\documentclass[11pt,openany]{amsart}
\usepackage{a4wide,amssymb,color}
\usepackage{enumitem}
\usepackage{mathrsfs,mathtools}
\usepackage{textcomp}
\numberwithin{equation}{section}

\makeindex

\usepackage[all]{xy}
\usepackage{hyperref}
\usepackage{caption}

\usepackage{graphicx}

\makeatletter
\DeclareRobustCommand{\cev}[1]{%
  {\mathpalette\do@cev{#1}}%
}
\newcommand{\do@cev}[2]{%
  \vbox{\offinterlineskip
    \sbox\z@{$\m@th#1 x$}%
    \ialign{##\cr
      \hidewidth\reflectbox{$\m@th#1\vec{}\mkern4mu$}\hidewidth\cr
      \noalign{\kern-\ht\z@}
      $\m@th#1#2$\cr
    }%
  }%
}
\makeatother

\newcommand{\fdom}{\mbox{\rm\textborn}}
\newcommand{\fran}{\mbox{\rm\textdied}}

\newcommand{\IR}{\mathbb R}
\newcommand{\IQ}{\mathbb Q}
\newcommand{\IC}{\mathbb C}
\newcommand{\IN}{\mathbb N}
\newcommand{\IZ}{\mathbb Z}
\newcommand{\UU}{\mathbf{U}}

\newcommand{\LL}{\mathbf L}
\newcommand{\Fun}{\mathbf{Fun}}
\newcommand{\pr}{\mathrm{pr}}
\newcommand{\cf}{\mathrm{cf}}
\newcommand{\op}{\mathrm{op}}
\newcommand{\true}{\mathsf{true}}
\newcommand{\false}{\mathsf{false}}
\newcommand{\ev}{\mathsf{ev}}

\newcommand{\Cut}{\mathsf{Cut}}
\newcommand{\Gap}{\mathsf{Gap}}
\newcommand{\birth}{\mathsf{birth}}
\newcommand{\minus}{\mbox{\tt-}}

\newcommand{\Mor}{\mathsf{Mor}}
\newcommand{\Ob}{\mathsf{Ob}}
\newcommand{\C}{\mathcal C}
\newcommand{\Set}{\mathbf{Set}}
\newcommand{\Iso}{\mathsf{Iso}}
\newcommand{\Mono}{\mathsf{Mono}}
\newcommand{\Epi}{\mathsf{Epi}}

\newcommand{\w}{\omega}
\newcommand{\e}{\varepsilon}

\newcommand{\dom}{\mathsf{dom}}
\newcommand{\rng}{\mathsf{rng}}
\newcommand{\A}{\mathcal A}
\newcommand{\Free}{\mathsf{Free}}

\newcommand{\Proj}{\mathsf{Pr}}
\newcommand{\Id}{\mathbf{Id}}
\newcommand{\Succ}{\mathsf{Succ}}

\newcommand{\NBG}{\mathsf{NBG}}

\newcommand{\CST}{\mathsf{CST}}

\newcommand{\Crd}{\mathbf{Card}}

\newcommand{\AC}{\mathsf{AC}}

\newcommand{\Ord}{\mathbf{On}}
\newcommand{\On}{\mathbf{On}}
\newcommand{\Tr}{\mathbf{Tr}}
\newcommand{\MM}{\mathbf{M}}

\newcommand{\E}{\mathbf E}

\newcommand{\IV}{\mathbf V}
\newcommand{\VV}{\mathbf V}
\newcommand{\UTC}{\mathsf{TC}}
\newcommand{\dG}{\ddot{\mathsf G}}

\newcommand{\No}{\mathbf{No}}

\newcommand{\Ra}{\Rightarrow}
\newcommand{\rank}{\mathsf{rank}}
\newcommand{\Go}{\mathsf{G}}

\newcommand{\sqcuplus}{\uplus}
\newcommand{\defeq}{\coloneqq}

\newtheorem{theorem}{Theorem}[section]
\newtheorem{proposition}[theorem]{Proposition}
\newtheorem{lemma}[theorem]{Lemma}
\newtheorem{corollary}[theorem]{Corollary}

\theoremstyle{definition}
\newtheorem{definition}[theorem]{Definition}
\newtheorem{example}[theorem]{Example}
\newtheorem{exercise}[theorem]{Exercise}
\newtheorem{Exercise}[theorem]{Exercise$^\star$}
\newtheorem{remark}[theorem]{Remark}

\newtheorem{claim}[theorem]{Claim}

\begin{document}

\title[Classical Set Theory: Theory of Sets and Classes]{Classical Set Theory:\\ Theory of Sets and Classes}
\author{Taras Banakh}
\maketitle

\setcounter{tocdepth}{1}
\tableofcontents
\newpage

\part*{Introduction}

\rightline{\em Make things as simple as possible, but not simpler.}

\rightline{Albert Einstein}
\vskip15pt

\rightline{\em Do not ask whether a statement is true until you know what it means.}

\rightline{Errett Bishop}
\vskip15pt





\rightline{\em Beauty is the first test: there is no permanent place in this world for ugly mathematics.}

\rightline{Godfrey Harold Hardy}
\vskip25pt

This is a short introductory course to Set Theory and Category Theory, based on axioms of von Neumann--Bernays--G\"odel (briefly NBG). The text can be used as a base for a  lecture course in Foundations of Mathematics, and contains a reasonable minimum which a good (post-graduate) student in Mathematics should know about foundations of this science. 

My aim is to give strict definitions of all set-theoretic notions and concepts that are widely used in mathematics. In particular, we shall introduce the sets $\IN,\IZ,\IQ,\IR,\IC$ of numbers (natural, integer, rational, real, complex) and will prove their basic order and algebraic properties. Since the system of NBG axioms is finite and does not involve advanced logics, it is more friendly for beginners than other axiomatic set theories (like ZFC). 

The legal use of classes in NBG will allow us to discuss freely Conway's surreal numbers that form an ordered field $\mathbf{No}$, which is a proper class and hence is not ``visible'' in ZFC. Also  the language of NBG allows to give natural definitions of some basic notions of Category Theory:  category, functor, natural transformation, which is done in the last part of this book.
\vskip30pt

I would like to express my thanks for the help in writing this text to:
\begin{itemize}
\item Pace Nielsen who motivated my interests in NBG;
\item Uliana Banakh (my daughter) who was the first reader of this text;
\item the participants of Zoom-seminar in Classical Set Theory (Serhii Bardyla, Oleksandr Maslyuchenko, Misha Popov, Alex Ravsky and others) for many valuable comments;
\item Asaf Karagila for his valuable remarks concerning the Axiom of Choice;
\item Alexandru-Andrei Bosinta, Emil  Je\v r\'abek, Joel David Hamkins and Ali Enayat for their help in understanding the phenomenon of constructibility; 
\item Yaroslav Knyazev and Artem Gak for careful reading the manuscript and many helpful suggestions concerning the presentation;
\item Kimiya Owashi for his valuable remark concerning the (im)possibility to avoid the Russell Paradox in Intuitionistic Mathematics.
\end{itemize}

\vskip30pt

\rightline{\em Lviv (at the time of COVID-19 quarantine)}

\rightline{\em March--May, 2020.}

\part{Naive Set Theory}

\rightline{\em A set is a Many that allows itself}

\rightline{\em to be thought of as a One.}
\smallskip

\rightline{Georg Cantor}

\vskip30pt

\section{Origins of Set Theory} The origins of (naive) Set Theory were laid at the end of XIXth century by Georg Cantor\footnote{{\bf Exercise:} Read about Georg Cantor in Wikipedia.} (1845--1918) in his papers published in  1874--1897. 

Cantor's ideas made the notion of a set the principal (undefined) notion of Mathematics, which can be used to give precise definitions of all other mathematical concepts such as numbers or functions.

According to Cantor, a set is an arbitrary collection of objects, called {\em elements} of the set. In particular, sets can be elements of other sets. The fact that a set $x$ is an element of a set $y$ is denoted by the symbol $x\in y$. If $x$ is not an element of $y$, then we write $x\notin y$.

A set consisting of finitely many elements $x_1,\dots,x_n$ is written as $\{x_1,\dots,x_n\}$. Two sets $x,y$ are {\em equal} (denoted by $x=y$) if they consist of the same elements. For example, the sets $\{x,y\}$ and $\{y,x\}$ both have elements $x,y$ and hence are equal. 

A set containing no elements at all is called {\em the empty set} and is denoted by $\emptyset$. Since sets with the same elements are equal, the empty set is unique.

The theory developed so far, allows us to give a precise meaning to natural numbers (which are abstractions created by humans to facilitate counting):
$$
\begin{aligned}
& 0=\emptyset,\\
& 1=\{ 0\},\\
& 2=\{ 0, 1\},\\
& 3=\{ 0,1,2\},\\
&4=\{0,1,2,3\},\\
&5=\{0,1,2,3,4\},\\
&\dots
\end{aligned}
$$
The set $\{0,1,2,3,4,5,\dots\}$ of all natural numbers\footnote{{\bf Remark:} There are two meanings (Eastern and Western) of what to understand by a natural number. The Western approach includes zero in the natural numbers whereas the Eastern tradition does not. This difference can be noticed in numbering  floors in buildings in western or eastern countries. } is denoted by $\omega$. The set $\{1,2,3,4,5,\dots\}$ of non-zero natural numbers is denoted by $\IN$.

Very often we need to create a set of objects possessing some property (for example, the set of odd numbers). In this case we use the \index{constructor}{\em constructor} $\{x:\varphi(x)\}$, which yields exactly what we need: the set $\{x:\varphi(x)\}$ of all objects $x$ that have certain property $\varphi(x)$. 

Using the constructor we can define some basic ``algebraic'' operations over sets $X,Y$:
\begin{itemize}
\item the \index{intersection}{\em intersection} $X\cap Y=\{x:x\in X\;\wedge\;x\in Y\}$ whose elements are objects that belong to $X$ and  $Y$;
\item the \index{union}{\em union} $X\cup Y=\{x:x\in X\;\vee\;x\in Y\}$ consisting of the objects that belong to $X$ or $Y$ or to both of them;
\item the \index{difference}{\em difference} $X\setminus Y=\{x:x\in X\;\wedge\;x\notin Y\}$ consisting of elements that belong to $X$ but not to $Y$;
\item the \index{symmetric difference}{\em symmetric difference} $X\triangle Y=(X\cup Y)\setminus(X\cap Y)$ consisting of elements that belong to the union $X\cup Y$ but not to the intersection $X\cap Y$.
\end{itemize}

In the formulas for the union and intersection we used the logical  connectives $\wedge$ and $\vee$ denoting the logical operations {\tt and} and {\tt or}. Below we present the truth table for these logical operations and also for three other logical operations: the {\tt negation} $\neg$, the {\tt implication} $\Ra$,  the {\tt  equivalence} $\Leftrightarrow$, and the {\tt Sheffer stroke} $|$, called also {\tt nand}.

\begin{center}
\begin{tabular}{|c|c||c|c|c|c|c|c|}
\hline
$x$&$y$&$x\wedge y$&$x\vee y$&$x\Rightarrow y$&$x\Leftrightarrow y$&$x|y$&$\neg x$\\
\hline
0&0&0&0&1&1&1&1\\
0&1&0&1&1&0&1&1\\
1&0&0&1&0&0&1&0\\
1&1&1&1&1&1&0&0\\
\hline
\end{tabular}
\end{center}

Therefore, we have four basic operations over sets $X,Y$:
$$\begin{aligned}
&X\cap Y=\{x:x\in X\;\wedge\;x\in Y\},\quad X\cup Y=\{x:x\in X\;\vee\;x\in Y\},\\
&X\setminus Y=\{x:x\in X\;\wedge\;x\notin Y\},\quad X\triangle Y=(X\cup Y)\setminus (X\cap Y).
\end{aligned}
$$
\begin{exercise} For the sets $X=\{0,1,2,4,5\}$ and $Y=\{1,2,3,4\}$, find $X\cap Y$, $X\cup Y$, $X\setminus Y$, $X\triangle Y$.
\end{exercise}

\begin{exercise} Using truth tables, check that the logical functions $\neg,\vee,\wedge,\rightarrow,\leftrightarrow$ are expressible via the Sheffer stroke:
\begin{enumerate}
\item $\neg x$ is equal to $x|x$;
\item $x\wedge y$ is equal to $(x|y)|(x|y)$;
\item $x\vee y$ is equal to $(x|x)|(y|y)$;
\item $x\rightarrow y$ is equal to $x|(y|y)$; 
\item $x\leftrightarrow y$ is equal to $((x|x)|(y|y))|(x|y)$.
\end{enumerate}  
\end{exercise}

\section{Berry's Paradox}

\rightline{\em ``The essence of mathematics is its freedom"}

\rightline{Georg Cantor}
\vskip30pt

\index{Paradox!of Berry} In 1906 Bertrand Russell, a famous British philosopher, published a paradox, which he attributed to G.~Berry (1867--1928), a junior librarian at Oxford's Bodleian library. 

To formulate this paradox, observe that each natural number can be described by some property. For example, zero is the smallest natural number, one is the smallest nonzero natural number, two is the smallest prime number, three is the smallest odd prime number, four is the smallest square, five is the smallest odd prime number which is larger than the smallest square and so on. 

Since there are only finitely many sentences of a given length, such sentences (of given length) can describe only finitely many numbers, \footnote{The list of short descriptions of the first 10000 numbers can be found here:\newline {\tt https://erich-friedman.github.io/numbers.html}}. Consequently, infinitely many numbers cannot be described by short sentences, consisting of fewer than 100 symbols. Among such numbers take the smallest one and denote it by $s$. Now consider the characteristic property of this number: {\em $s$ is the smallest number that cannot be described by a sentence consisting less than 100 symbol}. But the latter sentence consists of 96 symbols, which is less than 100, and uniquely defines the number $s$. 

Now we have a paradox\footnote{A similar argument can be used to prove that every natural number has some interesting property. Assuming that there  exist natural numbers without interesting properties, we can consider the smallest element of the set of ``non-interesting'' numbers and this number has an interesting property: it is the smallest non-interesting number.}: on one hand, the number $s$ belongs to the set of numbers that cannot be described by short sentences, and on the other hand, it has a short description. Where is the problem?

The problem is that the description of $s$ contains a quantifier that runs over all sentences including itself. Using such self-referencing properties can lead to paradoxes, in particular, to Berry's Paradox. This means that not all properties $\varphi(x)$ can be used for defining mathematical objects, in particular, for constructing sets of the form $\{x:\varphi(x)\}$.

In order to avoid the Berry Paradox at constructing sets $\{x:\varphi(x)\}$,  mathematicians decided to use only precisely defined properties $\varphi(x)$, which do not include the property appearing in  Berry's paradox. Correct properties are described by formulas with one free variable in the language of Set Theory.

\section{The language and formulas of Set Theory} \label{s:language}

For describing the language of Set Theory we use our natural language, which will be called the {\em metalanguage} (with respect to the language of Set Theory). 

We start describing the language of Set Theory with describing its {\em alphabet}, which is an infinite list 
of symbols that necessarily includes the following {\em special symbols}:
\begin{itemize}
\item the symbols of binary relations: the equality ``$=$'' and membership  ``$\in$'';
\item logical connectives:  $|,\neg,\wedge,\vee,\rightarrow,\leftrightarrow$;
\item quantifiers: $\forall$ and $\exists$;
\item parentheses: ``('' and ``)''.
\end{itemize}All remaining (that is, non-special) symbols of the alphabet are called the {\em symbols of variables}. The list of those symbols is denoted by $\mathsf{Var}$. For symbols of variables we shall use small and large symbols of Latin alphabet: $a,A,b,B,c,C,u,U,v,V,w,W,x,X,y,Y,z,Z$ etc.

Sequences of symbols of the alphabet are called {\em words}. Well-defined words are called \index{Formula}{\em formulas}.

\begin{definition} \index{formula}Formulas are defined inductively by the following rules:
\begin{enumerate}
\item if $x,y$ are symbols of variables, then the word $x\in y$ is a formula (called an \index{atomic formula}{\em atomic formula});
\item if $\varphi,\psi$ are formulas, then $\neg(\varphi)$,  $(\varphi)\wedge(\psi)$, $(\varphi)\vee(\psi)$, $(\varphi)\to(\psi)$,  $(\varphi)\leftrightarrow(\psi)$ and $(\varphi)|(\psi)$ are formulas;
\item if $x$ is a symbol of a variable and $\varphi$ is a formula, then $\exists x\;(\varphi)$ and $\forall x\;(\varphi)$ are formulas.
\end{enumerate}
\end{definition}


For two symbols of variables $x,y$ the formula $x=y$ is a shorthand version of the formula $\forall z\;((z\in x)\leftrightarrow(z\in y))$ where $z$ is a symbol of variable, not equal to $x$ or $y$.

Writing formulas we shall often omit parentheses when there will be no ambiguity.
Putting back parentheses, we use the following preference order for logical operations:
$$\neg\;\; \wedge\;\;\vee\;\;\rightarrow\;\;\leftrightarrow.$$

\begin{example} The formula  
$\exists u\;\exists v\;(((\neg(u=v))\wedge (u\in x))\wedge(v\in x))$ can be written shortly as 
$\exists u\;\exists v\;(u\ne v\;\wedge\;u\in x\;\wedge\;v\in x)$. This formula describes the property of $x$ to have at least two elements.
\end{example}

For a formula $\varphi$ and symbols of variables $x,y$ the formulas $$\forall x\in y\;(\varphi)\mbox{  and }\exists x\in y\;(\varphi)$$ are shortened versions of the formulas $$\forall x\;((x\in y)\;\to\;(\varphi))\mbox{ and }\exists x\;((x\in y)\;\wedge\;(\varphi)),$$ respectively. In this case we say that the quantifiers $\forall x\in y$ and $\exists x\in y$ are {\em bounded} (more precisely, {\em $y$-bounded}).

Properties $\varphi(x)$ of sets that can be used in the constructors $\{x:\varphi(x)\}$ correspond to formulas with a unique free variable $x$. 

\begin{definition} For any formula $\varphi$ its \index{free variable}{\em set of free variables} $\Free(\varphi)$ is defined by induction on the complexity of the formula according to the following rules:
\begin{enumerate}
\item If $\varphi$ is the atomic formula $x\in y$, then $\Free(\varphi)=\{x,y\}$;
\item for any formulas $\varphi,\psi$ we have $\Free(\varphi|\psi)=\Free(\varphi)\cup\Free(\psi)$.
\item for any formula $\varphi$ we have $\Free(\exists x\;\varphi)=\Free(\varphi)\setminus\{x\}$.
\end{enumerate}
\end{definition}


\begin{example} The formula $\exists u\;(u\in x)$ has $x$ as a unique free variable and describes the property of a set $x$ to be non-empty. 
\end{example}

\begin{example} The formula $\exists u\;((u\in x)\;\wedge\;\forall v\;((v\in x)\to(v=u)))$ has $x$ as a unique free variable $x$ and describes the property of a set $x$ to be a singleton. 
\end{example}

\begin{exercise} Write down a formula representing the property of a set $x$ to contain
\begin{itemize}
\item at least two elements;
\item exactly two elements;
\item exactly three elements.
\end{itemize}
\end{exercise}

\begin{exercise} Write down  formulas representing the property of a set $x$ to be equal to $0$, $1$, $2$, etc.
\end{exercise}

\begin{Exercise} Suggest a formula $\varphi(x)$ such that each set $x$ satisfying this formula is infinite. 
\end{Exercise}

\begin{exercise} Explain why the property appearing in Berry's Paradox cannot be described by a formula of Set Theory.
\end{exercise}

\section{Russell's Paradox}

As we already know, Berry's Paradox can be avoided by formalizing the notion of a property. A much more serious problem for foundations of Set Theory   was discovered in 1901 by Bertrand Russell (1872--1970) who suggested the following paradox.
\medskip

\noindent{\bf Russell's Paradox.}\index{Paradox!of Russell} {\em Consider the property $\varphi$ of a set $x$ to not contain itself as an element. This property is represented by the well-defined formula $x\notin x$. Many sets, for example, all natural numbers, have the property $\varphi$. Next, consider the set $A=\{x:x\notin x\}$ of all sets $x$ that have the property $\varphi$. For this set $A$ two cases are possible:
\begin{enumerate}
\item[\textup{(1)}] $A$ has the property $\varphi$ and hence belongs to the set $A=\{x:\varphi(x)\}$, which contradicts the property $\varphi$ (saying that $A\notin A$);
\item[\textup{(2)}]  $A$ fails to have the property $\varphi$  and then $A\notin \{x:\varphi(x)\}=A$, which means that $A\notin A$ and $A$ has the property $\varphi$.
\end{enumerate}
In both cases we have a contradiction.}
\medskip

This paradox shows that some families of form $\{x:\varphi(x)\}$ are too large to be called sets. One of the ways of avoiding the Russell Paradox is to forbid the unrestricted constructor $\{x:\varphi(x)\}$ replacing it by its  restricted version $\{x\in y:\varphi(x)\}$.  The set $\{x\in y:\varphi(x)\}$ consists of all elements of a set $y$ that have property $\varphi(x)$. This approach resulted in appearance of the Axiomatic Set Theory of Zermelo--Fraenkel, used by many modern mathematicians. 



An alternative way of resolving the Russell Paradox consists in 
introducing a notion of \index{class}{\em class} for describing families of sets that are too big to be elements of other classes. An example of a class is the family $\{x:x\notin x\}$ appearing in the Russell's Paradox.  Then sets are defined as ``small'' classes. They are elements of other classes. This approach resulted in the appearance of the Axiomatic Set Theory of von Neumann--Bernays--G\"odel, abbreviated by NBG in Mendelson \cite{Mendelson}. Exactly this axiomatic system NBG will be taken as a base for presentation of Set Theory in this textbook. Since NBG allows us to speak about sets and classes and NBG was created by classics (von Neumann, Bernays, G\"odel), we will call it the Classical Set Theory.  
\newpage

\part{Axiomatic Theories of Sets and Classes}

\rightline{\em Aus dem Paradies, das Cantor uns geschaffen,}

\rightline{\em  soll uns niemand vertreiben k\"onnen}
\smallskip

\rightline{David Hilbert}
\vskip15pt

In this part we present axioms of von Neumann--Bernays--G\"odel and discuss the relation of these axioms to the Zermelo--Fraenkel axioms of Set Theory.

\section{Axioms of von Neumann--Bernays--G\"odel}\label{s:axioms}

In this section we shall list 14 axioms of von Neumann--Bernays--G\"odel and also introduce some new notions and notations on the base of these axioms.

In fact, the language of the Classical Set Theory has been described in Section~\ref{s:language}. We recommend the reader (if he or she is not fluent in Logics) to return to this section and read it once more. 

The unique undefined notions of the theory NBG are the notions of a \index{class}{\em class} and \index{element}{\em element}. 

Classes can be elements of other classes. The fact that a class $X$ is an element of a class $Y$ is written as $X\in Y$.
The negation of $X\in Y$ is written as $X\notin Y$. So, $X\notin Y$ is a short version of $\neg(X\in Y)$.

\begin{definition}\label{d:equality} Two classes $X,Y$ are {\em equal} (denoted by $X=Y$) if they have the same elements, i.e., $\forall Z\;(Z\in X\;\leftrightarrow\; Z\in Y)$.
\end{definition}

\begin{exercise}\label{ex:equality} Prove that the equality of classes has the following three properties:
\begin{enumerate}
\item $\forall X\;(X=X)$;
\item $\forall X\;\forall Y\;((X=Y)\;\rightarrow\;(Y=X))$;
\item $\forall X\;\forall Y\;\forall Z\;\;((X=Y)\;\wedge\;(Y=Z))\;\rightarrow\;(X=Z)$.
\end{enumerate}
\end{exercise}
 
The negation of the equality $X=Y$ is denoted by $X\ne Y$, i.e., $X\ne Y$ is a shorthand of the formula $\neg(X=Y)$.

\begin{definition} Given two classes $X,Y$, we write $X\subseteq Y$ and say that the class $X$ is a \index{subclass}{\em subclass} of a class $Y$ if $\forall z\in X\;(z\in Y)$, i.e., each element of the class $X$ is an element of the class $Y$. If $X$ is not a subclass of $Y$, then we write $X\not\subseteq Y$.

For two classes $X,Y$ we write $X\subset Y$ iff $X\subseteq Y$ and $X\ne Y$.
\end{definition}

Observe that two classes $X,Y$ are equal if and only if $X\subseteq Y$ and $Y\subseteq X$, i.e., $$\forall X\;\forall Y\;(X=Y\:\Leftrightarrow\;(X\subseteq Y\;\wedge\;Y\subseteq X)).$$

Now we start listing the axioms of NBG. 
\medskip

\index{Axiom of Equality}\noindent$\boxed{\mbox{\bf Axiom of Equality:\quad} \forall X\;\forall Y\; (X=Y\;\rightarrow\; \forall Z\;
(X\in Z\;\leftrightarrow\;Y\in Z))\;}$
\smallskip

The axiom of equality says that equal classes behave equally with respect to the membership relation $\in$. 
\medskip

\begin{definition} A class $X$ is defined to be a \index{set}{\em set} if $X$ is an element of some other class. More formally, a class $X$ is a set if $\exists Y\;(X\in Y)$. 
\end{definition}

\begin{definition} A class which is not a set is called a \index{proper class}{\em proper class}.
\end{definition}

To distinguish sets from proper classes, we shall use small characters (like $x,y,z,u,v,w,a,b,c$) for denoting sets, capital letters (like $X,Y,Z,U,V,W,A,B,C$) for denoting classes and boldface characters (like $\mathbf X,\mathbf Y,\UU,\mathbf E$) for denoting proper classes. Applying this convention, we write down our next axiom.
\medskip

\index{Axiom of Pair}\noindent$\boxed{\mbox{\bf Axiom of Pair:\quad}\\
\forall x\;\forall y\;\exists z\; \forall u\;(u\in z\leftrightarrow (u=x)\vee (u=y))\;\;}$
\medskip

The Axiom of Pair says that for any sets $x,y$ there exists a set $z$ whose unique elements are $x$ and $y$. By the definition of equality of sets, such set $z$ is unique. It is called the \index{unordered pair}\index{pair!unordered}{\em unordered pair} of the sets $x,y$ and is denoted by $\{x,y\}$. 

In the Axiom of Pair we implicitly assume that $x,y,z$ are sets. The expanded (true) version of this axiom reads as follows.
$$\boxed{\forall x\;\forall y\;\big((\exists X\;\exists Y\;(x\in X\;\wedge\;y\in Y))\to \exists z\;\exists Z\;(z\in Z\;\wedge\;\forall u\;(u\in z\;\leftrightarrow\; (u=x)\vee (u=y))\big)}$$



\begin{proposition}\label{p:equal-pairs} Let $x,y,u,v$ be sets. If $x=u$ and $y=v$, then $\{x,y\}=\{v,u\}$.
\end{proposition}

\begin{proof} Assume that $x=u$ and $y=v$. By Definition~\ref{d:equality}, the equality $\{x,y\}=\{v,u\}$ will follow as soon as for any set $z$ we prove that $(z\in\{x,y\})\;\leftrightarrow\;(z\in\{v,u\})$. If $z\in \{x,y\}$, then $((z=x)\vee (z=y))$ by Axiom of Pair and definition of the unordered pair $\{x,y\}$. By Exercise~\ref{ex:equality}(3), 
$((z=x)\vee (z=y))\wedge(x=u)\wedge(y=v)$ implies $(z=u)\vee(z=v)$. Then $z\in\{v,u\}$ by the definition of $\{v,u\}$. Therefore, $(z\in\{x,y\})\;\rightarrow\;(z\in\{v,u\})$. By analogy we can prove that  $(z\in\{v,u\})\;\rightarrow\;(z\in\{x,y\})$.
\end{proof}

\begin{corollary} For any sets $x,y$, we have $\{x,y\}=\{y,x\}$.
\end{corollary}

For any set $x$, the unordered pair $\{x,x\}$ is denoted by $\{x\}$ and is  called a \index{singleton}{\em singleton}. 

\begin{proposition}\label{p:equal}  Two sets $x,y$ are equal if and only if $\{x\}=\{y\}$.
\end{proposition}

\begin{proof} The ``only if'' part follows from Proposition~\ref{p:equal-pairs}. To prove the ``if'' part, assume that $\{x\}=\{y\}$. By definition of the singleton, $x\in\{x\}$. By the definition of equality, $x\in \{x\}=\{y\}$ implies $x\in\{y\}$. Then $x=y$ by the definition of the singleton $\{y\}$.
\end{proof}

\begin{exercise} Prove that two sets $x,y$ are equal if and only if $\forall Z\;(x\in Z\;\leftrightarrow\;y\in Z)$.
\smallskip

\noindent{\em Hint:} Apply Axiom of Equality and Proposition~\ref{p:equal}.
\end{exercise} 

\begin{definition}[Kuratowski, 1921] The \index{ordered pair}\index{pair!ordered}{\em ordered pair} $\langle x,y\rangle$ of sets $x,y$ is the set $\{\{x\},\{x,y\}\}$.
\end{definition} 

\begin{proposition}\label{p:pair} For sets $x,y,u,v$ the ordered pairs $\langle x,y\rangle$ and $\langle u,v\rangle$ are equal if and only if $x=u$ and $y=v$. More formally,
$$\forall x\;\forall y\;\forall u\;\forall v\;\big((\langle x,y\rangle=\langle u,v\rangle)\;\leftrightarrow\;(x=u\;\wedge\;y=v)\big).$$
\end{proposition}

\begin{proof} The ``if'' part follows from Proposition~\ref{p:equal-pairs}. To prove the ``only if'' part, assume that 
\begin{equation}\label{eq:pair}
\{\{x\},\{x,y\}\}=\langle x,y\rangle=\langle u,v\rangle=\{\{u\},\{u,v\}\}
\end{equation}
but $x\ne u$ or $y\ne v$.

First assume that $x\ne u$. By Proposition~\ref{p:equal}, $\{x\}\ne\{u\}$. 
By definition of the pair $\langle x,y\rangle$, we have $\{x\}\in\{\{x\},\{x,y\}\}=\langle x,y\rangle=\langle u,v\rangle$. By the definition of equality, $\{x\}\in \langle u,v\rangle=\{\{u\},\{u,v\}\}$. The definition of the pair $\{\{u\},\{u,v\}\}$ and non-equality $\{x\}\ne\{u\}$ imply $\{x\}=\{u,v\}$ and hence $u=x$, which contradicts our assumption. 

This contradiction shows that $x=u$. If $x=y$, then $\{\{u\},\{u,v\}\}=\{\{x\},\{x,y\}\}=\{\{x\}\}$ and Proposition~\ref{p:equal} imply that $u=v$ and hence $y=x=u=v$, according to Exercise~\ref{ex:equality}.
If $x\ne y$, then $\{u\}\ne\{x,y\}\in\{\{u\},\{u,v\}\}$ implies $\{x,y\}=\{u,v\}$ and hence $y=v$.
\end{proof}

Using the notion of an ordered pair, we can introduce ordered triples, quadruples etc.

Namely, for any sets $x,y,z$ the \index{ordered triple}{\em ordered triple} $\langle x,y,z\rangle$ is the set $\langle \langle x,y\rangle,z\rangle$.
\medskip

\begin{definition} A class $R$ is called a \index{relation}{\em relation} if its elements are ordered pairs. More formally,
$$\mbox{$R$ is a relation}\;\Leftrightarrow\;\forall z\in R\;\exists x\;\exists y\;(z=\langle x,y\rangle).$$
\end{definition}

The following axiom postulates the existence of the relation \index{class!$\mathbf E$}\index{membership relation}\index{relation!of membership} 
$$\mathbf E=\{\langle x,y\rangle:x\in y\},$$
which encodes the undefined relation $\in$.
\medskip
 
\index{Axiom of Membership}\noindent$\boxed{\mbox{\bf  Axiom of Membership:\quad} \exists\mathbf E\;\forall z\;(z\in \mathbf E\leftrightarrow \exists x\;\exists y\;(x\in y\;\wedge\;z=\langle x,y\rangle))\quad}$
\medskip

Next, we define 5 axioms describing some basic operations over classes.
\medskip

\index{Axiom of Domain}\noindent$\boxed{\mbox{\bf  Axiom of Domain:\quad} \forall X\;\exists D\;\forall x\;(x\in D\;\leftrightarrow\;\exists y\;(\langle x,y\rangle\in X))\quad}$
\medskip

\noindent By the Axioms of Domain, for each class $X$ its
 \index{domain}\index{class!domain of}{\em domain} $$\mathsf{dom}[X]=\{x:\exists y\;\langle x,y\rangle\in X\}$$ exists.

In particular, the class $\dom[\mathbf E]$ exists. 
Observe that for every set $x$ the ordered pair $\langle x,\{x\}\rangle$ belongs to the class $\mathbf E$ and hence $x\in \dom[\mathbf E]$, which implies that $\dom[\mathbf E]$ is the class of all sets. This important class will be denoted by $\mathbf U$ and called the {\em universe of sets}. The universe of sets $\mathbf U$ is unique by the definition of equality. The definition of the universe $\UU$ guarantees that $\forall X\;(X\subseteq\UU)$. 

\begin{exercise} Prove that $\dom[\UU]=\UU$.
\end{exercise}
\medskip

\index{Axiom of Difference}\noindent$\boxed{\mbox{\bf Axiom of Difference:\quad}\forall X\;\forall Y\;\exists Z\;\forall u\;(u\in Z\;\leftrightarrow\;(u\in X\;\wedge\;u\notin Y))\quad}$
\medskip

Axiom of Difference postulates that for any classes $X,Y$ the class 
\begin{itemize}
\item $X\setminus Y=\{x:x\in X\;\wedge\;x\notin Y\}$ exists.  
\end{itemize} Using the Axiom of difference, for any classes $X,Y$ we can define their
\begin{itemize}
\item \index{intersection}{\em intersection} $X\cap Y=X\setminus (X\setminus Y)$,
\item \index{union}{\em union}  $X\cup Y=\UU\setminus((\UU\setminus X)\cap(\UU\setminus Y))$, and
\item \index{symmetric difference}{\em symmetric difference} $X\triangle Y=(X\cup Y)\setminus (X\cap Y)$.
\end{itemize}

Applying the Axiom of Difference to the universal class $\UU$, we conclude that the \index{empty class}\index{class!$\emptyset$}{\em empty class} $$\emptyset=\UU\setminus\UU$$
 exists. The empty class contains no elements and is unique by the definition of equality.
 
 \begin{exercise} Prove that $\emptyset=\mathbf E\setminus\mathbf E$.
 \end{exercise}
\medskip

\index{Axiom of Product}\noindent$\boxed{\mbox{\bf  Axiom of Product:\quad} \forall X\;\forall Y\;\exists Z\;\forall z\; (z\in Z\;\leftrightarrow\;\exists x\in X\;\exists y\in Y\;(z=\langle x,y\rangle))\quad}$
\medskip

The Axiom of Product guarantees that for any classes $X,Y$, their \index{Cartesian product}\index{product}{\em Cartesian product}
$$X\times Y=\{\langle x,y\rangle:x\in X,\;y\in Y\}$$exists.

The product $\UU\times\UU$ will be denoted by $\ddot{\UU}$, and the product $\ddot{\UU}\times \UU$ by $\dddot{\UU}$.   The definition of a relation implies that a class $R$ is a relation if and only if $R\subseteq \ddot{\mathbf U}$.
\medskip


\medskip

\index{Axiom of Inversion}\noindent$\boxed{\mbox{\bf  Axiom of Inversion:\quad} \forall X\;\exists Y\;\forall p\;\big(p\in Y\;\leftrightarrow\;\exists x\;\exists y\;(\langle x,y\rangle\in  X\;\wedge\;p=\langle y,x\rangle)\big)\quad}$
\medskip

The Axiom of Inversion implies that for any class $X$ the relation $$X^{-1}=\{\langle y,x\rangle:\langle x,y\rangle\in X\}$$ exists. Observe that a class $X$ is a relation if and only if $X=(X^{-1})^{-1}$. 
\medskip


\noindent By the Axioms of Domain and Inversion, for each class $X$ its \index{range}\index{class!range of} {\em range} $$\mathsf{rng}[X]=\{y:\exists x\;\langle x,y\rangle\in X\}=\mathsf{dom}[X^{-1}]$$ 
exists.

\begin{exercise} Find $\mathsf{rng}[\mathbf E]$. 
\end{exercise}

\index{Axiom of Cycle}\noindent$\boxed{\mbox{\bf  Axiom of Cycle:\quad} \forall X\;\exists Y\;\forall t\;\big(t\in Y\;\leftrightarrow\;\exists x\;\exists y\;\exists z\; (\langle x,y,z\rangle \in X\;\wedge\;t=\langle z,x,y\rangle)\big)\quad}$
\medskip

The Axiom of Cycle implies that for every class $X$ the classes
$$X^\circlearrowright=\{\langle z,x,y\rangle:\langle x,y,z\rangle\in X\}\mbox{ \ and \ }
X^\circlearrowleft=\{\langle y,z,x\rangle:\langle x,y,z\rangle\in X\}=(X^\circlearrowright)^\circlearrowright$$
 exist.
\medskip




We can also define the union and intersection of sets that belong to a given class of sets. Namely, for a class $X$ consider its
\begin{itemize}
\item  \index{union}{\em union} $\bigcup X=\{z:\exists y\in X\;(z\in y)\}$,
\item  \index{intersection}{\em intersection} $\bigcap X=\{z:\forall y\in X\;(z\in y)\}$, and
\item \index{power-class}{\em power-class} $\mathcal P(X)=\{y:y\subseteq X\}$.
\end{itemize}

\begin{exercise} To show that the classes $\bigcup X$, $\bigcap X$ and $\mathcal P(X)$ exist, check that
$$
\begin{aligned}
&\textstyle{\bigcup}X=\mathsf{dom}[\mathbf E\cap(\UU\times X)]\\
&\textstyle{\bigcap}X=\UU\setminus \mathsf{dom}[(\UU\times X)\setminus \mathbf E], \mbox{ \ and \ }\\
&\mathcal P(X)=\UU\setminus \mathsf{dom}[\mathbf E^{-1}\setminus(\UU\times X)].
\end{aligned}
$$
\end{exercise}

\begin{exercise}  Prove that $\mathcal P(\UU)=\UU$.
\end{exercise}

For every relation $R$ and class $X$, the Axioms of Product, Inversion and Domain allow us to define the class
$$R[X]=\{y:\exists x\in X\;(\langle x,y\rangle\in R)\}=\mathsf{rng}[R\cap(X\times \UU)]$$ called the \index{image}{\em image} of the class $X$ under the relation $R$. The class $R^{-1}[X]$ is called the \index{preimage}{\em preimage} of $X$ under the relation $R$.

For two relations $F,G$ their\index{composition}\index{function!composition} {\em composition} $G\circ F$ is the relation
$$G\circ F=\{\langle x,z\rangle:\exists y\;(\langle x,y\rangle\in F\;\wedge\;\langle y,z\rangle\in G)\}.$$
The class $G\circ F$ exists since $G\circ F=\dom[T]$ where 
$$T=\{\langle x,z,y\rangle:\langle x,y\rangle\in F\;\wedge\;\langle y,z\rangle\in G\}=[F^{-1}\times\UU]^\circlearrowleft\cap [G^{-1}\times\UU]^\circlearrowright.$$



\smallskip

\begin{definition} A relation $F$ is called a \index{function}{\em function} if for any ordered pairs $\langle x,y\rangle,\langle x',y'\rangle\in F$ the equality $x=x'$ implies $y=y'$.
\end{definition}

Therefore, for any function $F$ and any $x\in\mathsf{dom}[F]$ there exists a unique set $y$ such that $\langle x,y\rangle \in F$. This unique set $y$ is called the \index{image}{\em image} of $x$ under the function $F$ and is denoted by $F(x)$. The round parentheses are used to distinguish the set $F(x)$ from the image $F[x]=\{y:\exists z\in x\;\langle z,y\rangle\in F\}$ of the set $x$ under the function $F$.



\smallskip

\begin{exercise} Prove that for any functions $F,G$ the relation $G\circ F$ is a function.
\end{exercise}

A function $F$ is called {\em injective} if the relation $F^{-1}$ is a function.

The next three axioms are called the axioms of existence of sets.
\medskip

\index{Axiom of Replacement}\noindent$\boxed{\mbox{{\bf  Axiom of Replacement:\quad} For every function $F$ and set $x$, the class $F[x]$ is a set}\quad}$
\medskip

\index{Axiom of Union}\noindent$\boxed{\mbox{{\bf  Axiom of Union:\quad} For every set $x$, the class $\bigcup x=\{z:\exists y\in x\;(z\in y)\}$ is a set}\quad}$
\medskip

\index{Axiom of Power-set}\noindent$\boxed{\mbox{{\bf  Axiom of Power-set:\quad} For every set $x$, the class $\mathcal P(x)=\{y:y\subseteq x\}$ is a set}\quad}$
\medskip

\begin{exercise} Write the Axioms of Replacement, Union and Power-set as formulas.
\end{exercise}

\begin{exercise} Show that for any sets $x,y$ the class $x\cup y$ is a set.
\smallskip

\noindent{\em Hint:} Use the Axioms of Pair and Union.
\end{exercise}


For a set $x$ the set $x\cup\{x\}$ is called the \index{successor}\index{set!successor}{\em successor} of $x$. The set $x\cup\{x\}$ is equal to $\cup\{x,\{x\}\}$ and hence exists by the  Axiom of Union.



At the moment no axiom guarantees that at least one set exists. This is done by

\medskip

\index{Axiom of Infinity}\noindent$\boxed{\mbox{{\bf  Axiom of Infinity:\quad} $\exists x\in\UU\;((\emptyset\in x)\wedge \forall n\;(n\in x\;\to\; n\cup\{n\}\in x))$}\quad}$
\medskip

A set $x$ is called \index{inductive class}\index{class!inductive}{\em inductive} if $(\emptyset\in x)\wedge \forall n\;(n\in x\;\to\; n\cup\{n\}\in x)$. The Axiom of Infinity guarantees the existence of an inductive set.
This axiom also implies that the empty class $\emptyset$ is a set. So, it is legal to form sets corresponding to natural numbers: $$ 0=\emptyset,\; 1=\{ 0\},\; 2=\{ 0, 1\},\;  3=\{ 0, 1, 2\},\; 4=\{ 0, 1, 2, 3\},\; 5=\{ 0, 1, 2, 3, 4\},\mbox{ \ and so on}.$$

 Let $\mathbf{Ind}$ be the class of all inductive sets (the existence of the class $\mathbf{Ind}$ is established in Exercise~\ref{ex:Ind}). The intersection $\bigcap\mathbf{Ind}$ of all inductive sets is the smallest inductive set, which is denoted by $\omega$. Elements of the $\omega$ are called  \index{natural number}{\em natural numbers} or else {\em finite ordinals}. The set\index{set!$\omega$}\index{$\omega$}\index{$\mathbb N$}\index{set!$\mathbb N$} $\IN=\w\setminus\{\emptyset\}$ is the set of non-zero natural numbers. 
 
 The definition of the set $\w$ as the smallest inductive set implies the well-known
 \smallskip
 
\noindent{\bf Principle of Mathematical Induction:}\index{Mathematical Induction}\index{Principle of Mathematical Induction} {\em If a set $X$ contains the empty set and for every $n\in X$ its successor $n\cup\{n\}$ belongs to $X$, then $X$ contains all natural numbers.}
\smallskip

A set $x$ is called \index{set!finite}\index{finite set}{\em finite} (resp. \index{set!countable}\index{countable set}{\em countable}) if there exists an  injective function $f$ such that $\dom[f]=x$ and $\rng[f]\in\w$ (resp. $\rng[f]\in \omega\cup\{\omega\}$). 



\smallskip

\index{Axiom of Foundation}\noindent$\boxed{\mbox{{\bf  Axiom of Foundation:\quad} $\forall x\in\UU\;(x\ne\emptyset\;\Rightarrow\;\exists y\in x\;\forall z\in y\;(z\notin x))$}\quad}$
\smallskip

The Axiom of Foundation says that each nonempty set $x$ contains an element $y\in x$ such that $y\cap x=\emptyset$. This axiom forbids the existence of a set $x$ such that $x\in x$. More generally, it forbids the existence of infinite sequences sets $x_1\ni x_2\ni x_3\ni\dots$.

The final axiom is  the
\smallskip

\index{Axiom of Global Choice}\index{({{\sf AGC}})}\noindent$\boxed{\mbox{\bf  Axiom of Global Choice:\;} \exists F\,((\mbox{$F$ is a function})\,\wedge\,\forall x\in\UU\,(x\ne\emptyset\,\Rightarrow\,\exists y\in x\,(\langle x,y\rangle{\in}F)))\,}$
\smallskip

The Axiom of Global Choice postulates the existence of a function $F:\UU\setminus\{\emptyset\}\to\UU$ assigning to each nonempty set $x$ some element $F(x)$ of $x$.

The Axiom of Global Choice implies its weaker version, called the
\smallskip

\index{Axiom of Choice}\noindent$\boxed{\mbox{\bf  Axiom of Choice:\;\;} \forall x\in\UU\,\exists f\,((\mbox{$f$ is a function})\,\wedge\,\forall y\in x\,(y\ne\emptyset\,\Rightarrow\,\exists z\in y\,(\langle y,z\rangle\in f)))\;}$
\medskip

The Axiom of Choice says that for any set $x$ there exists a function $f:x\setminus\{\emptyset\}\to\bigcup x$ assigning to every nonempty set $y\in x$ some element $f(y)$ of $y$.
\medskip

Therefore, the Classical  Set Theory is based on 14 axioms.
\medskip

\noindent\fbox{
\begin{minipage}{430pt}
\centerline{\bf \large Axioms of NBG}
\smallskip

\index{Axiom of Equality}\noindent{\bf Equality:} Two equal classes behave equally with respect to taking elements.
\smallskip


\index{Axiom of Pair}\noindent{\bf Pair:} For any sets $x,y$ the set $\{x,y\}$ exists.
\smallskip

\index{Axiom of Membership}\noindent{\bf Membership:} The class $\mathbf E=\{\langle x,y\rangle:\exists Y\;(x\in y\in Y)\}$ exists.
\smallskip

\index{Axiom of Domain}\noindent{\bf Domain:} For every class $X$ the class $\mathsf{dom}[X]=\{x:\exists y\;(\langle x,y\rangle\in X)\}$ exists.

\index{Axiom of Difference}\noindent {\bf Difference:} For any classes $X,Y$ the class $X\setminus Y=\{x:x\in X\;\wedge\;x\notin Y\}$ exists.

\index{Axiom of Product}\noindent{\bf Product:} For every classes $X,Y$ the class $X\times Y=\{\langle x,y\rangle:x\in X,\;y\in Y\}$ exists.

\index{Axiom of Inversion}\noindent{\bf Inversion:} For every class $X$ the class $X^{-1}=\{\langle y,x\rangle:\langle x,y\rangle\in X\}$ exists.

\index{Axiom of Cycle}\noindent{\bf Cycle:} For every class $X$ the class $X^\circlearrowright=\{\langle z,x,y\rangle:\langle x,y,z\rangle\in X\}$ exists.

\smallskip

\index{Axiom of Replacement}\noindent{\bf Replacement:} For every function $F$ and set $x$ the class $F[x]=\{F(y):y\in x\}$ is a set.

\index{Axiom of Union}\noindent{\bf Union:} For every set $x$ the class $\cup x=\{z:\exists y\in x\;(z\in y)\}$ is a set.

\index{Axiom of Power-set}\noindent{\bf Power-set:} For every set $x$ the class $\mathcal P(x)=\{y:y\subseteq x\}$ is a set.

\index{Axiom of Infinity}\noindent{\bf Infinity:} There exists an inductive set.
\medskip

\index{Axiom of Foundation}\noindent{\bf Foundation:} Every nonempty set $x$ contains an element $y\in x$ such that $y\cap x=\emptyset$.

\index{Axiom of Global Choice}\noindent{\bf Global Choice:} There is a function assigning to each nonempty set $x$ some element of $x$.
\end{minipage}
}
\medskip


\section{Basic Classes}

In this section we consider some basic classes which will often appear in the remaining part of the textbook. A class is called {\em basic} if it can be constructed from the membership relation $\E$ by application of finitely many operations of taking the difference, product, inversion, cycle and domain. For example, the universal class $\UU$ is basic because $\UU=\dom[\E]$. 

\begin{exercise} Prove that for basic classes $X,Y$, the classes $X\setminus Y$, $X\cap Y$, $X\cup Y$, $X\Delta Y$, $X\times Y$, $\dom[X]$, $X^{-1}$, $X^\circlearrowright$, $X^\circlearrowleft$, $\bigcup X$, $\bigcap X$, $\mathcal P(X)$ are basic.
\end{exercise}

The following results on the existence of basic classes  are written as exercises with solutions (called hints). Nonetheless we strongly recommend the reader to try to do all exercises without looking at hints (and without use the G\"odel's Theorem~\ref{t:class}  on existence of classes).

\begin{exercise}\label{ex:S} Prove that the class\index{class!$\mathbf{S}$}\index{{{\bf S}}} $\mathbf S=\{\langle x,y\rangle\in\ddot{\UU}:x\subseteq y\}$ exists and is basic.
\smallskip

\noindent {\it Hint:} Observe that $\ddot{\UU}\setminus \mathbf S=\dom[T]$ where\\
$T=\{\langle x,y,z\rangle\in\dddot{\UU}:z\in x\;\wedge\;z\notin y\}=[\mathbf E\times\UU]^\circlearrowleft\cap ((\ddot{\UU}\setminus \E^{-1})\times\UU)^\circlearrowright$.
\end{exercise}

\begin{exercise}\label{ex:Id} Prove that the identity function\index{class!{{\sf Id}}}\index{function!{{\sf Id}}}\index{{{\sf Id}}} $\mathsf{Id}=\{\langle x,y\rangle\in\ddot{\UU}:x=y\}$ exists and is a basic class.
\smallskip

\noindent{\it Hint:} Observe that $\mathsf{Id}=\mathbf S\cap\mathbf S^{-1}$ where $\mathbf S$ is the class from Exercise~\ref{ex:S}.
\end{exercise}

\begin{exercise}\label{ex:subclass} Prove that every subclass  $Y\subseteq x$ of a set $x$ is a set.
\smallskip

\noindent{\it Hint:} Observe that the identity function $F=\mathsf{Id}\cap(Y\times \UU)$ of $Y$ exists and apply the Axiom of Replacement to conclude that the class $F[x]=Y\cap x=Y$ is a set.
\end{exercise}




\begin{exercise}\label{ex:U-proper} Prove that the universe $\UU$ is a proper class.
\smallskip

\noindent{\em Hint:} Repeat the argument of Russell.
\end{exercise}

\begin{exercise}\label{ex:dom} Prove that the function\index{function!{\sf dom}} $\dom:\ddot{\UU}\to\UU,\;\dom:\langle x,y\rangle\mapsto x$,  exists and is a basic class.
\smallskip

\noindent{\it Hint:}  Observe that $\dom=\{\langle x,y,z\rangle\in\dddot{\UU}:z=x\}=[\mathsf{Id}\times\UU]^\circlearrowleft$
\end{exercise}

\begin{exercise}\label{ex:rng} Prove that the function\index{function!{\sf rng}} $\rng:\ddot{\UU}\to\UU,\;\rng:\langle x,y\rangle\mapsto y$,  exists and is a basic class.
\smallskip

\noindent{\it Hint:} Observe that $\rng=\{\langle x,y,z\rangle\in\dddot{\UU}:z=y\}=[\mathsf{Id}\times\UU]^\circlearrowright$.
\end{exercise}

\begin{exercise} Prove that the class $\mathbf E=\{\langle x,y\rangle:x\in y\in\mathbf U\}$ is proper.
\smallskip

\noindent{\em Hint:} Apply Exercises~\ref{ex:U-proper}, \ref{ex:dom}, and the Axiom of Replacement.
\end{exercise}

\begin{exercise} Prove that the function \index{function!{\sf pair}} $\mathsf{pair}:\ddot{\UU}\to\UU$, $\mathsf{pair}:\langle x,y\rangle\mapsto \langle x,y\rangle$ exists.
\smallskip

\noindent{\em Hint:} Observe that $\mathsf{pair}=\dddot{\UU}\cap\mathsf{Id}$.
\end{exercise}

\begin{exercise} Prove that the composition $G\circ F\defeq\{\langle x,z\rangle:\exists y\;\langle x,y\rangle\in F\;\wedge\;\langle y,z\rangle\in G\}$ of two basic relations is a basic relation.
\smallskip

\noindent{\em Hint:} Observe that $G\circ F=\dom[[F^{-1}\times\UU]^\circlearrowleft\cap[G^{-1}\times\UU]^\circlearrowright]$.
\end{exercise}

\begin{exercise}\label{ex:FRG} Let $R$ be a relation and $F,G$ be functions. Prove that the class\\ $F_RG=\{x\in\dom[F]\cap\dom[G]:\langle F(x),G(x)\rangle\in R\}$ exists. Prove that the class $F_RG$ is basic if so are the classes $F,R,G$.
\smallskip

\noindent{\it Hint:} Observe that 
$F_RG=
\{x:\exists y\;\exists z\;(\langle x,y\rangle\in F\;\wedge\;\langle x,z\rangle\in G\;\wedge\;\langle y,z\rangle\in R)\}=\mathsf{dom}[\mathsf{dom}[T]],$
where\\ $T=\{\langle x,y,z\rangle\in\dddot{\UU}:\langle x,y\rangle\in F\;\wedge\;\langle x,z\rangle\in G\;\wedge\;\langle y,z\rangle\in R\}=(F\times\UU)\cap [G^{-1}\times\UU]^\circlearrowleft\cap[R\times\UU]^\circlearrowright.$ 
\end{exercise}

\begin{exercise}\label{ex:Inv} Prove that the function \index{function!$\mathsf{Inv}$} $\mathsf{Inv}=\{\langle\langle x,y\rangle,\langle u,v\rangle\rangle\in \ddot\UU\times\ddot\UU:x=v\;\wedge\;y=u\}$ exists and is basic.
\smallskip

\noindent{\em Hint:} Observe that\\ $\mathsf{Inv}=\{z\in\ddot\UU\times\ddot\UU:\dom\circ \dom(z)=\rng\circ\rng(z)\;\wedge\;\rng\circ\dom(z)=\dom\circ\rng(z)\}$ and apply Exercises~\ref{ex:FRG} and \ref{ex:Id}.
\end{exercise}

\begin{exercise}\label{ex:succ} Prove that the function $\Succ=\{\langle x,y\rangle\in\ddot{\UU}:y=x\cup\{x\}\}$ exists and is basic.
\smallskip

\noindent{\em Hint:} Observe that 
$\ddot{\UU}\setminus\Succ=\{\langle x,y\rangle\in\ddot{\UU}:\exists z\;\neg(z\in y\;\Leftrightarrow\;(z\in x\;\vee\;z=x))\}=\mathsf{dom}[T],$
where $T=\{\langle x,y,z\rangle\in\dddot{\UU}:\neg (z\in y\;\Leftrightarrow\;(z\in x\;\vee\;z=x))\}=T_1\cup T_2\cup T_3$, and\\
$T_1=\{\langle x,y,z\rangle\in\dddot{\UU}:z\notin y\;\wedge\;z\in x\}=[(\ddot{\UU}\setminus \mathbf E^{-1})\times\UU]^\circlearrowright\cap[\mathbf E\times\UU]^\circlearrowleft;$\\
$T_2=\{\langle x,y,z\rangle\in\dddot{\UU}:z\notin y\;\wedge\;z=x\}=[(\ddot{\UU}\setminus \mathbf E^{-1})\times\UU]^\circlearrowright\cap [\mathsf{Id}\times\UU]^\circlearrowleft;$\\ 
$T_3=\{\langle x,y,z\rangle\in\dddot{\UU}:z\in y\;\wedge\;z\notin x\;\wedge\;z\ne x\}=[\mathbf E^{-1}\times\UU]^\circlearrowright\cap [(\ddot{\UU}\setminus(\mathbf E\cup \mathsf{Id}))\times\UU]^\circlearrowleft$.
\end{exercise}

\begin{exercise}\label{ex:Ind} The class $\mathbf{Ind}$ of all inductive sets exists and is basic.
\smallskip

\noindent{\em Hint:} Observe that $
\UU\setminus\mathbf{Ind}=\{x\in\UU:\emptyset\notin x\;\vee\;(\exists y\in x\;(y\cup\{y\}\notin x))\}=$\\
$\mathsf{rng}[(\{\emptyset\}{\times}\UU){\setminus}\mathbf E]\cup \mathsf{dom}[P]$, where $P=\{\langle x,y\rangle\in\ddot{\UU}:(y\in x)\;\wedge\;(y\cup\{y\}\notin x)\}=\mathbf E^{-1}\cap P'$ and
$P'=F_RG$ and $R=\ddot{\UU}\setminus \mathbf E$, $F=\Succ\circ \rng$, $G=\dom$.
\end{exercise}

\begin{exercise} Prove that the set $\w$ of natural numbers is a basic class.
\smallskip

\noindent{\em Hint:} Observe that $\w=\bigcap\mathbf{Ind}$ and apply Exercise~\ref{ex:Ind}.
\end{exercise}

\begin{exercise} Prove that the functions
$$\circlearrowright=\{\langle \langle x,y,z\rangle,\langle z,x,y\rangle\rangle:x,y,z\in \UU\}\quad\mbox{and}\quad \circlearrowleft=\{\langle \langle x,y,z\rangle,\langle y,z,x\rangle\rangle:x,y,z\in \UU\}$$exist and that $\circlearrowright=\circlearrowleft\circ\circlearrowleft$ are basic classes. Also check that $\circlearrowleft=\circlearrowright\circ\circlearrowright$.
\smallskip

\noindent{\em Hint:} Apply Exercises~\ref{ex:dom}, \ref{ex:rng}, \ref{ex:FRG}, \ref{ex:Id}.
\end{exercise} 

\begin{exercise} Prove that the function $\bigcup:\UU\to\UU$, $\bigcup:x\mapsto\bigcup x$, exists and is basic.
\smallskip

\noindent{\em Hint:} Observe that ${\bigcup}=\{\langle x,y\rangle:\forall z\;(z\in y\leftrightarrow \exists u\;(z\in u\;\wedge\;u\in x))\}=\ddot\UU\setminus\dom[(T_1\setminus T_2)\cup (T_2\setminus T_1)]$, where $T_1\defeq\{\langle\langle x,y\rangle,z\rangle: z\in y\}=\{\langle\langle y,z\rangle,x\rangle: z\in y\}^\circlearrowright=[(\ddot \UU\setminus\mathbf E^{-1})\times \UU]^\circlearrowright$ and $T_2\defeq \dom[\{\langle\langle\langle x,y\rangle,z\rangle, u\rangle: z\in u\;\wedge\;u\in x\}]=\dom[(\rng\circ\rng_{\mathbf E}\rng)\cap(\rng_{\mathbf E}\dom\circ\dom\circ\dom)]$.
\end{exercise}

\begin{exercise} Prove that the function $\mathsf{upair}:\UU\times\UU\to\UU$, $\mathsf{upair}:\langle x,y\rangle\mapsto \{x,y\}$, is basic.
\smallskip

\noindent{\em Hint:} Observe that $\mathsf{upair}=\bigcup\circ \mathsf{pair}$.
\end{exercise}








\begin{exercise}\label{ex:basic-singleton} Prove that for any basic set $x$, the set $\{x\}$ is basic.
\smallskip

\noindent{\em Hint:} Since $\{x\}=\{z:z=x\}=\{z:\forall u\;(u\in z\leftrightarrow u\in x\}$, it suffices to check that the class $\UU\setminus\{x\}=\{z:\exists u\;(u\in z\not\leftrightarrow u\in x)\}=\dom[\{\langle z,u\rangle:(u\in z\not\leftrightarrow u\in x\}]$ is basic, which follows from the basic property of the classes $\{\langle z,u\rangle:u\in z\}=\mathbf E^{-1}$ and $\{\langle z,u\rangle:u\in x\}=\UU\times x$.
\end{exercise} 

\begin{exercise}\label{ex:basic-doubleton} Prove that for any basic sets $x,y$, the set $\{x,y\}$ is basic.
\smallskip

\noindent{\em Hint:} By Exercise~\ref{ex:basic-singleton}, the singletons $\{x\}$ and $\{y\}$ are basic classes, Then the set $\{x,y\}=\UU\setminus((\UU\setminus\{x\})\setminus\{y\})$ is basic.
\end{exercise}

\begin{exercise} Prove that for every basic function $F$ and basic set $x\in\dom[F]$, the set $y\defeq F(x)$ is basic.
\smallskip

\noindent{\em Hint:} Observe that $y=\bigcup\rng[F\cap(\{x\}\times\UU)]$.
\end{exercise}

\begin{exercise}\label{ex:numbers-basic} Prove that the numbers $0,1,2,3,4,...$ are basic sets.
\end{exercise}

\begin{exercise} Can we conclude from Exercise~\ref{ex:numbers-basic} that every element of the set $\w$ is a basic set?

\noindent{\em Hint:} No. Why? 
\end{exercise}

\section{The complexity of formulas}

In this section we discuss the L\'evy complexity of formulas of the Classical Set Theory and will evaluate the complexity of some important formulas. Basic definitions related to formulas can be found in Section~\ref{s:language}.

A formula $\varphi$ is said to have \index{formula!with bounded quantifiers}{\em bounded quantifiers} if all quantifiers in $\varphi$ are of the form $\forall X\in Y$ or $\exists X\in Y$, where $X,Y$ are symbols of variables.  Such formulas form the first level $\Sigma_0=\Pi_0=\Delta_0$ of the \index{L\'evy hierarchy}{\em L\'evy hierarchy} of complexity of formulas.  More precisely, the level $\Sigma_0=\Pi_0=\Delta_0$ consists of all formulas which are equivalent to the formulas with bounded quantifiers. The higher levels of the L\'evy hierarchy are defined by induction: a formula $\varphi$ is 
\begin{itemize}
\item \index{$\Sigma_i$-formula}a {\em $\Sigma_{i+1}$-formula} if it is equivalent to a formula of the form $\exists X_1\exists X_2\dots\exists X_n\;(\psi)$ where $\psi$ is a $\Pi_i$-formula;
\item  \index{$\Pi_i$-formula}a {\em $\Pi_{i+1}$-formula} if it is equivalent to a formula of the form $\forall X_1\forall X_2\dots\forall X_n\;(\psi)$ where $\psi$ is a $\Sigma_i$-formula;
\item \index{$\Delta_i$-formula}a {\em $\Delta_{i+1}$-formula} if it is both a $\Sigma_{i+1}$-formula and a $\Pi_{i+1}$-formula.
\end{itemize}
So, $\Pi_0$-formulas, $\Sigma_0$-formulas, and $\Delta_0$-formulas are formulas equivalents to formulas with bounded quantifiers. 

The atomic formula $X\in Y$ has no quantifiers and hence is a $\Delta_0$-formula.

\begin{example} The $\Sigma_1$-formula $\exists Y\;(X\in Y)$ expresses the fact that $X$ is a set.
\end{example}

\begin{exercise}\label{ex:X=Y-Delta} Show that the formulas $X\subseteq Y$ and $X=Y$ are $\Delta_0$-formulas.
\smallskip

\noindent{\em Hint:} Observe that $X\subseteq Y$ is equivalent to the $\Delta_0$-formula  $(\forall z\in X\; (z\in Y))$.
\end{exercise}

\begin{exercise} Show that the formulas $z=\{x,y\}$ and $z=\{x\}$ are $\Delta_0$-formulas.
\smallskip

\noindent{\em Hint:} Observe that $z=\{x,y\}$ is equivalent to the $\Delta_0$-formula\newline
\centerline{$((x\in z)\wedge (y\in z))\wedge (\forall u\in z\; (u=x\vee u=y)).$}
\end{exercise} 

\begin{exercise}\label{ex:DeltaOpair} Show that the formula $z=\langle x,y\rangle$ is a $\Delta_0$-formula.
\smallskip

\noindent{\em Hint:} This formula is equivalent to the $\Delta_0$-formula\newline
\centerline{$\exists s\in z\;\exists p\in z\;((z=\{s,p\})\wedge (s=\{x\})\wedge (p=\{x,y\}))$.}
\end{exercise}

\begin{exercise} Show that the formula $Z=X\setminus Y$ is a $\Delta_0$-formulas.
\smallskip

\noindent{\em Hint:} This formula is equivalent to the $\Delta_0$-formula\newline
\centerline{$\big(\forall z\in Z\;((z\in X)\wedge (z\notin Y))\big)\wedge\big(\forall x\in X\;((x\notin Y)\to (x\in Z))\big)$.}
\end{exercise}

\begin{exercise} Show that the formulas $x\in \dom[Y]$ and $X=\dom[Y]$ are $\Delta_0$-formulas.
\smallskip

\noindent{\em Hint:} Observe that $x\in \dom[Y]$ is equivalent to the $\Delta_0$-formula\newline
\centerline{$\exists p\in Y\;\exists s\in p\;\exists d\in p\;\exists y\in d\; (p=\langle x,y\rangle)$.}
\end{exercise} 

\begin{exercise} Show that the formulas $z\in X\times Y$ and $Z=X\times Y$ are $\Delta_0$-formulas.
\smallskip

\noindent{\em Hint:} Observe that $z\in X\times Y$ is equivalent to the $\Delta_0$-formula $\exists x\in X\;\exists y\in Y\; (z=\langle x,y\rangle)$.
\end{exercise} 

\begin{exercise} Show that the formulas $z\in X^{-1}$ and $Y=X^{-1}$ are $\Delta_0$-formulas.
\smallskip

\noindent{\em Hint:} Observe that $p\in X^{-1}$ is equivalent to the $\Delta_0$-formula\\
\centerline{$\exists q\in X\;\exists u\in q\;\exists x\in u\;\exists y\in u\;(q=\langle x,y\rangle\;\wedge\; p=\langle y,x\rangle)$.}
\end{exercise}

\begin{exercise} Show that the formula $t=\langle x,y,z\rangle$ is a $\Delta_0$-formula.
\smallskip

\noindent{\em Hint:} This formula is equivalent to the $\Delta_0$-formula\newline
\centerline{$\exists u\in t\;\exists p\in u\;\exists v\in p\;\big((x\in v)\wedge (y\in v)\;\wedge (z\in u)\wedge (p=\langle x,y\rangle)\wedge (t=\langle p,z\rangle)\big)$.}
\end{exercise} 

\begin{exercise} Show that the formulas $t\in X^\circlearrowright$ and $Y=X^\circlearrowright$ are $\Delta_0$-formulas.
\smallskip

\noindent{\em Hint:} Observe that $t\in X^\circlearrowright$ is equivalent to the $\Delta_0$-formula\\
\centerline{$\exists r\in X\;\exists u\in r\;\exists p\in u\;\exists v\in p\;\exists x\in v\;\exists y\in v\;\exists z\in u\;\big(r=\langle x,y,z\rangle\wedge t=\langle z,x,y\rangle)$.}
\end{exercise} 

\begin{exercise}\label{ex:X=0} Show that the formula $X=\emptyset$ is a $\Delta_0$-formula.
\smallskip

\noindent{\em Hint:} Observe that $X=\emptyset$ is equivalent to the $\Delta_0$-formula $\forall x\in X\;(x\ne x)$.
\end{exercise}


\section{G\"odel's Theorem on class existence}

This section is devoted to a fundamental result of G\"odel\footnote{{\bf Task:} Read about G\"odel in Wikipedia.} on the existence of the class  $\{x:\varphi(x)\}$ for any $\UU$-bounded formula $\varphi(x)$  with one free variable $x$.  It is formulated and proved in the metalanguage by induction on the complexity of a formula $\varphi(x)$. So it provides a scheme for proofs of concrete instances of the formula $\varphi(x)$, but some of them admit more simple and direct proofs, see (and solve) exercises throughout the book.

\begin{definition} A formula $\varphi$ of Set Theory is called \index{formula!{{\bf U}}-bounded}{\em $\UU$-bounded} if each quantifier appearing in this formula is of the form $\exists x\in\UU$ or $\forall x\in\UU$, where $x$ is a symbol of variable. 
\end{definition}

Restricting the domain of quantifiers to the class $\UU$ allows us to avoid Berry's paradox at forming classes by constructors.

For any natural number $n\ge 3$ define an \index{ordered $n$-tuple} {\em ordered $n$-tuple} $\langle x_1,\dots,x_n\rangle$ of sets $x_1,\dots,x_n$ by the recursive formula: $\langle\langle x_1,\dots,x_{n-1}\rangle,x_n\rangle$. 


\begin{theorem}[G\"odel, 1940]\label{t:class} Let $\varphi(x_1,\dots,x_n,Y_1,\dots,Y_m)$ be a $\UU$-bounded formula of Set Theory whose free variables belong to the list $x_1,\dots,x_n,Y_1,\dots,Y_m$. Then for any classes $Y_1,\dots,Y_m$ the class $\Phi=\{\langle x_1,\dots,x_{n}\rangle \in \UU^n:\varphi(x_1,\dots,x_n,Y_1,\dots,Y_m)\}$ exists. Moreover, if the classes $Y_1,\dots,Y_m$ are basic, then so is the class $\Phi$.
\end{theorem}

In this theorem the $n$-th power $\UU^n$ of the universe $\UU$ is defined inductively: $\UU^1=\UU$ and $\UU^{n+1}=\UU^n\times\UU$ for a natural number $n$. 
Using the Axiom of Product, we can prove inductively that for every $n\in\IN$ the class $\UU^n$ exists.

Theorem~\ref{t:class} is proved by induction on the complexity of the formula $\varphi$. We lose no generality assuming that the formula $\varphi$ contains only  the existential quantifiers and the Sheffer stroke as the only logical connective. Such a restriction does not reduce the generality since every  logical connective can be expressed via the Sheffer stroke and for every formula $\psi$ and every symbol of variable $x$, the formula $\forall x(\psi)$ is equivalent to the formula $\neg(\exists x(\neg(\psi)))$, according to De Morgan's law.

If the formula $\varphi$ is atomic, then it is equal to one of the following atomic formulas:
$$x_i\in x_j,\;\;x_i\in Y_j,\;\;Y_i\in x_j,\;\;Y_i\in Y_j.$$
These cases are treated separately in the following lemmas. 

\begin{lemma}\label{l:xixj} For every natural number $n$ and positive numbers $i,j\le n$, the class
$$\{\langle x_1,\dots,x_{n}\rangle\in\UU^n:x_i\in x_j\}$$exists and is basic.
\end{lemma}

\begin{proof} Consider the basic functions $$\dom:\UU^2\to\UU,\;\dom:\langle x,y\rangle\mapsto x,\quad\mbox{and}\quad \rng:\UU^2\to\UU,\;\rng:\langle x,y\rangle\mapsto y,$$whose existence was established in Exercises~\ref{ex:dom} and \ref{ex:rng}.

Let $\dom^0=\Id$ and $\dom^{n+1}=\dom\circ\dom^n$ for every natural number $n$ (from the metalanguage). For every natural number $n$, the function $\dom^n$ assigns to any $(n+1)$-tuple $\langle x_1,\dots,x_{n+1}\rangle$ its first element $x_1$. 


We can prove inductively that for every natural number $n$ the function $\dom^n:\UU^{n+1}\to\UU$ exists and is basic.

Now observe that for any numbers $i\le n$ the function 
$$\Proj^n_i:\UU^n\to\UU,\;\;\Proj^n_i:\langle x_1,\dots,x_n\rangle\mapsto x_i,$$ exists and is basic, being  equal to the composition $\rng\circ\dom^{n-i}{\restriction}_{\UU^n}$.

By the Axiom of Membership, the class $\E=\{\langle x,y\rangle\in\ddot\UU:x\in y\}$  exists.

Observing that for every non-zero natural numbers $i,j\le n$ 
$$
\{\langle x_1,\dots,x_n\rangle:x_i\in x_j\}=\{z\in\UU^n:\langle \Proj^n_i(z),\Proj^n_j(z)\rangle\in \E\}$$
we can apply Exercise~\ref{ex:FRG} and conclude that this class exists and is basic.
\end{proof}

\begin{lemma}\label{l:FinY} For every non-zero natural numbers $i\le n$ and every (basic) class $Y$ the classes
$$\{\langle x_1,\dots,x_n\rangle:x_i\in Y\}\mbox{ \ and \ }\{\langle x_1,\dots,x_n\rangle:Y\in x_i\}$$exist (and are basic).
\end{lemma}

\begin{proof} Observing that
$$\{\langle x_1,\dots,x_n\rangle:x_i\in Y\}=\{z\in\UU^n:\exists y\in Y\;(\langle z,y\rangle\in\Proj^n_i)\}=\dom[(\UU^n\times Y)\cap \Proj^n_i],$$
we see that the class  $\{\langle x_1,\dots,x_n\rangle:x_i\in Y\}$ exists by the Axioms of Membership, Domain, Product, Difference. Moreover this class is basic if $Y$ is basic.

If $Y$ is a proper class, then the class $\{\langle x_1,\dots,x_n\rangle:Y\in x_i\}$ is empty and hence exists and is basic by the Axioms of Membership and Difference.

If $Y$ is a set, then we can consider the function $G=\UU^n\times\{Y\}$ and conclude that the class $\{\langle x_1,\dots,x_n\rangle:Y\in x_i\}=\{z\in\UU^n:\langle G(z),\Proj^n_i(z)\rangle\in\E\}$ exists and is basic by Exercises~\ref{ex:FRG} and \ref{ex:basic-singleton}.
\end{proof}

\begin{lemma}\label{l:Y=Z} For every classes $Y,Z$ the class
$$\{\langle x_1,\dots,x_n\rangle:Y\in Z\}$$
exists.
\end{lemma}

\begin{proof} This class is equal to $\UU^n$ or $\emptyset$ and hence exists and is basic by the Axioms of Membership, Domain, Product and Difference.
\end{proof}
\medskip

By Lemmas~\ref{l:xixj}--\ref{l:Y=Z}, for any atomic formula $\varphi$ with free variables in the list $x_1,\dots,x_n$, $Y_1,\dots,Y_m$ and any (basic) classes $Y_1,\dots,Y_m$, the class $\{\langle x_1,\dots,x_n\rangle:\varphi(x_1,\dots,x_n)\}$ exists (and is basic). Observe that each atomic formula has exactly 3 symbols.
\smallskip
 
 Assume that for some natural number $k\ge 4$, Theorem~\ref{t:class} has  been proved for all formulas $\varphi$ of containing $<k$ symbols. Let $\varphi$ be a formula consisting of exactly $k$ symbols. We also assume that the free variables of the formula $\varphi$ are contained in the list $x_1,\dots,x_{n}, Y_1,\dots,Y_m$. Since the formula $\varphi$ is not atomic, there exist formulas $\phi,\psi$ such that $\varphi$ is equal to $(\phi)|(\psi)$ or $\exists x(\phi)$ for some symbol of variable $x$.
 
First assume that $\varphi$ is equal to the formula $(\phi)|(\psi)$. In this case the formulas $\phi,\psi$ consist of $<k$ symbols and have $\Free(\phi)\cup\Free(\psi)=\Free(\varphi)\subseteq\{x_1,\dots,x_{n},Y_1,\dots,Y_m\}$.  Applying the inductive assumption, we conclude that for any (basic) classes $Y_1,\dots,Y_m$ the classes
$$\Phi=\{\langle x_1,\dots,x_{n}\rangle\in \UU^n:\phi(x_1,\dots,x_{n},Y_1,\dots,Y_m)\}$$and 
$$\Psi=\{\langle x_1,\dots,x_{n}\rangle\in \UU^n:\psi(x_1,\dots,x_{n},Y_1,\dots,Y_m)\}$$
exist (and are basic).
Then the class
$$\{\langle x_1,\dots,x_{n}\rangle\in \UU^n:\varphi(x_1,\dots,x_{n},Y_1,\dots,Y_m)\}=\UU^n\setminus(\Phi\cap \Psi)=\UU^n\setminus(\Phi\setminus(\Phi\setminus\Psi))$$exists by the Axioms of Membership, Domain and Difference (and is basic, being a result of applying the operations of domain, product and difference to basic classes).
\smallskip

Next, assume that $\varphi$ is equal to the formula $\exists x(\phi)$. If $x\in\{x_1,\dots,x_{n},Y_1,\dots,Y_m\}$, then we can replace all free occurences of the symbol $x$ in the formula $\phi$ by some other symbol and assume that $x\notin \{x_1,\dots,x_{n},Y_1,\dots,Y_m\}$. Then the formula $\phi$ has all its free variables in the list $x_1,\dots,x_{n},x,Y_1,\dots,Y_n$. By the inductive assumption, the class $$\Phi=\{\langle x_1,\dots,x_{n},x\rangle\in\UU^{n+1}:\phi(x_1,\dots,x_{n},x,Y_1,\dots,Y_m)\}$$exists (and is basic) and then the class $$\{\langle x_1,\dots,x_{n}\rangle\in\UU^n:\varphi(x_1,\dots,x_{n},Y_1,\dots,Y_m)\}$$ is equal to the class $\mathsf{dom}[\Phi],$ which exists by the Axiom of Domain (and is basic because of the definability of the class $\Phi$).
 This completes the inductive step and also completes the proof of the theorem.$\hfill\square$
\medskip

\begin{exercise} Show that every $\Delta_0$-formula is equivalent to a $\UU$-bounded formula.
\smallskip

\noindent{\em Hint:} Observe that $\forall X\in Y\;(\varphi)$ is equivalent to $\forall X\in\UU\;((X\in Y)\to(\varphi))$ for every formula $\varphi$. 
\end{exercise}

\begin{exercise} Show that every $\UU$-bounded formula $\varphi$ is a $\Sigma_2$-formula.
\smallskip

\noindent{\em Hint:} Observe that $\varphi$ is equivalent to the $\Sigma_2$-formula $\exists U\, ((\forall C\,(C\subseteq U))\wedge (\psi))$ where $\psi$ is the modification of the formula $\varphi$ in which all quantifiers $\forall X\in \UU$ and $\exists X\in\UU$ are replaced with the bounded quantifiers $\forall X\in U$ and $\exists X\in U$, respectively.
\end{exercise}

\begin{remark} Theorem~\ref{t:class} implies that a class $X$ is basic if and only if there exists an $\UU$-bounded formula $\varphi(x,E)$ with two free variables $x,E$ such that $X=\{x:\varphi(x,\E)\}$. This shows that the definition of a basic class involves quantifiers over formulas and hence is a property defined in the metalanguage (used for the analysis of models of the Classical Set Theory). The property of a class to be basic cannot be expressed in the language of the Classical Set Theory. Nonetheless every $\UU$-bounded formula $\varphi(x,E)$ with two free variables $x,E$ does determine a unique basic class $\{x:\varphi(x,\mathbf E)\}$, according to G\"odel's Theorem~\ref{t:class}. Moreover, assuming that CST is consistent, it is possible to construct countable models of CST in which all classes are basic, see \cite{HLR}.
\end{remark}

\begin{Exercise}
Find a formula $\varphi(X,E)$ with free variables $X,E$ such that every class $X$ satisfying the formula $\varphi(X,\E)$ is not basic.

\noindent{\em Hint:} Let $2^{<\omega}=\bigcup_{n\in\w}2^n$ and for every $k\in 2$, consider the function $\vec k:2^{<\omega}\to 2^{<\omega}$, $\vec k:t\mapsto\{\langle 0,k\rangle\}\cup\{\langle n\cup\{n\},y\rangle:\langle n,y\rangle\in t\}$. For every $n\in\w$ let $2^{<n}\defeq\bigcup_{k\in n}2^k$ and $T\defeq\bigcup_{n\in\w}6^{2^{<n}}$. For a class $X$ and a set $a$, let $X_a\defeq\{x\in\UU:\langle a,x\rangle\in X\}$. Consider the formula $\varphi(X,\E)$ describing the following property of a class $X$: $X_\emptyset=\E$ and for every $n\in\IN$, $\lambda\in 6^{2^{<n}}$ and the classes $X_\lambda$, $X_{\lambda\circ\vec 0}$, $X_{\lambda\circ\vec1}$ the following properties hold:
\begin{enumerate}
\item if $\lambda(\emptyset)=0$, then $X_\lambda=X_{\lambda\circ\vec 0}$;
\item if $\lambda(\emptyset)=1$, then $X_\lambda=X_{\lambda\circ\vec 0}\setminus X_{\lambda\circ\vec 1}$;
\item if $\lambda(\emptyset)=2$, then $X_\lambda=X_{\lambda\circ\vec 0}\times X_{\lambda\circ\vec 1}$;
\item if  $\lambda(\emptyset)=3$, then $X_\lambda=(X_{\lambda\circ\vec 0})^{-1}$;
\item if $\lambda(\emptyset)=4$, then $X_\lambda=(X_{\lambda\circ\vec 0})^\circlearrowright$;
\item if $\lambda(\emptyset)=5$, then $X_\lambda=\dom[X_{\lambda\circ\vec 0}]$.
\end{enumerate}
Assuming that for some class $X$ the formula $\varphi(X,\E)$ holds, consider the set $\Lambda=\{\lambda\in T:\lambda\notin X_\lambda\}$. Show that for every basic class $D$ there exists $\lambda\in \bigcup_{n\in\w}6^{2^{<n}}$ such that $D=X_\lambda$. Use this fact to prove that the classes $\Lambda$ and $X$ are not basic.
\end{Exercise}
\newpage

 \section{Axiomatic Set Theory of Zermelo--Fraenkel}
 
The Set Theory of Zermelo--Fraenkel (briefly ZF) is a part of the theory NBG, which speaks only about sets and identifies classes with formulas (which are used for defining those classes). The undefined notions of Zermelo--Fraenkel Set Theory are the notions of set and membership. The language of ZF theory is the same as the language of NBG theory.

Since the classes formally do not exist in ZF, more axioms are necessary to ensure the existence of sufficiently many of sets. So, the list of ZF axioms is infinite. It includes two axiom schemas: of  separation and  replacement.  The axiom schema of separation substitutes  seven axioms of existence of classes and sets in NBG and the axiom schema of replacement is a substitute for the single axiom of replacement in the NBG axiom system.
\medskip

\begin{center}\fbox{
\begin{minipage}{395pt}
\centerline{\bf \large Axioms of Zermelo--Fraenkel:}\index{Axioms of Zermelo-Fraenkel}\index{Axioms of ZFC}\index{ZFC}
\vskip2pt

\parskip1pt\parindent=0pt

\index{Axiom of Equality}{\bf Axiom of Equality:} $\forall x\;\forall y\;\big((x=y)\;\rightarrow\;\forall z\;(x\in z\;\leftrightarrow\;y\in z)\big)$

\index{Axiom of Pair} {\bf Axiom of Pair:} $\forall x\;\forall y\;\exists z\;\forall u\;(u\in z\;\leftrightarrow\;(u=x\;\vee\;u=y))$

\index{Axiom of Union}{\bf Axiom of Union:} $\forall x\;\exists y\;\forall z\;(z\in y\;\leftrightarrow\;\exists u\;(z\in u\;\wedge\;u\in x))$

\index{Axiom of Power-set}{\bf Axiom of Power-set:} $\forall x\;\exists y\;\forall z\;(z\in y\;\leftrightarrow\;\forall u\;(u\in z\;\rightarrow\;u\in x))$


\index{Axiom of Infinity}{\bf Axiom of Infinity:} $\exists x\;\exists o\;\big((\forall y\;(y\notin o))\;\wedge\; o\in x\;\wedge\;(\forall n\;(n\in x\;\to\;n\cup\{n\}\in x))\big)$

\index{Axiom of Foundation} {\bf Axiom of Foundation:} $\forall x\;(\exists y\;(y\in x)\;\Rightarrow\;\exists z\;(z\in x\;\wedge\;\forall u\;(u\in x\;\rightarrow\;u\notin z)))$
 \vskip1pt
 
\index{Axiom Schema of Separation}{\bf Axiom Schema of Separation:} Let $\varphi$ be a formula whose free variables are in the list $x,z,c$ and $y$ is not free for $\varphi$. Then $\forall x\;\forall c\;\exists y\;\forall z\;(z\in y\;\leftrightarrow\;(z\in x\;\wedge\;\varphi(x,z,c)))$

\index{Axiom Schema of Replacement}{\bf Axiom Schema of Replacement:} Let $\varphi$ be a formula whose free variables are in the list $x,u,v,c$ and $y$ is not free for $\varphi$. Then\\
{$\forall x\,\forall c\,\big((\forall u\in x\,\exists! v\,\varphi(x,u,v,c))\,\rightarrow\,\exists y\,\forall v\,(v\in y\,\leftrightarrow\,\exists u\,(u\in x\,\wedge\,\varphi(x,u,v,c)))\big)$}
\vskip1pt
\end{minipage}
}
\end{center}
\smallskip

The axioms ZF with added Axiom of Choice form the axioms ZFC. 

Replacing the quantifiers $\forall x$ and $\exists x$ in the axioms ZF by bounded quantifiers $\forall x\in\UU$ and $\exists x\in\UU$, we can see that obtained statements are theorems of NBG. This means that $\UU$ is a model of ZFC within NBG. So, consistency of NBG implies the consistency of ZFC. The converse is also true: the consistency of ZFC implies the consistency of NBG. So these two theories are equiconsistent. Moreover NBG is a conservative extension of ZFC, which means that a $\UU$-bounded formula without free variable is a theorem of ZFC if and only if it is a theorem of NBG. This important fact was proved by Shoenfield, see  \cite[p.70]{Jech}. Therefore, if we are interested only in sets, there is no difference (except aesthetic) which theory to use. On the other hand, NBG has essential advantages: it has finite list of axioms and allows one to work freely with classes. 
This is a reason why we have chosen NBG for presentation of Set Theory. Since NBG deals with sets and classes and it was created by the classics of Set Theory and Logic (von Neumann, Robinson, Bernays, G\"odel) we refer to this theory as  the Classical Set Theory (shortly, CST).\index{CST} 
From now on we accept the following list of 12 axioms, called the {\em Axioms of Classical Set Theory}.
\medskip

\noindent\fbox{
\begin{minipage}{425pt}
\parskip1pt
\centerline{\bf \large Axioms of Classical Set Theory}\index{Axioms of Classical Set Theory}
\vskip2pt

\noindent{\bf Equality:}\index{Axiom of Equality} Equal classes behave equally with respect to taking elements.
\vskip1pt

\noindent{\bf Pair:}\index{Axiom of Pair} For any sets $x,y$ the set $\{x,y\}$ exists.
\vskip1pt


\noindent{\bf Membership:}\index{Axiom of Membership} The class $\mathbf E=\{\langle x,y\rangle:x\in y\}$ exists.

\noindent{\bf Domain:}\index{Axiom of Domain} For every class $X$ the class $\mathsf{dom}[X]=\{x:\exists y\;(\langle x,y\rangle\in X\}$ exists.

\noindent {\bf Difference:}\index{Axiom of Difference} For any classes $X,Y$ the class $X\setminus Y=\{x:x\in X\;\wedge\;x\notin Y\}$ exists.

\noindent{\bf Product:}\index{Axiom of Product} For every classes $X,Y$ the class $X\times Y=\{\langle x,y\rangle:x\in X,\;y\in Y\}$ exists.

\noindent{\bf Inversion:}\index{Axiom of Inversion} For every class $X$ the class $X^{-1}=\{\langle y,x\rangle:\langle x,y\rangle\in X\}$ exists.

\noindent{\bf Cycle:}\index{Axiom of Cycle} For every class $X$ the class $X^\circlearrowright=\{\langle z,x,y\rangle:\langle x,y,z\rangle\in X\}$ exists.
\vskip1pt

\noindent{\bf Replacement:}\index{Axiom of Replacement} For every function $F$ and set $x$ the class $F[x]=\{F(y):y\in x\}$ is a set.

\noindent{\bf Union:}\index{Axiom of Union} For every set $x$ the class $\bigcup x=\{z:\exists y\in x\;(z\in y)\}$ is a set.

\noindent{\bf Power-set:}\index{Axiom of Power-set} For every set $x$ the class $\mathcal P(x)=\{y:y\subseteq x\}$ is a set.
\vskip1pt

\noindent{\bf Infinity:}\index{Axiom of Infinity} There exists an inductive set.
\smallskip
\end{minipage}
} 
\vskip5pt

\noindent Whenever necessary, we will add to this list the {\bf Axiom of Foundation} or\\  the {\bf Axiom of (Global) Choice}, which will be specially acknowledged. 
\newpage

\part{Fundamental Constructions}

In this section we survey some fundamental constructions of the Classical Set Theory, which often appear in other areas of Mathematics:  relations, functions, indexed families of classes,  Cartesian products, equivalence relations. Often we shall  formulate the corresponding existence theorems as exercises with solutions or hints. We recall that $\ddot\UU=\UU\times\UU$ and ${\dddot{\UU}}{}={}{\ddot\UU\times\UU}$.

\section{Cartesian products of sets}

By the Axiom of Product, for two classes  $X,Y$ the class $$X\times Y=\{\langle x,y\rangle:x\in X\;\wedge\;y\in Y\}$$exists. If the classes $X,Y$ are basic, then so is the class $X\times Y$.

\begin{exercise} Write down the products $2\times 3$ and $3\times 2$. Are they equal? Are they equal to the number $6$? Find all natural numbers $n,m$ whose product $n\times m$ is a natural number.
\end{exercise}

\begin{exercise}\label{t:prodsets} Prove that for any sets $X,Y$, the class $X\times Y$ is a set.
\smallskip

\noindent{\em Hint}: Observe that $X\times Y\subseteq\mathcal P(\mathcal P(X\cup Y))$.
\end{exercise}

\begin{exercise}\label{ex:prodsets} Prove that for any sets $X,Y$, the class $X\times Y$ is a set, not using the Axiom of Power-Set.
\smallskip

\noindent{\em Hint}: Apply the Axioms of Replacement and Union.
\end{exercise}


\section{Relations}  We recall that a \index{relation}{\em relation} is a subclass of the class $\UU\times\UU=\ddot\UU$. For a class $R$ let $$R^{-1}=\{\langle y,x\rangle:\langle x,y\rangle\in R\}$$ be the {\em inverse relation} to $R$. 
Observe that a class $R$ is a relation if and only if $R=(R^{-1})^{-1}$. 

For a relation $R$ let $R^\pm=R\cup R^{-1}$. 
The class $\dom[R^\pm]=\rng[R^\pm]$ is called the \index{underlying class}\index{relation!underlying class of}{\em underlying class} of the relation $R$.

\begin{exercise}\label{p:Rel} Prove that the class of relations \index{{{\bf Rel}}}\index{class!{{\bf Rel}}}$\mathbf{Rel}=\{r\in\UU:\mbox{$r$ is a relation}\}$ exists.
\smallskip

\noindent{\it Hint:} Observe that $\UU\setminus\mathbf{Rel}=\{x\in\UU:\exists y\in x\;\,(y\notin\ddot{\UU})\}=\dom[P]$ where\\ $
P=\{\langle x,y\rangle\in \ddot{\UU}:y\in x\;\wedge\;y\notin \ddot{\UU}\}=
\mathbf E^{-1}\cap(\UU\times(\UU\setminus\ddot{\UU})).$
\end{exercise}

\begin{exercise}\label{ex:domrelsubset} Prove that for any class $X$ the classes $\{r\in\mathbf{Rel}:\dom[r]\subseteq X\}$ and\\  $\{r\in\mathbf{Rel}:\rng[r]\subseteq X\}$ exist.
\smallskip

\noindent{\it Hint:} Observe that $\mathbf{Rel}\setminus\{r\in\mathbf{Rel}:\dom[r]\not\subseteq X\}=\{r\in\mathbf{Rel}:\exists \langle x,y\rangle\in r\cap ((\UU\setminus X)\times\UU)\}= \rng[((\UU\setminus X)\times\UU)\times\mathbf{Rel})\cap\mathbf E]$.
\end{exercise}

\begin{exercise} Prove that for any class $X$ the classes $\{r\in\mathbf{Rel}:X\subseteq \dom[r]\}$ and\\  $\{r\in\mathbf{Rel}:X\subseteq \rng[r]\}$ exist.
\smallskip

\noindent{\em Hint:} These classes are empty if $X$ is a proper class.
\end{exercise}

\begin{exercise}\label{ex:domrel=X} Prove that for any class $X$ the classes $\{r\in\mathbf{Rel}:\dom[r]=X\}$ and\\  $\{r\in\mathbf{Rel}:\rng[r]=X\}$ exist.
\end{exercise}

For a relation $R$ and a class $X$ denote by $R{\restriction}X$ the relation $R\cap(X\times X)$. The relation $R{\restriction}X$ is called the \index{relation!restriction}{\em restriction} of the relation $R$ to the class $X$. If $R$ is a function with $R[X]\subseteq X$, then $R{\restriction}X=R{\restriction}_X$, where $R{\restriction}_X=R\cap(X\times\UU)$.

\begin{exercise} Show that for a basic relation $R$ and a basic class $X$, the relation $R{\restriction}X$ is basic.
\end{exercise}

\begin{exercise} Show that the property of a class $R$ to be a relation can be  described by a $\Delta_0$-formula.
\smallskip

\noindent{\em Hint:} Observe that $R$ is a relation if and only if $\forall r\in R\;\exists u\in r\;\exists x\in u\;\exists y\in u\;(r=\langle x,y\rangle)$, and apply Exercise~\ref{ex:DeltaOpair}.
\end{exercise}

It is convenient to think of a relation $R$ as the directed graph with vertices in the class $\dom[R^\pm]$ and edges in the class $R$. 

\begin{example} The relation 
$$\mathbf E{\restriction}4=\{\langle 0,1\rangle, \langle 0,2\rangle,\langle 0,3\rangle,\langle 1,2\rangle,\langle 1,3\rangle, \langle 2,3\rangle\}$$ can be drawn as the directed graph
$\xymatrix{0\ar[r]\ar@/_10pt/[rr]\ar@/_15pt/[rrr]&1\ar[r]\ar@/^10pt/[rr]&2\ar[r]&3
}$.
\end{example}
\medskip

\section{Functions} We recall that a relation $F$ is a \index{function}{\em function} if $$\forall x\;\forall y\;\forall z\;((\langle x,y\rangle\in F\;\wedge\;\langle x,z\rangle\in F)\;\rightarrow\;(y=z)).$$
Therefore, for any function $F$ and any $x\in\dom[F]$ there exists a unique set $y$ such that $\langle x,y\rangle\in F$. This unique set $y$ is denoted by $F(x)$ and called \index{function!the value of} {\em the value} of the function $F$ at $x$. The round parentheses are used to distinguish the element $F(x)$ from the set $$F[x]=\{y:\exists z\in x\;(\langle z,y\rangle\in F)\}=\{F(y):y\in x\cap \dom[F]\}.$$

Given a function $F$ and two classes $X,Y$, we write $F:X\to Y$ and say that $F$ is a function from $X$ to $Y$ if $\dom[F]=X$ and $\rng[F]\subseteq Y$. Often we shall use the notation $$F:X\to Y,\;\;F:x\mapsto F(x),$$ indicating that $F$ assigns to each element $x\in X$ some element $F(x)$ of $Y$.

For any function $F:X\to Y$ and class $A$, the function $F\cap(A\times \UU)$ is denoted by $F{\restriction}_A$ and is called the \index{function!restriction}{\em restriction} of $F$ to the class $A$. If $F[A]\subseteq A$, then $F{\restriction}_A=F{\restriction}A$ where $F{\restriction}A=F\cap(A\times A)$ is the restriction of the relation $F$.

\begin{exercise} Write down the function $F:4\to\w$, $F:x\mapsto x^2$, and find $\dom[F]$ and $\rng[F]$. Write down the restrictions $F{\restriction}3$ and $F{\restriction}_3$. List the elements of the sets $F(3)$, $F[3]$, $F[6]$.
\smallskip

{\em Hint:} $F=\{\langle 0,0\rangle,\langle 1,1\rangle,\langle 2,4\rangle,\langle 3,9\rangle\}$; $\dom[F]=4=\{0,1,2,3\}$, $\rng[F]=\{0,1,4,9\}$; $F{\restriction}3=\{\langle 0,0\rangle,\langle 1,1\rangle\}$, $F{\restriction}_3=\{\langle 0,0\rangle,\langle 1,1\rangle,\langle 2,4\rangle\}$; $F(3)=9=\{0,1,2,3,4,5,6,7,8\}$; $F[3]=\{0,1,4\}$.
\end{exercise}

\begin{exercise} Prove that a function $F$ is a set if and only if its domain $\dom[F]$ is a set if and only if $\dom[F]$ and $\rng[F]$ are sets.
\end{exercise}

\begin{exercise} Prove that the property of a class $F$ to be a function can be described by a $\Delta_0$-formula.
\smallskip

\noindent{\em Hint:} Observe that a class $F$ is a function if and only if $F$ is a relation and\newline $\forall p\in F\;\forall q\in F\;\forall u\in p\;\forall v\in q\;\forall x\in u\;\forall y\in u\;\forall z\in v\;(p=\langle x,y\rangle \wedge q=\langle x,z\rangle\to y=z)$. 
\end{exercise}

A function $F:X\to Y$ is called
\begin{itemize}
\item \index{surjective function}\index{function!surjective}{\em surjective} if $\rng[F]=Y$;
\item \index{injective function}\index{function!injective} {\em injective} if $F^{-1}$ is a function;
\item \index{bijective function}\index{function!bijective} {\em bijective} if $F$ is surjective and injective.
\end{itemize}

\begin{exercise} Prove that the class \index{{{\bf Fun}}}\index{class!{{\bf Fun}}}$\mathbf{Fun}=\{f\in\mathbf{Rel}:\mbox{$f$ is a function}\}$ of all functions exists.
\end{exercise}

For two classes $A,X$ denote by $X^A$ the class of all functions $f$ such that $\dom[f]=A$ and $\rng[f]\subseteq X$. Therefore, $$X^A=\{f\in\mathbf{Fun}:\dom[f]=A,\;\;\rng[f]\subseteq X\}.$$

\begin{exercise} Prove that for every class $A$ the class $\UU^A$ exists.
\smallskip

\noindent{\it Hint:} Observe that $\UU^A=\mathbf{Fun}\cap\{f\in\mathbf{Rel}:\dom[f]=A\}$ and apply Exercise~\ref{ex:domrel=X}.
\end{exercise}

\begin{exercise}\label{ex:functionclass} Prove that for every (basic) classes $A, X$ the class $X^A$ exists (and is basic).
\smallskip

\noindent{\it Hint:} Observe that $X^A=\UU^A\cap\{f\in\mathbf{Rel}:\rng[f]\subseteq X\}$ and apply Exercise~\ref{ex:domrelsubset}.
\end{exercise}

\begin{exercise}Prove that for classes $A,X$ the class $X^A$ is empty if and only if one of the following holds:
\begin{enumerate}
\item $A$ is a proper class;
\item $X=\emptyset$ and $A\ne\emptyset$.
\end{enumerate}
\end{exercise}  

\begin{exercise}\label{ex:XA} Prove that for any sets $X,A$ the class $X^A$ is a set.
\smallskip

\noindent{\em Hint:} Observe that $X^A\subseteq \mathcal P(A\times X)$ and apply Exercises~\ref{t:prodsets}, \ref{ex:subclass} and the Axiom of Power-Set.
\end{exercise}

We recall that $0=\emptyset$, $1=\{0\}$, $2=\{0,1\}$ and $3=2\cup\{2\}$.

\begin{exercise}\label{ex:F1} Prove that the function $F_1:\UU\to\UU^1,\;F_1:x\mapsto \{\langle 0,x\rangle\}$ exists and is bijective.
\end{exercise}

\begin{exercise} Prove that the function $F_2:\ddot\UU\to\UU^2,\;F_2:\langle x,y\rangle\mapsto \{\langle 0,x\rangle,\langle 1,y\rangle\}$, exists and is bijective.
\end{exercise}

\begin{exercise}\label{ex:F3} Prove that the function $F_3:\dddot\UU\to\UU^3,\;F_3:\langle x,y,z\rangle\mapsto \{\langle 0,x\rangle,\langle 1,y\rangle,\langle 2,z\rangle\}$ exists and is bijective.
\end{exercise}

\begin{exercise} Prove that the functions $F_1,F_2,F_3$ from Exercises~\ref{ex:F1}--\ref{ex:F3} are basic.
\end{exercise}

\section{Indexed families of classes}

In spite of the fact that in the Classical Set Theory  proper classes cannot be elements of other classes, we can legally speak about indexed families of classes. Namely, for any class $A$, any subclass $X\subseteq A\times \UU$ can be identified with the indexed family $(X_\alpha)_{\alpha\in A}$ of the classes $X_\alpha\defeq X[\{\alpha\}]$, where $X[\{\alpha\}]=\rng[X\cap(\{\alpha\}\times\UU)]$ for any  index $\alpha\in A$.

In this case we can define the union $\bigcup_{\alpha\in A}X_\alpha$ as the class $\mathsf{rng}[X]$ and the intersection $\bigcap_{\alpha\in A}X_\alpha$ as the class $\UU\setminus \mathsf{rng}[(A\times\mathbf U)\setminus X]$. 

The indexed family $(X_\alpha)_{\alpha\in A}$ can be also thought of as a \index{multifunction}{\em multifunction} $X:A\multimap \UU$ assigning to each element $\alpha\in A$ the class $X[\{\alpha\}]$, and to each subclass $B\subseteq A$ the class $X[B]=\rng[X\cap(B\times\UU)]$.

If for every $\alpha\in A$ the class $X_\alpha$ is a set, then $(X_\alpha)_{\alpha\in A}$ is an indexed family of sets and we can consider the function $X^\bullet_A:A\to \UU$, assigning to each $\alpha\in A$ the set $X_\alpha$. 

The following theorem shows that the function $X_A^\bullet$ exists.

\begin{theorem}\label{t:index} Let $A,X$ be two classes such that for every $\alpha\in A$ the class $X_\alpha=X[\{\alpha\}]$ is a set. Then the function $X^\bullet_A:A\to\UU$, $X^\bullet_A:\alpha\mapsto X_\alpha$, exists.
\end{theorem}

\begin{proof} Observe that
$$
\begin{aligned}
X^\bullet_A=\;&\{\langle \alpha,y\rangle\in A\times \UU:y=X_\alpha\}= 
\{\langle \alpha,y\rangle\in A\times \UU:\forall z\;(z\in y\:\leftrightarrow\;z\in X_\alpha)\}=\\
&\{\langle \alpha,y\rangle\in A\times \UU:\forall z\;(z\in y\:\leftrightarrow\;\langle \alpha,z\rangle\in X)\}
\end{aligned}
$$
and
$$(A\times\UU)\setminus X^\bullet_A=\{\langle \alpha,y\rangle\in A\times\UU:\exists z\;\neg(z\in y\;\leftrightarrow\;\langle\alpha,z\rangle\in X)\}=\mathsf{dom}[T\cup T'],$$
where 
$$
T=\{\langle \alpha,y,z\rangle\in (A\times \UU)\times\UU:z\notin y\;\wedge\;\langle\alpha,z\rangle\in X\}=
 [(\ddot\UU\setminus \mathbf E^{-1})\times A]^\circlearrowright\cap[X^{-1}\times\UU]^\circlearrowleft
$$
and
$$
T'=\{\langle \alpha,y,z\rangle\in (A\times\UU)\times\UU:z\in y\;\wedge\;\langle\alpha,z\rangle\notin X\}= [\mathbf E^{-1}\times A]^\circlearrowright\cap[(\ddot\UU\setminus X^{-1})\times\UU]^\circlearrowleft.
$$
Now we see that the axioms of the Classical Set Theory guarantee the existence of the considered classes including the function  $X_A^\bullet$.
\end{proof}

If $X=(X_\alpha)_{\alpha\in A}$ is an indexed family of sets, then the class $\{X_\alpha:\alpha\in A\}$ is equal to $X^\bullet_A[A]$ and hence  exists. This justifies the use of the constructor $\{X_\alpha:\alpha\in A\}$ in our theory.

 If $A$ is a set, then the class $\{X_\alpha:\alpha\in A\}=X^\bullet_A[A]$ is a set by the Axiom of Replacement. In particular, each set $x$ is equal to the set $\{y:y\in x\}$. 

\begin{exercise}\label{ex:Russel2} Let $A$ be a class and $X$ be a subclass of $A\times \UU$ thought as an indexed family $(X_\alpha)_{\alpha\in A}$ of the classes $X_\alpha=X[\{\alpha\}]$. Observe that the class $$B=\{\alpha\in A:\alpha\notin X_\alpha\}=\{\alpha\in A:\langle\alpha,\alpha\rangle\notin X\}=\dom[(\mathsf{Id}\cap (A\times A))\setminus X]$$ exists. Repeating the argument of Russell's Paradox, prove that $B\ne X_\alpha$ for every $\alpha\in A$. 
\end{exercise}

By a \index{sequence}{\em sequence of classes} we understand a subclass $X\subseteq \w\times\UU$ identified with the indexed family of classes $(X_n)_{n\in\w}$ where $X_n=\{x\in\UU:\langle n,x\rangle\in X\}$. If each class $X_n$ is a set, then the indexed family of sets  $(X_n)_{n\in\w}$ can be identified with the function $X_*:\w\to\UU$ assigning to each $n\in\w$ the set $X_n$. By Theorem~\ref{t:index} such function exists.

For a natural number $n\in\IN$ by an \index{tuple of classes}{\em $n$-tuple} of classes $(X_0,\dots,X_{n-1})$ we understand the indexed family of classes $(X_i)_{i\in n}$. A \index{pair!of classes}{\em pair} of classes $(X,Y)$ is identified with the $2$-tuple $(X_i)_{i\in 2}$ such that $X_0=X$ and $X_1=Y$. By analogy we can introduce a triple of classes, a quaduple of classes, and so on.

Therefore, for any sets $x,y$ we have three different notions related to pairs: 
\begin{itemize}
\item[(i)] the \index{unordered pair}\index{pair!unordered}{\em unordered pair of sets} $\{x,y\}$
\item[(ii)] the  \index{ordered pair}\index{pair!ordered}{\em ordered pair of sets} $\langle x,y\rangle=\{\{x\},\{x,y\}\}$, 
\item[(iii)] the \index{pair!of classes}{\em pair of classes} $(x,y)=(\{0\}\times x)\cup(\{1\}\times y)=\{\langle 0,u\rangle:u\in x\}\cup\{\langle 1,v\rangle:v\in y\}$.
\end{itemize}
The definition of a pair of classes uses ordered pairs of sets and the definition of an ordered pair of sets is based on the notion of an unordered pair of sets (which exists by the Axiom of Pair).

The following exercise shows that the notion of a pairs of classes has the characteristic property of an ordered pair.

\begin{exercise} Prove that for any classes $A,B,X,Y$ we have the equivalence
$$(A,B)=(X,Y)\;\leftrightarrow\;(A=X\;\wedge\;B=Y).$$
\end{exercise}  

\section{Cartesian products of classes}

In this section we define the Cartesian product\index{Cartesian product}\index{product!Cartesian} $\prod_{\alpha\in A}X_\alpha$ of an indexed family of classes $X=(X_\alpha)_{\alpha\in A}$. By definition, the class $\prod_{\alpha\in A}X_\alpha$ consists of all functions $f$ such that $\dom[f]=A$ and $f(\alpha)\in X_\alpha$ for every $\alpha\in A$. Equivalently, the Cartesian product can be defined as the class
$$\prod_{\alpha\in A}X_\alpha=\{f\in\mathbf{Fun}:(f\subseteq X)\;\wedge\;(\dom[f]=A)\}.$$

\begin{proposition} For any class $A$ and a subclass $X\subseteq A\times \UU$, the Cartesian product $\prod_{\alpha\in A}X_\alpha$ of the indexed family $X=(X_\alpha)_{\alpha\in A}$ exists. If $A$ is a proper class, then $\prod_{\alpha\in A}X_\alpha$ is the empty class.
\end{proposition}

\begin{proof} The class $\rng[X]=\bigcup_{\alpha\in A}X_\alpha$ exists by the Axioms of Domain and Inversion. By Exercise~\ref{ex:functionclass}, the class $(\rng[X])^A$ of functions from $A$ to $\rng[X]$ exists. Observe that 
$$\prod_{\alpha\in A}X_\alpha=\{f\in(\rng[X])^A:f\subseteq X\}$$and hence
$$(\rng[X])^A\setminus\prod_{\alpha\in A}X_\alpha=\{f\in(\rng[X])^A:\exists z\;( z\in f\;\wedge\;z\notin X)\}=\dom[((\rng[X])^A\times (\UU\setminus X))\cap\mathbf E^{-1}].$$
Now the axioms of the Classical Set Theory ensure that the class $\prod_{\alpha\in A}X_\alpha$ exists. If this class is not empty, then it contains some function $f$ with $\dom[f]=A$. Applying the Axiom of Replacement to the function $\dom$, we conclude that the class $A=\dom[f]$ is a set.
\end{proof} 

\begin{exercise} Show that for a set $A$ and an indexed family of sets $X=(X_\alpha)_{\alpha\in A}$ the Cartesian product $\prod_{\alpha\in A}X_\alpha$ is a set.
\smallskip

\noindent{\em Hint:} Since $\{X_\alpha:\alpha\in A\}$ is a set, its union $\bigcup_{\alpha\in A}X_\alpha=\rng[X]$ is a set by the Axiom of Union and then $\prod_{\alpha\in A}X_\alpha$ is a set, being a subclass of the set $(\rng[X])^A$, see Exercise~\ref{ex:XA}.
\end{exercise}

\begin{exercise}\label{ex:prod} Let $A$ be a class and $X\subseteq A\times\UU$ be a class such that $\prod_{\alpha\in A}X_\alpha$ is not empty. Prove  that the class $A$ is a set and for every $\alpha\in A$ the class $X_\alpha=\{x:\langle\alpha,x\rangle\in X\}$ is not empty.
\end{exercise}

Exercise~\ref{ex:prod} motivates the following definition. For a class $X$ by $\prod X$ we denote the Cartesian product $\prod_{\alpha\in\dom[X]}X[\{\alpha\}]$ of the indexed family of nonempty classes $(X[\{\alpha\}])_{\alpha\in \dom[X]}$.

\begin{exercise} Prove that $\prod X=\{f\in\mathbf{Fun}:(\dom[f]=\dom[X])\;\wedge\;(f\subseteq X)\}$.
\end{exercise} 

\begin{exercise} Show that $X^A=\prod(A\times X)$ for any classes $A,X$.
\end{exercise} 

\begin{exercise} Show that $\prod\emptyset=\{\emptyset\}$ and hence $\prod\emptyset$ is not empty.
\end{exercise}

\begin{exercise} Observe that the Axiom of Choice holds if and only if for any set $X$ its Cartesian product $\prod X$ is not empty.
\end{exercise} 

\section{Reflexive and irreflexive relations}

We recall that a relation is a class whose elements are ordered pairs of sets. For a relation $R$ the class $\dom[R^\pm]$ is called the \index{relation!underlying class of}{\em underlying class} of the relation. Here $R^\pm=R\cup R^{-1}$.

\begin{definition} A relation $R$ is called
\begin{itemize}
\item \index{reflexive relation}\index{relation!reflexive}{\em reflexive} if $\Id{\restriction}\dom[R^\pm]\subseteq R$;
\item \index{irreflexive relation}\index{relation!irreflexive}{\em irreflexive} if $R\cap\Id=\emptyset$.
\end{itemize}
\end{definition}

\begin{example}
\begin{enumerate}
\item The relation $\Id$ is reflexive. 
\item The relation $\ddot\UU\setminus\Id$ is irreflexive.
\end{enumerate}
\end{example}

\begin{example} For any relation $R$ the relation $R\setminus\Id$ is irreflexive and $R\cup\Id{\restriction}\dom[R^\pm]$ is reflexive.
\end{example}

\begin{exercise} Using the  Axiom of Foundation, prove that the Membership relation $\E$ is irreflexive.
\end{exercise}

\begin{exercise} Prove that the class $\{r\in\mathbf{Rel}:\mbox{$r$ is a reflexive relation}\}$ exists.
\end{exercise}

\section{Equivalence relations} 

\begin{definition} A relation $R$ is called 
\begin{itemize}
\item \index{symmetric relation}\index{relation!symmetric}{\em symmetric} if $R=R^{-1}$;
\item \index{transitive relation}\index{relation!transitive}{\em transitive} if $\{\langle x,z\rangle \in\UU:\exists y\in\UU\;(\langle x,y\rangle\in R\;\wedge\;\langle y,z\rangle\in R)\}\subseteq R$;
\item \index{equivalence relation}\index{relation!of equivalence} {\em an equivalence relation} if $R$ is symmetric and transitive.
\end{itemize} 
\end{definition} 

Usually equivalence relations are denoted by symbols $=$, $\equiv$, $\cong$, $\sim$, $\approx$, etc.

\begin{example} The identity function $\Id=\{\langle x,x\rangle:x\in\UU\}$ is an equivalence relation.
\end{example}

\begin{exercise} Prove the existence of the classes of sets which are  symmetric relations, transitive relations, equivalence relations.
\end{exercise}

Let $R$ be an equivalence relation. The symmetry of $R$ guarantees that $\dom[R]=\rng[R]$. 

\begin{exercise}\label{ex:reflex} Prove that any equivalence relation $R$ is reflexive.
\smallskip

\noindent{\em Hint:} Given any $x\in\dom[R]$, find $y\in\UU$ with $\langle x,y\rangle\in R$. By the symmetry of $R$, $\langle y,x\rangle\in R$ and by the transitivity, $\langle x,x\rangle\in R$.
\end{exercise}

Let $R$ be an equivalence relation. For any set $x$, the class $R[\{x\}]=\{y:\langle x,y\rangle\in R\}$ is called the \index{equivalence class}{\em $R$-equivalence class} of $x$.
If $R[\{x\}]$ is not empty, then $x\in R[\{x\}]$ by Exercise~\ref{ex:reflex}.

\begin{exercise} Prove that for any equivalence relation $R$ and sets $x,y$, the $R$-equivalence classes $R[\{x\}]$ and $R[\{y\}]$ are either disjoint or coincide.
\end{exercise}


Let $R$ be an equivalence relation. If for any set $x$ its $R$-equivalence class $R[\{x\}]$ is a set, then by Theorem~\ref{t:index}, the class $R^\bullet=\{\langle x,R[\{x\}]\rangle :x\in\dom[R]\}$ is a well-defined function assigning to each set $x\in\dom[R]$ its equivalence class $R^\bullet(x)=R[\{x\}]$. The range $\{R^\bullet(x):x\in\dom[R]\}$ of this function is called the \index{quotient class}\index{class!quotient}{\em quotient class} of the relation $R$. The quotient class is usually denoted by $\dom[R]/R$. The function $R^\bullet:\dom[R]\to \dom[R]/R$ is called the \index{quotient function}\index{function!quotient}{\em quotient function}.

If the relation $R$ is a set, then by the Axiom of Replacement, the quotient class $\dom[R]/R=R^\bullet[\dom[R]]$ is a set, called the \index{quotient set}\index{set!quotient}{\em quotient set} of the relation $R$.

By an {\em equivalence relation on a set} $X$ we understand any equivalence relation $R$ with $\dom[R^\pm]=X$. In this case $R\subseteq X\times X$ is a set and so are all  $R$-equivalence classes $R^\bullet(x)$. Consequently, the quotient class $X/R$ is a set, called the \index{quotient set}\index{set!quotient}{\em quotient set} of $X$ by the relation $R$.

\begin{example} Consider the equivalence relation 
$$|\cdot\cdot|=\{\langle x,y\rangle\in \UU\times\UU:\exists f\in\mathbf{Fun} \;(f^{-1}\in \mathbf{Fun}\;\wedge\;\dom[f]=x\;\wedge\;\rng[f]=y)\}.$$
The equivalence class of a set $x$ by this equivalence relation is called the \index{cardinality}{\em cardinality} of the set $x$ and is denoted by $|x|$.
\end{example}

\section{Set-like relations}

A relation $R$ is called {\em set-like} if for any $x\in \mathbf U$ the class $\cev R(x)=R^{-1}[\{x\}]\setminus\{x\}$ is a set. The set $\cev R(x)$ appearing in this definition is called the \index{initial interval}\index{relation!initial interval of}{\em initial $R$-interval} of $x$. It is equal to $\cev R(x)=\{z:\langle z,x\rangle\in R\}\setminus\{x\}$. The set $\cev R(x)$ is empty if $x\notin \rng[R]$.

\begin{remark} For the membership relation $\mathbf E$ the initial $\E$-interval $\cev \E(x)$ of a set $x$ coincides with the set $x\setminus\{x\}$. If the Axiom of Foundation holds, then $\cev\E(x)=x$. The relation $\E$ is set-like.
\end{remark}

\begin{theorem}\label{t:set-like} If $R$ is a set-like relation, then for any set $a$ there exists a sequence of sets $(a_n)_{n\in\w}$ such that $a_0=a$ and $a_{n+1}=R^{-1}[a_n]$ for all $n\in\w$. The union $b=\bigcup_{n\in \w}a_n$ has the properties: $a\subseteq b$ and $R^{-1}[b]\subseteq b$.
\end{theorem}

\begin{proof} Consider the class $T$ of all functions $f$ such that
\begin{itemize}
\item[(i)] $\dom[f]\in\IN$;
\item[(ii)] $f(0)\in a$;
\item[(iii)] $\langle f(i),f(i+1)\rangle \in R^{-1}$ for any $i\in\dom[f]$ with $i+1\in\dom[f]$.
\end{itemize}
The existence of the class $T$ can be proved using G\"odel's Theorem~\ref{t:class}.  Consider the class 
$$F\defeq\{\langle n,y\rangle:n\in\w\;\wedge\;\exists f\in T\;(\langle n,y\rangle\in f)\}.$$
For every $n\in\w$ consider the class $a_n\defeq\{y:\langle n,y\rangle\in F\}=\{y:\exists f\in T\;(\langle n,y\rangle\in f)\}$.

The condition (ii) implies that $a_0=\{y:\langle 0,y\rangle\in F\}=a$.
By induction on $n$ we shall prove that the class $a_{n+1}$ is a set with $a_{n+1}=R^{-1}[a_n]$. Assume that for some $n\in\w$ we have proved that the class $a_n$ is a set. The condition (iii) implies that $a_{n+1}\subseteq R^{-1}[a_n]$. On the other hand, for any $x\in R^{-1}[a_n]$ we can find $y\in a_n$ with $\langle x,y\rangle\in R$. Since $y\in a_n$, there exists a function $f\in T$ such that $\langle n,y\rangle\in f$. Replacing $f$ by $f{\restriction}_{n+1}$, we can assume that $\dom[f]=n+1$. Then the function $g=f\cup\{\langle n+1,x\rangle\}$ belongs to $T$ and witnesses that $x\in a_{n+1}$. This shows that $a_{n+1}=R^{-1}[a_n]$.

Since the relation $R$ is set-like, the class $R^{-1\bullet}\defeq\{\langle x,y\rangle:y=R^{-1}[\{x\}]\}$ is a well-defined function, according to Theorem~\ref{t:index}. Now the Axioms of Replacement and Union ensure that the class $a_{n+1}=\bigcup (R^{-1\bullet}[a_n])$ is a set. 
By the Principle of Mathematical Induction, for every $n\in\w$, the class $a_n$ is a set. 

By Theorem~\ref{t:index}, $\Phi=\{\langle n,y\rangle:n\in\w\;\wedge\;y=a_n\}$ is a well-defined function. By the Axiom of Replacement, $\Phi[\w]$ is a set and by the Axiom of Union, the union $b=\bigcup\Phi[\w]=\bigcup_{n\in\w}a_n$ is a set. It follows that $a=a_0\subseteq b$. Also for any $x\in R^{-1}[b]$ we can find $n\in \w$ such that $x\in R^{-1}[a_n]$ and conclude that $x\in R^{-1}[a_n]=a_{n+1}\subseteq b$.
\end{proof}

\begin{corollary}\label{c:cup-sequence} For every set $a$ there exists a sequence of sets $(a_n)_{n\in\omega}$ such that $a_0=a$ and $a_{n+1}=\bigcup a_n$ for every $n\in\w$.
\end{corollary}

\begin{proof} Applying Theorem~\ref{t:set-like} to the set-like relation $\E$, we obtain  a sequence of sets  $(a_n)_{n\in\omega}$ such that $a_0=a$ and for every $n\in\w$,
$$a_{n+1}=\E^{-1}[a_n]=\{y:\exists x\in a_n\;(\langle x,y\rangle\in\E^{-1})\}=\{y:\exists x\in a_n\;(y\in x)\}=\textstyle \bigcup a_n.$$
\end{proof}

\section{Well-founded relations}

In this section we introduce and discuss well-founded relations, which play an extremely important role in Classical Set Theory.

\begin{definition}\label{d:well-founded} A relation $R$ is defined to be
 \index{well-founded relation}\index{relation!well-founded}{\em well-founded} if every nonempty class $X$ contains an element $x\in X$ such that $\cev R(x)\cap X=\emptyset$.
\end{definition} 

For transitive set-like relations the quantifier over classes in the definition of well-foundedness can be replaced by a $\UU$-bounded quantifier.




\begin{proposition}\label{p:well-founded} A transitive set-like relation $R$ is well-founded if and only if every nonempty set $a\subseteq\rng[R]$ contains an element $y\in a$ such that $\cev R(y)\cap a=\emptyset$.
\end{proposition}

\begin{proof} The ``only if'' part is trivial. To prove the ``if'' part, fix a nonempty class $X$.  Take any element $x\in X$ and consider the  class $a=(\cev R(x)\cup\{x\})\cap X$, which is a set, being a subclass of the set $\cev R(x)\cup\{x\}$. 
If $a\not\subseteq\rng[R]$, then take any element $z\in a\setminus \rng[R]$ and observe that the class $\cev R(z)\subseteq R^{-1}[\{z\}]=\emptyset$ is empty and hence $\cev R(z)$ is a set with $\cev R(z)\cap X=\emptyset$.

So, we assume that $a\subseteq\rng[R]$. Since $x\in a$, the set $a$ is not empty and by the assumption, there exists an element $y\in a\subseteq X$ such that $\cev R(y)\cap a=\emptyset$. If $y=x$, then $$
X\cap \cev R(x)\subseteq X\cap(\cev R(x)\cup\{x\})\cap \cev R(x)=a\cap\cev R(y)=\emptyset$$and the point $x=y\in X$ has the required property:  $\cev R(x)\cap X=\emptyset$.

Now assume that $y\ne x$. In this case $$y\in a\setminus\{x\}\subseteq (\cev R(x)\cup\{x\})\setminus\{x\}\subseteq\cev R(x)\subseteq R^{-1}[\{x\}]$$ and then the transitivity of the relation $R$ ensures that $$\cev R(y)\subseteq R^{-1}[\{y\}]\subseteq R^{-1}[R^{-1}[\{x\}]]\subseteq R^{-1}[\{x\}]\subset R^{-1}[\{x\}]\cup\{x\}.$$ Then
$$X\cap \cev R(y)=X\cap(R^{-1}[\{x\}]\cup\{x\})\cap \cev R(y)=a\cap\cev R(y)=\emptyset$$and $y\in a\subset X$ is a required element such that $\cev R(y)\cap X=\emptyset$.  
\end{proof}

Well-founded relations allow us to generalize the Principle of Mathematical Induction to the Principle of Well-Founded Induction. In fact, the Principle of Mathematical Induction has two forms.
\smallskip

\noindent{\bf Principle of Mathematical Induction:}\index{Principle of Mathematical Induction}\index{Mathematical Induction}\\ {\em Let $X$ be a set of natural numbers. If $\emptyset\in X$ and for every $n\in X$ the number $n+1$ belongs to $X$, then $X=\w$.
}
\smallskip

\begin{proof} It follows that $X$ is an inductive set and hence $\w\subseteq X$ by the definition of $\w$. Since $X\subseteq \w$ we have $X=\w$.
\end{proof}

\noindent{\bf Principle of Mathematical Induction (metaversion):} {\em Let $\varphi(x)$ be a $\mathbf U$-bounded formula with free variables $x,Y_1,\dots,Y_m$. Let $Y_1,\dots,Y_m$ be any classes. Assume that $\varphi(\emptyset,Y_1,\dots,Y_m)$ holds and for every $n\in\w$ if $\varphi(n,Y_1,\dots,Y_m)$ holds, then $\varphi(n+1,Y_1,\dots,Y_m)$ holds. Then $\varphi(n,Y_1,\dots,Y_m)$ holds for all $n\in\w$.
}
\smallskip

\begin{proof} Consider the class $X=\{n\in\w:\varphi(n,Y_1,\dots,Y_m)\}$ which exists by G\"odel's  Theorem~\ref{t:class}. Since $X\subseteq \w$, the class $X$ is a set, see Exercise~\ref{ex:subclass}. Our assumptions on $\varphi$ ensure that the set $X$ is inductive. Then $X=\w$ by the minimality of the inductive set $\w$.
\end{proof}

The Principle of Well-Founded induction also has two versions.
\smallskip

\noindent{\bf Principle of Well-Founded Induction:} {\em Let $Y$ be a subclass of a class $X$ and $R$  be a well-founded relation such that $\dom[R]\subseteq X$. If each element $x\in X$ with $\cev R(x)\subseteq Y$ belongs to $Y$, then $Y=X$.\index{Principle of Well-Founded Induction}\index{Well-Founded Induction}
}
\smallskip

\begin{proof} Assuming that $Y\ne X$, consider the nonempty class $X\setminus Y$ and by the well-foundedness of the relation $R$, find an element $x\in X\setminus Y$ such that $\cev R(x)\cap (X\setminus Y)=\emptyset$ and hence $$\cev R(x)\subseteq \dom[R]\setminus(X\setminus Y)\subseteq X\setminus(X\setminus Y)=Y$$ Now the assumption ensures that $x\in Y$, which contradicts the choice of $x$.
\end{proof}

\noindent{\bf Principle of Well-Founded Induction (metaversion):} {\em Let $\varphi(x,Y_1,\dots,Y_m)$ be a $\mathbf U$-bounded formula with a free variables $x,Y_1,\dots,Y_m$. Let $Y_1,\dots,Y_m$ be any classes. Let $R$ be a well-founded relation and $X$ be a class such that $\dom[R]\subseteq X$. Assume that for any $x\in X$ the following implication holds:
$$(\forall z\in\cev R(x)\;\varphi(z,Y_1,\dots,Y_m))\;\Rightarrow\;\varphi(x,Y_1,\dots,Y_m).$$ Then $\varphi(x,Y_1,\dots,Y_m)$ holds for every $x\in X$.
}

\begin{proof} Consider the class $Y=\{x\in X:\varphi(x,Y_1,\dots,Y_m)\}$ which exists by G\"odel's class existence Theorem~\ref{t:class}. Applying the Principle of Well-Founded Induction, we conclude that $Y=X$.
\end{proof}

\begin{remark} If the well-founded relation $R$ in the Principle of Well-Founded induction coincides with the membership relation $\mathbf E{\restriction}\mathbf{On}$ restricted to the class of ordinals (see Section~\ref{s:ordinals} for definition of ordinals), then the Principle of Well-Founded Induction is called the \index{Principle of Transfinite Induction}{\em Principle of Transfinite Induction}. 
\end{remark}

\begin{Exercise}\label{ex:AF<=>E-WF} Prove that the Axiom of Foundation holds if and only if the relation $\mathbf E$ is well-founded and irreflexive.

{\em Hint:} Look at Theorem~\ref{t:AF<=>E-WF}.
\end{Exercise}
\newpage

\part{Order}

In this section we consider some notions related to order and introduce ordinals.

\section{Order relations} There exists a wide class of relations describing various types of order on classes and sets.  

\begin{definition}  A relation $R$ is called 
\begin{itemize}
\item \index{antisymmetric relation}\index{relation!antisymmetric}{\em antisymmetric} if $R\cap R^{-1}\subseteq \Id$;
\item \index{order}\index{relation!of order} an {\em order} if the relation $R$ is transitive and antisymmetric;
\item \index{linear order}\index{order!linear} a {\em linear order} if $R$ is an order such that $\dom[R^\pm]\times\dom[R^\pm]\subseteq R\cup\Id\cup R^{-1}$; 
\item a \index{well-order}\index{order!well-order}  {\em well-order} if $R$ is a well-founded linear order;
\item a \index{tree-order}\index{order!tree} {\em tree-order} if $R$ is an order such that for every $x\in \dom[R^\pm]$ the order $R{\restriction}\cev R(x)$ is a well-order.
\end{itemize} 
\end{definition}


\begin{exercise} Prove that any well-founded (transitive) relation is antisymmetric (and hence is an order relation).
\end{exercise}

For order relations we have the following implication.
$$
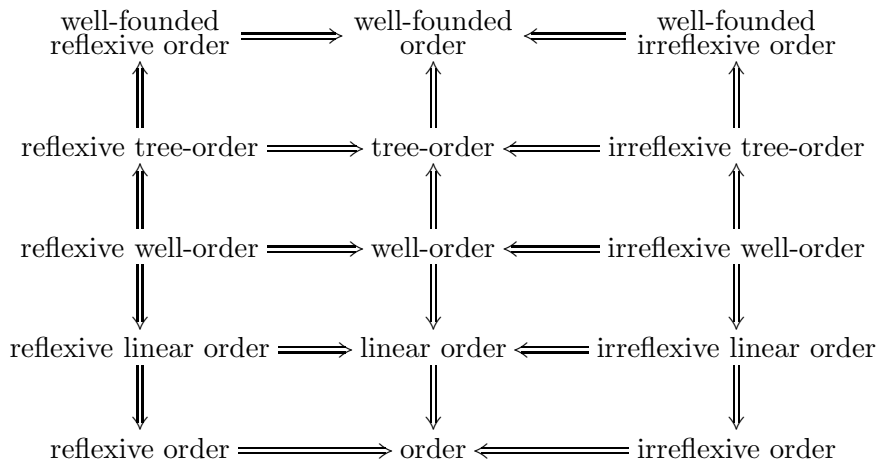
\xymatrix{
\mbox{well-founded}\atop\mbox{reflexive order}\ar@{=>}[r]&\mbox{well-founded}\atop\mbox{order}&\mbox{well-founded}\atop\mbox{irreflexive order}\ar@{=>}[l]\\
\mbox{reflexive tree-order}\ar@{=>}[u]\ar@{=>}[r]&\mbox{tree-order}\ar@{=>}[u]&\mbox{irreflexive tree-order}\ar@{=>}[u]\ar@{=>}[l]\\
\mbox{reflexive well-order}\ar@{=>}[r]\ar@{=>}[u]\ar@{=>}[d]&\mbox{well-order}\ar@{=>}[d]\ar@{=>}[u]&\mbox{irreflexive well-order}\ar@{=>}[l]\ar@{=>}[d]\ar@{=>}[u]\\
\mbox{reflexive linear order}\ar@{=>}[r]\ar@{=>}[d]&\mbox{linear order}\ar@{=>}[d]&\mbox{irreflexive linear order}\ar@{=>}[l]\ar@{=>}[d]\\
\mbox{reflexive order}\ar@{=>}[r]&\mbox{order}&\mbox{irreflexive order}\ar@{=>}[l]
}
$$
\smallskip

\begin{remark} Reflexive orders are called \index{partial order}\index{order!partial}{\em partial orders}, and irreflexive orders are called \index{strict order}\index{order!strict}{\em strict orders}.\\ Reflexive orders are usually denoted by $\leq$, $\preceq$, $\sqsubseteq$ and irreflexive orders  by $<$, $\prec$, $\sqsubset$ etc. 
\end{remark}

\begin{exercise} Show that for any order $R$, the relation $R\setminus\Id$ is an irreflexive order and the relation $R\cup\Id{\restriction}\dom[R^\pm]$ is a reflexive order.
\end{exercise}

\begin{definition} For a reflexive order $R$ (which is a set), the pair $(\dom[R],R)$ is called a \index{class!partially ordered}\index{partially ordered class}{\em partially ordered class} (resp. a {\em partially ordered set} or briefly, a {\em poset}).\index{partially ordered set}\index{set!partially ordered}\index{poset}
\end{definition}



\begin{exercise} Prove that the relation $ \mathbf S=\{\langle x,y\rangle\in\ddot\UU:x\subseteq y\}$ is a partial order.
\end{exercise}

\begin{exercise}\label{ex:char-WO} Prove that an order $R$ is a well-order if and only if every non-empty class $X\subseteq\dom[R^\pm]$ contains an element $x\in X$ such that $\langle x,y\rangle\in R$ for every $y\in X\setminus\{x\}$.
\end{exercise}

\begin{exercise}\label{ex:Lin} Prove that the class $\mathbf{Lin}$ of strict linear orders exists and is basic.

\end{exercise}

\begin{exercise}\label{ex:WF} Prove that the class $\mathbf{WF}$ of well-founded relations exists and is basic.
%
\end{exercise}

\begin{exercise} Prove that the classes of sets which are orders (linear orders, well-orders)  exist are basic.
\end{exercise}

Let $R$ be a relation and $X$ be a class. An element $x\in X\cap\dom[R^\pm]$ is called
\begin{itemize}
\item an\index{minimal element}\index{element!minimal} {\em $R$-minimal element of $X$} if $\forall y\in X\;(\langle y,x\rangle\in R\;\Rightarrow\;y=x)$;
\item an {\em $R$-maximal element of $X$}\index{maximal element}\index{element!maximal}  if $\forall y\in X\;(\langle x,y\rangle\in R\;\Rightarrow\;y=x)$;
\item an {\em $R$-least element of $X$}\index{least element}\index{element!least}  if  $\forall y\in X\;(\langle x,y\rangle\in R\cup\Id)$;
\item an {\em $R$-greatest element of $X$}\index{greatest element}\index{element!greatest}  if  $\forall y\in X\;(\langle y,x\rangle\in R\cup\Id)$.
\end{itemize}

\begin{exercise} Let $X$ be a class and $R$ be an antisymmetric relation.  Show that every $R$-least element of $X$ is $R$-minimal, and every $R$-greatest element of $X$ is $R$-maximal in $X$.
\end{exercise}

\begin{exercise} Let $R$ be a linear order and $X\subseteq \dom[R^\pm]$. Show that an element $x\in X$ is 
\begin{itemize}
\item $R$-minimal in $X$ if and only if $x$ is the $R$-least element of $X$;
\item $R$-maximal in $X$ if and only if $x$ is the $R$-greatest element of $X$.
\end{itemize}
\end{exercise}

\begin{example} Consider the partial order $\mathbf S=\{\langle x,y\rangle:x\subseteq y\}$. Observe that every element of the set $X=\{\{0\},\{1\}\}$ is $\mathbf S$-minimal and $\mathbf S$-maximal, but $X$ contains no $\mathbf S$-least and no $\mathbf S$-greatest elements.
\end{example}

\begin{exercise}\label{ex:fin-max} Prove that for every natural number $n\in\IN$ and any order $R$ with $\dom[R^\pm]\subseteq n$ there exist an $R$-minimal element $x\in n$ and an $R$-maximal element $y\in n$.
\smallskip

\noindent{\em Hint:} Apply the Principle of Mathematical Induction.
\end{exercise}

We recall that a set $x$ is called finite if there exists an injective function $f$ such that $\dom[f]=x$ and $\rng[f]\in\w$.

\begin{exercise}\label{ex:finite-order-mm}  Prove that for any order $R$ on a finite set $X=\dom[R^\pm]$ there exist an $R$-minimal element $x\in X$ and an $R$-maximal element $y\in X$.
\end{exercise}

\begin{definition} Let $L$ be an order on the class $X=\dom[L^\pm]$. A subclass $A\subseteq X$ is called \index{subset!upper bounded}\index{subset!lower bounded}{\em upper $L$-bounded} (resp. {\em lower $L$-bounded\/}) if the class $\overline{A}=\{b\in X:A\times\{b\}\subseteq L\cup\Id\}$ (resp. the class $\underline{A}=\{b\in X:\{b\}\times A\subseteq L\cup\Id\}$) is not empty.

If the class $\overline{A}$ (resp. $\underline{A}$) contains the $L$-least (resp. $L$-greatest) element, then this unique element is denoted by $\sup_L(A)$ (resp. $\inf_L(A)$) and called {\em the least upper $L$-bound} (resp. {\em the greatest lower $L$-bound}) of $A$. \index{least upper bound}\index{greatest lower bound}\index{$\sup$}\index{$\inf$}
\end{definition}

For a well-founded order $R$ and an element $x\in\dom[R^\pm]$, the class 
$$\Succ_R(x)={\min}_R (R[\{x\}]\setminus\{x\})$$
of $R$-minimal elements of the class $R[\{x\}]=\{y:\langle x,y\rangle\in R\}$ is called the {\em class of immediate $R$-successors} of $x$ in the well-founded order $R$.

\begin{definition} A well-founded order $R$ is called
\begin{itemize}
\item \index{well-founded order!locally finite}{\em locally finite}  if for every $x\in \dom[R^\pm]$ the class $\Succ_R(x)$  is a finite set;
\item \index{well-founded order!locally countable} {\em locally countable}   if for every $x\in \dom[R^\pm]$ the class $\Succ_R(x)$ a countable set; 
\item \index{well-founded order!locally set} {\em locally set}  if for every $x\in \dom[R^\pm]$ the class $\Succ_R(x)$ is a set.
\end{itemize} 
\end{definition}







\section{Transitivity}

\begin{definition} A class $X$ is called \index{transitive class}\index{class!transitive}{\em transitive} if $\forall y\in X\;\forall z\in y\;(z\in X)$.
\end{definition}

Observe that the transitivity is defined by a $\Delta_0$-formula.

\begin{exercise} Prove that a class $X$ is transitive if and only if $\forall x\in X\; (x\subseteq X)$ if and only if $\bigcup X\subseteq X$.
\end{exercise}

\begin{exercise}\label{ex:Tr} Prove that the class $\mathbf{Tr}$ of transitive sets exists and is basic.
\smallskip

\noindent{\em Hint:} Observe that $\UU\setminus\mathbf{Tr}=\{x:\exists y\;\exists z\;(y\in x\;\wedge\;z\in y\;\wedge\;z\notin x)\}=\rng[T_1\cap T_2\cap T_3]$ where 
$T_1=\{\langle z,y,x\rangle:y\in x\}=[\mathbf E\times\UU]^\circlearrowright$, 
$T_2=\{\langle z,y,x\rangle:z\in y\}=\mathbf E\times\UU$, and\\ $T_3=\{\langle z,y,x\rangle:z\notin x\}=[(\ddot\UU\setminus \E^{-1})\times\UU]^\circlearrowleft$.
\end{exercise} 

\begin{remark} The membership relation $\E{\restriction}\mathbf{Tr}$ on the class $\mathbf{Tr}$ is transitive.
\end{remark}

\begin{exercise} Let $(X_\alpha)_{\alpha\in A}$ be an indexed family of transitive classes. Prove that the union $\bigcup_{\alpha\in A}X_\alpha$ and intersection $\bigcap_{\alpha\in A}X_\alpha$ are transitive classes.
\end{exercise}



Let us recall that a class $X$ is called \index{class!inductive}\index{inductive class}{\em inductive} if $\emptyset\in X$ and for every $x\in X$ the set $x\cup\{x\}$ belongs to $X$. By the Axiom of Infinity there exists an inductive set. The intersection of all inductive sets is denoted by $\w$.

\begin{proposition}\label{p:trans=>ind} The class $\mathbf{Tr}$ of transitive sets is inductive. 
\end{proposition}

\begin{proof} It is clear that the empty set is transitive. Assume that a set $x$ is transitive and take any element $y\in x\cup\{x\}$. If $y\in x$, then for every $z\in y$ the element $z$ belongs to $x\cup\{x\}$ by the transitivity of $x$.  If $y=x$, then every element $z\in y=x$ belongs to $x\subseteq x\cup\{x\}$ as $y=x$.
\end{proof}

\begin{theorem}\label{t:wTr} $\w\subset\mathbf{Tr}$ and $\w\in\mathbf{Tr}$.
\end{theorem}

\begin{proof} Since the classes $\mathbf{Tr}$ and $\w$ are inductive, so is their intersection $\mathbf{Tr}\cap \w$. By Exercise~\ref{ex:subclass}, the class $\mathbf{Tr}\cap \w$ is a set and then $\w\subseteq\w\cap\mathbf{Tr}\subseteq\mathbf{Tr}$ by the definition of $\w$.

To prove that $\w\in\mathbf{Tr}$, we need to show that each element of $\w$ is a subset of $\w$. For this consider the class $T=\{n\in\w:n\subseteq\w\}$, which is equal to $(\w\times\UU)\cap\dom[\mathbf S\cap (\UU\times\{\w\})]$ and hence exists. The class $T$ is a set, being a subclass of the set $\w$, see Exercise~\ref{ex:subclass}. Let us show that the set $T$ is inductive. It is clear that $\emptyset\in T$. Assuming that $n\in T$, we conclude that $n\subseteq \w$ and also $n\in T\subseteq\w$. Then $n\cup\{n\}\subseteq\w$, which means that $n\cup\{n\}\in T$ and the set $T$ is inductive. Since $\w$ is the smallest inductive set, $\w=T$. Consequently, $\w$ is transitive set and $\w\in\mathbf{Tr}$.
\end{proof}

Given any set $x$, consider the class $\Tr(x)=\{y\in\Tr:x\subseteq y\}$ of transitive sets that contain $x$. This class is equal to $\Tr\cap \mathbf S[\{x\}]$ and hence exists by Exercises~\ref{ex:Tr} and \ref{ex:S}. By Corollary~\ref{c:cup-sequence}, there exists a sequence of sets $(x_n)_{n\in\omega}$ such that $x_0=x$ and $x_{n+1}=\bigcup x_n$ for all $n\in\w$. It follows that the set $x_{<\w}\defeq \bigcup_{n\in\w}x_n$ is transitive and contains $x$ as a subset. Therefore, $x_{<\w}\in \Tr(x)$, witnessing that the class $\Tr(x)$ is not empty. The intersection $$\UTC(x)\defeq{\textstyle\bigcap}\Tr(x)$$of the class $\Tr(x)$ is the smallest transitive set that contains $x$. This set is called the \index{transitive closure}{\em transitive closure} of $X$. The following proposition establishes some properties of the transitive closure.


\begin{proposition}\label{p:TC} $\UTC(x)=x\cup\UTC(\bigcup x)=x\cup\bigcup\limits_{y\in x}\UTC(y)$ for every set $x$.
\end{proposition}

\begin{proof}  The equality $\UTC(x)=x\cup\UTC(\bigcup x)$ will be established as soon as we check that the set $x\cup \UTC(\bigcup x)$ is transitive and is a subset of every transitive set $t$ that contains $x$. 

Given any sets $u\in x\cup\UTC(\bigcup x)$ and $v\in u$, we shall prove that $v\in x\cup\UTC(\bigcup x)$. If $u\in x$, then $v\in\bigcup x\subseteq\UTC(\bigcup x)\subseteq x\cup\UTC(\bigcup x)$.  If $u\in\UTC(\bigcup x)$, then $v\in\UTC(\bigcup x)$ by the transitivity of the set $\UTC(\bigcup x)$. Therefore the set $x\cup\UTC(\bigcup x)$ is transitive. 

Now let $T$ be any transitive set that contains $x$. Then $\bigcup x\subseteq T$ by transitivity of $T$ and then $\UTC(\bigcup x)\subseteq T$ since $\UTC(\bigcup x)$ is the smallest transitive set that contains $\bigcup x$.  Then $x\cup\UTC(\bigcup x)\subseteq T$. 
\end{proof}


\begin{exercise}\label{ex:TC} Given any sets $x,y,z$, prove the following equalities:
\begin{enumerate}
\item $\UTC(x\cup y)=\UTC(x)\cup\UTC(y)$
\item $\UTC(\{x,y\})=\{x,y\}\cup\UTC(x\cup y)$.
\item $\UTC(\langle x,y\rangle)=\{\{x\},\{x,y\}\}\cup\{x,y\}\cup\UTC(x\cup y)$.
\item $\UTC(x\times y)=(x\times y)\cup\{\{u\}:u\in x\}\cup\{\{u,v\}:u\in x\;\wedge\;v\in y\}\cup\UTC(x\cup y)$.
\item $\UTC(\dom[x])\subseteq\UTC(x)$.
\item Calculate $\UTC(\langle x,y,z\rangle)$. 
\end{enumerate}
\end{exercise}

Now we can present a solution to Exercise~\ref{ex:AF<=>E-WF}.

\begin{theorem}\label{t:AF<=>E-WF} Axiom of Foundation holds if and only if the relation $\mathbf E$ is well-founded and irreflexive.
\end{theorem}

\begin{proof} Assume that the Axiom of Foundation holds. We have to prove that the relation $\mathbf E$ is well-founded and irreflexive. Assuming that $\mathbf E$ is not irreflexive, we can find a set $x$ such that $x\in x$. Then every element of the set $\{x\}$ intersects the set $x$, which contradicts the Axiom of Foundation. To prove that the relation $\mathbf E$ is well-founded, take any nonempty class $X$. We have to find an element $a\in X$ such that $\cev{\mathbf E}(a)\cap X=\emptyset$. Choose any element $x\in X$. If $x\cap X=\varnothing$, then the element $a\defeq x$ has the required property. So, assume that $x\cap X\ne\varnothing$. Consider the transitive closure $\UTC(x)$ of $x$, and observe that $\UTC(x)\cap X$ is a nonempty set. By the Axiom of Foundation, there exists an element $a\in \UTC(x)\cap X$ such that $a\cap\UTC(x)\cap X=\varnothing$. Assuming that $a\cap X\ne\emptyset$, we can find an element $b\in a\cap X$ and conclude that $b\in a\cap \UTC(x)\cap X$, which contradicts the choice of $a$. This contradiction shows that $a\cap X=\emptyset$ and hence $\cev{\mathbf E}(a)\cap X\subseteq a\cap X=\varnothing$.
\smallskip

Now assume that the relation $\mathbf E$ is well-founded and irreflexive. To prove that the Axiom of Foundation holds, take any nonempty set $x$. By the well-foundedness of the relation $\mathbf E$, there exists an element $y\in X$ such that $\cev{\mathbf E}(y)\cap x=\varnothing$. The irreflexivity of the relation $\mathbf E$ ensures that $y\notin y$ and hence $\cev{\mathbf E}(y)=y$ and $y\cap x=\cev{\mathbf E}(y)\cap x=\emptyset$, witnessing that the Axiom of Foundation holds. 
\end{proof}

\section{Ordinals}\label{s:ordinals}


\begin{definition}[von Neumann]\label{d:ordinals} A set $x$ is called an \index{ordinal}{\em ordinal} if $x$ is transitive and the relation $\mathbf E{\restriction}x=\E\cap(x\times x)$ is an irreflexive well-order on $x$.
\end{definition}

\begin{exercise} Show that under the Axiom of Foundation, a transitive set $x$ is an ordinal if and only if the relation $\mathbf E{\restriction}x$ is a linear order.
\end{exercise}

We recall that for two classes $X,Y$ the notation $X\subset Y$ means that $X\subseteq Y$ and $X\ne Y$.

\begin{theorem}\label{t:ord}
\begin{enumerate}
\item[\textup{1)}] Each element of an ordinal is a transitive set.
\item[\textup{2)}]  Each element of an ordinal is an ordinal.
\item[\textup{3)}]  The intersection of two ordinals is an ordinal.
\item[\textup{4)}]  For any two ordinals $\alpha,\beta$, we have the equivalence: $\alpha\in\beta\;\Leftrightarrow\;\alpha\subset\beta$.
\item[\textup{5)}]  For any ordinals $\alpha,\beta$ we have the dichotomy: $\alpha\subseteq\beta\;\vee\;\beta\subseteq\alpha$.
\item[\textup{6)}]  For any ordinals $\alpha,\beta$ we have the trichotomy: $\alpha\in\beta\;\vee\;\alpha=\beta\;\vee\;\beta\in\alpha$.
\end{enumerate}
\end{theorem}

\begin{proof} 1. Let $\beta$ be an ordinal and $\alpha\in\beta$. To show that the set $\alpha$ is transitive, take any sets $y\in \alpha$ and $z\in y$. The transitivity of $\beta$ guarantees that $y\in\beta$ and $z\in\beta$. 
Since $z\in y\in\alpha$, the transitivity of the relation $\mathbf E{\restriction}\beta$ implies $z\in\alpha$, which means that the set $\alpha$ is transitive.
\smallskip

2. Let $\beta$ be an ordinal and $\alpha\in\beta$. By the preceding statement, the set $\alpha$ is transitive. The transitivity of the set $\beta$ implies that $\alpha\subseteq\beta$. Since the relation $\mathbf E{\restriction}\beta$ is an irreflexive well-order, its restriction $\mathbf E{\restriction}\alpha$ to the subset $\alpha\subseteq\beta$ is an irreflexive well-order, too. This means that the transitive set $\alpha$ is an ordinal.
\smallskip

3. The definition of an ordinal implies that the intersection of two ordinals is an ordinal.
\smallskip

4. Let $\alpha,\beta$ be two ordinals. If $\alpha\in\beta$, then $\alpha\subseteq \beta$ by the transitivity of $\beta$, and $\alpha\ne\beta$ by the irreflexivity of the relation $\mathbf E{\restriction}\beta$.  Therefore, $\alpha\in\beta$ implies $\alpha\subset\beta$.

Now assume that $\alpha\subset\beta$. Since the set $\beta\setminus\alpha$ is not empty and $\mathbf E{\restriction}\beta$ is a strict well-order, there exists an element $\gamma\in\beta\setminus\alpha$ such that $\gamma\cap(\beta\setminus\alpha)=\emptyset$. By the transitivity of the set $\beta$, the element $\gamma$ is a subset of $\beta$. Taking into account that $\gamma\cap(\beta\setminus\alpha)=\emptyset$, we conclude that $\gamma\subseteq\alpha$. We claim that $\gamma=\alpha$. In the opposite case, there exists an element $\delta\in\alpha\setminus\gamma$. Since $\delta\in\alpha\subset\beta$ and $\gamma\in\beta$, we can apply the linearity of the order $\mathbf E{\restriction}\beta$ to conclude that $\delta\in\gamma\;\vee\;\delta=\gamma\;\vee\;\gamma\in\delta$. The assumption $\delta\in\gamma$, contradicts the choice of $\delta\in\alpha\setminus\gamma$. The equality $\delta=\gamma$ implies that $\gamma=\delta\in\alpha$, which contradicts the choice of $\gamma\in\beta\setminus\alpha$. The assumption $\gamma\in\delta$ implies $\gamma\in\delta\subseteq\alpha$ (by the transitivity of $\alpha$) and this contradicts the choice of $\gamma\in\beta\setminus\alpha$. These contradictions imply that $\alpha=\gamma\in\beta$. 
\smallskip

5. Let $\alpha,\beta$ be two ordinals. Assuming that neither $\alpha\subseteq \beta$ nor $\beta\subseteq\alpha$, we conclude that the ordinal $\gamma=\alpha\cap\beta$ is a proper subset in $\alpha$ and $\beta$. Applying the preceding statement, we conclude that $\gamma\in\alpha$ and $\gamma\in\beta$. This implies that $\gamma\in\alpha\cap\beta=\gamma$. But this contradicts the irreflexivity of the well-order $\E{\restriction}\alpha$ on the ordinal $\alpha$. 
\smallskip

6. Let $\alpha,\beta$ be two ordinals. By the preceding statement, $\alpha\subseteq\beta$ or $\beta\subseteq\alpha$. This implies the trichotomy $\alpha\subset \beta\;\vee\;\alpha=\beta\;\vee\;\beta\subset\alpha$, which is equivalent to the trichotomy $\alpha\in \beta\;\vee\;\alpha=\beta\;\vee\;\beta\in\alpha$ according to the statement (4).
\end{proof}

\begin{exercise} Using the Axiom of Foundation prove that a set $x$ is an ordinal if and only if $x$ is transitive and each element $y$ of $x$ is a transitive set.
\smallskip

\noindent{\em Hint:} See Theorem~\ref{t:ordinal=ht}.
\end{exercise}

Now we establish some properties of the class of ordinals $\Ord$.\index{class!{{\bf On}}}\index{{{\bf On}}}

\begin{exercise} Prove that the class $\Ord$ exists and is basic.
\smallskip

\noindent{\em Hint:} Apply Exercises~\ref{ex:Tr}, \ref{ex:Lin}, \ref{ex:WF}.
\end{exercise}

\begin{theorem}\label{t:Ord}
\begin{enumerate}
\item[\textup{(1)}] The class $\Ord$ is transitive.
\item[\textup{(2)}] The relation $\E{\restriction}\Ord$ is an irreflexive well-order on the class $\Ord$.
\item[\textup{(3)}]  $\emptyset\in\Ord$.
\item[\textup{(4)}] $\forall \alpha\;(\alpha\in\Ord\;\Ra\;\alpha\cup\{\alpha\}\in\Ord)$.
\item[\textup{(5)}] $\forall x\in\UU\;(x\subseteq\Ord\;\Ra\;\bigcup x\in\Ord)$.
\item[\textup{(6)}] $\Ord$ is a proper class.
\item[\textup{(7)}] $\w\subset\Ord$.
\item[\textup{(8)}] $\w\in\Ord$.
\end{enumerate}
\end{theorem}

\begin{proof} 
1. The transitivity of the class $\Ord$ follows from Theorem~\ref{t:ord}(2) saying that any element of an ordinal is an ordinal.
\smallskip

2. The transitivity of ordinals implies the transitivity of the relation $\E{\restriction}\Ord$. The irreflexivity of the relation $\E{\restriction}\Ord$ follows from the irreflexivity of the relation $\E{\restriction}\alpha$ for each ordinal. Applying  Theorem~\ref{t:ord}(6), we see that the relation $\E{\restriction}\Ord$ is a linear order on the class $\Ord$. To see that $\E{\restriction}\Ord$ is well-founded, take any nonempty subclass $A\subseteq\Ord$. We should find an ordinal $\alpha\in A$ such that $\alpha\cap A=\emptyset$. Take any ordinal $\beta\in A$. If $\beta\cap A=\emptyset$, then we are done. In the opposite case, $\beta\cap A$ is a nonempty subset of the ordinal $\beta$. Since $\E{\restriction}\beta$ is an irreflexive well-order, there exists an ordinal $\alpha\in \beta\cap A$ such that $\alpha\cap(\beta\cap A)=\emptyset$. The transitivity of the set $\beta$ guarantees that $\alpha\subseteq\beta$ and then $\alpha\cap A=(\alpha\cap\beta)\cap A=\alpha\cap(\beta\cap A)=\emptyset$.
\smallskip

3. The inclusion $\emptyset\in\Ord$ is trivial.
\smallskip

4. Assume that $\alpha$ is an ordinal. Then $\alpha$ is a transitive set and by Proposition~\ref{p:trans=>ind}, its successor $\beta=\alpha\cup\{\alpha\}$ is transitive, too. It remains to prove that the relation $\E{\restriction}\beta$ is an irreflexive well-order. The transitivity of the relation $\E{\restriction}\beta$ follows from the transitivity of the set $\alpha$ and the transitivity of the relation $\E{\restriction}\alpha$.  The irreflexivity of this relation follows from the irreflexivity of the relation $\E{\restriction}\alpha$, which implies also that $\alpha\notin\alpha$. 

To see that $\E{\restriction}\beta$ is a  well-order, it suffices to show that any nonempty subset $X\subseteq \beta$ contains an element $x\in X$ such that $x\in y$ for every $y\in X\setminus\{x\}$. If $X=\{\alpha\}$, then $x=\alpha$ has the required property. If $X\ne\{\alpha\}$, then $X\cap \alpha$ is a non-empty set in $\alpha$. Since $\alpha$ is an ordinal, there exists $x\in X\cap\alpha$ such that $x\in y$ for every $y\in X\cap\alpha\setminus\{x\}$. For $y=\alpha$ we have $x\in X\cap\alpha\subseteq\alpha=y$, too.
\smallskip

5. Let $x$ be any subset of $\Ord$. By the Axiom of Union, the class $\cup x=\{z:\exists y\in x\;(z\in y)\}$ is a set. The transitivity of the class $\Ord$ guarantees that $\cup x\subseteq\Ord$. Since $\E{\restriction}\Ord$ is an irreflexive well-order, its restriction $\E{\restriction}{\cup}x$ is an irreflexive well-order, too. 
Since the elements of $x$ are transitive sets, the union $\cup x$ is a transitive set. Therefore, $\cup x$ is an ordinal and hence $\cup x\in\Ord$.
\smallskip

6. Assuming that the class $\Ord$ is a set, we can apply the preceding statement and conclude that $\bigcup\Ord=\Ord$ is an ordinal and hence $\Ord\in\Ord$ and $\Ord\in\Ord\in\Ord$, which contradicts the irreflexivity of the relation $\E{\restriction}\Ord$. This contradiction shows that $\Ord$ is a proper class.
\smallskip

7. The statements (3) and (4) imply that the class $\Ord$ is inductive. Then the intersection $\w\cap\Ord$ is an inductive class. By Exercise~\ref{ex:subclass}, the class $\w\cap \Ord$ is a set and hence $\w\subseteq\w\cap\Ord\subseteq\Ord$. Since $\Ord$ is a proper class, $\w\ne\Ord$ and hence $\w\subset\Ord$.
\smallskip

8. By the statements (7) and (5), $\w\subset\Ord$ and $\cup\w\in\Ord$. By Theorem~\ref{t:wTr}, the set $\w$ is transitive and hence $\cup\w\subseteq\w$. On the other hand, for any $x\in\w$ we have $x\in x\cup\{x\}\in\w$ and hence $x\in\bigcup\w$. Therefore, $\w=\cup\w\in\Ord$.
\end{proof}

\begin{definition} An ordinal $\alpha$ is called
\begin{itemize}
\item a \index{successor ordinal}\index{ordinal!successor} {\em successor ordinal} if $\alpha=\beta\cup\{\beta\}$ for some ordinal $\beta$;
\item a \index{limit ordinal}\index{ordinal!limit} {\em limit ordinal} if $\alpha$ is not empty and is not a successor ordinal;
\item a \index{finite ordinal}\index{ordinal!finite} {\em finite ordinal} if every ordinal $\beta\in(\alpha\cup\{\alpha\})\setminus\{\emptyset\}$ is a successor ordinal.
\end{itemize}
\end{definition}

\begin{exercise}\label{ex:FO} Show that the classes of successor ordinals, limit ordinals, finite ordinals exist.
\end{exercise}

\begin{exercise} Prove that for every ordinal $\alpha$ we have
$$\cup\alpha=\begin{cases}\alpha&\mbox{if $\alpha$ is a limit ordinal};\\
\beta&\mbox{if $\alpha=\beta\cup\{\beta\}$ is a successor ordinal.}
\end{cases}
$$
\end{exercise}

\begin{exercise} Prove that the class $\Ord$ coincides with the smallest class $X$ such that
\begin{enumerate}
\item $\forall x\in X\;(x\cup\{x\}\in X)$;
\item $\forall x\in\UU\;(x\subseteq X\;\Ra\;\bigcup x\in X)$.
\end{enumerate}
\end{exercise}  

Now we reveal the interplay between ordinals and natural numbers (i.e., the elements of the smallest inductive set $\w$).

\begin{theorem}
\begin{enumerate}
\item[\textup{1)}] The ordinal $\w$ is the smallest limit ordinal.
\item[\textup{2)}] The set of natural numbers $\w$ coincides with the set of finite ordinals.
\end{enumerate}
\end{theorem}

\begin{proof} Assuming that $\w$ is a successor ordinal, we can find an ordinal $\alpha$ such that $\w=\alpha\cup\{\alpha\}$. Then $\alpha\in\w$ and by the inductivity of $\w$, we obtain $\w=\alpha\cup\{\alpha\}\in\w$, which is not possible as $\w$ is an ordinal (by Theorem~\ref{t:ord}(8)). This contradiction shows that the ordinal $\w$ is limit. Assuming that $\w$ is not the smallest nonempty limit ordinal, we can find a nonempty limit ordinal $\alpha\in \w$. By Theorem~\ref{t:ord}(4), $\emptyset\in \alpha$ and $\alpha\subset \w$. Since $\alpha$ is a limit ordinal, for every $x\in\alpha$, $x\cup\{x\}\ne \alpha$. By the transitivity of $\alpha$, we obtain $x\subseteq \alpha$ and hence $x\cup\{x\}\subset\alpha$. Applying Theorem~\ref{t:ord}(4), we conclude that $x\cup\{x\}\in\alpha$, which means that the set $\alpha$ is inductive and hence $\w\subseteq\alpha$ as $\w$ is the smallest inductive set. Then $\w\subseteq\alpha\subset \w$ and Theorem~\ref{t:ord}(4), imply $\w\in\w$, which is not possible as $\w\in\Ord$.
\smallskip

2. Denote by $\mathsf{FO}$ the class of finite ordinals (it exists by Exercise~\ref{ex:FO}). It is easy to see that the class $\mathsf{FO}$ is inductive and hence $\w\subseteq\mathsf{FO}$. On the other hand, the limit property of $\w$ and  Theorem~\ref{t:ord}(6) ensure that $\mathsf{FO}\subseteq\w$. Therefore, $\w=\mathsf{FO}$ is the set of all finite ordinals.
\end{proof}
\smallskip

\noindent{\bf Some terminology and notation.} 
Since the Membership relation $\E{\restriction}\Ord$ on the class $\Ord$ is a linear order, it is often denoted by the symbol $<$. Therefore, given two ordinals $\alpha,\beta$ we write $\alpha<\beta$ if $\alpha\in\beta$ (which is equivalent to $\alpha\subset \beta$). In this case we say that $\alpha$ is {\em smaller} than $\beta$. Also we write $\alpha\le \beta$ if $\alpha\in \beta$ or $\alpha=\beta$.  The successor $\alpha\cup\{\alpha\}$ is often denoted by $\alpha+1$. For a set $A$ of ordinals let $$\sup A=\min\{\beta\in\Ord:\forall \alpha\in A\;\;(\alpha\subseteq \beta)\}.$$

\begin{lemma}\label{l:sup-ord} For any set $A\subseteq\Ord$,
$$\sup A=\textstyle\bigcup A.$$
\end{lemma}

\begin{proof} By Theorem~\ref{t:Ord}(5), the set $\bigcup A$ is an ordinal. For every $\alpha\in A$ the definition of the union $\bigcup A$ ensures that $\alpha\subseteq\bigcup A$, which implies $\sup A\le\bigcup A$. Assuming that $\sup A<\bigcup A$, we would conclude that $\sup A\in\bigcup A$ and hence $\sup A\in\alpha$ for some $\alpha\in A$.
By Theorem~\ref{t:ord}(4), $\sup A\in\alpha$ implies $\sup A\subset\alpha\subseteq\sup A$ and hence $\sup A\ne\sup A$, which is a contradiction showing that $\sup A=\bigcup A$.
\end{proof}

\section{Trees}\label{s:trees}

In this section we introduce some notions related to trees and tree-orders. We recall that a \index{tree-order}\index{order!tree-order} {\em tree-order} is an order $R$ such that for every $x\in\dom[R]$ the restriction $R{\restriction}\cev R(x)$ is a well-order.

\begin{example} Let $\UU^{<\Ord}$ be the class of functions $f$ with $\dom[f]\in\Ord$. Then for the reflexive order relation $\mathbf S=\{\langle x,y\rangle\in\ddot\UU:x\subseteq y\}$  the restriction $\mathbf S{\restriction}\UU^{<\Ord}$ is a tree-order.  
\end{example}

\begin{definition} A class $T$ is called an \index{ordinary tree}\index{tree!ordinary}{\em ordinary tree} if $T\subseteq \UU^{<\Ord}$ and for every $t\in T$ and any ordinal $\alpha$ the function $t{\restriction}_\alpha$ belongs to $T$. \\
The class $\dom[T]=\{\dom[t]:t\in T\}$ is called the \index{tree!height of}{\em height} of an ordinary tree $T$.
\end{definition}

\begin{example} For every class $A$ and every ordinal $\kappa$ the class $$A^{<\kappa}=\{t\in\UU^{<\Ord}:\dom[t]\in\kappa,\;\rng[t]\subseteq A\}$$ is an ordinary tree.
\end{example}

For every ordinal $\kappa$ the ordinary tree $2^{<\kappa}$ is called the \index{full binary tree}\index{tree!full binary}{\em full binary $\kappa$-tree}. Any ordinary tree $T\subseteq 2^{<\w}$ is called a \index{binary tree}\index{tree!binary}{\em binary tree}. 

The ordinary tree $\Ord^{<\Ord}$ of all functions $f$ with $\dom[f]\in\Ord$ and $\rng[f]\subseteq\Ord$ carries an irreflexive linear order $\mathsf{W}_{\!<}$ defined by the formula
$$
\begin{aligned}
\mathsf{W}_<=\;&\big\{\langle f,g\rangle \in\Ord^{<\Ord}\times\Ord^{<\Ord}:
\dom[f]\cup\textstyle\bigcup\rng[f]\subset\dom[g]\cup\bigcup\rng[g]\;\vee\\
&\hskip100pt (\dom[f]\cup\textstyle\bigcup\rng[f]=\dom[g]\cup\bigcup\rng[g]\;\wedge\;\dom[f]\subset\dom[g])\;\vee\;\\
&\hskip100pt  \big( \dom[f]\cup\textstyle\bigcup\rng[f]=\dom[g]\cup\bigcup\rng[g]\;\wedge\;\dom[f]=\dom[g]\;\wedge\;\\
&\hskip150pt\exists \alpha\in \dom[f]=\dom[g]\;\big(f{\restriction}_\alpha=g{\restriction}_\alpha\;\wedge\; f(\alpha)\in g(\alpha)\big)\big)\big\}
\end{aligned}
$$
and called the {\em canonical linear order} on $\Ord^{<\Ord}$.

\begin{exercise} Prove that the restriction $\mathsf{W}_{\!<}{\restriction}\Ord^{<\w}$ is a set-like well-order.
\end{exercise}

\begin{exercise} Prove that the order $\mathsf{W}_{\!<}{\restriction}2^{\w}$ is not well-founded.
\end{exercise}

For an ordinary tree $T$ and an element $x\in T$ the class 
$$\Succ_T(x)=\{t\in T:\dom[t]=\dom[x]+1\;\wedge\;t{\restriction}_{\dom[x]}=x\}$$ is called the {\em class of immediate successors} of $x$ in the tree $T$.

\begin{definition} An ordinary tree $T$ is called
\begin{itemize}
\item \index{locally finite tree}\index{tree!locally finite}{\em locally finite}  if for every $x\in \dom[T^\pm]$ the class $\Succ_T(x)$  is a finite set;
\item \index{locally countable tree}\index{tree!locally countable}{\em locally countable}   if for every $x\in \dom[T^\pm]$ the class $\Succ_T(x)$ a countable set; 
\item \index{locally set tree}\index{tree!locally set}{\em locally set}  if for every $x\in \dom[T^\pm]$ the class $\Succ_T(x)$ is a set.
\end{itemize} 
\end{definition}

\begin{Exercise} Show that an ordinary tree $T$ is a set if and only if it is set-like, locally set and its height $\dom[T]$ is a set.
\end{Exercise} 

\section{Recursion Theorem}

In this section we prove an important theorem guaranteeing the existence of functions defined by recursive procedures, which are often used in Mathematics and Computer Science. 

\begin{theorem}\label{t:Recursion} Let $X$ be a class, $F:X\times\UU\to\UU$ be a function and $R$ be a set-like well-founded relation such that $\dom[R]\subseteq X$. Then there exists a unique function $G:X\to\UU$ such that for every $x\in X$ $$G(x)=F(x,G{\restriction}_{\cev R(x)})\quad\mbox{where \ $\cev R(x)=R^{-1}[\{x\}]\setminus\{x\}$}.$$
If the classes $F$ and $R$ are basic, then so is the function $G$.
\end{theorem}

\begin{proof} Since the relation $R$ is set-like, for every $x\in X$ the class $$\cev R(x)=R^{-1}[\{x\}]\setminus\{x\}=\{z:z\ne x\;\wedge\;\langle z,x\rangle\in R\}$$is a set.

Let $\mathbb G$ be the class consisting of all functions $f\in\UU$ that satisfy the following conditions:
\begin{enumerate}
\item[(i)] $\forall x\in \dom[f]\;\;(\cev R(x)\subseteq\dom[f]\subseteq X)$;
\item[(ii)]  $f(x)=F(x,f{\restriction}_{\cev R(x)})$ for every $x\in\dom[f]$.
\end{enumerate}
We claim that any two functions $f,g\in  \mathbb G$ agree on the intersection of their domains. To derive a contradiction, assume that the set $A=\{x\in\dom[f]\cap\dom[g]:f(x)\ne g(x)\}$ is not empty. 
Since the relation $R$ is well-founded, the set $A$ contains an element $a$ such that $\cev R(a)\cap A=\emptyset$. The condition (i) ensures that $$\cev R(a)\subseteq (\dom[f]\cap\dom[g])\setminus A=\{x\in\dom[f]\cap \dom[g]:f(x)=g(x)\}$$ and then 
$$f(a)=F(a,f{\restriction}_{\cev R(a)})=F(a,g{\restriction}_{\cev R(a)})=g(a),$$
which contradicts $a\in A$. This contradiction shows that the functions $f$ and $g$ coincide on the intersection of their domains.

This property of the class $\mathbb G$ implies that the class
$$G=\textstyle{\bigcup}\mathbb G=\{\langle x,y\rangle:\exists f\in \mathbb G\;\;\langle x,y\rangle\in f\}$$
is a function. 

By definition of $G=\bigcup \mathbb G$, for every $x\in\dom[G]$ there exists $y\in \UU$ such that $\langle x,y\rangle\in G=\bigcup\mathbb G$ and hence $\langle x,y\rangle\in f$ for some $f\in \mathbb G$. Taking into account that $G\cap(\dom[f]\times\UU)=f$, we conclude that 
$$G(x)=y=f(x)=F( x,f{\restriction}_{\cev R(x)})=F( x,G{\restriction}_{\cev R(x)}).$$

Next, we prove that $\dom[G]=X$. The condition (i) guarantees that $\dom[G]\subseteq X$. Assuming that $\dom[G]\ne X$ and using the well-foundedness of the relation $R$, we can find an element $x\in X\setminus\dom[G]$ such that $\cev R(x)\cap (X\setminus \dom[G])=\emptyset$ and hence $\cev R(x)\subseteq \dom[G]$. By Theorem~\ref{t:set-like}, there exists a sequence of sets $(A_n)_{n\in\w}$ such that $A_0=\cev R(x)$ and $A_{n+1}=\cev R[A_n]$ for every $n\in\w$. By induction it can be shown that $A_n\subseteq\dom[G]$ for every $n\in\w$. Then the set $A=\bigcup_{n\in\w}A_n$ is a subset of $\dom[G]$.

Now consider the function 
$$f=G{\restriction}_{A}\cup \{\langle x,F(x,G{\restriction}_{\cev R(x)})\rangle\}.$$with domain $\dom[f]=A\cup\{x\}$ and observe that $f$ has properties (i),(ii) and hence $f\in \mathbb G$ and $x\in\dom[f]\subseteq\dom[G]$, which contradicts the choice of $x$. 
\smallskip

Finally, we show that the function $G$ is unique. Indeed, take any function $\Phi:X\to \UU$ such that $\Phi(x)=F(x,\Phi{\restriction}_{\cev R(x)})$ for all $x\in X$. Assuming that $\Phi\ne G$, we conclude that the class  $D=\{x\in X:\Phi(x)\ne G(x)\}$ is not empty and by the well-foundedness of the relation $R$ contains an element $a\in D$ such that $\cev R(a)\cap D=\emptyset$ and hence $\cev R(a)\subseteq\{x\in X:\Phi(x)=G(x)\}$. Then $$\Phi(a)=F(a,\Phi{\restriction}_{\vec R(a)})=F(a,G{\restriction}_{\vec R(a)})=G(a),$$which contradicts the choice of $a$.

If the classes $F$ and $R$ are basic, then so are the classes $X=\dom[\dom[F]]$, $\mathbb G$ and $G$, according to Theorem~\ref{t:class}.
\end{proof} 

Now we apply the Recursion Theorem to legalize recursive definitions of  (function) sequences.

\begin{theorem}\label{t:recursion2} For every class $X$ and functions $G_0:X\to\UU$ and $F:(\w\times X)\times \UU\to \UU$ there exists a unique function $G:\w\times X\to\UU$ such that $$G(0,x)=G_0(x)\mbox{ \ and \ }G(n+1,x)=F(n+1,x,G(n,x))$$ for every $\langle n,x\rangle\in \w\times X$.
\end{theorem}

\begin{proof}  Consider the function $\Phi:(\w\times X)\times\UU\to\UU$ defined by 
$$\Phi(n,x,y)=\begin{cases}
F(n,x,z)&\mbox{if $y$ is a function and $\exists k\in\w$ such that $n=k+1$ and $\langle \langle k,x\rangle,z\rangle\in y$};\\
G_0(x)&\mbox{otherwise}.
\end{cases}
$$
\smallskip

The function $\Phi$ exists by the G\"odel class existence theorem.
Consider the set-like well-founded order
$$R=\{\langle\langle k,x\rangle,\langle n,x\rangle\rangle:k\in n\in\w,\; x\in X\}$$on the class $\w\times X=\dom[R^\pm].$

By the Recursion Theorem~\ref{t:Recursion}, there exists a unique function $G:\w\times X\to\UU$ such that 
\begin{equation}\label{eq:recursion1}
G(n,x)=\Phi(n,x,\{\langle \langle k,x\rangle,G(k,x)\rangle:k\in n\})
\end{equation}
 for all $\langle n,x\rangle\in\w\times X$.
 
The equality (\ref{eq:recursion1}) and the definition of the function $\Phi$ imply that $G(0,x)=\Phi(0,x,\emptyset)=G_0(x)$ and 
$$
G(n+1,x)=\Phi(n+1,x,\{\langle \langle k,x\rangle,G(k,x)\rangle:k\in n+1\})=F(n+1,x, G(n,x)).$$\end{proof}

As a special case of Theorem~\ref{t:recursion2} for $X=\{0\}$, we obtain the following corollary justifying the definition of sequences by recursive formulas. 

\begin{corollary}\label{c:recursion3} For every function $F:\w\times \UU\to \UU$ and set $x$ there exists a unique sequence $(x_n)_{n\in\w}$ such that $$x_0=x\mbox{ \ and \ }x_{n+1}=F(n+1,x_n)$$ for every $\langle n,x\rangle\in \w\times X$.
\end{corollary}


Now we apply Theorem~\ref{t:recursion2} to legalize the widely used procedure of  iterations of functions.
Let $X$ be a class and $\Phi:X\to X$ be a function. Consider the sequence of functions $(\Phi^{\circ n})_{n\in\w}$ defined by the recursive formula:
$$\Phi^{\circ 0}=\Id{\restriction}_X\quad\mbox{and}\quad \Phi^{\circ(n+1)}=\Phi\circ \Phi^{\circ n}\quad\mbox{for every $n\in\w$}.$$ 

 Let us recall that for two functions $G,H$ their compositions $G\circ H$ is defined as the function
$$G\circ H=\{\langle x,z\rangle:\exists y\;\langle x,y\rangle\in H\;\wedge\;\langle y,z\rangle\in G\}.$$

\begin{theorem}\label{t:iterate} For every class $X$ and (basic) function $\Phi:X\to X$ the function sequence $(\Phi^{\circ n})_{n\in\w}$ exists (and is basic). 
\end{theorem}

\begin{proof} Consider the function $F:(\w\times X)\times\UU\to\UU$ defined by $$F(n,x,y)=\begin{cases}\Phi(y)&\mbox{if $y\in X$ and $n>0$};\\
x&\mbox{otherwise}.
\end{cases}
$$
If the function $\Phi$ is basic, then so is the class $X=\dom[\Phi]$ and the function $F$, see Theorem~\ref{t:class}.

By Theorem~\ref{t:recursion2}, there exists a (basic) function $G:\w\times X\to \UU$ such that $G(0,x)=x$ and $G(n+1,x)=F(n+1,x,G(n,x))$ for all $\langle n,x\rangle\in\w\times X$.

By induction we shall prove that for every $x\in X$ and $n\in\w$ we have
\begin{equation}\label{eq:iteration}
G(n,x)\in X\quad\mbox{and}\quad G(n,x)=\Phi^{\circ n}(x).
\end{equation}
Observe that $G(0,x)=x\in X$ and $G(0,x)=\Phi^{\circ 0}(x)$. Assume that for some $n\in\w$ the equality (\ref{eq:iteration}) holds.
Then $$G(n+1,x)=F(n+1,x,G(n,x))=\Phi(G(n,x))=\Phi(\Phi^{\circ n}(x))=\Phi^{\circ(n+1)}(x)\in X.$$ By the Principle of Matemathical Induction, the equality (\ref{eq:iteration}) holds for all $n\in\w$.
\end{proof}

As an application of iterations let us prove the existence of transitive closures. Consider the function of taking union $$\textstyle\bigcup:\UU\to\UU,\quad \bigcup:x\mapsto\bigcup x,$$and its iterations $\bigcup^{\circ n}$ for $n\in\w$. Taking the union of those iterations, we obtain the function $${\textstyle\bigcup}^{\circ\w}:\UU\to\UU,\quad{\textstyle\bigcup}^{\circ\w}:x\mapsto\textstyle\bigcup\{\bigcup^{\circ n}x:n\in\w\}.$$

\begin{theorem}\label{t:TC} The function $\bigcup^{\circ\w}$ coincides with the function $\UTC$ of transitive closure. 
\end{theorem}

\begin{proof} We need to show that for every set $x$ the set $\bigcup^{\circ \w}\!x$ is the smallest transitive set containing $x$ as a subset.

To see that the set $\bigcup^{\circ \w}\!x$ is transitive, take any element $y\in \bigcup^{\circ \w}\!x=\bigcup_{n\in\w}\bigcup^{\circ n}\!x$ and find $n\in\w$ such that $y\in \bigcup^{\circ n}\!x$. Then $y\subseteq \bigcup(\bigcup^{\circ n}\!x)=\bigcup^{\circ(n+1)}\!x\subseteq\bigcup^{\circ \w}\!x$. So, $\bigcup^{\circ \w}\!x$ is transitive.

 Next, we prove that $\bigcup^{\circ \w}\!x\subseteq Y$ for any transitive class $Y$ with $x\subseteq Y$.
Since $\bigcup^{\circ \w}\!x=\bigcup_{n\in\w}\bigcup^{\circ n}\!x$, it suffices to show that $\forall n\in\w\;\;\bigcup^{\circ n}\!x\subseteq Y$. For $n=0$ this follows from the equality $\bigcup^{\circ 0}\!x=x\subseteq Y$. Assume that for some $n\in\w$ we proved that $\bigcup^{\circ n}\!x\subseteq Y$. Then $\bigcup^{\circ(n+1)}\!x=\bigcup(\bigcup^{\circ n}\!x)\subseteq \bigcup Y\subseteq Y$ by the transitivity of $Y$. Applying the Principle of Mathematical Induction, we conclude that $\bigcup^{\circ n}\!x\subseteq Y$ for all $\langle n,x\rangle\in\w\times X$.
\end{proof}




\section{Transfinite Dynamics}\label{s:dynamics}

Let $X$ be a class and $P\subseteq X\times X$ be an order on $X$. A subset $C\subseteq X$ is called an {\em $P$-chain} if $C\times C\subseteq P^\pm\cup\Id$. The order $P$ is called {\em $\sup$-complete} if each $P$-chain $C\in\mathcal P(X)$ has the least upper bound $\sup_P C\in X$ with respect to the order $P$. The $\sup$-completeness of $P$ implies that $X$ contains the $P$-least element, equal to the least upper bound of the empty set $\sup_P \emptyset$.  

A function $f:X\to X$ is called {\em $P$-progressive} if for every $x\in X$ we have $\langle x,f(x)\rangle\in P\cup\Id$.

Given a class $X$ endowed with a $\sup$-complete order $P\subseteq X\times X$ and a $P$-progressive function, consider the transfinite sequence of functions $(f^{\circ\alpha}:X\to X)_{\alpha\in\Ord}$ defined by the recursive formula
\begin{equation}\label{eq:recursive-dymanics}
f^{\circ\alpha}(x)=\begin{cases}
x&\mbox{if $\alpha=0$};\\
f(f^{\circ\beta}(x))&\mbox{if $\alpha=\beta+1$ is a successor ordinal};\\
\sup_P\{f^{\circ\beta}(x):\beta<\alpha\}&\mbox{if $\alpha$ is a limit ordinal}.
\end{cases}
\end{equation}

\begin{theorem}\label{t:dynamics} The transfinite sequence of functions $(f^{\circ\alpha})_{\alpha\in\Ord}$ is well-defined and monotone in the sense that $\langle f^{\circ\alpha}(x),f^{\circ\beta}(x)\rangle\in P\cup\Id$ for any $x\in X$ and any ordinals $\alpha\le\beta$.
\end{theorem}

\begin{proof} Consider the function $F:(X\times \Ord)\times\UU\to X$ assigning to each triple $\langle x,\alpha,y\rangle\in(X\times\Ord)\times\UU$ the element
$$F(x,\alpha,y)=\begin{cases}\sup_P f[\rng[y]])&\mbox{if $\alpha>0$ and $f[\rng[y]]$ is an $P$-chain in $X$;}\\
x&\mbox{otherwise};
\end{cases}
$$
of the class $X$.

On the class $X\times\Ord$ consider the set-like well-founded order $$R=\{\langle\langle x,\alpha\rangle,\langle x,\beta\rangle\rangle:x\in X\;\wedge\;\alpha\in\beta\in\Ord\}$$and observe that $\dom[R^\pm]=X\times\Ord$ and $\cev R(x,\alpha)=\{x\}\times\alpha$ for any ordered pair $\langle x,\alpha\rangle\in X\times\Ord$.

By Recursion Theorem~\ref{t:Recursion}, there exists a unique function $G:X\times\Ord\to\UU$ such that $G(x,\alpha)=F(x,\alpha,G{\restriction}_{\{x\}\times\alpha})$ for every $\langle x,\alpha\rangle\in X\times\Ord$. In particular,
$$G(x,0)=F(x,0,G{\restriction}_{\{x\}\times 0})=F(x,0,\emptyset)=x$$
for every $x\in X$.

By transfinite induction we shall prove that for every $\langle x,\alpha\rangle\in X\times\Ord$, the following conditions are satisfied:
\begin{itemize}
\item[$(1_\alpha)$] $\forall\beta\in \alpha\;\;\langle G(x,\beta), G(x,\alpha)\rangle \in P\cup\Id$;
\item[$(2_\alpha)$] $G(x,\alpha+1)=f(G(x,\alpha))$;
\item[$(3_\alpha)$] if $\alpha$ is a nonzero limit ordinal, then $G(x,\alpha)=\textstyle{\sup_P}\{G(x,\beta):\beta\in\alpha\}$ for any $x\in X$.
\end{itemize}
\smallskip

Assume that for some ordinal $\alpha$ and all $\gamma\in\alpha$, the conditions $(1_\gamma)$--$(3_\gamma)$ hold. 

If $\alpha=0$, then the conditions $(1_0)$ and $(3_0)$ are vacuously true. The condition $(2_0)$ holds as
$$G(x,1)=F(x,1,G{\restriction}_{\{x\}\times 1})=F(x,1,\{\langle \langle x,0\rangle, G(x,0)\rangle\})=\sup_P \{f(G(x,0))\}=f(G(x,0)).$$

Next, assume that the ordinal $\alpha$ is limit and nonzero. In this case the conditions $(1_\gamma)$, $\gamma\in\alpha$, imply that for every $x\in X$ the set $\{G(x,\beta):\beta\in\alpha\}$ is an $P$-chain and hence has $\sup\{G(x,\beta):\beta\in\alpha\}$. The conditions $(2_\beta)$, $\beta\in\alpha$, imply that $\{f(G(x,\beta)):\beta\in\alpha\}=\{G(x,\beta+1):\beta\in\alpha\}$ is a subchain of the $P$-chain $\{G(x,\beta):\beta\in\alpha\}$ with $\sup_P\{G(x,\beta+1):\beta\in\alpha\}=\sup_P\{G(x,\beta):\beta\in\alpha\}$. Now we see that the condition $(3_\alpha)$ holds:
\begin{multline*}
G(x,\alpha)=F(x,\alpha,G{\restriction}_{\{x\}\times\alpha})=F(x,\alpha,\{\langle\langle x,\beta\rangle,G(x,\beta)\rangle:\beta\in\alpha\})=\\
=\textstyle{\sup_P} \{f(G(x,\beta)):\beta\in\alpha\}=\textstyle{\sup_P}\{G(x,\beta+1):\beta\in\alpha\}=\textstyle{\sup_P} \{G(x,\beta):\beta\in\alpha\}.
\end{multline*}
Therefore, $\{G(x,\beta):\beta\le \alpha\}$ is an $R$-chain and the condition $(1_\alpha)$ holds. The $P$-progressivity of $f$ and the conditions $(2_\beta)$, $\beta\in\alpha$, imply that the set 
$\{f(G(x,\beta)):\beta\le\alpha\}=\{G(x,\beta+1):\beta<\alpha\}\cup\{f(G(x,\alpha))\}$ is a $P$-chain whose largest element is $f(G(x,\alpha))$.
 Then
\begin{multline*}G(x,\alpha+1)=F(x,\alpha+1,G{\restriction}_{\{x\}\times(\alpha+1)})=F(x,\alpha+1,\{\langle \langle x,\beta\rangle,G(x,\beta)\rangle:\beta\in\alpha+1\})=\\
=\textstyle{\sup_P}\{f(G(x,\beta)):\beta\in\alpha+1\}=f(G(x,\alpha)),
\end{multline*}
which means that the condition $(2_\alpha)$ holds.

Finally, assume that $\alpha$ is a successor ordinal and hence $\alpha=\beta+1$ for some ordinal $\beta$. The condition $(1_\beta)$ guarantees that for every $x\in X$ the set $\{G(x,\gamma):\gamma\le\beta\}$ is an $R$-chain whose largest element is $G(x,\beta)$. The $P$-progressivity of $f$ and the conditions $(1_\gamma)$ and $(2_\gamma)$ for $\gamma\le\beta$   imply that 
$$\{f(G(x,\gamma)):\gamma\le\beta\}=\{G(x,\gamma+1):\gamma\le\beta\}$$ is a $P$-chain and hence
\begin{multline*}
G(x,\alpha+1)=F(x,\alpha+1,G{\restriction}_{\{x\}\times(\alpha+1)})=F(x,\alpha+1,\{\langle \langle x,\beta\rangle,G(x,\beta)\rangle:\beta<\alpha+1\})=\\
=\textstyle{\sup_P}\{f(G(x,\beta)):\beta\in\alpha+1\}=f(G(x,\alpha)),
\end{multline*}
which means that the condition $(2_\alpha)$ holds.

By the Principle of Transfinite Induction, the conditions $(1_\alpha)$--$(3_\alpha)$ hold for every ordinal $\alpha$.
\smallskip

Now we prove that $G(x,\alpha)=f^{\circ\alpha}(x)$ for any $x\in X$ and $\alpha\in \Ord$. For $\alpha=0$ this is true: $G(x,0)=x=f^{\circ0}(x)$. Assume that for some ordinal $\alpha$ and all its elements $\gamma\in\alpha$ we have proved that $G(x,\gamma)=f^{\circ\gamma}(x)$. If $\alpha$ is a successor ordinal, then $\alpha=\beta+1$ for some ordinal $\beta$ and by the inductive condition $(2_\beta)$,
$$G(x,\alpha)=G(x,\beta+1)=f(G(x,\beta))=f(f^{\circ\beta}(x))=f^{\circ\alpha}(x).$$
If $\alpha$ is a limit nonzero ordinal, then using the inductive assumption and the inductive condition $(3_\alpha)$, we obtain
$$
G(x,\alpha)=\textstyle{\sup_P}\{G(x,\gamma):\gamma\in\alpha\}=\textstyle{\sup_P}\{f^{\circ\gamma}(x):\gamma\in\alpha\}=f^{\circ\alpha}(x).
$$
\end{proof}

Theorem~\ref{t:dynamics} implies the following Fixed Point Theorem.

\begin{theorem}[Tarski]\label{t:Tarski-fix-point} Let $X$ be a nonempty set endowed with a $\sup$-complete order $P\subseteq X\times X$. Every $P$-progressive function $f:X\to X$ has a fixed point $x=f(x)\in X$.
\end{theorem}

\begin{proof} Given any element $x\in X$ consider the transfinite sequence $(f^{\circ\alpha}(x))_{\alpha\in\On}$ defined by the recursive formulas:
\begin{itemize}
\item $f^{\circ 0}(x)=x$;
\item $f^{\circ(\alpha+1)}(x)=f(f^{\circ \alpha}(x))$;
\item $f^{\circ\alpha}(x)=\sup_P\{f^{\circ\beta}(x):\beta\in\alpha\}$ if $\alpha$ is a nonzero limit ordinal.
\end{itemize}
Theorem~\ref{t:dynamics} guarantees that $\langle f^{\circ\alpha}(x),f^{\circ\beta}(x)\rangle\in R\cup\Id$ for any ordinals $\alpha\le \beta$. Assuming that the function $f$ has no fixed point, we would conclude that $f^{\circ(\alpha+1)}(x)=f(f^{\circ\alpha}(x))\ne f^{\circ\alpha}(x)$ for every $\alpha$ and hence the function $\Phi:\On\to X$, $\Phi:\alpha\mapsto f^{\circ \alpha}(x)$ is injective, which is not possible as $X$ is set and $\On$ is a proper class.
\end{proof}

\begin{remark} The function $\mathsf{succ}:\On\to \On$, $\mathsf{succ}:\alpha\mapsto \alpha+1$, is $\mathbf S{\restriction}\On$-progressive but has no fixed points, witnessing that the Fixed Point Theorem~\ref{t:Tarski-fix-point} is not true for progressive functions on proper classes.
\end{remark}

\begin{example} For the $\mathbf S$-progressive function $F:\UU\to\UU$, $F:x\mapsto x\cup\mathcal P(x)$, the transfinite iterations $F^{\circ\alpha}(\emptyset)$ are equal to the sets $V_\alpha$ of the von Neumann cumulative hierarchy $(V_\alpha)_{\alpha\in\Ord}$, see Section~\ref{s:vN}.
\end{example}

\section{Ranks}

In this section we discuss the rank functions induced by set-like well-founded relations. The intuition behind this notions is the following.

Given any well-founded relation $R$ on a set $X=\dom[R^\pm]$, we can consider the set $X_0$ of elements $x\in X$ of whose initial interval $\cev R(x)=\{z:\langle z,x\rangle\in R\}\setminus\{x\}$ is empty. The elements of the set $X_0$ are called $R$-minimal elements of $X$. The rank function $\rank_R$ assigns to elements of the set $X_0$ the ordinal $ 0=\emptyset$. Then we consider the set $X_1$ of $R$-minimal elements of the set $X\setminus X_0$ and assign to them the ordinal $1=\{0\}$. Next, consider the set $X_2$ of $R$-minimal elements of the set $X\setminus(X_0\cup X_1)$. Continuing by induction, we represent $X$ as the union $X=\bigcup_{\alpha\in\rank(R)}X_\alpha$ of sets $X_\alpha$ indexed by ordinals $\alpha$ that belong to some ordinal $\rank(R)$, called the rank of the well-founded order $R$. The function $\rank_R:X\to\rank(R)$ assigns to each $x\in X$ a unique ordinal $\alpha\in \rank(R)$ such that $x\in X_\alpha$. So, this is a rough idea.
\smallskip

 Now let us give the precise definition of the rank function $\rank_R$. In the definition for an ordinal $\alpha$ by $\alpha+1$ we denote the successor $\alpha\cup\{\alpha\}$ of $\alpha$, and for a set $A$ of ordinals, $\sup A=\min\{\beta\in\Ord:\forall \alpha\in A\;(\alpha\subseteq\beta)\}=\bigcup A$, see Lemma~\ref{l:sup-ord}.

\begin{definition} For a set-like well-founded relation $R$, the \index{rank}\index{well-founded relation!rank of}{\em $R$-rank} is the function $$\rank_R:\UU\to\Ord,\quad\rank_R:x\mapsto \sup\big\{\rank_R(y)+1:y\in\cev R(x)\big\}.$$
\end{definition}

The existence of the function $\rank_R$ follows from the Recursion Theorem~\ref{t:Recursion} applied to the function $F:\UU\times\UU\to\UU,\;\;F:\langle x,y\rangle\mapsto\textstyle{\bigcup}\{z\cup\{z\}:z\in\rng[y]\}$. 
The rank function $\rank_R$ can be characterized as the smallest $R$-increasing function $\UU\to\Ord$.

\begin{definition} Let $R,P$ be two relations and $X,Y$ be two classes. A function $F:X\to Y$ is called \index{function!$P$-$R$-increasing} {\em $R$-$P$-increasing} if for any distinct elements $x,x'\in X$ with $\langle x,x'\rangle\in R$ we have $\langle F(x),F(x')\rangle\in P$.
\end{definition}

\begin{theorem}\label{t:rank} Let $R$ be a set-like well-founded relation.
 \begin{enumerate} 
\item[\textup{1)}] The rank function $\rank_R:\UU\to\Ord$ is  $R$-$\E$-increasing.
\item[\textup{2)}] For every $R$-$\E$-increasing function $F:\UU\to\Ord$ we have $\rank_R(x)\le F(x)$ for all $x\in\UU$.
\item[\textup{3)}] If $\rank_R[\UU]\ne\Ord$, then $\rank_R[\UU]$ is an ordinal.
\end{enumerate}
\end{theorem}

\begin{proof} 1. To see that the rank function $\rank_R$ is $R$-$\E$-increasing, take any $x\in \UU$ and $y\in\cev R(x)$. Then $\rank_R(y)<\rank_R(y)+1\le\sup\{\rank_R(z)+1:z\in \cev R(x)]\}=\rank_R(x)$, which means that $\rank_R$ is $R$-$\E$-increasing.
\smallskip

2. Let $F:\UU\to\Ord$ be any $R$-$\E$-increasing function. To show that $\rank_R\le F$, it suffices to show that the class $X=\{x\in \UU:\rank_R(x)\not\le F(x)\}$ is empty. To derive a contradiction, assume that the class $X$ is not empty.
By the well-foundedness of the relation $R$, we can find an element $a\in X$ such that $\cev R(a)\cap X=\emptyset$. Then $\rank_R(x)\le F(x)$ for all $x\in \cev R(a)$. Since the function $F$ is $R$-$\E$-increasing, for every $x\in \cev R(a)$, we have $F(x)<F(a)$ and hence $F(x)+1\le F(a)$. 
Observe that for every $x\in\cev R(a)$ the inequality $\rank_R(x)\le F(x)$ implies $\rank_R(x)+1\le F(x)+1\le F(a)$ and hence $\rank_R(a)=\sup\{\rank_R(x)+1:x\in\cev R(a)\}\le F(a)$, which contradicts the choice of $a\in X$.
\smallskip

3.  Assuming that $\rank_R[\UU]\ne\Ord$, consider the smallest ordinal $\alpha$ in the class $\Ord\setminus\rank_R[\UU]$. Then every element of $\alpha$ belongs to $\rank_R[\UU]$ and hence $\alpha\subseteq\rank_R[\UU]$. Assuming that $\alpha\ne\rank_R[\UU]$, take the smallest ordinal $\beta$ in the set $\rank_R[\UU]\setminus\alpha$, and observe that $\alpha<\beta$ and $[\alpha,\beta]\cap \rank_R[\UU]=\{\beta\}$ where  $[\alpha,\beta]\defeq\{x\in\Ord:\alpha\le x\le \beta\}$. Consider the function $L:\Ord\to\Ord$ such that $L(\gamma)=\gamma$ for any $\gamma\in \Ord\setminus [\alpha,\beta]$ and $L(\gamma)=\alpha$ for every $\gamma\in [\alpha,\beta]$. Taking into account that $[\alpha,\beta]\cap\rank_R[\UU]=\{\beta\}$, we can show that the function $L\circ\rank_R:\UU\to\Ord$ is $R$-$\E$-increasing and $L\circ\rank_R(x)<\rank_R(x)$ for any $x\in\rank_R^{-1}[\{\beta\}]$. But this contradicts the preceding statement. This contradiction shows that $\rank_R[\UU]=\alpha\in\On$. 
\end{proof}

For every set-like well-founded relation $R$, let $\rank(R)=\rank_R[\UU]\subseteq\Ord$. By Theorem~\ref{t:rank}(3), the class $\rank(R)$ either coincides with the class $\Ord$ or is an ordinal. In the latter case this ordinal is called the \index{rank}\index{well-founded relation!rank of}{\em rank} of the well-founded relation $R$.

\section{Well-orders}

In this section we apply ranks to constructing isomorphisms between set-like well-orders.

\begin{definition} Let $R,P$ be two orders. A bijective function $F:\dom[R^\pm]\to\dom[P^\pm]$ is called an \index{order isomorphism}{\em order isomorphism} if the function $F$ is $R$-$P$-increasing and $F^{-1}$ is $P$-$R$-increasing. In this case the function $F^{-1}$ is also an order isomorphism. Two orders $R,P$ are called\index{isomorphic orders} {\em isomorphic} if there exists an order isomorphism $F:\dom[R^\pm]\to\dom[P^\pm]$.
\end{definition}

\begin{proposition}\label{p:iso-equivalent} Let $R,P$ be two linear orders. For an bijective function $F:\dom[R^\pm]\to\dom[P^\pm]$ the following conditions are equivalent:
\begin{enumerate}
\item[\textup{1)}] $F$ is an order isomorphism;
\item[\textup{2)}] $F$ is  $R$-$P$-increasing;
\item[\textup{3)}] $F^{-1}$ is $P$-$R$-increasing.
\end{enumerate}
\end{proposition} 

\begin{proof} The implication $(1)\Ra(2)$ is trivial.
\vskip3pt

$(2)\Ra(3)$. Assume that the condition (2) holds but (3) does not. Then there exist distinct points $x,x'\in X$ such that $\langle F(x),F(x')\rangle\in P$ but $\langle x,x'\rangle\notin R$. Then $\langle x',x\rangle\in R$ by the linearity of the order $R$. Applying the condition (2), we obtain $\langle F(x'),F(x)\rangle\in P$. Now the antisymmetry of the relation $P$ ensures that $F(x)=F(x')$ which contradicts our assumption.
\vskip3pt

$(3)\Ra(1)$  Assume that the condition (3) holds but (1) does not. Then there exist distinct points $x,x'\in \dom[R^\pm]$ such that $\langle x,x'\rangle\in R$ but $\langle F(x),F(x')\rangle\notin P$. Then $F(x)\ne F(x')$ by the injectivity of the function $F$ and  $\langle F(x'),F(x)\rangle\in P$ by the linearity of the order $P$. Applying the condition (3), we obtain $\langle x',x\rangle\in R$. Now the antisymmetry of the relation $R$ ensures that $x=x'$, which contradicts the choice of $x,x'$.
\end{proof}

\begin{proposition}\label{p:unique1} Let $R$ be a well-order. Every order isomorphism $F:\dom[R^\pm]\to\dom[R^\pm]$ is equal to the identity function $\mathsf{Id}{\restriction}\dom[R^\pm]$.
\end{proposition}

\begin{proof} To derive a contradiction, assume that $F(x)\ne x$ for some $x\in \dom[R^\pm]$. Then the class $A=\{z\in R^{-1}[\{x\}]:F(z)\ne z\}$ is not empty. 
Since $R$ is a well-order, the class $A$ contains an element $a\in A$ such that $\cev R(a)\cap A=\emptyset$. The transitivity of the relation $R$ ensures that $\cev R(a)\subset R^{-1}[\{x\}]$. It follows from $\cev R(a)\cap A=\emptyset$ that $F(z)=z=F^{-1}(z)$ for all $z\in\cev R(a)$. 

It follows from $a\in A$ that $F(a)\ne a$. Since the order $R$ is linear, either $F(a)\in \cev R(a)$ or $a\in\cev R(F(a))$. In the first case we get the equality $F(F(a))=F(a)$, which contradicts the injectivity of $F$. 
In the second case, the inclusion $a\in\cev R(F(a))$ and the $R$-$R$-increasing property of the isomorphism $F$ imply that $F^{-1}(a)\in\cev R(a)$, and then $F^{-1}(F^{-1}(a))=F^{-1}(a)$, which contradicts the injectivity of the function $F^{-1}$.
\end{proof}

\begin{corollary}\label{c:unique2} For any well-orders $P,R$ there exists at most one order-isomorphism from $P$ to $R$.
\end{corollary}

\begin{proof} Let $\Phi,\Psi:\dom[P^\pm]\to\dom[R^\pm]$ be two order-isomorphisms. Then $\Phi^{-1}\circ\Psi:\dom[P^\pm]\to\dom[P^\pm]$ is an order isomorphism of the well-order $P$. By Proposition~\ref{p:unique1}, $\Phi^{-1}\circ \Psi=\mathbf {Id}{\restriction}\dom[P^\pm]$. Applying to this equality the bijective function $\Phi$, we obtain the desired equality $\Psi=\Phi\circ \Phi^{-1}\circ\Psi=\Phi\circ\mathsf{Id}{\restriction}\dom[P^\pm]=\Phi$.
\end{proof}

Let $R$ be an order. For an element $x\in\dom[R^\pm]$, the set $\cev R(x)=R^{-1}[\{x\}]\setminus\{x\}$ is called the {\em initial interval} of $\dom[R^\pm]$ and the partial order $R{\restriction}\cev R(x)$ is called an \index{initial interval}{\em initial interval} of the partial order $R$.

\begin{proposition}\label{p:initial} A well-order $R$ cannot be isomorphic to its own initial interval.
\end{proposition}

\begin{proof} Assume that for some $x\in\dom[R^\pm]$ there exists an order isomorphism $F:\dom[R^\pm]\to\cev R(x)$. Then $F(x)\in\cev R(x)$ and hence $F(x)\ne x$. Consider the class $A=\{z\in R^{-1}[\{x\}]:F(z)\ne z\}$ which contains $x$ and hence is not empty. Since the order $R$ is well-founded, the class $A$ contains an element $a$ such that $\cev R(a)\cap A=\emptyset$. The transitivity of $R$ ensures that $\cev R(a)\subset R^{-1}[\{x\}]\setminus A$ and hence $F(z)=z=F^{-1}(z)$ for all $z\in\cev R(a)$. Since the order $R$ is linear and $F(a)\ne a$, either $F(a)\in\cev R(a)$ or $a\in\cev R(F(a))$. In the first case we obtain that $F(a)=F(F(a))$, which contradicts the injectivity of $F$.
If $a\in\cev R(F(a))$, then $a\in \cev R(F(a))\subset\cev R(x)=F[\dom[R]]$ and hence $F^{-1}(a)\in \dom[R]$ exists. Since $F$ is an order-isomorphism, $a\in\cev R(F(a))$ implies $F^{-1}(a)\in\cev R(a)$ and then $F^{-1}(F^{-1}(a))=F^{-1}(a)$, which contradicts the injectivity of $F^{-1}$.
\end{proof}

\begin{theorem}\label{t:wOrd} For any set-like well-order $R$ on a class $X=\dom[R^\pm]$, the function $\rank_R{\restriction}_X:X\to\rank(R)$ is an order isomorphism. 
\end{theorem}

\begin{proof} By Theorem~\ref{t:rank}, the function $\rank_R{\restriction}_X$ is $R$-$\E$-increasing. Since the order $R$ is linear, for any distinct elements $x,x'\in X$ we have $\langle x,x'\rangle\in R$ or $\langle x',x\rangle\in R$. Taking into account that $\rank_R$ is $P$-$\E$ increasing, we conclude that  $\rank_R(x)<\rank_R(x')$ or $\rank_R(x')<\rank_R(x)$. In both cases we have $\rank_R(x)\ne\rank_R(x')$, which means that the function $\rank_R$ is injective. Since $\rank(R)=\rank_R[\UU]=\rank_R[\dom[R^\pm]]$, the function $\rank_R$ is surjective and hence bijective. The $R$-$\E$-increasing property and  Proposition~\ref{p:iso-equivalent} imply that $\rank_R{\restriction}_X:X\to\rank(R)$ is an order isomorphism.
\end{proof}

\begin{corollary}[Cantor]\label{t:Cantor-trichotomy} For set-like well-orders $R,P$ one of the following conditions holds:
\begin{enumerate}
\item[\textup{1)}]  $R$ and $P$ are isomorphic;
\item[\textup{2)}]  $R$ is  isomorphic to a unique initial interval of $P$;
\item[\textup{3)}]  $P$ is isomorphic to a unique initial interval of $R$.
\end{enumerate}
\end{corollary}

\begin{proof} The uniqueness of the initial intervals in the statements (2),(3) follows from Proposition~\ref{p:initial}. It remains to prove the existence of order-isomorphisms in one of  the statements (1)--(3).
\smallskip

By Theorem~\ref{t:wOrd}, the functions $\rank_R{\restriction}_{\dom[R^\pm]}:\dom[R^\pm]\to\rank(R)$ and $\rank_P{\restriction}_{\dom[P^\pm]}:\dom[P^\pm]\to\rank(P)$ are order isomorphisms. Each of the ranks $\rank(R)$, $\rank(P)$ is either $\Ord$ or some ordinal. Consequently, three cases are possible.

1) $\rank(R)=\rank(P)$. In this case the well-orders $R,P$ are isomorphic.

2) $\rank(R)\in \rank(P)$. In this case the ordinal $\rank(R)$ is an initial interval of $\rank(P)$ and the well-order $R$ is isomorphic to an initial interval of the well-order $P$.

3) $\rank(P)\in\rank(R)$.    In this case the ordinal $\rank(P)$ is an initial interval of $\rank(R)$ and the well-order $P$ is isomorphic to an initial interval of the well-order $R$.
\end{proof}

Let $\mathbf{WO}$ be the class of well-orders which are sets. The function $\rank:\mathbf{WO}\to\Ord$ assigns to each well-order $R\in\mathbf{WO}$ the ordinal $\rank(R)$, called \index{order type}{\em the order type} of $R$. For any ordinal $\alpha$ the preimage $\rank^{-1}[\{\alpha\}]$ is the equivalence class of all well-orders that are isomorphic  to $\alpha$. Initially ordinals were thought as such equivalence classes (till John von Neumann discovered the notion of an ordinal we use nowadays).

\begin{exercise} Prove that the class $\mathbf{WO}$ and the function $\rank:\mathbf{WO}\to\Ord$ exist.
\end{exercise}


The following theorem was proved by Friedrich Hartogs in 1915.

\begin{theorem}[Hartogs]\label{t:Hartogs} For any set $x$ there exists an ordinal $\alpha$ admitting no injective function $f:\alpha\to x$.
\end{theorem}

\begin{proof} Let $WO(x)$ be the set whose elements are well-orders $w$ with $\dom[w^\pm]\subseteq x$. Since $WO(x)\subseteq\mathcal P(x\times x)$, the class $WO(x)$ is a set by the Axiom of Power-set and Exercises~\ref{t:prodsets}, \ref{ex:subclass}.
 Let $\rank:WO(x)\to\Ord$ be the function assigning to each well-order $w\in WO(x)$ its rank $\rank(w)=\rank_w[\UU]$. By the Axiom of Replacement, the image $\rank[WO(x)]\subseteq \Ord$ is a set and so is its union $\alpha=\bigcup(\rank[WO(x)]+1)$, which is an ordinal by Theorem~\ref{t:Ord}(5).

We claim that the ordinal $\alpha$ admits no injective function $f:\alpha\to x$. In the opposite case, $w=\{\langle f(\gamma), f(\beta)\rangle:\gamma\in\beta\in\alpha\}$ would be a well-order in the set $WO(x)$ such that $\rank(w)=\alpha\in\alpha$, which is forbidden by the definition of an ordinal.
\end{proof}

\section{Foundation}\label{s:vN}

In this section we construct a proper class $\IV$ called the \index{von Neumann universe $\mathbf V$}  \index{universe!von Neumann $\mathbf V$}\index{class!{{\bf V}}}\index{{{\bf V}}}{\em von Neumann universe}. This class is the smallest class that contains all ordinals and 
is closed under the operations of taking power-set and union. The restriction $\E{\restriction}\IV$ of the membership relation to this class is well-founded and the Axiom of Foundation is equivalent to the equality $\UU=\IV$. The class $\IV$ is defined as the union of the von Neumann cumulative hierarchy.

\begin{definition}[Cumulative hierarchy of von Neumann] The \index{cumulative hyerarchy of von Neumann}\index{von Neumann cumulative hierarchy}{\em cumulative hierarchy of von Neumann}  is the transfinite sequence of sets $(V_\alpha)_{\alpha\in\On}$, defined by the recursive formula
$$V_\alpha\defeq\textstyle{\bigcup}\{\mathcal P(V_\gamma):\gamma\in\alpha\},\quad\alpha\in\Ord.$$
The class $\VV\defeq\bigcup_{\alpha\in\On}V_\alpha$ is called the {\em von Neumann universe}.
\end{definition}

\begin{theorem}\label{t:vN} The von Neumann cumulative hierarchy $(V_\alpha)_{\alpha\in\On}$ is basic  and has the following properties:
\begin{enumerate}
\item[\textup{1)}] $\{\alpha\}\cup V_\alpha\subseteq V_{\alpha+1}=\mathcal P(V_\alpha)$ for every ordinal $\alpha$.
\item[\textup{2)}] $V_\alpha=\bigcup\{V_\gamma:\gamma\in\alpha\}$ for any limit ordinal $\alpha$.
\item[\textup{3)}]  For every ordinal $\alpha$ the set $V_\alpha$ is transitive.
\item[\textup{4)}] The class $\IV=\bigcup\{V_\alpha:\alpha\in\On\}$ is transitive, contains all ordinals and hence is proper.
\item[\textup{5)}] The relation $\E{\restriction}\IV$ is set-like, well-founded, irreflexive, and $\rank_{\E{\restriction}\IV}[V_\alpha]=\alpha$.
\item[\textup{6)}] Each subset of $\IV$ is an element of $\IV$, which can be written as $\mathcal P(\IV)\subseteq\IV$.
\item[\textup{7)}] $\IV$ is a subclass of any class $\mathbf X$ such that  $\mathcal P(\mathbf X)\subseteq\mathbf X$.
\end{enumerate}
\end{theorem}

\begin{proof} 0. The existence and definability of the function $$V_*:\Ord\to \UU,\quad V_*:\alpha\mapsto V_\alpha,$$determining the von Neumann cumulative hierarchy follows from the Recursion Theorem~\ref{t:Recursion} applied to the set-like well-order $\E{\restriction}\Ord$ and the function $$F:\Ord\times\UU\to\UU,\;\;F:\langle\alpha,y\rangle\mapsto \textstyle{\bigcup}\{\mathcal P(z):z\in \rng[y]\}.$$
The definability of $\E{\restriction}\Ord$ and $F$ follows from Theorem~\ref{t:class}.

The existence and definability of the function $V_*$ also implies the existence and the definability of the ``inverse function" $$\mathit\Lambda:\VV\to\Ord,\;\mathit\Lambda:x\mapsto\min\{\alpha\in\Ord:x\in V_\alpha\},$$
where $\VV=\bigcup_{\alpha\in\Ord}V_\alpha=\bigcup V_*[\Ord]$. The function $\mathit\Lambda$ exists and is basic since
$$\mathit\Lambda=\{\langle x,\alpha\rangle\in\VV\times\Ord:x\in V_\alpha\;\wedge\;\forall\gamma\in\alpha\;(x\notin V_\gamma)\}.$$

1. For any ordinal $\alpha$, the definition of $V_\alpha=\bigcup\{\mathcal P(V_\gamma):\gamma\in\alpha\}$ implies that $V_\alpha\subseteq V_\beta$ and hence $\mathcal P(V_\alpha)\subseteq \mathcal P(V_\beta)$ for any ordinals $\alpha\le\beta$. Then 
$$V_{\alpha+1}=\textstyle{\bigcup}\{\mathcal P(V_\gamma):\gamma\le\alpha\}=\mathcal P(V_\alpha).$$

Next we show that $\forall\alpha\in\On\;(\alpha\in V_{\alpha+1})$. Assuming that this is not true, we can use the well-foundedness of the order $\E{\restriction}\Ord$ and find an ordinal $\alpha$ such that $\alpha\notin V_{\alpha+1}$ but $\forall \gamma\in \alpha\;(\gamma\in V_{\gamma+1}\subseteq V_\alpha)$. Then $\alpha=\{\gamma:\gamma\in\alpha\}\subseteq V_\alpha$ and hence $\alpha\in \mathcal P(V_\alpha)=V_{\alpha+1}$, which contradicts the choice of $\alpha$. 
\smallskip

2. If $\alpha$ is a limit ordinal, then
$$V_\alpha=\textstyle{\bigcup}\{\mathcal P(V_\gamma):\gamma\in\alpha\}=\textstyle{\bigcup}\{V_{\gamma+1}:\gamma\in\alpha\}=\textstyle{\bigcup}\{V_{\gamma}:\gamma\in\alpha\}.$$
The above equality implies that for every $x\in \IV$ the ordinal $\mathit\Lambda(x)$ is a successor ordinal.
\smallskip

3. For every ordinal $\alpha$ and every set $x\in V_\alpha$, we can find an ordinal $\beta\in\alpha$ such that $x\in V_{\beta+1}=\mathcal P(V_\beta)$. Then $x\subseteq V_\beta\subseteq V_\alpha$.

4. By the statement (1), the class $\IV$ contains all ordinals and hence is a proper class according to Theorem~\ref{t:Ord}(6) and Exercise~\ref{ex:subclass}. The transitivity of the sets $V_\alpha$, $\alpha\in\Ord$, implies the transitivity of the union $\VV=\bigcup_{\alpha\in\Ord}V_\alpha$.
\smallskip

5. It is clear that the relation $\E{\restriction}\IV$ is set-like. To see that it is well-founded and irreflexive, take any nonempty subclass $X\subseteq\IV$. We should find an element $x\in X$ such that $x\cap X=\emptyset$. Since the relation $\E{\restriction}\Ord$ is well-founded, the nonempty subclass $\mathit \Lambda[X]$ of $\Ord$ contains the smallest ordinal, denoted by $\alpha$. For this ordinal $\alpha$ we have $X\cap V_\alpha\ne\emptyset$ but $X\cap V_\gamma=\emptyset$ for all $\gamma\in\alpha$. Take any set $x\in X\cap V_\alpha$. 
Since $V_0=\bigcup\{\mathcal P(V_\gamma):\gamma\in\emptyset\}=\bigcup\emptyset=\emptyset$, the ordinal $\alpha$ is not empty.

If $\alpha$ is a limit ordinal, then the statement (2) implies that $x\in V_\gamma$ for some $\gamma\in\alpha$, which contradicts the minimality of $\alpha$. Therefore, $\alpha=\beta+1$ for some ordinal $\beta$.  Then $x\in V_\alpha=V_{\beta+1}=\mathcal P(V_\beta)$ and hence $x\subseteq V_\beta$. The choice of $\alpha$ guarantees that $x\cap X\subseteq V_\beta\cap X=\emptyset$.

Since the relation $\E{\restriction}\IV$ is set-like and well-founded it has a well-defined function $$\rank_{\E{\restriction}\IV}:\IV\to\Ord,\quad\rank_{\E{\restriction}\IV}:x\mapsto\sup\{\rank_{\E{\restriction}\IV}(y)+1:y\in x\}.$$ The embedding $\rank_{\E{\restriction}\IV}[V_\alpha]\subseteq\alpha$ will be proved by transfinite induction. For $\alpha=0$ we have $\rank_{\E{\restriction}\IV}[V_0]=\rank_{\E{\restriction}\IV}[\emptyset]=\emptyset=0$. 
Assume that for some ordinal $\alpha$ and all its elements $\beta\in\alpha$ the embedding  $\rank_{\E{\restriction}\IV}[V_\beta]\subseteq\beta$ has been proved.
If $\alpha$ is a limit ordinal, then 
$$\rank_{\E{\restriction}\IV}[V_\alpha]=\rank_{\E{\restriction}\IV}[\textstyle\bigcup_{\beta\in\alpha}V_\beta]=\bigcup_{\beta\in\alpha}\rank_{\E{\restriction}\IV}[V_\beta]\subseteq\bigcup_{\beta\in\alpha}\beta=\alpha.$$
If $\alpha$ is a successor ordinal, then $\alpha=\beta+1$ for some ordinal $\beta\in\alpha$ and then for every $x\in V_{\alpha}=\mathcal P(V_\beta)$ and $y\in x$ we have $y\in x\subseteq V_\beta$ and hence $\rank_{\E{\restriction}\IV}(y)\in \rank_{\E{\restriction}\IV}[V_\beta]\subseteq\beta$. So, $\rank_{\E{\restriction}\IV}(y)\in\beta$ and $\rank_{\E{\restriction}\IV}(y)+1\le\beta$. Then
$$\rank_{\E{\restriction}\IV}(x)=\sup\{\rank_{\E{\restriction}\IV}(y)+1:y\in x\}\le \beta\in\alpha$$and hence $\rank_{\E{\restriction}\IV}(x)\in\alpha$ and $\rank_{\E{\restriction}\IV}[V_\alpha]\subseteq \alpha$. On the other hand, the inclusion $\alpha\subseteq V_\alpha$ implies $\alpha\subseteq\rank_{\E{\restriction}\IV}[V_\alpha]\subseteq\alpha$.
\smallskip

6. Assume that $x$ is a subset of $\IV$. By the Axiom of Replacement, the image $\mathit \Lambda[x]\subset\Ord$ is a set and hence $\alpha=\bigcup\mathit\Lambda[x]=\bigcup\{\mathit\Lambda(y):y\in x\}$ is an ordinal according to  Theorem~\ref{t:ord}(5). Then for every $y\in x$ we have $\mathit\Lambda(y)\subseteq\bigcup\Lambda[x]=\alpha$ and thus $y\in V_{\mathit\Lambda(y)}\subseteq V_\alpha$ and $x\in V_{\alpha+1}\subseteq \VV$.
\smallskip

7. Assume that $\mathbf X$ is a class such that $\mathcal P(\mathbf X)\subseteq\mathbf X$. We claim that for every $\alpha\in\Ord$ the set $V_\alpha$ is a subset of $\mathbf X$. It is easy to show that the class $A=\{\alpha\in\Ord: V_\alpha\subseteq \mathbf X\}$ exists. If $A=\Ord$, then $\VV=\bigcup_{\alpha\in\On}V_\alpha\subseteq\mathbf X$ and we are done. So assume that $A\ne\Ord$ and take the smallest ordinal $\alpha$ in the subclass $\Ord\setminus A$. Such an ordinal exists since the relation $\E{\restriction}\Ord$ is well-founded. Then $\alpha\subseteq A$ and hence $V_\gamma\subseteq \mathbf X$ for all $\gamma\in\alpha$. If $\alpha$ is a limit ordinal, then $V_\alpha=\bigcup_{\gamma\in\alpha}V_\gamma\subseteq \mathbf X$ and hence $\alpha\in A$, which contradicts the choice of $A$. This contradiction shows that $\alpha$ is not limit and hence $\alpha=\gamma+1$ for some ordinal $\gamma$. The choice of $\alpha$ guarantees that $V_\gamma\subseteq \mathbf X$. Then $V_\alpha=\mathcal P(V_\gamma)\subseteq \mathcal P(\mathbf X)\subseteq \mathbf X$. But this contradicts the choice of $\alpha$. This contradiction shows that $A=\Ord$ and $\VV=\bigcup_{\alpha\in\Ord}V_\alpha\subseteq\mathbf X$.
\end{proof}


\begin{remark} Theorem~\ref{t:vN}(6,7) implies that the von Neumann class $\VV$ is the smallest class $\mathbf X$ such that $\mathcal P(\mathbf X)\subseteq \mathbf X$.
\end{remark} 

A set $x$ is called \index{set!hereditarily finite}{\em hereditarily finite} if its transitive closure is finite. 

\begin{exercise} Prove that the set $V_\w$ coincides with the class of all hereditarily finite sets in $\VV$.
\end{exercise}

\begin{theorem}\label{t:AF} The Axiom of Foundation is equivalent to the equality $\UU=\IV$.
\end{theorem}

\begin{proof} If $\UU=\IV$ then the relation $\E=\E{\restriction}\IV$ is well-founded by Theorem~\ref{t:vN}(5) and hence the Axiom of Foundation holds.

Now assumming the Axiom of Foundation, we shall prove that $\UU=\IV$. To derive a contradiction, assume that $\UU\setminus \IV$ is not empty and fix any set $a\in\UU\setminus \IV$.  By Theorem~\ref{t:vN}(6), the set $a\setminus \IV$ is not empty. By Theorem~\ref{t:TC}, the set $a\setminus\IV$ is contained in some transitive set $t$. By the Axiom of Foundation, the set $t\setminus\IV$ contains an element $u\in t\setminus \IV$ such that $u\cap (t\setminus \IV)=\emptyset$. By the transitivity of the set $t$, we have $u\subseteq t\setminus(t\setminus\IV)=t\cap\IV\subset\IV$. 
Applying Theorem~\ref{t:vN}(6), we conclude that $u\in \IV$, which contradicts the choice of $u$.
\end{proof}

A set $x$ is called \index{set!hereditarily transitive}{\em hereditarily transitive} if $x$ is transitive and every element $y\in x$ is a transitive set.

\begin{theorem}\label{t:ordinal=ht} A set $x$ is an ordinal if and only if $x\in \mathbf V$ and $x$ is hereditarily transitive. Consequently, $\textstyle\On=\{x\in \mathbf V:\forall y\in x\cup\{x\}\;(\bigcup y\subseteq y)\}.$
\end{theorem}

\begin{proof} If $x$ is an ordinal, then $x\in \mathbf V$ by Theorem~\ref{t:vN} and $x$ is hereditarily transitive by Definition~\ref{d:ordinals} and Theorem~\ref{t:ord}.

Now assume that a set $x\in \mathbf V$ is hereditarily transitive. We claim that $x\subseteq\On$. In the opposite case, the set $x\setminus \On\in\mathbf V$ is not empty and by the well-foundedness of the relation $\E{\restriction}\mathbf V$, there exists an element $y\in x\setminus\On$ such that $y\cap (x\setminus\On)=\emptyset$. The transitivity of $x$ implies that $y\subseteq x$ and hence $y\subseteq x\setminus(x\setminus \On)=x\cap\On$. Being a transitive subset of $\On$, the set $y$ is an ordinal by Definition~\ref{d:ordinals} and Theorem~\ref{t:Ord}(2). But the inclusion $y\in\On$ contradicts the choice of $y\in x\setminus \On$. This contradiction shows that $x\subseteq\On$ and hence $x$ is an ordinal by 
Definition~\ref{d:ordinals} and Theorem~\ref{t:Ord}(2).
\end{proof}

\begin{exercise}\label{ex:vNG} Show that for every ordinal $\alpha$ and sets $x,y\in V_\alpha$ we have
\begin{enumerate}
\item $x\setminus y\in V_\alpha$;
\item $\bigcup x\in V_\alpha$;
\item $\dom[x]\in V_\alpha$;
\item $x^{-1}\in V_\alpha$;
\item $\{x,y\}\in V_{\alpha+1}$;
\item $\langle x,y\rangle\in V_{\alpha+2}$.
\item $x\times y\in V_{\alpha+2}$.
\item $\E\cap(x\times y)\in V_{\alpha+2}$;
\item $x^\circlearrowright\in V_{\alpha+2}$.
\end{enumerate}
\end{exercise}


\part{Constructibility}\label{part:constructibility}

\rightline{\em But above all I wish to designate the following}

\rightline{\em as the most important among the numerous questions}

\rightline{\em which can be asked with regard to the axioms:}

\rightline{\em To prove that they are not contradictory, that is,}

\rightline{\em that a definite number of logical steps based upon them}

\rightline{\em can never lead to contradictory results.}
\smallskip

\rightline{David Hilbert}
\bigskip

In this chapter we study the class $\mathbf L$ of constructible sets that was introduced by G\"odel for establishing the consistency of the Axiom of Choice and Continuum Hypothesis. We shall apply the class $\mathbf L$ to prove the consistency of the Axiom of Foundation and the Global Axiom of Choice. The class $\mathbf L$ is defined with the help of G\"odel's operations, so first we establish some properties of G\"odel's operations and their compositions.

\section{G\"odel's operations over sets}

\begin{definition}  By \index{G\"odel's operations}{\em G\"odel's operations} we understand the following nine functions:
\begin{itemize}
\item[\textup{(0)}] $\dG_0:\UU\times\UU\to\UU$, $\dG_0:\langle x,y\rangle \mapsto x$
\item[\textup{(1)}] $\mathsf{G}_1:\UU\times\UU\to\UU$, $\mathsf{G}_1:\langle x,y\rangle\mapsto x\setminus y=\{u\in x:u\notin y\}$;
\item[\textup{(2)}] $\dG_2:\UU\times\UU\to\UU$, $\dG_2:\langle x,y\rangle\mapsto x\times y=\{\langle u,v\rangle:u\in x\;\wedge\;v\in y\}$;
\item[\textup{(3)}] $\dG_3:\UU\times\UU\to\UU$, $\dG_3:\langle x,y\rangle\mapsto x^{-1}=\{\langle v,u\rangle:\langle u,v\rangle\in x\}$;
\item[\textup{(4)}] $\dG_4:\UU\times\UU\to\UU$, $\dG_4:\langle x,y\rangle\mapsto x^\circlearrowright=\{\langle w,u,v\rangle:\langle u,v,w\rangle\in x\};$
\item[\textup{(5)}] $\dG_5:\UU\times \UU\to\UU$, $\dG_5:\langle x,y\rangle\mapsto x\cap\E$;
\item[\textup{(6)}] $\dG_6:\UU\times\UU\to\UU$, $\dG_6:\langle x,y\rangle\mapsto \dom[x]=\{u:\exists v\;(\langle u,v\rangle\in x)\}$;
\item[\textup{(7)}] $\dG_7:\UU\times\UU\to\UU$, $\dG_7:\langle x,y\rangle\mapsto \bigcup x=\{z:\exists u\in x\;(z\in u)\}$;
\item[\textup{(8)}] $\dG_8:\UU\times\UU\to\UU$, $\dG_8:\langle x,y\rangle\mapsto \{x,y\}$.
\end{itemize}
\end{definition}

\begin{exercise} Prove that the functions $\dG_0$--$\dG_8$ exist and are basic  classes.
\end{exercise}

The G\"odel operations  $\dG_0$ and $\dG_3$--$\dG_7$ do not depend on the second variable. Nonetheless we have written them as functions of two variables for uniform treatment of all G\"odel's operations. To simplify expressions involving the G\"odel's operations, it will be convenient to consider the unary operations
$$\dot\Go_i:\UU\to\UU,\quad \dot\Go_i:x\mapsto \dG_i(x,x),\quad\mbox{where $i\in 9$}.$$

\begin{exercise} Show that for every sets $x,y$,$$\dot\Go_1(x)=\emptyset,\quad \dot\Go_2(x)=x\times x,\quad \dot\Go_8(x)=\{x\},\quad\mbox{and}\quad\dot\Go_i(x)=\dG_i(x,y)\mbox{ for $i\in\{0,3,4,5,6,7\}$}.$$
\end{exercise}




\begin{definition} The \index{G\"odel's extension}{\em G\"odel's extension} is the function $\Go:\UU\to\UU$ assigning to every set $x$ the set
$$
\Go(x)=\bigcup_{i=0}^8\dG_i[x\times x]=\{u, u\setminus v, u\times v, u^{-1}, u^\circlearrowright,\dom[u], {\textstyle\bigcup u},\{u,v\}:u,v\in x\}.
$$
\end{definition}
\noindent So, the set $\Go(x)$ consists of  the results of application of G\"odel's operations to {\em elements} of $x$.

\begin{exercise} Prove that the  function $\Go$ is a basic class.
\end{exercise}

Iterating the function $\Go$, we obtain the function sequence $(\Go^{\circ n})_{n\in\w}$ such that $\Go^{\circ 0}=\Id$ and $\Go^{\circ(n+1)}=\Go\circ \Go^{\circ n}$ for every $n\in\w$. The function sequence $(\Go^{\circ n})_{n\in\w}$ exists and is basic  by Theorem~\ref{t:iterate}. Finally, consider the function
$\Go^{\circ\w}:\UU\to\UU$ assigning to each set $x$, the set $\Go^{\circ\w}(x)\defeq\bigcup_{n\in\w}\Go^{\circ n}(x)$, called the \index{G\"odel's hull}{\em G\"odel's hull} of $x$. 

\begin{exercise} Prove that the G\"odel's extension is monotone in the sense that $$ \Go(x)\subseteq\Go(y)$$ for any sets $x\subseteq y$.
\end{exercise}

\begin{exercise} Prove that $$x\subseteq \Go^{\circ n}(x)\subseteq \Go^{\circ(n+1)}(x)\subseteq \Go^{\circ\w}(x)$$ for every set $x$ and natural number $n$.
\end{exercise} 

\begin{Exercise} Prove that for any transitive set $x$, its G\"odel's hull $\Go^{\circ\w}(x)$ is a transitive set.
\smallskip

\noindent{\em Hint:} This is not easy, see Theorem 27.9 in {\tt arxiv.org/pdf/2006.01613v2.pdf}.
\end{Exercise}

\begin{exercise} Find a set $x$ whose G\"odel's hull $\Go^{\circ\w}(x)$ is not transitive.
\end{exercise}

Many useful operations on sets can be expressed via G\"odel's operations $\dG_i$.

\begin{exercise}\label{ex:God-op} Prove that  for any sets $x,y$ we have
\begin{enumerate} 
\item $x\cap y=x\setminus(x\setminus y)=\dG_1(x,\dG_1(x,y))\in \Go^{\circ 2}(\{x,y\})$;
\item $x\cup y=\bigcup\{x,y\}=\dot\Go_7(\dG_8(x,y))\in \Go^{\circ2}(\{x,y\})$;
\item $x\setminus\E=x\setminus(x\cap \E)=\dG_1(x,\dot\Go_5(x))\in\Go^{\circ2}(\{x\})$;
\item $\{x,y\}=\bigcup\bigcup(\{x\}\times \{y\})=\dot \Go_7(\dot\Go_7(\ddot \Go_2(\dot\Go_8(x),\dot\Go_8(y))))$;
\end{enumerate}
\end{exercise}

\begin{remark} Exercise~\ref{ex:God-op}(4) shows that the G\"odel's operation $\ddot G_8$ is expressible via its one-variable version $\dot\Go_8$ and the G\"odel's operations $\ddot \Go_2$ and $\dot\Go_7$.
\end{remark}

\begin{proposition}\label{p:vOn} For any set $v$ the set 
$$\textstyle\gamma\defeq\{x\in v:\forall y\in x\cup\{x\}\;(\bigcup y\subseteq y)\}$$ is an element of the set $\Go^{\circ20}(\{v\})\subseteq\Go^{\circ\w}(\{v\})$. If $v\in\mathbf V$, then $\gamma=v\cap\On$.
\end{proposition}

\begin{proof} 
It follows that
\begin{multline*}
v\setminus\gamma=\textstyle\{x\in v:\bigcup x\not\subseteq x\}\cup\{x\in v:\exists y\;(y\in x\;\wedge\;\bigcup y\not\subseteq y)\}=\\
\{x\in v:\exists y\exists z\; (y\in x\wedge z\in y\wedge z\notin x)\}\cup\{x\in v:\exists y\exists z\exists s\; (y\in x\wedge z\in y\wedge s\in z\wedge s\notin y)\}.
\end{multline*}
Observe that for any sets $x\in v$, $y\in x$, $z\in y$ and $s\in z$, we have $$\textstyle y\in\bigcup v=\dot\Go_7(v),\quad z\in\bigcup\bigcup v=\bigcup^{\circ2}v=\dot\Go^{\circ 2}_7(v),\quad s\in\bigcup\bigcup\bigcup v=\bigcup^{\circ3}v=\dot\Go^{\circ 3}_7(v).$$ 
Then 
$v\setminus \gamma=\dom^{\circ2}[T_1\cap T_2\cap T_3]\cup\dom^{\circ 3}[Q_1\cap Q_2\cap Q_3\cap Q_4]$,
where
$$
\begin{aligned}
T_1&=\textstyle\{\langle\langle x,y\rangle, z\rangle\in (v\times \bigcup v)\times\bigcup^{\circ 2} v:y\in x\}=[(\bigcup v\times v)\cap\E]^{-1}\times\bigcup^{\circ 2}v
\in\Go^{\circ5}(\{v\});\\
T_2&=\textstyle\{\langle\langle x,y\rangle, z\rangle\in (v\times \bigcup v)\times\bigcup^{\circ2} v:z\in y\}=[[\mathbf E\cap(\bigcup^{\circ2} v\times \bigcup v)]^{-1}\times  v]^\circlearrowright
\in \Go^{\circ 7}(\{v\});\\
T_3&=\textstyle\{\langle\langle x,y\rangle, z\rangle\in (v\times \bigcup v)\times\bigcup^{\circ2} v:z\notin x\}=\textstyle[((\bigcup^{\circ2} v\times v)\setminus \mathbf E)\times  \bigcup v]^\circlearrowleft
\in\Go^{\circ7}(\{v\})
\end{aligned}
$$
and
$$
\begin{aligned}
Q_1&=\textstyle\{\langle \langle\langle x,y\rangle,z\rangle,s\rangle\in ((v\times\bigcup v)\times\bigcup^{\circ2} v)\times\bigcup^{\circ3} v:y\in x\}\\
&=\textstyle([(\bigcup v\times v)\cap\E]^{-1}\times \bigcup^{\circ2} v)\times\bigcup^{\circ3} v\in 
\Go^{\circ6}(\{v\}),\\
Q_2&=\textstyle\{\langle \langle\langle x,y\rangle,z\rangle,s\rangle\in ((v\times\bigcup v)\times\bigcup^{\circ2} v)\times\bigcup^{\circ3} v:z\in y\}\\
&=\textstyle[[(\bigcup^{\circ2} v\times \bigcup v)\cap\E]^{-1}\times v]^\circlearrowright\times\bigcup^{\circ3} v
\in\Go^{\circ8}(\{v\}),\\
Q_3&=\textstyle\{\langle \langle\langle x,y\rangle,z\rangle,s\rangle\in ((v\times\bigcup v)\times\bigcup^{\circ2} v)\times\bigcup^{\circ3} v:s\in z\}\\
&=\textstyle[\{\langle \langle z,s\rangle,\langle x,y\rangle\rangle\in (\bigcup^{\circ2}v\times\bigcup^{\circ3} v)\times(v\times \bigcup v):s\in z\}]^\circlearrowright\\
&=\textstyle[[(\bigcup^{\circ3}v\times\bigcup^{\circ2}v)\cap\E]^{-1}\times (v\times \bigcup v)]^\circlearrowright
\in\Go^{\circ8}(\{v\})\\
Q_4&=\textstyle\{\langle \langle\langle x,y\rangle,z\rangle,s\rangle\in ((v\times\bigcup v)\times\bigcup^{\circ2} v)\times\bigcup^{\circ3} v:s\notin y\}\\
&=[\textstyle\{\langle \langle z,s\rangle,\langle x,y\rangle\rangle\in (\bigcup^{\circ 2}v\times\bigcup^{\circ 3}v)\times(v\times \bigcup v):s\notin y\}]^\circlearrowright\\
&=[[\textstyle\{\langle \langle x,y\rangle,\langle z,s\rangle\rangle\in (v\times\bigcup v)\times(\bigcup^{\circ2} v\times\bigcup^{\circ3} v):s\notin y\}]^\circlearrowright]^{-1}\\
&=[[[\textstyle\{\langle \langle y,\langle z,s\rangle\rangle,x\rangle\in (\bigcup v\times(\bigcup^{\circ2} v\times\bigcup^{\circ3} v))\times v:s\notin y\}]^\circlearrowright]^{-1}]^\circlearrowright\\
&=[[[\textstyle\{\langle y,\langle z,s\rangle\rangle\in \bigcup v\times(\bigcup^{\circ2} v\times\bigcup^{\circ3} v):s\notin y\}\times v]^\circlearrowright]^{-1}]^\circlearrowright\\
&=[[[[\textstyle\{\langle s,y\rangle,z\rangle\in  (\bigcup^{\circ3}v\times\bigcup v)\times\bigcup^{\circ2} v:s\notin y\}]^\circlearrowleft\times v]^\circlearrowright]^{-1}]^\circlearrowleft\\
&=[[[[\textstyle((\bigcup^{\circ3}v\times\bigcup v)\setminus\E)\times\bigcup^{\circ2} v]^\circlearrowleft\times v]^\circlearrowright]^{-1}]^\circlearrowleft
\in\Go^{\circ12}(\{v\}).
\end{aligned}
$$

Applying Exercise~\ref{ex:God-op}(1),  we conclude that 
$T_1\cap T_2\in\Go^{\circ 9}(\{v\})$ and
$T_1\cap T_2\cap T_3\in\Go^{\circ11}(\{v\})$. 
By analogy, we can deduce from $Q_1,Q_2,Q_3\in \Go^{\circ8}(\{v\})$ and $Q_4\in\Go^{\circ12}(\{v\})$ that 
$Q_1\cap Q_2\cap Q_3\in\Go^{\circ12}(\{v\})$ and $Q_1\cap Q_2\cap Q_3\cap Q_4\in \Go^{\circ 14}(\{v\})$.
Then 
$$v\setminus \gamma=\dom^{\circ2}[T_1\cap T_2\cap T_3]\cup\dom^{\circ3}[Q_1\cap Q_2\cap Q_3\cap Q_4]\in \Go^{\circ19}(\{v\})$$
and finally,
$\gamma=v\setminus(v\setminus\gamma)\in\Go^{\circ20}(\{v\})$.

If $v\in\mathbf V$, then by Theorem~\ref{t:ordinal=ht}, $v\cap\On=\{x\in v:\forall y\in x\cup\{x\}\;(\bigcup y\subseteq y)\}=\gamma$. 
\end{proof}
\medskip




\section{G\"odel's Constructible Universe}\label{s:L}
\smallskip

\rightline{\em The universe is almost like a huge magic trick}

\rightline{\em  and scientists are trying to figure out}

\rightline{\em  how it does what it does.}
\smallskip

\rightline{Martin Gardner}
\bigskip

Consider an increasing transfinite sequence of sets $(L_\alpha)_{\alpha\in\On}$ defined by the recursive formula 
$$
L_\alpha=\bigcup_{\beta\in\alpha}\mathcal P(L_\beta)\cap\Go^{\circ\w}(L_\beta\cup\{L_\beta\})\quad\mbox{for}\quad \alpha\in\On.$$

\begin{definition} The class $$\mathbf L=\bigcup_{\alpha\in\On}L_\alpha$$ is called \index{G\"odel's constructible universe}\index{constructible universe of G\"odel}\index{class!{{\bf L}}}\index{{{\bf L}}}{\em the G\"odel's constructible universe}. Its elements are called \index{constructible set}\index{set!constructible}{\em constructible sets}.
\end{definition}

\begin{exercise} Prove that the classes $(L_\alpha)_{\alpha\in\On}$ and $\mathbf L$ are basic.
\smallskip

\noindent{\em Hint:} Apply Theorem~\ref{t:class}.
\end{exercise}

The following two propositions list some essential properties of the transfinite sequence $(L_\alpha)_{\alpha\in\On}$.

\begin{proposition}\label{p:L-seq} For every ordinal $\alpha$,
\begin{enumerate}
\item[\textup{1)}]  the set $L_\alpha$ is transitive;
\item[\textup{2)}] $L_{\alpha+1}=\mathcal P(L_\alpha)\cap\Go^{\circ\w}(L_\alpha\cup\{L_\alpha\})$;
\item[\textup{3)}] $L_\alpha\subseteq V_\alpha$;
\item[\textup{4)}] $\alpha\in L_{\alpha+1}$.
\end{enumerate}
\end{proposition}

\begin{proof} 1. The transitivity of the set $L_\alpha$ will be proved by induction. Assume that for some ordinal $\alpha$ all sets $L_\beta$, $\beta\in\alpha$, are transitive. To show that the set $L_\alpha$ is transitive, take any $x\in L_\alpha=\bigcup_{\beta\in\alpha}\mathcal P(L_\beta)\cap \Go^{\circ\w}(L_\beta\cup\{L_\beta\})$ and find an ordinal $\beta\in\alpha$ such that $x\in\mathcal P(L_\beta)\cap\Go^{\circ\w}(L_\beta\cup\{L_\beta\})\subseteq \mathcal P(L_\beta)$. Then $x\subseteq L_\beta\subseteq L_\alpha$.
\smallskip

2. The transitivity of the set $L_\alpha$ implies $L_\alpha\subseteq\mathcal P(L_\alpha)$. Also for every $x\in L_\alpha$ we have $x=\dG_0(x,x)\in \Go^{\circ\w}(L_\alpha)\subseteq \Go^{\circ\w}(L_\alpha\cup\{L_\alpha\})$ and hence $L_\alpha\subseteq\mathcal P(L_\alpha)\cap\Go^{\circ\w}(L_\alpha\cup\{L_\alpha\})$. Then $$L_{\alpha+1}=L_\alpha\cup(\mathcal P(L_\alpha)\cap\Go^{\circ\w}(L_\alpha\cup\{L_\alpha\}))=\mathcal P(L_\alpha)\cap\Go^{\circ\w}(L_\alpha\cup\{L_\alpha\}).$$
\smallskip

3. The inclusion $L_\alpha\subseteq V_\alpha$ will be proved by induction on $\alpha$. Assume that for some ordinal $\alpha$ we know that $L_\beta\subseteq V_\beta$ for all $\beta\in\alpha$. If $\alpha=0$, then $L_0=\emptyset=V_0$. If $\alpha$ is a limit ordinal, then $L_\alpha=\bigcup_{\beta\in\alpha}L_\beta\subseteq \bigcup_{\beta\in\alpha}V_\beta=V_\alpha$. If $\alpha=\beta+1$ for some ordinal $\beta$, then by the transitivity of the set $L_\beta$, 
$$L_\alpha=L_{\beta+1}\subseteq  \mathcal P(L_\beta)\subseteq \mathcal P(V_\beta)=V_{\beta+1}=V_\alpha.$$
\smallskip

4. The inclusion $\alpha\in L_{\alpha+1}$ will be proved by induction on $\alpha$. Assume that for some ordinal $\alpha$ we know that $\beta\in L_{\beta+1}$ for all ordinals $\beta<\alpha$. If  $\alpha=0$, then $\alpha=\emptyset=\dG_0(\emptyset,\emptyset)\in \mathcal P(\emptyset)\cap \Go^{\circ\w}(\{\emptyset\})=L_1=L_{\alpha+1}$. If $\alpha=\beta+1$ is a successor ordinal, then the inductive hypothesis ensures that $\beta\in L_{\beta+1}=L_\alpha$ and hence $\alpha=\beta\cup\{\beta\}\subseteq L_\alpha$ and $\alpha\in\mathcal P(L_\alpha)$. Then 
$$\alpha=\beta\cup\{\beta\}=\textstyle{\bigcup}\{\beta,\{\beta\}\}=\dot\Go_7(\dG_8(\beta,\dot\Go_8(\beta)))\in \mathcal P(L_\alpha)\cap \Go^{\circ\w}(L_\alpha)\subseteq L_{\alpha+1}.$$ Finally, assume that $\alpha$ is a limit ordinal. Then $\alpha=\bigcup_{\beta\in\alpha}\beta\subseteq\bigcup_{\beta\in\alpha}L_\beta=L_\alpha$. By Proposition~\ref{p:vOn}, the set $\gamma=\{x\in L_\alpha:\forall y\in x\cup\{x\}\;\;(\bigcup y\subseteq y)\}$ coincides with the intersection $L_\alpha\cap\On$ and is an element of the set $\Go^{\circ20}(\{L_\alpha\})$. By the transitivity of  $L_\alpha$ and $\On$, the set $\gamma=L_\alpha\cap\On$ is transitive and hence is an ordinal. 
By the inductive assumption, for every $\beta\in\alpha$ we have $\beta\in L_{\beta+1}\subseteq L_\alpha$ and hence $\beta\in L_\alpha\cap\On=\gamma$. Then $\alpha\le\gamma$. If $\alpha=\gamma$, then 
$\alpha=\gamma\in \mathcal P(L_\alpha)\cap\Go^{\circ20}(\{L_\alpha\})\subseteq L_{\alpha+1}$. If $\alpha\ne\gamma$, then $\alpha\in\gamma\subseteq L_\alpha\subseteq L_{\alpha+1}$ as well.
\end{proof}

\begin{exercise} Show that $L_\alpha=V_\alpha$ for all $\alpha\le \w$.
\end{exercise}

Next, we show that the class $\LL$ is closed under G\"odel's operations $\dG_0$--$\dG_8$.

\begin{proposition}\label{p:L-G-closed} Let $\alpha$ be an ordinal and $x,y\in L_\alpha$ be any sets. Then
\begin{enumerate}
\item[\textup{1)}] $x\setminus y\in L_{\alpha+1}$;
\item[\textup{2)}] $x\cap\E\in L_{\alpha+1}$;
\item[\textup{3)}] $\bigcup x\in L_{\alpha+1}$;
\item[\textup{4)}] $\dom[x]\in L_{\alpha+1}$;
\item[\textup{5)}] $\{x,y\}\in L_{\alpha+1}$;
\item[\textup{6)}] $\langle x,y\rangle\in L_{\alpha+2}$;
\item[\textup{7)}] $x\times y\in L_{\alpha+2}$;
\item[\textup{8)}] $x^{-1}\in L_{\alpha+1}$;
\item[\textup{9)}] $x^\circlearrowright\in L_{\alpha+2}$.
\end{enumerate}
\end{proposition}

\begin{proof} 
{\parskip2pt
The transitivity of the set $L_\alpha$ implies that $x,y\subseteq L_\alpha$.

1. It follows from $x,y\in L_\alpha$ and $x\subseteq L_\alpha$  that $\dG_1(x,y)=x\setminus y\in \mathcal P(L_\alpha)\cap\Go(L_\alpha)\subseteq L_{\alpha+1}$.

2. By analogy, $\dG_5(x,y)=x\cap\E\in \mathcal P(L_\alpha)\cap\Go(L_\alpha)\subseteq L_{\alpha+1}$.

3. The transitivity of the set $L_\alpha$ ensures that $\bigcup x\subseteq L_\alpha$ and hence\newline $\dG_7(x,y)=\bigcup x\in\mathcal P(L_\alpha)\cap\Go(L_\alpha)\subseteq L_{\alpha+1}$.

4. To see that $\dom[x]\subseteq L_\alpha$, take any element $u\in\dom[x]$ and find a set $v$ such that $\{u\}\in \langle u,v\rangle \in x$ and hence $\{u\}\in\bigcup x$ and $u\in\bigcup\bigcup x$. The transitivity of $L_\alpha$ ensures that $u\in\bigcup\bigcup x\subseteq L_\alpha$. Then $\dG_6(x,y)=\dom[x]\in\mathcal P(L_\alpha)\cap\Go(L_\alpha)\subseteq L_{\alpha+1}$.

5. It follows from $x,y\in L_\alpha$ that $\dG_8(x,y)=\{x,y\}\in \mathcal P(L_\alpha)\cap\Go(L_\alpha)\subseteq L_{\alpha+1}$.

6. By the preceding statement, $\{x\},\{x,y\}\in L_{\alpha+1}$ and $\langle x,y\rangle=\{\{x\},\{x,y\}\}\in L_{\alpha+2}$.

7. To see that $x\times y\subseteq L_{\alpha+1}$, take any element $z\in x\times y$ and find sets $u\in x$ and $v\in y$ such that $z=\langle u,v\rangle$. Let $\beta\le\alpha$ be the smallest ordinal such that $x,y\in L_\beta$. If $\beta$ is limit, then $L_\beta=\bigcup_{\gamma\in\beta}L_\gamma$ and hence $x,y\in L_\gamma$ for some ordinal $\gamma<\beta$, which contradicts the minimality of $\beta$. Therefore, $\beta$ is not limit and hence $\beta=\gamma+1$ for some ordinal $\gamma$. Then $x,y\in L_{\gamma+1}\subseteq\mathcal P(L_\gamma)$ and $u,v\in x\cup y\subseteq L_\gamma$. By the preceding statement, $z=\langle u,v\rangle\in L_{\gamma+2}=L_{\beta+1}\subseteq L_{\alpha+1}$ and hence $x\times y\in \mathcal P(L_{\alpha+1})\cap\Go(L_{\alpha})\subseteq L_{\alpha+2}$.

8. To prove that  $x^{-1}\in L_{\alpha+1}$, take any element $z\in x^{-1}$ and find sets $u,v$ such that $\langle u,v\rangle\in x$ and $\langle v,u\rangle =z$. Repeating the argument from the preceding paragraph, find an ordinals $\gamma<\beta<\alpha$ such that $\{u,v\}\in L_\beta$ and $u,v\in L_\gamma$. Then $z=\langle v,u\rangle\in L_{\gamma+2}\subseteq L_\alpha$. This shows that $x^{-1}\subseteq L_\alpha$ and hence $x^{-1}\in\mathcal P(L_\alpha)\cap\Go(L_\alpha)=L_{\alpha+1}$.

9. To prove that $x^\circlearrowright\in L_{\alpha+2}$, take any element $z\in x^\circlearrowright$ and find sets $u,v,w$ such that $\langle u,v,w\rangle\in x$ and $\langle w,u,v\rangle=z$. Repeating the preceding argument, find ordinals $\delta<\gamma<\beta<\alpha$ such that $\langle\langle u,v\rangle,w\rangle\in L_\beta$ and $\langle u,v\rangle,w\in L_\delta$. The transitivity of $L_\delta$ ensures that $u,v\in L_\delta$. Then $\langle w,u\rangle\in L_{\delta+2}\subseteq L_\beta$ and $z=\langle \langle w,u\rangle,v\rangle\in L_{\beta+2}\subseteq L_{\alpha+1}$ by Proposition~\ref{p:L-G-closed}(6). This shows that $x^\circlearrowright\in\mathcal P(L_{\alpha+1})\cap\Go(L_\alpha)\subseteq L_{\alpha+2}$.
}
\end{proof}

\begin{definition} A class $X$ is defined to be
\begin{itemize}
\item \index{$\dG_i$-closed class}\index{class!$\dG_i$-closed}{\em $\dG_i$-closed} for $i\in 9$ if $\dG_i(x,y)\in X$ for all $x,y\in X$;
\item \index{G\"odel-closed class}\index{class!G\"odel-closed}{\em G\"odel-closed} if $X$ is $\dG_i$-closed for every $i\in 9$. 
\end{itemize}
\end{definition}


Propositions~\ref{p:L-seq} and \ref{p:L-G-closed} imply the following important fact.

\begin{theorem}\label{t:L-tpc} The constructible universe $\mathbf L$ is a G\"odel-closed transitive proper class such that $\On\subseteq \mathbf L\subseteq\mathbf V.$
\end{theorem}

An important feature of the class $\mathbf L$ is its well-orderability. In Theorem~\ref{t:well-order} below we shall prove that there exists a basic set-like well-order $\mathbf W_{<}$ with $\dom[\mathbf W_{<}]=\mathbf L$. 

To construct such a well-order, we first construct an enumeration of all possible compositions of G\"odels operations. There are only countably many such compositions. We shall enumerate them by the countable set $\bigcup_{n\in\w}9^{2^{<n}}$. Since we included the identity operation $\dG_0$ in the list of G\"odel's operations, any composition of G\"odel's operations can be encoded by  a $9$-labeled full binary tree $2^{<n}$ of some finite height $n$.

For example, operation of union $x\cup y=\bigcup\{x,y\}$  can be written as the composition $\dG_7(\dG_8(x,y)),\dG_i(a,b))$, which is  represented by the full binary tree of height 2:
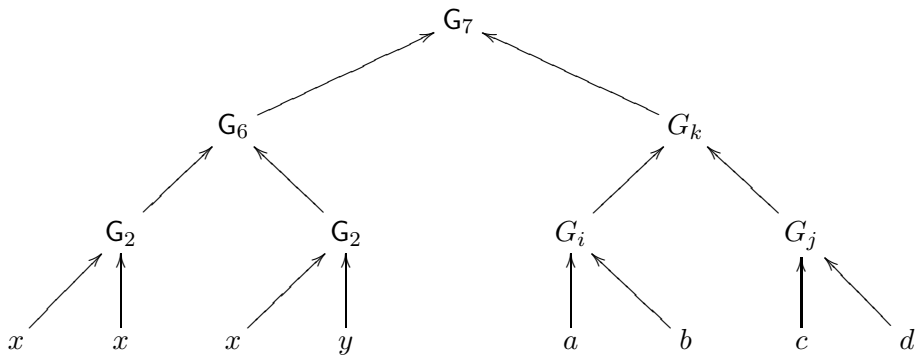
$$
\xymatrix{
&&&\dG_7\\
&\dG_8\ar[urr]&&&&\dG_i\ar[ull]\\
x\ar[ur]&&y\ar[ul]&&a\ar[ur]&&b\ar[ul]
}
$$
Since the operation $\dG_7$ of union does not depend on the second variable, in the right-hand part of the tree we can write any operations and variables. 

\begin{exercise} Draw the corresponding tree for representing the operation of forming the ordered pair $\langle x,y\rangle$ and the triple  $\langle x,y,z\rangle=\langle\langle x,y\rangle,z\rangle$.
\end{exercise}


Such a representation suggests the idea of encoding of all possible compositions of G\"odel's operations by binary trees whose vertices are labeled by numbers that belong to the set $9=\{0,1,\dots,8\}$.
 
By a \index{binary tree}{\em  binary tree} we understand any ordinary tree $T$ which is a subset of the full binary $\w$-tree $2^{<\w}$. We recall that $2^{<\w}$ consists of all functions $f$ with $\dom[f]\in\w$ and $\rng[f]\subseteq 2=\{0,1\}$. A subset $T\subseteq 2^{<\w}$ is called an {\em ordinary tree} if for any $t\in T$ and $n\in\w$ the function $t{\restriction}_n=t\cap(n\times \UU)$ belongs to $T$. For every $n\in\w$, the set $2^{<n}=\bigcup_{k\in n}2^k$ is an ordinary tree, called the {\em full binary tree of height $n$}. 

For every $k\in\{0,1\}$, the injective function $$\vec k:2^{<\w}\to 2^{<\w},\quad \vec k:t\mapsto \{\langle 0,k\rangle\}\cup\{\langle \alpha+1,y\rangle:\langle \alpha,y\rangle\in t\}$$is called the {\em $k$-transplantation} of the full binary $\w$-tree. 

\begin{exercise} Show that  $\rng[\vec k]=\{t\in 2^{<\w}:t(0)=k\}$ and for every $n\in\w$ we have $\vec k[2^n]=\{t\in 2^{n+1}:t(0)=k\}$.  Deduce from this that for every function $x:2^{n+1}\to\UU$ the composition $x\circ\vec k$ is a function with $\dom[x\circ \vec k]=2^n$. Consequently, the function 
$\UU^{2^{n+1}}\to \UU^{2^{n}}$, $x\mapsto x\circ\vec k$, is well-defined.
\end{exercise}

\begin{exercise} For a function $x=\{\langle \langle 0,0\rangle,x_{00}\rangle, \langle\langle 0,1\rangle,x_{01}\rangle,\langle\langle 1,0\rangle,x_{10}\rangle,\langle1,1\rangle, x_{11},\rangle\}\in \UU^{2^2}$, find the functions $x\circ \vec 0$ and $x\circ \vec 1$.
\smallskip

\noindent {\em Hint:} $x\circ\vec 0=\{\langle 0,x_{00}\rangle, \langle 1,x_{01}\rangle\}$ and $x\circ\vec 1=\{\langle 0,x_{10}\rangle, \langle 1,x_{11}\rangle\}$. 
\end{exercise}

By a \index{tree!labeling of}\index{labeled tree}{\em $9$-labeling of a binary tree $T$} we understand any function $\lambda:T\to 9=\{0,1,\dots,8\}$. Therefore a 9-labeling $\lambda$ assigns to each vertex $t\in T$ of the tree some number $\lambda(t)\in 9$. Observe that for any $n\in\w$ and a $9$-labeling $\lambda:2^{<(n+1)}\to 9$ of the full binary $(n+1)$-tree $2^{<(n+1)}$, the composition $\lambda\circ \vec k$ has $\dom[\lambda\circ\vec k]=2^{<n}$, so $\lambda\circ \vec k$ is a labeling of the tree $2^{<n}$.

Now for every $n\in\w$ and every $9$-labeling $\lambda:2^{<n}\to 9$ of the tree $2^{<n}$ we define the function $\Go_\lambda:\UU^{2^n}\to\UU$ by the recursive formulas:
\begin{enumerate}
\item If $n=0$, then $\Go_\lambda(\langle 0,x\rangle)=\Go_\emptyset(\langle 0,x\rangle)=x$ for any $\{\langle 0,x\rangle\}\in \UU^{2^0}=\UU^{1}$;
\item If $n>0$, then $\Go_{\lambda}(x)=\dG_{\lambda(0)}\big(\Go_{\lambda\circ\vec 0}(x\circ \vec 0),\Go_{\lambda\circ\vec 1}(x\circ\vec 1)\big)$ for every $x\in \UU^{2^n}$.
\end{enumerate}

\begin{lemma}\label{l:code} For every $n\in\w$ and set $x$ we have $$\Go^{\circ n}(x)=\bigcup\{\Go_\lambda[x^{2^n}]:\lambda\in 9^{2^{<n}}\}.$$
\end{lemma}

\begin{proof} For $n=0$ we have $2^{<0}=\emptyset$ and  
$\Go^{\circ0}(x)=x=\Go_\emptyset[x^{2^0}]=\bigcup\{\Go_\lambda[x^{2^0}]:\lambda\in 9^0=\{\emptyset\}\}$.
Assume that for some $n\in\w$ the equality
$$\Go^{\circ n}(x)=\bigcup\{\Go_\lambda[x^{2^n}]:\lambda\in 9^{2^{<n}}\}$$
has been proved.
Then 
$$
\begin{aligned}
\Go^{\circ(n+1)}(x)=\;&
\Go(\Go^{\circ n}(x))=\{\dG_i(u,v):i\in 9,\;u,v\in\Go^{\circ n}(x)\}\\
&=\big\{\dG_i(u,v):i\in 9,\;u,v\in\textstyle\bigcup\{\Go_\lambda[x^{2^n}]:\lambda\in 9^{2^{<n}}\}\big\}\\
&=\big\{\dG_i(\Go_\mu(f),\Go_\nu(g)):i\in 9,\;\mu,\nu\in 9^{2^{<n}},\;f,g\in x^{2^n}\big\}\\
&=\big\{\Go_\lambda(\varphi):\lambda\in 9^{2^{<(n+1)}},\;\varphi\in x^{2^{n+1}}\big\}.
\end{aligned}
$$
\end{proof}

Lemma~\ref{l:code} implies the following corollary.

\begin{corollary}\label{c:code} For every  set $x$ its G\"odel's hull $\Go^{\circ\w}(x)$  is equal to  $$\bigcup_{n\in\w}\bigcup_{\lambda\in 9^{2^{<n}}}\Go_\lambda[x^{2^n}].$$
\end{corollary}


Now we are able to prove the promised 

\begin{theorem}[G\"odel]\label{t:well-order} There exists a set-like well-order $\mathbf W_{<}$ such that $\dom[\mathbf W_{<}]=\mathbf L$.
\end{theorem}

\begin{proof} By Corollary~\ref{c:code}, $\mathbf L=\bigcup_{\alpha\in\On}L_\alpha$, where $$L_\alpha=\bigcup_{\beta\in\alpha}\mathcal P(L_\beta)\cap \Go^{\circ\w}(L_\beta\cup\{L_\beta\})=\bigcup_{\beta\in\alpha}\bigcup_{n\in\w}\bigcup_{\lambda\in 9^{<2^n}}\mathcal P(L_\beta)\cap\Go_\lambda[(L_\beta\cup\{L_\beta\})^{2^n}].$$

For every ordinal $\alpha$ we define a well-order $W_\alpha$ on the set $L_\alpha$ by recursion. First we fix well-orders on the sets $2^{<\w}=\bigcup_{n\in\w}2^n$ and $9^\star= \bigcup_{n\in\w}9^{2^{<n}}$. The set $2^{<\w}$ carries the well-order $$<_2=\{\langle f,g\rangle \in 2^{<\w}\times 2^{<\w}:|f|<|g|\;\vee\; \exists i \in \dom[f]=\dom[g]\;(f{\restriction}_i=g{\restriction}_i\wedge f(i)<g(i))\}$$  and the set $9^\star= \bigcup_{n\in\w}9^{2^{<n}}$ carries the well-order
 $$<_9=\{\langle f,g\rangle \in 9^\star\times 9^\star:|f|<|g|\;\vee\; \exists s \in \dom[f]=\dom[g]\;(f{\restriction}_{{\downarrow}s}=g{\restriction}_{{\downarrow}s}\;\wedge\; f(s)<g(s))\},$$
 where ${\downarrow}s=\{t\in 2^{<\w}:t<_2s\}$.

Assume that for some ordinal $\alpha$ we have defined a sequence $(W_\beta)_{\beta\in\alpha}$ of well-orders such that $W_\beta^\pm=(L_\beta\times L_\beta)\setminus\Id$ for every $\beta\in\alpha$. If $\alpha$ is a limit or zero ordinal, then let $$W_\alpha\defeq\bigcup_{\beta\in\alpha}\big(L_\beta\times (L_\alpha\setminus L_\beta)\big)\cup \big(W_\beta{\restriction} (L_{\beta+1}\setminus L_\beta)\big)$$and observe that $W_\alpha$ is a well-order on the set $L_\alpha$.

Next, assume that $\alpha=\beta+1$ for some ordinal $\beta$. Consider the well-order $W_\beta$ on the set $L_\beta$ and extend it to the well-order $\prec_\beta=W_\beta\cup (L_\beta\times\{L_\beta\})$ on the set $L_\beta\cup\{L_\beta\}$. Endow the set $T_\beta=\bigcup_{n\in\w}9^{2^{<n}}\times (L_\beta\cup\{L_\beta\})^{2^n}$ with the well-order $$W'_\beta=\{\langle \langle \lambda,f\rangle,\langle \mu,g\rangle\rangle\in T_\beta\times T_\beta:\lambda<_9\mu\vee (\lambda=\mu\wedge \exists s\in \dom[\lambda]\;(f{\restriction}_{{\downarrow}s}=g{\restriction}_{{\downarrow}s}\wedge f(s)\prec_\beta g(s))\},$$
where ${\downarrow} s=\{t\in\dom(\lambda):t<_2 s\}$.

Finally, consider the well-order
\begin{multline*}
W_\alpha\defeq W_\beta\cup (L_\beta\times (L_\alpha\setminus L_\beta))\cup
\{\langle a,b\rangle:a,b\in L_\alpha\setminus L_\beta\;\wedge \\
\exists \langle\lambda,f\rangle\in T_\beta\; (a=\Go_\lambda(f)\;\wedge\; \forall \langle \mu,g\rangle\in T_\beta\; (b=\Go_\mu(g)\Rightarrow \langle\langle\lambda,f\rangle,\langle \mu,g\rangle\rangle\in W_\beta'))\}
\end{multline*}
on the set $L_\alpha$.
Then the union
$$
\mathbf W_<=\bigcup_{\alpha\in\On}W_\alpha$$ is a 
 desired set-like well-order on the constructible universe $\mathbf L$.
\end{proof}

\begin{exercise} Prove that the well-order $\mathbf W_<$ is a basic class.
\smallskip

\noindent{\em Hint:} Analyse the definition of $\mathbf W_<$ and apply Theorem~\ref{t:class}.
\end{exercise}

In Theorem~\ref{t:Con(U=L)} we shall prove that the following statement does not contradict the axioms of the Classical Set Theory.
$$\boxed{\mbox{\bf Axiom of Constructibility:\quad}\mathbf U=\mathbf L \quad}$$

The Axiom of Constructibility postulates that every set is constructible.

\begin{theorem} The Axioms Classical Set Theory with added Axiom of Constructibility imply the Axiom of Foundation and the Axiom of Global Choice.
\end{theorem}

\begin{proof} Since $\mathbf L\subseteq\mathbf V\subseteq \mathbf U$, the equality $\mathbf L=\mathbf U$ implies the equality $\mathbf V=\mathbf U$, which is equivalent to the Axiom of Foundation by Theorem~\ref{t:AF}.

To prove that $\mathbf L=\mathbf U$ implies the Axiom of Global Choice, define the choice function $C:2^{\UU}\setminus\{\emptyset\}\to\UU$ assigning to every nonempty subset $x\subseteq \UU=\mathbf L$ the unique $\mathbf W_{<}$-least element of $x$, where $\mathbf W_<$ in the well-order on $\mathbf L$, constructed in Theorem~\ref{t:well-order}.
\end{proof}

\section{G\"odel classes}

In this section we study G\"odel classes, one of which is the G\"odel constructible universe $\mathbf L$.

\begin{definition} A class $X$ is defined to be
\begin{itemize} 
\item \index{almost universal class}\index{class!almost universal}{\em almost universal} if for every subset $x\subseteq X$ there exists an element $u\in X$ such that $x\subseteq u$;
\item \index{G\"odel class}\index{class!G\"odel}{\em G\"odel} if $X$ is transitive, almost universal and G\"odel-closed.
\end{itemize}
\end{definition}

\begin{proposition} The class of constructible sets $\mathbf L$ is G\"odel. 
\end{proposition}

\begin{proof} By Theorem~\ref{t:L-tpc}, the class $\LL$ is transitive and G\"odel-closed. To see that $\LL$ is almost universal, take any set $x\subseteq \mathbf L=\bigcup_{\alpha\in\On}L_\alpha$. Applying the Axiom of Replacement, find an ordinal $\alpha$ such that $x\subseteq L_\alpha$. Since $L_\alpha\in L_{\alpha+1}\subseteq\LL$, the class $\LL$ is almost universal. 
\end{proof}

\begin{proposition} Every G\"odel class $X$ is proper.
\end{proposition}

\begin{proof} Assuming that $X$ is a set, we can find an element $u\in X$ such that $X\subseteq u$. The transitivity of the class $X$ ensures that $u\subseteq X$. We claim that the set $$a=\{\langle x,y\rangle\in u\times u:x\ne y\}$$ is an element of  the class $X$. Indeed, for any $x\in u$ and $z\in x$ we have $z\in \bigcup u=\dot\Go_7(u)\in X$. Then $a=\dom[p_1\cup p_2]$ where $p_1=\{\langle\langle x,y\rangle, z\rangle\in (u\times u)\times \bigcup u:z\in x\wedge z\notin y\}$ and
$p_2=\{\langle\langle x,y\rangle, z\rangle\in (u\times u)\times \bigcup u:z\notin x\wedge z\in y\}$. It follows that $p_1=p_{11}\cap p_{12}$ and $p_2=p_{21}\cap p_{22}$, where
$$\begin{aligned}
p_{11}&=\textstyle\{\langle\langle x,y\rangle, z\rangle\in (u\times u)\times \bigcup u:z\in x\}=[(\mathbf E\cap(\bigcup u\times u))\times u]^\circlearrowleft\in X;\\ 
p_{12}&=\textstyle\{\langle\langle x,y\rangle, z\rangle\in (u\times u)\times \bigcup u:z\notin y\}=
[((u\times \bigcup u)\setminus [\mathbf E\cap (\bigcup u\times u)]^{-1})\times u]^\circlearrowright\in X;\\
p_{21}&=\textstyle\{\langle\langle x,y\rangle, z\rangle\in (u\times u)\times \bigcup u:z\notin x\}=
[((\bigcup u\times u)\setminus (\mathbf E\cap (\bigcup u\times u))\times u]^\circlearrowleft\in X;\\
p_{22}&=\textstyle\{\langle\langle x,y\rangle, z\rangle\in (u\times u)\times \bigcup u:z\in y\}=
[[\mathbf E\cap(\bigcup u\times u)]^{-1}\times u]^\circlearrowright\in X.
\end{aligned}
$$
Since the class $X$ is G\"odel-closed, the set $a=\dom[(p_{11}\cap p_{12})\cup(p_{21}\cap p_{22})]$ is an element of $X$. Then the set $b=\{\langle x,y\rangle\in u\times u:x=y\}=(u\times u)\setminus a=\dG_1(\dot\Go_2(u,u),a)$ also is an element of $X$, and so is the set $$c=\{x\in u:x\notin x\}=\dom[\{\langle x,y\rangle\in u\times u:x=y\wedge x\notin y\}]=\dom[b\setminus (\mathbf E\cap (u\times u))]\in X.$$ Then $c\in X\subseteq u$ and $c\in u$. Now we obtain a contradiction of Russell's type: if $c\in c$, then $c\notin c$, and if $c\notin c$, then $c\in c$. This contradiction shows that the class $X$ is proper.
\end{proof}  

\begin{proposition}\label{p:Go-Ord} If $\mathbf X$ is a G\"odel class \textup{(}with $\mathbf X\subseteq\mathbf V$\textup{)}, then $\w\subset \mathbf X$ \textup{(}and $\On\subseteq\mathbf X$\textup{)}.
\end{proposition}

\begin{proof} The inclusion $\w\subseteq\mathbf X$ will be proved by induction.  Since $\emptyset\subset \mathbf X$, by the almost universality of $\mathbf X$, there exists a set $y\in \mathbf X$ such that $\emptyset \subseteq y$. Since $\mathbf X$ is G\"odel-closed, $0=\emptyset=y\setminus y=\Go_1(y,y)\in X$. Assume that for some number $n\in\w$ we have proved that $n\in\mathbf X$. Then $n+1=n\cup\{n\}=\bigcup\{n,\{n\}\}=\dot\Go_7(\dG_8(n,\dot\Go_8(n)))\in X$. By the Principle of Mathematical Induction, $\w\subseteq \mathbf X$. Since $\mathbf X$ is a proper class, $\w\ne\mathbf X$ and hence $\w\subset \mathbf X$.

Now assume that $\mathbf X\subseteq \mathbf V$. By transfinite induction, we shall prove that every ordinal $\alpha$ belongs to $\mathbf X$. Assume that for some ordinal $\alpha$ we have proved that $\beta\in \mathbf X$ for all $\beta<\alpha$, and hence $\alpha\subseteq\mathbf X$.
If $\alpha=\beta+1$ for some ordinal $\beta$, then $\beta\in\mathbf X$ and $\alpha=\beta\cup\{\beta\}=\dot\Go_7(\dG_8(\beta,\dot\Go_8(\beta)))\in \mathbf X$, because the class $\mathbf X$ is G\"odel-closed. If $\alpha=0$, then $0\in\w\subset\mathbf X$. It remains to consider the case of limit ordinal $\alpha$. By our assumption, $\alpha\subseteq \mathbf X$ and by the almost universality of $\mathbf X$, there exists a set $u\in\mathbf X$ such that $\alpha\subseteq u$. By Proposition~\ref{p:vOn}, the set $\gamma=u\cap\On$ is an element of $\Go^{\circ20}(\{u\})\subseteq \mathbf X$ and so is the set $\bigcup\gamma=\dot\Go_7(\gamma)\in\mathbf X$. By Theorem~\ref{t:Ord}(5), the set $\bigcup\gamma$ is an ordinal. It follows from $\alpha\subseteq u\cap\On=\gamma$ that $\alpha=\bigcup\alpha\subseteq\bigcup\gamma$. If $\alpha=\bigcup\gamma$, then $\alpha=\bigcup\gamma\in\mathbf X$. If $\alpha\ne\bigcup\gamma$, then $\alpha\in \bigcup\gamma\subseteq \mathbf X$ by Theorem~\ref{t:ord}(4) and the transitivity of the G\"odel class $\mathbf X$.
\end{proof}

\section{G\"odel operations over classes}

Let us observe that the G\"odel's operations $\dG_i$ for $i\le 7$ are well-defined on classes. Namely, by the axioms of $\CST$, for any classes $X,Y$ the following classes are well-defined:
 \begin{itemize}
\item[\textup{(0)}] $\dG_0(X,Y)\defeq X$
\item[\textup{(1)}] $\dG_1(X,Y)\defeq X\setminus Y$
\item[\textup{(2)}] $\dG_2(X,Y)\defeq X\times Y$
\item[\textup{(3)}] $\dG_3(X,Y)\defeq X^{-1}$
\item[\textup{(4)}] $\dG_4(X,Y)\defeq X^\circlearrowright$
\item[\textup{(5)}] $\dG_5(X,Y)\defeq X\cap\mathbf E$
\item[\textup{(6)}] $\dG_6(X,Y)\defeq \dom[X]$
\item[\textup{(7)}] $\dG_7(X,Y)\defeq\bigcup X$.
\end{itemize}

In fact, the operations $\dG_0$ and $\dG_3$--$\dG_7$ do not depend on the second variable. We have written them as binary operations in the sake of uniform treatment of G\"odels operations over classes and enumerating their compositions by $8$-labeled full binary trees of finite height. However, for simplifying the expressions involving the G\"odel's operations, it will be convenient to consider for every $i\in 8$ the unary G\"odel's operation $\dot\Go_i$ assigning to every class $X$ the class
$$\dot\Go_i(X)\defeq\dG_i(X,X).$$

\begin{proposition}\label{p:G(X,Y)M} Let $\MM$ be a $\dG_8$-closed transitive class. For any subclasses $X,Y\subseteq\MM$ and every $i\in 8$, the class $\dG_i(X,Y)$ is a subclass of $\mathbf M$.
\end{proposition}

\begin{proof} Depending on the value of the number $i$ we consider eight cases.
{\parskip2pt

\noindent 0. If $i=0$, then $\dG_i(X,Y)=X\subseteq \mathbf M$.

\noindent 1. If $i=1$, then $\dG_i(X,Y)=X\setminus Y\subseteq X\subseteq \mathbf M$.

\noindent 2. If $i=2$, then $\dG_i(X,Y)=X\times Y=\{\langle x,y\rangle:x\in X\;\wedge\; y\in Y\}\subseteq\{\{\{x\},\{x,y\}\}:x,y\in\mathbf M\}=\{\dG_8(\dG_8(x,x),\dG_8(x,y)):x,y\in\mathbf M\}\subseteq\mathbf M$ because the class $\mathbf M$ is $\dG_8$-closed.

\noindent 3. If $i=3$, then $\dG_i(X,Y)=X^{-1}=\{\langle y,x\rangle:\langle x,y\rangle\in X\}\subseteq\{\langle y,x\rangle:x,y\in \bigcup\bigcup X\}\subseteq\mathbf M$ because the class $\mathbf M$ is transitive and $\dG_8$-closed.

\noindent 4. If $i=4$, then $\dG_i(X,Y)=X^\circlearrowright=\{\langle z,x,y\rangle:\langle x,y,z\rangle\in X\}\subseteq\{\langle z,x,y\rangle:x,y,z\in \mathbf M\}\subseteq\mathbf M$ because the class $\mathbf M$ is transitive and $\dG_8$-closed.

\noindent 5. If $i=5$, then $\dG_i(X,Y)=X\cap\E\subseteq X\subseteq \mathbf M$.

\noindent 6. If $i=6$, then $\dG_i(X,Y)=\dom[X]=\{x:\exists y\;\langle x,y\rangle\in X\}\subseteq \bigcup\bigcup X\subseteq\bigcup\bigcup\mathbf M\subseteq\mathbf M$, by the transitivity of the class $\mathbf M$.

\noindent 7. If $i=7$, then $\dG_i(X,Y)=\bigcup X\subseteq\mathbf M$, by the transitivity of the class $\mathbf M$.
}
\end{proof}

\begin{lemma}\label{l:M-reduction} Let $\mathbf M$ be a G\"odel class. For any  subclasses $X,Y\subseteq \mathbf M$ and set $u\in\mathbf M$, there exists a set $z\in\mathbf M$  such that for every number $i\in 6$ and every class $Z$ with $z\subseteq Z$ we have  the equality
$$u\cap \dG_i(X,Y)=u\cap\dG_i(Z\cap X,Z\cap Y).$$
\end{lemma}

\begin{proof} Fix  classes $X,Y\subseteq\mathbf M$ and a set $u\in\mathbf M$. 
Since the class $\mathbf M$ is G\"odel-closed and transitive, $$\textstyle\dom[u]\cup\rng[u]\cup u^{-1}\cup u^\circlearrowright\cup\bigcup u\subseteq \mathbf M.$$
By the almost universality of $\mathbf M$, there exists a set $z\in\mathbf M$ such that $$\textstyle u\cup \dom[u]\cup\rng[u]\cup u^{-1}\cup u^\circlearrowright\cup\bigcup u\subseteq z.$$
We claim that the set $z$ has the required property. Given a class $Z$ with $z\subseteq Z$, we should check that $u\cap \dG_i(X,Y)=u\cap\dG_i(Z\cap X,Z\cap Y)$ for every number $i\in 6$. Depending on the value of $i$, we consider six cases.
\smallskip

0. If $i=0$, then $u\cap \dG_i(X,Y)=u\cap X=u\cap Z\cap X=u\cap \dG_i(Z\cap X,Z\cap Y)$.
\smallskip

1. If $i=1$, then  $u\cap \dG_1(X,Y)=u\cap(X\setminus Y)=u\cap Z\cap (X\setminus Y)=u\cap ((Z\cap X)\setminus(Z\cap Y))=u\cap \dG_1(Z\cap X,Z\cap Y)$.
\smallskip

2. If $i=2$, then
$$
\begin{aligned}
u\cap\dG_i(X,Y)&=u\cap(X\times Y)=u\cap(\UU\times\UU)\cap (X\times Y)=u\cap(\dom[u]\times \rng[u])\cap(X\cap Y)\\
&=u\cap(Z\times Z)\cap (X\times Y)=u\cap\dG_i(Z\cap X,Z\cap Y).
\end{aligned}
$$ 
\smallskip

3. If $i=3$, then  
$$u\cap \dG_i(X,Y)=u\cap X^{-1}=u\cap(X\cap u^{-1})^{-1}=u\cap (X\cap Z)^{-1}=u\cap \dG_i(Z\cap X,Z\cap Y).$$

4. If $i=4$, then 
$$u\cap \dG_i(X,Y)=u\cap X^\circlearrowright=u\cap (u^{\circlearrowleft}\cap X)^{\circlearrowright}=u\cap(Z\cap X)^\circlearrowright=u\cap \dG_i(Z\cap X,Z\cap X).$$

5. If $i=5$, then $u\cap \dG_i(X,Y)=u\cap(X\cap \mathbf E)=(u\cap Z)\cap (X\cap\mathbf E)=u\cap\dG_i(Z\cap X,Z\cap Y)$.
\end{proof} 

For $\mathbf M$-collective G\"odel classes, Lemma~\ref{l:M-reduction} extends to all G\"odel's operation $\dG_0$--$\dG_7$. 

\begin{definition} A class $X$ is defined to be
\begin{itemize}
\item  \index{cumulative class}\index{class!cumulative}{\em cumulative} if there exists a function $F:X\to \On$ such that for any ordinal $\alpha$ the preimage $F^{-1}[\alpha]$ is a set; 
\item  \index{collective class}\index{class!collective}{\em collective} if for any set $u$ and any subclass $Y\subseteq u\times X$ there exists a set $v$  such that $\dom[Y]=\dom[Y\cap v]$;
\item \index{class!$\MM$-collective}{\em $X$-collective} if for any set $u\in X$ and any subclass $Y\subseteq u\times X$ there exists a set $v$  such that $\dom[Y]=\dom[Y\cap v]$.
\end{itemize}
\end{definition}

It is clear that every collective class $X$ is $X$-collective.

\begin{proposition}\label{p:V-cumucol} Every subclass of $\VV$  is cumulative, and every cumulative class is collective.
\end{proposition}

\begin{proof} The cumulativity of a subclass $X\subseteq\mathbf V=\bigcup_{\alpha\in\On}V_\alpha$ is witnessed by the function $\Lambda:X\to\On$ assigning to every $x\in X$ the smallest ordinal $\alpha$ such that $x\in V_\alpha$. 

Next, fix any cumulative class $Y$ and let $F:Y\to\On$ be a function such that for every ordinal $\alpha$ the preimage $F^{-1}[\alpha]$ is a set. To see that $Y$ is collective, fix any set $u$ and any subclass $Z\subseteq u\times Y$. Consider the function $\Lambda:\dom[Z]\to\On$ assigning to every $x\in \dom[X]$ the smallest ordinal $\alpha$ such that $\{y\in Y:\langle x,y\rangle\in Z\}\cap F^{-1}[\alpha]\ne\emptyset$. Since $\dom[Z]\subseteq u$ is a set, we can apply the  Axiom of Replacement and conclude that  the class $\Lambda[\dom[Z]]\subseteq \On$ is a set and hence $\Lambda[\dom[Z]]\subseteq\alpha$ for some ordinal $\alpha$. Then the set $v\defeq u\times F^{-1}[\alpha]$ has the required property $\dom[Z]=\dom[Z\cap v]$, witnessing that $Y$ is a collective class.
\end{proof} 

\begin{lemma}\label{l:dom-reduction} Let  $\mathbf M$ be an $\mathbf M$-collective G\"odel class. For any  subclass $X\subseteq \mathbf M$ and set $u\in\mathbf M$, there exists a set $z\in\mathbf M$  such that for any class $Z$ with $z\subseteq Z$ we have $u\cap \dom[X]=u\cap \dom[Z\cap X]$.
\end{lemma}

\begin{proof} Fix a subclass $X\subseteq\mathbf M$ and a set $u\in\mathbf M$. 
Since $\mathbf M$ is $\mathbf M$-collective, for the class $X\cap (u\times\mathbf M)$ there exists a set $v$ such that $\dom[X\cap(u\cap\mathbf M)]=\dom[v\cap X\cap(u\cap \mathbf M)]$. Since the class $\mathbf M$ is almost universal, there exists a set $w\in\mathbf M$ such that $\mathbf M\cap\rng[v]\subseteq w$. We claim that $u\cap\dom[X]=\dom[(u\times w)\cap X]$. The inclusion $\dom[(u\times w)\cap X]\subseteq u\cap\dom[X]$ is obvious. To see that $u\cap\dom[X]\subseteq\dom[(u\times w)\cap X]$, take any element $x\in u\cap\dom[X]$ and find a set $y$ such that $\langle x,y\rangle\in X\subseteq\mathbf M$. By transitivity of $\mathbf M$, the set $y\in\{x,y\}\in\langle x,y\rangle\in X\subseteq \mathbf M$ is an element of $\mathbf M$ and hence $\langle x,y\rangle\in X\cap(u\times\mathbf M)$ and $x\in \dom[X\cap(u\times\mathbf M)]=\dom[v\cap X\cap(u\times\mathbf M)]$. Then there exists a set $y'$ such that $\langle x,y'\rangle\in v\cap X\cap(u\times\mathbf M)$  and hence $y'\in \rng[v]\cap\mathbf M\subseteq w$. It follows from $\langle x,y'\rangle\in X\cap(u\times w)$ that $x\in\dom[X\cap(u\times w)]$, which completes the proof of the equality $u\cap\dom[X]=\dom((u\times w)\cap X]$. 

Since the class $\mathbf M$ is G\"odel-closed, the set $z=u\times w=\dG_2(u,w)$ is an element of $\mathbf M$. 
The set $z$ has the required property, because for every class $Z$ with $z\subseteq Z$, we have 
$$\dom[(u\times w)\cap X]=u\cap \dom[z\cap X]\subseteq u\cap\dom[Z\cap X]\subseteq u\cap\dom[X]=\dom[(u\times w)\cap X)$$and hence $u\cap\dom[X]=u\cap\dom[Z\cap X]$.
\end{proof}

\begin{lemma}\label{l:union-reduction} Let $\mathbf M$ be an $\mathbf M$-collective G\"odel class. For any  subclass $X\subseteq \mathbf M$ and set $u\in\mathbf M$, there exists a set $z\in\mathbf M$  such that for any class $Z$ with $z\subseteq Z$ we have $u\cap \bigcup X=u\cap \bigcup(Z\cap X)$.
\end{lemma}

\begin{proof} Fix an element $u\in\mathbf X$ and a subclass $X\subseteq\mathbf M$. The transitivity of the G\"odel class $\mathbf M$ implies $\bigcup X\subseteq \bigcup\mathbf M\subseteq\mathbf M$. Then 
$$\textstyle\bigcup X=\{x:\exists y\in X\;(x\in y)\}=\dom[\{\langle x,y\rangle\in \mathbf M\times X:x\in y\}]=\dom[(\mathbf M\times X)\cap \mathbf E].$$ Since the class $\mathbf M$ is $\dG_8$-closed the class $Y=(\mathbf M\times X)\cap\E$ is a subclass of $\mathbf M$. By Lemma~\ref{l:dom-reduction}, there exists a set $v\in\mathbf M$ such that $u\cap\dom[Y]=u\cap \dom[v\cap Y]$. Since the class $\mathbf M$ is G\"odel-closed, the set $z\defeq \rng[v]=\dom[v^{-1}]$ is an element of $\mathbf M$. Then $v\cap(\mathbf M\times X)=v\cap(\mathbf M\times(z\cap X))$. We claim that for every class $Z$ with $z\subseteq Z$ we have $u\cap \bigcup X=u\cap\bigcup(Z\cap X)$. The inclusion $u\cap\bigcup(Z\cap X)\subseteq u\cap\bigcup X$ is obvious. To see that $u\cap\bigcup X\subseteq u\cap\bigcup(Z\cap X)$, take any element $x\in u\cap\bigcup X$ and observe that $x\in u\cap\bigcup X=u\cap\dom[Y]=u\cap \dom[v\cap Y]$ and hence there exists a set $y$ such that $\langle x,y\rangle\in v\cap Y=v\cap(\mathbf M\times X)\cap\E=v\cap (\mathbf M\times (z\cap X))\cap \E$, which means that $x\in y\in z\cap X\subseteq Z\cap X$ and hence $x\in u\cap\bigcup(Z\cap X)$.
\end{proof}

Lemmas~\ref{l:M-reduction}, \ref{l:dom-reduction}, \ref{l:union-reduction} imply the following theorem which is a main result of this section.

\begin{theorem}\label{t:uG(X,Y)} If a G\"odel class $\mathbf M$ is $\mathbf M$-collective, then for every subclasses $X,Y\subseteq \mathbf M$ and set $u\in\mathbf M$ there exists a set $z\in\mathbf M$ such that $$\forall i\in 8\;\forall Z\;\big((z\subseteq Z)\to (u\cap\dG_i(X,Y)=u\cap \dG_i(Z\cap X,Z\cap Y))\big).$$
\end{theorem}

\section{Constructible classes}

In this section, given a G\"odel class $\mathbf M$, we introduce and study $\MM$-constructible classes. Those are classes that can be constructed by repeated application of G\"odel's operations $\dG_0$--$\dG_7$ to the class $\mathbf M$ and  elements of the class $\mathbf M$. 

Observe that all possible compositions of G\"odel's operations $\dG_0$--$\dG_7$ can be effectively enumerated by the set $\bigcup_{n\in\w}8^{2^{<n}}$ of 8-labeled full binary trees of finite height. This suggests the following definition of the $\MM$-constructibility.

\begin{definition}\label{d:Mdef} Let $\MM$ be a class. A class $X$ is defined to be  \index{$\MM$-constructible class}\index{class!$\MM$-constructible}{\em $\MM$-constructible} if there exist $n\in\w$, $\lambda\in 8^{2^{<n}}$, and an indexed family of classes $(X_a)_{a\in 2^{\le n}}$ satisfying the following conditions:
\begin{enumerate}
\item $X=X_\emptyset$;
\item $\forall a\in 2^n\;\; (X_a=\MM\;\vee\;X_a\in\MM)$;
\item $\forall a\in 2^{<n}\;\;\big(X_a=\dG_{\lambda(a)}(X_{a\hat{\;}0},X_{a\hat{\;}1})\big)$.
\end{enumerate}
\end{definition} 

In this definition $2^{\le n}\defeq\bigcup_{k\le n}2^k$ and for every function $a\in 2^{<\w}$ and number $y\in 2$ the function $a\hat{\;}y$ is defined as $a\cup\{\langle\dom[a],y\rangle\}$. If $a$ is a function with domain  $\dom[a]=n\in\w$, then $a\hat{\;}y$ is the unique function such that $\dom[a\hat{\;}y]=n+1$, $(a\hat{\;}y){\restriction}_n=a$ and $(a\hat{\;}y)(n)=y$.

\begin{exercise} Show that the class $\MM$ is $\MM$-constructible.
\end{exercise}

\begin{exercise} Show that every set $x\in\MM$ is an $\MM$-constructible class.
\end{exercise}

\begin{exercise}\label{ex:M-operations} Prove that for any $\MM$-constructible classes $X,Y$, the classes $X\setminus Y$, $X\cap Y$, $X\times Y$, $X^{-1}$, $X^\circlearrowright$, $\dom[X]$, $\bigcup X$ are $\MM$-constructible.
\end{exercise}

\begin{proposition}\label{p:MsubM} Let $\MM$ be a transitive $\dG_8$-closed class. Every $\mathbf M$-constructible class $X$ is a subclass of the class $\mathbf M$.
\end{proposition}

\begin{proof} Let $X$ be an $\mathbf M$-constructible class. Then there exist $n\in\w$, $\lambda\in 8^{2^{<n}}$ and an indexed family of classes $(X_a)_{a\in 2^{\le n}}$ satisfying the three conditions of Definition~\ref{d:Mdef}. Consider the set $A=\{a\in 2^{\le n}:X_a\not\subseteq\mathbf M\}$, which exists by Theorem~\ref{t:class}. Assuming that the set $A$ is not empty, find the largest number $k\le n$ such that $A\cap 2^k\ne\emptyset$ and choose an element $b\in A\cap 2^k$.  The transitivity of the class $\mathbf M$ and Definition~\ref{d:Mdef}(2) imply that $k<n$. The maximality of $k$ ensures that for every $i\in 2$ the function $b\hat{\;}i$ does not belong to the set $A$ and hence  $X_{b\hat{\;}0}\cup X_{b\hat{\;}1}\subseteq\mathbf M$. By Proposition~\ref{p:G(X,Y)M}, $X_b=\dG_{\lambda(b)}(X_{b\hat{\;}0},X_{b\hat{\;}1})\subseteq\mathbf M$, which contradicts the choice of $b\in A$. This contradiction shows that the set $A$ is empty and hence $X_a\subseteq\mathbf M$ for all $a\in 2^{\le n}$. In particular, $X=X_\emptyset\subseteq \mathbf M$.
\end{proof}

\begin{exercise}\label{ex:Mdefcup} Let $\MM$ be a transitive $\dG_8$-closed class. Prove that for any $\MM$-constructible classes $X,Y$ the class $X\cup Y$ is $\MM$-constructible.
\smallskip

\noindent{\em Hint:} By Proposition~\ref{p:MsubM}, $X\cup Y\subseteq \MM$, which implies $X\cup Y=\MM\setminus((\MM\setminus X)\cap(\MM\setminus Y))$.
\end{exercise}

\smallskip




\begin{proposition}\label{p:uXM} Let $\MM$ be an $\MM$-cumulative G\"odel class. 
For every $\mathbf M$-constructible class $X$ and every set $u\in\mathbf M$, the set $u\cap X$ is an element of the class $\mathbf M$.
\end{proposition}

\begin{proof} Let $u\in\mathbf M$ and $X$ be an $\mathbf M$-constructible class. Then there exist $n\in\w$, $\lambda\in 8^{2^{<n}}$ and an indexed family of classes $(X_a)_{a\in 2^n}$ satisfying the three conditions of Definition~\ref{d:Mdef}.

Let $\mathcal F$ be the class of all functions $f$ having the following properties:
\begin{enumerate}
\item $\dom[f]\subseteq 2^{\le n}$;
\item $\forall t\in \dom[f]\;\forall k\le n\;(t{\restriction}_k\in \dom[f])$;
\item $\forall t\in\dom[f]\;\big((\exists k\in 2\; (t\hat{\;}k\in\dom[f]))\to 
(\forall k\in 2\; (t\hat{\;}k\in\dom[f]))\big)$;
\item $\rng[f]\subseteq \mathbf M$ and $f(\emptyset)=u$;
\item $\forall a{\in} \dom[f]\;\big(\big(\{a\hat{\;}0,a\hat{\;}1\}{\subseteq}\dom[f]\big)\to \big(f(a){\cap} X_a=\dG_{\lambda(a)}(f(a\hat{\;}0){\cap} X_{a\hat{\;}0},f(a\hat{\;}1){\cap} X_{a\hat{\;}1})\big)\big)$.
\end{enumerate}
The existence of the class $\mathcal F$ follows from Theorem~\ref{t:class}.
The class $\mathcal F$ contains the empty function and hence is nonempty. Since the set $\{\dom[f]:f\in\mathcal F\}\subseteq\mathcal P(2^{\le n})$ is finite, there exists a function $f\in\mathcal F$ such that $\forall g\in\mathcal F\;(\dom[f]\subseteq \dom[g]\to \dom[f]=\dom[g])$. We claim that $\dom[f]=2^{\le n}$. To derive a contradiction, assume that $\dom[f]\ne 2^{\le n}$. The conditions (1)--(3) imply that there exists $t\in\dom[f]\cap 2^{<n}$ such that $t\hat{\;}0,t\hat{\;}1\notin \dom[f]$. Definition~\ref{d:Mdef}(3) ensures that $X_t=\dG_{\lambda(t)}(X_{t\hat{\;}0},X_{t\hat{\;}1})$. By Theorem~\ref{t:uG(X,Y)}, for the set $f(t)\in\mathbf M$ there exists a set $z\in\mathbf M$ such that $f(t)\cap \dG(X_{t\hat{\;}0},X_{t\hat{\;}1})=\dG_{\lambda(t)}(z\cap X_{t\hat{\;}0},z\cap X_{t\hat{\;}1})$. Consider the function $g=f\cup\{\langle t\hat{\;}0,z\rangle,\langle t\hat{\;}1,z\rangle\}$ and observe that $g\in \mathcal F$ and $\dom[f]$ is a proper subset of $\dom[g]$, which contradicts the choice of $f$. This contradiction shows that $\dom[f]=2^{\le n}$. 

Consider the set $A=\{a\in 2^{\le n}:f(a)\cap X_a\notin\mathbf M\}$. Assuming that $A$ is not empty, let $k\le n$ be the largest number such that $A\cap 2^k\ne\emptyset$. Take any $a\in A\cap 2^k$. Definition~\ref{d:Mdef}(2) implies that $k<n$. The maximality of $k$ guarantees that for every $k\in 2$ we have $a\hat{\;}k\notin A$ and hence $f(a\hat{\;}k)\cap X_{a\hat{\;}k}\in\mathbf M$. Then $f(a)\cap X_a=\dG_{\lambda(a)}(f(a\hat{\;}0)\cap X_{a\hat{\;}0},f(a\hat{\;}1)\cap X_{a\hat{\;}1})\in\mathbf M$ because the class $\mathbf M$ is  G\"odel-closed. But $f(a)\cap X_a\in\mathbf M$ contradicts the choice of $a\in A$. This contradiction shows that $A=\emptyset$ and hence $f(a)\cap X_a\in\mathbf M$ for all $a\in 2^{\le n}$. In particular, $u\cap X=f(\emptyset)\cap X_\emptyset\in\mathbf M$.
\end{proof}

\begin{corollary}\label{c:M-set}  Let $\MM$ be an $\MM$-collective G\"odel class. A set $x$ is $\MM$-constructible if and only if $x\in \MM$.
\end{corollary}

\begin{proof} If $x\in \MM$, then $x$ is $\MM$-constructible because $x=\dG_\emptyset((X_a)_{a\in 2^0})$ for the indexed family $(X_a)_{a\in 2^0}=(x)_{a\in 2^0}=2^0\times x$. Now assume that $x$ an  $\MM$-constructible set. By the almost universality of $\MM$, there exists a set $u\in\MM$ such that $x\subseteq u$. By Proposition~\ref{p:uXM}, $x=u\cap x\in\MM$.
\end{proof}

\begin{Exercise}\label{ex:M-standard-number} Let $\MM$ be a G\"odel class. A number $n\in\w$ is called \index{number!$\MM$-standard}{\em $\MM$-standard} if for every indexed family $(X_i)_{i\in n}$ consisting of $\MM$-constructible classes, the union $\bigcup_{i\in n}X_i$ is $\MM$-constructible. Show that $0$ is an $\MM$-standard number and for every $\MM$-standard number $n\in\w$, the number $n+1$ is $\MM$-standard. Is every number $\MM$-standard?
\smallskip

\noindent{\em Hint:} This question is independent of the axioms of $\CST$. More precisely, the axioms of $\CST$ do not imply the existence of the set of $\MM$-standard numbers, because the definition of the $\MM$-standardness uses a formula with $\UU$-unbounded quantifiers, so Theorem~\ref{t:class} on the existence of classes cannot be applied.
\end{Exercise}



\section{Absoluteness}

\rightline{\em The best material model of a cat is another,}

\rightline{\em  or preferably the same, cat.}
\smallskip

\rightline{Norbert Wiener}
\medskip

Let $\mathbf M$ be an $\mathbf M$-collective G\"odel class. For a formula $\varphi$ of the Classical Set Theory, its \index{$\MM$-relativization}\index{formula!$\MM$-relativization of}{\em $\mathbf M$-relativization} $\varphi^{\mathbf M}$ is the modification of the formula $\varphi$ in which all the quantifiers $\forall$ and $\exists$ are replaced with the quantifiers $\forall^{\mathbf M}$ and $\exists^{\mathbf M}$ that have meaning ``for all $\mathbf M$-constructible classes'' and ``there exists an $\mathbf M$-constructible class'', respectively. Corollary~\ref{c:M-set} implies that  the $\MM$-relativization of the $\UU$-bounded quantifiers $\forall x\in\UU$ and $\exists x\in\UU$ are the $\MM$-bounded quantifiers $\forall x\in\MM$ and $\exists x\in\MM$, respectively. 

A formula $\varphi$ with free variables $X_1,\dots,X_n$ is called \index{$\MM$-absolute formula}\index{formula!$\MM$-absolute}{\em $\mathbf M$-absolute} if for any $\mathbf M$-constructible classes $X_1,\dots,X_n$ the formula $\varphi(X_1,\dots,X_n)$ holds if and only if $\varphi^{\mathbf M}(X_1,\dots,X_n)$ holds. 

This definition implies that formulas without quantifiers are $\mathbf M$-absolute. In particular, the atomic formula $X\in Y$ is $\mathbf M$-absolute. Next, we show that  $\Delta_0$-formulas are $\mathbf M$-absolute.

\begin{theorem}\label{t:Delta-absolute} Every $\Delta_0$-formula is $\MM$-absolute. 
\end{theorem}

\begin{proof} Since every formula is equivalent to a formula containing only existential quantifiers and the Sheffer stoke as the only logical connectives, it suffices to prove $\MM$-absoluteness for $\Delta_0$-formulas of this special form. The proof is by induction on the complexity of a $\Delta_0$-formula.
 Assume that for some natural number $m$ we have proved that every $\Delta_0$-formula of length $<m$ is $\MM$-absolute. 

Let $\varphi$ be a $\Delta_0$-formula of length $m$. If $\varphi$ is atomic, then it has no quantifiers and hence is $\MM$-absolute. 

Next, assume that $\varphi$ is of the form $(\phi)|(\psi)$ for some formulas $\phi$ and $\psi$, which are $\Delta_0$-formulas and hence are $\MM$-absolute by the inductive hypothesis. Assume that the free variables of the formula $\varphi$ belong to the list $X_1,\dots,X_n$. Fix any $\MM$-constructible classes $X_1,\dots,X_n$. Since the formulas $\phi,\psi$ are $\MM$-absolute, $\phi^\MM(X_1,\dots,X_n)\leftrightarrow \phi(X_1,\dots,X_n)$ and $\psi^\MM(X_1,\dots,X_n)\leftrightarrow \psi(X_1,\dots,X_n)$. These two equivalences imply that $\varphi^\MM(X_1,\dots,X_n)\leftrightarrow\varphi(X_1,\dots,X_n)$.

Finally, assume that $\varphi$ has form $\exists X\in Y \;(\psi)$ for some formula $\psi$ whose free variables belong to the set $X,Y,X_1,\dots,X_n$. Then the free variables of the formula $\varphi$ belong to the list $Y,X_1,\dots,X_n$. Fix any $\MM$-constructible classes $Y,X_1,\dots,X_n$.  By Proposition~\ref{p:MsubM}, $Y\cup X_1\cup\dots\cup X_n\subseteq \MM$. If the formula $\varphi(Y,X_1,\dots,X_n)$ holds, then there exists $X_0\in Y$ such that $\psi(X_0,Y,X_1,\dots,X_n)$ holds. Since $X_0\in Y\subseteq\MM$, the class $X_0$ is $\MM$-constructible and by the inductive hypothesis, the formula $\psi^\MM(X_0,Y,X_1,\dots,X_n)$ holds being equivalent to $\psi(X_0,Y,X_1,\dots,X_n)$. Since the class $X_0$ is $\MM$-constructible, the formula $\exists^\MM X\in Y$ $\psi^\MM(X,Y,X_1,\dots,X_n)$ holds and so does the formula $\varphi^\MM(Y,X_1,\dots,X_n)$.

 Next, assume that the formula $\varphi^\MM(Y,X_1,\dots,X_n)$ holds. Then there exists an $\MM$-construc\-tible class $X_0\in Y$ such that the formula $\psi^\MM(X_0,Y,X_1,\dots,X_n)$ holds and so does the formula $\psi(X_0,Y,X_1,\dots,X_n)$, by the inductive hypothesis. Since $X_0\in Y$, the formula $\exists X\in Y$ $\psi(X,Y,X_1,\dots,X_n)$ holds and so does the formula $\varphi(Y,X_1,\dots,X_n)$.
 \end{proof} 
 
Theorem~\ref{t:Delta-absolute} and Exercises~\ref{ex:X=Y-Delta}--\ref{ex:X=0} imply the following corollary providing many examples of $\MM$-absolute  formulas.

\begin{corollary} The $\Delta_0$-formulas $X\subseteq Y$, $X=Y$, $z=\{x,y\}$, $z=\langle x,y\rangle$, $Z=X\setminus Y$, $Z=X\cap Y$, $Z=X\cup Y$, $Z=X\times Y$, $Y=X^{-1}$, $Y=X^\circlearrowright$, $Y=\dom[X]$, $Y=\bigcup X$, $X=\emptyset$, are $\MM$-absolute.
\end{corollary}

\begin{exercise} Show that the $\Sigma_1$-formula $\exists Y (X\in Y)$ is $\mathbf M$-absolute.
\end{exercise}

\begin{proof} Apply Corollary~\ref{c:M-set}.
\end{proof}

Now we will show that the $\MM$-relativizations of all axioms of $\CST$  are true (the relativization of the Axiom of Infinity holds under the additional assumption that $\MM$ contains an inductive set).

\begin{exercise}\label{ex:AE} Show that the $\MM$-relativization 
$$\forall^\MM X\;\forall^\MM Y\;(X=Y\to \forall^\MM Z (X\in Z\leftrightarrow Y\in Z))$$
of the Axiom of the Equality is true.
\smallskip

\noindent{\em Hint:} Apply the Axiom of Equality and the $\MM$-absoluteness of the $\Delta_0$-formula $X=Y$.
\end{exercise}

\begin{exercise} Check that the $\MM$-relativization 
$$\forall x\in\MM\;\forall x\in\MM\;\exists z\in\MM \;(z=\{x,y\})$$
of the Axiom of Pair is true.
\smallskip

\noindent{\em Hint:} Apply the Axiom of Pair, the $\dG_8$-closedness of the class $\MM$ and the $\MM$-absoluteness of the formula $z=\{x,y\}$.
\end{exercise}

\begin{exercise}  Check that the $\MM$-relativization  
$$\exists^\MM E\; \forall z\in\MM\;(z\in E\leftrightarrow \exists x\in\MM\; \exists y\in\MM\; (z=\langle x,y\rangle\wedge x\in y))$$of the Axiom of Membership is true.
\smallskip

\noindent{\em Hint:} Observe that the class $E=(\MM\times\MM)\cap \E=\dot\Go_5(\dG_2(\MM,\MM))$ is $\MM$-constructible and has the required property by the transitivity of the class $\MM$.
\end{exercise}

\begin{exercise} Show that the $\MM$-relativization  
$$\forall^\MM X\;\exists^\MM D\;(D=\dom[X])$$of the Axiom of Domain is true.
\smallskip

\noindent{\em Hint:} Given any $\MM$-constructible class $X$, consider the class $D=\dom[X]=\dot\Go_7(X)$ and observe that it is $\MM$-constructible and has the required property by the $\MM$-absoluteness of the $\Delta_0$-formula $D=\dom[X]$.
\end{exercise}

\begin{exercise} Show that the $\MM$-relativization  
$$\forall^\MM X\;\forall^\MM Y\;\exists^\MM Z\;(Z=X\setminus Y)$$
of the Axiom of Difference is true.
\smallskip

\noindent {\em Hint:} Given any $\MM$-constructible classes $X,Y$, consider the class $Z=X\setminus Y=\dG_1(X,Y)$ and observe that it is $\MM$-constructible and has the required property by the $\MM$-absoluteness of the $\Delta_0$-formula $Z=X\setminus Y$.
\end{exercise}

\begin{exercise} Show that the $\MM$-relativization 
$$\forall^\MM X\;\forall^\MM Y\;\exists^\MM Z\;(Z=X\times Y)$$
of the Axiom of Product is true.
\smallskip

\noindent {\em Hint:} Given an $\MM$-constructible classes $X,Y$, consider the class $Z=X\times Y=\dG_2(X,Y)$ and observe that it is $\MM$-constructible and has the required property by the $\MM$-absoluteness of the $\Delta_0$-formula $Z=X\times  Y$.
\end{exercise}

\begin{exercise} Show that the $\MM$-relativization 
$$\forall^\MM X\;\exists^\MM Y\;(Y=X^{-1})$$
of the Axiom of Inversion is true.
\smallskip

\noindent {\em Hint:}  Given an $\MM$-constructible class $X$, consider the class $Y=X^{-1}=\dG_3(X)$ and observe that it is $\MM$-constructible and has the required property by the $\MM$-absoluteness of the $\Delta_0$-formula $Y=X^{-1}$.
\end{exercise}

\begin{exercise} Show that the $\MM$-relativization 
$$\forall^\MM X\;\exists^\MM Y\;(Y=X^\circlearrowright)$$
of the Axiom of Cycle is true.
\smallskip

\noindent {\em Hint:}  Given an $\MM$-constructible class $X$, consider the class $Y=X^\circlearrowright=\dG_4(X)$ and observe that it is $\MM$-constructible and has the required property by the $\MM$-absoluteness of the $\Delta_0$-formula $Y=X^\circlearrowright$.
\end{exercise}

\begin{exercise} Show that the $\MM$-relativization of the Axiom of Replacement 
$$\forall F\,\forall x\big((\forall u\forall v\forall w(\langle u,v\rangle{\in} F\wedge\langle u,w\rangle{\in} F\to v=w))\to \exists y\forall z(z\in y\leftrightarrow \exists u\exists v(u\in x\wedge \langle u,v\rangle{\in} F)\big)$$
is true.
\smallskip

\noindent{\em Hint:} Take any $\MM$-constructible class $F$ such that 
$\forall u,v,w\in\MM\;(\langle u,v\rangle\in F\wedge\langle u,w\rangle\in F\to v=w)$, which means that $F\cap(\MM\times\MM)$ is a function.
By Proposition~\ref{p:MsubM}, $F\subseteq\MM$ and hence $F\cap(\UU\times\UU)=F\cap\MM\cap(\UU\times\UU)=F\cap(\MM\times\MM)$. By the Axiom of Replacement, for every set $x\in\MM$, the class $y=F[x]=\rng[F\cap (x\times\UU)]=\rng[F\cap(x\times\MM)]$ is a set. The $\MM$-constructibility of the classes $F$ and $x$ implies the $\MM$-constructibility of the set $y=\rng[F\cap(x\times\MM)]=\dom[[F\cap (x\times\MM)]^{-1}]$. It is easy to see that the $\MM$-constructible set  $y$ has the required property: $\forall z\in \MM\;\big(z\in y\leftrightarrow \exists u\in\MM \;\exists v\in\MM\;(u\in x\wedge \langle u,v\rangle\in F)\big)$.
\end{exercise}

\begin{exercise} Show that the $\MM$-relativization  
$$\textstyle\forall x\in\MM\; \exists y\in\MM\; (y=\bigcup x)$$of the Axiom of Union is true.
\smallskip

\noindent{\em Hint:} Given any set $x\in\MM$, consider the $\MM$-constructible class $y=\bigcup x=\dot\Go_7(x)$. The Axiom of Union ensures that $y$ is a set and hence $y\in\MM$, according to Corollary~\ref{c:M-set}. The $\MM$-absoluteness of the formula $y=\bigcup x$ completes the proof.
\end{exercise}

\begin{exercise}\label{ex:AP} Show that the $\MM$-relativization  
$$\forall x\in \MM\;\exists y\in\MM\; \forall z\in\MM\;\;(z\in y\leftrightarrow z\subseteq x)$$of the Axiom of Powerset is true.
\smallskip

\noindent{\em Hint:} Given any set $x\in\MM$, it suffices to check that the set $y=\mathcal P^\MM(x)=\{z\in\MM:z\subseteq x\}$ is $\MM$-constructible. Observe that $\MM\setminus\mathcal P^\MM(x)=\{z\in\MM:\exists u\;(u\in z\wedge u\notin x)\}$. The transitivity of $\MM$ guarantees that for every $z\in\MM$ and $u\in z$ we have $u\in\MM$. Consequently, the class 
$$\begin{aligned}
\MM\setminus\mathcal P^\MM(x)&=\{z\in\MM:\exists u\in\MM\;(u\in z\wedge u\notin x)\}\\
&=\dom[\{\langle z,u\rangle\in \MM\times\MM:u\in z\wedge u\notin x\}]\\
&=\dom[((\MM\times\MM)\cap \E^{-1})\setminus (\MM\times(\MM\setminus x))]\\
&=\dot\Go_6\big(\dG_1\big(\dot\Go_3(\dG_5(\MM,\MM)),\dG_2(\MM,\dG_1(\MM,x))\big)\big)
\end{aligned}
$$ is $\MM$-constructible and so is the set $\mathcal P^\MM(x)=\MM\setminus(\MM\setminus\mathcal P^\MM(x))$. 
\end{exercise}

\begin{exercise}\label{ex:AI} Assuming that the class $\MM$ contains an inductive set $i\in \MM$, prove that the $\MM$-relativization of the Axiom of Infinity
$$\exists x\;(\emptyset\in x\;\wedge\;\forall n (n\in x\to n\cup\{n\}\in x))$$is true.
\smallskip

\noindent{\em Hint:} Observe that the inductive set $i\in\MM$ witnesses that the $\MM$-relativization of the Axiom of Infinity holds. 
\end{exercise}

\begin{exercise}\label{ex:AF} Assuming that $\MM\subseteq\mathbf V$, prove that the $\MM$-relativization 
$$\forall x\in\MM \;\big(x\ne\emptyset\to \exists y\in\MM\;((y\in x)\wedge (y\cap x=\emptyset))\big)$$
of the Axiom of Foundation is true.
\smallskip

\noindent {\em Hint:} Given a nonempty $\MM$-constructible set $x\in\MM$, we need to find an $\MM$-constructible set $y\in x$ such that $x\cap y=\emptyset$. Since $x\in\MM\subset\mathbf V$, we can apply Theorem~\ref{t:vN}(5) and find a set $y\in x$ such that $x\cap y=\emptyset$. By the transitivity of $\MM$, $y\in\MM$. The set $y$ has the required property by the $\MM$-absoluteness of the $\Delta_0$-formula
$(y\in x\wedge y\cap x=\emptyset))$.
\end{exercise}

\begin{theorem}\label{t:CST=>M} Let $\MM$ be an $\MM$-collective G\"odel class containing an inductive set $x\in \MM$. If $\CST$ is consistent, then so is $\CST+(\UU=\MM)$.
\end{theorem}

\begin{proof} If $\CST$ is consistent, then it has a model $(M,E)$ whose elements are interpretations of the undefined notion ``class'' and $E\subseteq M\times M$ is an interpretation of the undefined relation $\in$. The class $\MM$ is some element of the model $M$. In the set $M$ consider the subset $N$ whose elements are $\MM$-constructible classes in the model $(M,E)$. In particular, $\MM\in N$. 

By Exercises~\ref{ex:AE}--\ref{ex:AI}, all the axioms of $\CST$ hold in the interpretation $(N,E{\restriction}N)$, which means that  $(N,E{\restriction}N)$ is a model of $\CST$. Since every $\MM$-constructible set is an element of $\MM$, in the model $(N,E{\restriction}N)$ the class $\MM$ coincides with the class of all sets, witnessing that $\CST+(\UU=\MM)$ is consistent.
\end{proof}

In order to obtain the consistency of the equality $\UU=\mathbf L$ we need to show that the notion of a constructible set is $\mathbf L$-absolute, i.e., does not change when we instead of all classes consider only $\mathbf L$-constructible classes. The $\LL$-absoluteness of the class $\LL$ will be proved in the next section.

\section{Absoluteness of the Constructibility}

In this section we assume that $\MM\subseteq\mathbf V$ is a G\"odel class. By Proposition~\ref{p:V-cumucol}, $\MM$ is cumulative and hence $\MM$-collective. By Proposition~\ref{p:Go-Ord}, $\On\subseteq\MM$. To prove the $\MM$-absoluteness of the definition of a constructible set, we will prove the $\MM$-absoluteness of all notions composing this definion.

\begin{lemma}\label{l:On-absolute} The formulas ``$X$ is a transitive class'' and ``$X$ is an ordinal'' are $\MM$-absolute.
\end{lemma}

\begin{proof} 1. The transitivity is expressed by the $\Delta_0$-formula $\forall x\in X\;\forall y\in x\;(y\in X)$ and hence is an $\MM$-absolute notion, according to Theorem  \ref{t:Delta-absolute}.
\smallskip

2. Denote by $\varphi(x)$ the formula ``$X$ is an ordinal''. Let $X$ be any $\MM$-constructible class. If the formula $\varphi(X)$ holds, then $X\in\On\subseteq\mathbf M$ and hence $X$ is a hereditarily transitive set, according to Theorem~\ref{t:ordinal=ht}. By Exercise~\ref{ex:AF}, the $\MM$-relativization of the Axiom of Foundation is true. Then the formula $\varphi^\MM(X)$ is true by Theorem~\ref{t:ordinal=ht} and by the $\MM$-absoluteness of the transitivity. Now assume that $\varphi^\MM(X)$ holds. Then  $X$ is a hereditarily transitive set by Theorem~\ref{t:ordinal=ht} and by the $\MM$-absolutenes of the notion of a transitivity set. By Corollary~\ref{c:M-set}, $X\in\MM\subseteq\mathbf V$. By Theorem~\ref{t:ordinal=ht}, $X$ is an ordinal, so $\varphi(X)$ holds.
\end{proof}

For a clas $X=\{x\in \UU:\varphi(x,Y_1,\dots,Y_n)\}$ described by a formula $\varphi(x,Y_1,\dots,Y_n)$ with parameters $Y_1,\dots,Y_n$ by $X^\MM$ we denote the class $\{x\in \MM:\varphi^\MM(x,Y_1,\dots,Y_n)\}$, called the {\em $\MM$-relativization} of the class $X$. Of course, the $\MM$-relativization depends on the formula $\varphi$ that defined $X$, so this notion will be applied to classes that have some fixed definitions, like $\On$, $\mathbf L$ or $\mathbf V$.

Lemma~\ref{l:On-absolute} implies that  $\On^{\MM}=\On$. 

\begin{exercise} Prove that for every $i\in 9$ the $\MM$-relativization $\dG_i^\MM$ of the G\"odel's operation $\dG_i:\UU\times\UU\to\UU$ coincides with the restriction $\dG_i{\restriction}_{\MM\times\MM}:\MM\times\MM\to\MM$.
\smallskip

\noindent{\em Hint:} This fact follows from the $\MM$-absoluteness of the $\Delta_0$-formulas $Z=X\setminus Y$, $Z=X\times Y$, $Z=X^{-1}$, $Z=X^\circlearrowright$, $Z=X\cap \E$, $Z=\dom[X]$, $Z=\bigcup X$, $z=\{x,y\}$.
\end{exercise}

\begin{exercise} Prove that the $\MM$-relativization of the function $\Go:\UU\to\UU$ coincides with the restriction $\Go{\restriction}_\MM:\Go\to\Go$. 
\end{exercise}

\begin{exercise} Prove that the $\MM$-relativization of the function $\Go^{\circ\w}:\UU\to\UU$ coincides with the restriction $\Go^{\circ\w}{\restriction}_\MM:\MM\to\MM$.
\smallskip

\noindent{\em Hint:} Use the fact that the ordinal $\w$ is $\MM$-absolute (as the smallest nonzero limit ordinal).
\end{exercise}

Let us recall that $\mathbf L=\bigcup_{\alpha\in\On}L_\alpha$ where $L_\alpha=\bigcup_{\beta\in\alpha}\mathcal P(L_\beta)\cap\Go^{\circ\w}(L_\alpha\cup\{L_\alpha\})$ for every ordinal $\alpha$.

\begin{lemma}\label{l:La-absolute} If\/ $\mathbf L\subseteq\MM$, then $\forall \alpha\in\On\;(L_\alpha^\MM=L_\alpha)$.
\end{lemma}

\begin{proof} The lemma will be proved by transfinite induction. Assume that for some ordinal $\alpha$ we know that $\forall\beta\in\alpha\;(L_\beta^\MM=L_\beta)$. If $\alpha=0$, then $L_\alpha^\MM=\emptyset^\MM=\emptyset=L_\alpha$. If $\alpha$ is a limit ordinal, then $L_\alpha^\MM=\bigcup_{\beta\in \alpha}L_\beta^\MM=\bigcup_{\beta\in\alpha}L_\beta=L_\alpha$ by the inductive hypothesis. Finally assume that $\alpha=\beta+1$ for some ordinal $\beta$. The inductive assumption ensures that $L^\MM_\beta=L_\beta$ and hence  $L_\beta^\MM\cup\{L_\beta^\MM\}=L_\beta\cup \{L_\beta\}\subseteq\mathbf L\subseteq\MM$ and $\Go^{\circ\w}(L_\beta\cup\{L_\beta\})\subseteq\mathbf L\subseteq\MM$. Then   
$$
L^\MM_\alpha=\mathcal P^\MM(L_\beta^\MM)\cap\Go^{\circ\w}{\restriction}_\MM(L_\beta^\MM\cup\{L_\beta^\MM\})=\mathcal P(L_\alpha)\cap\MM\cap \Go^{\circ\w}(L_\beta\cup\{L_\beta\})=L_\alpha.
$$
\end{proof}

\begin{Exercise} Prove that $\mathbf L\subseteq\MM$ and hence $\mathbf L^\MM=\mathbf L$ for every G\"odel class  $\MM\subseteq\mathbf V$.
\smallskip

\noindent{\em Hint:} Use the fact that $\On\subseteq\MM$ and every set $L_\alpha$ has a canonical well-order, so is the image of a suitable ordinal under some basic function.
\end{Exercise}

Finally we can prove the promised consistency of the Axiom of Constructibility $(\UU=\mathbf L)$.

\begin{theorem}\label{t:Con(U=L)}  If $\CST$ is consistent, then so is $\CST+(\UU=\LL)$.
\end{theorem}

\begin{proof} If $\CST$ is consistent, then it has a model $(M,E)$ whose elements are interpretations of the undefined notion ``class'' and $E\subseteq M\times M$ is an interpretation of the indefined relation $\in$. Let $\mathbf L\in M$ be the class of constructible sets in the model $(M,E)$. In the set $M$ consider the subset $N$ whose elements are $\mathbf L$-constructible classes in the model $(M,E)$. In particular, $\mathbf L\in N$. 

By Exercises~\ref{ex:AE}--\ref{ex:AI}, all the axioms of $\CST$ hold in the interpretation $(N,E{\restriction}N)$, which means that  $(N,E{\restriction}N)$ is a model of $\CST$.
Moreover, Lemma~\ref{l:La-absolute} implies that the class $\mathbf L$ coincides with the class of constructible sets in the sense of the model $(N,E{\restriction}N)$. Since every $\mathbf L$-constructible set is an element of $\mathbf L$, in the model $(N,E{\restriction}N)$ the class of all sets coincides with the class $\mathbf L$ of all constructible sets, witnessing that $\CST+(\UU=\mathbf L)$ is consistent.
\end{proof}

\begin{Exercise} Prove that it is consistent to assume that every class is $\mathbf L$-constructible (which implies $\UU=\mathbf L$).
\smallskip

\noindent{\em Hint:} Mimic the proof of Theorem~\ref{t:Con(U=L)} but take for $N$ the subset of $M$ consisting of all classes that can be constructed from the class $\mathbf L$ and its elements by applications of finitely many G\"odel's operations $\dG_0$--$\dG_7$, where ``finitely many'' is understood in the sense of metalanguage. In general such a model $N$ will be smaller comparing to the model consisting of all $\mathbf L$-constructible classes and all elements of $N$ will remain $\mathbf L$-constructible classes in the sense of the model $(N,E{\restriction}N)$.
\end{Exercise}

\part{Choice and Global Choice}

The Axiom of Choice is the most controversial axiom in mathematics.
It has many valuable implications (Tychonoff's compactness theorem in Topology, Hahn-Banach Theorem in Functional Analysis) but it also implies some highly counter-intuitive statements like the Banach-Tarski Paradox\footnote{{\bf Exercise:} Read about Banach-Tarski Paradox in Wikipedia.}. In this part we survey some implications or equivalents of Axiom of Choice and its stronger version, the Axiom of Global Choice.

\section{Choice}\label{s:choice}

\rightline{\em The axiom of choice is obviously true,}

\rightline{\em  the well-ordering principle obviously false,}

\rightline{\em and who can tell about Zorn's lemma?}
\smallskip

\rightline{Jerry Lloyd Bona}

\vskip20pt

In this section we discuss some statements related to the Axiom of Choice. In fact, this subject is immense and there are good and complete books covering this topic in many details, see for example, \cite{JechAC}, \cite{Herrlich}, \cite{HR},  \cite{Moore}. So, we shall recall only the most important choice principles that have applications in mathematics.  


\begin{definition} We say that a class $X$ 
\begin{itemize} 
\item can be \index{well-ordered class}\index{class!well-ordered} {\em well-ordered}  if there exists a well-order $W$ such that $\dom[W^\pm]=X$;
\item is \index{class!well-orderable}\index{well-orderable class}{\em well-orderable} if there exists a set-like well-order $W$ such that $\dom[W^\pm]=X$;
\item has a \index{choice function}\index{function!choice}{\em choice function} if there exists a function $F:X\setminus\{\emptyset\}\to \bigcup X$ such that $F(x)\in x$ for every nonempty set $x\in X$; the function $F$ is called a {\em choice function} for $X$.
\end{itemize}
\end{definition}
Observe that a set can be well-ordered if and only if it is well-orderable. 

We recall that the {\bf Axiom of Choice} postulates that each set has a choice function.

The following fundamental result is known as the well-ordering theorem of Zermelo.

\begin{theorem}[Zermelo]\label{t:WO} For any set $x$ the following statements are equivalent.
\begin{enumerate} \setlength{\itemindent}{12pt}
\item[$\mathsf{WO}(x){:}$] The set $x$ can be well-ordered.
\item[$\mathsf{AC}(\mathcal Px){:}$] The power-set $\mathcal P(x)$ of $x$ has a choice function.
\end{enumerate}
\end{theorem}

\begin{proof} $\mathsf{AC}(\mathcal Px)\Ra\mathsf{WO}(x)$: Assume that there exists a choice function $c:\mathcal P(x)\setminus \{\emptyset\}\to x$ for the power-set $\mathcal P(x)$ of $x$.   Since $\UU$ is a proper class, there exists an element $z\in\UU\setminus  x$. Consider the function $\bar c=\{\langle\emptyset,z\rangle\}\cup c$.

Applying the Recursion Theorem~\ref{t:Recursion} to the function $$\textstyle F:\Ord\times\UU\to\{z\}\cup x,\quad F:\langle \alpha,y\rangle\mapsto \bar c(x\setminus \rng[y]),$$
and the set-like well-order $\E{\restriction}\Ord$, we obtain a (unique) function $G:\Ord\to\{z\}\cup  x$ such that $$G(\alpha)=F(\alpha,G{\restriction}_\alpha)$$ for every $\alpha\in\Ord$.   

We claim that $z\in G[\Ord]$. To derive a contradiction assume that $z\notin G[\Ord]$.  In this case for every ordinals $\beta\in\alpha$ we have
$$z\ne G(\alpha)=F(\alpha,G{\restriction}_\alpha)=c(x\setminus G[\alpha])\in x\setminus G[\alpha]\subseteq x\setminus\{G(\beta)\}$$
and hence $G(\beta)\ne G(\alpha)$. The injectivity of the function $G$ guarantees that $G^{-1}$ is a function, too. Then $\Ord=G^{-1}[x]$ is a set by the Axiom of Replacement. But this contradicts Theorem~\ref{t:Ord}(6).
This contradiction shows that $z\in G[\Ord]$. Since the order $E{\restriction}\Ord$ is well-founded, for the nonempty set $G^{-1}[\{z\}]\subseteq\Ord$ there exists  an ordinal  $\alpha\in G^{-1}[\{z\}]$ such that $\alpha\cap G^{-1}[\{z\}]=\emptyset$. Then $z\notin G[\alpha]$ and hence $G[\alpha]\subseteq x$. Repeating the above argument, we can prove that the function $G{\restriction}_\alpha$ is injective. 

We claim that $G[\alpha]=x$. Assuming that $G[\alpha]\ne x$, we see that the set $x\setminus G[\alpha]$ is not empty and then the definition of the function $F$ ensures that $G(\alpha)=F(\alpha,G{\restriction}_\alpha)=c(x\setminus G[\alpha])\in x\setminus G[\alpha]\subseteq \bigcup x$ and hence $G(\alpha)\ne z$, which contradicts the choice of $\alpha$.

Therefore, the function $G{\restriction}_\alpha:\alpha\to x$ is bijective and we can define an irreflexive well-order on $x$ by the formula
$$w=\{\langle G(\beta),G(\alpha)\rangle:\beta\in\alpha\}.$$
\smallskip

$\mathsf{WO}(x)\Ra\mathsf{AC}(\mathcal Px):$  If there exists a well-order $w$ with $\dom[w^\pm]=x$, then the formula
$$c:\mathcal P(x)\setminus \{\emptyset\}\to  x,\quad c:a\mapsto {\min}_w(a)$$determines a choice function for $\mathcal P(x)$. In this formula by $\min_w(a)$ we denote the unique $w$-minimal element of a nonempty set $a\subseteq x$.
\end{proof}

An important statement which is equivalent to the Axiom of Choice was found by Kuratowski in 1922 and (independently) Zorn in 1935. It concerns the existence of maximal elements in orders.

Let us recall that for an order $R$, an element $x\in\dom[R^\pm]$ is called {\em $R$-maximal} if $$\forall y\in \dom[R^\pm]\;(\langle x,y\rangle\in R\;\Rightarrow\;y=x).$$

A subclass $L\subseteq \dom[R^\pm]$ is called 
\begin{itemize}
\item an \index{chain}{\em $R$-chain} if $L\times L\subseteq R^\pm\cup\Id$;
\item an \index{antichain}{\em $R$-antichain} if $(L\times L)\cap R\subseteq\Id$;
\item a \index{maximal chain}{\em maximal $R$-chain} if $L$ is an $R$-chain and $L$ is equal to any $R$-chain $L'\subseteq \dom[R^\pm]$ with $L\subseteq L'$;
\item a \index{maximal antichain}{\em maximal $R$-antichain} if $L$ is an $R$-antichain and  $L$ is equal to any $R$-antichain $L'\subseteq \dom[R^\pm]$ with $L\subseteq L'$.
\end{itemize}
An element $b\in\dom[R^\pm]$ is called an \index{upper bound}{\em upper bound} of a set $L\subseteq\dom[R^\pm]$ if $L\times\{b\}\subseteq R\cup\Id$.

We say that an order $R$ is \index{chain-bounded order}\index{order!chain-bounded}{\em chain-bounded} if each $R$-chain $L\subseteq \dom[R^\pm]$ has an upper bound in $\dom[R^\pm]$.

\begin{lemma}[Kuratowski--Zorn]\label{l:KZL}\index{Kuratowski--Zorn Lemma} Let $r\in\UU$ be a chain-bounded order on a nonempty set $x=\dom[r^\pm]$. If the power-set $\mathcal P(x)$ of $x$ has a choice function, then there exists an $R$-maximal element $z\in\dom[r^\pm]$.
\end{lemma}

\begin{proof} To derive a contradiction, assume that $\dom[r^\pm]$ contains no $r$-maximal elements.

Let $c$ be the set of all $r$-chains in the set $x=\dom[r^\pm]$. Since the order $r$ is chain-bounded, for every chain $\ell\in c$ the set $u(\ell)=\{b\in\dom[r^\pm]:\ell\times\{b\}\subseteq r\cup\Id\}$ of its upper bounds is not empty. We claim that the subset $v(\ell)=\{b\in\dom[r^\pm]:\ell\times\{b\}\subseteq r\setminus\Id\}$ of $u(\ell)$ is not empty, too. For this take any element $b\in u(\ell)$. By our assumption, the element $b$ is not $r$-maximal. Consequently, there exists an element $d\in\dom[r^\pm]$ such that $\langle b,d\rangle\in r\setminus\Id$. The transitivity of the relation $r$ and the inequality $b\ne d$ guarantees that $d\notin \ell$ and hence $d\in v(\ell)$.

By our assumption, the power-set $\mathcal P(x)$ has a choice function $f:\mathcal P(x)\setminus\{\emptyset\}\to x$. 
Let $z\in\UU\setminus \dom[r^\pm]$ be any set (which exists as $\dom[r^\pm]$ is a set and $\UU$ is a proper class).  Consider the function $F:\Ord\times\UU\to\UU$ defined by the formula
$$F(\alpha,y)=\begin{cases} f(v(\rng[y]))&\mbox{if $\rng[y]\in c$};\\
z&\mbox{otherwise}.
\end{cases}
$$By the Recursion Theorem~\ref{t:Recursion}, there exists a (unique) function $G:\Ord\to \UU$ such that $G(\alpha)=F(\alpha,G{\restriction}_\alpha)$ for every ordinal $\alpha$.

We claim that for every ordinal $\alpha$ the image $G[\alpha]$ is an $r$-chain. Assuming that this is not true, we can find the smallest ordinal $\alpha$ such that $G[\alpha]$ is not an $r$-chain but for every $\beta\in \alpha$ the set $G[\beta]$ is an $r$-chain. Since $G[\alpha]$ is not an $r$-chain, there two elements $y,z\in G[\alpha]$ such that $\langle y,z\rangle\notin r^\pm\cup\Id$. Find two ordinals $\beta,\gamma\in\alpha$ such that $y=G(\beta)$ and $z=G(\gamma)$. Since the relation $r^\pm\cup\Id$ is symmetric, we lose no generality assuming that $\beta\le\gamma$. We claim that  $\gamma+1=\alpha$. In the opposite case, the minimality of $\alpha$ guarantees that $G[\gamma+1]$ is an $r$-chain and then $\langle y,z\rangle\in r^\pm\cup\Id$, which contradicts the choice of the elements $y,z$.  Therefore, $\alpha=\gamma+1$. Since $G[\gamma]$ is an $r$-chain, the definition of the function $F$ guarantees that $G(\gamma)=F(\gamma,G{\restriction}_\gamma)=f(v(G[\gamma]))$ and then $G[\alpha]=G[\gamma]\cup G(\gamma)=G[\gamma]\cup f(v(G[\gamma]))$ is an $r$-chain. But this contradicts the choice of $\alpha$. This contradiction shows that for all ordinals $\alpha$ the set $G[\alpha]$ is an $r$-chain. In this case for every ordinal $\alpha$ we have $G(\alpha)=F(\alpha,G{\restriction}_\alpha)=f(v(G[\alpha]))\in v(G[\alpha])\subseteq \bigcup x\setminus G[\alpha]$, which implies that the function $G:\Ord\to\dom[r^\pm]$ is injective. Then $G^{-1}$ is a function and $\Ord=G^{-1}[\dom[r^\pm]]$ is a set by the Axiom of Replacement. But this contradicts the properness of the class $\Ord$, see Theorem~\ref{t:Ord}(6).
\end{proof}

Another statement, which is equivalent to the Axiom of Choice is

\begin{center}
\index{Hausdorff's Maximality Principle ({{\sf MP}})}{\em  Hausdorff's Maximality Principle}
\smallskip

\index{({{\sf MP}})}\fbox{$\mbox{$(\mathsf{MP})${:}\;\; For every  order $r\in\UU$, the set $\dom[r^\pm]$ contains a maximal $r$-chain.}$}
\end{center}
\smallskip

Hausdorff's Maximality Principle restricted to ordinary trees is called the \index{Principle of Tree Choice ({{\sf TC}})}{\em Principle of Tree Choice} and is denoted by $(\mathsf{TC})$. 

Let us recall that an order $R$ is called a \index{tree-order}\index{order!tree-order}{\em tree-order} if for every $x\in\dom[R^\pm]$ the initial interval $\cev R(t)$ is well-ordered by the relation $R{\restriction}\cev R(x)$. A standard example of a tree-order is the order $\mathbf{S}{\restriction}\UU^{<\Ord}$ where $\mathbf S=\{\langle x,y\rangle:x\subseteq y\}$ and $\UU^{<\Ord}$ is the  class of all function $f$ with $\dom[f]\in\Ord$. An \index{ordinary tree}\index{tree!ordinary}{\em ordinary tree} is a subclass $T\subseteq \UU^{<\Ord}$ such that for every function $t\in T$ and ordinal $\alpha$ the function $t{\restriction}_\alpha=t\cap(\alpha\times\UU)$ belongs to $T$. For an ordinary tree $T$, a subclass $C\subseteq T$ is called a \index{maximal chain}\index{chain!maximal}({\em maximal}) {\em chain} if it is a (maximal) $\mathbf S{\restriction}T$-chain.
\smallskip

A set $\mathcal F$ is called {\em of finite character} if a set $x$ is an element of $\mathcal F$ if and only if every finite subset of $x$ is an element of $\mathcal F$. An important implication  Axiom of Choice was found by Oswald Teichm\"uller (1939) and John Tukey (1940):

\begin{center}
\index{Teichm\"uller--Tukey Lemma ({{\sf TT}})}{\em  Teichm\"uller--Tukey Lemma}
\smallskip

\index{({{\sf TT}})}\fbox{$\mbox{$(\mathsf{TT})${:}\;\; Every set $\mathcal F$ of finite character contains an $\mathbf S{\restriction}\mathcal F$-maximal element.}$}
\end{center}
\smallskip

Now we list some statements that are equivalent to the Axiom of Choice.

\begin{theorem}\label{t:mainAC} The following statements are equivalent:
\begin{enumerate}
\index{({{\sf AC}})}\index{Choice Principle!({{\sf AC}})}\item[$(\mathsf{AC})$] Every set  has a choice function.
\index{({{\sf KZ}})}\index{Choice Principle!({{\sf KZ}})}\item[$(\mathsf{KZ})$] Every chain-bounded order $r\in\UU$ has an $r$-maximal element.
\index{({{\sf MP}})}\index{Choice Principle!({{\sf MP}})}\item[$(\mathsf{MP})$] For every order $r\in\UU$, the set $\dom[r^\pm]$ contains an $r$-maximal chain.
\index{({{\sf TT}})}\index{Choice Principle!({{\sf TT}})}\item[$(\mathsf{TT})$] Every set $\mathcal F$ of finite character contains an $\mathbf S{\restriction}\mathcal F$-maximal element.
\index{({{\sf TC}})}\index{Choice Principle!({{\sf TC}})}\item[$(\mathsf{TC})$] Every ordinary tree $t\in\UU$ contains a maximal chain.
\index{({{\sf WO}})}\index{Choice Principle!({{\sf WO}})}\item[$(\mathsf{WO})$] Every set $x$ can be well-ordered.
\index{Choice Principle!$\Pi$}\item[$(\Pi)$] For any set $a$ and indexed family of nonempty sets $(x_\alpha)_{\alpha\in a}$,\\ the Cartesian product $\prod_{\alpha\in a}x_\alpha$ is not empty.
\end{enumerate}
\end{theorem}

\begin{proof} The 
 implication $(\mathsf{AC}\Ra\mathsf{KZ})$ has been proved in the Kuratowski--Zorn Lemma~\ref{l:KZL}.
\smallskip

$(\mathsf{KZ})\Ra(\mathsf{MP})$: Fix an order $r\in\UU$ and consider the set $$c=\{\ell\subseteq\dom[r^\pm]:\mbox{$\ell$ is an $r$-chain}\},$$ endowed with the partial order $\mathbf S{\restriction}c$. It is easy to see that for any $\mathbf S{\restriction}c$-chain $c'\subseteq c$, the union $\bigcup c'$ is an $r$-chain, which is an upper bound of the chain $c'$ in the set $c$ endowed with the partial order $\mathbf S{\restriction}c$. This means that the partial order $\mathbf S{\restriction}c$ is chain-bounded. By $(\mathsf{KZ})$, there exists an $\mathbf S{\restriction}c$-maximal element $c'\in c$, which is a required maximal $r$-chain in $\dom[r^\pm]$.
\smallskip

$(\mathsf{MP})\Ra(\mathsf{TT})$: Given a set $\mathcal F$ of finite character, use $(\mathsf{MP})$ and find a maximal $\mathbf S{\restriction}\mathcal F$-chain $C\subseteq\mathcal F$. We claim that the set $x=\bigcup C$ is $\mathbf S{\restriction}C$-maximal element of $\mathcal F$. To see that $x\in\mathcal F$, it suffices to check that every finite subset $y\subseteq x$ belongs to $\mathcal F$. For every $z\in y$ find a set $c_z\in C$ with $z\in c_z$. Then $\{c_z:z\in y\}$ is a finite subset of the chain $C$ and there exists $u\in y$ such that $c_z\subseteq c_u$ for all $z\in y$. Then $y\subseteq \bigcup_{z\in y}c_z=c_u\in C\subseteq\mathcal F$ and hence $y\in\mathcal F$ (because $\mathcal F$ has finite character). Therefore, $x=\bigcup C\in\mathcal F$. Assuming that the set $x\in\mathcal F$ is not $\mathcal S{\restriction}\mathcal F$-maximal, we can find a set $f\in\mathcal F$ such that $x\subset f$. Then $C\cup\{f\}$ is an $\mathbf S{\restriction}\mathcal F$-chain, which strictly contains $C$. But this contradicts the maximality of $C$.
\smallskip

$(\mathsf{TT})\Ra(\mathsf{TC})$: Given an ordinary tree $t$, observe that the family $\mathcal F$ of chains in $t$ is of finite character. By $(\mathsf{TT})$, this family contains an $\mathcal S{\restriction}\mathcal F$-maximal element $c\in\mathcal F$, which is a maximal chain in the ordinary tree $t$.
\smallskip

$(\mathsf{TC})\Ra(\mathsf{WO})$: Given any set $x$, consider the ordinary tree $$
T=\{f\in\UU^{<\Ord}:f^{-1}\in\mathbf{Fun}\;\wedge\;\rng[f]\subseteq x\}$$ consisting of all injective functions $f$ with $\dom[f]\in\Ord$ and $\rng[f]\subseteq x$. By the Hartogs Theorem~\ref{t:Hartogs}, there exists an ordinal $\alpha$ (equal to $\rank[WO(x)]$) that admits no injective function $f:\alpha\to x$.
Since $T\subseteq \mathcal P(\alpha\times x)$, the tree $T$ is a set. By $(\mathsf{TC})$, the ordinary tree $T$ contains a maximal chain $c\subseteq T$. It is easy to see that the union of this chain $f=\bigcup c$ is an injective function with $\dom[f]\in\Ord$ and $\rng[f]\subseteq x$. So, $f\in T$. We claim that $\rng[f]=x$.
Assuming that $x\ne\rng[f]$, take any element $y\in x\setminus\rng[f]$ and consider the function $\bar f=f\cup\{\langle\dom[f],y\rangle\}\in T$. The maximality of the chain $c$ guarantees that $c=c\cup\{\bar f\}$ and hence $\bar f\in c$. Then $\langle \dom[f],y\rangle\in \bar f\subseteq\bigcup c=f$ and $\dom[f]\in\dom[f]\in\Ord$, which contradicts the definition of an ordinal. Therefore, $f:\dom[f]\to x$ is a bijective function and 
$$w=\{\langle f(\gamma),f(\beta)\rangle:\gamma\in\beta\in\dom[f]\}$$is a well-order $w$ with $\dom[w^\pm]=x$, witnessing that the set $x$ is well-ordered. 
\smallskip

$(\mathsf{WO}\Ra\Pi)$: Let $a$ be a set and $(x_\alpha)_{\alpha\in a}$ be an indexed family of sets. By the Axiom of Union, the class $x=\bigcup_{\alpha\in A}x_\alpha$ is a set. By $(\mathsf {WO})$, the set $x$ admits a well-order $w$ with $\dom[w^\pm]=x$. Now consider the function $f:a\to x$ assigning to every $\alpha\in a$ the unique $w$-minimal element $\min_w(x_\alpha)$ of the set $x_\alpha$.
Then $f\in\prod_{\alpha\in a}x_a$, witnessing that the Cartesian product $\prod_{\alpha\in a}x_a$ is not empty.
\smallskip

$(\Pi)\Ra(\mathsf{AC})$: By $(\Pi)$, for any set $a$, the Cartesian product $\prod_{x\in a\setminus\{\emptyset\}}x$ contains some function $f:a\setminus\{\emptyset\}\to\bigcup a$, which is a choice function for the set $a$.
\end{proof}

Now we list some statements that are equivalent to the Axiom of Choice at presence of the Axiom of Foundation.

\begin{theorem}\label{t:AC+UV} Consider the following statements:
\begin{itemize}
\item[$(\mathsf{MC})$] {\sf Axiom of Multiple Choice.} For any set $x$ there exists a function $f:x\setminus\{\emptyset\}\to \UU$ assigning to each non-empty set $y\in x$ a nonempty finite set $f(y)\subseteq y$.
\item[$(\mathsf{AP})$] {\sf Antichain Principle.} Every order $r\in \UU$ has a maximal antichain $A\subseteq\dom[r^\pm]$.
\item[$(\mathsf{WL})$] Every linearly ordered set can be well-ordered.
\item[$(\mathsf{WP})$] The power-set of any well-ordered set can be well-ordered.
\item[$(\mathsf{WV})$] Every set $x\in\VV$ can be well-ordered.
\end{itemize}
Then $(\mathsf{AC})\Ra(\mathsf{MC})\Ra(\mathsf{AP})\Ra(\mathsf{WL})\Ra(\mathsf{WP})\Leftrightarrow(\mathsf{WV})$. If the Axiom of Foundation holds, then $(\mathsf{AC})\Leftrightarrow(\mathsf{MC})\Leftrightarrow(\mathsf{AP})\Leftrightarrow(\mathsf{WL})\Leftrightarrow(\mathsf{WP})\Leftrightarrow(\mathsf{WV})=(\mathsf{WO})$.
\end{theorem}

\begin{proof} The implication $(\mathsf{AC})\Ra(\mathsf{MC})$ is trivial. 
\smallskip

$(\mathsf{MC})\Ra(\mathsf{AP})$ Let $r\in\UU$ be any order and $x=\dom[r^\pm]$.
By $(\mathsf{MC})$, there exists a function $f:\mathcal P(x)\to\mathcal P(x)$ assigning to each (nonempty) subset $y\subseteq x$ a (nonempty) finite set $f(y)\subseteq y$. For every set $y$ let $$c(y)=\{z\in x:(\{z\}\times y)\cap (r^\pm\cup \Id)=\emptyset\}$$be the set of elements of $x$ that are $r$-incomparable with elements of the set $y$. Also let 
$$\mu(y)=\{z\in y:(y\times\{z\})\cap r\subseteq\Id\}$$ be the set of $r$-minimal elements of the set $y$. Consider the function
$$F:\On\times\UU\to\UU,\quad F(\alpha,y)=\textstyle(\bigcup \rng[y])\cup \mu(f(c(\bigcup \rng[y]))).$$ By the Recursion Theorem~\ref{t:Recursion}, there exists a function $G:\On\to\UU$ such that $G(\alpha)=F(\alpha,G{\restriction}_\alpha)$ for every ordinal $\alpha$. By transfinite induction it can be shown that for every ordinal $\alpha$ the set $G(\alpha)$ is an $r$-antichain in $x$. Since the transfinite sequence of sets $(G(\alpha))_{\alpha\in\On}$ is non-decreasing, there exists an ordinal $\alpha$ such that $G(\alpha+1)=G(\alpha)$. We claim that $G(\alpha)$ is a maximal $r$-antichain in $x$. In the opposite case the set $c(G(\alpha))$ is not empty and so are the sets $f(c(G(\alpha))$ and $\mu(f(c(G_\alpha)))$. Then 
\begin{multline*}
G(\alpha+1)=F(\alpha+1,G{\restriction}_{\alpha+1})=\\
\textstyle\bigcup\{G(\beta)\}_{\beta\le\alpha}\cup\mu(f(c(\bigcup\{G(\beta)\}_{\beta\le \alpha})))=G(\alpha)\cup\mu(f(c(G(\alpha)))\ne G(\alpha),
\end{multline*}
which contradicts the choice of $\alpha$. Therefore, $G(\alpha)=G[\On]$ is a maximal $r$-antichain for the order $r$.
\smallskip

$(\mathsf{AP})\Ra(\mathsf{WL})$ Given any linear order $l\in\UU$ we should prove that the set $x=\dom[l^\pm]$ can be well-ordered. Consider the order 
$$r=\{\langle \langle a,y\rangle,\langle b,z\rangle\rangle\in(\mathcal P(x)\times x)\times(\mathcal P(x)\times x):a=b\;\wedge\;y\in a\;\wedge\;z\in b\;\wedge\;\langle y,z\rangle\in l\cup\Id\}$$ on the set $\mathcal P(x)\times x$. By $(\mathsf{AP})$,  there exists a maximal $r$-antichain $A\subseteq P(x)\times x$. The maximality and the antichain property of $A$ ensures that for every nonempty set $a\subseteq x$ there exists a unique element $f(a)\in a$ such that $\langle a,f(a)\rangle\in A$. Then $f$ is a choice function for the power-set $\mathcal P(x)$. By the Zermelo Theorem~\ref{t:WO}, the set $x$ can be well-ordered.
\smallskip

$(\mathsf{WL})\Ra(\mathsf{WP})$ Let $x$ be a well-ordered set and $w$ be a well-order such that $x=\dom[w^\pm]$. Then the power-set $\mathcal P(x)$ carries the lexicographic linear order $l$ defined by
$$l=\{\langle a,b\rangle\in\mathcal P(x)\times\mathcal P(x):\exists \alpha\in b\setminus a\;\wedge\;\forall \beta\in x\;(\langle\beta,\alpha\rangle\in w\setminus\Id\;\Ra\;(\beta\in a\;\Leftrightarrow\;\beta\in b))\}.$$ By $(\mathsf{WL})$, the linearly ordered set $\mathcal P(x)$ can be well-ordered.
\smallskip

$(\mathsf{WP})\Ra(\mathsf{WV})$ Since every set $x\in \VV$ is contained in some set $V_\alpha$, it suffices to prove that for every ordinal $\beta$ the set $V_\beta$ can be well-ordered. By the Hartogs Theorem \ref{t:Hartogs}, for every ordinal $\beta$ there exists an ordinal $\gamma$ that does not admit an injective function into $V_\beta$. By $(\mathsf{WP})$, the power-set $\mathcal P(\gamma)$ admits a well-order $w$. For every ordinal $\alpha\le \beta$, define a well-order $w_\alpha$ on the set $V_\alpha$ by the recursive formula:
$$w_\alpha=\bigcup_{\delta\in\alpha}((V_\delta\times (V_{\alpha}\setminus V_\delta))\cup\{\langle y,z\rangle\in (\mathcal P(V_\delta)\setminus V_\delta)\times\mathcal (\mathcal P(V_\delta)\setminus V_\delta):\langle \rank_{w_\delta}[y],\rank_{w_\delta}[z]\rangle\in w\}.$$
Then $w_\beta$ is a well-order of the set $V_\beta$.
\smallskip

The implication $(\mathsf{WV})\Ra(\mathsf{WP})$ follows from Theorem~\ref{t:wOrd} and the fact that the class $\VV$ contains all ordinals and is closed under taking the power-set.
\smallskip

If the Axiom of Foundation holds, then $\UU=\VV$ (by Theorem~\ref{t:vN}) and the statement $(\mathsf{WV})$ coincides with the statement $(\mathsf{WO})$, which is equivalent to the Axiom of Choice by Zermelo Theorem~\ref{t:WO}. In this case the chain of implications  $(\mathsf{AC})\Ra(\mathsf{MC})\Ra(\mathsf{AP})\Ra(\mathsf{WL})\Ra(\mathsf{WP})\Leftrightarrow(\mathsf{WV})=(\mathsf{WO})$ turns into a chain of equivalences.
\end{proof}

Now we survey some statements, which are weaker than the Axiom of Choice. We start with an application of the Kuratowski-Zorn Lemma to ultrafilters.

\begin{definition}[Cartan] A class  $F$ is called \index{filter}{\em a filter} if the following conditions are satisfied:
\begin{enumerate}
\item $\emptyset\notin F$;
\item $\forall x\in F\;\forall y\in F\;(x\cap y\in F)$;
\item $\forall x\;\forall y\;((x\in F\;\wedge\;x\subseteq y\subseteq\bigcup F)\;\Rightarrow\;y\in F)$.
\end{enumerate}
A filter $F$ is called a {\em filter on a class} $X$ if $X=\bigcup F$.\\
\smallskip
A filter $F$ is called an \index{ultrafilter}{\em ultrafilter} if $F$ is equal to any filter $F'$ such that $F\subseteq F'$ and $\bigcup F=\bigcup F'$.
\end{definition}

\begin{example} The family $F=\{x\in\mathcal P(\w):\exists n\in\w\;(\w\setminus x\subseteq n)\}$ is a filter, called the {\em Fr\'echet filter} on $\w$.
\end{example}

The Kuratowski-Zorn Lemma has the following implication (the statement $\mathsf{UF}$ below is called the \index{Ultrafilter Lemma ({{\sf UF}})}{\em Ultrafilter Lemma}).

\begin{lemma}\label{c:UF} The Axiom of Choice implies the following statement:
\begin{itemize}
\index{({{\sf UF}})}\index{Choice Principle!({{\sf UF}})}\item[$(\mathsf{UF}){:}$] Each filter $\varphi$ on a set $x$ is a subset of some ultrafilter $u$ on $x$.
\end{itemize}
\end{lemma}

\begin{proof} Given a filter $\varphi$ on a set $x$, consider the set $\hat\varphi$ of all filters on $x$ that contain the filter $\varphi$ as a subset. The set $\hat\varphi$ is a subset of the double exponent $\mathcal P(\mathcal P(x))$, so it exists by the Axiom of Power-Set. By the Kuratowski-Zorn Lemma, the set $\hat\varphi$ endowed with the partial order $\mathbf S{\restriction}\hat\varphi$ contains an $\mathbf S{\restriction}\hat\varphi$-maximal element $u$, which is the required ultrafilter on $x$ that contains $\varphi$.
\end{proof}

It is known \cite{JechAC} that the statement $(\mathsf{UF})$ appearing in Lemma~\ref{c:UF} is not equivalent to the Axiom of Choice. Nonetheless, it implies that each set admits a linear order.

\begin{proposition}\label{p:UF=>LO} The statement $(\mathsf{UF})$ implies the following weak version of $(\mathsf{WO})$:
\begin{itemize}
\index{({{\sf LO}})}\index{Choice Principle!({{\sf LO}})}\item[$(\mathsf{LO}){:}$] for every set $x$ there exists a linear order $\ell$ such that $x=\dom[\ell^\pm]$.
\end{itemize}
\end{proposition}

\begin{proof} Given a set $x$, consider the set $\lambda$ of all linear orders $f$, which are finite subsets of $x\times x$. Let $\mathcal P_{<\w}(x)$ be the set of finite subsets of $x$. We recall that a set $s$ is called {\em finite} if there exists an injective function $g$ such that $\dom[g]=s$ and $\rng[g]\in\w$.

For any finite subset $a\subseteq x$ let $\lambda_a=\{f\in \lambda:a\subseteq\dom[f^\pm]\}$. Consider the set 
$$\varphi=\{l\subseteq \lambda:\exists a\in \mathcal P_{<\w}(x)\;(\lambda_a\subseteq l)\}.$$
It is easy to see that $\varphi$ is a filter with $\bigcup\varphi=\lambda$. By $(\mathsf{UF})$, the filter $\varphi$ is contained in some ultrafilter $u$ with $\bigcup u=\lambda$. It can be shown that $\ell=\bigcap_{l\in u}\bigcup l$ is a linear order with $\dom[\ell^\pm]=x$. 
\end{proof}

\begin{exercise} Fill all the details in the proof of Proposition~\ref{p:UF=>LO}.
\smallskip

\noindent{\em Hint:} First show that $x\times x\subseteq \ell^\pm\cup\Id$, then that $\ell$ is antisymmetric and finally that $\ell$ is transitive.
\end{exercise}

\begin{remark} The statement $(\mathsf{UF})$ is equivalent to many important statements in Mathematics. For example, it is equivalent to the compactness of the Tychonoff product of any family of compact Hausdorff spaces, see \cite[2.6.15]{JechAC}. For a long list of statements which are equivalent to $(\mathsf{UF})$, see \cite[Form 14]{HR}.
\end{remark}

\begin{exercise}[Kurepa] Prove that $(\mathsf{AC})\Leftrightarrow(\mathsf{AP}+\mathsf{LO})$.
\end{exercise}

\begin{proposition} The statement $(\mathsf{LO})$ implies the following statement
\begin{itemize} \setlength{\itemindent}{3pt}
\index{({{\sf AC$^{<\w}$}})}\index{Choice Principle!({{\sf AC$^{<\w}$}})}\item[$(\mathsf{AC}^{<\w}){:}$] For every set $A$ and every indexed family of finite nonempty sets $(X_\alpha)_{\alpha\in A}$, there exists a function $f:A\to\bigcup_{\alpha\in A}X_\alpha$ such that $f(\alpha)\in X_\alpha$ for every $\alpha\in A$.
\end{itemize}
\end{proposition}

\begin{proof} By the statement $\mathsf{(LO)}$, for the set $X=\bigcup_{\alpha\in A}X_\alpha$ there exists a linear order $\ell$ such that $\dom[\ell^\pm]=X$. Let $f:A\to X$ be the function assigning to every $\alpha\in A$ the unique $\ell$-minimal element of the finite set $X_\alpha$. It is clear that $f\in\prod_{n\in\w}X_\alpha$.
\end{proof}

\begin{exercise} Prove that for every finite linear order $\ell$, the finite set $\dom[\ell^\pm]$ contains a unique $\ell$-minimal element.
\smallskip

\noindent{\em Hint:} Apply the Principle of Mathematical Induction.
\end{exercise}

Observe that the statement $(\mathsf{TC})$ from Theorem~\ref{t:mainAC} implies the following weaker statement 
\vskip3pt
\index{({{\sf TC$_\w$}})}
\index{Choice Principle!({{\sf TC$_\w$}})}
\fbox{$\mbox{$(\mathsf{TC}_\w)${:}\;\; Every ordinary tree $t\subseteq\UU^{<\w}$ with $t\in\UU$ contains a maximal chain.}$}
\vskip5pt

The statement $(\mathsf{TC}_\w)$ is equivalent to the \index{Axiom of Dependent Choice ({{\sf DC}})}\index{({{\sf DC}})}{\em Axiom of Dependent Choice} $(\mathsf{DC})$, introduced by Paul Bernays in 1942 whose aim was to suggest an axiom which is weaker than $\mathsf{AC}$ and does not have strange consequences like the Banach-Tarski Paradox\footnote{{\bf Task:} Read about the Banach-Tarski Paradox in Wikipedia.} but still is sufficient for normal development of Mathematical Analysis.

\begin{proposition}\label{p:TC=>DC} The principle $(\mathsf{TC}_\w)$ is equivalent to the following statement
\begin{itemize}
\index{({{\sf DC}})}\index{Choice Principle!({{\sf DC}})}\index{Axiom of Dependent Choice ({{\sf DC}})}\item[$(\mathsf{DC}){:}$] For any relation $r\in\UU$ with $\dom[r]=\dom[r^\pm]$ there exists a function $f$ such that $\dom[f]=\w$ and $\langle f(n),f(n+1)\rangle\in r$ for all $n\in\w$.
\end{itemize}
\end{proposition}

\begin{proof} $(\mathsf{TC}_\w)\Ra(\mathsf{DC})$: Given any relation $r\in \UU$ with $\dom[r]=\dom[r^\pm]$, consider the set $t$ of all functions $f$ such that $\dom[f]\in\w$, $\rng[f]\subseteq\dom[r^\pm]$ and $\langle f(k),f(k+1)\rangle\in r$ for any $k\in\dom[f]$ with $k+1\in\dom[f]$. It is clear that $t$ is an ordinary tree with $t\subseteq\UU^{<\w}$. By $(\mathsf{TC}_\w)$, $t$ contains a maximal chain $c$. Then the union $f=\bigcup c$ is a required function with $\dom[f]=\w$ and $\langle f(n),f(n+1)\rangle\in r$ for all $n\in\w$. 
\smallskip

\noindent$(\mathsf{DC})\Ra(\mathsf{TC}_\w)$: Let $t\subseteq\UU^{<\w}$ be an ordinary tree, which is a set. Then the relation $$r=\big\{\langle\langle f,k,f(k)\rangle,\langle f,k+1,f(k+1)\rangle\rangle:(f\in t)\;\wedge\;(k\in \dom[f])\;\wedge\;(k+1\in\dom[f])\big\}$$
is a set, too. If $t$ has a maximal chain, then we are done. If $t$ has no maximal chains, then $t$ contains no $\mathbf S{\restriction}T$-maximal elements and hence $\dom[r]=\dom[r^\pm]$. By $(\mathsf{DC})$, there exists a function $g$ such that $\dom[g]=\w$ and $\langle g(n),g(n+1)\rangle\in r$ for all $n\in\w$. Then $g(0)=\langle f,k,f(k)\rangle$ for some $f\in t$ and $k\in\w$. We claim that $g(n)=\langle f,k+n,f(k+n)\rangle$ for all $n\in\w$. For $n=0$ this follows from the choice of $k$.
Assume that for some $n\in\w$ we proved that $g(n)= \langle f,k+n,f(k+n)\rangle$.
Since $\langle g(n),g(n+1)\rangle\in r$, the definition of the relation $r$ ensures that $g(n+1)=\langle f,k+n+1,f(k+n+1)\rangle$. By the Principle of Mathematical Induction, $g(n)=\langle f,k+n,f(k+n)\rangle$ for all $n\in\w$. Then $\dom[f]=\w$ and $\{f{\restriction}_n:n\in\w\}$ is a maximal chain in the tree $t$. 
\end{proof}

In its turns, the principle $(\mathsf{TC}_\w)$ implies the Axiom of Countable Choice $(\mathsf{AC_\w})$, introduced in the following proposition. 

\begin{proposition}\label{p:TCw=>ACw} The Principle $(\mathsf{TC}_\w)$  implies the following statement:
\begin{itemize}
\index{({{\sf AC$_\w$}})}\index{Choice Principle!({{\sf AC$_\w$}})}\item[$(\mathsf{AC}_\w){:}$] For any indexed sequence of nonempty sets $(X_n)_{n\in\w}$ there exists a function $f:\w\to\bigcup_{n\in\w}X_n$ such that $f(n)\in X_n$ for every $n\in\w$.
\end{itemize}
\end{proposition}

\begin{proof} Consider an ordinary tree $t$ consisting of functions $f$ such that $\dom[f]\in \w$ and $f(k)\in X_k$ for all $k\in\dom[f]$. By $(\mathsf{TC}_\w)$, the tree $t$ contains a maximal chain $c$. By the maximality of $c$, the union $f=\bigcup c$ is a function such that $\dom[f]=\w$ and $f{\restriction}_n\in t$ for all $n\in\w$. The definition of the tree $t$ ensures that  $f(k)\in X_k$ for all $k\in\w$.
\end{proof}

\begin{exercise} Prove that for any finite set $X$ there exists a choice function $f:X\setminus\{\emptyset\}\to\bigcup X$.
\vskip5pt

\noindent {\em Hint:} Use the Mathematical Induction on the cardinality of $X$.
\end{exercise}

\begin{exercise} Prove that $(\mathsf{AC}_\w)$ is equivalent to the existence of a choice function $c:X\setminus\{\emptyset\}\to \bigcup X$ for any countable set $X$.
\end{exercise}

In its turn, the Axiom of Countable Choice implies the following statement $(\mathsf{UT_\w})$ called the \index{Countable Union Theorem ($\mathsf{UT_\omega}$)}{\em Countable Union Theorem.}

\begin{proposition}\label{p:UT} The Axiom of Countable Choice $(\mathsf{AC}_\w)$ implies the following statement.
\begin{enumerate}
\index{({{\sf UT$_\w^{}$}})}\index{Choice Principle!({{\sf UT$_\w$}})}\item[$(\mathsf{UT}_\w){:}$] For any indexed family of countable sets $(X_n)_{n\in\w}$ the set $\bigcup_{n\in\w}X_n$ is countable.
\end{enumerate}
\end{proposition}

\begin{proof} By the Axiom of Union, the union $X=\bigcup_{n\in\w}X_n$ is a set. Consider the function $\nu:X\to \w$ assigning to every element $x\in X$ the smallest ordinal $\nu(x)\in\w$ such that $x\in X_{\nu(x)}$. 

For every $n\in\w$ consider the set $F_n$ of all injective functions $f$ such that $\dom[f]=X_n$ and $\rng[f]\in\w\cup\{\w\}$. Since the sets $X_n$ are countable, the sets $F_n$ are not empty. By the Axiom of Countable Choice, there exists a function $\varphi\in\prod_{n\in\w}F_n$. For every $n\in\w$ denote the function $\varphi(n)\in F_{n}$ by $\varphi_n$. Consider the function $$\mu:X\to\w,\quad\mu:x\mapsto\max\{\nu(x),\varphi_{\nu(x)}(x)\}=\mu(x)\cup\varphi_{\nu(x)}(x).$$

On the set $X$ consider the irreflexive set-like well-order 
\begin{multline*}
W=\big\{\langle x,y\rangle\in X\times X: \mu(x)<\mu(y)\;\vee\;\big(\mu(x)=\mu(y)\;\wedge\; \nu(x)<\nu(y)\big)\;\vee\\\big(\mu(x)=\mu(y)\;\wedge\;\nu(x)=\nu(y)\;\wedge\;\varphi_{\nu(x)}(x)<\varphi_{\nu(y)}(y)\big)\big\}.
\end{multline*}

By Theorem~\ref{t:wOrd}, the function $\rank_W:X\to\rank(W)\in\Ord$ is an order isomorphism. The definition of the well-order $W$ implies that each initial interval of $W$ is finite\footnote{{\bf Exercise:} Prove (by Mathematical Induction) that all initial intervals of the well-order $W$ are finite.} and hence $\rank(W)\le\w$. Then $\rank_W$ is an injective function with $\dom[\rank_W]=X$ and $\rng[\rank_W]=\rank(W)\in\w\cup\{\w\}$, witnessing that the set $X$ is countable.
\end{proof}

\begin{proposition} The Countable Union Theorem implies the following statement:
\begin{enumerate}
\index{({{\sf AC$^\w_\w$}})}\index{Choice Principle!({{\sf AC$^\w_\w$}})}\item[$(\mathsf{AC}^\w_\w)$] For any indexed sequence of countable sets $(X_n)_{n\in\w}$, there exists a function\\ $f:\w\to\bigcup_{n\in\w}X_n$ such that $f(n)\in X_n$ for every $n\in\w$.
\end{enumerate}
\end{proposition}

\begin{proof} Let $(X_n)_{n\in\w}$ be an indexed sequence of nonempty countable sets. By Countable Union Theorem, the set $X=\bigcup_{n\in\w}X_n$ is countable and hence $X=\dom[W^\pm]$ for some well-order $W$. Consider the function $f:\w\to X$ assigning to each $n\in\w$ the unique $W$-minimal element of the nonempty set $X_n$. 
\end{proof}  

Each of the statements $(\mathsf{AC}_\w^\w)$ or $(\mathsf{AC}^{<\w})$ imply the equivalent statements in the following theorem.

\begin{theorem}\label{t:Konig} The following statements are equivalent:
\begin{itemize}\setlength{\itemindent}{6pt}
\index{({{\sf AC$^{<\w}_\w$}})}\index{Choice Principle!({{\sf AC$^\{<\w}_\w$}})}\item[$(\mathsf{AC}_{\w}^{<\w}){:}$] For any indexed sequence of nonempty finite sets $(X_n)_{n\in\w}$, there exists a function $f:\w\to\bigcup_{n\in\w}X_n$ such that $f(n)\in X_n$ for every $n\in\w$.
\index{({{\sf UT$_\w^{<\w}$}})}\index{Choice Principle!({{\sf UT$_\w^{<\w}$}})}\item[$(\mathsf{UT}_{\w}^{<\w}){:}$] For any indexed sequence of finite sets $(X_n)_{n\in\w}$ the union $\bigcup_{n\in\w}X_n$ is countable.
\index{({{\sf TC$_\w^{<\w}$}})}\index{Choice Principle!({{\sf TC$_\w^{<\w}$}})}\item[$(\mathsf{TC}_{\w}^{<\w}){:}$] Any locally finite ordinary tree $T\subseteq \UU^{<\w}$ contains a maximal chain.
\index{({{\sf DC$_\w^{<\w}$}})}\index{Choice Principle!({{\sf DC$_\w^{<\w}$}})}\item[$(\mathsf{DC}^{<\w}){:}$] For any relation $r\in \UU$ such that for every $x\in \dom[r]$ the set $\{y:\langle x,y\rangle\in r\}$ is finite and non-empty, there exists a function $f$ such that $\dom[f]=\w$ and\\ $\forall n\in\w\;\langle f(n),f(n+1)\rangle\in r$. 
\end{itemize}
\end{theorem}

\begin{proof} $(\mathsf{AC}_{\w}^{<\w})\;\Ra\; (\mathsf{UT}_{\w}^{<\w})$: Given any sequence of finite sets $(X_n)_{n\in\w}$, for every $n\in\w$ consider the set $$F_n=\{f\in\Fun:f^{-1}\in\Fun\;\wedge\;\dom[f]=X_n\;\wedge\;\rng[f]\in\w\}.$$ Since the sets $X_n$ are finite, the sets $F_n$ are finite and nonempty (this can be proved by Mathematical Induction). By $(\mathsf{AC}_{\w}^{<\w})$, there exists a function $\varphi\in\prod_{n\in\w}F_n$. Repeating the argument of the proof of Proposition~\ref{p:UT}, we can prove that the union $\bigcup_{n\in\w}X_n$ is countable.
\smallskip

$(\mathsf{UT}_{\w}^{<\w})\;\Ra\; (\mathsf{TC}_{\w}^{<\w})$: Let $T\subseteq\UU^{<\w}$ be a locally finite ordinary tree. For every $n\in\w$ consider the class $T_n=\{t\in T:\dom[t]=n\}$. Using the Principle of Mathematical Induction and the local finiteness of the tree $T$, one can prove that for every $n\in\w$ the class $T_n$ is a finite set. By  $(\mathsf{UT}_{\w}^{<\w})$, the union $T=\bigcup_{n\in\w}T_n$ is a countable set. Consequently, the set $T$ admits a well-order $W$ such that $\dom[W^\pm]=T$.

Let $L$ be the class of chains in the tree $T$. For every chain $\ell\in L$, the union $\bigcup\ell$ is a function with $\dom[\bigcup\ell]\subseteq\w$. Let $\Succ_T(\bigcup\ell)=\{t\in T:\bigcup\ell\subset t\;\wedge\;\dom[t]=\dom[\bigcup\ell]+1\}$ be the (finite) set of immediate successors of the function $\bigcup\ell$ in the tree $T$.
Consider the function 
$F:\w\times\UU\to\UU$ assigning to every ordered pair $\langle n,\ell\rangle\in \w\times\UU$ the set
$$F(n,\ell)=\begin{cases}\min_W(\Succ_T(\bigcup\rng[\ell]))&\mbox{if $\rng[\ell]\in L$ and $\Succ_T(\bigcup\ell)\ne\emptyset$};\\
\emptyset&\mbox{otherwise}.
\end{cases}
$$By the Recursion Theorem~\ref{t:Recursion}, there exists a function $G\colon\w\to \UU$ such that $G(n)=F(n,G{\restriction}_n)$ for every $n\in\w$. It can be shown that $G[\w]$ is a maximal chain in the tree $T$.
\smallskip

The implication $(\mathsf{TC}_{\w}^{<\w})\;\Ra\;(\mathsf{AC}_{\w}^{<\w})$ can be proved by analogy with Proposition~\ref{p:TCw=>ACw}, and the equivalence $(\mathsf{TC}_{\w}^{<\w})\;\Leftrightarrow\;(\mathsf{DC}_{\w}^{<\w})$ can be proved by analogy with Proposition~\ref{p:TC=>DC}.
\end{proof}



\begin{theorem}[K\H onig, 1927]\index{K\H onig Theorem} The statement $(\mathsf{AC}_\w^{<\w})$ is equivalent to the statement
\begin{itemize}\setlength{\itemindent}{8pt}
\index{({{\sf TC$_\w^{<\w}$}})}\index{Choice Principle!({{\sf TC$_\w^{<\w}$}})}\item[$(\mathsf{TC}^{<\w}_\w){:}$] Every locally finite ordinary tree of countable height has a maximal chain.
\end{itemize}
\end{theorem}

Therefore, we have the following diagram of statements related to the Axiom of Choice.
$$
\xymatrix@C=21pt{
&&&\mathsf{MC}\ar@{=>}[r]&\mathsf{AP}\ar@{=>}[r]&\mathsf{WL}\ar@{=>}[r]&\mathsf{WP}\ar@{<=>}[r]&\mathsf{WV}\ar^{\UU=\VV}[d]\\
\mathsf{LO}\ar@{=>}[d]&\mathsf{UF}\ar@{=>}[l]&\mathsf{WO}\ar@{<=>}[r]\ar@{=>}[l]&\mathsf{AC}\ar@{<=>}[r]\ar@{=>}[u]&\mathsf{KZ}\ar@{<=>}[r]&\mathsf{MP}\ar@{<=>}[r]&\mathsf{TT}\ar@{<=>}[r]&\mathsf{TC}\ar@{=>}[r]&\mathsf{DC}\ar@{<=>}[d]\\
\mathsf{AC}^{<\w}\ar@{=>}[r]&\mathsf{AC}^{<\w}_\w\ar@{<=>}[r]&\mathsf{TC}^{<\w}_\w\ar@{<=>}[r]&\mathsf{DC}^{<\w}_\w&\mathsf{UT}^{<\w}_\w\ar@{<=>}[l]&\AC^\w_\w\ar@{=>}[l]&\mathsf{UT}_\w\ar@{=>}[l]&\mathsf{AC}_\w\ar@{=>}[l]&\mathsf{TC}_\w\ar@{=>}[l]\\
}
$$

\begin{remark} All implications in this diagram are strict (i.e., cannot be reversed). The proof of this fact requires more advanced technique, see \cite{JechAC}. For applications of weaker forms of choice in topology, see \cite{Wajch}. 
\end{remark}  

\newpage

\section{Global Choice}\label{s:GChoice}

\rightline{\em The strongest principle of growth lies in human choice}

\rightline{George Eliot}
\vskip25pt

In this section we study the interplay between global versions of Choice Principles that were analyzed in the preceding section. First recall some notions related to well-orderability of classes.

A class $X$ 
\begin{itemize}
\item can be \index{class!well-ordered} {\em well-ordered} if there exists a well-order $W$ such that $X=\dom[W^\pm]$;
\item is \index{class!well-orderable}{\em well-orderable} if there exists a set-like well-order $W$ such that $X=\dom[W^\pm]$;
\item {\em admits a choice function} if there exists a function $f:X\to\UU$ assigning to each non-empty set $x\in X$ some element $f(x)\in x$;
\item is \index{class!cumulative}{\em cumulative} if $X$ admits a function $f:X\to\Ord$ such that for every ordinal $\alpha$ the preimage  $f^{-1}[\alpha]=\{x\in X:f(x)\in \alpha\}$ is a set.
\end{itemize} 
A class $X$ is cumulative if and only if its can be written as the union $X=\bigcup_{\alpha\in \Ord}X_\alpha$ of a transfinite sequence of sets $(X_\alpha)_{\alpha\in\Ord}$. The von Neumann  cumulative hierarchy $(V_\alpha)_{\alpha\in\Ord}$ witnesses that the class $\VV$ is cumulative.

The following theorem is the class generalization of the Zermelo Theorem~\ref{t:WO}.

\begin{theorem}\label{t:CZ} For a class $X$ consider the statements:
\begin{enumerate}\setlength{\itemindent}{15pt}
\item[$\mathsf{WO}(X){:}$] The class $X$ is well-orderable.
\item[$\mathsf{Wo}(X){:}$] The class $X$ can be well-ordered.
\item[$\mathsf{wo}(X){:}$] The class $X$ admits an order $W$ such that $X=\dom[W^\pm]$ and each non-empty subset $y\subseteq X$ has the $W$-smallest element.
\item[$\mathsf{C}(\mathcal PX){:}$] The class $\mathcal P(X)$ admits a choice function $f:\mathcal P(X)\setminus\{\emptyset\}\to X$.
\end{enumerate}
Then $\mathsf{WO}(X)\Ra \mathsf{Wo}(X) \Ra \mathsf{wo}(X)\Ra \mathsf{C}(\mathcal PX)$. 

\noindent If the class $X$ is cumulative, then $\mathsf{C}(\mathcal P X)\Ra \mathsf{WO}(X)$ and hence $$\mathsf{WO}(X)\Leftrightarrow \mathsf{Wo}(X) \Leftrightarrow \mathsf{wo}(X)\Leftrightarrow \mathsf{C}(\mathcal P X).$$
\end{theorem}

\begin{proof} The implications $\mathsf{WO}(X)\Ra \mathsf{Wo}(X) \Ra \mathsf{wo}(X)$ are trivial. To prove that $\mathsf{wo}(X)\Ra \mathsf{C}(\mathcal PX)$, assume that the class admits an order $W$ such that $X=\dom[W^\pm]$ and each non-empty subset $y\subseteq X$ has the $W$-smallest element $f(y)\in y$. Then $f:\mathcal P(X)\setminus\{\emptyset\}\to X$ is a well-defined choice function for the class $\mathcal P(X)$ of all subsets of $X$.

Now assuming that the class $X$ is cumulative, we prove that $\mathsf{C}(\mathcal PX)\Ra\mathsf{WO}(X)$. By $\mathsf{C}(\mathcal PX)$, there exists a function $C:\mathcal P(X)\setminus\{\emptyset\}\to \UU$ such that  $C(y)\in y$ for every nonempty set $y\subseteq X$. If $X$ is a set, then the well-orderability of the set $X$ follows from Zermelo's Theorem~\ref{t:WO}. So, we assume that $X$ is a proper class. By the cumulativity of the class $X$, there exists a function $\lambda:X\to\Ord$ such that for every ordinal $\alpha$ the preimage $X_\alpha=\lambda^{-1}(\alpha)=\{x\in X:\lambda(x)=\alpha\}$ is a set. 

Observe that for every set $y$ the class $X\setminus y$ is proper and hence non-empty and then $\lambda[X\setminus y]$ is a nonempty subclass of $\Ord$, so we can consider the smallest ordinal $\min \lambda[X\setminus y]$ in $\lambda[X\setminus y]$. 
Then the function 
$$F:\Ord\times \UU\to\UU,\quad F:\langle\alpha,y\rangle\mapsto C(X_{\min\lambda[X\setminus \rng[y]]}\setminus \rng[y])$$is well-defined. By the Recursion Theorem~\ref{t:Recursion}, there exists a function $G:\Ord\to\UU$ such that 
$$G(\alpha)=F(\alpha,G{\restriction}_\alpha)=C(X_{\min\lambda(X\setminus G[\alpha])}\setminus G[\alpha])\in X\setminus G[\alpha]$$ for every $\alpha\in \Ord$.
This property of $G$ implies that $G$ is injective. Next, we show that $G[\Ord]=X$. Assuming that $G[\Ord]\ne X$, we can find a set $z\notin G[\Ord]$ and an ordinal $\alpha$ such that $z\in X_\alpha$. It follows that for every ordinal $\beta$, the set $X_\alpha\setminus G[\beta]\ni z$ is not empty. The definition of the function $F$ guarantees that $G(\beta)=F(\beta,G{\restriction}_\beta)\in \bigcup_{\gamma\le\alpha}X_\gamma$. Now we see that $G$ is an injective function from $\Ord$ to the set $\bigcup_{\gamma\le\alpha}X_\gamma$ which contradicts the Axiom of Replacement. This contradiction shows that $G[\Ord]=X$. Then we can define a set-like well-order $W$ on $X$ by the formula
$$W=\{\langle G(\alpha),G(\beta)\rangle:\alpha\in\beta\in\Ord\}.$$
\end{proof}

\begin{corollary}\label{c:CZ} A class $X$ is well-orderable if and only if it is cumulative and its power-class $\mathcal P(X)$ has a choice function $C:\mathcal P(X)\setminus\{\emptyset\}\to X$.
\end{corollary}

\begin{proof} The ``if'' part follows from Theorem~\ref{t:CZ}. To prove the ``only if'' part, assume that the class $X$ is well-orderable and hence admits a set-like well-order $W$ such that $\dom[W^\pm]=X$. By Theorem~\ref{t:CZ}, the power-set $\mathcal P(X)$ has a choice function. It remains to prove that the class $X$ is cumulative. By Theorem~\ref{t:wOrd}, the rank function $\rank_W:X\to \rank(W)\subseteq \Ord$ is an order isomorphism. Since the order $W$ is set-like, for every ordinal $\alpha$ the class $X_\alpha=\{x\in X:\rank_W(x)\in\alpha\}$ is a set, witnessing that the class $X$ is cumulative.
\end{proof}

In the following theorem we prove that under the assumption of cumulativity of the universe (denoted by \index{$\mathsf{C}{\mathbf U}$}$(\mathsf{C}\UU)$), the Axiom of Global Choice $(\mathsf{AGC})$ is equivalent to many global versions of the statements, equivalent to the Axiom of Choice. 

We recall that an ordinary tree $T\in\UU^{<\Ord}$ is called \index{tree!locally set}{\em locally set} if for each $t\in T$ the class $\Succ_T(t)$ of its immediate successors in $T$ is a set.

\begin{theorem}\label{t:mainAGC} If the universe $\UU$ is cumulative \textup{(}which follows from $\UU=\VV$\textup{)}, then the following statements are equivalent:
\begin{enumerate}\setlength{\itemindent}{6pt}
\index{({{\sf GWO}})}\index{Choice Principle!({{\sf GWO}})}\item[$(\mathsf{GWO}){:}$] There exists a set-like well-order $W$ such that $\dom[W^\pm]=\UU$.
\index{({{\sf GwO}})}\index{Choice Principle!({{\sf GwO}})}\item[$(\mathsf{GwO}){:}$] There exists a well-order $W$ such that $\dom[W^\pm]=\UU$.
\index{({{\sf Gwo}})}\index{Choice Principle!({{\sf Gwo}})}\item[$(\mathsf{Gwo}){:}$] There exists a linear order $W$ such that $\dom[W^\pm]=\UU$ and each nonempty set contains a $W$-minimal element.
\index{({{\sf GMP}})}\index{Choice Principle!({{\sf GMP}})}\item[$(\mathsf{GMP}){:}$] For every order $R$ there exists a maximal $R$-chain $C\subseteq\dom[R^\pm]$.
\item[$(\mathsf{GKZ}){:}$]\index{({{\sf GKZ}})}\index{Choice Principle!({{\sf GKZ}})} For every chain-bounded order $R$ there exists an $R$-maximal element $x\in\dom[R^\pm]$.
\item[$(\mathsf{GTC}){:}$]\index{({{\sf GTC}})}\index{Choice Principle!({{\sf GTC}})} Every ordinary tree has a maximal chain.
\item[$(\mathsf{TC^s}){:}$]\index{({{\sf TC$^{\mathsf s}$}})}\index{Choice Principle!({{\sf TC$^{\mathsf s}$}})} Every locally set ordinary tree has a maximal chain.
\item[$(\mathsf{EC}){:}$]\index{({{\sf EC}})}\index{Choice Principle!({{\sf EC}})} For every equivalence relation $R$ there exists a class $C$ such that for every $x\in\dom[R]$ the intersection $R[\{x\}]\cap C$ is a singleton.
\item[$(\mathsf{AC^c_c}){:}$]\index{({{\sf AC$_{\mathsf c}^{\mathsf c}$}})}\index{Choice Principle!({{\sf AC$_{\mathsf c}^{\mathsf c}$}})} For every indexed family of non-empty classes $(X_\alpha)_{\alpha\in A}$ there exists a function\\ $F\colon A\to\bigcup_{\alpha\in A}X_\alpha$ such that $F(\alpha)\in X_\alpha$ for all $\alpha\in A$.
\item[$(\mathsf{AC^s_c}){:}$]\index{({{\sf AC$_{\mathsf c}^{\mathsf s}$}})}\index{Choice Principle!({{\sf AC$_{\mathsf c}^{\mathsf s}$}})} For every indexed family of non-empty sets $(X_\alpha)_{\alpha\in A}$ there exists a function\\ $F\colon A\to\bigcup_{\alpha\in A}X_\alpha$ such that $F(\alpha)\in X_\alpha$ for all $\alpha\in A$.
\item[$(\mathsf{AGC}){:}$]\index{({{\sf AGC}})}\index{Choice Principle!({{\sf AGC}})} There exists a function $F{:}\UU{\setminus}\{\emptyset\}\to\UU$ such that $F(x){\in}x$ for every nonempty set $x$.
\end{enumerate}
Under the Axiom of Foundation $(\UU=\VV)$ these statements are equivalent to any of the following statements:
\begin{itemize}\setlength{\itemindent}{6pt}
\item[$(\mathsf{GMC}){:}$]\index{({{\sf GMC}})}\index{Choice Principle!({{\sf GMC}})} There exists a function $F\colon\UU\setminus\{\emptyset\}\to\UU$ assigning to each nonempty set $x$ a nonempty finite subset $F(x)\subseteq x$.
\item[$(\mathsf{GAP}){:}$]\index{({{\sf GAP}})}\index{Choice Principle!({{\sf GMC}})} For every order $R$ there exists a maximal $R$-antichain $C\subseteq\dom[R^\pm]$.
\item[$(\mathsf{GWL}){:}$]\index{({{\sf GWL}})}\index{Choice Principle!({{\sf GWL}})} Every linearly ordered class can be well-ordered.
\item[$(\mathsf{WPO}){:}$]\index{({{\sf WPO}})}\index{Choice Principle!({{\sf WPO}})} The class $\mathcal P(\Ord)$ can be well-ordered.
\item[$(\mathsf{GWV}){:}$]\index{({{\sf GWV}})}\index{Choice Principle!({{\sf GWV}})} The class $\VV$ is well-orderable.
\end{itemize}
\end{theorem} 

To prove Theorem~\ref{t:mainAGC}, in Lemmas~\ref{l:AV+TCs=>GWO}--\ref{l:AC=>TC} we shall prove the following implications.

$$
\xymatrix{
&(\mathsf{GWV})\ar@{<=>}^{\ref{l:WPO<=>GWV}}[r]\ar@{<->}^{\UU=\VV}[d]&(\mathsf{WPO})&(\mathsf{GWL})\ar@{=>}^{\ref{l:GWL=>WPO}}[l]&(\mathsf{GAP})\ar_{\mathsf{C}\UU}^{\ref{l:GAP=>GWL}}[l]&(\mathsf{GMC})\ar_{\mathsf{C}\UU}^{\ref{l:GMC=>GAP}}[l]\\
(\mathsf{TC^s})\ar^{\mathsf{C}\UU}_{\ref{l:AV+TCs=>GWO}}[r]&(\mathsf{GWO})\ar@{=>}[r]\ar@{=>}_{\ref{l:GWO=>GMP}}[d]&(\mathsf{GwO})\ar@{=>}[r]\ar@{=>}_{\ref{l:GMP=>EC}}[d]&(\mathsf{Gwo})\ar@{=>}_{\ref{t:CZ}}[r]&\mathsf{AGC}\ar@{<=>}[d]\ar@{=>}[ur]\\
&(\mathsf{GMP})\ar@{=>}_{\ref{l:GMP=>GKZ}}[d]&(\mathsf{EC})\ar@{=>}^{\ref{l:EC=>ACcc}}[r]&(\mathsf{AC^c_c})\ar@{=>}_{\ref{l:AC=>TC}}[d]\ar@{=>}[r]&(\mathsf{AC^s_c})\ar@{=>}^{\ref{l:AC=>TC}}[d]\\
&(\mathsf{GKZ})\ar@{=>}_{\ref{l:GKZ=>GTC}}[rr]&&(\mathsf{GTC})\ar@{=>}[r]&(\mathsf{TC^s})\\
}
$$

\begin{lemma}\label{l:AV+TCs=>GWO} $(\mathsf{C}\UU+\mathsf{TC^s})\;\Ra\;(\mathsf{GWO})$.
\end{lemma}

\begin{proof} By $(\mathsf{C}\UU)$, there exists an indexed family of sets $(U_\alpha)_{\alpha\in\Ord}$ such that $\UU=\bigcup_{\alpha\in\Ord}U_\alpha$. This family induces the function $\mu:\UU\to\Ord$ assigning to every set $y$ the smallest ordinal $\alpha$ such that $U_\alpha\setminus y\ne\emptyset$. Let $T$ be the class of functions $f$ such that $\dom[f]\in\Ord$ and for every $\alpha\in\dom[f]$ with $\alpha+1\in\dom[f]$ we have $f(\alpha)\in U_{\mu(f[\alpha])}\setminus f[\alpha]$. It is easy to see that $T$ is a locally set ordinary tree. By $(\mathsf{TC^s})$, this tree contains a maximal chain $C$. Its union $F=\bigcup C$ is a function such that $\dom[f]\subseteq\Ord$ and $f(\alpha)\in U_{\mu(f[\alpha])}\setminus f[\alpha]$ for every $\alpha\in\dom[f]$. The latter condition ensures that $f$ is injective. 

We claim that $\rng[f]=\UU$. Assuming that $\rng[f]\ne\UU$, we can find the smallest ordinal $\alpha$ such that $U_\alpha\cap (\UU\setminus\rng[f])\ne\emptyset$. Then for every $\beta\in\Ord$ the set $U_\alpha\setminus f[\beta]\supseteq U_\alpha\setminus\rng[f]$ is not empty, which implies that $\mu(f[\beta])\le\alpha$ and $f(\beta)\in \bigcup_{\gamma\le\alpha}U_\gamma$. Consequently, $\rng[f]\subseteq\bigcup_{\gamma\le\alpha}U_\gamma$. By the injectivity of $f$ and the Axiom of Replacement, $\dom[f]\subseteq\Ord$ is a set and hence $\dom[f]\in\Ord$. Now take any element $z\in U_\alpha\setminus\rng[f]$ and consider the chain $\bar C\cup\{\langle\dom[f],z\rangle\}\subseteq T$, witnessing that the chain $C$ is not maximal. But this contradicts the choice of $C$. This contradiction shows that $\rng[f]=\UU$. Now we see that $$W=\{\langle f(\beta),f(\gamma)\rangle:\beta\in\gamma\in\dom[f]\subseteq\Ord\}$$ is a set-like well-order with $\dom[W^\pm]=\UU$.
\end{proof}

\begin{lemma}\label{l:GWO=>GKZ} $(\mathsf{GWO})\;\Ra\;(\mathsf{GKZ})$.
\end{lemma}

\begin{proof} The proof is a suitable modification of the Kuratowski-Zorn Lemma~\ref{l:KZL}. By $(\mathsf{GWO})$, there exists a set-like well-order $W$ with $\dom[W]=\UU$. Since $W$ is set-like, for every $x\in\UU$ the initial interval $\cev W(x)$ is a set.

To prove $(\mathsf{GKZ})$, fix any chain-bounded order $R$. If $R$ is a set, then we can apply the Kuratowski--Zorn Lemma~\ref{l:KZL} and conclude that the set $\dom[R^\pm]$ contains an $R$-maximal element. So, we assume that $R$ is a proper class. To derive a contradiction, assume that no element $x\in\dom[R^\pm]$ is $R$-maximal. 

Let $C$ be the class of all $R$-chains which are subsets of the class $X=\dom[R^\pm]$. Repeating the argument of the proof of Lemma~\ref{l:KZL}, we can show that for every $R$-chain $\ell\in C$, the class $V(\ell)=\{b\in\dom[R^\pm]:\ell\times\{b\}\subseteq R\setminus\Id\}$ is not empty. The well-foundedness of $W$ guarantees that the class $V(\ell)$ contains a unique $W$-minimal element $\min_W(V(\ell))$.

Consider the function $F:\Ord\times\UU\to\UU$ defined by the formula
$$F(\alpha,y)=\begin{cases} \min_W(V(\rng[y]))&\mbox{if $y\in C$};\\
\emptyset&\mbox{otherwise}.
\end{cases}
$$By the Recursion Theorem~\ref{t:Recursion}, there exists a (unique) function $G:\Ord\to \UU$ such that $G(\alpha)=F(\alpha,G{\restriction}_\alpha)$ for every ordinal $\alpha$.

Repeating the argument of the proof of Lemma~\ref{l:KZL}, we can show that for every ordinal $\alpha$ the image $G[\alpha]$ is an $R$-chain and hence $$G(\alpha)=F(\alpha,G{\restriction}_\alpha)={\min}_W(V(G[\alpha]))\in V(G[\alpha])\subseteq \UU\setminus G[\alpha],$$ which implies that the function $G:\Ord\to\dom[R^\pm]$ is injective. Since for every $\alpha\in\Ord$ the element $G(\alpha)=V(G[\alpha])$ is an upper bound of the $R$-chain $G[\alpha]$, the image $G[\Ord]$ is an $R$-chain in $\dom[R^\pm]$. Since $R$ is chain-bounded, the $R$-chain $G[\Ord]$ has an upper bound $b$. It follows that for every ordinal $\alpha$ the element $b$ belongs to the set $V(G[\alpha])$ and hence $G(\alpha)=\min_W(V(G[\alpha]))\in\cev W(b)$. Therefore $G[\Ord]\subseteq\cev W(b)$ and hence $\Ord=G^{-1}[\cev W(b)]$ is a set by the Axiom of Replacement. But this contradicts Theorem~\ref{t:Ord}(6).
\end{proof}

\begin{lemma}\label{l:GWO=>GMP} $(\mathsf{GWO})\;\Ra\;(\mathsf{GMP})$.
\end{lemma}

\begin{proof} Assume that $(\mathsf{GWO})$ holds and fix a set-like well-order $W$ such that $\dom[W^\pm]=\UU$. By Corollary~\ref{c:CZ}, the universe is cumulative, so we can find an indexed family of sets $(U_\alpha)_{\alpha\in\Ord}$ such that $\UU=\bigcup_{\alpha\in\Ord}U_\alpha$. Replacing each set $U_\alpha$ by the union $\bigcup_{\beta\le\alpha}U_\beta$, we can assume that $U_\beta\subseteq U_\alpha$ for any ordinals $\beta\in\alpha$.

To prove the $(\mathsf{GMP})$, take any order $R$. For every ordinal $\alpha$ let $\Lambda_\alpha$ be the set of all $R$-chains $\ell\subseteq U_\alpha\cap\dom[R^\pm]$. The set $\Lambda_\alpha$ is endowed with the partial order $\mathbf S{\restriction}\Lambda_\alpha$. Let $M_\alpha$ be the subset of $\Lambda_\alpha$ consisting of $\mathbf S{\restriction}\Lambda_\alpha$-maximal chains. By the $(\mathsf{GKZ})$ (which follows from $(\mathsf{GWO})$ by Lemma~\ref{l:GWO=>GKZ}), for every chain $\ell\in\Lambda_\alpha$ the set $M_\alpha(\ell)=\{\lambda\in M_\alpha:\ell\subseteq \lambda\}$ is not empty and hence contains a unique $\min_W$-minimal element $\min_W(M_\alpha(\ell))$.

So, we can define the function $F:\Ord\times\UU\to\UU$ by the formula
$$F\colon\langle \alpha,y\rangle\mapsto\begin{cases}\min_W(M_\alpha(\bigcup \rng[y]))&\mbox{if $\bigcup \rng[y]\in\Lambda_\alpha$};\\
\emptyset&\mbox{otherwise}.
\end{cases}
$$
By the Recursion Theorem~\ref{t:Recursion}, there exists a function $G:\Ord\to\UU$ such that $G(\alpha)=F(\alpha,G{\restriction}_\alpha)$ for every $\alpha\in\Ord$.

We claim that for every ordinal $\alpha$ the set $G(\alpha)$ is an element of the set $M_\alpha$ and $G(\beta)\subseteq G(\alpha)$ for all $\beta\in\alpha$. For $\alpha=0$, $G(0)=F(0,\emptyset)=\min_W(M_0(\emptyset))\in M_0$. Assume that for some ordinal $\alpha$ we have proved that for every ordinal $\beta\in\alpha$ the set $G(\beta)$ is an element of $M_\beta$ and for every ordinal $\gamma\in\beta$  we have $G(\gamma)\subseteq G(\beta)$. Then the union $\ell=\bigcup_{\beta\in\alpha}G(\beta)$ is a chain in the set $\dom[R^\pm]\cap\bigcup_{\beta\in\alpha}U_\beta\subseteq\dom[R^\pm]\cap U_\alpha$ and hence $\ell\in\Lambda_\alpha$. Now the definition of the function $F$ guarantees that
$$G(\alpha)=F(\alpha,G{\restriction}_\alpha)={\min}_W(M_\alpha(\ell))\in M_\alpha$$is a maximal $\mathbf S{\restriction}\Lambda_\alpha$-chain containing the chain $\ell$ as a subset. Then for every $\beta\in\alpha$ we have $G_\beta\subseteq\ell\subseteq G_\alpha$. 
\smallskip

Since the transfinite sequence of $R$-chains $(G(\alpha))_{\alpha\in\Ord}$ is increasing, its union $$L=\bigcup_{\alpha\in\Ord}G(\alpha)$$ is an $R$-chain in $\dom[R^\pm]$. We claim that $L$ is a maximal $R$-chain. In the opposite case, we could find an element $b\in \dom[R^\pm]\setminus L$ such that $L\cup\{b\}$ is an $R$-chain. Find an ordinal $\alpha$ such that $b\in U_\alpha$. Consider the $R$-chain $G(\alpha)\subseteq L$ and observe that $G(\alpha)\cup\{b\}$ is a chain in $\dom[R^\pm]\cap U_\alpha$, which implies that $G(\alpha)\notin M_\alpha$. But this contradicts the choice of $G(\alpha)$. This contradiction completes the proof of the maximality of the $R$-chain $L$.
\end{proof}

\begin{lemma}\label{l:GMP=>GKZ} $(\mathsf{GMP})\;\Ra\;(\mathsf{GKZ})$.
\end{lemma}

\begin{proof} Assume that $(\mathsf{GMP})$ holds. To prove $(\mathsf{GKZ})$, take any chain-bounded order $R$. We need to find an $R$-maximal element $b\in\dom[R^\pm]$. By  $(\mathsf{GMP})$, the class $\dom[R^\pm]$ contains a maximal $R$-chain $L$. By the chain-boundedness of the order $R$, the chain $L$ has an upper bound $b\in\dom[R^\pm]$. We claim that the element $b$ is $R$-maximal. In the opposite case we can find an element $b'\in\dom[R^\pm]$ such that $\langle b,b'\rangle\in R\setminus\Id$. Then the $R$-chain is contained in the strictly larger $R$-chain $C\cup\{b'\}$, which contradicts the maximality of $L$.
\end{proof} 

 \begin{lemma}\label{l:GMP=>EC} $(\mathsf{GwO})\;\Ra\;(\mathsf{EC}).$
\end{lemma}

\begin{proof} Assume that $(\mathsf{GwO})$ holds, which means that there exists a well-order $W$ such that $\dom[W^\pm]=\UU$.  To prove $(\mathsf{EC})$, take any equivalence relation $R$. Consider the function $F:\dom[R]\to\dom[R]$ assigning to every $x\in\dom[R]$ the unique $W$-minimal element of the class $R[\{x\}]$. It is easy to see that the class $C=\rng[F]$ has the required property: for every $x\in\dom[R]$ the intersection $R[\{x\}]\cap C$ is a singleton.
\end{proof}
 
 \begin{lemma}\label{l:EC=>ACcc} $(\mathsf{EC})\;\Ra\;(\mathsf{AC^c_c}).$
\end{lemma}

\begin{proof} Assume that $(\mathsf{EC})$ holds. To prove $(\mathsf{AC^c_c})$, take any indexed family of non-empty classes $X=(X_\alpha)_{\alpha\in A}$. Consider the equivalence relation $$R=\bigcup_{\alpha\in A}((\{\alpha\}\times X_\alpha)\times(\{\alpha\}\times X_\alpha)).$$ By $(\mathsf{EC})$, there exists a class $C$ such that for every $\alpha\in A$ and $x\in X_\alpha$ the intersection $C\cap(\{\alpha\}\times X_\alpha)$ is a singleton. Replacing $C$ by $C\cap X$, we can assume that $C\subseteq X=\bigcup_{\alpha\in A}(\{\alpha\}\times X_\alpha)$. In this case, $C$ is a function such that $\dom[C]=A$ and $\forall \alpha\in A\;\;C(\alpha)\in X_\alpha$.
\end{proof}

\begin{lemma}\label{l:GKZ=>GTC} $(\mathsf{GKZ})\;\Ra\;(\mathsf{GTC})$.
\end{lemma}

\begin{proof} Assume that $(\mathsf{GKZ})$ hold. To check $(\mathsf{GTC})$, we should prove that any ordinary tree $T$ has a maximal $\mathbf S{\restriction}T$-chain. To derive a contradiction, assume that $T$ does not contain maximal $\mathbf S{\restriction}T$-chains. In this case we shall show that the partial order $ \mathbf{S}{\restriction}T$ is chain-bounded. Take any chain $L\subseteq T$ and consider its union $f=\bigcup L$, which is function with $\dom[f]\subseteq\Ord$. Then $C=\{f{\restriction}_\alpha:\alpha\in\dom[f]\}$ is an $\mathbf S{\restriction}T$-chain in $T$. By our assumption this chain is not maximal and hence there exists an $\mathbf S{\restriction}T$-chain $C'\subseteq T$ such that $C\subsetneq C'$. Take any element $c'\in C'\setminus C$ and observe that $f\subseteq c'$ for any function $f\in L$. This means that $c'$ is an upper bound of the chain $L$, and hence the partial order $\mathbf S{\restriction}T$ is chain-bounded. By $(\mathsf{GKZ})$, this order has a maximal element $t\in T$. This maximal element $t$ generates the maximal chain $M=\{t\}\cup\{t{\restriction}_\alpha:\alpha\in\dom[t]\}$ in $T$. But this contradicts our assumption.  
\end{proof}

\begin{lemma}\label{l:AC=>TC} $(\mathsf{AC^c_c})\;\Ra\;(\mathsf{GTC})$ and $(\mathsf{AC^s_c})\;\Ra\;(\mathsf{TC^s})$.
\end{lemma}

\begin{proof} Given a (locally set) ordinary tree $T\subseteq\UU^{<\Ord}$, we should find a maximal chain in $T$. We endow $T$ with the partial order $\mathbf S{\restriction}T$. To derive a contradiction, assume that $T$ contains no maximal chains. Observe that for every chain $C\subseteq T$ its union $f=\bigcup C\subseteq T$ is a function with $\dom[f]\subseteq \Ord$. If $\dom[f]=\Ord$, then $\{f{\restriction}_\alpha:\alpha\in\Ord\}$ is a maximal chain in $T$, which contradicts our assumption. Then $\dom[f]$ is some ordinal, which implies that the chain $C$ is a set. Therefore, the class $L$ of all chains in $T$ is well-defined. We say that a chain $\ell\in L$ is {\em limit} if $\dom[\bigcup\ell]$ is a limit ordinal. Let $L'$ be the subclass of $L$ consisting of limit chains. By our assumption, every chain in $T$ is not maximal. Consequently, for any chain $\ell\in L\setminus L'$, the class $T_\ell=\{t\in T:\bigcup\ell\subset t,\;\dom[t]=\dom[\bigcup\ell]+1\}$ is not empty. If the tree $T$ is locally set, then $T_\ell$ is a set. 

Using $(\mathsf{AC^c_c})$ (or $(\mathsf{AC}^s_c)$), we can construct a function $\Psi:\UU\to\UU$ satisfying the following conditions:
\begin{enumerate}
\item $\Psi(\ell)=\emptyset$ if $\ell\in\UU\setminus L$;
\item $\Psi(\ell)=\bigcup\ell$ if $\ell\in L'$;
\item $\Psi(\ell)\in T_\ell$ if $\ell\in L\setminus L'$.
\end{enumerate}
In fact,  $(\mathsf{AC^c_c})$ (or  $(\mathsf{AC}^s_c)$) were used only for satisfying condition (3).

Consider the function $$F:\Ord\times\UU\to\UU,\quad F:\langle \alpha,\ell\rangle\mapsto \Psi(\textstyle\bigcup \rng[\ell]).$$
Applying the Recursion Theorem~\ref{t:Recursion} to the function $F$ and the well-order $\E{\restriction}\Ord$, we can find a unique function $G:\Ord\to\UU$ such that $G(\alpha)=F(\alpha,G{\restriction}_\alpha)$ for every $\alpha\in\Ord$. By transfinite induction it can be shown that for every $\alpha$ the set $G[\alpha]$ is a chain in $T$ and $\dom[G(\alpha)]=\alpha$. Then $G[\Ord]$ is a maximal chain in $T$, which contradicts our assumption. 
\end{proof} 

\begin{lemma}\label{l:GMC=>GAP} $(\mathsf{GMC}+\mathsf{C}\UU)\Ra(\mathsf{GAP})$.
\end{lemma}

\begin{proof} Let $R$ be any order and $X=\dom[R^\pm]$. Assuming that the universe is cumulative, write the universe as the union $\UU=\bigcup_{\alpha\in \Ord}U_\alpha$ of a transfinite sequence of sets $(U_\alpha)_{\alpha\in\Ord}$ such that $U_0=\emptyset$ and $U_\alpha\subseteq U_\beta$ for any ordinals $\alpha\le\beta$. Let $\lambda:\UU\to\Ord$ be the function assigning to every set $x$ the smallest ordinal $\alpha$ such that $x\in U_\alpha$. 

By $(\mathsf{GMC})$, there exists a function $f:\UU\to\UU$ assigning to each (nonempty) set $x$ a (nonempty) finite set $f(x)\subseteq x$. For every set $y$ let $$C(y)=\{z\in X:(\{z\}\times y)\cap (R^\pm\cup \Id)=\emptyset\}$$be the class of elements of $X$ that are $R$-incomparable with elements of the set $y$. Let $c(y)=C(y)\cap U_{\alpha(y)}$ where $\alpha(y)=\min(\{0\}\cup \lambda[C(y)])$. If the class $C(y)$ is not empty, then $c(y)$ is a nonempty subset of $C(y)$ and $f(c(y))$ is a nonempty finite subset of $c(y)\subseteq C(y)$.

For every set $y$ let 
$$\mu(y)=\{z\in y:(y\times\{z\})\cap R\subseteq\Id\}$$ be the set of $R$-minimal elements of the set $y$. Consider the function
$$F:\On\times\UU\to\UU,\quad F(\alpha,y)=\textstyle(\bigcup \rng[y])\cup \mu(f(c(\bigcup \rng[y]))).$$ By the Recursion Theorem~\ref{t:Recursion}, there exists a function $G:\On\to\UU$ such that $G(\alpha)=F(\alpha,G{\restriction}_\alpha)$ for every ordinal $\alpha$. By transfinite induction it can be shown that for every ordinal $\alpha$ the set $G(\alpha)$ is an $R$-antichain in $X$ and so is the class $G[\Ord]$. We claim that $G[\Ord]$ is a maximal $R$-antichain in $X$. In the opposite case we can find an element $z\in C(X\setminus G[\Ord])$ and conclude that for every ordinal $\alpha$ the set $U_{\lambda(z)}\cap C(\bigcup G[\alpha])$ is not empty, which implies that $(G(\alpha))_{\alpha\in\Ord}$ is a strictly increasing sequence of subsets of the set $U_{\lambda(z)}$. But the existence of such sequence contradicts the Axiom of Replacement. Therefore, $G[\On]$ is a maximal $R$-antichain for the order $R$.
\end{proof}

\begin{lemma}\label{l:GAP=>GWL} $(\mathsf{GAP}+\mathsf{C}\UU)\Ra(\mathsf{GWL})$.
\end{lemma}

\begin{proof} Assume that the universe is cumulative and each order has a maximal antichain. Given any linear order $L$ we should prove that the class $X=\dom[L^\pm]$ can be well-ordered. By Theorem~\ref{t:CZ}, the well-orderability of $X$ will follow as soon as we show that the power-class $\mathcal P(X)$ has choice function. Consider the order 
$$R=\{\langle \langle a,y\rangle,\langle b,z\rangle\rangle\in(\mathcal P X\times X)\times(\mathcal PX\times X):a=b\;\wedge\;y\in a\;\wedge\;z\in b\;\wedge\;\langle y,z\rangle\in L\cup\Id\}$$ on the class $\mathcal P X\times X$. By $(\mathsf{GAP})$,  there exists a maximal $R$-antichain $A\subseteq P X\times X$. The maximality and the antichain property of $A$ ensures that for every nonempty set $a\subseteq X$ there exists a unique element $f(a)\in a$ such that $\langle a,f(a)\rangle\in A$. Then $f$ is a choice function for the power-class $\mathcal P(X)$. By Theorem~\ref{t:CZ}, the class $X$ can be well-ordered.
\end{proof}

\begin{lemma}\label{l:GWL=>WPO} $(\mathsf{GWL})\Ra(\mathsf{WPO})$.
\end{lemma}

\begin{proof} Observe that the power-set $\mathcal P(\Ord)$ carries the lexicographic linear order $L$ defined by
$$L=\{\langle a,b\rangle\in\mathcal P(\Ord)\times\mathcal P(\Ord):\exists \alpha\in b\setminus a\;\wedge\;\forall \beta\in \alpha\;(\beta\in a\;\Leftrightarrow\;\beta\in b))\}.$$ By $(\mathsf{GWL})$, the linearly ordered set $\mathcal P(\Ord)$ can be well-ordered.
\end{proof}

\begin{lemma}\label{l:WPO<=>GWV} $(\mathsf{WPO})\Leftrightarrow(\mathsf{GWV})$. 
\end{lemma}

\begin{proof} The implication $(\mathsf{GWV})\Ra(\mathsf{WPO})$ implies from the inclusion $\mathcal P(\Ord)\subseteq \VV$. To prove that  $(\mathsf{WPO})\Ra(\mathsf{GWV})$, assume that the power-class $\mathcal P(\Ord)$ can be well-ordered and fix a well-order $W$ such that $\dom[W^\pm]=\mathcal P(\Ord)$. 

For every ordinal $\alpha\le \beta$, define a well-order $w_\alpha$ on the set $V_\alpha$ by the recursive formula:
$$w_\alpha=\bigcup_{\delta\in\alpha}((V_\delta\times (V_{\alpha}\setminus V_\delta))\cup\{\langle y,z\rangle\in (\mathcal P(V_\delta)\setminus V_\delta)\times\mathcal (\mathcal P(V_\delta)\setminus V_\delta):\langle \rank_{w_\delta}[y],\rank_{w_\delta}[z]\rangle\in W\}.$$
Then $\bigcup_{\alpha\in\Ord}w_\alpha$ is a set-like well-order of the class $\VV$, witnessing that this class is well-orderable.
\end{proof}

\begin{exercise} (i) Prove that $(\mathsf{GTC})$ implies the principle
\begin{itemize}\setlength{\itemindent}{8pt}
\item[$(\mathsf{AC^c_s}){:}$] For every set $A$ and indexed family of nonempty classes $(X_\alpha)_{\alpha\in A}$ there exists a function $f:A\to\bigcup_{\alpha\in A}X_\alpha$ such that $f(\alpha)\in X_\alpha$ for all $\alpha\in A$.
\end{itemize}
(ii) Prove that $(\mathsf{C}\UU+\mathsf{AC})$ implies $\mathsf{AC^c_s}$.
\end{exercise}

\begin{exercise} Prove that $(\mathsf{C}\UU+\mathsf{AC}_\w)$ implies the principle
\begin{itemize}\setlength{\itemindent}{8pt}
\item[$(\mathsf{AC^c_\w}){:}$] For every  indexed sequence of classes $(X_n)_{n\in\w}$ there exists a function $f:\w\to\bigcup_{n\in\w}X_n$ such that $f(n)\in X_n$ for every $n\in\w$.
\end{itemize}
\end{exercise}

\begin{exercise} Prove that $(\mathsf{C}\UU+\mathsf{TC}_\w)$ implies the principle
\begin{itemize}\setlength{\itemindent}{8pt}
\item[$(\mathsf{TC^c_\w}){:}$] Every ordinary tree $T\subseteq\UU^{<\w}$ has a maximal chain.
\end{itemize}
\end{exercise}

\begin{exercise} Prove that $(\mathsf{TC_\w^c})\;\Ra\;(\mathsf{AC^c_\w})$.
\end{exercise}

\begin{exercise} Prove that $(\mathsf{GwO})\;\Leftrightarrow\;(\mathsf{Gwo}+\mathsf{TC^c_\w})$.
\end{exercise}

\newpage

\part{Ordinal Arithmetics}

In this section we define algebraic operations on ordinals: addition, multiplication, exponentiation. 

\section{Successors}

In this section we analyze the operation 
$$\Succ:\UU\to\UU,\quad \Succ:x\mapsto x\cup\{x\},$$
of taking the successor set.  For a set $x$ its successor $x\cup\{x\}$ will be denoted by $x+1$. 

Let us observe some immediate properties of the function $\Succ$.

\begin{proposition}\label{p:plus1} Let $x,y$ be two sets.
\begin{enumerate}
\item[\textup{1)}] $x+1\subseteq y$ if and only if $x\subseteq y$ and $x\in y$.
\item[\textup{2)}] $x=x+1$ if and only if $x\in x$;
\item[\textup{3)}] $x+1=y+1$ if and only if $(x=y)\;\vee\;((x\ne y)\wedge (x\setminus\{x,y\}{=}y\setminus\{x,y\})\wedge (x\in y)\wedge (y\in x))$.
\end{enumerate}
\end{proposition}

\begin{corollary} If the Axiom of Foundation holds, then for any sets $x,y$
\begin{enumerate}
\item[\textup{1)}] $x\subset x+1$;
\item[\textup{2)}] $x=y$ if and only if $x+1=y+1$.
\end{enumerate}
\end{corollary}

Since the membership relation $\E{\restriction}\Ord$ is an irreflexive well-order on $\Ord$, Proposition~\ref{p:plus1} has the following 

\begin{corollary}\label{c:Succ(alpha)} Let $\alpha,\beta$ be two ordinals.
\begin{enumerate}
\item[\textup{1)}] $\alpha+1\le\beta$ iff $\alpha<\beta$.
\item[\textup{2)}] $\alpha<\alpha+1$.
\item[\textup{3)}] $\alpha+1=\beta+1$ iff $\alpha=\beta$.
\item[\textup{4)}] $\alpha+1<\beta+1$ iff $\alpha<\beta$.
\end{enumerate}
\end{corollary}

Applying Theorem~\ref{t:dynamics} to the $\mathbf S$-progressive function $\Succ:\UU\to\UU$, we obtain the following 

\begin{corollary}\label{c:transucc} There exists a transfinite sequence of functions $(\Succ^{\circ\alpha})_{\alpha\in\Ord}$ such that for every set $x$ and ordinal $\alpha$ the following conditions are satisfied:
\begin{enumerate}
\item[\textup{1)}] $\Succ^{\circ0}(x)=x$;
\item[\textup{2)}] $\Succ^{\circ(\alpha+1)}(x)=\Succ(\Succ^{\circ\alpha}(x))$;
\item[\textup{3)}] $\Succ^{\circ\alpha}(x)=\sup\{\Succ^{\circ\gamma}(x):\gamma\in\alpha\}$ if the ordinal $\alpha>0$ is limit;
\item[\textup{4)}] $\Succ^{\circ\alpha}(x)\subseteq \Succ^{\circ\beta}(x)$ for every ordinal $\beta\ge \alpha$.
\end{enumerate}
\end{corollary}

\section{Addition}

The following theorem  introduces the addition of ordinals.

\begin{theorem} There exists aa unique function $$+:\Ord\times\Ord\to\Ord,\quad+:\langle \alpha,\beta\rangle\mapsto\alpha+\beta$$such that for all  ordinals $\alpha,\beta$ the following conditions are satisfied:
\begin{itemize}
\item[\textup{0)}] $\alpha+0=\alpha$ for any ordinal $\alpha$;
\item[\textup{1)}] $\alpha+(\beta+1)=(\alpha+\beta)+1$ for any ordinals $\alpha,\beta$;
\item[\textup{2)}] $\alpha+\beta=\bigcup\{\alpha+\gamma:\gamma\in\beta\}$ if the ordinal $\beta$ is limit.
\end{itemize} 
\end{theorem}

\begin{proof} Define the addition letting $\alpha+\beta=\Succ^{\circ\beta}(\alpha)$ and apply Corollary~\ref{c:transucc}. The conditions (0)--(2) and Theorem~\ref{t:Ord}(3,4) imply that for any ordinals $\alpha,\beta$ their sum  $\alpha+\beta$ is an ordinal.
\end{proof}


\begin{theorem}\label{t:as} Let $X$ be a chain-inclusive class and $\Phi:X\to X$ be an expansive function. Then for any set $x\in X$ and ordinals $\alpha,\beta$ we have
$$\Phi^{\circ\beta}(\Phi^{\circ\alpha}(x))=\Phi^{\circ(\alpha+\beta)}(x).$$
\end{theorem}

\begin{proof}  This equality will be proved by transfinite induction on $\beta$. Fix $x\in X$ and an ordinal $\alpha$. Observe that
$$\Phi^{\circ0}(\Phi^{\circ\alpha}(x))=\Phi^{\circ\alpha}(x)=\Phi^{\circ(\alpha+0)}(x).$$
Assume that for some ordinal $\beta$ and all its elements $\gamma\in\beta$ we have proved that $\Phi^{\circ\gamma}(\Phi^{\circ\alpha}(x))=\Phi^{\circ(\alpha+\gamma)}(x)$. 

If $\beta$ is a successor ordinal, then $\beta=\gamma+1$ for some ordinal $\gamma<\beta$ and then
\begin{multline*}
\Phi^{\circ\beta}(\Phi^{\circ\alpha}(x))=\Phi^{\circ(\gamma+1)}(\Phi^{\circ\alpha}(x))=
\Phi(\Phi^{\circ\gamma}(\Phi^{\circ\alpha}(x)))=\\
\Phi(\Phi^{\circ(\alpha+\gamma)}(x))=\Phi^{\circ((\alpha+\gamma)+1)}(x)=\Phi^{\circ(\alpha+(\gamma+1))}(x)=\Phi^{\circ(\alpha+\beta)}(x).
\end{multline*}

Next, assume that $\beta$ is a nonzero limit ordinal. In this case the ordinal $\alpha+\beta=\bigcup\{\alpha+\gamma:\gamma\in\beta\}$ is also limit (since for any ordinal $\alpha+\gamma\in\alpha+\beta$ the ordinal $(\alpha+\gamma)+1=\alpha+(\gamma+1)$ also belongs to $\alpha+\beta$). Moreover, $\{\alpha+\gamma:\gamma\in\beta\}\subseteq\alpha+\beta$ and for every $\delta\in\alpha+\beta$ there exists $\gamma\in\beta$ such that $\delta\le\alpha+\gamma$. Then
\begin{multline*}
\textstyle\Phi^{\circ\beta}(\Phi^{\circ\alpha}(x))=\bigcup\{\Phi^{\circ\gamma}(\Phi^{\circ\alpha}(x)):\gamma\in\beta\}=\bigcup\{\Phi^{\circ(\alpha+\gamma)}(x):\gamma\in\beta\}=\\
\textstyle\bigcup\{\Phi^{\circ\delta}(x):\delta\in\alpha+\beta\}=\Phi^{\circ(\alpha+\beta)}(x).
\end{multline*}
\end{proof}

Next we establish some properties of addition of ordinals.

\begin{theorem}\label{t:add-ord}  Let $\alpha,\beta,\gamma$ be ordinals.
\begin{enumerate}
\item[\textup{1)}] $(\alpha+\beta)+\gamma=\alpha+(\beta+\gamma)$.
\item[\textup{2)}] $0+\alpha=\alpha=\Succ^{\circ\alpha}(0)$ for any ordinal.
\item[\textup{3)}] If $\alpha\le\beta$, then $\alpha+\gamma\le\beta+\gamma$.
\item[\textup{4)}] $\beta<\gamma$ if and only if $\alpha+\beta<\alpha+\gamma$.
\item[\textup{5)}] For any ordinals $\alpha\le\beta$ there exists a unique ordinal $\gamma$ such that $\alpha+\gamma=\beta$.
\end{enumerate}
\end{theorem}

\begin{proof} 1. The equality $(\alpha+\beta)+\gamma=\alpha+(\beta+\gamma)$ is nothing else but the equality $$\Succ^{\circ\gamma}(\Succ^{\circ\beta}(\alpha))=\Succ^{\circ(\beta+\gamma)}(\alpha)$$established in Theorem~\ref{t:as}.
\smallskip

2.  The equality $0+\alpha=\alpha$ will be proved by transfinite induction. For $\alpha=0$ the equality $0+0=0$ holds. Assume that for some ordinal $\alpha>0$ we have proved that $0+\beta=\beta$ for all $\beta\in\alpha$. If $\alpha=\beta+1$ is a successor ordinal, then $$0+\alpha=0+(\beta+1)=(0+\beta)+1=\beta+1=\alpha$$by the induction hypothesis. 

If $\alpha$ is a limit ordinal, then 
$$0+\alpha=\sup\{0+\beta:\beta\in\alpha\}=\sup\{\beta:\beta\in\alpha\}=\alpha.$$
\smallskip

3.  Assume that $\alpha\le \beta$. The inequality $\alpha+\gamma\le\beta+\gamma$ will be proved by transfinite induction. For $\gamma=0$ the inequality $\alpha+0=\alpha\le\beta=\beta+0$ trivially holds. Assume that for some nonzero ordinal $\gamma$ and all its elements $\delta\in\gamma$ we have proved that $\alpha+\delta\le\beta+\delta$, which implies $(\alpha+\delta)+1\subseteq(\beta+\delta)+1$. 
If $\gamma=\delta+1$ for some ordinal $\delta$, then
$$\alpha+\gamma=\alpha+(\delta+1)=(\alpha+\delta)+1\subseteq(\beta+\delta)+1=\beta+(\delta+1)=\beta+\gamma.$$
If $\gamma$ is a limit ordinal, then 
$\alpha+\gamma=\sup\{\alpha+\delta:\delta\in\gamma\}\le\sup\{\beta+\delta:\delta\in\gamma\}=\beta+\gamma$.
\smallskip

4. If $\beta<\gamma$, then $\beta+1\le\gamma$ and $\alpha+\beta=\Succ^{\circ\beta}(\alpha)\subset \Succ(\Succ^{\circ\beta}(\alpha))\subseteq \Succ^{\circ\gamma}(\alpha)=\alpha+\gamma$, 
by Corollaries~\ref{c:Succ(alpha)} and \ref{c:transucc}. Therefore, $\beta<\gamma$ implies $\alpha+\beta<\alpha+\gamma$, and $\beta\le\gamma$ implies $\alpha+\beta\le\alpha+\gamma$. Now we prove that $\alpha+\beta<\alpha+\gamma$ implies $\beta<\gamma$. In the opposite case, we get $\gamma\le\beta$ (by Theorem~\ref{t:ord}(6)) and hence $\alpha+\gamma\le\alpha+\beta$, which contradicts our assumption.
\smallskip

5. Given two ordinals $\alpha\le\beta$, consider the class $\Gamma=\{\gamma\in\Ord:\alpha+\gamma\subseteq\beta\}$. By Theorem~\ref{t:add-ord}(3), for every $\gamma\in\Gamma$ we have $\gamma=0+\gamma\le\alpha+\gamma\le\beta$ and hence $\Gamma\subseteq\mathcal P(\beta)$ is a set. By Theorem~\ref{t:Ord}(5) and Lemma~\ref{l:sup-ord}, the set $\gamma=\sup \Gamma=\bigcup\Gamma$ is an ordinal.  First we prove that $\alpha+\gamma\le\beta$.

If $\gamma=\sup \Gamma$ is a successor ordinal, then $\gamma\in \Gamma$ and hence $\alpha+\gamma\subseteq\beta$ by the definition of the set $\Gamma$. If $\gamma$ is a limit ordinal, then 
$$
\alpha+\gamma=\alpha+\sup \Gamma=\sup\{\alpha+\delta:\delta\in \Gamma\}\le \beta$$
as $\alpha+\delta\subseteq\beta$ for all $\delta\in \Gamma$.
Therefore, $\alpha+\gamma\le\beta$. Assuming that $\alpha+\gamma\ne\beta$, we conclude that $\alpha+\gamma<\beta$ and hence $\alpha+(\gamma+1)=(\alpha+\gamma)+1\le\beta$ by  Proposition~\ref{p:plus1}(1). Then $\gamma+1\in \Gamma$ and $\gamma\in \gamma+1\subseteq\sup \Gamma=\gamma$, which contradicts the irreflexivity of the relation $\E{\restriction}\Ord$. This contradiction shows that $\alpha+\gamma=\beta$. The uniqueness of the ordinal $\gamma$ follows from Theorems~\ref{t:add-ord}(4) and \ref{t:ord}(6).
\end{proof}

Finally we establish the commutativity of  addition for natural numbers.

\begin{theorem}\label{t:add-nat} Let $k,n\in\w$ be two natural numbers.
\begin{enumerate}
\item[\textup{1)}] $k+n\in\w$;
\item[\textup{2)}] $n+1=1+n$;
\item[\textup{3)}] $n+k=k+n$.
\end{enumerate}
\end{theorem}

\begin{proof} Fix any natural number $k\in\w$.
\smallskip

1. The inclusion  $k+n\in\w$ will be proved by Mathematical Induction. For $n=0$, we have $k+0=k\in\w$ by Theorem~\ref{t:add-ord}(2). Assume that for some $n\in\w$ we have proved that $k+n\in\w$. Taking into account that $\w$ is a limit ordinal, we conclude that $k+(n+1)=(k+n)+1\in\w$. By the Principle of Mathematical Induction, $\forall n\in\w\;(k+n\in\w)$.
\smallskip

2. By Theorems~\ref{t:add-ord}(2), $1+0=1=0+1$. Assume that for some $n\in\w$ we have proved that $1+n=n+1$. Then by Theorem~\ref{t:add-ord}(1), $1+(n+1)=(1+n)+1=(n+1)+1$. By the Principle of Mathematical Induction, the equality $1+n=n+1$ holds for all $n\in\w$.
\smallskip

3. Take any $k\in\w$. The equality $k+n=n+k$ will be proved by induction on $n\in\w$. For $n=0$ the equality $k+0=k=0+k$ holds by Theorem~\ref{t:add-ord}(2). Assume that for some $n\in\w$ the equality $k+n=n+k$ has been proved. By the inductive assumption and  Theorems~\ref{t:add-ord}(1), \ref{t:add-nat}(2), we obtain
$$k+(n+1)=(k+n)+1=1+(k+n)=1+(n+k)=(1+n)+k=(n+1)+k.$$
\end{proof}

\begin{exercise} Find two ordinals $\alpha,\beta$ such that $\alpha+\beta\ne\beta+\alpha$.
\smallskip

\noindent{\em Hint:} Show that $1+\w=\w\ne\w+1$.
\end{exercise}

\begin{exercise} Prove that every ordinal $\alpha$ can be uniquely written as the sum $\alpha=\beta+n$ of a limit ordinal $\beta$ and a natural number $n$.
\end{exercise}

In fact, the operation of addition can be defined ``geometrically'' for any relations. Namely, for two relations $R,P$ their sum $R\sqcuplus P$ is defined as the relation $$
\begin{aligned}
R\uplus P=\;&\{\langle\langle 0,x\rangle,\langle 0,y\rangle\rangle:\langle x,y\rangle\in R\}
\cup\{\langle\langle 1,x\rangle,\langle 1,y\rangle\rangle:\langle x,y\rangle\in P\}\cup\\&\{\langle\langle 0,x\rangle,\langle 1,y\rangle\rangle:\langle x,y\rangle\in\dom[R^\pm]\times \dom[P^\pm]\}
\end{aligned}
$$
with the underlying class $\dom[(R+P)^\pm]=(\{0\}\times \dom[R^\pm])\cup(\{1\}\times\dom[P^\pm])$.

\begin{exercise} Given ordinals $\alpha,\beta$, prove that
\begin{enumerate}
\item the order $\E{\restriction}(\alpha+\beta)$ is isomorphic to $\E{\restriction}\alpha\sqcuplus \E{\restriction}\beta$;
\item $\alpha+\beta=\rank(\E{\restriction}\alpha\sqcuplus \E{\restriction}\beta)$.
\end{enumerate}
\end{exercise}

\section{Multiplication}

The following theorem introduces the operation of multiplication of ordinals.

\begin{theorem}\label{t:mult-ord} There exists a unique function $$\cdot:\Ord\times\Ord\to\Ord,\quad \cdot:\langle\alpha,\beta\rangle\mapsto \alpha\cdot\beta,$$ such that for any ordinals $\alpha,\beta$ the following properties are satisfied:
\begin{enumerate}
\item[\textup{0)}] $\alpha\cdot 0=0$;
\item[\textup{1)}]  $\alpha\cdot(\beta+1)=\alpha\cdot \beta+\alpha$;
\item[\textup{2)}]  $\alpha\cdot \beta=\sup\{\alpha\cdot\gamma:\gamma\in\beta\}$ if the ordinal $\beta$ is limit.
\end{enumerate}
\end{theorem}

\begin{proof} The uniqueness of $\cdot$ can be proved by transfinite induction on $\beta$. The existence of $\cdot$ follows from Theorem~\ref{t:dynamics} applied to the function
$$\Phi_\alpha:\Ord\to\Ord,\quad\Phi_\alpha:x\mapsto x+\alpha=\Succ^{\circ\alpha}(0).$$
Then $\alpha\cdot\beta=\Phi_\alpha^{\circ\beta}(0)$.
\end{proof}

\begin{lemma}\label{l:0alpha} $0\cdot\alpha=0$ for any ordinal $\alpha$.
\end{lemma}

\begin{proof} This equality will be proved by transfinite induction. For $\alpha=0$ it follows from the definition of multiplication by zero. Assume that for a nonzero ordinal $\beta$ we proved that $0\cdot\gamma=0$ for all $\gamma\in\beta$. If $\beta$ is a successor ordinal, then $\beta=\gamma+1$ for some ordinal $\gamma\in\beta$ and hence
$$0\cdot\beta=0\cdot(\gamma+1)=0\cdot\gamma+0=0+0=0.$$
If $\beta$ is a limit ordinal, then
$$0\cdot\beta=\sup\{0\cdot\gamma:\gamma\in\beta\}=\sup\{0\}=0.$$
\end{proof}

\begin{theorem}\label{t:ass} Let $X$ be a chain-inclusive class and $\Phi:X\to X$ be an expansive function. Then for any set $x\in X$ and ordinals $\alpha,\beta,\gamma$ we have
\begin{enumerate}
\item[\textup{1)}] $(\Phi^{\circ\alpha})^{\circ\beta}(x)=\Phi^{\circ(\alpha\cdot\beta)}(x)$;
\item[\textup{2)}] $\Phi^{\circ(\alpha\cdot(\beta+\gamma))}(x)=\Phi^{\circ(\alpha\cdot\beta+\alpha\cdot\gamma)}(x)$.
\end{enumerate}
\end{theorem}

\begin{proof} 1. If $\alpha=0$, then by Lemma~\ref{l:0alpha},
$$(\Phi^{\circ0})^{\circ\beta}(x)=x=\Phi^{\circ0}(x)=\Phi^{\circ(\alpha\cdot0)}(x).$$
So, we assume that $\alpha\ne0$.

 The equality $(\Phi^{\circ\alpha})^{\circ\beta}(x)=\Phi^{\circ(\alpha\cdot\beta)}(x)$ will be proved by transfinite induction on $\beta$. Fix $x\in X$ and an ordinal $\alpha$. Observe that $(\Phi^{\circ\alpha})^{\circ0}(x)=x=\Phi^{\circ0}(x)=\Phi^{\circ(\alpha\cdot 0)}(x)$. Assume that for some ordinal $\beta$ and all its elements $\gamma\in\beta$ we have proved that $(\Phi^{\circ\alpha})^{\circ\gamma}=\Phi^{\circ(\alpha\cdot\gamma)}(x)$. 

If $\beta$ is a successor ordinal, then $\beta=\gamma+1$ for some ordinal $\gamma$ and by the inductive assumption, Theorem~\ref{t:as} and the definition of ordinal multiplication,
\begin{multline*}
(\Phi^{\circ\alpha})^{\circ\beta}(x)=(\Phi^{\circ\alpha})^{\circ(\gamma+1)}(x)=\Phi^{\circ\alpha}((\Phi^{\circ\alpha})^{\circ\gamma}(x))=\Phi^{\circ\alpha}(\Phi^{\circ(\alpha\cdot\gamma)}(x))=\\
\Phi^{\circ((\alpha\cdot\gamma)+\alpha)}(x)=\Phi^{\circ(\alpha\cdot(\gamma+1))}(x)=\Phi^{\circ(\alpha\cdot\beta)}(x).
\end{multline*}

Next, assume that $\beta$ is a limit ordinal. In this case the ordinal $\alpha\cdot\beta=\bigcup\{\alpha\cdot\gamma:\gamma\in\beta\}$ is also limit (since for any ordinal $\alpha\cdot\gamma\in\beta$ the ordinal $(\alpha\cdot\gamma)+1\le\alpha\cdot\gamma+\alpha=\alpha\cdot(\gamma+1)$ also belongs to $\alpha+\beta$). Moreover, $\{\alpha\cdot\gamma:\gamma\in\beta\}\subseteq\alpha\cdot\beta$ and for every $\delta\in\alpha\cdot\beta$ there exists $\gamma\in\beta$ such that $\delta\le\alpha\cdot\gamma$. Then
$$
\textstyle(\Phi^{\circ\alpha})^{\circ\beta}(x)=\bigcup\{(\Phi^{\circ\alpha})^{\circ\gamma}(x):\gamma\in\beta\}=\bigcup\{\Phi^{\circ(\alpha\cdot\gamma)}(x):\gamma\in\beta\}=
\textstyle\bigcup\{\Phi^{\circ\delta}:\delta\in\alpha{\cdot}\beta\}=\Phi^{\circ(\alpha{\cdot}\beta)}(x).
$$
\smallskip

2. Applying Theorems~\ref{t:as} and \ref{t:ass}(1), we see that
$$
\Phi^{\circ(\alpha\cdot(\beta+\gamma))}(x)=(\Phi^{\circ\alpha})^{\circ(\beta+\gamma)}(x)=(\Phi^{\circ\alpha})^{\circ\gamma}((\Phi^{\circ\alpha})^{\circ\beta}(x))=
\Phi^{\circ(\alpha\cdot\gamma)}(\Phi^{\circ(\alpha\cdot\beta)}(x))=\Phi^{\circ(\alpha\cdot\beta+\alpha\cdot\gamma)}(x).
$$
\end{proof}

Next we establish some properties of multiplication of ordinals.

\begin{theorem}\label{t:mult-ord-prop}  Let $\alpha,\beta,\gamma$ be ordinals.
\begin{enumerate}
\item[\textup{1)}] $(\alpha\cdot\beta)\cdot\gamma=\alpha\cdot(\beta\cdot\gamma)$.
\item[\textup{2)}] $\alpha\cdot(\beta+\gamma)=\alpha\cdot\beta+\alpha\cdot\gamma$.
\item[\textup{3)}] If $\alpha\le\beta$, then $\alpha\cdot\gamma\le\beta\cdot\gamma$.
\item[\textup{4)}] If $\alpha>0$ and $\beta>0$, then $\alpha\cdot\beta>0$.
\item[\textup{5)}] If $\alpha>0$ and $\beta<\gamma$, then $\alpha\cdot\beta<\alpha\cdot\gamma$.
\item[\textup{6)}] For any ordinals $\alpha\ne0$ and $\beta$ there exist unique ordinals $\gamma$ and $\delta$ such that $\beta=\alpha\cdot\gamma+\delta$ and $\delta<\alpha$.
\end{enumerate}
\end{theorem}

\begin{proof} 1. Applying Theorems~\ref{t:add-ord}(2) and \ref{t:ass}(1), we obtain
\begin{multline*}
\alpha\cdot(\beta\cdot\gamma)=\Succ^{\circ(\alpha\cdot(\beta\cdot\gamma))}(0)=(\Succ^{\circ\alpha})^{\circ{(\beta\cdot\gamma)}}(0)=((\Succ^{\circ\alpha})^{\circ\beta})^{\circ\gamma}(0)=\\ (\Succ^{\circ(\alpha\cdot\beta)})^{\circ\gamma}(0)=
\Succ^{\circ((\alpha\cdot\beta)\cdot\gamma)}(0)=(\alpha\cdot\beta)\cdot\gamma.
\end{multline*}

2.  Applying Theorems~\ref{t:add-ord}(2) and \ref{t:ass}(2), we obtain
$$
\alpha\cdot(\beta+\gamma)=\Succ^{\circ(\alpha\cdot(\beta+\gamma))}(0)=\Succ^{\circ(\alpha{\cdot}\beta+\alpha{\cdot}\gamma)}(0)=\alpha\cdot\beta+\alpha\cdot\gamma.$$


3.  Assume that $\alpha\le \beta$. The inequality $\alpha\cdot\gamma\le\beta\cdot\gamma$ will be proved by transfinite induction. For $\gamma=0$ the inequality $\alpha\cdot0=0=\beta\cdot0$ trivially holds. Assume that for some nonzero ordinal $\gamma$ and all its elements $\delta\in\gamma$ we have proved that $\alpha\cdot\delta\le\beta\cdot\delta$, which implies $\alpha{\cdot}\delta+\alpha\le\beta{\cdot}\delta+\alpha\le\beta{\cdot}\delta+\beta$, see Theorem~\ref{t:add-ord}(3,4). 
If $\gamma=\delta+1$ for some ordinal $\delta$, then
$$\alpha\cdot\gamma=\alpha\cdot(\delta+1)=\alpha\cdot\delta+\alpha\le\beta\cdot\delta+\beta=\beta\cdot(\delta+1)=\beta\cdot\gamma.$$
If $\gamma$ is a limit ordinal, then 
$\alpha\cdot\gamma=\sup\{\alpha\cdot\delta:\delta\in\gamma\}\le\sup\{\beta\cdot\delta:\delta\in\gamma\}=\beta\cdot\gamma$.
\smallskip

4. Assume that $\alpha$ and $\beta$ are nonzero ordinals. By the preceding statement $\alpha\cdot\beta\ge 1\cdot\beta=\beta>0$.
\smallskip

5. If $\alpha>0$ and $\beta<\gamma$, then by Theorem~\ref{t:add-ord}(5), we can find a unique ordinal $\delta$ such that $\beta+\delta=\gamma$. Since $\beta\ne\gamma$, the ordinal $\delta$ is nonzero. By Theorem~\ref{t:mult-ord-prop}(4), $0<\alpha\cdot\delta$. By Theorems~\ref{t:add-ord}(4) and \ref{t:mult-ord}(2), we have
$$\alpha\cdot\beta=\alpha\cdot\beta+0<\alpha\cdot\beta+\alpha\cdot\delta=\alpha\cdot(\beta+\delta)=\alpha\cdot\gamma.$$
\smallskip

6. Given two ordinals $\alpha>0$ and $\beta$, consider the class $\Gamma=\{\gamma\in\Ord:\alpha\cdot\gamma\le\beta\}$.  By Theorem~\ref{t:mult-ord-prop}(3), for every $\gamma\in\Gamma$ we have $\gamma=1\cdot\gamma\le\alpha\cdot\gamma\le\beta$  and hence $\Gamma\subseteq\mathcal P(\beta)$ is a set. By Theorem~\ref{t:Ord}(5) and Lemma~\ref{l:sup-ord}, the set $\gamma=\sup \Gamma=\bigcup\Gamma$ is an ordinal.  First we prove that $\alpha\cdot \gamma\le\beta$.

If $\gamma=\sup \Gamma$ is a successor ordinal, then $\gamma\in \Gamma$ and hence $\alpha\cdot\gamma\subseteq\beta$ by the definition of the set $\Gamma$. If $\gamma$ is a limit ordinal, then 
$$
\alpha\cdot\gamma=\alpha\cdot\sup \Gamma=\sup\{\alpha\cdot\delta:\delta\in \Gamma\}\le \beta$$
as $\alpha\cdot\delta\subseteq\beta$ for all $\delta\in \Gamma$.
Therefore, $\alpha\cdot\gamma\le\beta$. By Theorem~\ref{t:add-ord}(5), there exists a unique ordinal $\delta$ such that $\alpha\cdot\gamma+\delta=\beta$. We claim that $\delta<\alpha$. Assuming that $\delta\not<\alpha$ and applying Theorem~\ref{t:ord}(6), we conclude that $\alpha\le\delta$. By Theorem~\ref{t:add-ord}(5), there exists an ordinal $\delta'$ such that $\alpha+\delta'=\delta$. Then $$\beta=\alpha\cdot\gamma+\delta=\alpha\cdot\gamma+(\alpha+\delta')=(\alpha\cdot\gamma+\alpha)+\delta'=\alpha\cdot(\gamma+1)+\delta'\ge \alpha\cdot(\gamma+1)$$
and hence $\gamma+1\in\Gamma$ and $\gamma+1\le\sup\Gamma=\gamma$, which contradicts the irreflexivity of the relation $\E{\restriction}\Ord$. This contradiction shows that $\delta<\alpha$. 

It remains to show that the ordinals $\gamma$ and $\delta$ are unique. Assume that $\gamma',\delta'$ are ordinals such that $\delta'<\alpha$ and $\alpha{\cdot}\gamma'+\delta'=\beta$. Since $$\alpha\cdot\gamma'=\alpha\cdot\gamma'+0\le\alpha\cdot\gamma'+\delta'=\beta,$$
the ordinal $\gamma'$ belongs to $\Gamma$ and hence $\gamma'\le\gamma$. Assuming that $\gamma'\ne\gamma$, we conclude that $\gamma'<\gamma$. By Theorem~\ref{t:add-ord}(5), there exists a unique ordinal $\delta''>0$ such that $\gamma'+\delta''=\gamma$. Then $$\alpha\cdot\gamma'+\delta'=\beta=\alpha\cdot\gamma+\delta=\alpha\cdot(\gamma'+\delta'')+\delta=(\alpha\cdot\gamma'+\alpha\cdot\delta'')+\delta=\alpha\cdot\gamma'+(\alpha\cdot\delta''+\delta).$$
By Theorem~\ref{t:add-ord}(5,4), \ref{t:mult-ord-prop}(5) $$\delta'=\alpha\cdot\delta''+\delta\ge \alpha\cdot\delta''+0=\alpha\cdot\delta''\ge\alpha\cdot 1=\alpha,$$
which contradicts the choice of $\delta'<\alpha$. This contradiction completes the proof of the equality $\gamma=\gamma'$. Now the equality 
$$\alpha\cdot\gamma+\delta'=\alpha\cdot\gamma'+\delta'=\beta=\alpha\cdot\gamma+\delta$$ and Theorem~\ref{t:add-ord}(5) imply $\delta=\delta'$.
\end{proof}

\begin{exercise} Find two ordinals $\alpha,\beta$ such that $\alpha\cdot\beta\ne\beta\cdot\alpha$.
\smallskip

\noindent{\em Hint:} Show that $\w\cdot 2=\w+\w\ne\w=2\cdot\w$.
\end{exercise}

\begin{theorem}\label{t:mult-nat} For any natural numbers $n,k\in\w$ the following conditions hold:
\begin{enumerate}
\item[\textup{1)}] $n\cdot k\in\w$;
\item[\textup{2)}] $(n+1)\cdot k=n\cdot k+k$;
\item[\textup{3)}] $n\cdot k=k\cdot n$.
\end{enumerate}
\end{theorem}

\begin{proof} 1. Fix any natural number $n$. The inclusion $n\cdot k\in\w$ will be proved by induction on $k$. For $k=0$ we have $n\cdot 0=0\in\w$. Assume that for some $k\in\w$ we have proved that $n{\cdot}k\in\w$. Then $n\cdot(k+1)=n{\cdot}k+n\in\w$ by Theorem~\ref{t:add-nat}(1).
\smallskip

2. The equality $(n+1)\cdot k=n{\cdot}k+k$ will be proved by induction on $k$. For $k=0$ we have $(n+1)\cdot 0=0=n{\cdot}0+0$. Assume that for some $k\in\w$ we have proved that $(n+1)\cdot k=n\cdot k+k$. Taking into account that the addition of natural numbers is associative and commutative, we conclude that
$$(n+1)\cdot(k+1)=(n+1)\cdot k+(n+1)=n\cdot k+k+n+1=n\cdot k+n+k+1=n\cdot (k+1)+(k+1).$$By the Principle of Mathematical Induction, the equality $(n+1)\cdot k=n\cdot k+k$ holds for all natural numbers. 
\smallskip

3. The equality $n\cdot k=k\cdot n$ will be proved by induction on $n\in\w$. For $n=0$ we have $0\cdot k=0=k\cdot 0$, by Lemma~\ref{l:0alpha} and the definition of the multiplication. Assume that for some $n$ we have proved that $k\cdot n=n\cdot k$. By the preceding statement,
$$(n+1)\cdot k=n\cdot k+k=k\cdot n+k=k\cdot(n+1).$$
\end{proof}

In fact, the operation of multiplication can be defined ``geometrically'' for any relations. Namely, for two relations $R,P$ we can defined their left and right lexicographic products $R\ltimes P$ and $R\rtimes P$ by the formulas:
\begin{multline*}
R\ltimes P=\{\langle \langle x,y\rangle,\langle x',y'\rangle\rangle\colon\\
(\langle x,x'\rangle\in R\setminus \Id\;\wedge\;\{y,y'\}\subseteq\dom[P^\pm])\;\vee\;(x=x'\in\dom[R^\pm]\;\wedge\;\langle y,y'\rangle\in P)\}
\end{multline*}
and 
\begin{multline*}
R\rtimes P=\{\langle \langle x,y\rangle,\langle x',y'\rangle\rangle\colon\\
(\langle y,y'\rangle\in P\setminus\Id\;\wedge\;\{x,x'\}\subseteq\dom[R^\pm])\;\vee\;(y=y'\in\dom[P^\pm]\;\wedge\;\langle x,x'\rangle\in R)\}.
\end{multline*}

\begin{exercise} Given any  ordinals $\alpha,\beta$, prove that
\begin{enumerate}
\item the well-order $\E{\restriction}\alpha{\cdot}\beta$ is isomorphic to the orders $\E{\restriction}\alpha\rtimes \E{\restriction}\beta$ and $\E{\restriction}\beta\ltimes \E{\restriction}\alpha$;
\item $\alpha{\cdot}\beta=\rank(\E{\restriction}\alpha\rtimes \E{\restriction}\beta)=\rank(\E{\restriction}\beta\ltimes \E{\restriction}\alpha)$.
\end{enumerate}
\end{exercise}

\section{Exponentiation}

The following theorem introduces the operation of exponentiation of ordinals.

\begin{theorem}\label{t:exp-ord-def} There exists a unique function $$\exp:\Ord\times\Ord\to\Ord,\quad \exp:\langle\alpha,\beta\rangle\mapsto \alpha^{\cdot\beta},$$ such that for any ordinals $\alpha,\beta$ the following properties are satisfied:
\begin{enumerate}
\item[\textup{0)}] $\alpha^{\cdot0}=1$;
\item[\textup{1)}]  $\alpha^{\cdot(\beta+1)}=\alpha^{\cdot\beta}\cdot\alpha$;
\item[\textup{2)}]  $\alpha^{\cdot\beta}=\sup\{\alpha^{\cdot\gamma}:\gamma\in\beta\}$ if the ordinal $\beta$ is limit.
\end{enumerate}
\end{theorem}

\begin{proof} Given an ordinal $\alpha$, consider the expanding function $$\Phi_\alpha:\Ord\to\Ord,\quad\Phi_\alpha:x\mapsto x\cdot \alpha.$$ By Theorem~\ref{t:dynamics} there exists a transfinite function sequence $(\Phi_\alpha^{\circ\beta})_{\beta\in\Ord}$ such that for any ordinals $x$ and $\beta$ the following conditions hold:
\begin{itemize}
\item $\Phi_\alpha^{\circ0}(x)=x$;
\item $\Phi_\alpha^{\circ(\beta+1)}(x)=\Phi_\alpha(\Phi_\alpha^{\circ\beta}(x))=\Phi^{\circ\beta}(x)\cdot \alpha$;
\item $\Phi_\alpha^{\circ\beta}(x)=\sup\{\Phi_\alpha^{\circ\gamma}(x):\gamma\in\beta\}$ if the ordinal $\beta$ is limit.
\end{itemize}
Comparing these three conditions with the definition of exponentiation, we can see that $\alpha^\beta=\Phi_\alpha^{\circ\beta}(1)$ for all $\alpha,\beta$.
\end{proof}

\begin{remark} In most textbooks on Set Theory,  the ordinal exponentiation is denoted by $\alpha^\beta$, which unfortunately coincides with the notation $\alpha^\beta$ for the set of all functions from $\beta$ to $\alpha$. For distinguishing these two notions we denote the ordinal exponentiation $\alpha^{\cdot\beta}$ using the dot before $\beta$, which indicates that $\alpha^{\cdot\beta}$ is the result of repeated multiplication of $1$ by $\alpha$, $\beta$ times.
\end{remark}   

Now we establish some properties of the exponentiation of ordinals.

\begin{theorem}\label{t:exp-ord} Let $\alpha,\beta,\gamma$ be any non-zero ordinals.
\begin{enumerate}
\item[\textup{1)}] If $\alpha\le\beta$, then $\alpha^{\cdot\gamma}\le\beta^{\cdot\gamma}$.
\item[\textup{2)}] $\alpha^{\cdot(\beta+\gamma)}=\alpha^{\cdot\beta}\cdot\alpha^{\cdot\gamma}$.
\item[\textup{3)}] $\alpha^{\cdot(\beta\cdot\gamma)}=(\alpha^{\cdot\beta})^{\cdot\gamma}$.
\item[\textup{4)}] If $\beta\le\gamma$, then $\alpha^{\cdot\beta}\le\alpha^{\cdot\gamma}$.
\item[\textup{5)}] If $\alpha>1$ and $\beta<\gamma$, then $\alpha^{\cdot\beta}<\alpha^{\cdot\gamma}$.
\item[\textup{6)}] For any ordinals $\alpha>1$ and $\beta\ge \alpha$ there exists unique ordinals $x,y,z$ such that $\beta=\alpha^{\cdot x}\cdot y+z$, $0<y<\alpha$, and $z<\alpha^{\cdot x}$.
\item[\textup{7)}] For any ordinal $\beta\ge\w$ there are unique ordinals $x,y,z$ such that $0<y<\w$, $z<\w^{\cdot x}$ and $\beta=\w^{\cdot x}{\cdot}y+z$.
\end{enumerate}
\end{theorem}

\begin{proof} 1. Take any ordinals $\alpha\le\beta$. The inequality $\alpha^{\cdot\gamma}\le\beta^{\cdot\gamma}$ will be proved by induction on the ordinal $\gamma$. For $\gamma=0$ the inequality $\alpha^{\cdot0}=1=\beta^{\cdot0}$ is true. Assume that for some ordinal $\gamma$ and all ordinals $\delta\in\gamma$ we have proved that $\alpha^{\cdot\delta}\le\beta^{\cdot\delta}$.

If $\gamma$ is a successor ordinal, then $\gamma=\delta+1$ for some $\delta\in\gamma$. By Theorem~\ref{t:mult-ord-prop}(3,5), 
$$\alpha^{\cdot\gamma}=\alpha^{\cdot(\delta+1)}=\alpha^{\cdot\delta}\cdot\alpha\le\beta^{\cdot\delta}\cdot \beta=\beta^{\cdot(\delta+1)}=\beta^{\cdot\gamma}.$$
If the ordinal $\gamma$ is limit, then
$$\alpha^{\cdot\gamma}=\sup\{\alpha^{\cdot\delta}:\delta\in\gamma\}\le\sup\{\beta^{\cdot\delta}:\delta\in\gamma\}=\beta^{\cdot\gamma}.$$
\smallskip

2. Fix ordinals $\alpha,\beta$. The equality $\alpha^{\cdot(\beta+\gamma)}=\alpha^{\cdot\beta}\cdot\alpha^{\cdot\gamma}$ will be proved by transfinite induction on $\gamma$. For $\gamma=0$ we have $$\alpha^{\cdot(\beta+0)}=\alpha^{\cdot\beta}=\alpha^{\cdot\beta}\cdot 1=\alpha^{\cdot\beta}\cdot\alpha^{\cdot0}.$$
 
 Assume that for some ordinal $\gamma$ and all its elements $\delta\in\gamma$ we have proved that $\alpha^{\cdot(\beta+\delta)}=\alpha^{\cdot\beta}\cdot\alpha^{\cdot\delta}$. If $\gamma$ is a successor ordinal, then $\gamma=\delta+1$ for some $\delta\in\gamma$ and then
 $$\alpha^{\cdot(\beta+\gamma)}=\alpha^{\cdot(\beta+\delta+1)}=\alpha^{\cdot(\beta+\delta)}\cdot \alpha=(\alpha^{\cdot\beta}\cdot\alpha^{\cdot\delta})\cdot\alpha=\alpha^{\cdot\beta}\cdot(\alpha^{\cdot\delta}\cdot\alpha)=\alpha^{\cdot\beta}\cdot\alpha^{\cdot(\delta+1)}=\alpha^{\cdot\beta}\cdot\alpha^{\cdot\gamma}.$$
 If $\gamma$ is a limit ordinal, then $\beta+\gamma=\sup\{\beta+\delta:\delta\in\gamma\}$ is limit too and hence
 $$\alpha^{\cdot(\beta+\gamma)}=\sup\{\alpha^{\cdot(\beta+\delta)}:\delta\in\gamma\}=\sup\{\alpha^{\cdot\beta}\cdot\alpha^{\cdot\delta}:\delta\in\gamma\}=\alpha^{\cdot\beta}\cdot\sup\{\alpha^{\cdot\delta}:\delta\in\gamma\}=\alpha^{\cdot\beta}\cdot\alpha^{\cdot\gamma}.$$
\smallskip

3. Fix ordinals $\alpha,\beta$. If $\beta=0$, then for any ordinal $\gamma$ we have
$$\alpha^{\cdot(\beta\cdot\gamma)}=\alpha^{\cdot(0\cdot\gamma)}\alpha^{\cdot0}=1=1^{\cdot \gamma}=(\alpha^{\cdot0})^{\cdot\gamma}=(\alpha^{\cdot\beta})^{\cdot\gamma}.$$ So, we assume that $\beta>0$.

 The equality $\alpha^{\cdot(\beta\cdot\gamma)}=(\alpha^{\cdot\beta})^{\cdot\gamma}$ will be proved by transfinite induction on $\gamma$. For $\gamma=0$ we have $$\alpha^{\cdot(\beta\cdot 0)}=\alpha^{\cdot0}=1=(\alpha^{\cdot\beta})^{\cdot0}.$$
Assume that for some ordinal $\gamma$ and all its elements $\delta\in\gamma$ we have proved that $\alpha^{\cdot(\beta\cdot\delta)}=(\alpha^{\cdot\beta})^{\cdot\delta}$. If $\gamma$ is a successor ordinal, then $\gamma=\delta+1$ for some $\delta\in\gamma$ and then
 $$\alpha^{\cdot(\beta\cdot\gamma)}=\alpha^{\cdot(\beta\cdot(\delta+1))}=\alpha^{\cdot(\beta\cdot\delta+\beta)}=\alpha^{\cdot(\beta\cdot\delta)}\cdot \alpha^{\cdot\beta}=(\alpha^{\cdot\beta})^{\cdot\delta}\cdot\alpha^{\cdot\beta}=(\alpha^{\cdot\beta})^{\cdot(\delta+1)}=(\alpha^{\cdot\beta})^{\cdot\gamma}.$$
 
If $\gamma$ is a limit ordinal, then $\beta\cdot\gamma=\sup\{\beta\cdot\delta:\delta\in\gamma\}$ is limit, too and hence
 $$\alpha^{\cdot(\beta\cdot\gamma)}=\sup\{\alpha^{\cdot(\beta\cdot\delta)}:\delta\in\gamma\}=\sup\{(\alpha^{\cdot\beta})^{\cdot\delta}:\delta\in\gamma\}=(\alpha^{\cdot\beta})^{\cdot\gamma}.$$
 \smallskip
 
4. If $\beta\le\gamma$, then by Theorem~\ref{t:add-ord}(5), there exists a unique ordinal $\delta$ such that $\gamma=\beta+\delta$. Applying Theorem~\ref{t:mult-ord-prop}(5), we obtain 
$$\alpha^{\cdot\beta}=\alpha^{\cdot\beta}\cdot 1\le\alpha^{\cdot\beta}\cdot\alpha^{\cdot\delta}=\alpha^{\cdot(\beta+\delta)}=\alpha^{\cdot\gamma}.$$
\smallskip
 
5. Let $\alpha>1$ and $\beta<\gamma$ be any ordinals.  By Theorem~\ref{t:add-ord}(5), there exists a unique ordinal $\delta\ge 1$ such that $\beta+\delta=\gamma$. By Theorem~\ref{t:mult-ord-prop}(5), $$\alpha^{\cdot\beta}=\alpha^{\cdot\beta}\cdot 1<\alpha^{\cdot\beta}\cdot\alpha\le \alpha^{\cdot\beta}\cdot\alpha^{\cdot\delta}=\alpha^{\cdot(\beta+\delta)}=\alpha^{\cdot\gamma}.$$
\smallskip

6. Fix any ordinals $\alpha>1$ and $\beta\ge\alpha$. Let $X=\{x\in\Ord:\alpha^{\cdot x}\le\beta\}$ and $x=\sup X$. It can be shown that $\alpha^{\cdot x}\le\beta$ and $\alpha^{\cdot(x+1)}>\beta$. Let $Y=\{y\in\Ord:\alpha^{\cdot x}\cdot y\le\beta\}$. It can be shown that for the ordinal $y=\sup Y$ we have $\alpha^{\cdot x}{\cdot}y\le\beta$ but $\alpha^{\cdot x}\cdot (y+1)>\beta$.  The choice of $x$ implies that $0<y<\alpha$. Let $Z=\{z\in\Ord:\alpha^{\cdot x}\cdot y+z\le\beta\}$. It can be shown that the ordinal $z=\sup Z$ has the required property: $\alpha^{\cdot x}{\cdot}y+z=\beta$. It follows from $\alpha^{\cdot x}\cdot(y+1)=\alpha^{\cdot x}\cdot y+\alpha^{\cdot x}>\beta=\alpha^{\cdot x}\cdot y+z$ that $z<\alpha^{\cdot x}$. The proof of the uniqueness of the ordinals $x,y,z$ is left to the reader.
 \smallskip
 
7. The seventh statement is a partial case of the sixth statement for $\alpha=\w$.
 \end{proof}
 
 \begin{exercise} Prove that $\forall n\in\w\;\forall k\in\w\;(n^{\cdot k}\in\w)$.
 \end{exercise}
 
 \begin{exercise} Simplify:
 \begin{itemize}
 \item $(\w+2)\cdot\w$;
 \item $(\w+\w^{\cdot 2})\cdot(\w^{\cdot3}+\w^{\cdot4})$;
 \item $\w+\w^{\cdot2}+\w^{\cdot3}$.
 \end{itemize} 
 \end{exercise}

\begin{exercise} Find $\alpha<\beta$ and $\gamma$ such that
\begin{itemize}
\item $\alpha+\gamma=\beta+\gamma$;
\item $\alpha\cdot\gamma=\beta\cdot\gamma$;
\item $\alpha^{\cdot\gamma}=\beta^{\cdot\gamma}$.
\end{itemize}
\end{exercise}

\begin{exercise} Consider the sequence of ordinals $(\alpha_n)_{n\in\w}$ defined by the recursive formula $\alpha_0=1$ and $\alpha_{n+1}=\w^{\cdot \alpha_n}$. Prove that $\e_0=\sup_{n\in\w}\alpha_n$ the smallest ordinal $\e$ such that $\e=\w^{\cdot \e}$.
\end{exercise}

An  ordinal $\gamma>1$ is called \index{ordinal!indecomposable}\index{indecomposable ordinal}{\em indecomposable} if $\gamma\ne\alpha+\beta$ for any ordinals $\alpha,\beta<\gamma$.

\begin{exercise} Prove that an ordinal $\alpha>1$ is indecomposable if and only if $\alpha=\w^{\cdot \beta}$ for some ordinal $\beta$. 
\end{exercise}

 The exponentiation of ordinals  has a nice geometric model. Let $L$ be a linear order on a set $X=\dom[L^\pm]$ and $R$ be an order on a set $Y=\dom[R^\pm]$. Let $Y^{<X}$ be the set of all functions $f$ such that $f$ is a finite set with $\dom[f]\subseteq X$ and $\rng[X]\subseteq Y$. We endow the set $Y^{<X}$ with the irreflexive order $W$ consisting of all ordered pairs $\langle f,g\rangle$ such that there exists $x\in\dom[g]$ such that $f\cap((L[\{x\}]\setminus\{x\})\times Y)=g\cap((L[\{x\}]\setminus\{x\})\times Y)$ and either $x\notin\dom[f]$ or $x\in\dom[f]$ and $\langle f(x),g(x)\rangle\in R\setminus\Id$.
 
\begin{exercise} Prove that for any ordinals $\alpha,\beta$ the order $\E{\restriction}\alpha^{\cdot\beta}$ is isomorphic to the order $W$ on the set $\alpha^{<\beta}$.
\end{exercise}

\section{Cantor's normal form}

The Cantor's normal form of an ordinal is its power expansion with base $\w$.

\begin{theorem}\label{t:Cantor-norm} Each nonzero ordinal $\alpha$ can be uniquely written as $$\alpha=\w^{\cdot\beta_1}\cdot k_1+\dots+\w^{\cdot\beta_n}\cdot k_n$$for some ordinals $\beta_1>\dots>\beta_n$  and nonzero natural numbers $k_1,\dots,k_n$.
\end{theorem}

\begin{proof} The proof is by induction on $\alpha>0$. For $\alpha=1$, we have the representation $\alpha=\w^{\cdot0}\cdot 1$. Assume that the theorem has been proved for all ordinal smaller than some ordinal $\alpha$. By Theorem~\ref{t:exp-ord}(7), there are unique ordinals $\beta_1,k_1,\alpha_1$ such that $\alpha=\w^{\cdot\beta_1}\cdot k_1+\alpha_1$ and $0<k_1<\w$, $\alpha_1<\w^{\cdot\beta_1}$. Since $\alpha_1<\alpha$ we can apply the inductive assumption and find ordinals  $\beta_2>\dots>\beta_n$  and nonzero natural numbers $k_2,\dots,k_n$ such that $\alpha_1=\w^{\cdot\beta_2}{\cdot}k_2+\dots+\w^{\cdot\beta_n}{\cdot}k_n$. Then $$\alpha=\w^{\cdot\beta_1}\cdot k_1+\dots+\w^{\cdot\beta_n}\cdot k_n$$
the required decomposition of $\alpha$.
\end{proof}

Theorem~\ref{t:Cantor-norm} allows to identify ordinals with functions $f:\w^\Ord\to\w$ defined on the class $\w^\Ord=\{\w^{\cdot\beta}:\beta\in\Ord\}$ such that $\mathrm{supp}(f)=\{\beta\in\w^\Ord:f(\w^{\cdot\beta})>0\}$ is finite. The coordinatewise addition of such functions induces the so-called \index{normal addition of ordinals}{\em normal addition} $\alpha\oplus\beta$ of ordinals. The normal addition of ordinals is both associative and commutative.  

\begin{example} $(\w^{\cdot2}{\cdot}3+\w{\cdot}5+1)\oplus(\w^{\cdot3}{\cdot}4+\w^{\cdot2}\cdot{2}+3)=\w^{\cdot 3}{\cdot}4+\w^{\cdot2}{\cdot}5+\w{\cdot}5+4$.
\end{example}

\newpage

\part{Cardinals}

This part is devoted to cardinals and cardinalities. Since we will mostly speak about sets (and rarely about classes), we will deviate from our convention of denoting sets by small characters and shall use both small and capital letters for denoting sets.

\section{Cardinalities and cardinals}


\begin{definition} Two classes $X,Y$ are defined to be \index{sets!equipotent}\index{equipotent sets}{\em equipotent} if there exists a bijective function $F:X\to Y$. In this case we write $|X|=|Y|$ and say that $X$ and $Y$ have the same cardinality.
\end{definition}

\begin{exercise}\label{ex:equipotent} Prove that for any classes $X,Y,Z$
\begin{itemize}
\item[1)] $|X|=|X|$;
\item[2)] $|X|=|Y|\;\Ra\;|Y|=|X|$;
\item[3)] $(|X|=|Y|\;\wedge\;|Y|=|Z|)\;\Ra\;|X|=|Z|)$.
\end{itemize}
\end{exercise}

Exercise~\ref{ex:equipotent} witnesses that
$$\{\langle x,y\rangle\in\ddot\UU:|x|=|y|\}$$is an equivalence relation on the universe $\UU$. For every set $x$ its cardinality $|x|$ is the unique equivalence class $\{y\in\UU:|y|=|x|\}$ of this relation, containing the set $x$. Let us write down this as a formal definition.

\begin{definition} For a set $x$ its \index{cardinality}{\em cardinality} $|x|$ is the class
$$\{y\in\UU:\exists f\in\Fun\;\wedge\;(f^{-1}\in\Fun\;\wedge\;\dom[f]=x\;\wedge\;\rng[f]=y)\}$$
consisting of all sets $y$ that are equipotent with $x$.
\end{definition}

\begin{example} 
\begin{enumerate}
\item[0)] The cardinality $|0|$ of the emptyset $0=\emptyset$ is the singleton $\{\emptyset\}=1$.
\item[1)] The cardinality $|1|$ of the natural number $1=\{\emptyset\}$ is the proper class of all singletons $\{\{x\}:x\in\UU\}$. 
\item[2)] The cardinality $|2|$ of the natural number $2=\{0,1\}$ is the proper class of all  doubletons $\{\{x,y\}:x\in\UU\;\wedge\;y\in\UU\;\wedge\;\;x\ne y\}$.
\end{enumerate}
\end{example}

The Hartogs' Theorem~\ref{t:Hartogs} implies that for every set $x$ the class $|x|\cap\On$ is a set. This  intersection is not empty if and only if the set $x$ can be  well-ordered.

\begin{definition} An ordinal $\alpha$ is called a \index{cardinal}{\em cardinal} if $\alpha$ is the $\E$-least element in the set $|\alpha|\cap\Ord$. The class of cardinals will be denoted by \index{{{\bf Card}}}\index{class!{{\bf Card}}}$\Crd$.
\end{definition}

If the Axiom of Choice holds, then by Zermelo's Theorem~\ref{t:WO}, each set $x$ can be well-ordered and then Theorem~\ref{t:wOrd} ensures that the cardinality $|x|$ contains an ordinal and hence contains a unique cardinal. This unique cardinal is called the {\em cardinal} of the set $x$. 
For example, the cardinal of the ordinal $\w+1$ is $\w$.

The cardinal of a set $x$ is well-defined if and only if the cardinality $|x|$ contains some ordinal. 
\smallskip

Let us recall that for two classes $X,Y$ we write $|X|=|Y|$ iff there exists a bijective function $F:X\to Y$.

Given two classes $X,Y$ we write  $|X|\le|Y|$ if there exists an injective function $F:X\to Y$. Also we write $|X|<|Y|$ if $|X|\le |Y|$ but $|X|\ne|Y|$.
 
\begin{theorem}[Cantor--Bernstein--Schr\"oder]\label{t:CBDS} For two classes $X,Y$,
$$|X|=|Y|\;\Leftrightarrow\;(|X|\le|Y|\;\wedge\;|Y|\le|X|).$$
\end{theorem}
 
\begin{proof} The implication $|X|=|Y|\;\Ra\;(|X|\le|Y|\;\wedge\;|Y|\le|X|)$ is trivial. To prove the reverse implication, assume that $|X|\le|Y|$, $|Y|\le|X|$ and fix injective functions $F:X\to Y$ and $G:Y\to X$.  For every $n\in\w$ let $(G\circ F)^{\circ n}:X\to X$ and $(F\circ G)^{\circ n}:Y\to Y$ be the $n$-th iterations of the functions $G\circ F$ and $F\circ G$, respectively. These iterations exist by Theorem~\ref{t:iterate}.

For every $n\in\w$ let $X_{2n}=(G\circ F)^{\circ n}[X]$, $Y_{2n}=(F\circ G)^{\circ n}[Y]$, $X_{2n+1}=(G\circ F)^{\circ n}[G[Y]]$ and $Y_{2n+1}=(F\circ G)^{\circ n}[F[X]]$.

By induction we can prove that $X_{n+1}\subseteq X_n$ and $Y_{n+1}\subseteq Y_n$ for every $n\in\w$. Let $X_\w=\bigcap_{n\in\w}X_n$ and $Y_\w=\bigcap_{n\in\w}Y_n$, and observe that
$$F[X_\w]=\textstyle\bigcap_{n\in\w}F[X_n]=\bigcap_{n\in\w}Y_{n+1}=Y_\w.$$
It is easy to check that the function $H:X\to Y$ defined by
$$H(x)=\begin{cases}
F(x)&\mbox{if $x\in \bigcup_{n\in\w}(X_{2n}\setminus X_{2n+1})$};\\
G^{-1}(x)&\mbox{if $x\in\bigcup_{n\in\w}(X_{2n+1}\setminus X_{2n+2})$};\\
F(x)&\mbox{if $x\in X_\w$}
\end{cases}
$$is bijective and witnesses that $|X|=|Y|$.
\end{proof}

Since the cardinalities of sets $|x|$ are proper classes, we cannot speak about the class of cardinalities. Nonetheless, the indexed family of cardinalities $(|x|)_{x\in\UU}$ is absolutely legal; and Theorem~\ref{t:CBDS} implies that $\le$ is a partial order on this indexed family.

Given two classes $X,Y$ we write $|X|\le^*|Y|$ if there exists a surjective function $f:Y\to X$ or $X$ is empty.

The following proposition shows that for any sets $x,y$ we have the implications
$$(|x|\le|y|)\;\Ra\;(|x|\le^*|y|)\;\Ra\;(|\mathcal P(x)|\le|\mathcal P(y)|).$$

\begin{proposition}\label{p:card-ineq} Let $X,Y$ be nonempty classes.
\begin{enumerate}
\item[\textup{1)}] If $|X|\le|Y|$, then $|X|\le^*|Y|$.
\item[\textup{2)}] If $Y$ is well-ordered, then $|X|\le|Y|$ is equivalent to $|X|\le^*|Y|$.
\item[\textup{3)}] If $X,Y$ are sets and $|X|\le^*|Y|$, then $|\mathcal P(X)|\le|\mathcal P(Y)|$.
\end{enumerate}
\end{proposition}

\begin{proof} 1. If $|X|\le|Y|$, then there exists an injective function $F:X\to Y$. If $X=\emptyset$, then $|X|\le^*|Y|$ by definition. If $X\ne\emptyset$, then we can choose an element $b\in X$ and consider the function $G:Y\to X$ defined by
$$G(y)=\begin{cases}F^{-1}(y)&\mbox{if $y\in F[X]$};\\
b&\mbox{if $y\in Y\setminus F[X]$}.
\end{cases}
$$
Using Theorem~\ref{t:class}, one can check that the function $G$ is well-defined and surjective.
\smallskip

2. Assume that $Y$ is well-ordered and fix a well-order $W$ on $Y=\dom[W^\pm]$. If $|X|\le|Y|$, then $|X|\le^*|Y|$ by the preceding statement. Now assume that $|X|\le^*|Y|$. If $X$ is empty, then the empty injective function $\emptyset:\emptyset\to Y$ witnesses that $|X|\le|Y|$. If $X$ is not empty, then there exists a surjective function $F:Y\to X$. For every $x\in X$ let $G(x)$ be the unique $W$-minimal element of the nonempty class $\{y\in Y:F(y)=x\}$. Then $G:X\to Y$, $G:x\mapsto G(x)$, is a well-defined injective function witnessing that $|X|\le|Y|$.
\smallskip

3. If $|X|\le^*|Y|$ and $X,Y$ are sets, then either $X=\emptyset$ or there exists a surjective function $G:Y\to X$. In the first case we  have $|\mathcal P(X)|=|\{\emptyset\}|\le|\mathcal P(Y)|$. In the second case we can consider the injective function $G^{-1}:\mathcal P(X)\to\mathcal P(Y)$, $G^{-1}:a\mapsto G^{-1}[a]$, witnessing that $|\mathcal P(X)|\le|\mathcal P(Y)|$.
\end{proof}

\begin{exercise} Show the equivalence of two statements:
\begin{itemize} 
\item[(i)] For any sets $x,y$ and a surjective function $f:x\to y$ there exists an injective function $g:y\to x$ such that $f\circ g=\Id{\restriction}y$.
\item[(ii)] The Axiom of Choice holds.
\end{itemize}
Prove that the equivalent conditions (i),(ii) imply the condition
\begin{itemize}
\item[(PP)] For any sets $x,y$\;$(|x|\le|y|\;\Leftrightarrow\;|x|\le^*|y|)$.
\end{itemize}
\end{exercise}

\begin{remark} The statement (PP) is known in Set Theory as the Partition Principle.\\ It is an open problem whether (PP) implies (AC), see\\ {\tt\small https://karagila.org/2014/on-the-partition-principle/}
\end{remark}


\begin{theorem}[Cantor]\label{t:CantorP} For any set $x$ there is no surjective function $f:x\to\mathcal P(x)$.
\end{theorem}

\begin{proof} Given any function $f:x\to\mathcal P(x)$, consider the set $a=\{z\in x:z\notin f(z)\}\in\mathcal P(x)$. Assuming that $a\in\rng[f]$, we could find an element $z\in x$ such that $a=f(z)$. If $z\notin f(z)$, then $z\in a=f(z)$ which is a contradiction. If $z\in f(z)$, then $z\in a$ and hence $z\notin f(z)$. In both cases we obtain a contradiction, which shows that $a\notin\rng[f]$ and hence $f$ is not surjective.
\end{proof}

\begin{corollary}[Cantor]\label{c:CantorP} For any set $x$ we have $|x|<|\mathcal P(x)|$.
\end{corollary}

\begin{proof} The inequality $|x|\le|\mathcal P(x)|$ follows from the injectivity of the function $x\to\mathcal P(x)$, $z\mapsto \{z\}$. Assuming that $|x|=|\mathcal P(x)|$, we would get a bijective (and hence surjective) function $x\to\mathcal P(x)$, which is forbidden by Cantor's Theorem~\ref{t:CantorP}.
\end{proof}

\section{Finite and Countable sets}

In this section we shall establish some elementary facts about finite and countable sets. Let us recall that a set $x$ is {\em countable} (resp. {\em finite}) if $|x|=|\alpha|$ for some ordinal $\alpha\in\w\cup\{\w\}$ (resp. $\alpha\in\w$). Elements of the set $\w$ are called \index{natural number}{\em natural numbers}.

\begin{lemma}\label{l:injn} Let $n$ be a natural number. Every injective function $f:n\to n$ is bijective.
\end{lemma}

\begin{proof} This lemma will be proved by induction on $n$. For $n=0$ the unique function $\emptyset:\emptyset\to\emptyset$ is bijective. Assume that for some $n\in\w$ we have proved that any injective function $f:n\to n$ is bijective.  

Take any injective function $f:n+1\to n+1$. If $f(n)=n$, then the injectivity of $f$ guarantees that $n\notin f[n]$ and hence $f[n]\subseteq (n+1)\setminus\{n\}=n$. By the inductive assumption, the injective function $f{\restriction}_n:n\to n$ is bijective and hence $$f[n+1]=f[n\cup\{n\}]=f[n]\cup\{f(n)\}=n\cup\{n\}=n+1,$$ which means that $f$ is surjective and hence bijective.  If $f(n)\ne n$, then consider the bijective function $g:n+1\to n+1$ defined by
$$g(x)=\begin{cases} f(n)&\mbox{if $x=n$};\\
n&\mbox{if $x=f(n)$};\\
x&\mbox{otherwise}.
\end{cases}$$
Then $g\circ f:n+1\to n+1$ is an injective function with $g\circ f(n)=n$. As we already proved, such function is bijective. Now the bijectivity of $g$ implies that $f=g^{-1}\circ g\circ f$ is bijective, too.
\end{proof}

\begin{lemma}\label{l:surjn} Let $f:x\to n$ be a surjective function from a set $x$ onto a natural number $n\in\w$. Then there exists a function $g:n\to x$ such that $f\circ g=\Id{\restriction}n$.
\end{lemma}

\begin{proof} This lemma will be proved by induction on $n$. For $n=\emptyset$ the statement of the lemma is trivially true. Assume that the lemma has been proved for some natural number $n\in\w$. Take any surjective function $f:x\to n+1$. Consider the subset $x'=f^{-1}[n]\subseteq x$ and observe that the function $f{\restriction}_{x'}:x'\to n$ is surjective. By the inductive assumption, there exists a function $g:n\to x'$ such that $f{\restriction}{x'}\circ g=\Id{\restriction}n$. By the surjectivity of $f$, there exists an element $z\in x$ such that $f(z)=n$. Define the function $\bar g:n+1\to x$ by the formula $$\bar g(y)=\begin{cases} g(y)&\mbox{if $y\in n$};\\
z&\mbox{if $y=n$}.
\end{cases}
$$It is clear that $f\circ\bar g=\Id{\restriction}(n+1)$.
\end{proof}

\begin{theorem}\label{t:inj=sur} For any natural number $n$ and function $f:n\to n$ the following conditions are equivalent:
\begin{enumerate}
\item[\textup{1)}] $f$ is injective;
\item[\textup{2)}] $f$ is surjective;
\item[\textup{3)}] $f$ is bijective.
\end{enumerate}
\end{theorem}

\begin{proof} The implication $(1)\Ra(3)$ was proved in Lemma~\ref{l:injn} and $(3)\Ra(2)$ is trivial. To prove $(2)\Ra(1)$, assume that the function $f:n\to n$ is surjective. By Lemma~\ref{l:surjn}, there exists a function $g:n\to n$ such that $f\circ g=\Id{\restriction}n$. The function $g$ is injective since for any distinct elements $x,y\in n$ we have $f(g(x))=x\ne y=f(g(y))$ and hence $g(x)\ne g(y)$. By Lemma~\ref{l:injn}, the injective function $g:n\to n$ is bijective. Then for any distinct elements $x,y\in n$ the injectivity of $g$ implies that $g^{-1}(x)\ne g^{-1}(y)$ and finally
$$f(x)=f\circ g(g^{-1}(x))=\Id(g^{-1}(x))=g^{-1}(x)\ne g^{-1}(y)=f\circ g(g^{-1}(y))=f(y),$$ which means that $f$ is injective.
\end{proof}

\begin{corollary}\label{c:fin=>Dfin} For any finite set $x$ and function $f:x\to x$ the following conditions are equivalent:
\begin{enumerate}
\item[\textup{1)}] $f$ is injective;
\item[\textup{2)}] $f$ is surjective;
\item[\textup{3)}] $f$ is bijective.
\end{enumerate}
\end{corollary}

\begin{theorem} Every natural number is a cardinal.
\end{theorem}

\begin{proof} We need to prove that every natural number $n$ is the smallest ordinal in the set $|n|\cap\Ord$. Assuming the opposite, we can find an ordinal $k<n$ and a bijective function $f:n\to k$. By Theorem~\ref{t:ord}(4), $k\subset n$ and hence the function $f:n\to k\subset n$ is injective but not surjective. But this contradicts Theorem~\ref{t:inj=sur}.
\end{proof}

\begin{theorem} The ordinal $\w$ is a cardinal.
\end{theorem}

\begin{proof} Assuming that $\w$ is not a cardinal, we could find a bijective function $f:\w\to n$ to some natural number $n\in\w$. Then  for the natural number $n+1$ the function $f{\restriction}_{n+1}:n+1\to n\subset n+1$ is injective but not surjective, which contradicts Theorem~\ref{t:inj=sur}.
\end{proof}

\begin{theorem}\label{t:comp-n} For a set $x$ and a natural number $n$ we have $|n|\le|x|$ or $|x|\le |n|$.
\end{theorem}

\begin{proof} This theorem will be proved by induction on $n$. For $n=0$ the inequality $|\emptyset|\le|x|$ holds for every set $x$ as the empty function $\emptyset:\emptyset\to x$ is injective.

Assume that for some natural number $n$ and all sets $x$ we have proved that $|x|\le|n|$ or $|n|\le|x|$. Now consider the natural number $n+1$ and take any set $x$. If $|x|\le n+1$, then we are done. So, assume that $|x|\not\le n+1$. By the inductive assumption, $|x|\le|n|\;\vee\;|n|\le|x|$. The first case is impossible since it leads to the contradiction: $|x|\le|n|\le|n+1|$. Therefore, $|n|\le|x|$ and hence there exists an injective function $f:n\to x$. If $f$ is surjective, then $f$ is bijective and hence $|x|=|n|\le|n+1|$, which contradicts our assumption. Therefore, $f$ is not surjective and we can choose an element $z\in x\setminus f[n]$ and extend the function $f$ to the injective function $\bar f=f\cup\{\langle n,z\rangle\}$ from $n+1$ to $x$, witnessing that $|n+1|\le|x|$. 
\end{proof}

\begin{proposition}\label{p:lew} If for some set $x$ we have $|x|<|\w|$, then $x$ is finite.
\end{proposition}

\begin{proof} The inequality $|x|<|\w|$ implies that $x$ admits an injective function $x\to\w$ and hence is well-orderable. By Theorem~\ref{t:wOrd}, there exists a bijective function $f:x\to\alpha$ to some ordinal $\alpha$. We can assume that this ordinal is the smallest possible.  We claim that $\alpha<\w$. Assuming that $\alpha\not<\w$ and applying Theorem~\ref{t:ord}(6), we conclude that $\w\le\alpha$ and hence $|\w|\le|\alpha|=|x|$. Since $|x|\le|\w|$, we can apply Theorem~\ref{t:CBDS} and conclude that $|x|=|\w|$, which contradicts our assumption. This contradiction shows that $\alpha<\w$ and then $|x|=|\alpha|\in\w$, which means that the set $x$ is finite.
\end{proof}  

\begin{corollary}\label{c:count=lew} A set $x$ is countable if and only if $|x|\le|\w|$.
\end{corollary}

\begin{proposition}\label{p:Df} For a set $x$ the following conditions are equivalent:
\begin{enumerate}
\item[\textup{1)}] $|\w|\le|x|$;
\item[\textup{2)}] there exists an injective function $f:x\to x$ which is not surjective.
\end{enumerate}
\end{proposition}

\begin{proof} $(1)\Ra(2)$ If $|\w|\le|x|$, then there exists an injective function $g:\w\to x$. 
Define a function $f:x\to x$ by the formula
$$f(z)=\begin{cases}
g(g^{-1}(z)+1)&\mbox{if $z\in g[\w]$};\\
z&\mbox{if $z\in x\setminus g[\w]$};
\end{cases}
$$and observe that $f$ is injective but $g(0)\notin f[x]$, so $f$ is not surjective.
\smallskip

$(2)\Ra(1)$ Assume that there exists an injective function $f:x\to x$ which is not surjective. Choose any element $x_0\in x\setminus f[x]$ and consider the sequence $(x_n)_{n\in\w}$  defined by the recursive formula $x_{n+1}=f(x_n)$ for $n\in\w$. We claim that the function $g:\w\to x$, $g:n\mapsto x_n$ is injective. This will follow from Theorem~\ref{t:ord}(6) as soon as we check that $g(k)\ne g(n)$ for any natural numbers $k<n$. This will be proved by induction on $k\in\w$. For $k=0$ this follows from the choice of $x_0\notin f[x]$. Assume that for some $k\in\w$ we have proved that $g(k)\ne g(n)$ for all $n>k$. Take any natural number $n>k+1$. By Theorem~\ref{t:add-ord}(5), there exists an ordinal $\alpha>0$ such that $(k+1)+\alpha=n$. Theorem~\ref{t:add-ord}(3) implies that $\alpha=0+\alpha\le (k+1)+\alpha=n$ and hence $\alpha\in\w$. By Theorem~\ref{t:add-nat}, $n=(k+1)+\alpha=(k+\alpha)+1$. By Theorem~\ref{t:add-ord}(4), $k=k+0<k+\alpha$ and then $g(k)\ne g(k+\alpha)$ by the inductive assumption. The injectivity of $f$ guarantees that $g(k+1)=x_{k+1}=f(x_k)=f(g(k))\ne f(g(k+\alpha))=g(k+\alpha+1)=g(n)$. This completes the proof of the inductive step. Now the Principle of Mathematical Induction implies that $g(k)\ne g(n)$ for all natural numbers $k<n$. Finally, Theorem~\ref{t:ord}(6) implies that the function $g:\w\to x$ is injective and hence $|\w|\le|x|$.
\end{proof}

\begin{definition} A set $x$ is called \index{Dedekind finite set}\index{set!Dedekind finite}{\em Dedekind-finite} if every injective function $f:x\to x$ is surjective.
\end{definition}

By Proposition~\ref{p:Df}, a set $x$ is Dedekind-finite if and only if $|\w|\not\le|x|$. By Theorem~\ref{t:inj=sur}, every finite set is Dedekind-finite.

\begin{proposition}\label{p:DC=>Df} Assume that the Axiom of Dependent Choice $(\mathsf{DC})$ holds. A set $x$ is Dedekind-finite if and only if it is finite.
\end{proposition}

\begin{proof} The ``if'' part follows from Corollary~\ref{c:fin=>Dfin}. To prove the ``only if'' part, assume that a set $x$ is not finite. Consider the ordinary tree $T\subseteq x^{<\w}$ whose elements are injective functions $f$ with $\dom[f]\in\w$ and $\rng[f]\subseteq x$. By  Proposition~\ref{p:TC=>DC}, the $(\mathsf{DC})$ is equivalent to $(\mathsf{TC}_\w)$ and hence the tree $T$ contains a maximal chain $C\subseteq T$. It follows that $f=\bigcup C$ is an injective function such that $\dom[f]\in\w\cup\{\w\}$ and $\rng[f]\subseteq x$. We claim that $\dom[f]=\w$. To derive a contradiction, assume that $\dom[f]=n\in\w$. Since $x$ is not finite, the injective function $f:n\to x$ is not surjective. Consequently, we can find an element $z\in x\setminus f[n]$ and consider the function $\bar f=f\cup\{\langle n,z\rangle\}\in T$ and the chain $\bar C=C\cup\{\bar f\}$, which is strictly larger than $C$. But this contradicts the maximality of $C$. This contradiction shows that $\dom[f]=\w$ and hence $f:\w\to x$ is an injective function witnessing that $|\w|\le|x|$ and hence $x$ is not Dedekind-finite.
\end{proof}

Theorem~\ref{t:comp-n} and Propositions~\ref{p:lew}, \ref{p:Df} imply

\begin{corollary} If a set $x$ is Dedekind-finite but not finite, then 
\begin{enumerate}
\item[\textup{1)}] $|n|\le|x|$ for all $n\in\w$;
\item[\textup{2)}] $|\w|\not\le|x|$;
\item[\textup{3)}] $|x|\not\le|\w|$.
\end{enumerate}
\end{corollary}


\begin{remark} The existence of infinite Dedekind-finite sets does not contradicts the Axioms of $\CST$, see \cite[\S4.6]{JechAC} for the proof of this fact. On the other hand, such sets do not exist under the axioms $(\CST+\mathsf{DC})$, see \ref{p:DC=>Df}.
\end{remark}

Next, we show that the countability is preserved by some operations over sets.

\begin{proposition} For any function $F$ and a countable set $x$ the image $F[x]$ is a countable set.
\end{proposition}

\begin{proof} Let $y=F[x]$. Since $x$ is countable, there exists an injective function $g$ such that $\dom[g]\in\w\cup\{\w\}$ and $\rng[g]=x$. Then $f=F\circ g$ is a  function such that $y=\rng[f]$ and $\dom[f]\subseteq\w$. Consider the function $h:y\to\w$ assigning to each element $v\in y$ the unique $\E$-minimal element of the set $f^{-1}[\{v\}]\subseteq\w$. It is easy to see that the function $h:y\to\w$ is injective and hence $|y|\le|\w|$.   By Proposition~\ref{p:lew}, the set $y$ is countable.
\end{proof}

\begin{corollary} If $x$ is a countable set, then each subset of $x$ is countable.
\end{corollary}

\begin{proof} Observe that each subset $y\subseteq x$ has $|y|\le|x|\le|\w|$. Applying Proposition~\ref{p:lew}, we obtain that $|y|=|n|$ for some $n\in\w\cup\{\w\}$.
\end{proof}

\begin{proposition} The set $\w^{<\w}=\bigcup_{n\in\w}\w^n$ is countable.
\end{proposition}

\begin{proof} Observe that the set $\w^{<\w}$ consists of all functions $f$ such that $\dom[f]\in\w$ and $\rng[f]\subseteq\w$. It follows that the set $\rng[f]$ is finite and hence has an $\E$-maximal element $\max\rng[f]\in\w$, see Exercise~\ref{ex:fin-max}. Then the function $\mu:\w^{<\w}\to\w$, $\mu:f\mapsto \max\{\dom[f],\max\rng[f]\}$, is well-defined.

On the set $\w^{<\w}$ consider the well-order
\begin{multline*}
W=\{\langle f,g\rangle\in\w^{<\w}\times\w^{<\w}:\mu(f)<\mu(g)\;\vee\;(\mu(f)=\mu(g)\wedge\;\dom[f]<\dom[g])\;\vee\;\\
(\mu(f)=\mu(g)\;\wedge\;\dom[f]=\dom[g]\;\wedge\;\exists n\in\dom[f]\;(f(n)<g(n)\;\wedge\;\forall i\in n\;f(i)=g(i))\}.
\end{multline*}
By Theorem~\ref{t:wOrd}, the function $\rank_W:\w^{<\w}\to \rank(W)$ is bijective.
Observe that for any $f\in\w^{<\w}$ the initial interval $\cev W(f)$ is contained in the finite set $k^k$ for some natural number $k\in\w$. This implies that $\rank(W)\le\w$. Then $|\w^{<\w}|=|\rank(W)|\le|\w|$ and the set $\w^{<\w}$ is countable.
\end{proof}

\begin{corollary}\label{c:exp-w} For any countable set $x$ the set $x^{<\w}$ is countable.
\end{corollary}

\begin{exercise} Prove that for any countable ordinals $\alpha,\beta$ the ordinals $\alpha+\beta$, $\alpha\cdot \beta$ and $\alpha^{\cdot\beta}$ are countable.
\end{exercise}

\begin{proposition}\label{p:f-union} For any natural number $n$ and an indexed family of countable sets $(x_i)_{i\in n}$ the union $\bigcup_{i\in n}x_i$ is countable.
\end{proposition}

\begin{proof} The proof is inductive. For $n=0$ the union $\bigcup_{i\in k}x_i$ is empty and hence countable. Assume that for some $n\in\w$ we proved that the union $\bigcup_{i\in n}x_i$ of any family $(x_i)_{i\in n}$ of countable sets is countable. 

Take any family of countable sets $(x_i)_{i\in n+1}$ and consider its union $x=\bigcup_{i\in n+1}x_i$. Using Theorems~\ref{t:mult-ord-prop}(6) it can be shown that the functions $e:\w\to\w$, $e:n\mapsto 2n$, and $o:\w\to\w$, $o:n\mapsto 2n+1$, are injective and have disjoint images.

By the inductive assumption, the set $\bigcup_{i\in n}x_i$ is countable and hence admits an injective function $f:\bigcup_{i\in n}x_i\to \w$. Since the set $x_n$ is countable there exists an injective function $g:x_n\to\w$. Consider the injective function $h:x\to\w$ defined by the formula
$$h(u)=\begin{cases}2\cdot f(z)&\mbox{if $u\in\bigcup_{i\in n}x_i$};\\
2\cdot g(z)+1&\mbox{if  $u\in x_n\setminus\bigcup_{i\in n}x_i$}.
\end{cases}
$$
The injective function $h$ witnesses that $|x|\le|\w|$. By Corollary~\ref{c:count=lew}, the set $x$ is countable.
\end{proof}

\begin{remark} The axioms of $\CST$ do not imply that the union $\bigcup_{n\in\w}x_n$ of an indexed family $(x_n)_{n\in\w}$ of countable sets is countable. This holds only under the axiom $(\mathsf{UT}_\w)$, which is  weak version of the Axiom of Choice, see Chapter~\ref{s:choice}.
\end{remark}

\begin{proposition} For two countable sets $x,y$ their Cartesian product $x\times y$ is countable.
\end{proposition}

\begin{proof} By Proposition~\ref{p:f-union} and Corollary~\ref{c:exp-w}, the sets $x\cup y$ and $(x\cup y)^{<\w}$ are countable. Since $|x\times y|\le|(x\cup y)^2|\le|(x\cup y)^{<\w}|\le|\w|$, the set $x\times y$ is countable by Corollary~\ref{c:count=lew}.
\end{proof}

\begin{definition} A set $x$ is called \index{set!hereditarily countable}\index{hereditarily countable set}\index{set!hereditarily finite}\index{hereditarily finite set}{\em hereditarily countable} (resp. {\em hereditarily finite}) if $x$ is countable (resp. finite) and each set $y\in\mathsf{TC}(x)$ is countable (resp. finite).
\end{definition}

\begin{exercise} Prove that a set $x\in \mathbf V$ is hereditarily finite if and only if its transitive closure $\mathsf{TC}(x)$ is finite.
\end{exercise}

\begin{exercise} Assuming the principle ($\mathsf{UT_\w})$, prove that a set $x$ is hereditarily countable if and only if its transitive closure $\mathsf{TC}(x)$ is countable.
\end{exercise}

\section{Successor cardinals and alephs}

In the following theorem  for a set $x$ by $W\!O(x)$ we denote the set of all well-orders $w\subseteq x\times x$.

\begin{theorem}[Hartogs--Sierpi\'nski]\label{t:Hartogs2}\index{Hartogs--Sierpi\'nski Theorem}\index{Theorem!Hartogs--Sierpi\'nski} There exists a function $(\cdot)^+:\UU\to\Crd$ assigning to every set $x$ its \index{successor cardinal}\index{cardinal!successor}{\em successor cardinal} $x^+=\{\alpha\in\Ord:|\alpha|\le|x|\}$. For this cardinal we have the upper bounds $$|x^+|\le^*|W\!O(x)|\le|\mathcal P(x\times x)|,\;\;
|x^+|\le \min\{|\mathcal P^{\circ 2}(x\times x)|,|\mathcal P^{\circ3}(x)|\}\mbox{  and\/ }|x^+\cap\;\Crd|\le |\mathcal P^{\circ 2}(x)|.$$
\end{theorem}

\begin{proof} By Theorem~\ref{t:Hartogs} for any set $x$ there exists an ordinal $\alpha$ admitting no injective map $\alpha\to x$. So, we can define $x^+$ as the smallest ordinal with this property. The minimality of the ordinal $x^+$ ensures that $x^+$ is a cardinal and for any $\alpha\in x^+$ there exists an injective function $\alpha\to x$ and hence $|\alpha|\le|x|$. The function $(\cdot)^+:\UU\to\Crd,\quad (\cdot)^+:x\mapsto x^+$, exists  by Theorem~\ref{t:class} on the existence of classes. 

Next we prove the upper bounds for the successor cardinal $x^+$ of a set $x$. Let $W\!O(x)$ be the set of all well-orders $w\subseteq x\times x$. It is easy to see that $WO(x)\subseteq\mathcal P(x\times x)$ and $x^+=\rank[WO(x)]$, which implies
$$|x^+|\le^*|WO(x)|\le|\mathcal P(x\times x)|\quad\mbox{and hence}\quad|x^+|\le|\mathcal P(\mathcal P(x\times x))|=|\mathcal P^{\circ2}(x\times x)|$$
by  Proposition~\ref{p:card-ineq}(3).

On the other hand, any injective function $f$ with $\dom[f]\in x^+$ and $\rng[f]\subseteq x$ is uniquely determined by the chain of subsets $C_f=\{f[\beta]:\beta\le\dom[f]\}\subseteq \mathcal P(x)$ and each ordinal $\alpha\in x^+$ is uniquely determined by the subset $$F_\alpha=\{C_f:f\in\Fun\;\wedge\;f^{-1}\in\Fun\;\wedge\;\dom[f]=\alpha\;\wedge\;\rng[f]\subseteq x\}\subseteq\mathcal P(\mathcal P(x)),$$ which implies that
$$|x^+|\le|\{F_\alpha:\alpha\in x^+\}|\le |\mathcal P(\mathcal P(\mathcal P(x)))|=|\mathcal P^{\circ 3}(x)|.$$

To see that $|x^+\cap\;\Crd|\le|\mathcal P^{\circ2}(x)|$, observe that the function $$x^+\cap\;\Crd\to\mathcal P^{\circ 2}(x),\quad \kappa\mapsto [x]^\kappa:=\{y\in\mathcal P(x):|y|=|\kappa|\},$$ is injective.
\end{proof}

Theorem~\ref{t:inj=sur} implies 

\begin{corollary} For every natural number $n$ its successor cardinal $n^+$ is equal to $n+1$.
\end{corollary}

The Hartogs--Sierpi\'nski Theorem~\ref{t:Hartogs2} has an interesting implication. Given a cardinality $\kappa$, we say that a cardinality $\kappa^{+}$ is a \index{successor cardinality}{\em successor cardinality} of $\kappa$ if $\kappa<\kappa^+$ and $\kappa^+\le\lambda$ for any cardinality $\lambda$ with $\kappa<\lambda$. By Theorem~\ref{t:CBDS} a successor cardinality if it exists, then it is unique.

\begin{theorem}[Tarski]\label{t:card-compare} The following statements are equivalent:
\begin{enumerate}
\item[\textup{1)}] The Axiom of Choice holds.
\item[\textup{2)}] For any sets $x,y$ we have $|x|\le|y|$ or $|y|\le|x|$.
\item[\textup{3)}] For any set $x$ we have $|x|<|x^+|$.
\item[\textup{4)}] For any set $x$ the cardinality $|x^+|$ of the successor cardinal $x^+$ is the  successor cardinality of $|x|$.
\end{enumerate}
\end{theorem}

\begin{proof} $(1)\Ra(2)$ If the Axiom of Choice holds, then by Zermelo Theorem~\ref{t:WO}, every set is equipotent to some ordinal. Now the comparability of cardinalities follows from the comparability of ordinals, see Theorem~\ref{t:ord}(6).
\smallskip

$(2)\Ra(3)$ Assume that any two cardinalities are comparable. For any set $x$, the definition of the successor cardinal $x^+$ implies that $|x^+|\not\le|x|$ and hence $|x|<|x^+|$.
\smallskip

$(3)\Ra(1)$ If for every $x$ we have $|x|<|x^+|$, then there exists an injective function $f:x\to x^+$ and hence $x$ is well-orderable. By Theorem~\ref{t:mainAC}, the Axiom of Choice holds.
\smallskip

$(2)\Ra(4)$ Take any set $x$ and consider its successor cardinal $x^+$. The definition of $x^+$ guarantees that $|x^+|\not\le |x|$. Now the condition (2) implies that $|x|<|x^+|$. 
Assuming that $|x^+|$ is not a successor cardinality of $|x|$, we can find a set $y$ such that $|x|<|y|$ but $|x^+|\not\le |y|$. The comparability of the cardinalities $|x^+|$ and $|y|$ implies that $|y|<|x^+|$. Then the set $y$ admits an injective function to $x^+$ and hence is well-orderable. By Theorem~\ref{t:wOrd}, $y$ admits a bijective function on some ordinal $\alpha$. We claim that $\alpha<x^+$. In the opposite case, we can apply Theorem~\ref{t:ord}(6) and conclude that $x^+\le\alpha$ and then $|x^+|\le|\alpha|=|y|$. Since $|y|<|x^+|$ we can apply Theorem~\ref{t:CBDS} and conclude that $|y|=|x^+|$ which contradicts our assumption. This contradiction shows that $\alpha<x^+$.  Since $|x|<|y|=|\alpha|$, the cardinal $\alpha$ admits no injective function into $x$ and hence $x^+\le\alpha$ as $x^+$ is the smallest ordinal with this property. But this contradicts the strict inequality $\alpha<x^+$ established earlier. This contradiction shows that $|x^+|$ is the successor cardinality of $|x|$.
\smallskip

$(4)\Ra(1)$ If for every set $x$, the cardinality $|x^+|$ is a successor cardinality of $|x|$, then $|x|<|x^+|$ and hence $x$ admits an injective function $f:x\to x^+$ to the cardinal $x^+$, which implies that $x$ can be well-ordered. By Theorem~\ref{t:mainAC}, the Axiom of Choice holds.
\end{proof}

Transfinite iterations of the operation of taking the successor cardinals yield a nice paramet\-rization of the class $\Crd$ of cardinals by ordinals.

Consider the transfinite sequence of cardinals \index{$\w_\alpha$}$(\w_\alpha)_{\alpha\in\Ord}$ defined by the recursive formula:
\begin{itemize}
\item $\w_0=\w$;
\item $\w_{\alpha+1}=\w_\alpha^+$ for any ordinal $\alpha$;
\item $\w_\alpha=\sup\{\w_\beta:\beta\in\alpha\}$ for any limit ordinal $\alpha>0$.
\end{itemize}
Therefore,  $\w_\alpha=\w^{+\circ\alpha}$ for every ordinal $\alpha$.

\begin{proposition}\label{p:walpha} The function $\w_*:\Ord\to\Crd$, $\w_*:\alpha\to\w_\alpha$, is well-defined. For every ordinal $\alpha$ we have $\alpha\le\w_\alpha<\w_{\alpha+1}$.
\end{proposition}

\begin{proof} The existence of the function $\w_*$ follows from Theorem~\ref{t:dynamics} applied to the Hartogs' function $\Ord\to\Ord$, $\alpha\mapsto\alpha^+$, of taking the successor cardinal. 
\smallskip

The inequality $\alpha\le\w_\alpha$ will be proved by transfinite induction on $\alpha$.
 For $\alpha=0$ we have $0<\w$. Assume that for some ordinal $\alpha$ and all its elements $\beta\in\alpha$ we proved that $\beta\le\w_\beta$. If $\alpha$ is a successor ordinal, then $\alpha=\beta+1$ for some $\beta\in\alpha$ and hence $\w_{\alpha}=\w_{\beta+1}=\w_{\beta}^+>\w_\beta\ge\beta$, 
which implies $\alpha=\beta+1\le \w_{\alpha}$. 

If $\alpha$ is a limit ordinal, then $$\w_\alpha=\sup\{\w_\beta:\beta\in\alpha\}\ge\sup\{\beta:\beta\in\alpha\}=\alpha$$
by the inductive assumption. By the Principle of Transfinite Induction, the inequality $\alpha\le\w_\alpha$ is true for all ordinals $\alpha$.
\smallskip

The strict inequality $\w_\alpha<\w_{\alpha+1}=\w_\alpha^+$ follows from the definition of the successor cardinal $\w_\alpha^+>\w_\alpha$. 
\end{proof}

\begin{theorem}\label{t:alephs} For every infinite cardinal $\kappa$ there exists an ordinal $\alpha$ such that $\kappa=\w_\alpha$.
\end{theorem}

\begin{proof} Given a cardinal $\kappa$, consider the class $A=\{\alpha\in\Ord:\w_\alpha\le\kappa\}$. 

Since the function $\w_*:\Ord\to\Crd$, $\w_*:\alpha\mapsto \w_\alpha$, is injective, the class $A$ is a set by the Axiom of Replacement. So, we can consider the ordinal $\alpha=\sup A$. If $\alpha$ is a limit ordinal, then $\w_\alpha=\sup_{\beta\in\alpha}\w_\beta=\sup_{\beta\in A}\w_\beta\le\sup_{\beta\in A}\kappa=\kappa$. If $\alpha=\sup A$ is a successor ordinal, then $\alpha\in A$ and again $\w_\alpha\le\kappa$. In both cases we obtain $\w_\alpha\le\kappa$. Assuming that $\w_\alpha\ne\kappa$, we conclude that $\w_\alpha<\kappa$. Since $\kappa$ is a cardinal, there exists no bijective function $\kappa\to\w_\alpha$. By Theorem~\ref{t:CBDS}, there is no injective functions from $\kappa\to\w_\alpha$. Then $\w_{\alpha+1}=\w_\alpha^+\le\kappa$ by the definition of the successor cardinal $\w_\alpha^+$. Then $\alpha+1\in A$ and hence $\alpha\in \alpha+1\le\sup A=\alpha$, which contradicts the irreflexivity of the relation $\E{\restriction}\Ord$.
 This contradiction shows that $\kappa=\w_\alpha$.
 \end{proof}  

For every  ordinal $\alpha$, denote by \index{$\aleph_\alpha$}$\aleph_\alpha$ the cardinality $|\w_\alpha|$ of the cardinal $\w_\alpha$. 

Theorems~\ref{t:WO}, \ref{t:wOrd}  and \ref{t:alephs} imply the following characterization. 

\begin{corollary}\label{c:alephAC} The following statements are equivalent:
\begin{enumerate}
\item[\textup{1)}] For every infinite set $x$ there exists an ordinal $\alpha$ such that $|x|=\aleph_\alpha$.
\item[\textup{2)}] The Axiom of Choice holds.
\end{enumerate}
\end{corollary}

\section{Arithmetics of Cardinals}

In this section, we study the operations of sum, product and exponent of cardinalities.

Namely, for any cardinalities $|x|$, $|y|$ we put 
\begin{itemize}
\item $|x|+|y|=|(\{0\}\times x)\cup(\{1\}\times y)|$;
\item $|x|\cdot|y|=|x\times y|$;
\item $|x|^{|y|}=|x^y|$.
\end{itemize}

\begin{exercise} Show that the sum, product and exponent of cardinalities are well-defined, i.e., do not depend on the choice of sets in the corresponding equivalence classes.
\end{exercise}

\begin{exercise} Prove that for any ordinals $\alpha,\beta$ we have $|\alpha|+|\beta|=|\alpha+\beta|$ and $|\alpha|\cdot|\beta|=|\alpha\cdot\beta|$. If the ordinal $\beta$ is finite, then $|\alpha|^{|\beta|}=|\alpha^{\cdot\beta}|$.
\end{exercise}

\begin{exercise} Find two ordinals $\alpha,\beta$ such that $|\alpha|^{|\beta|}\ne |\alpha^{\cdot\beta}|$.
\smallskip

\noindent{\em Hint}: Observe that $|2^{\cdot\w}|=|\w|<|2^\w|$.
\end{exercise} 

\begin{exercise}\label{ex:sum-prod-card} Given cardinalities $\kappa,\lambda,\mu$, prove that
\begin{enumerate}
\item $\kappa+\lambda=\lambda+\kappa$;
\item $(\kappa+\lambda)+\mu=\kappa+(\lambda+\mu)$;
\item $\kappa\cdot\lambda=\lambda\cdot\kappa$;
\item $(\kappa\cdot\lambda)\cdot\mu=\kappa\cdot(\lambda\cdot\mu)$;
\item $\kappa\cdot(\lambda+\mu)=(\kappa\cdot\lambda)+(\kappa\cdot\mu)$;
\item If $\kappa\le\lambda$, then $\kappa+\mu\le\lambda+\mu$ and $\kappa\cdot\mu\le\lambda\cdot\mu$;
\item $\kappa+|\emptyset|=\kappa=|\kappa|\cdot|1|$;
\item  if $|2|\le\kappa$ and $|2|\le\lambda$, then $\kappa+\lambda\le\kappa\cdot \lambda$.
\end{enumerate}
\smallskip

\noindent{\em Hint to} (8): Fix two sets $x,y$ with $|x|=\kappa\ge|2|$ and $|y|=\lambda\ge |2|$. Fix points $a,b\in x$ and  $c,d\in y$ with $a\ne b$ and $c\ne d$. Consider the injective function $f:(\{0\}\times x)\cup(\{1\}\times y)\to x\times y$ assigning to each point $\langle 0,z\rangle\in \{0\}\times x)$ the ordered pair $\langle z,c\rangle$, to each point $\langle 1,z\rangle\in\{1\}\times (y\setminus\{c\})$ the ordered pair $\langle a,z\rangle$, and to the ordered pair $\langle 1,c\rangle$ the ordered pair $\langle b,d\rangle$. 
\end{exercise}


\begin{theorem}\label{t:aleph-times} For any ordinal $\alpha$ we have $\aleph_\alpha\cdot\aleph_\alpha=\aleph_\alpha$.
\end{theorem}

\begin{proof} This theorem will be proved by transfinite induction on $\alpha$. Assume that for some ordinal $\alpha$ and all its elements $\beta\in\alpha$ we have proved that $\aleph_\beta\cdot\aleph_\beta=\aleph_\beta$. 

Consider the cardinal $\w_\alpha$ and its square $\w_\alpha\times\w_\alpha$ endowed with the canonical well-order
\begin{multline*}
W=\{\langle\langle x,y\rangle,\langle x',y'\rangle\rangle\in(\w_\alpha\times\w_\alpha)\times(\w_\alpha\times\w_\alpha):\\
(x\cup y\subset x'\cup y')\;\vee\;(x\cup y=x'\cup y'\;\wedge\;x\in x')\;\vee\;(x\cup y=x'\cup y'\;\wedge\; x=x'\;\wedge\; y\in y')\}.
\end{multline*} 
By Theorem~\ref{t:wOrd}, there exists an order isomorphism $\rank_W:\dom[W^\pm]\to \rank(W)$. The definition of the well-order $W$ guarantees that for any $z\in \w_\alpha\times\w_\alpha$ the initial interval $\cev W(z)$ is contained in the square $\beta\times\beta$ of some ordinal $\beta\in\w_\alpha$. Since $\w_\alpha$ is a cardinal, $|\beta|<|\w_\alpha|=\aleph_\alpha$. By Theorem~\ref{t:alephs}, there exists an ordinal $\gamma$ such that $|\beta|=\aleph_\gamma$. Taking into account that $\aleph_\gamma=|\beta|<|\w_\alpha|=\aleph_\alpha$ and applying Proposition~\ref{p:walpha} and Theorem~\ref{t:ord}(6), we conclude that $\gamma<\alpha$. Then by the inductive assumption, $|\beta\times\beta|=\aleph_\gamma\cdot\aleph_\gamma=\aleph_\gamma<\aleph_\alpha=|\w_\alpha|$ and consequently, $|\cev W(z)| \le|\beta\times\beta|<|\w_\alpha|$.
Since $\rank_W:\w_\alpha\times\w_\alpha\to\rank(W)\subset\Ord$ is an order isomorphism, for every $z\in\rank(W)$ the initial interval $\cev E(z)=z\in\Ord$ has cardinality $|z|<|\w_\alpha|$ and hence $z\subset \w_\alpha$. Then $\rank(W)=\bigcup\{z:z\in\rank(W)\}\subseteq\w_\alpha$ and hence $\aleph_\alpha\cdot\aleph_\alpha\le |\w_\alpha\times\w_\alpha|=|\rank(W)|\le|\w_\alpha|=\aleph_\alpha$. The inequality $\aleph_\alpha\le\aleph_\alpha\cdot\aleph_\alpha$ is trivial. By Theorem~\ref{t:CBDS}, $\aleph_\alpha\cdot\aleph_\alpha=\aleph_\alpha$.
\end{proof}

Theorems~\ref{t:aleph-times} and \ref{t:CBDS} imply:

\begin{corollary}\label{c:aleph-plus} For any ordinals $\alpha\le\beta$ we have 
$$\aleph_\alpha+\aleph_\beta=\aleph_\alpha\cdot\aleph_\beta=\aleph_\beta.$$
\end{corollary}

In its turn, Corollaries~\ref{c:aleph-plus}  and \ref{c:alephAC} imply

\begin{corollary}\label{c:card-plus-AC} Assume that Axiom of Choice. Then for any infinite cardinalities $\kappa,\lambda$ we have $$\kappa+\lambda=\kappa\cdot\lambda=\max\{\kappa,\lambda\}.$$
\end{corollary}

We are going to show that the Axiom of Choice cannot be removed from Corollary~\ref{c:card-plus-AC}.

\begin{lemma}\label{l:tarski0} If for some sets $x,y,z,\alpha$ we have $|x|+|\alpha|=|y\times z|$, then $|z|\le|x|$ or $|y|\le^*|\alpha|$.
\end{lemma}

\begin{proof} The equality  $|x|+|\alpha|=|y\times z|$ implies that $y\times z=f[x]\cup g[\alpha]$ for some injective functions $f:x\mapsto y\times z$ and $g:\alpha\to y\times z$ with $f[x]\cap g[\alpha]=\emptyset$. If for some $v\in y$ the set $\{v\}\times z$ is a subset of $f[x]$, then the injective function $f^{-1}{\restriction}_{\{v\}\times z}$ witnesses that $|z|\le|x|$. So, assume that for every $v\in y$ the set $\{v\}\times z$ is not contained in $f[x]$.
 Then it intersects the set $g[\alpha]=(y\times z)\setminus f[x]$ and the function $\dom\circ g:\alpha\to y$ is surjective, witnessing that $|y|\le^*|\alpha|$.
\end{proof}

Combining Lemma~\ref{l:tarski0} with Proposition~\ref{p:card-ineq}(2), we obtain the following lemma.

\begin{lemma}\label{l:tarski} If for some sets $x,y,z$ and ordinal $\alpha$ we have $|x|+|\alpha|=|y\times z|$, then $|z|\le|x|$ or $|y|\le|\alpha|$.
\end{lemma}
\begin{theorem}[Tarski] The following conditions are equivalent.
\begin{enumerate}
\item[\textup{1)}] For any infinite sets $x,y$ we have $|x|+|y|=|x|\cdot|y|$;
\item[\textup{2)}]  For any infinite set $x$ we have $|x|+|x^+|=|x|\cdot|x^+|$;
\item[\textup{3)}]  For any infinite set $x$ we have $|x\times x|=|x|$;
\item[\textup{4)}]  The Axiom of Choice holds.
\end{enumerate}
\end{theorem}

\begin{proof} The implication $(1)\Ra(2)$ is trivial.
\smallskip

$(2)\Ra(4)$: Assuming some set $x$ has $|x|+|x^+|=|x|\cdot|x^+|$, we can apply Lemma~\ref{l:tarski} and conclude that $|x^+|\le|x|$ or $|x|\le|x^+|$. The inequality $|x^+|\le|x|$ contradicts the definition of the ordinal $x^+$. Therefore, $|x|\le|x^+|$ which implies that the set $x$ admits an injective function into the ordinal $x^+$ and hence $x$ can be well-ordered. By Theorem~\ref{t:mainAC}, the Axiom of Choice holds.

The implication $(4)\Ra(1)$ has been proved in Corollary~\ref{c:card-plus-AC}.

The implication $(4)\Ra(3)$ follows from Corollary~\ref{c:card-plus-AC}. 

$(3)\Ra(2)$: Given any infinite set $x$, consider the set $y=(\{0\}\times x)\cup(\{1\}\times x^+)$. By (3) and Exercise~\ref{ex:sum-prod-card}(8), we have
$$ 
|y\times y|=|y|=|x|+|x^+|\le|x\times x^+|\le |y\times y|$$ and hence 
$|x|+|x^+|=|x\times x^+|$.
\end{proof}

\begin{theorem}\label{t:AC-union} Let $I$ be a set and $(x_i)_{i\in I}$ be an indexed family of sets and $\kappa$ be an infinite cardinal such that $|I|\le|\kappa|$ and $|x_i|\le|\kappa|$ for all $i\in I$. If the Axiom of Choice holds, then $|\bigcup_{i\in I}x_i|\le|\kappa|$.
\end{theorem} 

\begin{proof} For every $i\in I$ consider the set $F_i$ of all injective functions from the set $x_i$ to the cardinal $\kappa$. By our assumption, $|x_i|\le|\kappa|$ and hence the set $F_i$ is not empty. By the Axiom of Choice, the Cartesian product $\prod_{i\in I}F_i$ is not empty and hence contains some indexed family of injective functions $(f_i)_{i\in I}$. Since $|I|\le\kappa$, there exists an injective function $g:I\to\kappa$. The function $g$ induced the well-order $W=\{\langle i,j\rangle\in I\times I:g(i)\in g(j)\}$ on the index set $I$. For every $u\in\bigcup_{i\in I}X_i$ let $\mu(u)$ be the unique $W$-minimal element of the non-empty set $\{i\in I: u\in x_i\}$. Then the injective function $$h:\bigcup_{i\in I}x_i\to \kappa\times\kappa,\quad h:u\mapsto \langle g(u),f_{\mu(u)}(u)\rangle$$witnesses that
$|\textstyle\bigcup_{i\in I}x_i|\le |\kappa\times\kappa|=|\kappa|$, where the last equality follows from Theorems~\ref{t:aleph-times} and \ref{t:alephs}.
\end{proof}

Next, we consider the exponentiation of cardinalities.

\begin{exercise} For any nonzero  cardinalities $\kappa,\lambda,\mu$, the exponentiation has the following properties:
\begin{itemize}
\item $\kappa^{\lambda+\mu}=\kappa^\lambda\cdot\kappa^\mu$;
\item $\kappa^{\lambda\cdot\mu}=(\kappa^\lambda)^{\mu}$;
\item If $\kappa\le\lambda$, then $\kappa^\mu\le\lambda^\mu$;
\item If $\lambda\le^*\mu$, then $\kappa^\lambda\le\kappa^\mu$.
\end{itemize}
\end{exercise}

\begin{exercise} Prove that $|2^x|=|\mathcal P(x)|$ for every set $x$.
\smallskip

\noindent{\em Hint}: Observe that the function $\xi:2^x\to\mathcal P(x)$, $\xi:f\mapsto f^{-1}[\{1\}]$, is bijective.
\end{exercise} 

\begin{exercise} Prove that $|x|<|2^x|$ for any set $x$.
\end{exercise}

Many results on exponents of cardinalities can be derived from K\H onig's Theorem~\ref{t:Konig}, which compares the cardinalities of sum and products of cardinals.

For an indexed family of classes $(X _i)_{i\in I}$, consider the class 
$$\textstyle\sum\limits_{i\in I}X_i=\bigcup\limits_{i\in I}\{i\}\times X_i.$$

\begin{lemma}\label{l:Konig1} Let $I$ be a set and $(\kappa_i)_{i\in I}$ and $(\lambda_i)_{i\in I}$ be two indexed families of cardinals. If  $\forall i\in I\;(\max\{2,\kappa_i\}\le\lambda_i)$, then $|\sum_{i\in I}\kappa_i|\le|\prod_{i\in I}\lambda_i|$.
\end{lemma}

\begin{proof}  If $|I|\le |2|$ then the inequality $|\sum_{i\in I}\kappa_i|\le|\prod_{i\in I}\lambda_i|$ follows from Exercise~\ref{ex:sum-prod-card}(8). So, we assume that $|I|\ge |3|$.

In this case the inequality $|\sum_{i\in I}\kappa_i|\le|\prod_{i\in I}\lambda_i|$ is witnessed by the injective function $f:\sum_{i\in I}\kappa_i\to\prod_{i\in I}\lambda_i$ assigning to every ordered pair $\langle i,x\rangle\in \bigcup_{j\in I}(\{j\}\times\kappa_j)$ the function $f_{\langle i,x\rangle}:I\to \bigcup_{i\in I}\lambda_i$ such that 
$$f_{\langle i,x\rangle}(j)=
\begin{cases}
0&\mbox{if $x>0$ and $j\in I\setminus\{i\}$};\\
x&\mbox{if $x>0$ and $j=i$};\\
1&\mbox{$x=0$ and $j\in I\setminus\{i\}$};\\
0&\mbox{$x=0$ and $j=i$}.
\end{cases}
$$
\end{proof}

\begin{lemma}\label{l:Konig2}  Let $I$ be a set,  $(x_i)_{i\in I}$ an indexed family of sets and $(\lambda_i)_{i\in I}$ an indexed family of cardinals. If  $\forall i\in I\;|\lambda_i|\not\le^*|x_i|$, then $|\prod_{i\in I}\lambda_i|\not\le^* |\sum_{i\in I}x_i|$.
\end{lemma}

\begin{proof} To derive a contradiction, assume that $|\prod_{i\in I}\lambda_i|\le^* |\sum_{i\in I}x_i|$ and find a surjective function $f:\sum_{i\in I}x_i\to\prod_{i\in I}\lambda_i$. For every $i\in I$ denote by $$\pr_i:\prod_{j\in I}\lambda_j\to \lambda_i,\quad \pr_i:g\mapsto g(i),$$ the projection onto the $i$-th coordinate. It follows from $|\lambda_i|\not\le^*|x_i|$ that $\pr_i\circ f[\{i\}\times x_i]\ne\lambda_i$. So, we can define a function $g\in\prod_{j\in I}\lambda_j$ assigning to every $i\in I$ the unique $\E$-minimal element of the nonempty subset $\lambda_i\setminus\pr_i\circ f[\{i\}\times x _i]$ of the cardinal $\lambda_i$.
 
 By the surjectivity of $f$, there exists $i\in I$ and $z\in x_i$ such that $g=f(i,z )$. Then $g(i)=\pr_i\circ f(i,z)\in\pr_i\circ f[\{i\}\times x_i]$, which contradicts the definition of $g$. 
 \end{proof}

Lemmas~\ref{l:Konig1} and \ref{l:Konig2} imply the following 

\begin{theorem}[K\H onig]\label{t:Konig2} Let $I$ be a set, and $(\kappa_i)_{i\in I}$ and $(\lambda_i)_{i\in I}$ be two indexed families of cardinals. If\/ $\forall i\in I\;(\kappa_i<\lambda_i)$, then $|\sum_{i\in I}\kappa_i|<|\prod_{i\in I}\lambda_i|$.
\end{theorem}
 
 \begin{remark} For $\kappa_i=1$ and $\lambda_i=2$, K\H onig's Theorem~\ref{t:Konig2} implies the strict inequality $|I|=|\sum_{i\in I}\kappa_i|<|\prod_{i\in I}\lambda_i|=|2^I|$, which has been established in Corollary~\ref{c:CantorP}.
 \end{remark}

\section{Cofinality of cardinals}

\begin{definition} The \index{cofinality}\index{ordinal!cofinality of}{\em cofinality} $\cf(\alpha)$ of an ordinal $\alpha$ is the smallest cardinal $\lambda$ for which there exists a function $f:\lambda\to\alpha$ which is \index{unbounded function}\index{function!unbounded}{\em unbounded} in the sense that for every $\beta<\kappa$ there exists $\gamma\in\lambda$ such that $\beta\le f(\gamma)$.
\end{definition}

\begin{exercise} Check that:
\begin{itemize}
\item $\cf(0)=0$;
\item $\cf(n)=1$ for any natural number $n>0$;
\item $\cf(\w)=\w$;
\item $\cf(\alpha+1)=1$ for any ordinal $\alpha$;
\item $\cf(\omega_\alpha)\le\cf(\alpha)$ for any limit ordinal $\alpha$. 
\item $\cf(\alpha)\le\alpha$ for any ordinal $\alpha$.
\end{itemize}
\end{exercise}

\begin{proposition}\label{p:moncf} For every limit ordinal $\kappa$ there exists an unbounded increasing function $f:\cf(\kappa)\to\kappa$.
\end{proposition}

\begin{proof} By the definition of the cardinal $\cf(\kappa)$, there exists an unbounded function $\varphi:\cf(\kappa)\to\kappa$. Consider the function $F:\cf(\kappa)\times\UU\to\UU$ assigning to each ordered pair $\langle \alpha,x\rangle\in\cf(\kappa)\times\UU$ the set $\varphi(\alpha)\cup (\bigcup\rng[x])\cup\{\bigcup \rng[x]\}$. By Recursion Theorem~\ref{t:Recursion}, there exists a function $f:\cf(\kappa)\to\UU$ such that $f(\alpha)=F(\alpha,f{\restriction}_\alpha)$ for every ordinal $\alpha\in\cf(\kappa)$.  We claim that for every ordinal $\alpha\in\cf(\kappa)$, the set $f(\alpha)$ is an ordinal such that $f(\alpha)\in\kappa$ and $f(\beta)\in f(\alpha)$ for every $\beta\in\alpha$. Indeed, for $\alpha=0$, the set $f(0)=F(0,\emptyset)=\varphi(0)\cup(\bigcup\emptyset)\cup\{\bigcup\emptyset\}=\varphi(0)\cup\{0\}=\max\{\varphi(0),1\}$ is an ordinal with $f(0)\in\kappa$. Assume that for some ordinal $\alpha\in\cf(\kappa)$ and all ordinals $\gamma<\beta<\alpha$ we have proved that the set $f(\beta)$ is an ordinal such that $f(\gamma)<f(\beta)<\kappa$.  The minimality of $\cf(\kappa)>\alpha$ implies that the set $f[\alpha]$ is bounded in $\kappa$ and hence $\bigcup f[\alpha]$ is an element of $\kappa$. Since the ordinal $\kappa$ is limit, the set $f(\alpha)=\varphi(\alpha)\cup(\bigcup f[\alpha])\cup\{\bigcup f[\alpha]\}$ is an element of $\kappa$, too. Moreover, for every $\beta\in\alpha$ we have $f(\beta)\subseteq\bigcup f[\alpha]\subset f(\alpha)$ and hence $f(\beta)<f(\alpha)$ according to Theorem~\ref{t:ord}(4).

Since the function $\varphi:\cf(\kappa)\to\kappa$ is unbounded and $\varphi(\alpha)\le f(\alpha)$ for all $\alpha\in\cf(\kappa)$, the strictly increasing function $f:\cf(\kappa)\to\kappa$ is unbounded as well. 
\end{proof}

\begin{definition} An infinite cardinal $\kappa$ is called
\begin{itemize}
\item \index{regular cardinal}\index{cardinal!regular}{\em regular} if $\cf(\kappa)=\kappa$;
\item  \index{singular cardinal}\index{cardinal!singular} {\em singular} if $\cf(\kappa)<\kappa$.
\end{itemize}
\end{definition}

\begin{example} The cardinal $\w_\w$ is singular because $\w_\w>\cf(\w_\w)=\w$.
\end{example}

\begin{proposition} For any limit ordinal $\alpha$ its cofinality $\cf(\alpha)$ is a regular cardinal.
\end{proposition}

\begin{proof} By Proposition~\ref{p:moncf}, there exists an unbounded increasing function $f:\cf(\alpha)\to\alpha$.

Since the ordinal $\alpha$ is limit, the cardinal $\cf(\alpha)$ is infinite and hence is a limit ordinal. Assuming that $\cf(\alpha)$ is a singular, we can find a cardinal $\kappa<\cf(\alpha)$ and an unbounded increasing function $g:\kappa\to\cf(\kappa)$.
Then $f\circ g:\kappa\to \alpha$ is an unbounded function, which contradicts the minimality of the cardinal $\cf(\alpha)>\kappa$.
\end{proof}

\begin{theorem}\label{t:w-reg} Under the Axiom of Choice, for every ordinal $\alpha$ the cardinal $\w_{\alpha+1}$ is regular.
\end{theorem}

\begin{proof} To derive a contradiction, assume that the cardinal $\w_{\alpha+1}$ is singular and hence the cardinal $\kappa=\cf(\w_{\alpha+1})$ is strictly smaller than $\w_{\alpha+1}$. Then $|\kappa|\le\w_\alpha$ by the definition of $\w_{\alpha+1}=\w_\alpha^+$.

By the definition of the cofinality $\cf(\w_{\alpha+1})=\kappa$, there exists an unbounded function $f:\kappa\to \w_{\alpha+1}$. Since $\w_{\alpha+1}$ is a cardinal, for every $\gamma\in\kappa$, the ordinal $f(\gamma)\subset\w_{\alpha+1}$ has cardinality $|f(\gamma)|<|\w_{\alpha+1}|$. By the definition of the successor cardinal $\w_{\alpha}^+=\w_{\alpha+1}$, the ordinal $f(\gamma)$ admits an injective function to $\w_\alpha$ and hence $|f(\gamma)|\le\w_\alpha$. 

Applying Theorem~\ref{t:AC-union}, we conclude that
$$|\w_\alpha^+|=|\w_{\alpha+1}|=|\textstyle\bigcup_{\alpha\in\kappa}f(\alpha)|\le|\w_\alpha|,$$
which contradicts the definition of the successor cardinal $\w_\alpha^+$. This contradiction shows that the cardinal $\w_{\alpha+1}$ is regular.
\end{proof}

\begin{remark} The Axiom of Choice is essential in the proof of Theorem~\ref{t:w-reg}:   by \cite[Theorem 10.6]{Jech} it is consistent with the axioms of ZF that the cardinal $\w_1$ is a countable union of countable sets and hence $\cf(\w_1)=\w$.
\end{remark}

Now derive some corollaries of K\H onig's Theorem~\ref{t:Konig2} that involve the cofinality.

\begin{corollary}\label{c:k<kcf} Every infinite cardinal $\kappa$ has cardinality $|\kappa|<|\kappa^{\cf(\kappa)}|$.
\end{corollary}

\begin{proof} By the definition of the cardinal $\cf(\kappa)$, there exists an unbounded function $f:\cf(\kappa)\to\kappa$. The unboundedness of $f$ guarantees that $\kappa=\bigcup_{\alpha\in\cf(\kappa)}f(\alpha)$.  For every $x\in\kappa$ let $\alpha(x)$ be the smallest ordinal such that $x\in f(\alpha(x))$. The injective map 
$$\textstyle g:\kappa\to\sum_{\alpha\in\cf(\kappa)}f(\alpha),\quad g:x\mapsto \langle \alpha(x),x\rangle$$ witnesses that $|\kappa|\le |\sum_{\alpha\in\cf(\kappa)}f(\alpha)|$. Since $\kappa$ is a cardinal, for every $\alpha\in\cf(\kappa)$, the ordinal $f(\alpha)\in\kappa$ has cardinality $|f(\alpha)|<|\kappa|$. Applying Theorem~\ref{t:Konig2},  we obtain that
$$\textstyle|\kappa|\le|\sum_{\alpha\in\cf(\kappa)}f(\alpha)|<|\prod_{\alpha\in\cf(\kappa)}\kappa|=|\kappa^{\cf(\kappa)}|.$$
\end{proof}

\begin{corollary} Let $\kappa,\lambda$ be infinite cardinals. If $|\lambda|=|a^\kappa|$ for some set $a$, then $\kappa<\cf(\lambda)$.
\end{corollary}

\begin{proof} Assuming that $\cf(\lambda)\le\kappa$, find an unbounded function $f:\kappa\to \lambda$. For every $i\in\kappa$ consider the ordinal $\kappa_i=f(i)\subset\lambda$. Since $\lambda$ is a cardinal, $|\kappa_i|<|\lambda|$. Since $f$ is unbounded, $\lambda=\bigcup_{i\in\kappa}f(i)$.  By Theorems~\ref{t:Konig2} and \ref{t:aleph-times},
$$|\lambda|=|\bigcup_{i\in\kappa}f(i)|\le|\sum_{i\in\kappa}\kappa_i|<|\prod_{i\in\kappa}\lambda|=|\lambda^\kappa|=|(a^{\kappa})^\kappa|=|a^{\kappa\cdot\kappa}|=|a^\kappa|=|\lambda|,$$
which is a desired contradiction.
\end{proof}

\section{(Generalized) Continuum Hypothesis}

\rightline{\em Wir m\"ussen wissen, wir werden wissen}
\rightline{David Hilbert}
\vskip10pt

By Cantor's Theorem~\ref{t:CantorP}, for every set $x$ the cardinality of its power-set is strictly larger than the cardinality of $x$. Observe that for a finite set $x$ the set $\{y\in\Crd:|x|<|y|<|\mathcal P(x)|\}$ has cardinality $|2^x|-|x+1|$ and hence contains many cardinals.

In 1878 Cantor made a conjecture that for infinite sets the situation is different: there is no cardinality $|x|$ such that $|\w|<|x|<|\mathcal P(\w)|$. This conjecture is known as the Continuum Hypothesis\index{Continuum Hypothesis ({{\sf CH}})}\index{({{\sf CH}})} (briefly $(\mathsf{CH})$). At the presence of the Axiom of Choice the Continuum Hypothesis is equivalent to the equality $\aleph_1=\mathfrak c$, where $\mathfrak c$ denotes the cardinality $|\mathcal P(\w)|$ of the power-set $\mathcal P(\w)$ and is called the  \index{cardinality of continuum $\mathfrak c$}\index{$\mathfrak c$}{\em cardinality of continuum}. 

Cantor himself tried to prove the Continuum Hypothesis but without success. David Hilbert included the Continuum Hypothesis as problem number one in his famous list of open problems announced in the II World Congress of mathematicians in 1900. For the complete solution, this problem waited more than 60 years. The final solution appeared to be a bit unexpected. First, in 1939 Kurt G\"odel proved that the Continuum Hypothesis does not contradict the axioms $\NBG$ or {\sf ZFC}. Twenty four years later, in 1963 Paul Cohen proved that the negation {\sf CH} does not contradict the axioms {\sf ZFC}. So, CH turned out to be independent of {\sf ZFC}. It can be neither proved nor disproved. To prove that CH does not contradicts the axioms $\NBG$, G\"odel established that it holds in the constructible universe. Moreover, G\"odel showed that his Axiom of Constructibility $\UU=\LL$ implies the following more general statement, called the \index{Generalized Continuum Hypothesis ({{\sf GCH}})}\index{({{\sf GCH}})}{\em Generalized Continuum Hypothesis}
\begin{center}
\fbox{$\mbox{({\sf GCH}): For any infinite set $x$ there is no cardinality $\kappa$ such that $|x|<\kappa<|2^x|$.}$}
\end{center}
\smallskip

Under the Axiom of Choice the Generalized Continuum Hypothesis is equivalent to the statement that $|\kappa^+|=|2^\kappa|$ for every infinite cardinal $\kappa$. 

\begin{theorem}[G\"odel] $(\UU=\LL)\;\Ra\;(\mathsf{GCH})$.
\end{theorem}

\begin{proof} Assume that $\UU=\LL$. By Theorem~\ref{t:well-order}, $(\UU=\LL)$ implies the Axiom of (Global) Choice. So, ({\sf GCH}) will be established as soon as we prove the equality $|\mathcal P(\kappa)|=|\kappa^+|$ for every infinite cardinal $\kappa$. We recall that $\mathbf L=\bigcup_{\alpha\in\On}L_\alpha$ where $L_\alpha=\bigcup_{\beta\in\alpha}\mathcal P(L_\beta)\cap\Go^{\circ\w}(L_\beta\cup\{L_\beta\})$ for every ordinal $\alpha$. Here $\Go^{\circ\w}(L_\beta\cup\{L_\beta\})$ is the smallest set that contains $L_\beta\cup\{L_\beta\}$ as a subset and is closed under G\"odel's operations $\dG_0$--$\dG_8$, see Section~\ref{s:L}. It is easy to see that $|L_{\beta+1}|\le|\Go^{\circ\w}(L_\beta\cup\{L_\beta\})|\le\max\{\w,|L_\beta|\}$ for every ordinal $\beta$. This implies that $|L_\alpha|=|\alpha|$ for every infinite ordinal $\alpha$. It can be shown (but it is difficult and requires more advanced model-theoretic tools, see e.g. \cite[13.20]{Jech} or \cite{Koepke}) that $\mathcal P(\kappa)\cap\mathbf L=\mathcal P(\kappa)\cap L_{\kappa^+}$ and hence  
$|\mathcal P(\kappa)|=|\mathcal P(\kappa)\cap\mathbf L|\le|L_{\kappa^+}|=|\kappa^+|$ under $(\UU=\LL)$.
\end{proof}  

\begin{Exercise} Prove\footnote{In case you find an elementary proof of this fact, write for a prize to {\tt t.o.banakh@gmail.com}.}  that $\mathcal P(\kappa)\cap\LL\subseteq L_{\kappa^+}$ for any infinite cardinal $\kappa$.
\end{Exercise}


It turns out that $(\mathsf{GCH})$ implies $(\mathsf{AC})$. The following theorem was announced by Lindenbaum and Tarski in 1926 but the first written proof was published only in 1947 by Sierpi\'nski \cite{GCH}. 

\begin{theorem}[Sierpi\'nski]\label{t:GCH=>AC} The Generalized Continuum Hypothesis implies the Axiom of Choice, i.e., $\mathsf{(GCH)}\;\Ra\;\mathsf{(AC)}$.
\end{theorem}

Theorem~\ref{t:GCH=>AC} will be derived from its local version.

\begin{lemma} A set $x$ can be well-ordered if for every $n\in\{3,4,5\}$ the class $$\{y\in\UU: |\mathcal P^{\circ n}(x)|<|y|<|\mathcal P^{\circ(n+1)}(x)|\}$$ is empty.
\end{lemma}

\begin{proof} If the set $x$ is finite, then it can be well-ordered without any additional assumptions. So, assume that $x$ is infinite and for every $n\in\{3,4,5\}$ every cardinality $\kappa$ with $|\mathcal P^{\circ n}(x)|\le \kappa\le|\mathcal P^{\circ(n+1)}(x)|$ is equal either to $|\mathcal P^{\circ n}(x)|$ or to $|\mathcal P^{\circ(n+1)}(x)|$.

By Theorem~\ref{t:comp-n}, the successor cardinal $x^+$ of $x$ is infinite and hence $\w\le x^+$. By the Hartogs--Sierpi\'nski Theorem~\ref{t:Hartogs2}, $|\w|\le |x^+\cap\;\Crd|\le|\mathcal P^{\circ2}(x)|$. 

For every $n\in\w$ consider the iterated power-set $p_n=\mathcal P^{\circ n}(x)$ of $x$.
The inequality $|\w|\le|x^+\cap\;\Crd|\le|\mathcal P^{\circ2}(x)|=|p_2|$ implies that $1+|p_2|=|p_2|$.

\begin{claim}\label{cl:GCH1} For every natural number $n\ge 3$ we have $2\cdot|p_{n}|=|p_n|$. 
\end{claim}

\begin{proof} For $n=3$ we have $$2\cdot|p_3|=2\cdot |2^{p_2}|=2^{1+|p_2|}=2^{|p_2|}=|p_3|.$$
Assume that for some $n\ge 3$ we proved that $2\cdot |p_n|=|p_n|$.
Then $|p_n|\le 1+|p_n|\le|p_n|+|p_n|=2\cdot|p_n|=|p_n|$ and hence $1+|p_n|=|p_n|$ by Theorem~\ref{t:CBDS}.
Finally, 
 $$2\cdot |p_{n+1}|=2\cdot 2^{|p_n|}=2^{1+|p_n|}=2^{|p_n|}=|p_{n+1}|.$$
\end{proof}

\begin{claim}\label{cl:GCH2} For every  natural number $n\ge 3$ and every ordinal $\alpha$ the inequality $|\alpha|+|p_n|=|p_{n+1}|$ implies $|p_{n+1}|\le|\alpha|$.
\end{claim} 

\begin{proof} By Claim~\ref{cl:GCH1}, $$|p_{n+1}|=2^{|p_n|}=2^{|p_n|+|p_n|}=2^{|p_n|}\cdot 2^{|p_n|}=|p_{n+1}|\cdot|p_{n+1}|$$
and hence $|\alpha|+|p_n|=|p_{n+1}|\cdot|p_{n+1}|$. By Lemma~\ref{l:tarski}, either $|p_{n+1}|\le|p_n|$ or $|p_{n+1}|\le|\alpha|$. The first case is excluded by Cantor's Theorem~\ref{t:CantorP}. Therefore $|p_{n+1}|\le|\alpha|$.
\end{proof}

By Hartogs--Sierpi\'nski Theorem~\ref{t:Hartogs2}, the successor cardinal $\alpha=p_3^+$ of the set $p_3$ has cardinality $|\alpha|\le|\mathcal P^{\circ3}(p_3)|=|p_6|$.

Then $|p_5|\le |\alpha|+|p_5|\le 2\cdot|p_6|=|p_6|$. By our assumption, either $|\alpha|+|p_5|=|p_6|$ or $|\alpha|+|p_5|=|p_5|$.
In the first case we can apply Claim~\ref{cl:GCH2} and conclude that $|x|\le|p_6|\le|\alpha|$, which implies that $x$ admits an injective function into the cardinal $\alpha=p_3^+$ and hence $x$ can be well-ordered.

So, consider the second case $|\alpha|+|p_5|=|p_5|$. In this case $|p_4|\le |\alpha|+|p_4|\le|p_5|+|p_5|=|p_5|$ and by our assumption, the cardinality $|\alpha|+|p_4|$ is equal either to $|p_5|$ or to $|p_6|$. If $|\alpha|+|p_4|=|p_5|$, then by Claim~\ref{cl:GCH2}, $|x|\le|p_5|\le|\alpha|$ and hence $x$ can be well-ordered. 

It remains to consider the case $|\alpha|+|p_4|=|p_4|$. Then $|p_3|\le|\alpha|+|p_3|\le |p_4|+|p_4|=|p_4|$ and by our assumption, either $|\alpha|+|p_3|=|p_4|$ or  $|\alpha|+|p_3|=|p_3|$. In fact, the latter case is not possible as $|\alpha|=|p_3^+|\not\le|p_3|$. So,  $|\alpha|+|p_3|=|p_4|$ and by Claim~\ref{cl:GCH2}, $|x|\le |p_4|\le|\alpha|$ and $x$ can be well-ordered.
\end{proof}

The Generalized Continuum Hypothesis can be characterized as follows.

\begin{theorem}\label{t:GCH} The following statements are equivalent:
\begin{enumerate}
\item[\textup{(1)}] \textup{({\sf GCH})} holds;
\item[\textup{(2)}] $|2^x|=|x^+|$ for every infinite set $x$.
\end{enumerate}
Under $(\mathbf U=\mathbf V)$ the statements \textup{(1)--(2)} are equivalent to the equivalent statements:
\begin{enumerate}
\item[\textup{(3)}] $|2^\kappa|=|\kappa^+|$ for any infinite cardinal $\kappa$;
\item[\textup{(4)}] $|2^{<\kappa}|=|\kappa|$ for any infinite cardinal $\kappa$.
\end{enumerate}
 \end{theorem}
 
\begin{proof} $(1)\Ra(2)$: Assume ({\sf GCH}). By Theorem~\ref{t:GCH=>AC}, the Axiom of Choice holds. By Theorems~\ref{t:card-compare}, any two cardinalities are comparable. Then for every infinite subset $x$ we have $|x|<|x^+|$ (as $|x^+|\le|x|$ is forbidden by the definition of the cardinal $x^+$). By Theorem~\ref{t:CantorP}, $|x|<|2^x|$. The minimality of the cardinal $x^+$ and the comparability of the cardinalities $|x^+|$ and $|2^x|$ implies $|x^+|\le |2^x|$. Therefore, $|x|<|x^+|\le|2^x|$. Now ({\sf GCH}) implies $|x^+|=|2^x|$.
\smallskip

$(2)\Ra(1)$: If $|2^x|=|x^+|$ for any infinite set $x$, then $|x|\le |2^x|=|x^+|$ and hence $x$ admits an injective function into the cardinal $x^+$, which implies that $x$ can be well-ordered. Therefore, the Axiom of Choice holds. Assuming that ({\sf GCH}) fails, we can find infinite sets $x,y$ such that $|x|<|y|<|2^x|$. Then $|x^+|\le|y|$ by the definition of the cardinal $x^+$ and comparability of the cardinalities $|x^+|$ and $|y|>|x|$. Then $|x^+|\le|y|<|2^x|$, which contradicts (2).
\smallskip

$(3)\Ra(4)$: Take any infinite cardinal $\kappa$ and assume that $|2^\lambda|=|\lambda^+|$ for any infinite cardinal $\lambda\le \kappa^+$. Using Theorem~\ref{t:alephs}, Corollary~\ref{c:aleph-plus} and the equality $|\mathcal P(\kappa)|=|2^\kappa|=|\kappa^+|$, we can show that $$|\mathcal P(\mathcal P(\kappa)\times\kappa)|=|\mathcal P(\kappa^+\times\kappa)|=|\mathcal P(\kappa^+)|=|2^{\kappa^+}|=|\kappa^{++}|,$$
which implies that the set $\mathcal P(\mathcal P(\kappa)\times \kappa)$ admits a well-order $\prec$. For every ordinal $\alpha\in\kappa$, consider the set $I_\alpha$ of injective functions $f:\mathcal P(\alpha)\to\kappa$. Observe that $I_\alpha\subseteq \mathcal P(\mathcal P(\kappa)\times\kappa)$.  Since $|\mathcal P(\alpha)|=|2^{\alpha}|=|\alpha^+|\le |\kappa|$, the set $I_\alpha$ is not empty. Let $f_\alpha$ be the smallest element of $I_\alpha$ with respect to the well-order $\prec$. Consider the function $\Phi:2^{<\kappa}\to\kappa\times\kappa$ assigning to every function $\varphi\in 2^{<\kappa}$ the ordered pair $\langle\dom[\varphi],f_{\dom[\varphi]}(\varphi^{-1}[\{1\}])\rangle$. Observe that the function $\Phi$ is injective, witnessing that $|2^{<\kappa}|\le|\kappa\times\kappa|=|\kappa|$.
On the other hand, the injective function $$\Psi:\kappa\to 2^{<\kappa},\quad\Psi:\alpha\mapsto\alpha\times\{1\},$$witnesses that $|\kappa|\le |2^{<\kappa}|$. Now Cantor--Bernstein--Schr\"oder Theorem~\ref{t:CBDS} implies $|\kappa|=|2^{<\kappa}|$.
\smallskip

$(4)\Ra(3)$: If for any infinite cardinal $\kappa$ we have $|2^{<\kappa}|=|\kappa|$, then for any infinite cardinal $\kappa$ we also have
$$|\kappa|<|2^\kappa|\le |2^{<\kappa^+}|=|\kappa^+|.$$ Now the minimality of the cardinal $\kappa^+$ implies $|\kappa^+|\le |2^\kappa|\le|\kappa^+|$ and hence $|\kappa^+|=|2^\kappa|$ by the Cantor-Bertstein-Schr\"oder Theorem~\ref{t:CBDS}. 
\smallskip

The implication $(2)\Ra(4)$ is trivial. Finally, assuming that $(\mathbf U=\mathbf V)$, we shall prove that $(3)\Ra(2)$. If for any cardinal $\kappa$ we have $|2^\kappa|=|\kappa^+|$, then the set $2^\kappa$ is well-orderable. By Theorem~\ref{t:AC+UV}, the Axiom of Choice holds. Then for every infinite set $x$ there exists a cardinal $\kappa$ such that $|x|=|\kappa|$ and hence $|2^x|=|2^\kappa|=|\kappa^+|=|x^+|$.
\end{proof}

Under ({\sf GCH}) the exponentiation of cardinals can be described by a simple formula, presented in the following theorem.

\begin{theorem} Assume {\sf (GCH)}. Let $\kappa,\lambda$ be infinite cardinals.
\begin{enumerate}
\item[\textup{1)}] If $\kappa\le\lambda$, then $|\kappa^\lambda|=|\lambda^+|$.
\item[\textup{2)}] If $\cf(\kappa)\le\lambda<\kappa$, then $|\kappa^\lambda|=|\kappa^+|$.
\item[\textup{3)}] If $\lambda<\cf(\kappa)$, then $|\kappa^\lambda|=|\kappa|$.
\end{enumerate}
\end{theorem}

\begin{proof} 1. If $\kappa\le \lambda$, then by ({\sf GCH}), 
$$|\lambda^+|=|2^\lambda|\le|\kappa^\lambda|\le |\lambda^\lambda|\le |(2^\lambda)^\lambda|=|2^{\lambda\times\lambda}|=|2^\lambda|=|\lambda^+|.$$
\smallskip

2. If $\cf(\kappa)\le\lambda<\kappa$, then by Corollary~\ref{c:k<kcf}, 
$$|\kappa|<|\kappa^{\cf(\kappa)}|\le |\kappa^\lambda|\le |(2^\kappa)^\lambda|=2^{|\kappa\times\lambda|}=2^{|\kappa|}=|\kappa^+|$$and hence $|\kappa^\lambda|=|\kappa^+|$.
\smallskip

3. Finally, assume that $\lambda<\cf(\kappa)$.  Observe that for any cardinal $\mu<\kappa$ and the cardinal $\nu=\max\{\mu,\lambda\}<\kappa$ we have $|\mu^\lambda|\le |\nu^\nu|\le |(2^\nu)^\nu|=2^{|\nu\times \nu|}=|2^\nu|=|\nu^+|\le\kappa$. The strict inequality $\lambda<\cf(\kappa)$ implies that $\kappa^\lambda=\bigcup_{\alpha\in\kappa}\alpha^\lambda$ and hence
$$|\kappa|\le|\kappa^\lambda|\le|\kappa|\cdot\sup_{\alpha\in\kappa}|\alpha^\lambda|\le |\kappa|\cdot|\kappa|=|\kappa|.$$
\end{proof}

\section{Inaccessible and measurable cardinals}



\begin{definition} An uncountable cardinal $\kappa$ is called
\begin{itemize}
\item \index{cardinal!weakly inaccessible}\index{weakly inaccessible cardinal}{\em weakly inaccessible} if $\kappa$ is  regular and $|\lambda^+|<|\kappa|$ for every cardinal $\lambda<\kappa$;
\item \index{cardinal!strongly inaccessible}\index{strongly inaccessible cardinal} {\em strongly inaccessible} if $\kappa$ is regular and $|2^\lambda|<|\kappa|$ for every cardinal $\lambda<\kappa$.
\end{itemize}
\end{definition}

\begin{remark} Under ({\sf GCH}) a cardinal is weakly inaccessible if and only if it is strongly inaccessible. The existence of  weakly inaccessible or strongly inaccessible cardinals can not be proved within the axioms ZFC since for the smallest strongly inaccessible cardinal $\kappa$ the set $V_{\kappa}$ and its elements is a model of ZFC (in which strongly inaccessible cardinals do not exist).
\end{remark}

Inaccessible cardinals are examples of large cardinals, i.e., cardinals that are so large that their existence cannot be derived from the axioms of NBG or ZFC.
Important examples of large cardinals are measurable cardinals, defined with the help of 2-valued measures.

\begin{definition}
A function $\mu:\mathcal P(x)\to\{0,1\}$ is called a \index{measure}{\em $2$-valued measure} on a set $x$ if 
\begin{enumerate}
\item $\mu(x)=1$;
\item for any disjoint subsets $a,b\subseteq x$ we have $\mu(a\cup b)=\mu(a)+\mu(b)$;
\item any finite subset $a\subseteq x$ has measure $\mu(a)=0$.
\end{enumerate}
\end{definition}

\begin{exercise} Show that for any 2-valued measure $\mu:\mathcal P(x)\to 2$ the family $U=\{a\in\mathcal P(x):\mu(a)=1\}$ is an ultrafilter with $\bigcup U=x$ and $\bigcap U=\emptyset$. 
\end{exercise}

\begin{exercise} Show that for any ultrafilter $U$ with $\bigcap U=\emptyset$, the function $\mu:\mathcal P(\bigcup U)\to2$ such that $\mu^{-1}[\{1\}]=U$  is a $2$-valued measure on the set $\bigcup U$.
\end{exercise}

\begin{exercise} Show that under the Axiom of Choice for every infinite set $x$ there exists a $2$-valued measure $\mu:\mathcal P(x)\to\{0,1\}$.
\end{exercise}

\begin{definition} 
A $2$-valued measure $\mu:\mathcal P(x)\to 2$ is called 
\begin{itemize} 
\item \index{measure!$\kappa$-additive}{\em $\kappa$-additive}  if for any subset $y\subseteq\{a\in\mathcal P(x):\mu(a)=0\}$ of cardinality $|y|\le |\kappa|$ the union $\bigcup y$ has measure $\mu(\bigcup y)=0$;
\item \index{measure!$\kappa^<$-additive}{\em $\kappa^<$-additive} if $\mu$ is $\lambda$-additive for every cardinal $\lambda<\kappa$.
\end{itemize}
\end{definition}

The existence of a $\kappa$-additive 2-valued measure on a set $x$ imposes the following restriction on the cardinality of $x$.

\begin{lemma}\label{l:measurable} Let $\kappa$ be a cardinal. If a $2$-valued measure $\mu:\mathcal P(x)\to 2$ on some set $x$ is $\kappa$-additive, then $|x|\not\le|2^\kappa|$.
\end{lemma}

\begin{proof} To derive a contradiction, assume that $|x|\le |2^{\kappa}|$. Then there exists an injective function $f:x\to 2^\kappa$. For every $i\in\kappa$  consider the projection $\pr_i:2^\kappa\to 2$, $\pr_i:g\mapsto g(i)$, onto the $i$-th coordinate. Next, for every $i\in\kappa$ and $k\in 2$, consider the set $a_{i,k}=\{y\in x:\pr_i\circ f(y)=k\}$. 
Since $x=a_{i,0}\cup a_{i,1}$ and $1=\mu(x)=\mu(a_{i,0})+\mu(a_{i,1})$ there exists a number $k_i\in 2$ such that $\mu(a_{i,k_i})=1$.  By the $\kappa$-additivity of the measure $\mu$, the intersection $a=\bigcap_{i\in\kappa}a_{i,k_i}$ has measure $\mu(a)=1$. On the other hand, the injectivity of $f$ implies that $|a|\le 1$ and hence $\mu(a)=0$.
\end{proof}

\begin{definition} An uncountable cardinal $\kappa$ is defined to be \index{cardinal!measurable}\index{measurable cardinal}{\em measurable} if it carries a  $\kappa^<$-additive $2$-valued measure $\mu:\mathcal P(\kappa)\to 2$. 
\end{definition}

\begin{theorem}[Tarski--Ulam] Under Axiom of Choice, each measurable cardinal is  strongly inaccessible.
\end{theorem}

\begin{proof} Let $\kappa$ be a  measurable cardinal and $\mu:\mathcal P(\kappa)\to 2$ be a 2-valued measure which is $\lambda$-additive for every cardinal $\lambda<\kappa$. By Lemma~\ref{l:measurable} and Theorem~\ref{t:card-compare}, for every cardinal $\lambda<\kappa$ we have $|2^\lambda|<|\kappa|$. To show that $\kappa$ is strongly inaccessible, it remains to prove that $\kappa$ is regular. To derive a contradiction, assume that $\cf(\kappa)<\kappa$ and choose an unbounded function $f:\cf(\kappa)\to\kappa$. Then $\kappa=\bigcup_{\alpha\in\cf(\kappa)}f(\alpha)$. Since $\kappa$ is a cardinal, for every $\alpha\in\cf(\kappa)$ the ordinal $f(\alpha)\in\kappa$ has cardinality $|f(\alpha)|<|\kappa|$. The $|f(\alpha)|$-additivity of the measure $\mu$ ensures that $\mu(f(\alpha))=\bigcup_{\beta\in f(\alpha)}\mu(\{\beta\})=0$. Applying the $\cf(\kappa)$-additivity of $\mu$, we conclude that $1=\mu(\kappa)=\mu(\bigcup_{\alpha\in\cf(\mu)}f(\alpha))=0$, which is a desired contradiction showing that the cardinal $\kappa$ is regular and hence strongly inaccessible.
\end{proof}

The cardinality of a set carrying an $\w$-additive $2$-valued measure still is very large.

\begin{proposition} The smallest cardinal $\kappa$ carrying an $\w$-additive $2$-valued measure is measurable and hence $\kappa$ strongly inaccessible under the Axiom of Choice.
\end{proposition}

\begin{proof} By our assumption, there exists a $\w$-additive $2$-valued measure $\mu:\mathcal P(\kappa)\to 2$. To show that $\kappa$ is measurable, it suffices to prove that the measure $\mu$ is $\kappa^<$-additive. To derive a contradiction, assume that $\mu$ is not  $\lambda$-additive for some cardinal $\lambda<\kappa$. Then there exists a family $(x_i)_{i\in\lambda}$ of sets of measure $\mu(x_i)=0$ such that $\mu(\bigcup_{i\in\lambda}x_i)=1$. Replacing each set $x_i$ by $x_i\setminus\bigcup_{j\in i}x_j$, we can assume that the family $(x_i)_{i\in\lambda}$ consists of pairwise disjoint sets. Observe that the function $\nu:\mathcal P(\lambda)\to 2$ defined by $\nu(a)=\mu(\bigcup_{i\in a}x_i)$ is an $\omega$-additive $2$-valued measure on the cardinal $\lambda<\kappa$. But this contradicts the minimality of $\kappa$.
\end{proof}

\begin{Exercise} Prove that for every measurable cardinal $\kappa$ and function $h:\mathcal P_{2}(\kappa)\to 2$ on the set $\mathcal P_{2}(\kappa)=\{x\in\mathcal P(\kappa):|x|=|2|\}$ there exists a subset $a\subseteq \kappa$ of cardinality $|a|=|\kappa|$ such that $|h[\mathcal P_{2}(a)]|=1$.
\end{Exercise}

\begin{remark} It is consistent with ZF that the cardinal $\w_1$ is measurable, see \cite[12.2]{JechAC}.
\end{remark}

The existence of a measurable cardinal contradicts the Axiom of Constructibility. This non-trivial fact was discovered by D.S.~Scott \cite{Scott} in 1961.

\begin{theorem}[Scott]\label{t:Scott} If a measurable cardinal exists, then $\UU\ne\LL$.
\end{theorem}

The proof of Theorem~\ref{t:Scott} is not elementary and can be found in \cite[17.1]{Jech}. 
\newpage

\part{Linear orders}

In this part we study linear orders. Linear orders often arise in mathematical practice. For example, the natural order on numbers (integer, rational or real) is a linear order, which fails to be a well-order. 

\section{Completeness}

In this section we study complete and boundedly complete linear orders. 

\begin{definition} An order $R$ is called \index{order!complete}\index{complete order}{\em complete} if every subclass $A\subseteq\dom[R^\pm]$ has $\sup_R(A)$ and $\inf_R(A)$.
\end{definition}

We recall that $\sup_R(A)$ is the unique $R$-least element of the class $\{b\in\dom[R^\pm]:A\times\{b\}\subseteq R\cup\Id\}$ and $\inf_R(A)$ the unique $R$-greatest element of the class $\{b\in\dom[R^\pm]:\{b\}\times A\subseteq R\cup\Id\}$.


By Exercise~\ref{ex:finite-order-mm}, any finite linear order is complete. 
Complete linear orders can be characterized as follows.

\begin{proposition}\label{p:comp} For an order $R$ the following conditions are equivalent:
\begin{enumerate}
\item[\textup{1)}] $R$ is  complete;
\item[\textup{2)}] each  subclass $A\subseteq\dom[R^\pm]$ has $\sup_R(A)$;
\item[\textup{3)}] each  subclass $B\subseteq\dom[R^\pm]$ has $\inf_R(B)$.
\end{enumerate}
\end{proposition}

\begin{proof} The implications $(1)\Ra(2,3)$ are trivial.
\smallskip

$(2)\Ra(3)$ Assume that every  subclass $A\subseteq\dom[R^\pm]$ has $\sup_R(A)$. In particular, the empty subset of $\dom[R^\pm]$ has $\sup_R(\emptyset)$, which is the $R$-least element of $\dom[R^\pm]$. Let $B$ be any subclass of $\dom[R^\pm]$. The subclass $A=\{a\in\dom[R^\pm]:\{a\}\times B\subseteq R\cup\Id\}$ of $\dom[R^\pm]$ contains the element $\sup_R(\emptyset)$ and hence it is not empty. By our assumption, the class $A$ has $\sup_R(A)$, which is the $R$-least element of the class $\overline{A}=\{b\in\dom[R^\pm]:A\times\{b\}\subseteq R\cup\Id\}$ of upper $R$-bounds of $A$ in $\dom[R^\pm]$. Since $B\subseteq\overline{A}$, the element $\sup_R(A)$ is  a lower $R$-bound for $B$. On the other hand, each lower $R$-bound $b\in\dom[R^\pm]$ for $B$ belongs to the class $\overline{A}$ and hence $\langle b,\sup_R(A)\rangle\in R\cup\Id$ by the definition of $\sup_R(A)$. This means that $\sup_R(A)=\inf_R(B)$, so $B$ has $\inf_R(B)$.
\smallskip

By analogy we can prove that $(3)\Ra(2)$.
\end{proof} 

 The following theorem implies that completeness is preserved by lexicographic powers of linear orders.

\begin{theorem}\label{t:complete-ord} For any complete linear order $L$ on a set $X=\dom[L^\pm]$ and ordinal $\alpha$ the lexicographic order $$L_\alpha=\{\langle f,g\rangle\in X^\alpha\times X^\alpha:\exists \beta\in\alpha\;(f{\restriction}_\beta=g{\restriction}_\beta\;\wedge\;\langle f(\beta),g(\beta)\rangle \in L\setminus \Id)\}$$
on the class $X^\alpha$ is complete.
\end{theorem}

\begin{proof} This theorem will be proved by transfinite induction. Observe that the set $X^0$ is a singleton and the order $L_0\subseteq X^0\times X^0$ complete.  Assume that for some ordinal $\alpha$ and all its elements $\beta\in\alpha$ we have proved that the order $L_\beta$ is complete. To show that the order $L_\alpha$ is complete, take any subclass $A\subseteq X^\alpha$. For every $\beta\in\alpha$, consider the projection $\pr_\beta:X^\alpha\to X^\beta$, $\pr_\beta:f\mapsto f{\restriction}_\beta$.

By the induction hypothesis, for every $\beta\in\alpha$ the linear order $L_\beta$ is complete and hence the set $\pr_\beta[A]\subseteq X^\beta$ has the smallest upper $L_\beta$-bound $s_\beta=\sup_{L_\beta}(\pr_\beta[A])\in X^\beta$.

If $\alpha$ is a successor ordinal, then $\alpha=\beta+1$ for some ordinal $\beta\in\alpha$. Consider the function $s_\beta=\sup_{L_\beta}(\pr_\beta[A])$ and the subset $A'=\{x\in X:s_\beta\cup\{\langle \beta,x\rangle\}\in A\}$. Since the order $L$ is complete, the set $A'$ has $\sup_L(A')\in X$ (remark that $\sup_L(\emptyset)=\inf_L(X)$). It can be shown that the function $$s_\alpha=s_\beta\cup\{\langle \beta,\sup{}_L(A')\rangle\}$$ is the required least upper bound $\sup_{L_\alpha}(A)$ on the set $A$ in $X^\alpha$.

If $\alpha$ is a limit ordinal, then we can show that the union $s_\alpha=\bigcup_{\beta\in\alpha}s_\beta$ is a function, which is the required least upper bound $\sup_{L_\alpha}(A)$ of the set $A$.

By analogy we can prove that $A$ has the greatest lower $L_\alpha$-bound $\inf_{L_\alpha}(A)$.
\end{proof}

\begin{corollary} For every ordinal $\alpha$ the lexicographic order
$$L=\{\langle f,g\rangle\in 2^\alpha\times 2^\alpha:\exists \beta\in\alpha\;(f{\restriction}_\beta=g{\restriction}_\beta\;\wedge\;f(\beta)\in g(\beta))\}$$on $2^\alpha$ 
is complete.
\end{corollary}

\begin{exercise} Complete all omitted details in the proof of Theorem~\ref{t:complete-ord}.
\end{exercise}

\begin{definition} An order $R$ is called \index{linear order!boundedly complete}\index{boundedly complete linear order}{\em boundedly complete} if 
\begin{itemize}
\item each  upper $R$-bounded nonempty subclass $A\subseteq\dom[R^\pm]$ has $\sup_R(A)$, and 
\item each  lower $R$-bounded nonempty subclass $A\subseteq\dom[R^\pm]$ has $\inf_R(A)$.
\end{itemize}
\end{definition} 

The following characterization of boundedly complete linear order can be proved by analogy with Proposition~\ref{p:comp}.

\begin{proposition}\label{p:bcomp} For an order $R$ the following conditions are equivalent:
\begin{enumerate}
\item[\textup{1)}] $R$ is boundedly complete;
\item[\textup{2)}] each  upper $R$-bounded nonempty subclass $A\subseteq\dom[R^\pm]$ has $\sup_R(A)$;
\item[\textup{3)}] each  lower $R$-bounded nonempty subclass $B\subseteq\dom[R^\pm]$ has $\inf_R(B)$.
\end{enumerate}
\end{proposition}




\section{Universality}


\begin{definition}  A linear order $L$ is called \index{linear order!universal}\index{universal linear order}{\em universal} if for any subsets $a,b$ of $\dom[L^\pm]$ with $|a\cup b|<|\dom[L^\pm]|$ and $a\times b\subseteq L\setminus\Id$ there exists an element $x\in \dom[L^\pm]$ such that $(a\times\{x\})\cup(\{x\}\times b)\subseteq L\setminus\Id$. 
\end{definition}

Examples of universal orders can be constructed as follows.

Given a subclass $\kappa\subseteq \Ord$, consider the class $2^{<\kappa}$ of all functions $f$ with $\dom[f]\in\kappa$ and $\rng[f]\subseteq 2=\{0,1\}$, endowed with  the linear order
\begin{multline*}
\mathsf{U}_{2^{<\kappa}}=\{\langle f,g\rangle\in 2^{<\kappa}\times 2^{<\kappa}:\exists \alpha\in\dom[f]\cap\dom[g]\;(f{\restriction}_\alpha=g{\restriction}_\alpha\;\wedge\;f(\alpha)=0\;\wedge\;g(\alpha)=1)\\
\vee\;(f\cup\{\langle \dom[f],1\rangle\}\subseteq g)\;\vee\;(g\cup\{\langle\dom[g],0\rangle\}\subseteq f)\}.
\end{multline*}

\begin{theorem}\label{t:univ-ord} The linear order $\mathsf{U}_{2^{<\kappa}}$ is universal if $\kappa=\Ord$ or $\kappa$ is a regular cardinal with $|\kappa|=|2^{<\kappa}|$. 
\end{theorem}

\begin{proof} Assume that $\kappa=\Ord$ or $\kappa$ is a regular cardinal with $|\kappa|=|2^{<\kappa}|$. Given any  subsets $a,b\subseteq 2^{<\kappa}$ with $|a\cup b|<|\dom[\mathsf{U}_{2^{<\kappa}}^\pm]|$ and $a\times b\subseteq \mathsf{U}_{2^{<\kappa}}$, we need to find an element $z\in 2^{<\kappa}$ such that $(a\times\{z\})\cup(\{z\}\times b)\subseteq\mathsf{U}_{2^{<\kappa}}$. By the Axiom of Replacement, the set $\gamma=\bigcup\{\dom[f]:f\in a\cup b\}\subseteq\Ord$ is an ordinal. 

We claim that $\gamma\subset\kappa$. This is clear if $\kappa=\Ord$. If $\kappa$ is a regular cardinal with $|\kappa|=|2^{<\kappa}|$, then $|a\cup b|<|2^{<\kappa}|=|\kappa|$. In this case the set $2^{<\kappa}$ is well-orderable and so is the set $a\cup b$. It follows that the set $\Gamma=\{\dom[f]:f\in a\cup b\}\subseteq\kappa$ has cardinality  $|\Gamma|\le|a\cup b|<|\kappa|$, see Proposition~\ref{p:card-ineq}(2). By the regularity of the cardinal $\kappa$, the union $\gamma=\bigcup\Gamma=\bigcup\{\dom[f]:f\in a\cup b\}$ is a proper subset of $\kappa$. Since $\kappa$ is a limit ordinal, $\gamma\subset\kappa$ implies $\gamma+1\in\kappa$.

If $a=\emptyset$ (resp. $b=\emptyset$), then the constant function $z=(\gamma+1)\times\{0\}$ (resp. $z=(\gamma+1)\times\{1\}$) is an element of $2^{<\kappa}$ that has the required property: $(a\times\{z\})\cup(\{z\}\times b)\subseteq\mathsf{U}_{2^{<\kappa}}$.

So, we assume that the sets $a,b$ are not empty. In this case, consider the set $u=\bigcup\{f\cap g:f\in a,\;g\in b\}$ and observe that it is a subset of $\gamma\times 2$.  Assuming that the set $u$ is not a function, we can find the smallest ordinal $\alpha$ such that $u{\restriction}_\alpha$ is not a function. It is easy to see that $\alpha$ is a successor ordinal and hence $\alpha=\beta+1$ for some ordinal $\beta\in\alpha$. It follows that $u{\restriction}\beta$ is a function but $u{\restriction}\alpha$ is not a function. Then both pairs $\langle\beta,0\rangle$  and $\langle\beta,1\rangle$ belong to the set $u$ and hence $\langle \beta,0\rangle\in f\cap g$ and $\langle \beta,1\rangle\in f'\cap g'$ for some functions $f,f'\in a$ and $g,g'\in b$. Then $\langle g,f'\rangle\in \mathsf{U}_{2^{<\kappa}}$ which contradicts $\langle f',g\rangle\in a\times b\subseteq\mathsf{U}_{2^{<\kappa}}$. This contradiction shows that $u$ is a function and hence $u\in 2^{<\kappa}$. 


Now three cases are possible.

1. There exist functions $f\in a$, $g\in b$ such that $\dom[u]\subset\dom[f]$, $\dom[u]\subset\dom[g]$ and $u=f{\restriction}\dom[u]=g{\restriction}\dom[u]$.
Taking into account that $f\cap g\subseteq u$, $\langle f,g\rangle\in a\times b\subseteq\mathsf{U}_{2^{<\kappa}}$,  we conclude that $f(\dom[u])=0$ and $g(\dom[u])=1$. 

If $u\notin a$, then define a function $z\in 2^{\gamma+1}\subset 2^{<\kappa}$ by the formula $$z(\alpha)=\begin{cases}f(\alpha)&\mbox{if $\alpha\le\dom[u]$};\\
1&\mbox{if $\dom[u]<\alpha\le\gamma$}.
\end{cases}
$$
If $u\in a$, then define a function $z\in 2^{\gamma+1}\subset 2^{<\kappa}$ by $$z(\alpha)=\begin{cases}g(\alpha)&\mbox{if $\alpha\le\dom[u]$};\\
0&\mbox{if $\dom[u]<\alpha\le\gamma$}.
\end{cases}
$$
It can be shown that the function $z$ has the required property: $(a\times\{z\})\cup(\{z\}\times b)\subseteq \mathsf{U}_{2^{<\kappa}}$.

2. There exist no functions $f\in a$ such that $\dom[u]\subset\dom[f]$ and $u=f{\restriction}_{\dom[u]}$. In this case define the function  $z\in 2^{\gamma+1}\subset 2^{<\kappa}$ by the formula $$z(\alpha)=\begin{cases}u(\alpha)&\mbox{if $\alpha<\dom[u]$};\\
0&\mbox{if $\dom[u]\le\alpha\le\gamma$};
\end{cases}
$$
and prove that $z$ has the required property: $(a\times\{z\})\cup(\{z\}\times b)\subseteq \mathsf{U}_{2^{<\kappa}}$. 

3. There exist no functions $g\in b$ such that $\dom[u]\subseteq\dom[g]$ and $u=g{\restriction}_{\dom[u]}$. In this case define the function  $z\in 2^{\gamma+1}\subset 2^{<\kappa}$ by the formula $$z(\alpha)=\begin{cases}u(\alpha)&\mbox{if $\alpha<\dom[u]$};\\
1&\mbox{if $\dom[u]\le\alpha\le\gamma$};
\end{cases}
$$
and prove that $z$ has the required property: $(a\times\{z\})\cup(\{z\}\times b)\subseteq \mathsf{U}_{2^{<\kappa}}$.
\end{proof}

\begin{remark} By Theorem~\ref{t:GCH}, under $(\mathsf{GCH})$, every infinite cardinal $\kappa$ satisfies the equality $|\kappa|=|2^{<\kappa}|$.
\end{remark}


The following theorem explains why universal orders are called universal.

\begin{theorem}\label{t:lin-univ-inj} Let $U$ be a universal linear order whose underlying class $\dom[U^\pm]$ is well-orderable. For a linear order $L$ the following conditions are equivalent.
\begin{enumerate}
\item[\textup{1)}] There exists an $L$-$U$-increasing function $\dom[L^\pm]\to\dom[U^\pm]$.
\item[\textup{2)}] There exists an injective function $\dom[L^\pm]\to\dom[U^\pm]$.
\end{enumerate}
\end{theorem}

\begin{proof} The implication $(1)\Ra(2)$ is trivial. To prove that $(2)\Ra(1)$, assume that there exists an injective function $J:\dom[L^\pm]\to\dom[U^\pm]$. The well-orderability of the class $\dom[U^\pm]$ and the injectivity of the function $J$ imply that the class $\dom[L^\pm]$ is well-orderable. If $\dom[L^\pm]$ is a proper class, then put $\kappa_1=\Ord$. If $\dom[L^\pm]$ is a set, then let $\kappa_1$ be the unique cardinal such that $|\kappa_1|=|\dom[L^\pm]|$ (the cardinal $\kappa_1$ exists since the set $\dom[L^\pm]$ is well-orderable). 
Using Theorem~\ref{t:wOrd}, find a bijective function $N_1:\kappa_1\to\dom[L^\pm]$. 

Now do the same with the order $U$. If $\dom[U^\pm]$ is a proper class, then put $\kappa_2=\Ord$ and if $\dom[U^\pm]$ is a set, then let $\kappa_2$ be the unique cardinal such that $|\kappa_2|=|\dom[U^\pm]$. 
Using Theorem~\ref{t:wOrd}, find a bijective function $N_2:\kappa_2\to\dom[U^\pm]$.

 Let $I$ be the class whose elements are injective functions $\varphi\subseteq \dom[L^\pm]\times\dom[U^\pm]$ such that $|\varphi|<|\kappa_1|$ and $\varphi$ is an isomorphism of the linear orders $L{\restriction}\dom[\varphi]$ and $U{\restriction}\rng[\varphi]$.

Let $\Phi:\kappa_1\times I\to I$ be the function assigning to each ordered pair $\langle \alpha,\varphi\rangle\in\kappa_1\times I$ the function $\Phi(\alpha,\varphi)\in I$ defined as follows. If $N_1(\alpha)\in\dom[\varphi]$, then $\Phi(\alpha,\varphi)=\varphi$. If $N_1(\alpha)\notin\dom[\varphi]$ then let $\beta(\alpha,\varphi)$ be the smallest ordinal $\beta\in\kappa_2$ such that the function $\varphi\cup\{\langle N_1(\alpha),N_2(\beta)\rangle\}$ is an element of the class $I$. Let us show that the ordinal $\beta(\alpha,\varphi)$ exists. Consider the sets $a=\{x\in\dom[\varphi]:\langle x,N_1(\alpha)\rangle\in L\}$ and $b=\{y\in\dom[\varphi]:\langle N_1(\alpha),y\rangle\in L\}$ and observe that $a\times b\subseteq L$. Taking into account that the function $\varphi$ is an isomorphism of the orders $L{\restriction}\dom[\varphi]$ and $U{\restriction}\rng[\varphi]$, we conclude that $\varphi[a]\times\varphi[b]\subseteq U\setminus\Id$. It follows that $|\varphi[a]\cup\varphi[b]|=|a\cup b|<|\kappa_1|\le|\kappa_2|\le|U|$. By the universality of the linear order $U$, 
there exists an element $z\in\dom[U^\pm]$ such that $(\varphi[a]\times\{z\})\cup(\{z\}\times\varphi[b])\subseteq U\setminus\Id$. Then the function $\varphi\cup\{\langle N_1(\alpha+1),z\rangle\}$ is an element of the class $I$. Since the function $N_2:\kappa_2\to\dom[U^\pm]$ is surjective, $z=N_2(\beta)$ for some ordinal $\beta\in\kappa_2$. Then $\varphi\cup\{\langle N_1(\alpha),N_2(\beta)\rangle\}\in I$, which completes the proof of the existence of the ordinal $\beta(\alpha,\varphi)$.

 Then put $\Phi(\alpha,\varphi)=\varphi\cup\{\langle N_1(\alpha),N_2(\beta(\alpha,\varphi))\rangle\}$.
Observe that the function $\Phi(\alpha,\varphi)$ has the properties: $\varphi\subseteq\Phi(\alpha,\varphi)\in I$ and $N_1(\alpha)\in\dom[\Phi(\alpha,\varphi)]$.

Finally, consider the function $F:\kappa_1\times \UU\to\UU$ assigning to every ordered pair $\langle \alpha,x\rangle\in\kappa_1\times \UU$ the set $$F(\alpha,x)=\begin{cases}\Phi(\alpha,\bigcup \rng[x])&\mbox{if $\bigcup \rng[x]\in I$};\\
\emptyset&\mbox{otherwise}.
\end{cases}
$$

By Recursion Theorem~\ref{t:Recursion}, there exists a transfinite sequence  $(\varphi_\alpha)_{\alpha\in\kappa_1}$ such that $\varphi_0=\emptyset$ and $\varphi_{\alpha}=F(\alpha,\{\langle \beta,\varphi_\beta\rangle\}_{\beta\in\alpha})$ for every ordinal $\alpha\in\kappa_1$. Using the Principle of Transfinite Induction, it can be proved that for every ordinal $\alpha\in\kappa_1$ the following conditions are satisfied:
\begin{itemize}
\item $\varphi_\alpha\in I$;
\item $\forall\beta\in\alpha\;(\varphi_\beta\subseteq \varphi_{\alpha})$;
\item $\varphi_{\alpha}=F(\alpha,\{\langle\beta,\varphi_\beta\rangle\}_{\beta\in\alpha+1})=\Phi(\alpha,\bigcup_{\beta\in\alpha}\varphi_\beta)$;
\item $N_1(\alpha)\in\dom[\varphi_{\alpha}]$.
\end{itemize}
Then $\varphi=\bigcup_{\alpha\in\kappa_1}\varphi_\alpha$ is a required $L$-$U$-increasing function from $\dom[L^\pm]$ to $\dom[U^\pm]$. 
\end{proof}

\begin{theorem}\label{t:lin-univ} Let $U$ be a universal linear order whose underlying class $\dom[U^\pm]$ is well-orderable. For a linear order $L$ the following conditions are equivalent:
\begin{enumerate}
\item[\textup{1)}]  there exists an  isomorphism $\dom[L^\pm]\to\dom[U^\pm]$ of the orders $L,U$;
\item[\textup{2)}] there exists a bijective  function $F:\dom[L^\pm]\to\dom[U^\pm]$ and the order $L$ is universal.
\end{enumerate}
\end{theorem}

\begin{proof} The implication $(1)\Ra(2)$ is trivial. To prove that $(2)\Ra(1)$, assume that there exists a bijective function $J:\dom[L^\pm]\to\dom[U^\pm]$ and the order $L$ is universal. If $\dom[U^\pm]$ is a set, then let $\kappa$ be a unique cardinal such that $|\kappa|=|\dom[U^\pm]|=|\dom[L^\pm]|$. If $\dom[U^\pm]$ is a proper class, then put $\kappa=\Ord$. 
Using Theorem~\ref{t:wOrd}, find  bijective functions $N_1:\kappa\to\dom[L^\pm]$ and $N_2:\kappa\to\dom[U^\pm]$. 

 Let $I$ be the class whose elements are injective functions $\varphi\subseteq \dom[L^\pm]\times\dom[U^\pm]$ such that $|\varphi|<|\kappa|$ and $\varphi$ is an isomorphism of the linear orders $L{\restriction}\dom[\varphi]$ and $U{\restriction}\rng[\varphi]$.

Let $\Phi:\kappa\times I\to I$ be the function assigning to each ordered pair $\langle \alpha,\varphi\rangle\in\kappa\times I$ the function $\Phi(\alpha,\varphi)\in I$ defined as follows. If $N_1(\alpha)\in\dom[\varphi]$, then $\Phi(\alpha,\varphi)=\varphi$. If $N_1(\alpha)\notin\dom[\varphi]$ then let $\beta(\alpha,\varphi)$ be the smallest ordinal $\beta\in\kappa$ such that the function $\varphi\cup\{\langle N_1(\alpha),N_2(\beta)\rangle\}$ is an element of the class $I$. Repeating the argument from the proof of Theorem~\ref{t:lin-univ-inj}, we can show that the ordinal $\beta(\alpha,\varphi)$ is well-defined.

By analogy define a function $\Psi:\kappa\times I\to I$ such that $\varphi\subseteq \Psi(\alpha,\varphi)\in I$ and $N_2(\alpha)\in\rng[\Psi(\alpha,\varphi)]$ for any $\langle \alpha,\varphi\rangle\in\kappa\times I$.

Finally, consider the function $F:\Ord\times \UU\to\UU$ assigning to every ordered pair $\langle \alpha,x\rangle\in\Ord\times \UU$ the set $$F(\alpha,x)=\begin{cases}\Psi(\alpha,\Phi(\alpha,\bigcup \rng[x]))&\mbox{if $\bigcup \rng[x]\in I$};\\
\emptyset&\mbox{otherwise}.
\end{cases}
$$

By Recursion Theorem~\ref{t:Recursion}, there exists a transfinite sequence  $(\varphi_\alpha)_{\alpha\in\kappa}$ such that $\varphi_0=\emptyset$ and $\varphi_{\alpha}=F(\alpha,\{\langle \beta,\varphi_\beta\rangle\}_{\beta\in\alpha})$ for every ordinal $\alpha\in\kappa$. By the Transfinite Induction it can be proved that for every ordinal $\alpha\in\kappa$ the following conditions are satisfied:
\begin{itemize}
\item $\varphi_\alpha\in I$;
\item $\forall\beta\in\alpha\;(\varphi_\beta\subseteq \varphi_{\alpha})$;
\item $\varphi_{\alpha}=\Psi(\alpha,\Phi(\alpha,\bigcup_{\beta\in\alpha}\varphi_\beta))$;
\item $N_1(\alpha)\in\dom[\varphi_{\alpha}]$ and $N_2(\alpha)\in\rng[\varphi_{\alpha}]$.
\end{itemize}
Then $\varphi=\bigcup_{\alpha\in\kappa}\varphi_\alpha$ is a required isomorphism of the linear orders $L$ and $U$.
\end{proof}  

\begin{corollary} Under $(\mathsf{GWO})$ every linear order $L$ admits an $L$-$\mathsf{U}_{2^{<\Ord}}$-increasing function $f:\dom[L]\to 2^{<\Ord}$.
\end{corollary}

\begin{exercise} Prove that a nonempty countable order $L$ is universal if and only if it has two properties:
\begin{itemize}
\item[\textup{1)}] for any elements $x<_L y$ of $\dom[L^\pm]$ there exists $z\in\dom[L^\pm]$ such that $x<_L z<_L y$;
\item[\textup{2)}] for any element $z\in\dom[L^\pm]$ there are $x,y\in\dom[L^\pm]$ such that $x<_L z<_L y$.
\end{itemize}
\end{exercise}
In this exercise we write $x<_L y$ instead of $\langle x,y\rangle\in L\setminus\Id$. 

For countable orders, Theorems~\ref{t:lin-univ-inj} and \ref{t:lin-univ} have the following corollaries, proved by Georg Cantor.

\begin{corollary}[Cantor] Any countable linear order $L$ admits $L$-$\mathsf{U}_{2^{<\w}}$-increasing function $f:\dom[L]\to 2^{<\w}$.
\end{corollary}

\begin{corollary}[Cantor]\label{l:Cantor-lin-univ} An order $L$ is isomorphic to the universal linear order $\mathsf{U}_{2^{<\w}}$ if and only if $L$ is nonempty, countable, and has two properties:
\begin{itemize}
\item[\textup{1)}] for any elements $x<_L y$ of $\dom[L^\pm]$ there exists $z\in\dom[L^\pm]$ such that $x<_L z<_L y$;
\item[\textup{2)}] for any element $z\in\dom[L^\pm]$ there are $x,y\in\dom[L^\pm]$ such that $x<_L z<_L y$.
\end{itemize}
\end{corollary}

\section{Cuts}\label{s:cuts}

In this section we study cuts of linear orders. 

Let $L$ be a linear order. An ordered pair of sets $\langle a,b\rangle$ is called an \index{cut}\index{linear order!cut of}{\em $L$-cut} if  $$a\cup b=\dom[L^\pm],\quad a\cap b=\emptyset\quad\mbox{and}\quad a\times b\subseteq L.$$ Let $\Cut(L)$ be the class of all $L$-cuts. For every $L$-cut $x=\langle a,b\rangle$ the sets $a$ and $b$ will be denoted by $\cev x$ and  $\vec x$, respectively.

In the following lemma, $(V_\alpha)_{\alpha\in\Ord}$ is von Neumann's cumulative hierarchy, studied in Section~\ref{s:vN}.

\begin{lemma}\label{l:sur1} Let $L$ be a linear order.
\begin{enumerate}
\item[\textup{1)}] The class $\Cut(L)$ exists and is a set.
\item[\textup{2)}] $\Cut(L)\ne\emptyset$ if and only if $L$ is a set.
\item[\textup{3)}] If $L\subseteq V_\alpha$ for some ordinal $\alpha$, then $\Cut(L)\subseteq V_{\alpha+3}$ and $L\cap\Cut(L)=\emptyset$.
\end{enumerate}
\end{lemma}

\begin{proof} The class $\Cut(L)$ exists by Theorem~\ref{t:class}. If $L$ is a proper class, then $\dom[L^\pm]$ is a proper class and then $\Cut(L)$ is the  empty set (since the proper class $\dom[L^\pm]$ cannot be represented as the union of two sets). If $L$ is a set, then the class $\Cut(L)$ is a set, being a subclass of the set $\mathcal P(\dom[L^\pm])\times\mathcal P(\dom[L^\pm])$. The set $\Cut(L)$ contains the $L$-cuts $\langle L,\emptyset\rangle$ and $\langle \emptyset,L\rangle$, witnessing that $\Cut(L)\ne\emptyset$.

 If $\Cut(L)$ is a nonempty set, then for any ordered pair $\langle a,b\rangle\in\Cut(L)$, the union $a\cup b=\dom[L^\pm]$ is a set and so is the linear order $L\subseteq\dom[L^\pm]\times\dom[L^\pm]$.

If $L\subseteq V_\alpha$ for some ordinal, then $\dom[L^\pm]=\dom[L]\cup\rng[L]\subseteq\bigcup\bigcup L\subseteq V_\alpha$ by the transitivity of the set $V_\alpha$, see Theorem~\ref{t:vN}(3). For every ordered pair $\langle a,b\rangle \in\Cut(L)$ we have $a,b\subseteq \dom[L^\pm]\subseteq V_\alpha$ and hence $a,b\in \mathcal P(V_\alpha)=V_{\alpha+1}$ and $\langle a,b\rangle\in V_{\alpha+3}$. Therefore, $\Cut(L)\subseteq V_{\alpha+3}$ and $\Cut(L)\in V_{\alpha+4}$.

Assuming that $L\cap \Cut(L)\ne\emptyset$, we would find an ordered pair $\langle a,b\rangle\in \Cut(L)\cap L$ and conclude that $a,b\in \dom[L^\pm]=a\cup b$. Taking into account that $a,b\in\VV$ and the relation $\E{\restriction}\VV$ is well-founded (see Theorem~\ref{t:vN}(5)), we conclude that $a\notin a$ and $b\notin b$. Then $a,b\in a\cup b$ implies $a\in b\in a$, which contradicts the well-foundedness of the relation $\E{\restriction}\VV$. This contradiction shows that $L\cap\Cut(L)=\emptyset$.
\end{proof}

The \index{cut extension}\index{linear order!cut extension of}{\em cut extension} of a linear order $L$ is the relation  
\begin{multline*}
\Xi(L)=L\cup \{\langle\langle a,b\rangle,\langle a',b'\rangle\rangle\in\Cut(L)\times\Cut(L):a\subseteq a'\}\cup\\ \{\langle x,\langle a,b\rangle\rangle\in \dom[L^\pm]\times\Cut(L):x\in a\}\cup\{\langle \langle a,b\rangle,y\rangle\in\Cut(L)\times \dom[L^\pm]:y\in b\}
\end{multline*}

\begin{theorem}\label{t:cut2} For any linear order $L\subseteq\VV$ its cut extension $\Xi[L]$ has the following properties:
\begin{enumerate}
\item[\textup{1)}] $L\subseteq \Xi(L)\subseteq\VV$.
\item[\textup{2)}] $\dom[\Xi(L)^\pm]=\dom[L^\pm]\cup\Cut(L)$.
\item[\textup{3)}] If $L$ is a set, then $L\ne \Xi(L)$.
\item[\textup{4)}] The relation $\Xi(L)$ is transitive.
\item[\textup{5)}] The relation $\Xi(L)$ is antisymmetric. 
\item[\textup{6)}] The relation $\Xi(L)$ is a linear order.
\item[\textup{7)}] If $L$ is reflexive, then $\Xi(L)$ is a reflexive linear order.
\end{enumerate}
\end{theorem}

\begin{proof} The first three statements follow from the definition of the relation $\Xi(L)$ and Lemma~\ref{l:sur1}.
\smallskip

4. To prove that the relation $\Xi(L)$ is transitive, fix any elements $x,y,z\in \dom[\Xi(L)^\pm]=\dom[L^\pm]\cup\Cut(L)$ with $\langle x,y\rangle,\langle y,z\rangle\in\Xi(L)$. We need to check that $\langle x,z\rangle\in\Xi(L)$. Since $y\in \dom[L^\pm]\cup\Cut(L)$, two cases are possible.

4a. First we assume that $y\in\dom[L^\pm]$. This case has four subcases.

4a1. If $x\in\dom[L^\pm]$ and $z\in \dom[L^\pm]$, then $\langle x,z\rangle\in L\subseteq \Xi(L)$ by the transitivity of the linear order $L$.

4a2. If $x\in \dom[L^\pm]$ and $z\in \Cut(L)$, then $z=\langle a,b\rangle$ for some 
$L$-cut $\langle a,b\rangle$ such that $y\in a$ (the latter inclusion follows from $\langle y,z\rangle\in\Xi(L)$). We claim that $x\in a$, too. In the opposite case $x\in b$ and then $\langle y,x\rangle\in a\times b\subseteq L$ and $x=y\in a$ by the antisymmetricity of $L$. But the inclusion $x\in a$ contradicts our assumption. This contradiction shows that $x\in a$ and hence $\langle x,z\rangle=\langle x,\langle a,b\rangle\rangle\in \Xi(L)$.

4a3. If $x\in \Cut(L)$ and $z\in \dom[L^\pm]$, then $x=\langle a,b\rangle$ for some $L$-cut $\langle a,b\rangle$ with  $y\in b$ (the  latter inclusion follows from $\langle x,y\rangle\in \Xi(L)$). We claim that $z\in b$. In the opposite case $z\in a$ and then $\langle z,y\rangle\in a\times b\subseteq L$. On the other hand, $\langle y,z\rangle\in\Xi(L)\cap(\dom[L^\pm]\times\dom[L^\pm])=L$ and the antisymetry of $L$ imply that $z=y\in b$, which contradict our assumption. This contradiction shows that $z\in b$ and hence $\langle x,z\rangle=\langle\langle a,b\rangle,z\rangle\in\Xi(L)$.

4a4. If $x\in\Cut(L)$ and $z\in\Cut(L)$, then $x=\langle a,b\rangle$ and $z=\langle a',b'\rangle$ for some $L$-cuts $\langle a,b\rangle$ and $\langle a',b'\rangle$. It follows from $\langle x,y\rangle\in \Xi(L)$ and $\langle y,z\rangle\in\Xi(L)$ that $y\in b$ and $y\in a'$. To show that $\langle x,z\rangle\in\Xi(L)$ we need to check that $a\subseteq a'$. Given any element $\alpha\in a$, observe that $\langle \alpha,y\rangle\in a\times b\subseteq L$. Assuming that $\alpha\notin a'$, we conclude that $\alpha\in b'$ and hence $\langle y,\alpha\rangle\in a'\times b'\subseteq L$. The antisymmetry of $L$ implies $\alpha=y\in a'$, which contradicts our assumption. This contradiction shows that $\alpha\in a'$ and hence $a\subseteq a'$ and finally $\langle x,z\rangle=\langle\langle a,b\rangle,\langle a',b'\rangle\rangle\in \Xi(L)$.
\smallskip

Now consider the second case.

4b. $y\in\Cut(L)$. In this case $y=\langle a,b\rangle$ for some $L$-cut. This case also has four subcases.

4b1. $x,z\in\dom[L^\pm]$. In this subcase, $\langle x,y\rangle\in\Xi(L)$ and $\langle y,z\rangle\in\Xi(L)$ imply that $x\in a$ and $z\in b$. Then $\langle x,z\rangle\in a\times b\subseteq L$.

4b2. $x\in\dom[L^\pm]$ and $z\in\Cut(L)$. In this subcase $x\in a$ and $z=\langle a',b'\rangle$ for some $L$-cut $\langle a',b'\rangle$ such that $a\subseteq a'$ (the latter embedding follows from $\langle y,z\rangle\in\Xi(L)$). Then $x\in a\subseteq a'$ implies $\langle x,z\rangle=\langle x,\langle a',b'\rangle\rangle\in\Xi(L)$.

4b3. $x\in\Cut(L)$ and $z\in\dom[L^\pm]$. In this subcase $x=\langle a',b'\rangle$ for some $L$-cut $\langle a',b'\rangle$. It follows from $\langle x,y\rangle\in\Xi(L)$ and $\langle y,z\rangle\in\Xi(L)$ that $a'\subseteq a$ and $z\in b$. Taking into account that $a'\subseteq a$ and $a\cup b=a'\cup b'=\dom[L^\pm]$, we conclude that $z\in b\subseteq b'$ and hence $\langle x,z\rangle=\langle\langle a',b'\rangle,z\rangle\in \Xi(L)$.

4b4. $x,z\in\Cut(L)$. In this subcase $x=\langle a',b'\rangle$ and $y=\langle a'',b''\rangle$ for some $L$-cuts $\langle a',b'\rangle$ and $\langle a'',b''\rangle$. It follows from $\langle x,y\rangle\in\Xi(L)$ and $\langle y,z\rangle\in\Xi(L)$ that $a'\subseteq a\subseteq a''$ and hence $a'\subseteq a''$, which means that $\langle x,z\rangle=\langle\langle a',b'\rangle,\langle a'',b''\rangle\rangle\in \Xi(L)$.

Therefore, we have considered all 8 cases and thus proved the transitivity of the relation $\Xi(L)$.
\smallskip

5. To show that the relation $\Xi(L)$ is antisymmetric, take any elements $x,y\in\dom[\Xi(L)^\pm]=\dom[L^\pm]\cup \Cut(L)$ and assume that $\langle x,y\rangle\in\Xi(L)$ and $\langle y,x\rangle\in\Xi(L)$. By Lemma~\ref{l:sur1}(3), the sets $\dom[L^\pm]$ and $\Cut(L)$ are disjoint. Now we consider four possible cases.

5a. If $x,y\in\dom[L^\pm]$ then  $\langle x,y\rangle,\langle y,z\rangle\in\Xi(L)\cap(\dom[L^\pm]\times\dom[L^\pm])=L$ and $x=y$ by the antisymmetricity of $L$.

5b. If $x\in\dom[L^\pm]$ and $y\in\Cut(L)$, then $y=\langle a,b\rangle$ for some $L$-cut $\langle a,b\rangle$ and the inclusions $\langle x,y\rangle,\langle y,x\rangle\in\Xi(L)$ imply $x\in a$ and $x\in b$ which is not possible as $a\cap b=\emptyset$.

5c. By analogy we can show that the case $x\in\Cut(L)$, $y\in\dom[L^\pm]$ is incompatible with $\langle x,y\rangle,\langle y,x\rangle\in\Xi(L)$.

5d. If $x,y\in\Cut(L)$, then $x=\langle a,b\rangle$ and $y=\langle a',b'\rangle$ for some $L$-cuts $\langle a,b\rangle$ and $\langle a',b'\rangle$. The inclusions $\langle x,y\rangle,\langle y,x\rangle\in\Xi(L)$ imply $a\subseteq a'\subseteq a$ and hence $a=a'$ and $b=\dom[L^\pm]\setminus a=\dom[L^\pm]\setminus a'=b'$, which implies $x=\langle a,b\rangle=\langle a',b'\rangle=y$.
\smallskip

6. The statements 4,5 imply that the relation $\Xi(L)$ is a partial order. To show that it is a linear order, we should prove that for any $x,y\in\dom[\Xi(L)^\pm]=\dom[L^\pm]\cup\Cut(L)$ we have $\langle x,y\rangle\in\Xi(L)^\pm\cup\Id$. Four cases are possible.

6a. If $x,y\in\dom[\Xi(L)^\pm]$, then $\langle x,y\rangle\in L^\pm\cup\Id$ since $L$ is a linear order.

6b. If $x\in\dom[\Xi(L)^\pm]$ and $y\in\Cut(L)$, then $y=\langle a,b\rangle$ for some $L$-cut $\langle a,b\rangle$. Since $x\in\dom[L^\pm]=a\cup b$, either $x\in a$ and then $\langle x,y\rangle\in\Xi(L)$ or $x\in b$ and then $\langle y,x\rangle\in\Xi(L)$.

6c. By analogy with (6b) we can treat the case $x\in\Cut(L)$ and $y\in\dom[L^\pm]$.

6d. If $x,y\in\Cut(L)$, then $x=\langle a,b\rangle$  and $y=\langle a',b'\rangle$ for some $L$-cuts $\langle a,b\rangle$ and $\langle a',b'\rangle$. We claim that either $a\subseteq a'$ or $a'\subseteq a$. In the opposite case we can find elements $x\in a\setminus a'$ and $x'\in a'\setminus a$. It follows from $a\cup b=\dom[L^\pm]=a'\cup b'$ that $x\in b'$ and $x'\in b$. Then $\langle x,x'\rangle\in a\times b\subseteq L$ and $\langle x',x\rangle\in a'\times b'\subseteq L$. The antisymmetricity of $L$ ensures that $x=x'\in a\cap b=\emptyset$ which is a desired contradiction showing that  $a\subseteq a'$ or $a'\subseteq a$ and hence $\langle x,y\rangle\in\Xi(L)$ or $\langle y,x\rangle\in\Xi(L)$.
\smallskip

7. If the linear order $L$ is reflexive, then the linear order $\Xi(L)$ is reflexive by the definition of $\Xi(L)$.
\end{proof}

\section{Gaps} 

Let $L$ be a linear order. An $L$-cut $\langle a,b\rangle$ is called an \index{gap}\index{linear order!gap in}{\em $L$-gap} if the sets $a,b$ are not empty and for any $x\in a$ and $y\in b$ there are elements $x'\in a$ and $y'\in b$ such that $x<_L x'$ and $y'<_L y$. By $\Gap(L)$ we denote the set of $L$-gaps. It is a subset of the set $\Cut(L)$ of $L$-cuts.

The linear order $$\Theta(L)=\{\langle x,y\rangle\in \Xi(L):x,y\in\dom[L^\pm]\cup\Gap(L)\}=\Xi(L)\restriction (\dom[L^\pm]\cup\Gap(L))$$is called the \index{gap extension}\index{linear order!gap extension of}{\em gap extension} of the linear order. 

\begin{theorem}\label{t:gap-ext} For any linear order $L\in\VV$ its gap extension has the following properties:
\begin{enumerate}
\item[\textup{1)}] $L\subseteq \Theta(L)\subseteq \Xi(L)\in\VV$.
\item[\textup{2)}] $\dom[\Theta(L)^\pm]=\dom[L^\pm]\cup\Gap(L)$.
\item[\textup{3)}] The relation $\Theta(L)$ is a linear order.
\item[\textup{4)}] If $L$ is reflexive, then $\Theta(L)$ is a reflexive linear order.
\item[\textup{5)}] $\Gap(\Theta(L))=\emptyset$ and hence $\Theta(\Theta(L))=\Theta(L)$.
\end{enumerate}
\end{theorem}

\begin{proof} The first four statements follow from Lemma~\ref{l:sur1} and Theorem~\ref{t:cut2}. It remains to prove that the gap extension $\Theta(L)$ of any linear order $L\in\VV$ has no $\Theta(L)$-gaps. Let $X=\dom[L^\pm]$ be the underlying set of the linear order $L$.

To derive a contradiction, assume that the linear order $\Theta(L)$ has a $\Theta(L)$-gap $\langle A,B\rangle$.  Consider the sets $a=A\cap X$ and $b=B\cap X$ and observe that $a\cap b\subseteq A\cap B=\emptyset$, $a\cup b=X\cap(A\cup B)=X$ and $a\times b=(A\times B)\cap (X\times X)\subseteq\Theta(L)\cap (X\times X)=L$ by Lemma~\ref{l:sur1}. Therefore, $\langle a,b\rangle$ is an $L$-cut. We claim that $\langle a,b\rangle$ is an $L$-gap. 

First we prove that the sets $a,b$ are not empty. To derive a contradiction, assume that the set $a=A\cap X$ is empty. 
Since the left set $A\subseteq X\cup \Gap(L)$ of the $\Theta(L)$-gap $\langle A,B\rangle$ is not empty, it contains some $L$-gap $\langle u,v\rangle$, whose left side $u$ is a non-empty subset of $X$. Then for every $x\in u$ we have $\langle x,\langle u,v\rangle\rangle\in \Theta(L)$. On the other hand, $x\in u\subseteq a\cup b=\emptyset \cup b=b\subseteq B$ and hence $\langle\langle u,v\rangle, x\rangle\in A\times B\subseteq\Theta(L)$, which implies $x=\langle u,v\rangle$ by the antisymmetricity of the linear order $\Theta(L)$. Then $x=\langle u,v\rangle\in X\cap\Gap(L)=\emptyset$, which is a contradiction showing that $a\ne\emptyset$. By analogy we can prove that $b\ne\emptyset$. 

To show that $\langle a,b\rangle$ is an $L$-gap, fix any elements $x\in a$ and $y\in b$. We need to find elements $x'\in a$ and $y'\in b$ such that $x<_L x'$ and $y'<_L y$.
Since $\langle A,B\rangle$ is a $\Theta(L)$-gap, for the elements $x\in a\subseteq A$ and $y\in b\subseteq B$ there are elements $x''\in A$ and $y''\in B$ such that $\langle x,x''\rangle\in\Theta(L)\setminus\Id$ and $\langle y'',y\rangle\in\Theta(L)\setminus\Id$. If $x''\in X$, then $x'=x''\in A\cap X=a$ is a required element of $a$ with $\langle x,x'\rangle\in (X\times X)\cap\Theta(L)\setminus\Id=L\setminus\Id$. If $x''\notin X$, then $x''=\langle u,v\rangle$ is an $L$-gap. It follows from $\langle x,\langle u,v\rangle\rangle=\langle x,x''\rangle\in\Theta(L)$ that $x\in u$. Since $\langle u,v\rangle$ is an $L$-gap, there exists an element $x'\in u$ such that $x<_L x'$. Then $\langle x',x''\rangle=\langle x',\langle u,v\rangle\rangle\in\Theta(L)$ and $x''\in A$ imply $x'\in A\cap X=a$. By analogy we can prove that the set $b$ contains an element $y'$ such that $y'<_L y$. 

Therefore, $\langle a,b\rangle$ is an $L$-gap and $\langle a,b\rangle\in\Gap(L)\in A\cup B$. If $\langle a,b\rangle\in A$, then by the definition of a $\Theta(L)$-gap, we can find an element $g\in A$ such that $\langle \langle a,b\rangle,g\rangle\in\Theta(L)\setminus \Id$. If $g\in X$, then $g\in A\cap X=a$ and hence $\langle g,\langle a,b\rangle\rangle\in\Theta(L)$, which contradicts the antisymmetricity of $\Theta(L)$. So, $g\notin X$ and hence $g=\langle u,v\rangle$ is an $L$-gap. Then $\langle \langle a,b\rangle,g\rangle\in\Theta(L)\setminus \Id$ and the definition of the linear order $\Theta(L)$ implies that $a\subset u$. Choose any $y\in u\setminus a\subseteq X\setminus a=b\subseteq B$ and conclude that $\langle g,y\rangle\in A\times B\subseteq \Theta(L)$. On the other hand, $y\in u$ implies $\langle y,g\rangle=\langle y,\langle u,v\rangle\rangle\in\Theta(L)$ and hence $g=y\in X$ by the antisymmetricity of $\Theta(L)$, which contradicts our assumption. By analogy we can prove that the inclusion $\langle a,b\rangle\in B$ leads to a contradiction. 
\end{proof}

\begin{definition} A linear order $L$ is called \index{gapless linear order}\index{linear order!gapless}{\em gapless} if $\Gap(L)=\emptyset$. 
\end{definition}

By Theorem~\ref{t:gap-ext}, a linear order $L\subseteq\VV$ is gapless if and only if $L=\Theta(L)$.

\begin{proposition}\label{p:gap<=>bc} A linear order $L$ \textup{(}on a set $X=\dom[L^\pm]$\textup{)} is gapless if \textup{(}and only if\textup{)} it is boundedly complete.
\end{proposition}

\begin{proof} Let $L$ be a boundedly complete linear order. To derive a contradiction, assume that $L$ has an $L$-gap $\langle a,b\rangle$. Then the sets $a$, $b$ are not empty and hence the set $a$ is upper $L$-bounded. By the bounded completeness of $L$, the set $a$ has $\sup_L(a)$. If $\sup_L(a)\in a$, then $\sup_L(a)$ is the $L$-greatest element of $a$. If $\sup_L(a)\in b$, then $\sup_L(a)$ is the $L$-least element of $b$. In both cases, $\langle a,b\rangle$ is not an $L$-gap.

Now assume that the linear order $L$ is gapless and $X=\dom[L^\pm]$ is a set. Assuming that  $L$ is not boundedly complete and applying Proposition~\ref{p:bcomp}, we conclude that $\dom[L^\pm]$ contains an upper $L$-bounded subset $A\subseteq X$ that has no $\sup_L(A)$. Then $A$ does not have an $L$-maximal element. Since $A$ is upper $L$-bounded, then set $b=\{x\in X:A\times\{x\}\subseteq L\cup\Id\}$ is not empty.
Consider the set $a=L^{-1}[A]$. Since $L$ is a linear order, $a\cup b=\dom[L^\pm]$. Since $L$ is gapless, the ordered pair $\langle a,b\rangle$ is not an $L$-gap and hence the set $b$ has an $L$-minimal element which is equal to $\sup_L(A)$.
\end{proof}

\begin{definition} Let $\kappa$ be a cardinal. A linear order $L$ is called \index{linear order!$\kappa$-universal}{\em $\kappa$-universal} if there exists a subset $U\subseteq \dom[L^\pm]$ of cardinality $|U|=\kappa$ such that the order $L\restriction U$ is universal and for any ordered pair $\langle x,y\rangle\in L\setminus\Id$ there are elements $u,v,w\in U$ such that $\{\langle u,x\rangle,\langle x,v\rangle, \langle v,y\rangle,\langle y,w\rangle\}\subseteq L\setminus\Id$.
\end{definition}

\begin{remark} Each universal linear order $L$ with $|\dom[L^\pm]|=\kappa$ is $\kappa$-universal.
\end{remark}

It can be shown that for any infinite cardinal $\kappa$ and universal linear order $L\in\VV$ with $|\dom[L^\pm]|=\kappa$, its gap extension $\Theta(L)$ is $\kappa$-universal. The following theorem  shows that all such orders are pairwise isomorphic.

\begin{theorem}\label{t:R-iso} Let $\kappa$ be an infinite cardinal. Any $\kappa$-universal gapless linear orders are isomorphic.
\end{theorem}

\begin{proof} Fix two $\kappa$-universal gapless linear orders $L_1,L_2$. For every $i\in\{1,2\}$, the underlying set $X_i=\dom[L_i^\pm]$ of the order $L_i$ contains a subset $D_i\subseteq X_i$ of cardinality $|D_i|=|\kappa|$ such that the linear order $U_i=L_i{\restriction}D_i$ is universal and for any  ordered pair $\langle x,y\rangle\in L_i\setminus\Id$ there are elements $u,v,w\in D_i$ such that $\{\langle u,x\rangle,\langle x,v\rangle, \langle v,y\rangle,\langle y,w\rangle\}\subseteq L_i\setminus\Id$.

 By Theorem~\ref{t:lin-univ}, there exists an isomorphism $f:D_1\to D_2$ of the universal  linear orders $U_1$ and $U_2$. Now we show that $f$ admits a unique extension $F$ to an isomorphism of the linear orders $L_1$ and $L_2$.
 
For every $x\in X_1\setminus D_1$, consider the $U_1$-gap $\langle \cev x,\vec x\rangle$ consisting of the sets $\cev x=\{y\in D_1:\langle y,x\rangle\in L_1\}$ and $\vec x=\{y\in D_1:\langle x,y\rangle\in L_1\}$. Since $f$ is an order isomorphism, the pair $\langle f[\cev x],f[\vec x]\rangle$ is a $U_2$-gap. Consider the sets ${\downarrow}f[\cev x]=L_2^{-1}[f[\cev x]]$ and ${\uparrow}f[\vec x]=L_2[f[\vec x]]$. Since the order $L_2$ is gapless, the ordered pair $\langle {\downarrow}f[\cev x],{\uparrow}f[\vec x]\rangle$ is not an $L_2$-gap, which implies that the complement $X_2\setminus( {\downarrow}f[\cev x]\cup {\uparrow}f[\vec x])\subseteq X_2\setminus D_2$ is not empty. The choice of the set $D_2$ ensures that this complement contains a unique point. We  denote this unique point by $F(x)$. Therefore, we have constructed an extension $F=f\cup\{\langle x,F(x)\rangle:x\in X_1\setminus D_1\}$ of the function $f$ to a function $F:X_1\to X_2$. It can be shown that the function $F$ is injective and $L_1$-$L_2$-increasing. By analogy we can extend the function $f^{-1}:D_2\to D_1$ to an injective $L_2$-$L_1$-increasing function $G:X_2\to X_1$. Then the composition $G\circ F:X_1\to X_1$ is an $L_1$-$L_1$-increasing function such that $G\circ F(x)=x$ for every $x\in D_1$. The $L_1$-density of the set $D_1$ implies that $G\circ F(x)=x$ for all $x\in X_1$. By analogy we can prove that $F\circ G$ is the identity function of the set $X_2$. Therefore, $F:X_1\to X_2$ is an isomorphism of the linear orders $L_1$ and $L_2$. 
\end{proof}

\begin{corollary}[Cantor]\label{t:wR-iso}  Any $\w$-universal gapless linear orders are isomorphic.
\end{corollary}

\begin{theorem}\label{t:order-continuum} Let $\kappa$ be a regular cardinal with $|\kappa|=|2^{<\kappa}|$.  Any $\kappa$-universal gapless linear order $L$ has cardinality $$|L|=|\dom[L^\pm]|=|2^\kappa|.$$
\end{theorem}

\begin{proof} By Theorem~\ref{t:univ-ord}, the linear order $\mathsf{U}_{2^{<\kappa}}$ is universal. 
By Theorem~\ref{t:R-iso}, the gap extension $\Theta(\mathsf{U}_{2^{<\kappa}})$ of $\mathsf{U}_{2^{<\kappa}}$ is isomorphic to the order $L$. Therefore, 
\begin{multline*}
|\dom[L^\pm]|=|\dom[\Theta(\mathsf{U}_{2^{<\kappa}})^\pm]|=|\dom[\mathsf{U}_{2^{<\kappa}}]|+|\Gap(\mathsf{U}_{2^{<\kappa}})|\le|2^{<\kappa}|+|\mathcal P(2^{<\kappa})\times\mathcal P(2^{<\kappa})|=\\
|\kappa|+|\mathcal P(\kappa)\times\mathcal P(\kappa)|\le |2^\kappa|+2^{|\kappa|+|\kappa|}=|2^\kappa|
\end{multline*}
and $|L|\le|\dom[L^\pm]\times\dom[L^\pm]|\le|2^\kappa\times 2^\kappa|=|2^\kappa|.$

To prove that $|\dom[L^\pm]|\ge |2^\kappa|$ and $|L|\ge |2^\kappa|$, consider the  set  $F=(2^{<\kappa})^\kappa$, endowed with the lexicographic linear order
$$R=\{\langle f,g\rangle\in F\times F :\exists \alpha\in\kappa\;(f{\restriction}_\alpha=g{\restriction}_\alpha\;\wedge\;\langle f(\alpha) ,g(\alpha)\rangle\in\mathsf{U}_{2^{<\kappa}})\}.$$
The set $$D=\{f\in F:|\{\alpha\in\kappa:f(\alpha)\ne \emptyset\}|<\kappa\}$$witnesses that the order $R$ is $\kappa$-universal. Then its gap extension $\Theta(R)$ is gapless and $\kappa$-universal. By Theorem~\ref{t:R-iso}, the orders $\Theta(R)$ and $L$ are isomorphic. Then $|\dom[L^\pm]|=|\dom[\Theta(R)]|\ge|\dom[R]|=|F|\ge|2^\kappa|$ and $|L|=|\Theta(R)|\ge|R|\ge|2^\kappa|$. 
Therefore, we have the inequalities
$$|2^\kappa|\le|\dom[L^\pm]|\le |2^\kappa|\quad\mbox{and}\quad|2^\kappa|\le|L|\le |2^\kappa|.$$
By Theorem~\ref{t:CBDS}, $|L|=|\dom[L^\pm]|=|2^\kappa|$.
\end{proof}

\newpage

\part{Numbers}\label{p:numbers}

\rightline{\em They [Pythagoreans] thought they found}

\rightline{\em in numbers more than in fire, earth, or water;}

\rightline{\em  all things seemed in their whole nature to be}





\rightline{\em assimilated to numbers, while numbers seemed}

\rightline{\em to be the first things in the whole of nature,}

\rightline{\em they supposed the elements of numbers}

\rightline{\em to be the elements of all things,}

\rightline{\em and the whole heaven to be}

\rightline{\em a musical scale and a number.}
\smallskip

\rightline{Aristotle, {\sf ``Metaphysics''}, 350 BC.}
\bigskip

\rightline{\em Die ganzen Zahlen hat der liebe Gott gemacht,}

\rightline{\em alles andere ist Menschenwerk.}

\rightline{Leopold Kronecker\footnote{{\bf Task:} Read at ({\tt https://mathshistory.st-andrews.ac.uk/Biographies/Kronecker/}) about Kronecker and his attitude to Set Theory.}}
\bigskip

The aim of this part is to introduce the sets $\IN,\IZ,\IQ,\IR,\IC$ which are of crucial importance for whole mathematics. The elements of those sets are numbers: natural, integer, rational, real, complex, respectively. In fact, the set $\IN$ of nonzero natural numbers has been introduced in Section~\ref{s:axioms} as the set $\w\setminus\{0\}$.\index{natural number}\index{number!natural}\index{$\mathbb N$}\index{set!$\mathbb N$}

\section{Integer numbers}

Theorem~\ref{t:add-ord}(5) implies that for any natural numbers $n\le m$ there exists a unique natural number $k$ such that $m=n+k$. This natural number $k$ is denoted by $m-n$ and called the result of subtraction of $n$ from $m$.  If $m<n$, then $m-n$ is not a natural number but is a negative integer. But what is a negative integer? For example, what is $\minus 1$? It should be equal to $1-2$, but also to $2-3$ and $3-4$ and so on. So, it is natural to define the negative integer $\minus 1$ as the set of ordered pairs $\{\langle 0,1\rangle,\langle 1,2\rangle,\langle 2,3\rangle,\dots\}$. The simplest  (in the sense of von Neumann hierarchy) element of this set is the  ordered pair $\langle 0,1\rangle$. We take this simplest pair $\langle 0,1\rangle$ to represent the negative integer $\minus 1$. As a result, the negative number $\minus 1$ becomes a relatively simple set $\langle 0,1\rangle=\{\{0\},\{0,1\}\}=\{\{\emptyset\},\{\emptyset,\{\emptyset\}\}\}=\{1,2\}$, which is an element of the set $V_4$ of the von Neumann hierarchy. On the other hand, the set of ordered pairs $\{\langle n,m\rangle\in\w\times\w:n+1=m\}$ which can be taken as an alternative definition of $\minus 1$ appears only at the stage $V_{\w+1}$ of von Neumann hierarchy.

Realizing this idea, for every nonzero natural number $n$ define the negative integer $\minus n$ as the ordered pair $\langle 0,n\rangle$. Then $\minus\IN=\{\langle 0,n\rangle:n\in\IN\}=\{0\}\times\IN$ is the set of {\em negative integers} and the union $$\IZ=\minus\IN\cup\{0\}\cup \IN$$ is the {\em set of integer numbers}. \index{negative integer}\index{integer number}\index{number!integer}\index{$\mathbb Z$}\index{set!$\mathbb Z$}

\begin{exercise} Show that the sets $\minus\IN$ and $\w$ are disjoint.
\smallskip

\noindent{\em Hint}: A unique two-element set in $\w$ is $2=\{\emptyset,\{\emptyset\}\}$, which does not belong to the class $\ddot\UU\supset \minus\IN$.
\end{exercise} 

The reflexive linear order $\mathbf S{\restriction}\w$ on the set $\w$ can be extended to  the reflexive linear order  $$\le_\IZ\;=(\minus\IN\times\w)\cup\{\langle n,m\rangle:n\subseteq m\}\cup\{\langle\minus n,\minus m\rangle:m\subseteq n\}$$
on the set $\IZ=\dom[\le_\IZ]=\rng[\le_\IZ]=\dom[\le_\IZ^\pm]$. The irreflexive order
$$<_\IZ\;=\;\le_\IZ\!\setminus\Id$$ is called the {\em strict linear order} on $\IZ$.

\begin{exercise} Show that $\minus\IN=\cev{\le}_\IZ(0)=\cev{<}_\IZ(0)$.
\end{exercise}

Now we introduce some arithmetic operations on the set $\IZ$. The first one is {\em additive inversion} $\minus:\IZ\to\IZ$, $\minus:z\mapsto \minus z$, defined by the formula
$$\minus z=\begin{cases}\langle 0,z\rangle&\mbox{if $z\in\IN$};\\
0&\mbox{if $z=0$};\\
n&\mbox{if $z=\langle 0,n\rangle$ for some $n\in\IN$}.
\end{cases}
$$
The definition of the function $\minus$ implies that $\minus (\minus z)=z$ for any integer number $z$.

\begin{exercise} Show that $\forall x,y\in\IZ\;(x\le_\IZ y\;\Leftrightarrow\;\minus y\le_\IZ\minus x)\;\wedge\;(x<_\IZ y\;\Leftrightarrow\;\minus y<_\IZ\minus x)$.
\end{exercise}

Now we shall extend the operation of addition to integer numbers. Since the addition of natural numbers is already defined, we need to define the sum $x+y$ only for  pairs $$\langle x,y\rangle\in(\IZ\times\IZ)\setminus(\w\times\w)=((\minus\IN)\times(\minus \IN))\cup(\w\times(\minus\IN))\cup((\minus \IN)\times\w).$$ This is done by the formulas
$$
(\minus m)+(\minus n)=\minus(m+n)\quad\mbox{and}\quad m+(\minus n)=(\minus n)+m=\begin{cases}m-n&\mbox{if $n\le m$},\\
\minus(n-m)&\mbox{if $m\le n$},
\end{cases}
$$
for any $n,m\in\w$.

\begin{exercise}\label{ex:add-Z} Check that the addition $+:\IZ\times\IZ\to\IZ$ has the following properties for every integer numbers $x,y,z$:
\begin{enumerate}
\item $(x+y)+z=x+(y+z)$;
\item $x+y=y+x$;
\item $x+0=x$;
\item $x+(\minus x)=0$;
\item $x<_\IZ y\;\Leftrightarrow\;x+z<_\IZ y+z$.
\end{enumerate}
\smallskip

\noindent{\em Hint:} Prove these properties for the isomorphic copy of $\IZ$, which is the quotient set $\tilde \IZ=(\w\times\w)/Z$ of $\w\times\w$ by the equivalence relation 
$$Z=\{\langle\langle k,l\rangle,\langle m,n\rangle\rangle\in(\w\times\w)\times(\w\times\w):k+n=m+l\}.$$ For every ordered pair $\langle m,n\rangle\in\w\times\w$ its equivalence class $Z^\bullet(m,n)$ represents the integer number $m-n$. For two equivalence classes $Z^\bullet(m,n),Z^\bullet(k,l)$ their sum $Z^\bullet(m,n)+Z^\bullet(k,l)$ is defined as the equivalence class $Z^\bullet(m+k,n+l)$. 
\end{exercise}

\smallskip


Next, we extend the multiplication to integer numbers. Since the multiplication of natural numbers is already defined, we need to define the product $x{\cdot}y$ only for  pairs $$\langle x,y\rangle\in (\IZ\times\IZ)\setminus(\w\times\w)=((\minus\IN)\times(\minus \IN))\cup(\w\times(\minus\IN))\cup((\minus \IN)\times\w).$$ This is done by the formulas
$$
(\minus m)\cdot(\minus n)=m\cdot n\quad\mbox{and}\quad
m\cdot(\minus n)=(\minus n)\cdot m=\minus(n\cdot m)
$$
for any $n,m\in\w$.

\begin{exercise} Check that the multiplication $\cdot:\IZ\times\IZ\to\IZ$ has the following properties for every $x,y,z\in\IZ$:
\begin{enumerate}
\item $(x\cdot y)\cdot z=x\cdot(y\cdot z)$;
\item $x\cdot y=y \cdot x$;
\item $x\cdot 0=0$;
\item $x\cdot 1=x$;
\item $x \cdot(\minus 1)=\minus x$;
\item If $0<_\IZ x$ and $0<_\IZ y$, then $0<_\IZ x\cdot y$;
\item $x\cdot(y+z)=x{\cdot} y+x{\cdot}z$; 
\item If $0<_\IZ x$, then $y<_\IZ z\;\Leftrightarrow\;x \cdot y<_\IZ x\cdot z$;
\item $x\cdot y=0\;\Leftrightarrow\;(x=0\;\vee\;y=0)$.
\end{enumerate}
\smallskip

\noindent{\em Hint}: Prove these properties for the isomorphic copy $\tilde \IZ$ of $\IZ$, considered in the hint to Exercise~\ref{ex:add-Z}. In the proof apply Theorems~\ref{t:mult-ord-prop} and \ref{t:mult-nat}.
\end{exercise}

To introduce rational numbers, we shall need some standard facts about the divisibility of integer numbers. We say that a natural number $d$ {\em divides} an integer number $z$ if $z=d\cdot k$ for some integer number $k$, which is denoted by $\frac{z}{d}$.

 For an integer number $z$ by $\mathsf{Div}(z)$ we denote the set of natural numbers dividing $z$. The set $\mathsf{Div}(z)$ contains $1$ and hence is not empty. Two integer numbers $a,b$ are called \index{coprime integers}{\em coprime} if $\mathsf{Div}(a)\cap\mathsf{Div}(b)=\{1\}$. For two numbers $a,b$ the largest element of the set $\mathsf{Div}(a)\cap \mathsf{Div}(b)$ is called the \index{greatest common divisor}{\em greatest common divisor} of $a$ and $b$ and is denoted by $\gcd(a,b)$. If $d$ is the largest common divisor of integer numbers $a,b$, then the integer numbers $\frac{a}{d}$ and $\frac{b}{d}$ are coprime.

\section{Rational numbers} 

Rational numbers are introduced to ``materialize'' the result of division of integer numbers. For example, $\frac 12$ represents the result of division of $1$ by $2$ but also $2$ by $4$ and $3$ by $6$, etc. So, $\frac12$ can be defined as the set of pairs $\{\langle 1,2\rangle,\langle 2,4\rangle,\langle 3,6\rangle,\dots\}$. Among such pairs the simplest (in the sense of von Neumann hierarchy) is the pair $\langle 1,2\rangle$, which can be taken as the representative of $\frac12$. 

More generally, for any pair $\langle a,b\rangle\in\IZ\times\IN$ the fraction $\frac{a}{b}$ can be defined as the set of ordered pairs $\{\langle m,n\rangle\in\IZ\times\IN:a\cdot n=b\cdot m\}$. The simplest element of this set is the ordered pair $\langle \frac{a}{\gcd(a,b)},\frac{b}{\gcd(a,b)}\rangle$ consisting of relatively prime integers $ \frac{a}{\gcd(a,b)}$ and $\frac{b}{\gcd(a,b)}$. The pairs $\langle m,n\rangle\in\IZ\times\IN$ with  relatively prime numbers $m,n$ thus encode all rational numbers $\frac{m}{n}$.

This suggests defining the set of rational numbers as the set
$$\IQ=\IZ\cup\{\langle m,n\rangle\in \IZ\times\IN:\mbox{$m$ and $n$ are coprime and $n\ge 2$}\}.$$ 
Pairs $\langle m,n\rangle\in\IQ\setminus\IZ$ will be denoted as fractions $\frac{m}{n}$, and integers $z\in\IZ$ as fractions $\frac{z}1$. Therefore, the set $\IQ$ can be written more uniformly as
$$\big\{\tfrac{m}{n}:\mbox{$m\in\IZ$ and $n\in\IN$ are coprime}\big\}.$$\index{number!rational}\index{$\mathbb Q$}\index{set!$\mathbb Q$}
\begin{exercise} Which number is represented by the ordered pair $\langle 0,2\rangle$? Why the pairs used for encoding negative integers do not intersect the pairs used for encoding non-integer rationals?
\end{exercise}



The linear order $\le_\IZ$ can be extended to the linear order
$$\le_\IQ\;=\{\langle\tfrac{m}{n},\tfrac{k}{l}\rangle\in\IQ\times\IQ:m{\cdot}l\le_\IZ k{\cdot}n\}.$$
on the set $\IQ$.

\begin{exercise} Show that $\le_\IQ$ is indeed a reflexive linear order on $\IQ$.
\end{exercise}

The irreflexive linear order
$$<_\IQ\;=\;\le_\IQ\!\setminus\Id$$ is called the {\em strict linear order} on $\IQ$.

Next we extend the arithmetic operations from $\IZ$ to $\IQ$. The additive inversion is extended to $\IQ$ letting $\minus\langle m,n\rangle=\langle \minus m,n\rangle$ for any $\langle m,n\rangle\in \IQ\setminus\IZ$.

Let  $\div:\IZ\times\IN\to\IQ$ be the function assigning to any ordered pair $\langle m,n\rangle\in\IZ\times\IN$ the rational number  
$$\frac{m}{n}=\frac{\frac{m}{\gcd(m,n)}}{\frac{n}{\gcd(m,n)}}=\begin{cases}\langle \frac{m}{\gcd(m,n)},\frac{n}{\gcd(m,n)}\rangle&\mbox{if $\gcd(m,n)<n$};\\
\frac{m}{\gcd(m,n)}&\mbox{if $\gcd(m,n)=n$}.
\end{cases}
$$
We recall that for a nonzero integer $d$ dividing an integer number $z$ the fraction $\frac{z}{d}$ denotes the unique integer number $k$ such that $d\cdot k=z$. 

Given any rational numbers $\frac{m}{n},\frac{k}{l}\in\IQ$ define their sum $\frac{m}{n}+\frac{k}{l}$ and product $\frac{m}{n}\cdot \frac{k}{l}$ as 
$$\frac{m}{n}+\frac{k}{l}=\frac{m{\cdot}l+k{\cdot}n}{n\cdot l}\mbox{  \ and \ }\frac{m}{n}\cdot\frac{k}{l}=\frac{m\cdot k}{n\cdot l}.$$

\begin{exercise} Given any $x,y,z\in\IQ$, check the following properties of the addition and multiplication of rational numbers:
\begin{enumerate}
\item $x+(y+z)=(x+y)+z$;
\item $x+y=y+x$;
\item $x+0=x$;
\item $x+(\minus x)=0$;
\smallskip
\item $x\cdot(y\cdot z)=(x\cdot y)\cdot z$;
\item $x\cdot y=y\cdot x$;
\item $x\cdot 1=x$;
\item $(x\ne 0)\;\Rightarrow\;(x\cdot\frac1x=1)$;
\item $x\cdot(y+z)=(x\cdot y)+(x\cdot z)$;
\smallskip

\item $x<_\IQ y\;\Rightarrow\;x+z<_\IQ y+z$;
\item $(0<_\IQ x\;\wedge\;0<_\IQ y)\;\Ra\;0<_\IQ x\cdot y$.
\end{enumerate}
\end{exercise}

Since the linear orders $\le_\IQ$ and $<_\IQ$ are countable and universal, Cantor's Corollary~\ref{l:Cantor-lin-univ}  implies the following characterization.

\begin{theorem}[Cantor] An (ir)reflexive order $L$ is isomorphic to the (strict) linear order on $\IQ$ if and only if $L$ is nonempty, countable and  has two properties:
\begin{itemize}
\item[\textup{1)}] for any elements $x<_L y$ of $\dom[L^\pm]$ there exists $z\in\dom[L^\pm]$ such that $x<_L z<_L y$;
\item[\textup{2)}] for any element $z\in\dom[L^\pm]$ there are $x,y\in\dom[L^\pm]$ such that $x<_L z<_L y$.
\end{itemize}
\end{theorem}

\section{Real numbers}

The set of real numbers  is defined as the set
$$\IR=\IQ\cup\Gap(\le_\IQ)$$where $\Gap(\le_\IQ)$ is the set of $\le_\IQ$-gaps. We recall that a {\em $\le_\IQ$-gap} is an ordered pair $\langle a,b\rangle$ of nonempty disjoint sets $a,b$ such that $a\cup b=\IQ$ and for any elements $x\in a$ and $y\in b$ there are elements $x'\in a$ and $y'\in b$ such that $x<_\IQ x'<_\IQ y'<_\IQ y$.

The set $\IR$\index{$\mathbb R$}\index{set!$\mathbb R$} is called the \index{real line}\index{real number}\index{number!real}{\em real line} and its elements are called {\em real numbers}. The set $\IR$ carries the reflexive linear order $\le_\IR$ equal to the gap extension $\Theta(\le_\IQ)$ of the linear order $\le_\IQ$. Also $\IR$ carries the {\em strict linear order} $<_\IR\;=\;\le_\IR\!\setminus\Id$. By Theorem~\ref{t:gap-ext}(5), the linear orders $\le_\IR$ and $<_\IR$ are gapless.

Now we extend the arithmetic operations from the set of rationals $\IQ$ to the set of reals $\IR$. The operation of additive inverse $\minus:\IQ\to\IQ$ is extended to the operation $\minus:\IR\to\IR$ assigning to each $\le_\IQ$-gap $\langle a,b\rangle$ the $\le_\IQ$-gap  $\langle\minus b,\minus a\rangle$, where $\minus a=\{\minus x:x\in a\}$ and $\minus b=\{\minus y:y\in b\}$.

To define the addition and multiplication of real numbers, we need some preparation.

For every real number $x\in\IR$, consider its left and right sets
$$
\cev x=\begin{cases}a&\mbox{if $x=\langle a,b\rangle\in\Gap(\le_\IQ)$};\\
\{y\in\IQ:y<_\IQ x\}&\mbox{if $x\in\IQ$};
\end{cases}
$$and
$$
\vec x=\begin{cases}b&\mbox{if $x=\langle a,b\rangle\in\Gap(\le_\IQ)$};\\
\{y\in\IQ:x<_\IQ y\}&\mbox{if $x\in\IQ$}.
\end{cases}
$$
Let $\check\IR=\{\langle \cev x,\vec x\rangle:x\in\IR\}$ and $f:\IR\to\check \IR$ be the function assigning to each real number $x$ the ordered pair $\langle \cev x,\vec x\rangle$. It is clear that this function is bijective and $f{\restriction}_{\Gap(\le_\IQ)}=\Id{\restriction}_{\Gap(\le_\IQ)}$. The inverse function $f^{-1}:\check \IR\to\IR$ assigns to each ordered pair $\langle a,b\rangle$ the unique real number $y$ such that $\langle a,b\rangle=\langle\cev y,\vec y\rangle$. This unique real number $y$ will be denoted by $a\curlyvee b$. Therefore, $x=\cev x\curlyvee \vec x$ for every real number $x$. 

For subsets $a,b\subseteq\IQ$ let $a+b=\{x+y:x\in a,\;y\in b\}$. 
In these notations the addition of  real numbers $x,y\in\IR$ is defined by the simple formula:
$$x+y=(\cev x+\cev y)\curlyvee(\vec x+\vec y).$$

\begin{exercise} Show that the addition of real numbers is well-defined and prove the following properties for any real numbers $x,y,z\in\IR$:
\begin{enumerate}
\item $(x+y)+z=x+(y+z)$;
\item $x+y=y+x$;
\item $x+0=x$;
\item $x+(\minus x)=0$;
\item If $x<_\IR y$, then $x+z<_\IR y+z$.
\end{enumerate}
\end{exercise}

Now we define the multiplication of  real numbers $x,y$. If one of these numbers is equal to zero, then we put $x\cdot y=0$. If $0<_\IR x$ and $0<_\IR y$, then $x\cdot y=a\curlyvee b$ where $$a=\{z\in\IQ:\exists u,v\in\IQ\;((0<_\IR u<_\IR x)\;\wedge\;(0<_\IR v<_\IR y)\;\wedge\;(z<u\cdot v))\}$$
and 
$$b=\{z\in\IQ:\exists u,v\in\IQ\;((x<_\IR u)\;\wedge\;(y<_\IR v)\;\wedge\;(u\cdot v<z)\}.$$
Having defined the product of $x\cdot y$ of strictly positive real numbers $x,y$, we also put
$$(\minus x)\cdot y=x\cdot(\minus y)=\minus(x\cdot y)\quad\mbox{and}\quad(\minus x)\cdot(\minus y)=x\cdot y.$$
Those formulas define the product of arbitrary real numbers.

\begin{exercise}   Show that the multiplication of real numbers is well-defined and prove the following its properties for any real numbers $x,y,z\in\IR$:
\begin{enumerate}
 \item $x\cdot(y\cdot z)=(x\cdot y)\cdot z$;
\item $x\cdot y=y\cdot x$;
\item $x\cdot 1=x$;
\item $x\cdot(y+z)=(x\cdot y)+(x\cdot z)$;
\item $(0<_\IR x\;\wedge\;0<_\IR y)\;\Ra\;0<_\IR x\cdot y$.
\end{enumerate}
\end{exercise}

\begin{Exercise} Show that the real line $\IR$ is {\em real closed} in the sense that for every odd number $n\in\IN$ and any real numbers $a_0,\dots,a_n$ with $a_n\ne 0$ there exists a real number $x$ such that 
$$a_0+a_1x+\dots+a_nx^n=\mathtt 0.$$
In particular, for any nonzero real number $x$ there exists a unique real number $y$ such that $x\cdot y=1$.
\end{Exercise}

A subset $a\subseteq\IR$ is called {\em upper-bounded} if $\exists y\in\IR\;\forall x\in a\;(x\le_\IR y)$.

\begin{exercise} Prove that for every nonempty upper bounded set $A\subseteq\IR$ the set $B=\{b\in\IR:\forall x\in A\;(x\le_\IR b)\}$ has the smallest element. This smallest element is denoted by $\sup A$.
\end{exercise}

\begin{exercise}[Axiom of Archimedes] Prove that for every positive real numbers $a$ and $\e$ there exists a natural number $n$ such that  $a\le n\cdot\e$.  
\end{exercise}

Since the linear orders $\le_\IR$ and $<_\IR$ are countably dense, endless and gapless, Proposition~\ref{p:gap<=>bc} and Cantor's Theorem~\ref{t:wR-iso} imply the following characterization.

\begin{theorem}\label{t:R-ord} A (ir)reflexive order $L$ is isomorphic to the (strict) linear order of the real line $\IR$ if and only if $L$ is boundedly complete and there exists a nonempty countable subset $D\subseteq\dom[L^\pm]$ that  has two properties:
\begin{itemize}
\item[\textup{1)}] for any elements $x<_L y$ of $\dom[L^\pm]$ there exists $z\in D$ such that $x<_L z<_L y$;
\item[\textup{2)}] for any element $z\in \dom[L^\pm]$ there are $x,y\in D$ such that $x<_L z<_L y$.
\end{itemize}
\end{theorem}

Theorem~\ref{t:order-continuum} implies

\begin{corollary} The real line has cardinality $|\IR|=|2^\w|$.
\end{corollary}

\begin{exercise} Show that $\IR\subset V_{\w+3}$ and $\IR\in V_{\w+4}$.
\end{exercise}

\begin{exercise} Show that $\IR\cap(\{1\}\times\IR)=\{\frac1m:2\le m\in\IN\}$.
\end{exercise}

\noindent{\em Hint:} Choose any $x\in\IR\cap(\{1\}\times \IR)$ and find $y\in\IR$ such that $x=\langle 1,y\rangle$. Since the set $\langle 1,y\rangle$ is not empty, the real number $x$ is nonzero. Since the set $1$ is finite, the real number $x$ is rational and $|x|=|\langle 1,y\rangle|=|\{\{1\},\{1,y\}|\le 2$. If $x\in\w$, then $x\subseteq \{0,1\}$ and $\{\{1\},\{1,y\}\}=\langle 1,y\rangle\subseteq\{0,1\}$, which is not true.  Therefore, $\langle 1,y\rangle=x\in\IQ\setminus\w\subseteq \IZ\times \IN$ and $x=\{\langle 1,y\rangle:2\le y\in\IN\}$ by the definition of numbers in the set $\IQ\setminus\w$. \hfill $\square$ 
\medskip

\begin{proposition}\label{p:real-complex} Every real number $x$ has $x\cap( \{1\}\times\mathbf U)=\emptyset$.
\end{proposition}

\begin{proof} To derive a contradiction, assume that there exists a set $z\in x\cap(\{1\}\times \mathbf U)$, and find a set $u$  such that $z=\langle 1,u\rangle$. Assuming that the real number $x$ is irrational, we can find two disjoint infinite subsets $A,B$ of $\IQ$ such that $x=\langle A,B\rangle$. Then $\{\{1\},\{1,u\}\}=\langle 1,u\rangle=z\in x=\langle A,B\rangle=\{\{A\},\{A,B\}\}$ and hence $\{\{1\},\{1,u\}\}=\{A\}$ or $\{\{1\},\{1,u\}\}=\{A,B\}$. Then $\{1\}=A$ or $\{1\}=B$, which is not possible because the sets $A,B$ are infinite. This contradiction shows that the real number $x$ is rational. It follows from $z=\langle 1,u \rangle=\{\{1\},\{1,u\}\}\notin\w$ that $x\notin \w$ and hence $x\in\IQ\setminus \w$. The definition of rational numbers guarantees that $x=\langle n,m\rangle$ for some integer numbers $n,m$ with $n\ne m\ge 2$. It follows from $\langle 1,u\rangle=z\in x=\langle n,m\rangle=\{\{n\},\{n,m\}\}$ that $\langle 1,u\rangle=\{n\}$ or $\langle 1, u\rangle=\{n,m\}$ and hence $\{1\}\in\{\{1\},\{1,u\}\}=\langle 1,u\rangle \subseteq\{n,m\}\subseteq\IZ$, which contradicts the definition of the set $\IZ=\w\cup\{\langle 0,k\rangle:k\in\IN\}$.
\end{proof}

\section{Complex numbers}

In this section we construct a canonical model for the set $\IC$ of complex numbers.

For real numbers $x,y$, the set
$$x+iy\defeq x\cup (\{1\}\times y)$$is called the {\em complex number} with real part $x$ and imaginary part $y$. The set
$$\IC\defeq\{x+iy:x,y\in\IR\}$$ is called the {\em set of complex numbers} or else the {\em complex plane}.

Observe that every real number $x$ is equal to the complex number $x+i0=x\cup(\{1\}\times\emptyset)$, so $\IR\subset\IC$.

Proposition~\ref{p:real-complex} implies that two complex numbers $a+ib$ and $x+iy$ are equal if and only if $a=x$ and $b=y$. Therefore, the functions $$\Re:\IC\to\IR,\quad \Re:x+iy\mapsto x,\quad\mbox{and}\quad\Im:\IC\to\IR,\quad \Im: x+iy\mapsto y,$$of taking the real and imaginary parts of a complex number are well-defined. 

For two complex numbers $a+ib$ and $x+iy$ define their addition and multiplication by the formulas
$$(a+ib)+(x+iy)\defeq(a+x)+i(b+y)\quad\mbox{and}\quad(a+ib)\cdot(x+iy)\defeq(ax-by)+i(ay+bx).$$
The definition of the multiplication implies that $(0+i1)^2=\minus1$.

For a complex number $z=x+iy$, let $\minus z=(\minus x)+i(\minus y)$ be the {\em additive inverse} to $z$ and $\bar z\defeq x-iy\defeq x+i(\minus y)$ be the {\em conjugated complex number} to $z$. The definition of the multiplication of complex numbers ensures that $z\bar z=x^2+y^2$ is a nonnegative real number. If $z\ne 0$, then $z\bar z=x^2+y^2>0$ and we can consider the complex number
$$z^{-1}=\frac{x}{x^2+y^2}-i\frac{y}{x^2+y^2}=\frac{\bar z}{z\bar z},$$
called the {\em multiplicative inverse} to $z$.

\begin{exercise} Show that the addition and multiplication of complex numbers is well-defined and prove the following its properties for any complex numbers $x,y,z\in\IC$:
\begin{enumerate}
\item $(x+y)+z=x+(y+z)$;
\item $x+y=y+x$;
\item $x+0=x$;
\item $x+(\minus x)=0$;
 \item $x\cdot(y\cdot z)=(x\cdot y)\cdot z$;
\item $x\cdot y=y\cdot x$;
\item $x\cdot 1=x$;
\item $x\cdot(y+z)=(x\cdot y)+(x\cdot z)$;
\item $z\cdot \bar z\in\IR$ and $z\cdot\bar z\ge0$;
\item $z=0$ if and only if $z\bar z=0$;
\item $z\ne 0\to z\cdot z^{-1}=1$.
\end{enumerate}
\end{exercise}

\begin{Exercise} Show that the complex plane $\IR$ is {\em algebraically closed} in the sense that for every number $n\in\IN$ and any complex  numbers $a_0,\dots,a_n$ with $a_n\ne 0$ there exists a complex number $z$ such that 
$$a_0+a_1z+\dots+a_nz^n=0.$$
\end{Exercise}

\begin{exercise} Show that $\IC\subset V_{\w+4}$ and $\IC\in V_{\w+5}$.
\end{exercise}

\section{Surreal numbers}\label{s:surreal}

{\small 
\rightline{\em An empty hat rests on a table made of a few axioms of standard set theory.}

\rightline{\em Conway waves two simple rules in the air, then reaches into almost nothing}

\rightline{\em and pulls out an infinitely rich tapestry of numbers.}
\smallskip

\rightline{Martin Gardner}
\vskip20pt

\rightline{\em I walked around for about six weeks after discovering the surreal numbers}

\rightline{\em  in a sort of permanent daydream, in danger of being run over...}
\smallskip

\rightline{John Horton Conway}
\vskip20pt

\rightline{\em The usual numbers are very familiar,}

\rightline{\em but at root they have a very  complicated structure.}

\rightline{\em Surreals are in every logical, mathematical and aesthetic
sense better.}
\smallskip

\rightline{Martin Kruskal}
}
\vskip30pt

In this section we make a brief introduction to Conway's surreal numbers. Surreal numbers are elements of a proper class $\No$ called the surreal line. The surreal line carries a natural structure of an ordered field, which contains an isomorphic copy of the field $\IR$ but also contains surreal numbers that can be identified with arbitrary ordinals.

Surreal numbers were introduced by John Horton Conway \cite{ONAG} who called them numbers (the adjective ``surreal'' was suggested by Donald Knuth \cite{Knuth}). %

The surreal line is obtained by transfinite iterations of cut extensions of linear orders, starting from the empty order. We recall (see Section~\ref{s:cuts}) that for a linear order $L$ an {\em $L$-cut} is an ordered pair of sets $\langle a,b\rangle$ such that $a\cap b=\emptyset$, $a\cup b=\dom[L^\pm]$ and $a\times b\subseteq L$. For an $L$-cut $x=\langle a,b\rangle$ the sets $a$ and $b$ will be denoted by $\cev x$ and $\vec x$ and called the {\em left} and {\em right parts} of the cut $x$.

 If the set $L$ is well-founded (i.e., $L\in\VV$), then the set $\dom[L^\pm]$ is disjoint with the set $\Cut(L)$ of $L$-cuts and the set $\dom[L^\pm]\cup\Cut(L)$ carries the linear order \begin{multline*}
 \Xi(L)=L\cup\{\langle\langle a,b\rangle,\langle a',b'\rangle\rangle\times\Cut(L)\times\Cut(L):a\subseteq a'\}\cup\\
 \{\langle x,\langle a,b\rangle\rangle\in \dom[L^\pm]\times\Cut(L):x\in a\}\cup\{\langle \langle a,b\rangle,y\rangle\in\Cut(L)\times \dom[L^\pm]:y\in b\}
 \end{multline*}
 called the {\em cut extension} of $L$, see Section~\ref{s:cuts}.
 
Let $(L_\alpha)_{\alpha\in\Ord}$ be the transfinite sequence of reflexive linear orders defined by the recursive formula
$$L_\alpha=\bigcup_{\beta\in\alpha}\Xi(L_\beta),$$
where $\Xi(L_\beta)$ is the cut extension of the linear order $L_\beta$.

 So,
 \begin{itemize}
 \item $L_0=\emptyset$, 
 \item $L_{\alpha+1}=\Xi(L_\alpha)$ for any ordinal $\alpha$;
 \item $L_\alpha=\bigcup_{\beta\in \alpha}L_\beta$ for any limit ordinal $\alpha$.
 \end{itemize}

 The existence of the transfinite sequence $(L_\alpha)_{\alpha\in\Ord}$ follows from Theorem~\ref{t:dynamics} applied to the expansive function
$$\Phi:\VV\to \VV,\quad\Phi(y)=\begin{cases}\Xi(y)&\mbox{if $y$ is a linear order};\\
y&\mbox{otherwise}.\end{cases}
$$ 
Using Lemma~\ref{l:sur1} it can be shown that for every ordinal $\alpha$ the underlying set $\mathsf{No}_\alpha=\dom[L_\alpha^\pm]$ of the linear order $L_\alpha$ can be written as the union $\bigcup_{\beta\in\alpha}\Cut(L_\beta)$ of the indexed family of pairwise disjoint sets $(\Cut(L_\beta))_{\beta\in\alpha}$.

The union $$\No=\bigcup_{\alpha\in\Ord}\mathsf{No}_\alpha=\bigcup_{\alpha\in\Ord}\Cut(L_\alpha)$$ is called the {\em surreal line}\index{$\mathbf{No}$} and its elements are called {\em surreal numbers}.\index{number!surreal}\index{surreal line}\index{surreal number} The surrreal line $\No$ carries the reflexive linear order 
$$\leqslant_\No\; =\bigcup_{\alpha\in\Ord}L_\alpha$$and the strict linear order
$$<_\No\;=\;\leqslant_\No\!\!\setminus\Id.$$
 \smallskip
 
Let us look at the structure of the sets $\mathsf{No}_\alpha$ for small ordinals $\alpha$.
The set $\mathsf{No}_0$ is empty and carries the empty linear order $L_0$ whose set of cuts $\Cut(L_0)=\{\langle \emptyset,\emptyset\rangle\}$ contains the unique ordered pair $\langle \emptyset,\emptyset\rangle$ denoted by $\mathtt 0$.

Therefore, $\mathsf{No}_1=\Cut(L_0)=\{\mathtt 0\}$. This set carries the linear order $L_1=\{\langle \mathtt 0,\mathtt 0\rangle\}$ that has two cuts denoted by 
$$\minus\mathtt 1:=\langle \emptyset,\{\mathtt 0\}\rangle\quad\mbox{and}\quad\mathtt{1}:=\langle\{\mathtt 0\},\emptyset\rangle.$$
The set $\mathsf{No}_2=\mathsf{No}_1\cup\Cut(L_1)=\{\minus\mathtt{1,0,1}\}$ carries the linear order $$L_2=\{\langle\mathtt{\minus 1},\mathtt{\minus 1}\rangle, \langle\mathtt{\minus 1},\mathtt{0}\rangle, \langle\mathtt{\minus 1},\mathtt{1}\rangle, \langle\mathtt 0,\mathtt 0\rangle, \langle\mathtt{0},\mathtt{1}\rangle,\langle\mathtt{1},\mathtt{1}\rangle\}$$  that has four cuts denoted by $$\mathtt{\minus 2}=\langle\emptyset,\{\mathtt{\minus 1},\mathtt{0},\mathtt{1}\}\rangle,\quad\mathtt{\minus }\tfrac{\mathtt1}{\mathtt2}:= \langle\{\mathtt{\minus 1}\},\{\mathtt{0},\mathtt{1}\}\rangle,\quad\tfrac{\mathtt1}{\mathtt2}:= \langle\{\mathtt{\minus 1},\mathtt{0}\},\{1\}\rangle,\quad\mathtt 2= \langle\{\mathtt{\minus 1},\mathtt{0},\mathtt{1}\},\emptyset\rangle.$$
Therefore, $\mathsf{No}_3=\{\minus{\mathtt2},\minus{\mathtt 1},\minus\tfrac{\mathtt1}{\mathtt 2},\mathtt{0,\tfrac12,1,2}\}$. 

The set $\mathsf{No}_4=\{\mathtt{\minus 3,\minus 2},\minus \tfrac{\mathtt3}{\mathtt2},\mathtt{\minus 1},\minus \tfrac{\mathtt 3}{\mathtt 4},\minus \tfrac{\mathtt 1}{\mathtt 2},\minus \tfrac{\mathtt1}{\mathtt 4},\mathtt 0,\tfrac{\mathtt1}{\mathtt4},\tfrac{\mathtt1}{\mathtt2},\tfrac{\mathtt3}{\mathtt 4},\mathtt{1},\tfrac{\mathtt3}{\mathtt 2},\mathtt{2,3}\}$ has 15 elements, in particlular:
$$
\begin{aligned}
&\qquad\quad\mathtt{\minus 3}=\langle \emptyset,\{\mathtt{\minus 2},\mathtt{\minus 1},\minus \tfrac{\mathtt 1}{\mathtt 2},\mathtt 0,\tfrac{\mathtt1}{\mathtt2},\mathtt{1},\mathtt{2}\}\rangle,\;\;
\minus \tfrac{\mathtt3}{\mathtt 2}=\langle \{\mathtt{\minus 2}\},\{\mathtt{\minus 1},\minus \tfrac{\mathtt 1}{\mathtt 2},\mathtt 0,\tfrac{\mathtt1}{\mathtt2},\mathtt{1},\mathtt{2}\}\rangle,\;\;\\
&\qquad\quad\minus \tfrac{\mathtt3}{\mathtt4}=\langle \{\mathtt{\minus 2},\mathtt{\minus 1}\},\{\minus \tfrac{\mathtt 1}{\mathtt 2},\mathtt 0,\tfrac{\mathtt1}{\mathtt2},\mathtt{1},\mathtt{2}\}\rangle,\;\;
\minus\tfrac{\mathtt1}{\mathtt 4}=\langle \{\mathtt{\minus 2},\mathtt{\minus 1},\minus \tfrac{\mathtt 1}{\mathtt 2}\},\{\mathtt 0,\tfrac{\mathtt1}{\mathtt2},\mathtt{1},\mathtt{2}\}\rangle,\;\;\\
&\qquad\quad\;\;\tfrac{\mathtt1}{\mathtt4}=\langle \{\mathtt{\minus 2},\mathtt{\minus 1},\minus \tfrac{\mathtt 1}{\mathtt 2},\mathtt 0\},\{\tfrac{\mathtt1}{\mathtt2},\mathtt{1},\mathtt{2}\}\rangle,\;\;\;\;
\tfrac{\mathtt3}{\mathtt 4}=\langle \{\mathtt{\minus 2},\mathtt{\minus 1},\minus \tfrac{\mathtt 1}{\mathtt 2},\mathtt 0,\tfrac{\mathtt1}{\mathtt2}\},\{\mathtt{1},\mathtt{2}\}\rangle,\;\;\\
&\qquad\quad\;\;\tfrac{\mathtt3}{\mathtt2}=\langle \{\mathtt{\minus 2},\mathtt{\minus 1},\minus \tfrac{\mathtt 1}{\mathtt 2},\mathtt 0,\tfrac{\mathtt1}{\mathtt2},\mathtt{1}\},\{\mathtt{2}\}\rangle,\;\;\;\;
\mathtt3=\langle \{\mathtt{\minus 2},\mathtt{\minus 1},\minus \tfrac{\mathtt 1}{\mathtt 2},\mathtt 0,\tfrac{\mathtt1}{\mathtt2},\mathtt{1},\mathtt{2}\},\emptyset\rangle.
\end{aligned}
$$The set $\mathsf{No}_5$ consists of 31 elements:
$$\mathtt{\mbox{\minus }4,\mbox{\minus }3},\minus \tfrac{\mathtt5}{\mathtt2},\minus \mathtt2,\minus \tfrac{\mathtt7}{\mathtt4},\minus \tfrac{\mathtt3}{\mathtt2},\minus \tfrac{\mathtt5}{\mathtt4},\mathtt{\minus 1},\minus \tfrac{\mathtt 7}{\mathtt 8},\minus \tfrac{\mathtt 3}{\mathtt 4},\minus \tfrac{\mathtt 5}{\mathtt 8},\minus \tfrac{\mathtt 1}{\mathtt 2},\minus \tfrac{\mathtt 3}{\mathtt 8},\minus \tfrac{\mathtt1}{\mathtt 4},\minus \tfrac{\mathtt1}{\mathtt8},\mathtt 0,\tfrac{\mathtt1}{\mathtt8}, \tfrac{\mathtt1}{\mathtt4},\tfrac{\mathtt3}{\mathtt8},\tfrac{\mathtt1}{\mathtt2},\tfrac{\mathtt5}{\mathtt8},\tfrac{\mathtt3}{\mathtt 4},\tfrac{\mathtt7}{\mathtt8},\mathtt{1},\tfrac{\mathtt5}{\mathtt4},\tfrac{\mathtt3}{\mathtt 2},\tfrac{\mathtt7}{\mathtt4},\mathtt{2},\tfrac{\mathtt5}{\mathtt2},\mathtt 3,
$$
and so on. 

Now we see that the set $\mathsf{No}_\w$ is countable and its elements can be labeled by dyadic rational numbers (with preservation of order).

The set $\mathsf{No}_{\w+1}=\mathsf{No}_\w\cup\Cut(L_\w)$ contains all $L_\w$-cuts which include all $L_\w$-gaps that can be interpreted as real numbers. Besides the real numbers the set $\mathsf{No}_{\w+1}$ contains the infinitely large number $\langle {\mathsf{No}_\w},\emptyset\rangle$ that can be identified with the ordinal $\w$ and also  its additive inverse $\minus\w=\langle\emptyset,\mathsf{No}_\w\rangle$.  Also for any dyadic rational $x$ the set $\mathsf{No}_{\w+1}$ contains two $L_\w$-cuts $$x_-=\langle \{y\in\mathsf{No}_\w:y<_\No x\},\{z\in\mathsf{No}_\w:x\le_\No z\}\rangle$$ and $$x_+=\langle \{y\in\mathsf{No}_\w:y\le_\No x\},\{z\in\mathsf{No}_\w:x<_\No z\}\rangle$$ that can be interpreted as numbers that are infinitely close to $x$.

The exists a natural injective function $\Ord\to\No$ which assigns to every ordinal $\alpha$ the $L_\alpha$-cut $\langle \mathsf{No}_\alpha,\emptyset\rangle$. Therefore, the surreal line contains a copy of the ordinal line, which implies that $\No$ is a proper class.





 
Let $\birth:\No\to\Ord$ be the function assigning to each element $x\in\No$ the unique ordinal $\alpha$ such that $x\in \Cut(L_\alpha)$. The function $\birth$ is called the \index{birthday function}\index{function!birthday} {\em birthday function}. 

Consider the class of ordered pairs
$$\mathcal P_<(\No)=\{\langle a,b\rangle:a,b\in\mathcal P(\No),\;a\times b\subset\; <_\No\}$$
and observe that $$\No=\bigcup_{\alpha\in\Ord}\Cut(L_\alpha)\subseteq\mathcal P_{<}(\No).$$

The following property of the birthday function is crucial for introducing various algebraic structures on the surreal line.

\begin{theorem}\label{t:birth} For any ordered pair $\langle a,b\rangle\in\mathcal P_{<}(\No)$ there exists a unique number $x\in\No$ such that 
\begin{itemize}
\item[\textup{1)}] $(a\times\{x\})\cup(\{x\}\times b)\subseteq\;<_{\No}$ and
\item[\textup{2)}] $(a\times\{y\})\cup(\{y\}\times b)\not\subseteq\;<_{\No}$ for any element $y\in\No$ with $\birth(y)<\birth(x)$.
\end{itemize}
This unique number $x$ will be denoted by $a\curlyvee b$.
\end{theorem}

\begin{proof} First we show that the class 
$$\No(a,b)=\{z\in\No:(a\times\{z\})\cup(\{z\}\times b)\subseteq\; <_\No\}$$is not empty.

By the Axiom of Replacement, the image $\birth[a\cup b]\subseteq\Ord$ of the set $a\cup b$ under the birthday function is a set, which implies that $a\cup b\subseteq \mathsf{No}_\alpha$ for some ordinal $\alpha$. Consider the sets ${\downarrow}a=\{x\in\mathsf{No}_\alpha:\exists y\in a\;\;(\langle x,y\rangle\in L_\alpha)\}$ and ${\uparrow}b=\mathsf{No}_\alpha\setminus{\downarrow}a$. Observe that $\langle {\downarrow}a,{\uparrow}b\rangle\in \Xi(L_\alpha)=L_{\alpha+1}$ and hence $\langle{\downarrow}a,{\uparrow}b\rangle\in\mathsf{No}(a,b)$, witnessing that the class $\No(a,b)$ is nonempty. Since the relation $\E{\restriction}\Ord$ is well-founded, the nonempty class $\birth[\No(a,b)]$ contains the smallest ordinal $\alpha$. Since $\mathsf{No}_\gamma=\bigcup_{\beta\in\gamma}\mathsf{No}_\beta$ for any limit ordinal $\gamma$, the ordinal $\alpha$ is not limit and hence $\alpha=\beta+1$ for some ordinal $\beta\in\alpha$. 

To finish the proof, it suffices to check that the class $\No(a,b)\cap \mathsf{No}_\alpha$ is a singleton. To derive a contradiction, assume that $\No(a,b)\cap\mathsf{No}_\alpha$ contains two distinct elements $x,y$. The minimality of $\alpha$ ensures that $x,y\notin\mathsf{No}_\beta$. Then $x,y\in \mathsf{No}_{\beta+1}\setminus\mathsf{No}_\beta=\Cut(L_\beta)$ are two distinct $L_\beta$-cuts. So, $x=\langle a',b'\rangle$ and $y=\langle a'',b''\rangle$ for some $L_\beta$-cuts $\langle a',b'\rangle$ and $\langle a'',b''\rangle$. Since $L_\alpha$ is a linear order, either $\langle x,y\rangle\in L_\alpha$ or $\langle y,x\rangle\in L_\alpha$. We lose no generality assuming that $\langle x,y\rangle\in L_\alpha$ and hence $a'\subset a''$ by the definition of the linear order $L_\alpha=L_{\beta+1}$. Choose any point $z\in a''\setminus a'\subseteq \No_\beta$ and observe that $z\in b'$ and hence $x=\langle a',b'\rangle < z< \langle a'',b''\rangle=y$. For every $u\in a$ and $v\in b$, we have $u<_\No x<_\No z<_\No y<_\No v$, which implies that $z\in\No(a,b)\cap \mathsf{No}_\beta$. But this contradicts the minimality of the ordinal $\alpha$. 
\end{proof}

\begin{exercise} Given any ordinal $\alpha$ and an $L_\alpha$-cut $x=\langle a,b\rangle\in \Cut(L_\alpha)\subset\No$, show that $x=a\curlyvee b$.
\end{exercise}

Theorem~\ref{t:birth} implies that the (strict) linear order of the surreal line is universal.
Under the Global Well-Orderability Principle ({\sf  GWO}) the surreal line in well-orderable.
Applying Theorem~\ref{t:lin-univ}, we obtain the following characterization of the (strict) linear order of the surreal line.

\begin{theorem} Assume $(\mathsf{GWO})$. An (ir)reflexive linear order $L$ is isomorphic to the (strict) linear order of the surreal line if and only if $L$ is universal and $\dom[L^\pm]$ is a proper class.
\end{theorem}

\begin{Exercise} Show that the linear order of the sureal line is isomorphic to the universal linear order $\mathsf{U}_{2^{<\Ord}}$ on $2^{<\Ord}$ (this gives the so-called sign representation of surreal numbers).
\end{Exercise}


It turns out that the surreal line carries a natural structure of an ordered field.
The operation of addition on $\No$ is defined by the recursive formula:
$$x+y=\big((\cev x+y)\cup(x+\cev y))\curlyvee((\vec x+y)\cup(x+\vec y)).$$
In this formula $x=\langle \cev x,\vec x\rangle$, $y=\langle \cev y,\vec y\rangle$, $\cev x+y=\{z+y:z\in \cev x\}$ etc.

\begin{exercise} Check that $\mathtt{0+0=0}$.
\smallskip

\noindent{\em Solution:} $\mathtt{0+0}=\langle \emptyset,\emptyset\rangle+\langle\emptyset,\emptyset\rangle=\\ ((\cev{\mathtt 0}+\mathtt 0)\cup(\mathtt 0+\cev{\mathtt 0}))\curlyvee ((\vec{\mathtt 0}+\mathtt 0)\cup(\mathtt 0+\vec{\mathtt 0}))=((\emptyset+\mathtt 0)\cup (\mathtt 0+\emptyset))\curlyvee((\emptyset+\mathtt 0)\cup(\mathtt 0+\emptyset))=\emptyset\curlyvee \emptyset=\mathtt 0$.
\end{exercise}

\begin{exercise} Check that $\mathtt{0+1=1}$.
\smallskip

\noindent{\em Solution:} $\mathtt{0+1}=\langle \emptyset,\emptyset\rangle+\langle\mathtt \{0\},\emptyset\rangle=((\cev{\mathtt 0}+\mathtt 1)\cup(\mathtt 0+\cev{\mathtt1}))\curlyvee((\vec{\mathtt 0}+\mathtt 1)\cup(\mathtt 0+\vec{\mathtt 1}))=\\ ((\emptyset+\mathtt 1)\cup(\mathtt 0+\mathtt\{0\}))\curlyvee((\emptyset+\mathtt 1)\cup(\mathtt 0+\emptyset))=\{\mathtt 0\}\curlyvee\emptyset=\mathtt 1$.
\end{exercise}

\begin{exercise} Check that $\mathtt{1+1=2}$.
\smallskip

\noindent{\em Solution:} $\mathtt{1+1}=\langle \{\mathtt 0\},\emptyset\rangle+\langle\mathtt \{0\},\emptyset\rangle=((\cev{\mathtt 1}+\mathtt 1)\cup(\mathtt 1+\cev{\mathtt 1}))\curlyvee((\vec{\mathtt 1}+\mathtt 1)\cup(\mathtt 1+\vec{\mathtt 1}))=\\ ((\{\mathtt 0\}+\mathtt 1)\cup(\mathtt 1+\{\mathtt 0\}))\curlyvee((\emptyset+\mathtt 1)\cup(\mathtt 1+\emptyset)) =(\{\mathtt{0+1,1+0}\}\curlyvee\emptyset=\{\mathtt 1\}\curlyvee\emptyset=\mathtt 2$.
\end{exercise}

By transfinite induction the following properties of the addition can be established.

\begin{proposition}\label{p:add-sur} For every numbers $x,y,z\in\No$ we have
\begin{enumerate}
\item[\textup{1)}] $x+\mathtt 0=x$;
\item[\textup{2)}] $x+y=y+x$;
\item[\textup{3)}] $x+(y+z)=(x+y)+z$;
\item[\textup{4)}] $x<_\No y\;\Ra\;x+z<_\No y+z$.
\end{enumerate}
\end{proposition}

To introduce the subtraction of Conway's numbers, for every surreal number $x\in\No$ consider its inverse $-x$ defined by the recursive formula $-x=-\langle \cev x,\vec x\rangle:=\{-z:z\in \vec x\}\curlyvee\{-z:z\in \cev x\}\rangle$.

\begin{exercise} Show that $-\mathtt 0=\mathtt 0$.
\smallskip

\noindent{\em Solution}: $-\mathtt 0=-\langle \emptyset,\emptyset\rangle=\{- z:z\in\emptyset\}\curlyvee\{- z:z\in\emptyset\}=\emptyset\curlyvee\emptyset=\mathtt0$.
\end{exercise}

\begin{example} Show that $-\mathtt 1=\minus \mathtt 1$.
\smallskip

\noindent{\em Solution}:  $-\mathtt 1=-\langle \{\mathtt 0\},\emptyset\rangle=\{-z:z\in\emptyset\}\curlyvee\{-z:z\in \{\mathtt 0\}\}=\emptyset\curlyvee\{-\mathtt 0\}=\emptyset\curlyvee\{\mathtt 0\}=\minus1$.
\end{example}

The following proposition can be proved by transfinite induction.

\begin{proposition}\label{p:minus-sur} For every number $x\in\No$ we have $x+(-x)=\mathtt 0$.
\end{proposition}

Proposition~\ref{p:add-sur} and \ref{p:minus-sur} imply that the surreal line $\No$ endowed with the operation of addition has the structure of an ordered commutative group.




The multiplication of surreal numbers is defined by the recursive formula
$$xy=L\curlyvee R$$where
$$
\begin{aligned}
&L=\{\dot xy+x\dot y-\dot x\dot y:\dot x\in \cev x,\;\dot y\in \cev y\}\cup\{\ddot xy+x\ddot y-\ddot x\ddot y:\ddot x\in \vec x,\;\ddot y\in \vec y\},\\
&R=\{\dot xy+x\ddot y-\dot x\ddot y:\dot x\in \cev x,\;\ddot y\in \vec y\}\cup\{x\dot y+\ddot xy-\ddot x\dot y:\ddot x\in \vec x,\;\dot y\in \cev y\}.
\end{aligned}
$$

\begin{exercise} Prove that $\mathtt 0\cdot\mathtt 0=\mathtt 0$, $\mathtt 1\cdot\mathtt 1=\mathtt 1$, $\mathtt 1\cdot\mathtt 2=\mathtt 2$.
\end{exercise}

\begin{exercise} Prove that $\mathtt 2\cdot \mathtt 2=\mathtt 4$.
\smallskip

\noindent{\em Solution}: $\mathtt 2\cdot\mathtt 2=\{\mathtt{1\cdot 2+ 2\cdot 1- 1\cdot 1}\}\curlyvee \emptyset=\{\mathtt 3\}\curlyvee\emptyset=\mathtt 4$.
\end{exercise}

\begin{Exercise} Prove that the multiplication of surreal numbers is well-defined and has the following properties for any $x,y,z\in\No$:
\begin{enumerate}
 \item $x\cdot(y\cdot z)=(x\cdot y)\cdot z$;
\item $x\cdot y=y\cdot x$;
\item $x\cdot \mathtt 1=x$;
\item $x\cdot(y+z)=(x\cdot y)+(x\cdot z)$;
\item $(0<_\No x\;\wedge\;0<_\No y)\;\Ra\;0<_\No x\cdot y$.
\end{enumerate}
\smallskip

\noindent{\em Hint}: See \cite[pp.19--20]{ONAG}.
\end{Exercise}

\begin{Exercise} Prove that the surreal line is real closed in the sense that for every odd number $n\in\w$ and any surreal numbers $a_0,\dots,a_n$ with $a_n\ne\mathtt 0$ there exists a surreal number $x$ such that 
$$a_0+a_1x+\dots+a_nx^n=\mathtt 0.$$
In particular, for any nonzero number $x\in\No$ there exists a unique surreal number $y\in\No$ such that $x\cdot y=\mathtt 1$.
\smallskip

\noindent{\em Hint}: See \cite[Theorem 25]{ONAG}.
\end{Exercise}

\newpage

\part{Mathematical Structures}


{\small\em

\rightline{Mathematics has less than ever been reduced to}

\rightline{a purely mechanical game of isolated formulas;}

\rightline{more than ever does intuition dominate}

\rightline{in the genesis of discoveries.}

\rightline{But henceforth, it possesses the powerful tools}

\rightline{furnished by the theory of the great types of structures;}

\rightline{in a single view, it sweeps over immense domains,}

\rightline{now unified by the axiomatic method,}

\rightline{but which were formerly in a completely chaotic state.}
\medskip
}

\rightline{\small ``The Architecture of Mathematics''}
\rightline{\small Bourbaki,  1950}
\vskip25pt

According to Nicolas Bourbaki\footnote{{\bf Task:} Read about Bourbaki in Wikipedia}, mathematics is a science about mathematical structures. Many mathematical structures are studied by various areas of mathematics: graphs, partially ordered sets, semigroups, monoids, groups, rings, ordered fields, topological spaces, topological groups, topological vector spaces, Banach spaces, Banach algebras, Boolean algebras, etc etc. 

A particular type of a mathematical structure is determined by a list of axioms $\A$ it should satisfy. By an axiom in the list $\A$ we understand any formula $\varphi(x,s,C_1,\dots,C_n)$ in the language of CST with free variables $x,s$, and parameters $C_1,\dots,C_n$ which are some fixed classes.
\medskip

\noindent{\bf Definition.} A \index{mathematical structure}{\em mathematical structure} satisfying a list of axioms $\A$ is any pair of classes $(X,S)$ such that for every axiom $\varphi(x,s,C_1,\dots,C_n)$ in the list $\A$,  the formula $\varphi(X,S,C_1,\dots,C_n)$ is true.  The classes $X$ and $S$ are called respectively the \index{structure}{\em underlying class} and the {\em structure} of the mathematical structure $(X,S)$.
\medskip

Mathematical structures satisfying certain specific lists of axioms have special names and are studied by the corresponding fields of mathematics. 

\section{Examples of Mathematical Structures}

In this section we present some important examples of mathematical structures, studies in various areas of mathematics.

\subsection{Set Theory}

\begin{example} The structure of a set can be considered as a mathematical structure $(x,S)$ with empty structure $S=\emptyset$, i.e., sets are mathematical structures without structure. 
\end{example}

Another option is to endow the set $x$ with the structure $S\defeq \mathbf E{\restriction}x$ of membership between its elements.

\subsection{Graph Theory}

\begin{example} A \index{graph}\index{mathematical structure!graph}{\em graph} is a mathematical structure $(X,S)$ satisfying the axiom
\begin{itemize}
\item $S\subseteq\{\{u,v\}:u,v\in X\}$.
\end{itemize} 
\end{example} 

The list $\A$ of axioms for the structure of a graph consists of the unique formula $\varphi(x,s)$:
$$\forall z\;(z\in s\;\Rightarrow\;\exists u\;\exists v\;(u\in x\;\wedge\;v\in x\;\wedge\;\forall w\;(w\in z\;\Leftrightarrow\;(w=u\;\vee\;w=v))))$$
expressing the fact that elements of $s$ are unordered pairs of elements of $x$. A less formal way of writing this formula is $\forall z\in s\;\exists u\in x\;\exists v\in x\;(s=\{u,v\})$.
\smallskip

For our next examples of mathematical structures we shall use such shorthand versions of formulas in the axiom lists.

\begin{example} A  \index{directed graph}\index{mathematical structure!directed graph}\index{digraph}\index{mathematical structure!digraph}{\em directed graph} (or else a {\em digraph}) is a mathematical structure $(X,S)$ satisfying the axiom
\begin{itemize}
\item $S\subseteq X\times X$.
\end{itemize}
A directed graph $(X,S)$ is {\em simple} if $\forall x\in X\;\;(\langle x,x\rangle\notin S)$.
\end{example}

\subsection{Order Theory}

\begin{example} An {\em ordered class}\index{ordered class}\index{mathematical structure!ordered class} is a mathematical structure $(X,S)$ consisting of a class $X$ and an order $S\subseteq X\times X$. The list of axioms determining this mathematical structure consists of three axioms:
\begin{itemize}
\item $S\subseteq X\times X$;
\item $S\cap S^{-1}\subseteq\Id$;
\item $S\circ S\subseteq S$.
\end{itemize}
An {\em ordered set}\index{ordered set}\index{mathematical structure!ordered set} is an ordered class $(X,S)$ such that $X\in\UU$.
\end{example}

\begin{example} A {\em partially ordered class}\index{partially ordered class}\index{mathematical structure!partially ordered class} is a mathematical structure $(X,S)$ consisting of a class and an order $S$ such that $\Id{\restriction}X\subseteq S\subseteq X\times X$. The list of axioms determining this mathematical structure consists of three axioms:
\begin{itemize}
\item $S\subseteq X\times X$;
\item $S\cap S^{-1}=\Id{\restriction}X$;
\item $S\circ S\subseteq S$.
\end{itemize}
A {\em partially ordered set} (briefly, a {\em poset})\index{partially ordered set}\index{mathematical structure!partially ordered set}\index{poset}\index{mathematical structure!poset} is a partially ordered class $(X,S)$ such that $X\in\UU$.
\end{example}
 
\begin{example} A {\em linearly ordered class}\index{linearly ordered class}\index{mathematical structure!linearly ordered class} is a mathematical structure $(X,S)$ satisfying the axioms
\begin{itemize}
\item $S\subseteq X\times X\subseteq S\cup S^{-1}\cup\Id$;
\item $S\cap S^{-1}\subseteq\Id$;
\item $S\circ S\subseteq S$.
\end{itemize}
A {\em linearly ordered set}\index{linearly ordered set}\index{mathematical structure!linearly ordered set} is a linear ordered class $(X,S)$ such that $X\in\UU$.
\end{example}

\begin{example} A {\em well-ordered class}\index{well-ordered class}\index{mathematical structure!well-ordered class} is a mathematical structure $(X,S)$ satisfying the axioms:
\begin{itemize}
\item $S\subseteq X\times X\subseteq S\cup S^{-1}\cup\Id$;
\item $S\cap S^{-1}\subseteq\Id$;
\item $S\circ S\subseteq S$;
\item $\forall Y\;(\emptyset\ne Y\subseteq X\;\Ra\;\exists y\in Y\;\forall x\in X\;(\langle x,y\rangle\in S\;\Ra\;x=y))$.
\end{itemize}
A {\em well-ordered set}\index{well-ordered set}\index{mathematical structure!well-ordered set} is a well-ordered class $(X,S)$ such that $X\in\UU$.
\end{example}

These mathematical structures relate as follows:
$$\mbox{well-ordered class $\Ra$ linearly ordered class $\Ra$ ordered class $\Ra$ directed graph}.$$

\begin{exercise} Find examples of:
\begin{enumerate}
\item an ordered set which is not partially ordered;
\item an ordered class which is not an ordered set;
\item a partially ordered set which is not linearly ordered;
\item a linearly ordered set which is not well-ordered.
\item a well-ordered class which is not a well-ordered set.
\end{enumerate}
\end{exercise}

For two directed graphs $(X,S_X)$, $(Y,S_Y)$ a function $f:X\to Y$ is called \index{function!increasing}\index{increasing function}{\em increasing} if $$\forall x\in X\;\forall y\in Y\;\;(\langle x,y\rangle\in S_X\setminus\Id\;\Ra\;\langle f(x),f(y)\rangle\in S_Y\setminus \Id).$$

\subsection{Algebra}

In this subsection we present examples of some elementary mathematical structures arising in Algebra.

\begin{example} A {\em magma}\index{magma}\index{mathematical structure!magma} is a mathematical structure $(X,S)$ such that $S$ is a function with $\dom[S]=X\times X$ and $\rng[S]\subseteq X$. This mathematical structure is determined by two axioms:
\begin{itemize}
\item $\forall t\;(t\in S\;\Rightarrow\;\exists x\;\exists y\;\exists z\;(x\in X\;\wedge\;y\in X\;\wedge\;z\in X\;\wedge \;(\langle \langle x,y\rangle,z\rangle=t))$;
\item $\forall x\;\forall y\;\;(x\in X\;\wedge\;y\in Y\;\Rightarrow\;\exists z\;(z\in X\;\wedge\; \langle \langle x,y\rangle, z\rangle\in S))$;
\item $\forall x\;\forall y\;\forall u\;\forall v\;(\langle \langle x,y\rangle, u\rangle\in S\;\wedge\;\langle\langle x,y\rangle,v\rangle\in S\;\Ra\;u=v)$.
\end{itemize}
\end{example}

The structure $S$ of a magma $(X,S)$ is called a \index{binary operation}{\em binary operation on $X$}. Binary operations are usually denoted by symbols: $+,\cdot,*,\star$, etc.

\begin{definition} For two magmas $(X,M_X)$ and $(Y,M_Y)$ a function $f:X\to Y$ is called a \index{magma homomorphism}{\em magma homomorphism} if $\forall x\in X\;\forall y\in X\;\;M_Y(f(x),f(y))=f(M_X(x,y))$.
\end{definition}

\begin{example} A {\em commutative magma}\index{commutative magma}\index{mathematical structure!commutative magma} is a magma $(X,S)$ such that
\begin{itemize}
\item $\forall x\;\forall y\;S(x,y)=S(y,x)$.
\end{itemize}
\end{example}

\begin{example} A \index{unital magma}\index{mathematical structure!unital magma}{\em unital magma} is a magma $(X,S)$ possessing a {\em two-sided unit}, which is a unique element $e\in X$ such that $S(x,e)=x=S(e,x)$.
\end{example}

\begin{example} A \index{quasigroup}\index{mathematical structure!quasigroup}{\em quasigroup} is a magma $(X,S)$ such that for every $a,b\in X$ there exist unique elements $x,y\in X$ such that $S(a,x)=b=S(y,a)$.
\end{example}

\begin{example} A \index{loop}\index{mathematical structure!loop}{\em loop} is a unital quasigroup.
\end{example}

\begin{exercise} Find an example of a magma, which is
\begin{enumerate}
\item commutative;
\item not commutative;
\item unital;
\item not unital;
\item a loop;
\item a quasigroup but not a loop.
\end{enumerate}
\end{exercise}

\begin{example} A {\em semigroup}\index{semigroup}\index{mathematical structure!semigroup} is a magma $(X,S)$ whose binary operation $S$ is \index{binary operation!associative}{\em associative} in the sense that $S(S(x,y),z)=S(x,S(y,z))$ for all $x,y,z\in X$.
\end{example}

\begin{example} A {\em regular semigroup}\index{regular semigroup}\index{mathematical structure!regular semigroup} is a semigroup $(X,S)$ such that
\begin{itemize}
\item $\forall x\in X\;\exists y\in X\;\;(S(S(x,y),x)=x\;\wedge\;S(S(y,x),y)=y)$.
\end{itemize}
\end{example}

\begin{example} An {\em inverse semigroup}\index{inverse semigroup}\index{mathematical structure!inverse semigroup} is a semigroup $(X,S)$ such that for every $x\in X$ there exists a unique element $y\in X$ such that $S(S(x,y),x)=x$ and $S(S(y,x),y)=y$. This unique element $y$ is denoted by $x^{-1}$.
\end{example}

\begin{example} A {\em Clifford semigroup}\index{Clifford semigroup}\index{mathematical structure!Clifford semigroup} is a regular semigroup $(X,S)$ such that
\begin{itemize}
\item $\forall x\in X\;(S(x,x)=x\;\Ra\;\forall y\in X\;(S(x,y)=S(y,x))$. 
\end{itemize}
\end{example}

\begin{Exercise} Prove that every Clifford semigroup is inverse.
\end{Exercise}

\begin{example} A {\em monoid}\index{monoid}\index{mathematical structure!monoid} is a unital semigroup.
\end{example}

\begin{example} A {\em group}\index{group}\index{mathematical structure!group} is an associative loop.
\end{example}

For these algebraic structures we have the implications:
$$
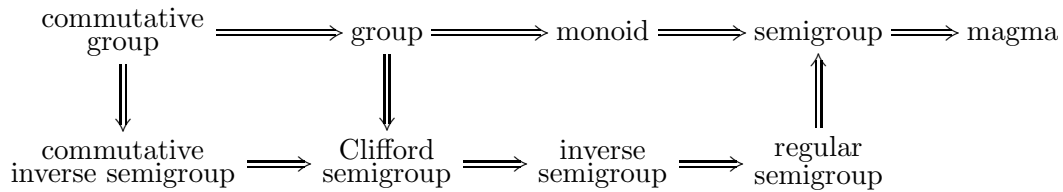
\xymatrix{
\mbox{commutative}\atop\mbox{quasigroup}\ar@{=>}[r]&\mbox{quasigroup}\\
\mbox{commutative}\atop\mbox{loop}\ar@{=>}[u]\ar@{=>}[r]&\mbox{loop}\ar@{=>}[u]\ar@{=>}[r]&\mbox{unital magma}\ar@{=>}[r]&\mbox{magma}\\
\mbox{commutative}\atop\mbox{group}\ar@{=>}[u]\ar@{=>}[r]\ar@{=>}[d]&\mbox{group}\ar@{=>}[d]\ar@{=>}[r]\ar@{=>}[u]&\mbox{monoid}\ar@{=>}[r]\ar@{=>}[u]&\mbox{semigroup}\ar@{=>}[u]\\
\mbox{commutative}\atop\mbox{inverse semigroup}\ar@{=>}[r]&\mbox{Clifford}\atop\mbox{semigroup}\ar@{=>}[r]&\mbox{inverse}\atop\mbox{semigroup}\ar@{=>}[r]&\mbox{regular}\atop\mbox{semigroup}\ar@{=>}[u]
}
$$

\begin{exercise} Find examples of:
\begin{enumerate}
\item a magma which is not a semigroup;
\item a semigroup which is not a monoid;
\item a monoid which is not a group;
\item a group which is not a commutative group;
\item a semigroup which is not regular;
 \item a regular semigroup which is not inverse;
\item an inverse semigroup which is not Clifford;
\item a Clifford semigroup which is not a group;
\item a loop which is not a group.
\end{enumerate}
\end{exercise}

\subsection{Group actions}  According to Klein’s Erlangen Program\footnote{{\bf Task:} Read about Klein's Erlangen Program in Wikipedia.} (1872), geometry studies spaces together with their symmetry groups. In this section we introduce $G$-spaces and permutation groups which are mathematical structures that realize this idea.

\begin{example} Let $(G,\cdot,1)$ be a group with identity element $1$. A \index{left $G$-space}\index{mathematical structure!left $G$-space}{\em left $G$-space} is a mathematical structure $(X,S)$ such that $S$ is a function with $\dom[S]=G\times X$ and $\rng[S]=X$ satisfying the axioms:
\begin{itemize}
\item $\forall g,h\in G\;\forall x\in X\;\;S(g,S(h,x))=S(g{\cdot}h,x)$;
\item $\forall x\in X\;\;S(1,x)=x$.
\end{itemize}
\end{example}

\begin{example} Let $(G,\cdot,1)$ be a group with identity element $1$. A \index{right $G$-space}\index{mathematical structure!right $G$-space}{\em right $G$-space} is a mathematical structure $(X,S)$ such that $S$ is a function with $\dom[S]=X\times G$ and $\rng[S]=X$ satisfying the axioms:
\begin{itemize}
\item $\forall x\in X\;\forall g,h\in G\;\;S(S(x,g),h)=S(x,g{\cdot}h)$;
\item $\forall x\in X\;\;S(x,1)=x$.
\end{itemize}
\end{example}

\begin{exercise} Show that for every left $G$-space $(X,L)$, the mathematical structure $(X,R)$ endowed with the function $R:X\times G\to X$, $R:(x,g)\mapsto L(g^{-1},x)$ is a right $G$-space.
\end{exercise}

\begin{example} A \index{permutation group}\index{mathematical structure!permutation group}{\em permutation group} is a mathematical structure $(X,G)$ whose structure $G$ is a subgroup of the symmetric group $\mathrm{Sym}(X)$ of all bijections of $X$, endowed with the operation of composition.
\end{example}

\begin{exercise} Let $(X,G)$ be a permutation group. Consider the functions $L:G\times X\to X$, $L:(g,x)\mapsto g(x)$, and $R:X\times G\to X$, $R:(x,g)\mapsto g^{-1}(x)$. Show that $(X,L)$ is a left $G$-space  and $(X,R)$ is a right $G$-space.
\end{exercise}

\subsection{Dynamical systems}

\begin{example} A \index{dynamical system}\index{mathematical structure!dynamical system}{\em dynamical system} is a mathematical structure $(X,S)$ with $S$ being a function $S:X\to X$. 
\end{example}

\begin{remark} The evolution of a dynamical system $(X,S)$ is described by iterations $S^{\circ n}$ of the function $S$, producing for every $x\in X$ the sequence $(S^{\circ n}(x))_{n\in\omega}$ called the {\em orbit} of the point $x$. In many applications the set 
$X$ carries additional structure, for example a topology or a probability measure, and the function $S$ is required to respect that structure (e.g., to be continuous or measure-preserving).
\end{remark}

\subsection{Geometry} In this section we present some mathematical structures studied in (synthetic) geometry. 

\begin{example} A \index{liner}\index{mathematical structure!liner}{\em liner} is a mathematical structure $(X,S)$ satisfying the axioms:
\begin{itemize}
\item $S\subseteq\mathcal P(X)$;
\item $\forall x,y\in X\;(x\ne y\Ra \exists !L\in S\;\;(\{x,y\}\subseteq L))$;
\item $\forall L\in S\;\exists x,y\in X\;(x\ne y\;\wedge\;\{x,y\}\subseteq L)$.
\end{itemize}
\end{example}

\begin{exercise} Find a natural structure of a liner on the Euclidean plane $\IR\times\IR$.
\end{exercise}

\begin{example} A \index{betweenness space}\index{mathematical structure!betweenness space}{\em betweenness space} (briefly, a \index{betspace}\index{mathematical structure!betspace}{\em betspace})  is a mathematical structure $(X,B)$ satisfying three axioms:
\begin{itemize}
\item $B\subseteq X^3$;
\item $\forall a,x,y\in X\;\;\{\langle x,y,a\rangle,\langle y,x,a\rangle\}\subseteq B\Leftrightarrow (x=y)\Leftrightarrow \{\langle a,x,y\rangle,\langle a,y,x\rangle\}\subseteq B$;
\item $\forall a,b,x,y,z\in X\;\{\langle a,x,b\rangle,\langle a,y,b\rangle,\langle x,z,y\rangle\}\subseteq B\;\Ra\;\langle a,z,b\rangle\in B$.
\end{itemize}
\end{example}

\begin{example} A \index{Pasch betliner}\index{mathematical structure!betweenness space}{\em Pasch betliner} is a mathematical structure $(X,B)$ satisfying seven axioms:
\begin{itemize}
\item $B\subseteq X^3$;
\item $\forall a,x,y\in X\;\;\{\langle x,y,a\rangle,\langle y,x,a\rangle\}\subseteq B\Leftrightarrow (x=y)\Leftrightarrow \{\langle a,x,y\rangle,\langle a,y,x\rangle\}\subseteq B$;
\item $\forall a,b,x,y\in X\;\{\langle a,x,b\rangle,\langle a,y,b\rangle\}\subseteq B\;\Ra\; (\langle a,x,y\rangle\in B\;\vee \;\langle a,y,x\rangle\in B)$;
\item $\forall a,b,x,y\in X\;\big((a\ne b)\wedge \{\langle a,b,x\rangle,\langle a,b,y\rangle\}\subseteq B \big)\;\Ra(\langle a,x,y\rangle\in B\;\vee\;\langle a,y,x\rangle\in B)$;
\item $\forall a,b,x,y\in X\;((a\ne b)\wedge \{\langle x,a,b\rangle,\langle a,b,y\rangle\}\subseteq B)\;\Ra\;\{\langle x,a,y\rangle,\langle x,b,y\rangle\}\subseteq B$;
\item $\forall a,b,c,v,p\in X\;\; \{\langle a,b,c\rangle,\langle v,p,c\rangle\}\subseteq B\;\Ra\;\exists x\;\{\langle p,x,a\rangle,\langle v,x,b\rangle\}\subseteq B$;
\item  $\forall a,b,c,v,p\in X\; \;\{\langle a,b,c\rangle,\langle a,p,b\rangle\}\subseteq B\;\Ra\;\exists x\;\{\langle x,p,a\rangle,\langle v,x,c\rangle\}\subseteq B$.
\end{itemize}
\end{example}

\begin{exercise} Show that every Pasch betliner is a betspace.
\end{exercise}

\begin{exercise} Find a natural structure of a Pasch betliner on the Euclidean plane $\IR\times\IR$.
\end{exercise}

\begin{example} A \index{Tarski geometry}\index{mathematical structure!Tarski geometry}{\em Tarski geometry} is a mathematical structure $(X,(B,E))$ satisfying the axioms:
\begin{itemize}
\item $B\subseteq X^3$;
\item $E\subseteq (X\times X)\times(X\times X)$;
\item $\forall a,b\in X\;\langle \langle a,b\rangle,\langle b,a\rangle\rangle\in E$; 
\item  $\forall a,b,c,d,e,f\;\{\langle \langle a,b\rangle,\langle c,d\rangle\rangle,\langle\langle c,d\rangle,\langle e,f\rangle\rangle\}\subseteq E\;\Ra\;\langle\langle a,b\rangle,\langle e,f\rangle\rangle\in E$;
\item $\forall a,b,c\;(\langle \langle a,b\rangle,\langle c,c\rangle\rangle\in E\;\Ra a=b)$;
 \item $\forall x,y,a,b\in X\;\exists z\in X\;(\langle x,y,z\rangle\in B\;\wedge\;\langle\langle y,z\rangle,\langle a,b\rangle\rangle\in E)$;
\item $\forall a,p,b,q,c\in X\;\{\langle a,p,b\rangle,\langle b,q,c\rangle\}\subseteq B\;\Rightarrow\;\exists x\in X\;\{\langle a,x,q\rangle,\langle p,x,c\rangle\}\subseteq B$;
\item $\forall \hat x,\hat y,\hat z,\hat v,\check x,\check y,\check z,\check v\in X\;\;\big((\hat x\ne \hat y\,\wedge\,\{\langle \hat x,\hat y,\hat z\rangle,\langle\check x,\check y,\check z\rangle\}\subseteq B\;\wedge$\newline $\{\langle\langle\hat x,\hat y\rangle,\langle\check x,\check y\rangle\rangle,\langle\langle\hat y,\hat z\rangle,\langle\check y,\check z\rangle\rangle,\langle\langle\hat v,\hat x\rangle,\langle\check v,\check x\rangle\rangle,\langle\langle\hat v,\hat y\rangle,\langle\check v,\check y\rangle\rangle\}\subseteq E)\,\Rightarrow\,\langle\langle\hat v,\hat z\rangle,\langle\check v,\check z\rangle\rangle{\in} E$;
\item $\forall a,b,c,d,t\; ((a\ne t)\wedge\{\langle a,t,c\rangle,\langle b,t,d\rangle\}\subseteq B)\Ra\exists x,y\; 
\{\langle a,b,x\rangle,\langle a,d,y\rangle,\langle x,c,y\rangle\}\subseteq B.$
\end{itemize}
\end{example}

\begin{Exercise} Prove that for every Tarski geometry $(X,(B,E))$, the mathematical structure $(X,B)$ is a Pasch betliner.
\smallskip

{\em Hint:} See \cite{SST} and \cite{RTG}.
\end{Exercise}

\begin{example} A  \index{rope space}\index{mathematical structure!rope space}{\em rope space} is a mathematical structure $(X,S)$ satisfying the axioms:
\begin{itemize}
\item $S\subseteq (X\times X)\times (X\times X)$;
\item $\forall a,b,u,v,x,y\in X\; \{\langle \langle a,b\rangle,\langle u,v\rangle\rangle,\langle \langle u,v\rangle,\langle x,y\rangle\rangle\}\subseteq S\;\Rightarrow\;\langle \langle a,b\rangle,\langle x,y\rangle\rangle\in S$;
\item $\forall x,y\in X\;\;\langle\langle x,y\rangle,\langle y,x\rangle\rangle\in S$;
\item $\forall x,y,z\in X\;\;\langle\langle z,z\rangle,\langle x,y\rangle\rangle\in S$;
\item $\forall x,y,z\in X\;\;\langle\langle x,y\rangle,\langle z,z\rangle\rangle \in S\;\Rightarrow\;x=y$.
\end{itemize}
\end{example}

\begin{exercise} Given a Tarski Geometry $(X,(B,E))$, consider the relation
$$S\defeq\{\langle \langle x,y\rangle,\langle a,b\rangle\rangle:\exists c\;\;(\langle \langle x,y\rangle,\langle a,c\rangle\rangle\in E\;\wedge\;\langle a,c,b\rangle\in B)\}.$$Prove that $(X,S)$ is a rope space.
\end{exercise}
 
\begin{example} A {\em metric space}\index{metric space}\index{mathematical structure!metric space} is a mathematical structure $(X,S)$ satisfying the axioms
\begin{itemize}
\item $S$ is a function with $\dom[S]=X\times X$ and $\rng[S]\subseteq\IR$;
\item $\forall x\in X\;\forall y\in X\;(S(x,y)=0\;\Leftrightarrow\;x=y)$;
\item $\forall x\in X\;\forall y\in X\;S(x,y)=S(y,x)$;
\item $\forall x\in X\;\forall y\in X\;\forall z\in X\;(S(x,z)\le S(x,y)+S(y,z))$.
\end{itemize}
\end{example}

\begin{exercise} Find a canonical rope structure on each metric space $(X,\rho)$.

{\em Hint:} $S\defeq\{\langle\langle a,b\rangle,\langle x,y\rangle\rangle:\rho(a,b)\le \rho(x,y)\}$.
\end{exercise}

Two rope spaces $(X,S_X)$ and $(Y,S_Y)$ are {\em isomorphic} if there exists a bijection $f:X\to Y$ such that $S_Y=\{\langle \langle f(x),f(y)\rangle,\langle f(u),f(v)\rangle\rangle:\langle \langle x,y\rangle,\langle u,v\rangle\rangle\in S_X\}$.

\begin{Exercise} Characterize rope spaces, isomorphic to the Euclidean plane endowed with its natural rope structure.
\smallskip

{\em Hint:} See \cite{BFG}.
\end{Exercise}

\subsection{Measure Theory} Observe that the axioms of a metric space contain as parameters the following sets: the set of real numbers $\IR$, the natural number $0=\emptyset$, the linear order $\le$ of the real line, and the addition operation $+:\IR\times\IR\to\IR$ on the real line.

Another structure whose definition involves the real line as a parameter is the structure of a measure space.

\begin{example} A {\em measure space}\index{measure space}\index{mathematical structure!measure space} is a mathematical structure $(X,S)$ satisfying the following axioms
\begin{itemize}
\item $S$ is a function with $\dom[S]\subseteq \mathcal P(X)$ and $\rng[S]\subset[0,\infty)\subset\IR$;
\item $\forall u\;\forall v\;(u\in \dom[S]\;\wedge\;v\in\dom[S]\;\Ra\;(u\setminus v\in \dom[S]\;\wedge\;u\cup v\in\dom[S])$;
\item $\forall u\;\forall v\;\big((u\in\dom[S]\;\wedge\;v\in\dom[S]\;\wedge\;u\cap v=\emptyset)\;\Ra\;S(u\cup v)=S(u)+S(v)\big)$.
\end{itemize}
\end{example}

\subsection{Topology}

\begin{example} A {\em topological space}\index{topological space}\index{mathematical structure!topological space} is a mathematical structure $(X,S)$ satisfying the following axioms:
\begin{itemize}
\item $\{\emptyset,X\}\subseteq S\subseteq\mathcal P(X)$;
\item $\forall x\;\forall y\;(x\in S\;\wedge\;y\in S\;\Ra\;x\cap y\in S)$;
\item $\forall u\;(u\subseteq S\;\Ra\;\bigcup u\in S)$.
\end{itemize}
\end{example}

For two topological spaces $(X,S_X)$, $(Y,S_Y)$ a function $f:X\to Y$ is called \index{function!continuous}\index{continuous function}{\em continuous} if $\forall u\;(u\in S_Y\;\Ra\;f^{-1}[u]\in S_X)$.

\begin{example} A {\em bornological space}\index{bornological space}\index{mathematical structure!bornological space} is a mathematical structure $(X,S)$ satisfying the following axioms:
\begin{itemize}
\item $\bigcup S=X$;
\item $\forall x\;\forall y\;(x\in S\;\wedge\;y\in S\;\Ra\;x\cup y\in S)$;
\item $\forall x\;\forall y\;(x\subseteq y\in S\;\Ra\;x\in S)$.
\end{itemize}
\end{example}

\begin{example} A {\em uniform space}\index{uniform space}\index{mathematical structure!uniform space} is a mathematical structure $(X,S)$ satisfying the following axioms:
\begin{itemize}
\item $\forall u\in S\;(\Id{\restriction}X\subseteq u\subseteq X\times X)$;
\item $\forall u\in S\;\forall v\in S\;\exists w\in S\;(w\circ w\subseteq u\cap v^{-1}$);
\item $\forall u\in S\;\forall v\;(u\subseteq v\subseteq X\times X\;\Ra\;v\in S$).
\end{itemize}
\end{example}

\begin{example} A {\em coarse space}\index{coarse space}\index{mathematical structure!coarse space} is a mathematical structure $(X,S)$ satisfying the following axioms:
\begin{itemize}
\item $\forall u\in S\;(\Id{\restriction}X\subseteq u\subseteq X\times X)$;
\item $\forall u\in S\;\forall v\in S\;\exists w\in S\;(u\circ v^{-1}\subseteq w$);
\item $\forall u\in S\;\forall v\;(\Id{\restriction}X\subseteq v\subseteq u\;\Ra\;v\in S$).
\end{itemize}
\end{example}

\begin{example} A {\em duoform space}\index{duoform space}\index{mathematical structure!duoform space} is a mathematical structure $(X,S)$ satisfying the following axioms:
\begin{itemize}
\item $\forall u\in S\;(\Id{\restriction}X\subseteq u\subseteq X\times X)$;
\item $\forall u\in S\;\forall v\in S\;\exists w\in S\;(w\circ w\subseteq u\cap v^{-1}$);
\item $\forall u\in S\;\forall v\in S\;\exists w\in S\;(u\circ v^{-1}\subseteq w$);
\item $\forall u\in S\;\forall v\in S\;\forall w\;(u\subseteq w\subseteq v\;\Ra\;w\in S$).
\end{itemize}
\end{example}

\begin{exercise} Prove that a duoform space $(X,S)$ is 
\begin{enumerate}
\item a uniform space if and only if $X\times X\in S$;
\item a coarse space if and only if $\Id{\restriction}X\in S$.
\end{enumerate}
\end{exercise} 

\begin{example} A {\em learning space} \index{learning space}\index{mathematical structure!learning space}  is a mathematical structure $(X,S)$ satisfying the following axioms:
\begin{itemize}
\item $\{\emptyset,X\}\subseteq S\subseteq \mathcal P(X)$;
\item $\forall A,B\in S\;\big(A\subset B\;\Rightarrow\;\exists n\in\IN\;\wedge \;\exists f\in X^{n}\;(B=A\cup f[n]\;\wedge\;(\forall k\in n\;A\cup f[k]\in S))\big)$;
\item $\forall A,B\in S\;\forall x\in X\;(A\subset B\;\wedge A\cup\{x\}\in S\;\Rightarrow B\cup\{x\}\in S)$.
\end{itemize}
\end{example}

\begin{exercise} Prove that for any learning space $(X,S)$ the sets $X$ and $S$ are finite. 
\end{exercise}

\begin{remark} More information on learning spaces and their applications in didactics can be found in the book \cite{FD}.
\end{remark}

\subsection{Ordered groups} Complex mathematical structures consists of several substructures which can be related each with the other.  A typical example is the structures of an ordered group.

\begin{example} An {\em ordered group}\index{ordered group}\index{mathematical structure!ordered group} is a mathematical structure $(X,S)$ whose structure $S$ is a pair of classes $(+,<)$ such that $(X,+)$ is a commutative group, and $(X,<)$ is a linearly ordered class such that 
 $\forall x\in X\;\forall y\in X\;\forall z\in X\;(x<y\;\Ra\;x+z<y+z)$.
 An ordered group $(X,(+,<))$ is called \index{Archimedean ordered group}\index{ordered group!Archimedean}{\em Archimedean} if for any elements $0<\e<a$ in $X$ there exists $n\in\IN$ such that $a<n\e$.
\end{example}

\begin{example} Any subgroup of the real line $(\mathbb R,(+,<))$ is an Archimedean ordered group. In particular, $(\mathbb R,(+,<))$ and $(\IQ,(+,<))$ are Archimedean ordered groups.
\end{example}

\begin{exercise} Construct an example of a non-Archimedean ordered group.

\noindent{\em Hint:} Consider the group $\IZ\times\IZ$ endowed with the lexicographic order.
\end{exercise}

\begin{exercise} Prove that an ordered group $(X,(+,<))$ is Archimedean (and trivial) if its linear order $<$ is (boundedly) complete.
\end{exercise}

\subsection{Ordered topological spaces} 
Another example of an important complex mathematical structure is that of partially ordered space.

\begin{example} A {\em partially ordered space}\index{partially ordered space}\index{mathematical structure!partially ordered space} is a mathematical structure $(X,S)$ whose structure $S$ is a pair of sets $(\le,\tau)$ such that $(X,\le)$ is a partially ordered set, $(X,\tau)$ is a topological space and the following conditions are satisfied:
\begin{itemize}
\item $\forall x,y\in X\;\;(x\not\le y\;\Ra\;\exists U,V\in \tau\;\;(\langle x,y\rangle\in U\times V\subseteq (X\times X)\setminus{\le}\;)$.
\end{itemize}
\end{example}

\subsection{Rings} 

\begin{example} A {\em ring}\index{ring}\index{mathematical structure!ring} is a mathematical structure $(X,S)$ whose structure $S$ is a pair of classes $(+,\cdot)$ such that $(X,+)$ is a commutative group, $(X,\cdot)$ is a monoid and the following axiom (called the {\em distributivity}) is satisfied:
\begin{itemize}
\item $\forall x,y,z\in X\;\;(x\cdot(y+z)=(x\cdot y)+(x\cdot z)\;\wedge\;(x+y)\cdot z=(x\cdot z)+(y\cdot z))$.
\end{itemize}
\end{example}

\begin{example} A {\em commutative ring}\index{commutative ring}\index{mathematical structure!commutative ring} is a ring $(X,(+,\cdot))$ such that the monoid $(X,\cdot)$ is commutative.
\end{example}

\begin{example} A {\em field}\index{field}\index{mathematical structure!field} is a commutative ring $(X,(+,\cdot))$ such that 
\begin{itemize}
\item $\forall x\in X\;(x+x\ne x\;\Ra\;\exists y\in X\;\forall z\in X\;z\cdot(x\cdot y)=z))$.
\end{itemize}
\end{example}

\subsection{Ordered fields}

\begin{example}\label{ex:ord-field} An {\em ordered field}\index{ordered field}\index{mathematical structure!ordered field} is a mathematical structure $(X,S)$ whose structure $S$ is a triple of classes $(+,\cdot,<)$ such that $(X,(+,\cdot))$ is a field, $(X,(+,<))$ is an ordered group and $\forall x,y,z\in X\;\;((z+z=z\;\wedge\;z<x\;\wedge\;z<y)\;\Ra\; \langle z<x\cdot y)$.
\end{example}

\begin{exercise} Show that $(\IR,(+,\cdot,<_\IR))$ is an ordered field.
\end{exercise}

\begin{exercise} Show that $(\No,(+,\cdot,<_\No))$ is an ordered field.
\end{exercise}

\begin{definition} An ordered field $(X,(+,\cdot,<))$ is called
\begin{itemize}
\item {\em Archimedean} if the ordered group $(X,(+,<))$ is Archimedean;
\item {\em Pithagorean} if for any $x,y\in X$ there exists $z\in X$ such that $x^2+y^2=z^2$;
\item {\em Euclidean} if for any $x\in X$ with $x>0$ there exists $y\in X$ such that $x=y^2$.
\end{itemize}
\end{definition}

\begin{Exercise}[Hilbert] Prove that the smallest Pithagorean subfield of the real line is not Euclidean.
\smallskip

\noindent{\em Hint:} Show that $1+\sqrt{2}$ does not belong to the smallest Pithagorean subfield of $\IR$.
\end{Exercise}

\subsection{Linear Algebra and Geometry}

\begin{example} A \index{linear space}\index{mathematical structure!linear space}{\em linear space over a field} $(F,(+,\cdot))$ is a mathematical structure $(X,S)$ whose structure $S$ is a pair $(+,\cdot)$ of two maps ${+}:X\times X\to X$, ${+}:\langle x,y\rangle\mapsto x+y$, and ${\cdot}:F\times X\to X$, ${\cdot}:\langle a,x\rangle\mapsto ax$, such that $(X,+)$ is a commutative group and the following axioms are satisfied:
\begin{itemize}
\item $a(x+y)=ax+ay$;
\item $(a+b)x=ax+bx$;
\item $(ab)x=a(bx)$;
\item $1x=x$;
\end{itemize}
for every $a,b\in F$ and $x,y\in X$.
\end{example}

\begin{example}\label{ex:inps} An \index{inner product space}\index{mathematical structure!inner product space}{\em inner product space over an ordered field} $(F,(+,\cdot,<))$ is a mathematical structure $(X,S)$ whose structure $S$ is a triple $(+,\cdot,\langle{\cdot}|{\cdot}\rangle)$ of three functions $+:X\times X\to X$, $\cdot:F\times X\to X$ and $\langle{\cdot}|{\cdot}\rangle:X\times X\to F$, $\langle{\cdot}|{\cdot}\rangle:\langle x,y\rangle\mapsto,\langle x|y\rangle$, such that $(X,(+,\cdot))$ is a linear space of the field $(F,(+,\cdot))$ and the following axioms are satisfied:
\begin{itemize}
\item $x\ne 0\Rightarrow \langle x|x\rangle> 0$;
\item $\langle x|y\rangle=\langle y|x\rangle$;
\item $\langle (x+y)|z\rangle=\langle x|z\rangle+\langle y|z\rangle$;
\item $\langle ax|y\rangle=a\langle x|y\rangle$;
\end{itemize}
for any $a\in F$ and $x,y,z\in X$.
\end{example}

\begin{exercise} Show that every inner product space over an ordered field carries the canonical structure of a Tarski geometry. 
\end{exercise}

\begin{Exercise} Define a notion of isomorphism for Tarski geometries and prove that every Tarski geometry $(X,(B,E))$ with $X\ne \varnothing$ is isomorphic to the Tarski geometry of a suitable inner product space over an ordered field. 
\smallskip

{\em Hint:} See \cite{SST}.
\end{Exercise}

\subsection{Algebraic structures in Logic} In this subsection, we present the most important algebraic structures that arise in formal logic.

\begin{definition}
A  \index{semilattice}\index{mathematical structure!semilattice}\emph{semilattice} is a commutative semigroup $(X,\cdot)$ such that $x\cdot x=x$ for all $x\in X$.
\end{definition}

\begin{exercise} Let $(X,\cdot)$ be a semilattice. Show that a relation $x\le y$ defined by $x\cdot y=x$ is a partial order on $X$ such that for any $x,y$, the element $x{\cdot}y$ coincides with the greatest lower bound $\inf\{x,y\}\defeq\{z\in X:z\le x,\;z\le y\}$ of the set $\{x,y\}$.
\end{exercise}

\begin{exercise} Let $(X,\le)$ be a partially ordered set such that for any elements $x,y\in X$ the set $\{x,y\}$ has the greatest lower bound $\inf\{x,y\}\defeq\{z\in X:z\le x,\;z\le y\}$. Show that $X$ endowed with the binary operation $x\cdot y\defeq \inf\{x,y\}$ is a semilattice.
\end{exercise}

\begin{definition}
A  \index{lattice}\index{mathematical structure!lattice}\emph{lattice} is a mathematical structure $(X,(\wedge,\vee))$ such that $(X,\vee)$ and $(X,\wedge)$ are semilattices satisfying the {\em absorption laws}: $x \wedge (x \vee y) = x$ and $x \vee (x \wedge y) = x$ holding for all $x,y\in X$.
\end{definition}

\begin{exercise} Let $(X,(\wedge,\vee))$ be a lattice. Show that $$\forall x,y\in X\;\;x\wedge y=x\;\Leftrightarrow\;x\vee y=y.$$Define a partial order $\le$ on $X$ by $x\le y$ iff $x\wedge y=x$ iff $x\vee y=y$. Show that for any elements $x,y\in X$, the elements $x\wedge y$ and $x\vee y$ coincide with the greatest lower bound $\inf\{x,y\}$ and the least upper bound $\sup\{x,y\}$ of the set $\{x,y\}$, respectively.
\end{exercise}

\begin{exercise} Let $(X,\le)$ be a partially ordered set such that for any elements $x,y\in X$, the set $\{x,y\}$ has the largest lower bound $\inf\{x,y\}=\max\{z\in X:z\le x,\;\;z\le y\}$ and the least upper bound $\sup\{x,y\}=\min\{z\in X:x\le z,\;y\le z\}$. Prove that the partially ordered set $(X,\le)$ endowed with the binary operations $x\wedge y\defeq \inf\{x,y\}$ and $x\vee y\defeq\sup\{x,y\}$ is a lattice.
\end{exercise}

\begin{definition} A lattice  \index{distributive lattice}\index{mathematical structure!distributive lattice}$(X,(\wedge,\vee))$ is {\em distributive} if $$x\wedge(y\vee z)=(x\wedge y)\vee (x\wedge z)\quad\mbox{and}\quad x\vee(y\wedge z)=(x\vee y)\wedge(x\vee z)$$for all $x,y,z\in X$.
\end{definition}

\begin{exercise} Show that a lattice $(X,(\wedge,\vee))$ is distributive if and only if\\ $\forall x,y,z\in X\;\;x\wedge(y\vee z)=(x\wedge y)\vee (x\wedge z)$ iff $\forall x,y,z\in X\;\;x\vee(y\wedge z)=(x\vee y)\wedge(x\vee z)$.
\smallskip

{\em Hint:} Assuming that $x\wedge (y\vee z)=(x\wedge y)\vee (x\wedge z)$ for all $x,y,z\in X$, observe that $(x\vee y)\wedge (x\vee z)=(x\wedge (x\vee z))\vee (y\wedge(x\vee z))=x\vee y\wedge x)\vee (y\wedge z)=x\vee(y\wedge z)$.
\end{exercise}

\begin{exercise} Find an example of a non-distributive lattice.
\smallskip

{\em Hint:} Consider the set  $X=\{0,1,a,b,c\}$ endowed with the lattice operations $\wedge$ and $\vee$, induced by the partial order $\{(x,x):x\in X\}\cup(\{0\}\times X)\cup(X\times\{1\})$ and observe that $a\wedge(b\vee c)=a\ne 0=(a\wedge b)\vee(a\wedge c)$.
\end{exercise}

\begin{definition} A  \index{bounded lattice}\index{mathematical structure!bounded lattice}{\em bounded lattice} is a mathematical structure $(X,(\wedge, \vee,0,1))$  such that $(X,(\wedge,\vee))$ is a lattice and $0,1\in X$ are two constants such that $0\wedge x=0$ and $1\vee x=1$ for all $x\in X$.
\end{definition}

\begin{definition}A \index{Boolean algebra}\index{mathematical structure!Boolean algebra}\emph{Boolean algebra} is a mathematical structure $X = (X,(\wedge, \vee, \neg, 0, 1))$ such that $(X,(\wedge,\vee,0,1))$ is a bounded distributive lattice and $\neg:X\to X$ is an unary operation such that $x\wedge \neg x=0$ and $x\vee\neg x=1$ for all $x\in X$.
\end{definition}

\begin{exercise} Prove that every Boolean algebra $(X,(\wedge,\vee,\neg,0,1))$ satisfies the {\em De Morgan\footnote{{\bf Task:} read about Augustus De Morgan (1806--1871) in Wikipedia.} laws}: $\neg(x\wedge y)=\neg x\vee \neg y$ and $\neg (x\vee y)=\neg x\wedge\neg y$ for all $x,y\in X$.
\smallskip

{\em Hint:} First prove that the negation reverses the order on a Boolean algebra.
\end{exercise}

\begin{exercise} For any set $X$, the powerset $\mathcal P(X)$ endowed with the operations $x\wedge y\defeq x\cap y$, $x\vee y\defeq x\cup y$, $\neg x\defeq X\setminus x$, is a Boolean algebra.
\end{exercise}

\begin{definition}
A \index{Heyting algebra}\index{mathematical structure!Heyting algebra}\emph{Heyting algebra} is a mathematical structure $(X,(\wedge,\vee,\to,0,1))$ such that $(X,(\wedge,\vee,0,1))$ is a bounded distributive lattice and $\to$ is a binary operation on $X$ such that for all $x,y,z \in H$,
\[
x \wedge y \le z \iff x \le (y \to z),
\]
where the partial order $\le$ is defined by $a \le b$ iff $a \wedge b = a$. 
 
The {\em negation operation} $\neg:X\to X$ on a Heyting algebra $(X,(\wedge,\vee,\to,0,1))$ is defined by $\neg x \defeq x \to 0$.
\end{definition}

\begin{exercise}Given a Boolean algebra $(X,(\wedge,\vee,\neg,0,1))$, consider the implication operation $\to$ on $X$, defined by $x\to y\defeq \neg x\vee y$. Show that $(X,(\wedge,\vee,\to,0,1))$ is a Heyting algebra.
\end{exercise}

\begin{exercise} Given a Heyting algebra $(X,(\wedge,\vee,\to,0,1))$, consider the negation operation $\neg$ on $X$ defined by $\neg x\defeq x\to 0$. Show that $(X,(\wedge,\vee,\neg,0,1))$ is a Boolean algebra if and only if the {\em law of excluded middle} $(\forall x\in X\;x\vee \neg x=1)$ holds if and only if the {\em double negation law} $(\forall x\in X\;\neg \neg x=x)$ holds.  
\end{exercise}

\begin{exercise} Let $(X,\mathcal{T})$ be a topological space. Show that the family $\mathcal{T}$ of open sets endowed with the operations $$U\wedge V\defeq U\cap V,\;\;U\vee V\defeq U\cup V,\;\;U\to V\defeq \textstyle{\bigcup}\{W\in\mathcal {T}:W\subseteq (X\setminus U)\cup V\},$$ and constants $0\defeq\varnothing$ and $1\defeq X$ is a Heyting algebra.
\end{exercise}

\begin{remark} Boolean algebras provide an algebraic model of classical propositional logic, where the operations $\wedge,\vee,\neg$ correspond to AND, OR, and NOT. Heyting algebras play a similar role for intuitionistic logic, capturing constructive reasoning where the law of excluded middle may fail. For further reading on the algebraic and semantic aspects of classics and intuitionistic logics, see the textbooks \cite{Mendelson}, \cite{MintsIntuitionistic}, \cite{PriestNonClassical}, \cite{SmithLogic}. 
\end{remark}

\subsection{Universal Algebra}

A quite general algebraic structure is that of universal algebra of a given signature.

\begin{definition} Let $\sigma$ be a function with $\rng[\sigma]\subseteq\w$. A {\em universal algebra of signature $\sigma$}\index{universal algebra}\index{mathematical structure!universal algebra} is a mathematical structure $(X,S)$ whose structure is an indexed family of classes $(S_i)_{i\in\dom[\sigma]}$ such that for every $i\in\dom[\sigma]$, $S_i$ is a function with $\dom[S_i]=X^{\sigma(i)}$ and $\rng[S_i]\subseteq X$.
\end{definition}

Even more general is a universal relation structure of a given signature.

\begin{definition} Let $\sigma$ be a function with $\rng[\sigma]\subseteq\w$. A {\em universal relation structure of signature $\sigma$}\index{universal relation structure}\index{mathematical structure!universal relation structure} is a mathematical structure $(X,S)$ whose structure is an indexed family $(S_i)_{i\in\dom[\sigma]}$ such that for every $i\in\dom[\sigma]$, $S_i\subseteq X^{\sigma(i)}$.
\end{definition}

Finally we consider the structure of an objectless category.

\begin{definition}\label{d:objectless} An \index{objectless category}\index{category!objectless}{\em objectless category} is a mathematical structure $(X,S)$ whose underlying class $X$ is called the {\em class of arrows} of the objectless category and the structure $S$ is a triple $(\fdom,\fran,\circ)$ consisting of functions $\fdom:X\to X$, $\fran:X\to X$ assigning to each arrow $x\in X$ its {\em source} $\fdom(x)$ and {\em target} $\fran(x)$, and a function $\circ$ called the function of {\em composition} of arrows satisfying the following axioms:
\begin{itemize}
\item $\forall x\in X\;\;\big(\fdom(\fdom(x))=\fran(\fdom(x))=\fdom(x)\;\wedge\;\fdom(\fran(x))=\fran(\fran(x))=\fran(x)\big)$;
\item $\dom[\circ]=\{\langle x,y\rangle:\fran(x)=\fdom(y)\}$ and $\rng[\circ]\subseteq X$;
\item $\forall x\in X\;\;(\circ(\fdom(x),x)=x=\circ(x,\fran(x))$;
\item $\forall x\in X\;\forall y\in X\;\big(\fran(x)=\fdom(y)\;\Ra\;\fdom(\circ(x,y))=\fdom(x)\;\wedge\;\fran(\circ(x,y))=\fran(y))\big)$;
\item $\forall x\in X\;\forall y\in X\;\forall z\in X\;((\fran(x)=\fdom(y)\;\wedge\;\fran(y)=\fdom(z))\;\Ra\;
\circ(\circ(x,y),z)=\circ(x,\circ(y,z))$.
\end{itemize}
\end{definition}

\section{Morphisms of Mathematical Structures}

For two mathematical structures $(X,S)$, $(X',S')$ a function $f:X\to X'$ is called a \index{morphism}\index{mathematical structure!morphism of}{\em morphism} of the mathematical structures if $f$ respects the structures in a certain sense (depending on the type of the structures). Definitions of morphisms should be chosen so that the composition of two morphisms between mathematical structures of the same type remains a morphism of mathematical structures of that type.

For the structure of magma and its specifications (semigroups, monoids, groups) morphisms  are called homomorphims and are defined as follows. 

\begin{definition}\label{d:magma-morphism} For two magmas $(X,S_X)$ and $(Y,S_Y)$ a function $f:X\to Y$ is called a \index{magma homomorphism}{\em homomorphism} if $\forall x\in X\;\forall y\in X\;\big(f(S_X(x,y))=S_Y(f(x),f(y))\big)$.
\end{definition}

For the structure of a ring (and field) homomorphisms should preserve both operations (of addition and multiplication).

\begin{definition} For two rings $(X,(+,\cdot))$ and $(Y,(\oplus,\odot))$ a function $f:X\to Y$ is called a \index{ring homomorphism}{\em homomorphism} if $\forall x\in X\;\forall y\in X\;\big(f(x+y)=f(x)\oplus f(y)\;\wedge\;f(x\cdot y)=f(x)\odot f(y)\big)$.
\end{definition}

A far generalization of a magma homomorphisms are homomorphisms of universal algebras.

\begin{definition} Let $\sigma$ be a function with $\rng[\sigma]\subseteq\w$ and $(X,S)$, $(Y,S')$ be two universal algebras of signature $\sigma$. A function $f:X\to Y$ is called a {\em homomorphism of universal algebras} if 
\begin{itemize}
\item $\forall i\in\dom[\sigma]\;\;\forall x\in X^{\sigma(i)}\;(f(S_i(x))=S_i'(f\circ x))$.
\end{itemize}
\end{definition}

\begin{definition} Let $\sigma$ be a function with $\rng[\sigma]\subseteq\w$ and $(X,S)$, $(Y,S')$ be two universal relation structures of signature $\sigma$. A function $f:X\to Y$ is called a {\em homomorphism of relation structures} if 
\begin{itemize}
\item $\forall i\in\dom[\sigma]\;\;\forall x\;(x\in S_i\;\Ra\;f\circ x\in S_i'))$.
\end{itemize}
\end{definition}

\begin{definition} Let $(X,(+,<))$ and $(Y,(+,<))$ be two ordered groups. A function $f:X\to Y$ is called an \index{order homomorphism}\index{ordered group!order homomorphism of}{\em order homomorphism} if $f$ is a homomorphism of the abelian groups $(X,+)$ and $(Y,+)$ and $f$ preserves the order in the sense that $f(x)<f(x')$ for any elements $x<x'$ of $X$. 
\end{definition}

\begin{exercise}[H\"older] Prove that an ordered group is Archimedean if and only if it admits an order homomorphism to the real line.
\end{exercise}

For the structure of a topological space, morphisms are defined as continuous functions.

\begin{definition} For two topological spaces $(X,S_X)$ and $(Y,S_Y)$ a function $f:X\to Y$ is {\em continuous} if $\forall u\;(u\in S_Y\;\Ra\;f^{-1}[u]\in S_X)$.
\end{definition} 

For the structure of a duoform space, morphisms are defined as duomorph functions.

\begin{definition}\label{d:duoform-morphism} For two duoform spaces $(X,S_X)$ and $(Y,S_Y)$ a function $f:X\to Y$ is called {\em duoform} if 
\begin{itemize}
\item $\forall u\in S_Y\;\exists v\in  S_X\;\forall x\;\forall y\;\;(\langle x,y\rangle\in v\;\Ra\;\langle f(x),f(y)\rangle\in u)$;
\item $\forall v\in S_X\;\exists u\in  S_Y\;\forall x\;\forall y\;\;(\langle x,y\rangle\in v\;\Ra\;\langle f(x),f(y)\rangle\in u)$.
\end{itemize}
\end{definition}

\begin{definition} For two objectness categories $(X,(\fdom,\fran,\circ))$ and  $(X',(\fdom',\fran',\circ'))$ a function $F:X\to X'$ is called a {\em functor} between objectless categories if 
\begin{itemize}
\item $\forall x\in X\;\;\big(\fdom'(F(x))=F(\fdom(x))\;\wedge\;\fran'(F(x))=F(\fran(x))\big)$;
\item $\forall x\in X\;\forall y\in X\;\;(\fran(x)=\fdom(y)\;\Ra\;F(\circ(x,y))=\circ'(F(x),F(y))$.
\end{itemize}
\end{definition}
 
\begin{exercise} Prove that the compositions of morphisms considered in Definitions~\ref{d:magma-morphism}--\ref{d:duoform-morphism} remain morphisms in the sense of those definitions.
\end{exercise} 

\begin{definition} Let $(X,\Gamma_X)$ and $(Y,\Gamma_Y)$ be two digraphs. 
 A function $f:X\to Y$ is called
\begin{itemize}
\item a {\em digraph  homomorphism} if $\forall x\;\forall x'\in X\;\;(\langle x,x'\rangle\in\Gamma_X\;\Ra\;\langle f(x),f(x')\rangle\in \Gamma_Y)$;
\item a {\em digraph isomomorphism} if $f$ is bijective and the functions $f$ and $f^{-1}$ are digraph homomorphisms.
\end{itemize}
\end{definition}
 
 A digraph $(X,\Gamma)$ is called {\em extensional} if $$\forall x\in X\;\forall y\in X\;\;(x=y\;\Leftrightarrow\;\cev \Gamma(x)=\cev\Gamma(y),\quad\mbox{where}\quad\cev \Gamma(x)=\{x'\in X:\langle x',x\rangle\in\Gamma\}\setminus\{x\}.$$  
 
 \begin{theorem}[Mostowski--Shepherdson collapse]\label{Mostowski--Shepherdson collapse} Let $(X,\Gamma)$ be a simple directed graph such that the relation $\Gamma$ is set-like and well-founded. Then there exists a unique function $f:X\to\UU$ such that $f(x)=f[\cev \Gamma(x)]$ for all $x\in X$ and $Y=f[X]\subseteq\IV$ is a transitive class. The function $f$ is a homomophism between the digraphs $(X,\Gamma)$ and $(Y,\E{\restriction}Y)$. The relation $\Gamma$ is extensional if and only if $f$ is a isomorphism of the digraphs $(X,\Gamma)$ and $(Y,\E{\restriction}Y)$.
\end{theorem}

\begin{proof} By Recursion Theorem~\ref{t:Recursion}, for the function
$$F:X\times\UU\to\UU,\quad F:\langle x,u\rangle\mapsto \rng[u],$$
there exists a unique function 
$f:X\to\UU$ such that $$f(x)=F(x,f{\restriction}_{\cev\Gamma(x)})=f[\cev\Gamma(x)]$$ for all $x\in X$.

Consider the class $Y=f[X]$. Given any set $y\in Y$, find $x\in X$ such that $y=f(x)$ and observe that $y=f(x)=f[\cev\Gamma(x)]\subseteq f[X]=Y$, which means that the class $Y$ is transitive. 

Assuming that $Y\not\subseteq\VV$, we conclude that the set $A=\{x\in X:f(x)\notin\VV\}$ is not empty. By the well-foundedness of the relation $\Gamma$, there exists an element $a\in A$ such that $\cev \Gamma(a)\cap A=\emptyset$ and hence $f[\cev \Gamma(a)]\subseteq\VV$ and $f(a)=f[\cev\Gamma(a)]\in\VV$, see Theorem~\ref{t:vN}(6). But the inclusion $f(a)\in\VV$ contradicts the choice of $a$.

To see that the function $f:X\to Y$ is a homomorphism of the digraphs $(X,\Gamma)$ and $(Y,\E{\restriction}Y)$, take any pair $\langle x',x\rangle\in \Gamma$. Then $x'\in\cev \Gamma(x)$ and hence $f(x)\in f(x')$, which is equivalent to $\langle f(x),f(x')\rangle\in\E$.

If the function $f:X\to Y$ is an isomorphism of the digraphs $(X,\Gamma)$ and $(Y,\E{\restriction}Y)$, then the relation $\Gamma$ is extensional by the extensionality of the membership relation, which follows from Definition~\ref{d:equality}. Now assume that the relation $\Gamma$ is extensional. 

We claim that the function $f$ is injective. In the opposite case we can find a set $z\in Y\subseteq\IV$ such that $z=f(x)=f(x')$ for some distinct $x,x'\in X$. Find an ordinal $\alpha$ such that $z\in V_\alpha$. We can assume that $\alpha$ is the smallest possible, i.e., for any $y\in Y\cap \bigcup_{\beta\in\alpha}V_\beta$ with $y=f(x)=f(x')$ we have $x=x'$. The minimality of $\alpha$ ensures that $\alpha=\beta+1$ for some $\beta\in\alpha$. By the extensionality of the relation $\Gamma$, the sets $\cev\Gamma(x)$ and $\cev\Gamma(x')$ are distinct. Consequently, there exists $x''\in X$ such that $x''\in\cev\Gamma(x)\setminus\cev\Gamma(x')$ or $x''\in\cev\Gamma(x')\setminus\cev\Gamma(x)$. In the first case we have $f(x'')\in f(x)\in V_{\beta+1}=\mathcal P(V_\beta)$ and hence $f(x'')\in V_\beta$. The minimality of $\alpha$, and $x'\notin\cev\Gamma(x')$ imply $f(x'')\notin f[\cev \Gamma(x')]=f(x')=f(x)$ which is a contradiction. By analogy we can derive a contradiction from the assumption $x''\in\cev\Gamma(x')\setminus\cev\Gamma(x)$. This contradiction completes the proof of the injectivity of $f$. Since $Y=f[X]$, the function $f:X\to Y$ is bijective. 

It remains to prove that $f^{-1}:Y\to X$ is a digraph homomorphism. Indeed, for any $x',x\in X$  with $\langle f(x'),f(x)\rangle\in \E{\restriction}Y$ we have $f(x')\in f(x)=f[\cev \Gamma(x)]$ and by the injectivity of $f$, $x'\in\cev \Gamma(x)$ and finally $\langle x',x\rangle\in\Gamma$.
\end{proof}

\begin{remark} The Mostowski--Shepherdson collapse is often applied to extensional  digraphs $(X,\Gamma)$ whose relation $\Gamma$ satisfies the axioms of ZFC or NBG. Such digraphs are called {\em models} of ZFC or NBG, respectively. In this case Mostowski--Shepherdson collapse says that each model of ZFC or NBG is isomorphic to the digraph $(X,\E{\restriction}X)$, where $X$ is a suitable class (which can be also a set).
\end{remark}

\section{A characterization of the real line}

The real line is the most fundamental object in mathematics. It carries a bunch of mathematical structures: additive group, multiplicative semigroup, ring, field, linearly ordered set, metric space, topological space, bornological space, topological field etc etc. Some combinations of these structures determine the real line uniquely up to an isomorphisms. For Analysis the most important structure determining the real line uniquely is the structure of an ordered field. \index{real line}

We recall that an ordered field is a mathematical structure $(X,S)$ whose structure is a triple $S=(+,\cdot,<)$ consisting of two binary operations and a linear order on $X$ satisfying the axioms described in Example~\ref{ex:ord-field}.

We say that two ordered fields $(X,(+,\cdot,<))$ and $(Y,(\oplus,\odot,\prec))$ are {\em isomorphic} if there exists a bijective function $f:X\to X'$ that preserves the structure in the sense that
$$f(x+y)=f(x)\oplus f(y),\quad f(x\cdot y)=f(x)\odot f(y)\quad\mbox{and}\quad x<y\;\Leftrightarrow\;f(x)\prec f(y)$$
for any elements $x,y\in X$. If an isomorphism  $f$ between  ordered fields exists and is unique, then we say that these fields are uniquely isomorphic.

\begin{theorem} An ordered field $(X,(\oplus,\otimes,\prec))$ is (uniquely) isomorphic to the real line $(\IR,(+,\cdot,<))$ if and only if its order $\prec$ is boundedly complete.
\end{theorem}

\begin{proof} The ``only if'' part follows from the bounded completeness of the linear order on the real line, which was established in Theorem~\ref{t:R-ord}.

To prove the ``if'' part, assume that $(X,(\oplus,\otimes,\prec))$ is an ordered field with boundedly complete linear order $\prec$. Let $\mathtt 0$ and $\mathtt 1$ be the identity element of the  group $(X,\oplus)$ and $\mathtt 1$ be the multiplicative element of the semigroup $(X,\odot)$. Let $f_{\w}:\w\to X$ be the function defined by the recursive formula: $f_{\w}(0)=\mathtt 0$ and $f_\w(n+1)=f_\w(n)\oplus \mathtt 1$ for every $n\in\w$. Using the connection between the addition and order in the ordered group $(X,(\oplus,\prec))$ we can prove (by Mathematical Induction) that $f_\w(n)\prec f_\w(n+1)$ for all $n\in\w$. This implies that the function $f_\w$ is injective. By Mathematical Induction, it can be shown that $f_\w(n+m)=f_\w(n)\oplus f_\w(m)$ and $f_\w(n\cdot m)=f(n)\odot f(m)$ for any $n,m\in\w$.

Extend $f_\w$ to a function $f_{\IZ}:\IZ\to X$ such that $f_{\IZ}(\minus n)=-f_\w(n)$ for every $n\in\IN$. Here $-f_\w(n)$ is the additive inverse of $f_\w(n)$ in the group $(X,\oplus)$. Using suitable properties of the ordered group $(X,(\oplus,\prec))$ we can show that 
$f_\IZ(\minus n)\prec 0$ for every $n\in\IN$. Using the properties of the addition and multiplication in the field $(X,(\oplus,\odot))$ it can be shown that $f_\IZ(n+m)=f_\IZ(n)\oplus f_\IZ(m)$ and $f_\IZ(n\cdot m)=f_\IZ(n)\odot f_\IZ(m)$ for all $n,m\in\IZ$.

Next, extend the function $f_\IZ$ to a function $f_{\IQ}:\IQ\to X$ letting $f_{\IQ}(\tfrac{m}{n})=f_\IZ(m)\odot (f_\w(n))^{-1}$ for any rational number $\frac{m}{n}\in\IQ\setminus\IZ$. In this formula $(f_\w(n))^{-1}$ is the multiplicative inverse of $f_\w(n)$ in the multiplicative group $(X\setminus\{\mathtt 0\},\odot)$. Using algebraic properties of the field $(X,(\oplus,\odot))$, it can be shown that  the function $f_\IQ:\IQ\to X$ is injective and preserves the field operations and the order.

The bounded completeness of the orderd field $(X,(\oplus,\odot,\prec))$ implies that this field is Archimedean in the sense that for any $x\in X$ there exists $n\in\w$ such that $x\prec f(n)$. Assuming that such a number $n$ does not exist, we conclude that the set $f_\w[\w]\subseteq X$ is upper bounded by $x$ in the linear order $(X,\prec)$ and by the bounded completeness, $f_\w[\w]$ has the least upper bound $\sup f_\w[\w]$. On the other hand, $\sup f_\w[\w]\oplus (-\mathtt 1)\prec \sup f_\w[\w]$ also is an upper bound for $f_\w[\w]$, which a contradiction showing that the field $(X,(\oplus,\odot,\prec))$ is Archimedean.

Now extend $f_\IQ$ to a function $f_{\IR}:\IR\to X$ assigning to each real number $r\in\IR\setminus\IQ$ the element $\sup f_\IQ(\{q\in\IQ:q<r\})$ of the set $X$. This element exists by the bounded completeness of $X$. Using the Archimedean property of the field $(X,(\oplus,\odot,\prec))$, it can be shown that $f_\IR$ is a required isomorphism of the fields $(\IR,(+,\cdot,<))$ and $(X,(\oplus,\odot,\prec))$. The uniqueness of $f_{\IR}$ follows from the construction: at each step there is a unique way to extend the function so that the field operations and the order are preserved.
\end{proof}

\begin{exercise}[Tarski] Prove that a bijective function $f:\IR\to\IR$ is the identity function of $\IR$ if and only if it has the following three properties:
\begin{enumerate}
\item $f(1)=1$;
\item $\forall x\in\IR\;\forall y\in\IR\;\;(f(x+y)=x+y)$;
\item $\forall x\in\IR\;\forall y\in\IR\;\; (x<y\;\Leftrightarrow\;f(x)<f(y))$.
\end{enumerate}
\end{exercise} 
\newpage

\newpage

\part{Elements of Category Theory}

{\small

\rightline{\em A mathematician is a person who can find analogies between theorems;}

\rightline{\em a better mathematician is one who can see analogies between proofs}

\rightline{\em and the best mathematician can notice analogies between theories.}

\rightline{\em One can imagine that the ultimate mathematician is one}

\rightline{\em who can see analogies between analogies.}
\smallskip

\rightline{Stefan Banach}
\bigskip

\rightline{\em The language of categories is affectionately known as}

\rightline{\em  ``abstract nonsense'', so named by Norman Steenrod.}

\rightline{\em  This term is essentially accurate and not necessarily derogatory:}

\rightline{\em categories refer to ``nonsense" in the sense}

\rightline{\em  that they are all about the ``structure",}

\rightline{\em and not about the ``meaning", of what they represent.}
\smallskip

\rightline{Paolo Aluffi, 2009}
}
 \bigskip
 
 Theory of Categories was founded in 1942--45 by Samuel Eilenberg and Saunders Mac Lane who worked in Algebraic Topology and needed tools for describing common patterns appearing in topology and algebra. Category Theory was created as a science about  structures and patterns appearing in mathematics. Rephrasing Stefan Banach, we could say that Category Theory is a science about analogies between analogies.
 
By some mathematicians, Category Theory is considered as an alternative (to Set Theory) foundation for mathematics.
In this respect, Saunders MacLane, one of founders of Category Theory writes the following  \cite{McLane86}.
\medskip
 




{\small \noindent\em
$\dots$ the membership relation for sets can often be replaced by the composition operation for functions. This leads to an alternative foundation for Mathematics upon categories — specifically, on the category of all functions. Now much of Mathematics is dynamic, in that it deals with morphisms of an object into another object of the same kind. Such morphisms (like functions) form categories, and so the approach via categories fits well with the objective of organizing and understanding Mathematics. That, in truth, should be the goal of a proper philosophy of Mathematics.

The standard ``foundation" for mathematics starts with sets and their elements. It is possible to start differently, by axiomatising not elements of sets but functions between sets. This can be done by using the language of categories and universal constructions.}
\bigskip

In this chapter we discuss some basic concepts of Category Theorem, but define them using the language of the Classical Set Theory. The culmination result of this part are Theorem~\ref{t:EC=Set} and \ref{t:GWO=Set} characterizing categories that are isomorphic to the category of sets.

\section{Categories}

\begin{definition} A {\em category}\index{category} is a $6$-tuple $\C=(\Ob,\Mor,\fdom,\fran,\mathsf 1,\circ)$ consisting of
\begin{itemize}
\item a class $\Ob$ whose elements are called {\em objects}\index{object of category}\index{category!object of} of the category $\C$ (briefly, {\em $\C$-objects});
\item a class $\Mor$ whose elements are called {\em morphisms}\index{morphism of category}\index{category!morphism of} of the category $\C$ (briefly, {\em $\C$-morphisms});
\item two functions $\fdom\colon\Mor\to\Ob$ and $\fran\colon\Mor\to\Ob$ assigning to each morphism $f\in\Mor$ its \index{source $\fdom$}{\em source} $\fdom(f)\in\Ob$ and \index{target $\fran$}{\em target} $\fran(f)\in\Ob$;
\item a function $\mathsf 1:\Ob\to\Mor$ assigning to each object $X\in\Ob$ a morphism $\mathsf 1_X\in\Mor$, called the \index{identity morphism} {\em identity morphism} of $X$, and satisfying the equality $\fdom(\mathsf 1_X)=X=\fran(\mathsf 1_X)$;
\item a function $\circ$ with domain $\dom[\circ]=\{\langle f,g\rangle\in \Mor\times\Mor:\fran(g)=\fdom(f)\}$ and range $\rng[\circ]\subseteq\Mor$ assigning to any $\langle f,g\rangle\in\dom[\circ]$ a morphism $f\circ g\in\Mor$ such that $\fdom(f\circ g)=\fdom(g)$ and $\fran(f\circ g)=\fran(f)$ and the following axioms are satisfied:
\begin{itemize}
\item [(A)] for any morphisms $f,g,h\in\Mor$ with $\fdom(f)=\fran(g)$ and $\fdom(g)=\fran(h)$ we have\\ $(f\circ g)\circ h=f\circ(g\circ h)$;
\item [(U)] for any morphism $f\in\Mor$ we have $\mathsf 1_{\fran(f)}\circ f=f=f\circ\mathsf 1_{\fdom(f)}$.
\end{itemize}
\end{itemize}
\end{definition}
The function $\circ$ is called the \index{composition of morphisms}{\em operation of composition} of morphisms and the axiom $(A)$ is called the {\em associativity} of the composition. 

Discussing several categories simultaneously, it will be convenient to label the classes of objects and and morphisms of a category $\C$ with subscripts writing $\Ob_\C$ and $\Mor_\C$.

For explaining definitions and results of Category Theory it is convenient to use arrow notations. Morphisms between objects are denoted by arrows with subscripts or supersripts, and equalities of compositions of morphisms are expressed by commutative diagrams.

For example, the associativity of the composition can be expresses as the commutativity of the diagram
$$
\xymatrix{
&B\ar^g[r]&C\ar^f[rd]\\
A\ar^h[ru]\ar_h[rd]\ar_{g\circ h}[rru]&&&D\\
&B\ar_g[r]\ar^{f\circ g}[rru]&C\ar_f[ru]
}
$$


\begin{definition} A category $\C=(\Ob,\Mor,\fdom ,\fran,\mathsf 1,\circ)$ is called
\begin{itemize}
\item \index{category!small}\index{small category}{\em small} if its class of morphisms $\Mor$ is a set;
\item \index{category!locally small}\index{locally small category}{\em locally small} if for any objects $X,Y\in\Ob$ the class of morphisms\\ $\Mor(X,Y)=\{f\in\Mor:\fdom(f)=X,\;\fran(f)=Y\}$ is a set;
\item \index{category!discrete}\index{discrete category}{\em discrete} if for any objects $X,Y\in\Ob$
$$\Mor(X,Y)=\begin{cases}\{\mathsf 1_X\}&\mbox{if $X=Y$};\\
\emptyset&\mbox{if $X\ne Y$}.
\end{cases}
$$
\end{itemize}
\end{definition}

Mathematics is literally saturated with categories. We start with the category of sets, one of the most important categories in Mathematics.

\begin{example} The \index{category of sets $\Set$}\index{category!{{\bf Set}}}\index{{{\bf Set}}}{\em category of sets} $\mathbf{Set}$ is the 6-tuple $(\Ob,\Mor,\fdom,\fran,\mathsf 1,\circ)$ consisting of
\begin{itemize}
\item the class $\Ob=\UU$;
\item the class $\Mor=\{\langle X,f,Y\rangle\in \UU\times \Fun\times\UU:X=\dom[f],\;\rng[f]\subseteq Y\}$;
\item the function $\fdom:\Mor\to\Ob$, $\fdom :\langle X,f,Y\rangle\mapsto X$;
\item the function $\fran:\Mor\to\Ob$, $\fran :\langle X,f,Y\rangle\mapsto Y$;
\item the function $\mathsf 1:\Ob\to\Mor$, $\mathsf 1:X\mapsto \langle X,\Id{\restriction}_X,X\rangle$;
\item the function $\circ=\{\langle\langle\langle A,f,B\rangle,\langle C,g,D\rangle\rangle,\langle A,gf,D\rangle\rangle\in(\Mor\times\Mor)\times\Mor:B=C\}$ where $ gf=\{\langle x,z\rangle:\exists y\;(\langle x,y\rangle\in f\;\wedge\;\langle y,z\rangle\in g)\}$.
\end{itemize}
\end{example}

Taking for morphisms the class of functions $\Fun$, we obtain the category of sets and their surjective maps.

\begin{example} The {\em category of sets and their surjective functions} is the 6-tuple\\ $(\Ob,\Mor,\fdom,\fran,\mathsf 1,\circ)$ consisting of
\begin{itemize}
\item the class $\Ob=\UU$;
\item the class $\Mor=\Fun$;
\item the function $\fdom=\dom{\restriction}_{\Fun}$;
\item the function $\fran=\rng{\restriction}_{\Fun}$;
\item the function $\mathsf 1:\Ob\to\Mor$, $\mathsf 1:X\mapsto \Id{\restriction}_X$;
\item the function $\circ=\{\langle\langle f,g\rangle,gf\rangle\in(\Mor\times\Mor)\times\Mor:\rng[f]=\dom[g]\}$ where\\ $ gf=\{\langle x,z\rangle:\exists y\;(\langle x,y\rangle\in f\;\wedge\;\langle y,z\rangle\in g)\}$.
\end{itemize}
\end{example}


\begin{definition} A category $\mathcal C=(\Ob,\Mor,\fdom ,\fran ,\mathsf 1,\circ)$ is a \index{subcategory}{\em subcategory} of a category\\ $\mathcal C'=(\Ob',\Mor',\fdom ',\fran ',\mathsf 1',\circ')$ if $$\Ob\subseteq\Ob',\;\Mor\subseteq\Mor',\;\fdom =\fdom '{\restriction}_{\Mor},\;\fran =\fran '{\restriction}_{\Mor},\;\mathsf 1=\mathsf 1'{\restriction}_{\Ob}, \mbox{ and }\circ=\circ'{\restriction}_{\Mor\times\Mor}.$$
A subcategory $\C$ of $\C'$ is called \index{subcategory!full}\index{full subcategory}{\em full} if $\forall X\in\Ob\;\forall Y\in\Ob\;\;\Mor(X,Y)=\Mor'(X,Y)$.
\end{definition}
A full subcategory is fully determined by its class of objects.
\smallskip

\begin{example} Let $\mathbf{FinSet}$ be the full subcategory of the category $\Set$, whose class of objects coincides with the class of finite sets.
\end{example}

\begin{example} Let $\Crd$ be the full subcategory of the category $\Set$, whose class of objects coincides with the class of cardinals.
\end{example}

An important example of a category is the category of mathematical structures. We recall that a mathematical structure is a pair of classes $(X,S)$ satisfying certain list of axioms. If $X$ and $S$ are sets, then the pair $(X,S)$ can be identified with the ordered pair $\langle X,S\rangle$, which is an element of the class $\ddot\UU=\UU\times\UU$. The underlying set $X$ and the structure $S$ can be recovered from the ordered pair $\langle X,S\rangle$ using the functions $\dom$ and $\rng$ as $X=\dom(\langle X,S\rangle)$ and $S=\rng(\langle X,S\rangle)$.

\begin{example} The \index{category of mathematical structures $\mathbf{MS}$}\index{category!of mathematical structures $\mathbf{MS}$}{\em category of mathematical structures} $\mathbf{MS}$ is the 6-tuple $(\Ob,\Mor,\fdom,\fran,\mathsf 1,\circ)$ consisting of
\begin{itemize}
\item the class $\Ob=\UU\times\UU$;
\item the class $\Mor=\{\langle X,f,Y\rangle\in\Ob\times \Fun\times\Ob:\dom[f]=\dom(X)\;\wedge\;\rng[f]\subseteq \dom(Y)\}$;
\item the function $\fdom:\Mor\to\Ob$, $\fdom :\langle X,f,Y\rangle\mapsto X$;
\item the function $\fran:\Mor\to\Ob$, $\langle X,f,Y\rangle\mapsto Y$;
\item the function $\mathsf 1:\Ob\to\Mor$, $\mathsf 1:X\mapsto\langle X,\Id{\restriction}_{\dom(X)},X\rangle$;
\item the function $\circ=\{\langle\langle A,f,B\rangle,\langle C,g,D\rangle\rangle,\langle A, gf,D\rangle\rangle\in(\Mor\times\Mor)\times\Mor:B=C\}$ where $ gf=\{\langle x,z\rangle:\exists y\;(\langle x,y\rangle\in f\;\wedge\;\langle y,z\rangle\in g)\}$.
\end{itemize}
\end{example}

Many important category arise as subcategories of the category $\mathbf{MS}$.

\begin{example} The {\em category of magmas} $\mathbf{Mag}$ is a subcategory of the category $\mathbf{MS}$. Its objects are magmas and morphisms are triples $\langle X,f,Y\rangle$ where $f$ is a homomorphism of magmas $X,Y$.
\end{example}

\begin{example}\label{ex:sMS} The categories of semigroups, inverse semigroups, Clifford semigroups, monoids, groups, commutative groups are full subcategories of the category of magmas. The objects of these categories are  semigroups,inverse semigroups, Clifford semigroups,  monoids, groups, commutative groups, respectively.
\end{example}

\begin{example} The category of topological spaces $\mathbf{Top}$ is the subcategory of the category $\mathbf{MS}$. The object of the category $\mathbf{Top}$ are topological spaces and morphisms are triples $\langle X,f,Y\rangle$ where $f$ is a continuous function between topological spaces $X$ and $Y$.
\end{example}

\begin{example} The category of directed graphs is the subcategory of the category $\mathbf{MS}$. The object of this category are directed graphs and morphisms are triples $\langle X,f,Y\rangle$ where $f$ is an increasing function between directed graphs $X$ and $Y$. The category of directed graphs contains full subcategories of ordered sets, partially ordered sets, linearly ordered sets, well-ordered sets.
\end{example}

\begin{example}\label{ex:monoid-category} Each monoid $(X,S)$ can be identified with the category $(\Ob,\Mor,\fdom ,\fran ,\mathsf 1,\circ)$   such that
\begin{itemize}
\item $\Ob=\{X\}$;
\item $\Mor=X$;
\item $\fdom =\fran =\Mor\times \Ob$;
\item $\mathsf 1=\Ob\times\{e\}$ where $e\in X$ is the unit of the monoid $(X,S)$;
\item $\circ=S$.
\end{itemize}
On the other hand, for any category $\C=(\Ob,\Mor,\fdom ,\fran ,\mathsf 1,\circ)$ with a single object, the pair $(\Mor,\circ)$ is a monoid.
\end{example}

\begin{example}\label{ex:po-cat} Each partially ordered class $(X,S)$ can be identified with the category $(\Ob,\Mor,\fdom ,\fran ,\mathsf 1,\circ)$   such that
\begin{itemize}
\item $\Ob=X$;
\item $\Mor=S$;
\item $\fdom =\dom{\restriction}_X$;
\item $\fran =\rng{\restriction}_X$;
\item $\mathsf 1=\{\langle x,\langle x,x\rangle\rangle:x\in X\}$; 
\item $\circ$ is the function assigning to any pair of pairs $\langle\langle x,y\rangle,\langle u,v\rangle\rangle\in S\times S$ with $y=u$ the pair $\langle x,v\rangle$ which belongs to the order $S$ by the transitivity of $S$.
\end{itemize}
\end{example}


\begin{definition} For any category $\C=(\Ob,\Mor,\fdom ,\fran ,\mathsf 1,\circ)$ the \index{dual category}\index{opposite category}\index{category!dual}\index{category!opposite} {\em dual} (or else {\em opposite}) category to $\C$ is the category $\C^{\op}=(\Ob^{\op},\Mor^\op,\fdom ^\op,\fran ^\op,\mathsf 1^\op,\circ^\op)$ such that
\begin{itemize}
\item $\Ob^\op=\Ob$, $\Mor^\op=\Mor$, and $\mathsf 1^\op=\mathsf 1$;
\item $\fdom ^\op=\fran $, $\fran ^\op=\fdom $;
\item $\circ^\op=\{\langle \langle g,f\rangle,h\rangle:\langle\langle f,g\rangle,h\rangle\in \circ\}$.
\end{itemize}
\end{definition}

In the dual category all arrows are reverted and the composition of arrows is taken in the reverse order.

The philosophy of Category Theory is to derive some properties of objects from the information about  morphisms related to these objects. A category is a kind of algebraic structure that operates with morphisms, not objects. Without any loss of information, objects can be identified with their unit morphisms. After such reduction a category becomes a typical algebraic structure on the class $\Mor$ of morphisms.

\begin{definition} Let   $\C$ be a category. A morphism $f\in\Mor_\C(X,Y)$ between $\C$-objects $X,Y$ is called an \index{isomorphism}\index{category!isomorphism in}{\em isomorphism} (more precisely, a {\em $\C$-isomorphism}) if there exists a morphism $g\in\Mor_\C(Y,X)$  such that $g\circ f=\mathsf 1_{X}$ and $f\circ g=1_Y$.
The morphism $g$ is unique and is denoted by $f^{-1}$. For $\C$-objects $X,Y$ by $\Iso_\C(X,Y)$ we shall denote the subclass of $\Mor_\C(X,Y)$ consisting of isomorphisms.
\end{definition}

\begin{exercise} Prove that for any isomorphism $f\in \Mor_\C(X,Y)$ of a category $\C$, the morphism $f^{-1}$ is unique.
\vskip5pt

\noindent{\em Hint:} If $g\in\Mor_\C(Y,X)$ is a morphism such that  $g\circ f=\mathsf 1_{X}$ and $f\circ g=\mathsf 1_{Y}$, then\\ $g=g\circ 1_{Y}=g\circ (f\circ f^{-1})=(g\circ f)\circ f^{-1}=\mathsf 1_{X}\circ f^{-1}=f^{-1}$.
\end{exercise} 

\begin{definition} Two objects $X,Y\in\Ob_\C$ of a category $\C$ are called \index{isomorphic objects}{\em isomorphic} (more precisely, {\em $\C$-isomorphic}) if there exists an isomorphism $f\in\Mor(X,Y)$. The isomorphness of $X,Y$ will be denoted as $X\cong Y$ or $X\cong_\C Y$.
\end{definition}

From the point of view of Category Theory, isomorphic objects have the same properties (which can be expresses in the language of morphisms).

\begin{exercise} Prove that (i) in the category of sets, isomorphisms are bijective maps; (ii) in the category of magmas, isomorphisms are bijective homomorphisms of magmas.
\end{exercise}

\begin{example}\label{ex:inverse-category} Any inverse semigroup $(X,S)$ can be identified with the category\\ $(\Ob,\Mor,\fdom ,\fran ,\mathsf 1,\circ)$   such that
\begin{itemize}
\item $\Ob=\{S(x,x^{-1}):x\in X\}$;
\item $\Mor=X$;
\item $\fdom =\{\langle x,S(x^{-1},x)\rangle:x\in X\}$;
\item $\fran =\{\langle x,S(x,x^{-1})\rangle:x\in X\}$;
\item $\mathsf 1=\Id{\restriction}_{\Ob}$;
\item $\circ=\{\langle x,y\rangle:\fran(x)=\fdom(y)\}$.
\end{itemize}
Each morphism of this category is an isomorphism.
\end{example}

Now we define category analogs of injective and surjective functions.

\begin{definition} Let $\C=(\Ob,\Mor,\fdom ,\fran ,\mathsf 1,\circ)$ be a category. A morphism $f\in\Mor(X,Y)$ between two $\C$-objects $X,Y$ is called
\begin{itemize}
\item a \index{monomorphism}{\em monomorphism} if $\forall Z\in\Ob\;\forall g,h\in\Mor(Z,X)\;(f\circ g=f\circ h\;\Ra\;g=h)$;
\item a \index{epimorphism}{\em epimorphism} if $\forall Z\in\Ob\;\forall g,h\in\Mor(Y,Z)\;(g\circ f=h\circ f\;\Ra\;g=h)$;
\item a \index{bimorphism}{\em bimorphism} if $f$ is both monomorphism and epimorphism.
\end{itemize}
For two $\C$-objects $X,Y$ by $\Mono_\C(X,Y)$ and $\Epi_\C(X,Y)$ we denote the subclasses of $\Mor_\C(X,Y)$ constisting of monomorphisms and epimorphisms from $X$ to $Y$, respectively.
\end{definition}

\begin{definition} A category is \index{balanced category}\index{category!balanced}{\em balanced} if each bimorphism of this category is an isomorphism.
\end{definition}

\begin{exercise} Prove that a morphism $f$ of a category $\C$ is a monomorphism if and only if $f$ is an epimorphism of the dual category $\C^\op$.
\end{exercise}

\begin{exercise} Prove that that in the category of sets (and in the category of topological spaces or magmas) monomorphisms are injective functions and epimorphisms are surjective functions.
\end{exercise}

\begin{exercise} Prove that the category of sets is balanced but the category of topological spaces is not balanced.
\end{exercise}

\begin{exercise} Find a homomorphism $h:X\to Y$ of two monoids, which is not a surjective function but is an epimorphism in the category of monoids.
\smallskip

\noindent{\em Hint}: Consider the identity function $\Id{\restriction}_\IN:\IN\to\IZ$ of the monoids $(\IN,+)$ and $(\IZ,+)$.
\end{exercise}

\begin{exercise} Given a monoid $(X,M)$ characterize monomorphisms and epimorphisms of the category described in Exercise~\ref{ex:monoid-category}.
\end{exercise}

Now using the properties of morphisms we distinguish two special types of objects. 

\begin{definition} Let $\C$ be a category. A $\C$-object $X$ is called
\begin{itemize}
\item \index{initial object}\index{category!initial object of}{\em initial} (more precisely, $\C$-{\em initial\/}) if for any $\C$-object $Y$ there exists a unique $\C$-morphism $X\to Y$;
\item \index{terminal object}\index{category!terminal object of}{\em terminal} (more precisely, {\em $\C$-terminal\/}) if for any $\C$-object $Z$ there exists a unique $\C$-morphism $Z\to X$.
\end{itemize}
\end{definition}

\begin{exercise} Prove that any initial (resp.  terminal) objects of a category are isomorphic.
\end{exercise}

\begin{exercise} Prove that an object  of a category $\C$ is terminal if and only if it is an initial object of the dual category $\C^\op$.
\end{exercise}

\begin{exercise} Prove that a set $X$ is an initial (resp. terminal) object of the category of sets if and only if $X$ is empty (resp. a singleton).
\end{exercise}

\begin{exercise} Prove that a group $G=(X,S)$ is an initial object of the category of groups $\mathbf{Grp}$ if and only if $G$ is a terminal object of the category $\mathbf{Grp}$ if and only if $X$ is a trivial group.
\end{exercise}

\begin{exercise} Describe initial and terminal objects in the category of magmas, semigroups, inverse semigroups, Clifford semigroups, monoids.
\end{exercise} 

\begin{definition} A \index{global element} {\em global element} of an object $X$ of a category $\C$ is any $\C$-morphism $f:\mathtt 1\to X$ form a terminal object $\mathtt 1$ of $\C$ to $X$.
\end{definition}

\begin{exercise} Describe global elements in the categories of sets, topological spaces, magmas, semigroups, inverse semigroups, Clifford semigroups, monoids, groups.
\end{exercise}

Finally, we define two operations on categories: product of categories and taking the category of morphisms.

\begin{definition} For two categories $\C=(\Ob,\Mor,\fdom ,\fran ,\mathsf 1,\circ)$ and $\C'=(\Ob',\Mor',\fdom ',\fran ',\mathsf 1',\circ')$ their \index{product of categories}\index{categories!product of}{\em product} $\C\times \C'$ is the category $(\Ob'',\Mor'',\fdom '',\fran '',\mathsf 1'',\circ'')$ with 
\begin{itemize}
\item $\Ob''=\Ob\times\Ob'$;
\item $\Mor''=\Mor\times\Mor'$;
\item $\fdom ''=\{\langle \langle f,f'\rangle,\langle X,X'\rangle\rangle:\langle f,X\rangle\in\fdom \;\wedge\;\langle f',X'\rangle\in\fdom '\}$;
\item $\fran ''=\{\langle \langle f,f'\rangle,\langle Y,Y'\rangle\rangle:\langle f,Y\rangle\in\fran \;\wedge\;\langle f',Y'\rangle\in\fran '\}$;
\item ${\mathsf 1}''=\{\langle \langle X,X'\rangle,\langle f,f'\rangle\rangle:\langle X,f\rangle\in\mathsf 1\;\wedge\;\langle X',f'\rangle\in\mathsf 1'\}$;
\item $\circ''=\{\langle\langle f,f'\rangle,\langle g,g'\rangle,\langle h,h'\rangle\rangle:(\langle f,g,h\rangle\in \circ)\;\wedge\;(\langle f',g',h'\rangle\in\circ')\}$.
\end{itemize}
\end{definition}

\begin{definition} For a category $\C=(\Ob,\Mor,\fdom ,\fran ,\mathsf 1,\circ)$ the \index{category!of morphisms}{\em category of $\C$-morphisms} is the category $\C^{\to}=(\Ob',\Mor',\fdom ',\fran ',\mathsf 1',\circ')$
where
\begin{itemize}
\item $\Ob'=\Mor$;
\item $\Mor'=\{\langle f,\langle \alpha,\beta\rangle,g\rangle:f,g,\alpha,\beta\in\Mor\;\wedge\;\fdom (g)=\fran (\alpha)\;\wedge\;\fdom (\beta)=\fran (f)\;\wedge\;g\circ\alpha=\beta\circ f\}$;
\item  $\fdom '=\{\langle \langle f,\langle \alpha,\beta\rangle,g\rangle,h\rangle\in\Mor'\times \Mor:h=f\}$, $\fdom': \langle f,\langle \alpha,\beta\rangle,g\rangle\mapsto f$;
\item $\fran '=\{\langle \langle f,\langle \alpha,\beta\rangle,g\rangle,h\rangle\in\Mor'\times \Mor:h=g\}$, $\fran':\langle f,\langle \alpha,\beta\rangle,g\rangle\mapsto g$;
\item $\mathsf 1'=\{\langle f,\langle f,\langle 1_{\fdom (f)},\mathsf 1_{\fran (f)}\rangle,f\rangle:f\in\Ob'\}$;
\item $\circ'=\{\langle\langle f,\langle\alpha,\beta\rangle,g\rangle,\langle g,\langle\alpha',\beta'\rangle,h\rangle,\langle f,\langle \alpha'\circ\alpha,\beta'\circ\beta\rangle,h\rangle\rangle:$

\rightline{$\langle f,\langle\alpha,\beta\rangle,g\rangle\in\Mor'\;\wedge\;\langle g,\langle\alpha',\beta'\rangle,h\rangle\in\Mor'\}$.}
\end{itemize}
\end{definition}

\begin{exercise} Illustrate the composition of morphisms of the category $\C^\to$ by commutative diagrams.
\end{exercise}

Each category can be identified with an objectless category, see Definition~\ref{d:objectless}.

\begin{remark} For each category $(\Ob,\Mor,\fdom,\fran,\mathsf 1,\circ)$, the mathematical structure $(\Mor,(\fdom',\fran',\circ'))$ where
\begin{itemize}
\item $\fdom'=\{\langle x,\mathsf 1_{\fdom(x)}\rangle:x\in\Mor\}$,
\item $\fran'=\{\langle x,\mathsf 1_{\fran(x)}\rangle:x\in\Mor\}$,
\item $\circ'=\{\langle x,y,z\rangle:\langle y,x,z\rangle\in\circ\}$
\end{itemize}
is an objectless category.

Conversely, for each objectless category $(X,(\fdom',\fran',\circ'))$ the $6$-tuple $(\Ob,\Mor,\fdom,\fran,\mathsf 1,\circ)$ consisting of
\begin{itemize}
\item the class $\Ob=\{x\in X:\fdom'(x)=x=\fran'(x)\}$;
\item the class $\Mor=X$;
\item the functions $\fdom=\fdom'$ and $\fran=\fran'$;
\item the function $\mathsf 1=\Id{\restriction}_{\Ob}$;
\item the function $\circ=\{\langle x,y,z\rangle:\langle y,x,z\rangle\in\circ'\}$
\end{itemize}
is a category.
\end{remark}

In fact, without loss of information, the theory of categories can be well developed in its objectless form, but human intuition is better fit to object version of category theory. 

\section{Functors}

Functors are functions between categories. The formal definition follows.

\begin{definition} A \index{functor}{\em functor} $F:\C\to\C'$ between two categories $\C=(\Ob,\Mor,\fdom ,\fran ,\mathsf 1,\circ)$ and $\C'=(\Ob',\Mor',\fdom ',\fran ',\mathsf 1',\circ')$ is a pair $F=(\dot F,\ddot F)$ of two functions $\dot F:\Ob\to\Ob'$ and $\ddot F:\Mor\to\Mor'$ such that
\begin{itemize}
\item $\forall X\in\Ob\;\;\big(\mathsf 1'_{\dot F(X)}=\ddot F(\mathsf 1_X)\big)$;
\item $\forall X,Y\in\Ob\;\;\big(\ddot F(\Mor(X,Y))\subseteq\Mor'(\dot F(X),\dot F(Y))\big)$;
\item $\forall f,g\in\Mor\;\;\big(\fdom (g)=\fran (f)\;\Rightarrow\;\ddot F(g\circ f)=\ddot F(g)\circ' \ddot F(f)\big)$. 
\end{itemize}
In the sequel, we shall write $FX$ and $Ff$ instead of $\dot F(X)$ and $\ddot F(f)$, respectively.
\end{definition}

\begin{definition} A functor $F:\C\to\C'$ is called \index{functor!faithful}\index{functor!full}\index{faithful functor}\index{full functor}{\em faithful} (resp. {\em full\/}) if for any $\C$-objects $X,Y$, the function $\ddot F{\restriction}_{\Mor_\C(X,Y)}:\Mor_\C(X,Y)\to \Mor_{\C'}(\dot FX,\dot FY)$ is injective (resp. surjective).
\end{definition}

\begin{example}[The embedding functor] For any (full) subcategory $\C$ of a category $\C'$ the \index{identity embedding functor}\index{functor!of identity embedding}{\em identity embedding functor} $\mathsf 1_{\C,\C'}:\C\to\C'$ is the pair of  functions $(\Id{\restriction}_{\Ob_\C},\Id{\restriction}_{\Mor_\C})$. This functor is faithful (and full). If $\C=\C'$, then the functor $\mathsf 1_{\C,\C'}$ is denoted by $\mathsf 1_\C$.
\end{example}

\begin{example}[Forgetful functor] Consider the functor $U:\mathbf{MS}\to\Set$ assigning to each mathematical structure $\langle X,S\rangle\in\UU\times\UU$ its underlying set $X$ and to each morphism $\langle X,f,Y\rangle$ of the category $\mathbf{MS}$ the morphism $\langle UX,f,UY\rangle$ of the category $\Set$. The functor $U$ is called the \index{forgetful functor}\index{functor!forgetful}{\em forgetful functor}. It is easy to see that this functor is faithful. Then the restriction of the functor $U$ to any subcategory of $\mathbf{MS}$ also is a faithful functor.
\end{example}

This example motivates the following definition.

\begin{definition} A category $\C$ is called \index{concrete category}\index{category!concrete}{\em concrete} if it admits a faithful functor $F:\C\to\Set$ to the category of sets $\Set$.
\end{definition}

So, categories of mathematical structures are concrete. 

\begin{Exercise} Give an example of a category which is not concrete.
\smallskip

\noindent{\em Hint}: Such categories exist in Algebraic Topology (for example, the category of topological spaces and classes of homotopic maps).
\end{Exercise}

Next, we consider two important functors reflecting the structure of any category in the categories $\Set$ and $\Set^\op$.

\begin{example} For any locally small category $\C=(\Mor,\Ob,\fdom ,\fran ,\mathsf 1,\circ)$ and any $\C$-object $C$ consider
\begin{enumerate}
\item[\textup{1)}] the functor $\Mor(C,-):\C\to\Set$ assigning to any object $X\in\Ob$ the set $\Mor(C,X)$ and to any morphism $f\in\Mor$  the function $\Mor(C,f):\Mor(C,\fdom (f))\to\Mor(C,\fran (f))$, $\Mor(C,f):g\mapsto f\circ g\in \Mor(C,\fran (f))$. 
\item[\textup{2)}] the functor $\Mor(-,C):\C\to\Set^\op$ assigning to any object $X\in\Ob$ the set $\Mor(X,C)$ and to any morphism $f\in\Mor$  the function $\Mor(f,C):\Mor(\fran (f),C)\to\Mor(\fdom (f),C)$, $\Mor(f,C):g\mapsto g\circ f$. 
\end{enumerate}
\end{example}

\begin{exercise} Study the functors $\Mor(c,-)$ and $\Mor(-,c)$ for categories defined in Example~\ref{ex:sMS}.
\end{exercise}

Let $F:\C\to\C'$ and $G:\C'\to\C''$ be functors between categories $\C,\C',\C''$. We recall that these functors are pairs of functions $(\dot F,\ddot F)$ and $(\dot G,\ddot G)$. Taking compositions of the corresponding components, we obtain a functor $(\dot G\circ \dot F,\ddot G\circ \ddot F):\C\to\C''$ denoted by $GF$ and called the \index{composition of functors}{\em composition} of the functors $F,G$.

\begin{exercise} Show the composition of faithful (resp. full) functors is a faithful (resp. full) functor.
\end{exercise} 

Any functor between small categories is a set. So, it is legal to consider the category $\mathbf{Cat}$ whose objects are small categories and morphisms are functors between small categories. Applying to this category the general notion of an isomorphism, we obtain the notion of isomorphic categories, which can be defined for any (not necessarily small) categories.

\begin{definition} Two categories $\C$ and $\C'$ are called \index{categories!isomorphic}{\em isomorphic} if 
there exist functors $F:\C\to\C'$ and $G:\C'\to\C$ such that $FG=\mathsf 1_{\C'}$ and $GF=\mathsf 1_{\C}$. In this case we write $\C\cong\C'$.
\end{definition}

A weaker notion is that of equivalent categories. To introduce this notion we need the notion of a natural transformation of functors.

\section{Natural transformations}

\begin{definition} Let $\C,\C'$ be two categories and $F,G:\C\to \C'$ be two functors. A \index{natural transformation}{\em natural transformation} $\eta:F\to G$ of the functors $F,G$ is a function $\eta:\Ob_\C\to \Mor_{\C'}$ assigning to each $\C$-object $X$ a $\C'$-morphism $\eta_X\in \Mor_{\C'}(FX,GX)$ so that for any $\C$-objects $X,X$ and $\C$-morphism $f\in\Mor_\C(X,Y)$ the following diagram commutes.
$$\xymatrix{
FX\ar^{\eta_X}[r]\ar_{Ff}[d]&GX\ar^{Gf}[d]\\
FY\ar_{\eta_y}[r]&GY
}
$$
A natural transformation $\eta:F\to G$ is called an \index{isomorphism of functors}{\em isomorphism} of the funtors $F,G$ if for every $\C$-object $X$ the morphism $\eta_X:FX\to GX$ is an isomorphism of the category $\C'$. 

Two functors $F,G:\C\to\C'$ are called {\em isomorphic} if there exists an isomorphism $\eta:F\to G$. In this case we write $F\cong G$.
\end{definition}

\begin{exercise}\label{ex:functor-iso} Prove that for any functors $F,G,H:\C\to\C'$ between categories $\C,\C'$ the following properties hold:
\begin{itemize}
\item $F\cong F$;
\item $F\cong G\;\Ra\;G\cong F$;
\item $(F\cong G\;\wedge\;G\cong H)\;\Ra\;F\cong H$.
\end{itemize}
\end{exercise}

Now we can introduce the notion of equivalence for categories.

\begin{definition} Two categories $\C$ and $\C'$ are defined to be \index{equivalent categories}\index{categories!equivalent}{\em equivalent} (denoted by $\C\simeq\C'$) if 
there are two functors $F:\C\to\C'$ and $G:\C'\to\C$ such that $FG\cong \mathsf 1_{\C'}$ and $GF\cong \mathsf 1_{\C}$.
\end{definition}

\begin{exercise} Prove that for any categories $\C,\C',\C''$ the following properties hold:
\begin{itemize}
\item $\C\simeq \C$;
\item $\C\cong \C'\;\Ra\;\C'\cong \C$;
\item $(\C\cong \C'\;\wedge\;\C'\cong \C'')\;\Ra\;\C\cong \C''$.
\end{itemize}
Here by $\simeq$ we denote the equivalence of categories.
\end{exercise}

Let $\mathcal X,\mathcal Y$ be two categories. If the category $\mathcal X$ is small, then any functor $F:\mathcal X\to\mathcal Y$ is a set, so it is legal to consider the category $\mathcal Y^{\mathcal X}$ whose objects are functors from $\mathcal X$ to $\mathcal Y$ and whose morphisms are natural transformations between functors. If the categories $\mathcal X$ and $\mathcal Y$ are small, then the set $\mathcal Y^{\mathcal X}$ coincides with the set of morphisms $\Mor(\mathcal X,\mathcal Y)$ of the category $\mathbf{Cat}$ of small categories.

\begin{exercise} Let $F:\C\to\C'$ be a functor. Prove that for any $\C$-isomorphism $f$ the morphism $Ff$ is a $\C'$-isomorphism.
\end{exercise}

\begin{exercise} Let $F,G:\C\to\C'$ be isomorphic functors. Prove that for any functor
\begin{itemize}
\item[\textup{1)}]  $H:\C'\to\C''$ the functors $HF$ and $HG$ are isomorphic;
\item[\textup{2)}]  $H:\C''\to\C$ the functors $FH$ and $GH$ are isomorphic.
\end{itemize}
\end{exercise}

\section{Skeleta and equivalence of categories}

A category $\C$ is called \index{category!skeletal}\index{skeletal category}{\em skeletal} if any isomorphic objects in $\C$ coincide.

A category $\mathcal S$ is called a \index{skeleton}\index{category!skeleton of}{\em skeleton} of a category $\C$ if $\mathcal S$ is a full subcategory of $\C$ such that for any $\C$-object $X$ there exists a unique $\mathcal S$-object $Y$, which is $\C$-isomorphic to $X$. This definition implies that each skeleton of a category is a skeletal category.

The existence of skeleta in various categories follows from suitable forms of the Axiom of Choice.

\begin{exercise} Using the Principle of Mathematical Induction, show that every finite category has a skeleton.
\end{exercise}

To prove the existence of skeleta in arbitrary categories we shall apply the choice principle $(\mathsf{EC})$. This principle asserts that for every equivalence relation $R$ there exists a class $C$ such that for every $x\in \dom[R]$ the intersection $R[\{x\}]\cap C$ is a singleton. The principle $(\mathsf{GMP})$ is weaker than the Global Well-Orderability Principle $(\mathsf{GWO})$ but stronger than the Axiom of Global Choice $(\mathsf{AGC})$. On the other hand, $(\mathsf{GWO})\Leftrightarrow(\mathsf{EC})\Leftrightarrow(\mathsf{AGC})$ under the assumption of cumulativity of the universe $(\mathsf{C}\UU)$ that follows from the Axiom of Foundation, see Section~\ref{s:GChoice}.

\begin{theorem}\label{t:skeleton} Under $(\mathsf{EC})$, each category $\C$ has a skeleton $\mathcal S$. Moreover, there exists a full faithful functor $F:\C\to\mathcal S$ such that for the identity embedding functor $J:\mathcal S\to \C$ we have $FJ=\mathsf 1_{\mathcal S}$ and $JF\cong \mathsf 1_\C$.
\end{theorem}

\begin{proof}  Consider the equivalence relation $R=\{\langle x,y\rangle\in\Ob_\C\times\Ob_\C:x\cong_\C y\}$ on the class $\Ob_\C=\dom[R^\pm]$. By $(\mathsf{EC})$, there exists a subclass $S\subseteq\Ob_\C$ such that for every object $x\in\Ob_\C$ the intersection $R[\{x\}]\cap S$ is a singleton. 
Let $\mathcal S$ be the full subcategory of the category $\C$ whose class of objects coincides with $S$. It follows that any $\mathcal S$-isomorphic objects in the class $S=\Ob_{\mathcal S}$ are equal, which means that $\mathcal S$ is a skeleton of the category $\C$.
\smallskip

Consider the function $\dot F:\Ob_\C\to \Ob_{\mathcal S}$ assigning to each $\C$-object $x$ the unique element of the intersection $R[\{x\}]\cap S$. 
For every object $x\in\Ob_\C$, consider the class $\mathsf{Iso}(x,\dot F(x))$ of $\C$-isomorphisms $f:x\to\dot F(x)$. By Lemma~\ref{l:EC=>ACcc}, $(\mathsf{EC})\Ra(\mathsf{AC^c_c})$ and $(\mathsf{AC^c_c})$ implies the existence of a function $i_*:\Ob_\C\to\Mor_\C$ assigning to  every $\C$-object $x$ some isomorphism $i_x\in\mathsf{Iso}_\C(x,\dot F(x))$. Define a function $\ddot F:\Mor_\C\to\Mor_{\mathcal S}$ assigning to any $\C$-objects $a,b$ and $\C$-morphism $f\in\Mor_\C(a,b)$  the morphism $$i_{b}\circ f\circ i_{a}^{-1}\in\Mor_{\mathcal S}(\dot F(a),\dot F(b))=\Mor_\C(\dot F(a),\dot F(b)).$$
It is easy to check that $F=(\dot F,\ddot F):\C\to\mathcal S$ is a full faithful functor such that for the identity embedding functor $J=\mathsf 1_{\mathcal S,\mathcal C}:\mathcal S\to\C$ we have $FJ=\mathsf 1_{\mathcal S}$ and $JF\cong\mathsf 1_{\C}$. The natural transformation $i=(i_x)_{x\in\Ob_\C}:\mathsf 1_{\C}\to JF$ witnesses that  $\mathsf 1_{\C}\cong JF$, which means that the categories $\C$ and $\C'$ are equivalent.
\end{proof}

For small categories, we can replace the principle $(\mathsf{EC})$ in Theorem~\ref{t:skeleton} by the Axiom of Choice and obtain the following ``small'' version of Theorem~\ref{t:skeleton}.

\begin{theorem} Under $(\mathsf{AC})$ each small category has a skeleton.
\end{theorem}

For locally small categories the second part of Theorem~\ref{t:skeleton} can be proved  using the Axiom of Global Choice instead of   $(\mathsf{EC})$. 

\begin{theorem}\label{t:AGC-skelet} Let $\mathcal S$ be a skeleton of a locally small category $\C$. Under $(\mathsf{AGC})$, the categories $\C$ and $\mathcal S$ are equivalent.
\end{theorem}

Let $\Crd$ be the full subcategory of the category $\Set$, whose class of objects coincides with the class of cardinals. It is clear that $\Crd$ is a skeletal category and under $(\mathsf{AC})$, $\Crd$ is a skeleton of the category $\Set$.  
Applying Theorem~\ref{t:AGC-skelet} to this skeleton, we obtain the following theorem.

\begin{theorem}\label{t:Set=Card} Under $(\mathsf{AGC})$, the category $\Set$ is equivalent to its skeleton $\Crd$. 
\end{theorem}

Next, we prove some  criteria of equivalence and isomophness of categories. 

We recall that two categories $\C$ and $\C'$ are {\em isomorphic} if there exist functors $F:\C\to\C'$ and $G:\C'\to\C$ such that $GF=\mathsf 1_\C$ and $FG=\mathsf 1_{\C'}$.

\begin{theorem}\label{t:isomorphism} Two categories $\C=(\Ob,\Mor,\fdom ,\fran ,\mathsf 1,\circ)$ and $\C'=(\Ob',\Mor',\fdom ',\fran ',\mathsf 1',\circ')$ are isomorphic if and only if there exists a full faithful functor $F:\C\to\C'$ whose object part $\dot F:\Ob\to\Ob'$ is bijective.
\end{theorem}

\begin{proof} The ``only if'' part is trivial. To prove the ``if'' part, assume that $F:\C\to\C'$ is a full faithful functor whose object part $\dot F:\Ob\to\Ob'$ is bijective. Consider the function $\dot G=(\dot F)^{-1}:\Ob'\to\Ob$. We claim that the morphism part $\ddot F:\Mor\to\Mor'$ of the functor $F$ is bijective, too. Given two distinct morphisms $f,g\in\Mor$ consider the following cases.
\smallskip

1. If $\fdom (f)\ne \fdom (g)$, then $\fdom '(\ddot Ff)=\dot F(\fdom (f))\ne\dot F(\fdom (g))=\fdom '(\ddot F(g))$ and hence $\ddot Ff\ne\ddot Fg$.

2. If  $\fran (f)\ne \fran (g)$, then $\fran '(\ddot Ff)=\dot F(\fran (f))\ne\dot F(\fran (g))=\fran '(\ddot F(g))$ and hence $\ddot Ff\ne\ddot Fg$.

3. If $\fdom (f)=\fdom (g)$ and $\fran (f)=\fran (g)$, then $\ddot F(f)\ne\ddot F(g)$ since the functor $F$ is faithful.  
\smallskip

Therefore, the function $\ddot F:\Mor\to\Mor'$ is injective. To see that it is surjective, take any morphism $f'\in\Mor'$. Since the function $\dot F:\Ob\to\Ob'$ is bijective, there are $\C$-objects $X,Y$ such that $\dot F(X)=\fdom '(f')$ and $\dot F(Y)=\fran '(f')$. Since the functor $F$ is full, the function $\ddot F{\restriction}_{\Mor(X,Y)}:\Mor(X,Y)\to\Mor'(\fdom '(f'),\fran '(f'))$ is surjective, so there exists a morphism $f\in\Mor(X,Y)$ such that $\ddot F(f)=f'$. Therefore, the function $\ddot F:\Mor\to\Mor'$ is bijective and we can consider the function $\ddot G=(\ddot F)^{-1}:\Mor'\to\Mor'$. 

It is easy to check that $G=(\dot G,\ddot G):\C'\to \C$ is a functor such that $GF=\mathsf 1_{\C}$ and $FG=\mathsf 1_{\C'}$.
\end{proof}

We recall that two categories $\C,\C'$ are called {\em equivalent} if there exist functors $F:\C\to\C'$ and $G:\C'\to\C$ such that $GF\cong \mathsf 1_{\C}$ and $FG\cong\mathsf 1_{\C'}$ where $\cong$ stands for the isomorphism of functors.

\begin{proposition}\label{p:skelet:i=e} Two skeletal categories $\C,\C'$ are isomorphic if and only if they are equivalent.
\end{proposition}

\begin{proof} If categories $\C,\C'$ are equivalent, then there exist functors $F:\C\to\C'$ and $G:\C'\to\C$ such that $GF\cong \mathsf 1_{\C}$ and $FG\cong \mathsf 1_{\C'}$. The latter isomorphisms imply that the functors $GF$ and $FG$ are full and faithful and so are the functors $F$ and $G$. Since the categories $\C,\C'$ are skeletal, for any $\C$-object $X$ and $\C'$-object $X'$, the isomorphisms $GFX\cong X$ and $FGX'\cong X'$ imply $GFX=X$ and $FGX'=X'$. This means that the functors $F$ and $G$ are bijective on objects. By Theorem~\ref{t:isomorphism}, the categories $\C,\C'$ are isomorphic.
\end{proof}

A functor $F:\C\to\C'$ is defined to be \index{functor!surjective on objects}\index{functor!essentially surjective on objects}({\em essentially}) {\em surjective on objects} if for any $\C'$-object $Y$ there exists a $\C$-object $X$ such that $FX=Y$ (resp. $FX\cong Y$).

\begin{theorem}\label{t:equivalent} Under $(\mathsf{EC})$, two categories $\C,\C'$ are equivalent if and only if there exists a full faithful functor $F:\C\to\C'$ which is essentially surjective on objects.
\end{theorem}

\begin{proof} The ``only if'' part is trivial. To prove the ``if'' part, assume that here exists a full faithful functor $F:\C\to\C'$ which is essentially surjective on objects. Write the categories $\C$ and $\C'$ in expanded form as $6$-tuples: $\C=(\Ob,\Mor,\fdom ,\fran ,\mathsf 1,\circ)$ and  $\C'=(\Ob',\Mor',\fdom ',\fran ',\mathsf 1',\circ')$. 

By Theorem~\ref{t:skeleton}, under $(\mathsf{EC})$, the categories $\C,\C'$ have skeleta $\mathcal S\subseteq\C$, $\mathcal S'\subseteq\C$, and there are full faithful functors $R:\C\to\mathcal S$ and $R':\C'\to\mathcal S'$ such that for the indentity embeddings $J:\mathcal S\to\C$ and $J':\mathcal S'\to\mathcal C'$ we have $RJ=\mathsf 1_{\mathcal S}$, $JR\cong 1_{\C}$, $R'J'=\mathsf 1_{\mathcal S'}$, $J'R'\cong \mathsf 1_{\mathcal C'}$.

Since the functors $R',F,J$ are full and faithful, so is their composition $\Phi=R'FJ:\mathcal S\to\mathcal S'$. Since $\mathcal S'$ is a skeleton of $\C'$ and the functor $F$ is essentially surjective on objects, the functor $\Phi$ is surjective on objects. Next, we show that $\Phi$ is injective on objects. Assuming that $\Phi X=\Phi Y$ for some $\mathcal S$-objects $X,Y$ and using the faithful property of $\Phi$, we conclude that the functions $\ddot\Phi{\restriction}_{\Mor(X,Y)}:\Mor(X,Y)\to\Mor'(\Phi X,\Phi Y)$ and 
$\ddot\Phi{\restriction}_{\Mor(Y,X)}:\Mor(Y,X)\to\Mor'(\Phi Y,\Phi X)$
are bijective and hence there exist $\mathcal S$-morphisms $f\in\Mor(X,Y)$ and $g\in\Mor(Y,X)$ such that $\Phi f=\mathsf 1_{\Phi X}$ and $\Phi g=\mathsf 1_{\Phi Y}$. Then $\Phi(f\circ g)=\Phi f\circ'\Phi g=\mathsf 1_{\Phi X}\circ \mathsf 1_{\Phi Y}=\mathsf 1_{\Phi Y}=\ddot \Phi(\mathsf 1_Y)$ and hence $f\circ g=\mathsf 1_Y$ by the injectivity of the restriction $\ddot\Phi{\restriction}_{\Mor(Y,Y)}$. By analogy we can prove that $g\circ f=\mathsf 1_X$. This means that $f:X\to Y$ is a $\C$-isomorphism. Since the category $\mathcal S$ is skeletal, $X=Y$. Therefore, the function $\dot\Phi$ is bijective. By (the proof of) Theorem~\ref{t:isomorphism},  there exists a functor $\Psi:\mathcal S'\to\mathcal S$ such that $\Psi\Phi=\mathsf 1_{\mathcal S}$ and $\Phi\Psi=\mathsf 1_{\mathcal S'}$. Then for the functor $G=J\Psi R'\colon\C'\to\C$ we have $FG=FJ\Psi R'\cong J'R'FJ\Psi R'=J'\Phi\Psi R'=J'R'\cong 1_{\C'}$ and $GF=J\Psi R' F\cong J\Psi R' FJR=J\Psi\Phi R=JR\cong 1_\C$, witnessing that the categories $\C,\C'$ are equivalent.
\end{proof}

In the following theorem we use $(\mathsf{GWO})$, the principle of Global Well-Orderability, which is the strongest among Choice Principles, considered in Section~\ref{s:GChoice}.

\begin{theorem}\label{t:GWO-iso-cat} Under $(\mathsf{GWO})$, two categories $\C$ and $\C'$ are isomorphic if and only if there exists a full faithful functor $F:\C\to\C'$ such that $F$ is essentially surjective on objects and for any object $A\in\Ob_\C$ we have $|\{X\in \Ob_\C:X\cong_\C A\}|=|\{Y\in\Ob_{\C'}:Y\cong_{\C'} FA\}|$.
\end{theorem}

\begin{proof} The ``only if'' part is trivial. To prove the ``if'' part, assume that  there exists a full faithful functor $F:\C\to\C'$ such that $F$ is essentially surjective on objects, and for any $\C$-object $a$ we have $|\{x\in \Ob_\C:x\cong_\C a\}|=|\{y\in\Ob_{\C'}:y\cong_{\C'} Fa\}|$. 

Since $(\mathsf{GWO})\Ra(\mathsf{EC})$, we can apply Theorem~\ref{t:skeleton} and conclude that the categories $\C,\C'$ have skeleta $\mathcal S\subseteq\C$ and $\mathcal S'\subseteq\C'$, and there exist full faithful functors $R:\C\to\mathcal S$ and $R':\C'\to\mathcal S'$ such that for the identity embeddings $J:\mathcal S\to\C$ and $J':\mathcal S'\to\mathcal C'$ we have $RJ=\mathsf 1_{\mathcal S}$, $JR\cong 1_{\C}$, $R'J'=\mathsf 1_{\mathcal S'}$, $J'R'\cong \mathsf 1_{\mathcal C'}$.

Consider the functor $\Phi=R'FJ:\mathcal S\to\mathcal S'$. 
By (the proof of) Theorem~\ref{t:equivalent}, there exists a functor $\Psi:\mathcal S'\to\mathcal S$ such that $\Psi\Phi=\mathsf 1_{\mathcal S}$ and $\Phi\Psi=\mathsf 1_{\mathcal S'}$.  Observe that for every $a\in\Ob_{\mathcal S}$ we have $\Phi a\cong Fa$ and hence
\begin{equation}\label{eq:card-iso}
|\{x\in\Ob_\C:x\cong_\C a\}|=|\{y\in \Ob_{\C'}:y\cong_{\C'} Fa\}|=|\{y\in \Ob_{\C'}:y\cong_{\C'} \Phi a\}|.
\end{equation}
Consider the classes $\Ob^{\mathsf s}_{\C}=\{x\in\Ob_\C:\{y\in \Ob_\C:y\cong_\C x\}\in\UU\}$ and $\Ob_{\mathcal S}^{\mathsf s}=\Ob_\C^{\mathsf s}\cap\Ob_{\mathcal S}$. The  existence of these classes follows from Theorem~\ref{t:class}. Using the Axiom of Global Choice (which follows from $(\mathsf{GWO}))$ and the equality (\ref{eq:card-iso}), we can find a function $\Theta_*: \Ob^\mathsf{s}_{\mathcal S}\to\UU$ assigning to each object $a\in \Ob^{\mathsf s}_{\mathcal S}$ a bijective function $\Theta_a$ such that $\dom[\Theta_a]=\{x\in\Ob_\C:x\cong_\C a\}$, $\rng[\Theta_a]=\{y\in\Ob_{\C'}:y\cong_{\C'} \Phi a\}$ and $\Theta_a(a)=\Phi a$.  

By $(\mathsf{GWO})$ there exists a set-like well-order $\mathbf W$ with $\dom[\mathbf W^\pm]=\UU$. This well-order  induces the set-like well-founded orders $$W=\{\langle x,y\rangle\in \mathbf W:x,y\in\Ob_{\mathcal C}\setminus\Ob_{\mathcal S},\;x\cong_\C y\}\cup\{\langle x,y\rangle\in \Ob_{\mathcal S}\times\Ob_{\C}:x\cong_\C y\}$$ and
 $$W'=\{\langle x,y\rangle\in \mathbf W:x,y\in\Ob_{\mathcal C'}\setminus\Ob_{\mathcal S'},\;x\cong_{\C'} y\}\cup\{\langle x,y\rangle\in \Ob_{\mathcal S'}\times\Ob_{\C'}:x\cong_{\C'} y\}.$$
For every $\C$-object $a$ the well-order $W$ determines a set-like well-order on the equivalence class $[a]_{\cong}=\{x\in\Ob_\C:a\cong_\C x\}$ such that the unique element of the intersection $\Ob_{\mathcal S}\cap [a]_{\cong}$ is the unique $W$-minimal element of  $[a]_{\cong}$. If $a\notin \Ob_\C^{\mathcal s}$, then $[a]_{\cong}$ is not a set and hence the well-ordered class $([a]_{\cong},W{\restriction}[a]_{\cong})$ is order-isomorphic to $\Ord$ by Theorem~\ref{t:wOrd}.
By (\ref{eq:card-iso}), the same is true for the object $b=\Phi a$ and its equivalnce class  $[b]_{\cong}=\{y\in\Ob_{\C'}:y\cong_{\C'}b\}$.  Since $|[b]_{\cong}|=|[a]_{\cong}|$, $[b]_{\cong}$ is not a set and then the well-ordered class $([b]_{\cong},W'{\restriction}[b]_{\cong})$ is order isomorphic to $\Ord$.

 Let $\rank_{W}:\UU\to\Ord$ and $\rank_{W'}:\UU\to\Ord$ be the rank functions of the well-founded set-like orders $W$ and $W'$. The definition of the orders $W,W'$ implies that for every $a\in\Ob_\C\setminus\Ob_\C^{\mathsf s}$ and $b=\Phi a$ the restrictions $\rank_W{\restriction}_{[a]_{\cong}}:[a]_{\cong}\to\Ord$ and $\rank_{W'}{\restriction}_{[b]_{\cong}}:[b]_{\cong}\to\Ord$ are bijections.

Now consider the function $\dot G:\Ob_\C\to\Ob_{\C'}$ assigning to each object $a\in \Ob^{\mathsf s}_\C$ the object $\Theta_{Ra}(a)$ and to each object $a\in \Ob_\C\setminus\Ob^{\mathsf s}_\C$ the unique object $b\in\Ob_{\C'}\setminus\Ob^{\mathsf s}_{\C'}$ such that $b\cong \Phi Ra$ and $\rank_{W'}(b)=\rank_W(a)$. The choice of the function $\Theta_*$ and the well-orders $W,W'$ ensures that the function $\dot G:\Ob_\C\to\Ob_{\C'}$ is bijective and $\dot\Phi\subseteq\dot G$. Consider the function $\ddot G:\Mor_{\C}\to\Mor_{\C'}$ assigning to any $\C$-objects $a,b$ and morphism $f\in\Mor_\C(a,b)$ a unique morphism $g\in\Mor_{\C'}(\dot G(a),\dot G(b))$ such that $R'g=\Phi Rf$. It can be shown that the function $\ddot G$ is bijective and hence $G:\C\to\C'$ is a full faithful functor whose object part $\dot G$ is bijective. By Theorem~\ref{t:isomorphism}, the categories $\C,\C'$ are isomorphic.
\end{proof}

\section{Limits and colimits}

In this section we discuss the notions of limit and colimit of a diagram in a category.

Limits  and colimits are general categorial notions whose partial cases are products and coproducts, pullbacks and pushouts, equalizers and coequalizes.

\subsection{Products and coproducts} Let $\C$ be a category. 

\begin{definition} For two $\C$-objects $X,Y$, their \index{product of objects}{\em $\C$-product} is a $\C$-object $X\times Y$ endowed with two $\C$-morphisms $\pi_X\in\Mor_\C(X\times Y,X)$ and $\pi_Y\in\Mor_\C(X\times Y, Y)$, called \index{cordinate projection}{\em the coordinate projections}, such that the triple $(X\times Y,\pi_X,\pi_Y)$ has the following universal property: for any $\C$-object $A$ and $\C$-morphisms $f\in\Mor_\C(A,X)$ and $g\in\Mor_\C(A,Y)$ there exists a unique $\C$-morphism $h\in\Mor_\C(A,X\times Y)$  such that $f=\pi_X\circ h$ and $g=\pi_Y\circ h$. 
\end{definition}
This definition is illustrated by the diagram:
$$
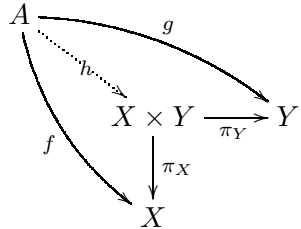
\xymatrix{
A\ar@/^10pt/^{g}[rrd]\ar@/_10pt/_{f}[rdd]\ar@{..>}|{h}[rd]\\
&X\times Y\ar_(.6){\pi_Y}[r]\ar^{\pi_X}[d]&Y\\
&X\\
}
$$

The uniqueness of the morphism $h$ implies that a $\C$-product $X\times Y$ if exists, then is unique up to a $\C$-isomorphism. 


\begin{exercise} Prove that for any sets $X,Y$, their $\C$-product $X\times Y$ endowed with the projections $\pi_X=\dom{\restriction}_{X\times Y}$ and $\pi_Y=\rng{\restriction}_{X\times Y}$ is a products of $X$ and $Y$ is the category $\Set$.
\end{exercise}

\begin{exercise} Identify products of objects in the categories considered in Examples~\ref{ex:sMS}--\ref{ex:po-cat}.
\end{exercise}

\begin{definition} A category $\C$ is defined to have \index{category!with binary products}{\em binary products} if for any $\C$-objects $X,Y$ there exists a product $X\times Y$ in $\C$.
\end{definition}

The dual notion to a product is that of a coproduct.

\begin{definition} For two $\C$-objects $X,Y$ their \index{coproduct}{\em $\C$-coproduct} is any $\C$-object $X\sqcup Y$ endowed with two $\C$-morphisms  $i_X:X\to X\sqcup Y$ and $i_Y:Y\to X\sqcup Y$, called the \index{coordinate coprojection}{\em coordinate coprojections}, such that the triple $(X\sqcup Y,i_X,i_Y)$ has the following universal property: for any $\C$-object $A$ and $\C$-morphisms $f:X\to A$ and $g:Y\to A$ there exists a unique $\C$-morphism $h:X\sqcup  Y\to A$ such that $f=i_X\circ h$ and $g=i_Y\circ h$. 
\end{definition}
This definition is illustrated by the diagram:
$$
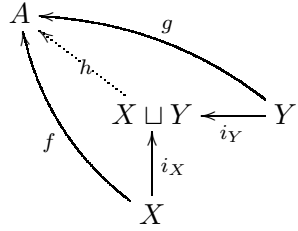
\xymatrix{
A\\
&X\sqcup Y\ar@{..>}|{h}[lu]&Y\ar@/_10pt/_{g}[llu]\ar^(.4){i_Y}[l]\\
&X\ar@/^10pt/^{f}[uul]\ar_{i_X}[u]\\
}
$$

The uniqueness of the morphism $h$ implies that a $\C$-coproduct $X\sqcup Y$ if exists, then is unique up to a $\C$-isomorphism. 

\begin{exercise} Prove that for any disjoint sets $X,Y$, their union $X\cup Y$ endowed with the identity embeddings $i_X=\Id{\restriction}_X:X\to X\cup Y$ and $i_Y=\Id_Y{\restriction}_Y:X\to X\cup Y$ is a $\Set$-coproduct of $X$ and $Y$.
\end{exercise}

\begin{exercise} Prove that for any sets $X,Y$ the set $(X,Y)=(\{0\}\times X)\cup(\{1\}\times Y)$ endowed with natural injective functions $i_X:X\to (X,Y)$ and $i_Y:Y\to(X,Y)$ is a $\Set$-coproduct of $X$ and $Y$.
\end{exercise}

\begin{exercise} Identify coproducts of objects in the categories considered in  Examples~\ref{ex:sMS}--\ref{ex:po-cat}.
\end{exercise}







In fact, products and coproducts can be defined for any indexed families of objects.

\begin{definition} Let $\C$ be a category. For an indexed family of $\C$-objects $(X_i)_{i\in I}$ its
\begin{itemize}
\item   \index{product}{\em $\C$-product} $\prod_{i\in I}X_i$ is any $\C$-object $X$ endowed with a family of $\C$-morphisms $(\pi_i)_{i\in I}\in\prod_{i\in I}\Mor_\C(X,X_j)$, which has the following universal property: for  any $\C$-object $Y$ and a family of $\C$-morphisms $(f_i)_{i\in I}\in\prod_{i\in I}\Mor_\C(Y,X_i)$ there exists a unique $\C$-morphism $h\in \Mor_\C(A,X)$ such that $\forall i\in I\;\;f_i=\pi_i\circ h$;
\item \index{coproduct}{\em $\C$-coproduct} $\amalg_{i\in I}X_i$ is any $\C$-object $X$ endowed with a family of $\C$-morphisms $(e_i)_{i\in I}\in\prod_{i\in I}\Mor_\C(X_i,X)$, which has the following universal property: for  any $\C$-object $Y$ and  a family of $\C$-morphisms $(f_i)_{i\in I}\in\prod_{i\in I}\Mor_\C(X_i,Y)$ there exists a unique $\C$-morphism $h\in \Mor_\C(X,Y)$ such that $\forall i\in I\;\;f_i=h\circ e_i$.
\end{itemize}
\end{definition}

\begin{definition} A category $\C$ is defined
\begin{itemize}
\item \index{category!with finite products}{\em to have finite products}  if for any finite set $I$ and indexed family of $\C$-objects $(X_i)_{i\in I}$, the category $\C$ contains a product $\prod_{i\in I}X_i$ of this family;
\item \index{category!with arbitrary products}{\em to have arbitrary products}  if for any set $I$ and indexed family of $\C$-objects $(X_i)_{i\in I}$, the category $\C$ contains a product $\prod_{i\in I}X_i$ of this family;

\item \index{category!with finite coproducts}{\em to have finite coproducts}  if for any finite set $I$ and indexed family of $\C$-objects $(X_i)_{i\in I}$, the category $\C$ contains a coproduct $\amalg_{i\in I}X_i$ of this family;
\item \index{category!with arbitrary coproducts}{\em to have arbitrary coproducts}  if for any set $I$ and indexed family of $\C$-objects $(X_i)_{i\in I}$, the category $\C$ contains a coproduct $\amalg_{i\in I}X_i$ of this family.
\end{itemize}
\end{definition}

\begin{exercise} Prove that a category has finite products if and only if it has binary products.
\end{exercise}

\begin{exercise} Prove that for any indexed family of sets $(X_i)_{i\in I}$ their Cartesian product $\prod_{i\in I}X_i$ endowed with the natural projections is a $\Set$-product of the family $(X_i)_{i\in I}$.
\end{exercise}

\begin{exercise} Prove that for any indexed family of pairwise disjoint sets $(X_i)_{i\in I}$ their union $\bigcup_{i\in I}X_i$ endowed with the identity inclusions of the sets $X_i$ is a $\Set$-coproduct of the family $(X_i)_{i\in I}$.
\end{exercise}

\begin{exercise} Prove that for any indexed family of   sets $(X_i)_{i\in I}$ the set $\bigcup_{i\in I}(\{i\}\times X_i)$ endowed with the natural embeddings of the sets $X_i$ is a $\Set$-coproduct of the family $(X_i)_{i\in I}$.
\end{exercise}

\begin{exercise} Identify products of objects in the categories considered in  Examples~\ref{ex:sMS}--\ref{ex:po-cat}.
\end{exercise}

\begin{exercise} Prove that the category $\mathbf{Cat}$ of small categories has arbitrary products and coproducts.
\end{exercise}

\subsection{Pullbacks and pushouts} In this subsection we consider pullbacks and pushouts, called also fibered products and coproducts.

Given a category $\C$, consider the following diagram consisting of three $\C$-objects $X,Y,Z$ and two $\C$-morphisms $g_X:X\to Z$ and $g_Y:Y\to Z$.
$$
\xymatrix{
&Y\ar^{g_Y}[d]\\
X\ar_{g_X}[r]&Z
}
$$
The \index{pullback}{\em pullback} $X\times_Z Y$ of this diagram is any $\C$-object $P$ endowed with two $\C$-morphisms $\pi_XLP\to X$ and $\pi_Y:P\to Y$ such that $g_X\circ\pi_X=g_Y\circ\pi_Y$ and the triple $(P,\pi_X,\pi_Y)$ has the following universality property: for any $\C$-object $P'$ and $\C$-morphisms $f_X:P'\to X$, $f_Y:P'\to Y$ with $g_X\circ f_X=g_Y\circ f_Y$, there exists a unique $\C$-morphism $h:P'\to P$ such that $f_X=\pi_X\circ h$ and $f_Y=\pi_Y\circ h$.
This definition is better seen at the diagram:
$$
\xymatrix{
P'\ar@/^10pt/^{\pi'_Y}[rrd]\ar@/_10pt/_{\pi'_X}[rdd]\ar@{..>}|{h}[rd]\\
&X\times_Z Y\ar^{\pi_Y}[r]\ar_{\pi_X}[d]&Y\ar^{g_Y}[d]\\
&X\ar_{g_X}[r]&Z
}
$$
If the category $\C$ has a terminal object $\mathtt 1$, then the product $X\times Y$ is a pullback $X\times_{\mathtt 1}Y$ of the diagram
$$
\xymatrix{
&Y\ar[d]\\
X\ar[r]&\mathtt 1.
}
$$

The notion of a pushout is dual to that of pullback.
\smallskip

Given a category $\C$, consider the following diagram consisting of three $\C$-objects $X,Y,Z$ and two $\C$-morphisms $g_X:Z\to X$ and $g_Y:Z\to Y$:
$$
\xymatrix{
&Y\\
X&Z\ar_{g_Y}[u]\ar^{g_X}[l]
}
$$
The \index{pushout}{\em pushout} $X\sqcup_Z Y$ of this diagram is any $\C$-object $P$ endowed with two $\C$-morphisms $i_X:X\to P$ and $i_Y:Y\to P$ such that $i_X\circ g_X=i_Y\circ g_Y$ and the triple $(P,i_X,i_Y)$ has the following universal property: for any $\C$-object $P'$ and $\C$-morphisms $f_X:X\to P'$, $f_Y:Y\to P'$ with $f_X\circ g_X=f_Y\circ g_Y$, there exists a unique $\C$-morphism $h: P'\to P$ such that $f_X=h\circ i_X$ and $f_Y=h\circ i_Y$.
This definition is better seen at the diagram:
$$
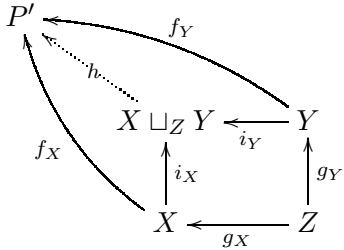
\xymatrix{
P'\\
&X\sqcup_Z Y\ar@{..>}|{h}[lu]&Y\ar^(.4){i_Y}[l]\ar@/_10pt/_{f_Y}[llu]\\
&X\ar_{i_X}[u]\ar@/^10pt/^{f_X}[luu]&Z\ar^{g_X}[l]\ar_{g_Y}[u]
}
$$
If the category $\C$ has an initial object $\mathtt 0$, then the coproduct $X\sqcup Y$ is a pullback $X\sqcup_{\mathtt 0}Y$ of the diagram
$$
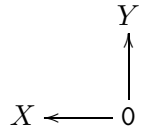
\xymatrix{
&Y\\
X&\mathtt 0\ar[l]\ar[u]
}
$$

A category $\C$ is defined 
\begin{itemize}
\item \index{category!with pullbacks}{\em to have pullbacks} if any diagram consisting of two  $\C$-morphisms $g_X:X\to Z$ and $g_Y:Y\to Z$  has a pullback. 
\item \index{category!with pushouts}{\em to have pushouts} if any diagram consisting of two  $\C$-morphisms $g_X:Z\to X$ and $g_Y:Z\to X$  has a pushout.
\end{itemize}
It is clear that a category has pullbacks if and only if the dual category has pushouts.

\subsection{Equalizers and coequalizers} For two objects $X,Y$ of a category $\C$ and two morphisms $f,g\in\Mor_\C(X,Y)$, an \index{equalizer}{\em equalizer} of the pair of $(f,g)$ is a $\C$-object $E$ endowed with a $\C$-morphism $e:E\to X$ such that $f\circ e=g\circ e$ and for any $\C$-object $E'$ and $\C$-morphism $e':E'\to X$ with $f\circ e'=g\circ e'$ there exists a unique $\C$-morphism $h:E'\to E$ such that $e'=e\circ h$.
This definition is better seen on the commutative diagram:
$$
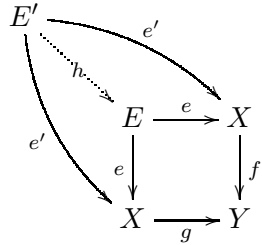
\xymatrix{
E'\ar@/^10pt/^{e'}[rrd]\ar@/_10pt/_{e'}[rdd]\ar@{..>}|{h}[rd]\\
&E\ar^{e}[r]\ar_{e}[d]&X\ar^{f}[d]\\
&X\ar_{g}[r]&Y
}
$$
\begin{example} In the category $\Set$ an equalizer of two functions $f,g:X\to Y$ is the  set $E=\{x\in X:f(x)=g(x)\}$ endowed with the identity embedding $e:E\to X$.
\end{example}

\begin{exercise}\label{ex:p+e=>pb} Prove that a category $\C$ has pullbacks if it has  equalizers and binary products.
\smallskip

\noindent{\em Hint}: Observe that for morphisms $f:X\to Z$ and $g:Y\to Z$ the pullback $X\times_Z Y$ is isomorphic to the equalizer $E$ of the pair $(f\circ\pr_X,g\circ \pr_Y)$ where $\pr_X,\pr_Y$ are coordinate projections of the product $X\times Y$.
\end{exercise}

Coequalizers are defined dually. 

 For two objects $X,Y$ of a category $\C$ and two morphisms $f,g\in\Mor_\C(X,Y)$, a \index{coequalizer}{\em coequalizer} of the pair $(f,g)$ is a $\C$-object $E$ endowed with a $\C$-morphism $e:Y\to E$ such that $e\circ f=e\circ g$ and for any $\C$-object $E'$ and $\C$-morphism $e':X\to E'$ with $e'\circ f=e'\circ g$ there exists a unique $\C$-morphism $h:E\to E'$ such that $e'=h\circ e$.
On a diagram this definition looks as follows.
$$
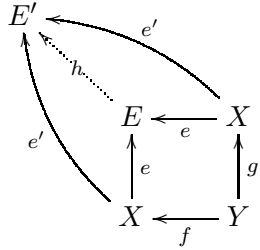
\xymatrix{
E'\\
&E\ar@{..>}|{h}[lu]&X\ar^e[l]\ar@/_10pt/_{e'}[llu]\\
&X\ar_{e}[u]\ar@/^10pt/^{e'}[luu]&Y\ar^{f}[l]\ar_{g}[u]
}
$$


\begin{exercise} Find a coequalizer of two functions in the category of sets.
\end{exercise}

A category $\C$ is defined \index{category!with equalizers}\index{category!with coequalizers}{\em to have} ({\em co}){\em equalizers} if any pair of $\C$-morphisms $f,g:X\to Y$ has a (co)equalizer.

\subsection{Limits and colimits of diagrams} Products, pullbacks, and equalizers are particular cases of limits of diagrams in a category. 

By definition, a \index{diagram}{\em diagram} in a category $\C$ is any functor $D:\mathcal D\to\mathcal C$ defined on a small category $\mathcal D$. For a fixed small category $\mathcal D$, a functor $D:\mathcal D\to\C$ is called a {\em $\mathcal D$-diagram} in $\C$.


\begin{definition}\label{d:limit} Let $\mathcal D$ be a small category and $D:\mathcal D\to\C$ be a $\mathcal D$-diagram in  a category $\C$.
\begin{itemize}
\item A \index{cone over a diagram}\index{diagram!cone over}{\em cone} over the $\mathcal D$-diagram $D$ is a pair $(V,f)$ consisting of a $\C$-object $V$ and a function $f:\Ob_{\mathcal D}\to\Mor_\C$ assigning to each $\mathcal D$-object $X\in\Ob_{\mathcal D}$ a $\C$-morphism $f_X\in \Mor_\C(V,FX)$  such that for any $\mathcal D$-morphism $g:X\to Y$ we have $f_Y=Fg\circ f_X$.
\smallskip
\item A \index{limit of a diagram}\index{diagram!limit of}{\em limit}  of the $\mathcal D$-diagram $D$ is any cone $(V,f)$ over $F$ that has the following universal property: for any cone $(V',f')$ over $F$ there exists a unique $\C$-morphism $h:V'\to V$ such that $\forall X\in \Ob_{\mathcal D}\;\;f'_X=f_X\circ h$.
\end{itemize}
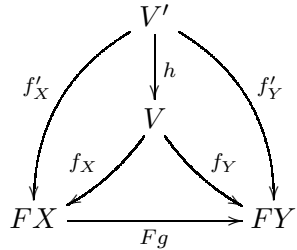
$$
\xymatrix{
&V'\ar@/_15pt/_{f'_X}[ldd]\ar@/^15pt/^{f'_Y}[ddr]\ar^h[d]\\
&V\ar@/^5pt/_{f_X}[ld]\ar@/_5pt/^{f_Y}[dr]\\
FX\ar_{Fg}[rr]&&FY
}
$$
\end{definition}

The uniqueness of the morphism $h$ in Definition~\ref{d:limit} implies the following useful  fact.

\begin{proposition} Let $(V,f)$ be a limit of a $\mathcal D$-diagram in a category $\C$. A $\C$-morphism $h:V\to V$ is equal to $\mathsf 1_V$ if and only if $\forall X\in\Ob(\mathcal D)\;\;f_X\circ h=f_X$.
\end{proposition}

The notion of a colimit is dual to the notion of a limit.

\begin{definition}\label{d:colimit}  Let $\mathcal D$ be a small diagram and $D:\mathcal D\to\C$ be a $\mathcal D$-diagram in a category $\C$.
\begin{itemize}
\item A \index{cocone of a diagram}\index{diagram!cocone of}{\em cocone} over the $\mathcal D$-diagram $D$ is a pair $(V,f)$ consisting of an object $V$ of the category $\C$ and a function $f:\Ob_{\mathcal D}\to\Mor_\C$ assigning to each $\mathcal D$-object $X\in\Ob_{\mathcal D}$ a $\C$-morphism $f_X:FX\to V$ such that for any $\mathcal D$-morphism $g:X\to Y$ we have $f_X=f_Y\circ Fg$.
\item A \index{colimit of a diagram}\index{diagram!colimit of}{\em colimit}  of the $\mathcal D$-diagram $D$ is any cocone $(V,f)$ over $F$ that has the following universal property: for any cocone $(V',f')$ over $F$ there exists a unique $\C$-morphism $h:V\to V'$ such that $\forall X\in\Ob_{\mathcal D}\;\;f'_X=h\circ f_X$.
\end{itemize}
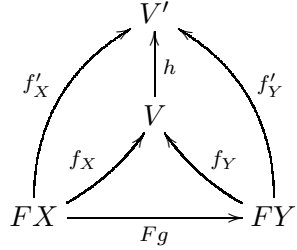
$$
\xymatrix{
&V'\\
&V\ar_h[u]\\
FX\ar@/^15pt/^{f'_X}[ruu]\ar_{Fg}[rr]\ar@/_5pt/^{f_X}[ru]&&FY\ar@/^5pt/_{f_Y}[ul]\ar@/_15pt/_{f'_Y}[uul]
}
$$
\end{definition}

The uniqueness of the morphism $h$ in Definition~\ref{d:colimit} implies the following useful fact that will be used in the proof of Lemma~\ref{l:Equi-Set}.

\begin{proposition}\label{p:colimit} Let $(V,f)$ be a colimit of a $\mathcal D$-diagram in a category $\C$. A $\C$-morphism $h\in\Mor_\C(V,V)$ is equal to $\mathsf 1_V$ if and only if $\forall X\in\Ob_{\mathcal D}\;\;(h\circ f_X=f_X)$.
\end{proposition}

Limits and colimits of diagrams are unique up to  (a properly defined notion of) an isomorphism.

\begin{remark} (Co)products are (co)limit of $\mathcal D$-liagram over discrete categories $\mathcal D$.
 \end{remark}
 
 \begin{exercise} Which diagrams $\mathcal D$ correspond to pullbacks and equalizers?
 \end{exercise} 
 
\begin{exercise} Investigate the existence and structure of limits and colimits in your favorable category.
\end{exercise}

\begin{Exercise}\label{ex:f-limits} Prove that a category has limits of finite diagrams if and only if it has binary products and equalizers.
\end{Exercise}


\section{Characterizations of the category $\Set$}
\bigskip

{\small
\em 

\rightline{We adjoin eight first-order axioms to the usual} 

\rightline{first-order theory of an abstract Eilenberg-Mac Lane category}

\rightline{to obtain an elementary theory with the following properties:}

\rightline{(a) There is essentially only one category which satisfies these eight
axioms}

\rightline{together with the additional (nonelementary) axiom of completeness,}

\rightline{namely, the category $\Set$ of sets and mappings.}

\rightline{Thus our theory distinguishes $\Set$ structurally from other complete categories,}

\rightline{such as those of topological spaces, groups, rings, partially ordered sets, etc.}

\rightline{(b) The theory provides a foundation for number theory, analysis, ... algebra and topology}

\rightline{even though no relation $\in$ with the traditional properties can be defined.}

\rightline{Thus we seem to have partially demonstrated that even in foundations}

\rightline{not Substance but invariant Form is the carrier of the relevant mathematical information.}
}
\smallskip

\rightline{\small William Lawvere, 1964}
\vskip30pt

In this section we characterize categories which are equivalent or isomorphic to the category $\Set$. The principal results are Theorems~\ref{t:EC=Set} and \ref{t:GWO=Set} which are simplified versions of the characterization of the category of sets, proved by  Lawvere in 1964 (before he created the theory of elementary topoi). 

A distinguishing property of the category of sets is that its terminal object is a generator for this category.

\begin{definition} Let $\C$ be a category. A $\C$-object $\Gamma$ is called a \index{generator}{\em generator} (more precisely, a {\em $\C$-generator}) if for any  $\C$-objects $X,Y$ and distinct $\C$-morphisms $f,g\in\Mor(X,Y)$ there exists a $\C$-morphism $h\in\Mor(\Gamma,X)$ such that $f\circ h\ne g\circ h$.
\end{definition}

\begin{exercise} Observe that in the categories $\Set$ and $\mathbf{Top}$ any non-initial  object is a generator. 
\end{exercise}

\begin{exercise} Prove that the additive group of integers $(\IZ,+)$ is a generator in the category of (commutative) groups.
\end{exercise}


\begin{definition} A category $\C$ with a terminal object $\mathtt 1$ is defined to be \index{category!element-separating}\index{element-separating category}{\em element-separating} if  for any morphism $x:\mathtt 1\to X$ there exist a $\C$-object $\Omega$ and  $\C$-morphisms $\chi_x:X\to \Omega$ and $\false:\mathtt 1\to \Omega$ such that $\chi_x\circ x\ne\false$ and $\forall x'\in\Mor(\mathtt 1,X)\setminus\{x\}\;\;(\chi_x\circ x'=\false)$.
\end{definition}

We recall that a category $\C$ is \index{category!balanced}{\em balanced} if a $\C$-morphism is an isomorphism if and only if it is a bimorphism (i.e., mono and epi).

\begin{lemma}\label{l:Equi-Set} A skeletal category $\C$ is equivalent to the category $\Set$ if and only if it satisfies the following properties:
\begin{enumerate}
\item[\textup{(1)}] $\C$ is locally small;
\item[\textup{(2)}] $\C$ is balanced;
\item[\textup{(3)}] $\C$ has equalizers;
\item[\textup{(4)}] $\C$ has arbitrary coproducts;
\item[\textup{(5)}] $\C$ has a terminal object $\mathtt 1$;
\item[\textup{(6)}] $\mathtt 1$ is a $\C$-generator;
\item[\textup{(7)}] $\C$ is element-separating.
\end{enumerate}
\end{lemma}

\begin{proof}  The ``only if'' part is trivial. To prove the ``if'' part, assume that a skeletal  category $\C$ has the properties (1)--(7). 
Then $\C$ has a  terminal object $\mathtt 1$, which is unique by the skeletality of $\C$. 

Consider the functor $G:\C\to\Set$ assigning to every $\C$-object $X$ the set $\Mor(\mathtt 1,X)$. Elements of the set $\Mor(\mathtt 1,X)$ are called \index{global elements} {\em global elements} of $X$. To every $\C$-morphism $f:X\to Y$ the functor $G$ assigns the function $$Gf:\Mor(\mathtt 1,X)\to\Mor(\mathtt 1,Y),\quad Gf:x\mapsto f\circ x.$$
Therefore, the functor $G$ assigns to each $\C$-object $X$ the set $\Mor(\mathtt 1,X)$ of its global element.

Now we describe a functor $F:\Set\to\C$ acting in the opposite direction. The functor $F$ assigns to every set $X$ the coproduct $\sqcup_{x\in X}\mathtt 1$ of $X$ many copies of the terminal object $\mathtt 1$. Since the category has arbitrary coproducts and is skeletal, the coproduct $\sqcup_{x\in X}\mathtt 1$ exists and is unique. Let $\eta_X:X\to \Mor(\mathtt 1,FX)$ be the function assigning to every element $x\in X$ the coordinate coprojection $\eta_X(x):\mathtt 1\to FX$. By definition of a coproduct, the function $\eta_X$ has the following universal property: for every $\C$-object $Y$ and function $g:X\to\Mor(\mathsf 1,Y)$ there exists a unique $\C$-morphism $h:FX\to Y$ such that $h\circ\eta_X(x)=g(x)$ for every $x\in X$. In particular, for every function $f:X\to X'$ between sets, there exists a unique $\C$-morphism $Ff:FX\to FX'$ such that 
\begin{equation}\label{eq:def-Ff}
Ff\circ \eta_X(x)=\eta_{X'}\circ f(x)\quad\mbox{for every $x\in X$}.
\end{equation}
 This formula defines the action of the functor $F$ on the morphisms of the category $\Set$.

Consider the natural transformation $\eta:\mathsf 1_{\Set}\to GF$ whose components are the $\Set$-morphisms $\eta_X:X\to \Mor(\mathtt 1,FX)=GFX$.

\begin{claim}\label{cl:eta-inj} The function $\eta_X:X\to GFX$ is injective.
\end{claim}

\begin{proof} To prove that the function $\eta_X:X\to GFX=\Mor(\mathtt 1,FX)$ is injective, fix any distinct elements $x,x'\in X$. Since $\C$ is element-separating, there exists a $\C$-object $\Omega$ such that $\Mor(\mathtt 1,\Omega)$ contains at least two distinct morphisms. Then we can choose a function $\chi:X\to\Mor(\mathtt 1,\Omega)$ such that $\chi(x)\ne\chi(x')$. By the definition of the natural transformation $\eta_X$, there exists a unique $\C$-morphism $u:FX\to \Omega$ such that $u\circ \eta_X(z)=\chi(z)$ for every $z\in X$. 
In particular, $$u\circ\eta_X(x)=\chi(x)\ne\chi(x')=u\circ\eta_X(x'),$$
which implies that $\eta_X(x)\ne \eta_X(x')$.
\end{proof}

\begin{claim}\label{cl:eta-sur} The function $\eta_X:X\to GFX$ is surjective.
\end{claim}

\begin{proof} 
Assuming that $\eta_X$ is not surjective, we can find a morphism $\psi\in GFX=\Mor(\mathtt 1,FX)$ such that $\psi\ne\eta_X(x)$ for any $x\in X$. 
Since the category $\C$ is element-separating, for the morphism $\psi:\mathtt 1\to FX$ there exist a $\C$-object $\Omega$ and $\C$-morphisms $\chi:FX\to\Omega$ and $\false,\true:\mathtt 1\to\Omega$ such that  $\chi\circ \psi=\true$ and for any global element $\varphi\in\Mor(\mathtt 1,FX)\setminus\{\psi\}$ we have $\chi\circ\varphi=\false\ne\true$.

Let $\tau:FX\to \mathtt 1$ be the unique $\C$-morphism and $\zeta=\mathsf{false}\circ \tau$. 
Since the category $\C$ has equalizers, there exist a $\C$-object $E$ and a $\C$-morphism $e\in \Mor(E,FX)$ such that $\chi\circ e=\zeta\circ e$ and for any $\C$-object $E'$ and morphism $e'\in\Mor(E',FX)$ with $\chi\circ e'=\zeta\circ e'$ there exists a unique morphism $h\in \Mor(E',E)$ such that $e\circ h=e'$. 
We apply this  universal property of $(E,e)$ to the pair $(E',e')=(\mathtt 1,\eta_{X}(x))$ where $x\in X$ is any element.
$$
\xymatrix{
&\mathtt 1\ar@/^10pt/^{\eta_X(x)}[ddr]\ar^{h_x}[d]\ar@/_10pt/_{\eta_X(x)}[ddl]\\
&E\ar_e[dr]\ar^e[dl]\\
FX\ar^{\zeta}[rd]\ar_{\tau}[d]&&FX\ar_{\chi}[ld]\\
\mathtt 1\ar_{\mathsf{false}}[r]&\Omega&\mathtt 1\ar^{\true}[l]\ar_{\psi}[u]\\
}
$$

By the choice of $\chi$, the inequality $\psi\ne\eta_X(x)$, implies
$$\chi\circ\eta_X(x)=\mathsf{false}=\mathsf{false}\circ\mathsf 1_{\mathtt 1}=\mathsf{false}\circ (\tau\circ \eta_X(x))=\zeta\circ\eta_{X}(f).$$  By the universal property of the equalizer  morphism $e:E\to FX$, there exists a unique morphism $h_f\in \Mor(\mathtt 1,E)$ such that $\eta_{X}(x)=e\circ h_f$. By the universal property of the coproduct $(FX,\eta_{X})$, there exists a unique $\C$-morphism $e':FX\to E$ such that $h_x=e'\circ \eta_X(x)$ for all $f\in\Mor_\C(\mathtt 1,X)$. Then $\eta_X(x)=e\circ h_x=e\circ e'\circ\eta_X(x)$ for all $x\in X$ and hence $e\circ e'=\mathsf 1_{FX}$, see Proposition~\ref{p:colimit}. 

We claim that $e'\circ e=\mathsf 1_{E}$. Assuming that $e'\circ e\ne \mathsf 1_E$, we can  use the generator property of $\mathtt 1$ and find a $\C$-morphism $u:\mathsf 1\to E$ such that $e'\circ e\circ u\ne \mathsf 1_E\circ u=u$. Let  $u'=e'\circ e\circ u$ and observe that $e\circ u'=e\circ e'\circ e\circ u=\mathsf 1_{FX}\circ e\circ u=e\circ u$. Then $u$ and $u'$ are two distinct morphisms  such that $e\circ u=e\circ u'$, which contradicts the uniqueness condition in the definition of an equalizer. This contradiction show that $e'\circ e=\mathsf 1_E$. Together with $e\circ e'=\mathsf 1_{FX}$ this implies that $e$ is an isomorphism and $e'$ is its inverse.

Now the equality $\chi\circ e=\zeta\circ e$ implies 
$$\chi=\chi\circ \mathsf 1_{FX}=\chi\circ (e\circ e')=(\chi\circ e)\circ e'=(\zeta\circ e)\circ e=\zeta\circ(e\circ e')=\zeta\circ \mathsf 1_{FX}=\zeta$$ and $$\true=\chi\circ\psi=\zeta\circ\psi=(\mathsf{false}\circ\tau)\circ\psi=\mathsf{false}\circ(\tau\circ\psi)=\mathsf{false}\circ\mathsf 1_{\mathtt 1}=\mathsf{false},$$
which contradicts the choice of the morphisms $\true$ and $\false$.
\end{proof}

Claims~\ref{cl:eta-inj} and \ref{cl:eta-sur} imply that the function $\eta_X$ is bijective and hence is an isomorphism of the category $\Set$,
 \smallskip

Next, we define a natural transformation $\e:FG\to\mathsf 1_{\C}$.   For every $\C$-object $Y$, consider the set $GY=\Mor(\mathtt 1,Y)$ and the coproduct $FGY=\amalg_{y\in\Mor(\mathtt 1,Y)}\mathtt 1$. By the universal property of the coproduct, there exists a unique $\C$-morphism $\e_Y:FGY\to Y$ such that
\begin{equation}\label{eq:epsilon}
y=\e_Y\circ \eta_{GY}(y)\quad\mbox{for every $y\in \Mor(\mathtt 1,Y)$}.
\end{equation}
The morphism $\e_Y$ is a component of the natural transformation $\e:FG\to\mathsf 1_\C$.

In the following two claims we show that for every $\C$-object $Y$ the morphism $\e_{Y}:FGY\to Y$ is an isomorphism in the category $\C$. 

\begin{claim}\label{cl:e-epi}  $\e_Y$ is an epimorphism. 
\end{claim}

\begin{proof} Assuming that $\e_Y$ is not epi, we can find a $\C$-object $Y'$ and two distinct morphisms $g,g'\in\Mor(Y,Y')$ such that $g\circ \e_Y=g'\circ\e_Y$. Since $\mathtt 1$ is a generator, there exists a morphism $y\in\Mor(\mathtt 1,Y)$ such that $g\circ y\ne g'\circ y$. Consider the morphism $\eta_{GY}(y)\in\Mor(\mathtt 1,FGY)$ and observe that $\e_Y\circ \eta_{GY}(y)=y$, see the equation (\ref{eq:epsilon}). Then $g\circ y=g\circ \e_Y\circ \eta_{GY}(y)=g'\circ \e_Y\circ\eta_{GY}(y)=g'\circ y$, which contradicts the choice of the morphisms $g,g'$. 
\end{proof}

\begin{claim}\label{cl:e-mono} $\e_Y$ is a monomorphism.
\end{claim}

\begin{proof} By Claims~\ref{cl:eta-inj}, \ref{cl:eta-sur}, the function $\eta_{GY}:\Mor(\mathtt 1,Y)\to \Mor(FGY)$ is bijective. Assuming that $\e_Y$ is not a monomorphism and taking into account that $\mathtt 1$ is a generator, we can find two distinct morphisms $\phi,\psi\in\Mor(\mathtt 1,FGY)$ such that $\e_Y\circ \phi=\e_Y\circ \psi$. By the bijectivity of the function $\eta_{GY}$, there are distinct morphisms $\phi',\psi'\in\Mor(\mathtt 1,Y)$ such that $\eta_{GY}(\phi')=\phi$ and $\eta_{GY}(\psi')=\psi$. The definition of the morphism $\e_Y$ guarantees that $$\phi'=\e_Y\circ\eta_{GY}(\phi')=\e_Y\circ\phi=\e_Y\circ\psi=\e_Y\circ\eta_{GY}(\psi')=\psi',$$which is a desired contradiction.
\end{proof}

By Claims~\ref{cl:e-epi}, \ref{cl:e-mono}, the morphism $\e_Y$ is a bimorphism. Since the category $\C$ is balanced, the morphism $\e_Y$ is an isomorphism. 

Therefore we proved that the natural transformations $\eta:\mathsf 1_{\Set}\to GF$ and $\e:FG\to \mathsf 1_\C$ are functor isomorphisms, witnessing that the categories $\C$ and $\Set$ are equivalent. 
\end{proof}

Lemma~\ref{l:Equi-Set} implies the following characterizations of the full subcategory $\Crd\subseteq\Set$ whose objects are cardinals.

\begin{theorem} Under $(\mathsf{AGC})$, a category $\C$ is isomorphic to the category $\Crd$ if and only if it satisfies the following properties:
\begin{enumerate}
\item[\textup{(0)}] $\C$ is skeletal;
\item[\textup{(1)}] $\C$ is locally small;
\item[\textup{(2)}] $\C$ is balanced;
\item[\textup{(3)}] $\C$ has equalizers;
\item[\textup{(4)}] $\C$ has arbitrary coproducts;
\item[\textup{(5)}] $\C$ has a terminal object $\mathtt 1$;
\item[\textup{(6)}] $\mathtt 1$ is a $\C$-generator;
\item[\textup{(7)}]  $\C$ is element-separating.
\end{enumerate}
\end{theorem}

\begin{proof} The ``only if'' part is trivial and holds without $(\mathsf{AGC})$. To prove the ``if'' part, assume that the Axiom of Global Choice holds and a category $\C$ has properties (0)--(7).  By Lemma~\ref{l:Equi-Set}, the skeletal category $\C$ is equivalent to the category $\Set$. By Theorem~\ref{t:Set=Card}, under $(\mathsf{AGC})$, the category $\Set$ is equivalent to its skeleton $\Crd$. Consequently, the skeletal categories $\C$ and $\Crd$ are equivalent, and by Proposition~\ref{p:skelet:i=e}, these categories are isomorphic.
\end{proof}

\begin{theorem}\label{t:EC=Set} Under $(\mathsf{EC})$, a category $\C$ is equivalent to the category $\Set$ if and only if it satisfies the following properties:
\begin{enumerate}
\item[\textup{(1)}] $\C$ is locally small;
\item[\textup{(2)}] $\C$ is balanced;
\item[\textup{(3)}] $\C$ has equalizers;
\item[\textup{(4)}] $\C$ has arbitrary coproducts;
\item[\textup{(5)}] $\C$ has a terminal object $\mathtt 1$;
\item[\textup{(6)}] $\mathtt 1$ is a $\C$-generator;
\item[\textup{(7)}] $\C$ is element-separating.
\end{enumerate}
\end{theorem}

\begin{proof}  The ``only if'' part is trivial and holds without $(\mathsf{EC})$. To prove the ``if'' part, assume that the principle $(\mathsf{EC})$ holds and a category $\C$ has properties (1)--(7).

By Theorem~\ref{t:skeleton}, under $(\mathsf{EC})$, the category $\C$  has a skeleton $\mathcal S$, which is equivalent to $\mathcal C$. Since the properties (1)--(7) are preserved by the equivalnce of categories, the category $\mathcal S$ has respective properties (1)--(7) and by Lemma~\ref{l:Equi-Set}, the category $\mathcal S$ is equivalent to the category $\Set$. Then $\C\simeq\mathcal S\simeq\Set$.
\end{proof} 

\begin{theorem}\label{t:GWO=Set} Under $(\mathsf{GWO})$, a category $\C$ is isomorphic to the category $\Set$ if and only if it satisfies the following properties:
\begin{enumerate}
\item[\textup{(1)}] $\C$ is locally small;
\item[\textup{(2)}] $\C$ is balanced;
\item[\textup{(3)}] $\C$ has equalizers;
\item[\textup{(4)}] $\C$ has arbitrary coproducts;
\item[\textup{(5)}] $\C$ has a terminal object $\mathtt 1$;
\item[\textup{(6)}] $\mathtt 1$ is a $\C$-generator;
\item[\textup{(7)}] $\C$ is element-separating;
\item[\textup{(8)}] $\C$ has a unique initial object $\mathtt 0$;
\item[\textup{(9)}] for any non-initial $\C$-object $x$ the class $\{y\in\Ob_\C:y\cong_\C x\}$ is a proper class.
\end{enumerate}
\end{theorem}

\begin{proof} The  ``only if'' part is trivial and holds without $(\mathsf{EC})$. To prove the ``if'' part, assume that the principle $(\mathsf{GWO})$ holds and a category $\C$ has properties (1)--(9). Since $(\mathsf{GWO})\Ra(\mathsf{EC})$, we can apply Theorem~\ref{t:EC=Set} and conclude that the categories $\C$ and $\Set$ are equivalent. By Theorem~\ref{t:equivalent}, there exists a full faithful functor $F:\C\to\Set$, which is essentially surjective on objects. 

By the condition (8), the category $\C$ contains a unique initial object $\mathtt 0$. Since the functor $F$ is essentially surjective on objects, for the empty set $\emptyset\in\Ob_{\Set}$, there exists a $\C$-object $Z$ such that $FZ\cong \emptyset$. Since $\emptyset$ is an initial object of the category $\Set$, the object $Z$ is initial in the category $\C$ and hence $Z=\mathtt 0$ by the uniqueness of the initial object $\mathtt 0$ in $\C$. Then $$|\{x\in \Ob_\C:x\cong\mathtt 0\}|=|1|=|\{y\in\Ob_{\Set}:y\cong\emptyset\}|.$$
On the other hand, for any $\C$-object $x\ne\mathtt 0$, the uniqueness of an initial object in $\C$ implies that $x$ is not initial in $\C$ and hence $Fx$ is not initial in $\Set$. The latter means that the set $Fx$ is not empty and then $\{z\in\UU:|z|=|Fx|\}$ is a proper class. By the condition (9), the class $\{y\in\Ob_\C:y\cong x\}$ is proper, too. By the principle $(\mathsf{GWO})$ the proper classes $\{y\in \Ob_\C:y\cong x\}$ and $\{z\in \UU:|z|=|Fx|\}$ are well-orderable. By Theorem~\ref{t:wOrd} these classes admit a bijective function onto the class $\Ord$, which implies that $|\{y\in\Ob_\C:y\cong Fx\}|=|\{z\in\UU:|z|=|x|\}|$. Applying Theorem~\ref{t:GWO-iso-cat}, we conclude that the categories $\C$ and $\Set$ are isomorphic.
\end{proof}

\begin{remark} Among conditions characterizing the category $\Set$ there are two conditions that have non-finitary nature, namely, the local smallness and the existence of arbitrary colimits. Attempts to give a finitary definition of a category that resembles the category of sets lead Lawvere and Tierney to discovering the notion of an elementary topos: this is a cartesian closed category with a subobject classifier. We shall briefly discuss these notions in the next three sections.
\end{remark}

\section{Cartesian closed categories}

\begin{definition} A category $\C$ with binary products is called \index{cartesian closed category}\index{category!cartesian closed}{\em cartesian closed} if for any $\C$-objects $X,Y$ there is an \index{exponential object}{\em exponential object} $Y^X\in\Ob_\C$ and an \index{evaluation morphism}{\em evaluation morphism} $\ev_{X,Y}:Y^X\times X\to Y$ with the universal property that for every $\C$-object $Z$ and $\C$-morphism $f:Z\times X\to Y$ there exist unique $\C$-morphisms $[f]:Z\to Y^X$ and $[f]{\times}\mathsf 1_X:Z\times X\to Y^X\times X$ making the following diagram commutative.
$$
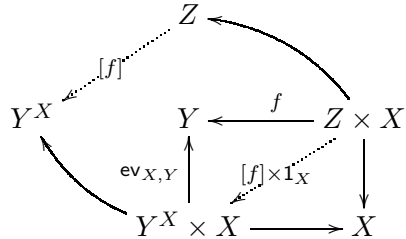
\xymatrix{
&Z\ar@{..>}|{[f]}[ld]\\
Y^X&Y&Z\times X\ar_f[l]\ar@{..>}|{[f]{\times}\mathsf 1_X}[ld]\ar[d]\ar@/_10pt/[lu]\\
&Y^X\times X\ar[r]\ar@/^10pt/[lu]\ar^{\ev_{X,Y}}[u]&X
}
$$In this diagram by arrows without labels we denote the coordinate projections.
\end{definition}

\begin{example} The category $\Set$ is cartesian closed: for any sets $X,Y$ the exponential object $Y^X$ is the set of all functions $f:X\to Y$, and the evaluation morphism $\ev_{X,Y}:Y^X\times X\to Y$ assigns to every ordered pair $\langle \varphi,x\rangle\in Y^X\times X$ the value $\varphi(x)$ of $\varphi$ at $x$. For every set $Z$ and function $f:Z\times X\to Y$ the function $[f]:Z\to Y^X$ assigns to every element $z\in Z$ the function $[f]_z:X\to Y$, $[f]_z:x\mapsto f(z,y)$.
\end{example}

\begin{exercise} Let $\C$ be a cartesian closed category. Prove that for any $\C$-objects $X,Y,Z$ we have $\C$-isomorphisms:
\begin{itemize}
\item $(Y^X)^Z\cong Y^{X\times Y}$;
\item $Y^X\times Z^X\cong (Y\times Z)^X$;
\item $X\cong X^{\mathtt 1}$ where $\mathtt 1$ is a terminal object in $\C$.
\end{itemize}
\end{exercise}

\section{Subobject classifiers}

In category theory subobjects correspond to subsets in the category of sets. Since the category theory does not ``see'' the inner structure of objects, subobjects should be defined via morphisms. The idea is to identify subobjects of a given object $A$ with equivalence classes of monomorphisms into $A$. 

We say that two $\C$-morphisms $f:X\to A$ and $g:Y\to A$ of a category $\C$ are {\em isomorphic} if there exists a $\C$-isomorphism $h:X\to Y$ such that $f=g\circ h$. For a $\C$-morphism $f$ by $[f]_{\cong}$ we denote the class of $\C$-morphisms, which are isomorphic to $f$.

By definition, a \index{subobject}{\em subobject} of a $\C$-object $A$ is the equivalence class $[i]_{\cong}$ of some monomoprphism $i:X\to A$.


Such definition of a subobject is not very convenient to work with because very often subobjects are  proper classes. So, it is not even possible to define the class of all subobjects of a given object of a category. In the category of sets subobjects of a given set $A$ can be identified with subsets of $A$. In its turn, using characteristic functions, we can identify each subset $X\subseteq A$ with the characteristic function $\chi_X:A\to 2$. So, function into the doubleton $2=\{0,1\}$ classify subobjects in the category of sets. This property of the doubleton motivates the following definition.

\begin{definition} Let $\C$ be a category that has a terminal object $\mathtt 1$. A \index{subobject classifier}{\em subobject classifier} is a $\C$-object $\Omega$ endowed with a $\C$-morphism $\true:\mathtt 1\to\Omega$  such that the following two properties are satisfied:

\noindent 1) for any $\C$-morphism $\chi:A\to\Omega$ the diagram $\mathtt 1\overset{\true}\longrightarrow \Omega\overset{\chi}\longleftarrow A$ has a pullback, and

\noindent 2) for any monomorphism $i:X\to A$ in the category $\C$ there exists a  unique $\C$-morphism $\chi_i:A\to\Omega$, called {\em the characteristic morphism for the monomorphism $i$}, such that for the unique $\C$-morphism $X\to \mathtt 1$, the square
\begin{equation}\label{eq:subobject}
\xymatrix{
X\ar^i[r]\ar[d]&A\ar^{\chi_i}[d]\\
\mathtt 1\ar_{\true}[r]&\Omega
}
\end{equation}
is a pullback, which means that for any $\C$-object $Y$ and $\C$-morphisms $f:Y\to A$ and $g:Y\to\mathtt 1$ with $\chi_i\circ f=\true\circ g$ there exists a unique $\C$-morphism $h:Y\to X$ such that $i\circ h=f$.
\end{definition}

The uniqueness of the morphism $\chi_i$ and the pullback property of the square (\ref{eq:subobject}) imply that for a $\C$-object $A$, two monomorphisms $i:X\to A$ and $j:y\to A$ are isomorphic if and only if $\chi_i=\chi_j$ if and only if $[i]_{\cong}=[j]_{\cong}$. This means that subobjects of a $\C$-object $A$ are in the bijective correspondence with  $\C$-morphism from $A$ to $\Omega$. The surjectivity of this correspondence follows from the following property of pullbacks.

\begin{exercise} Prove that for any pullback square
$$\xymatrix{
X\ar^i[r]\ar[d]&Y\ar[d]\\
\mathtt 1\ar[r]&Z
}
$$the morphism $i$ is always mono.
\end{exercise}

\begin{proposition} If a category $\C$ has a subobject classifier $\true:\mathtt 1\to\Omega$, then it is unique up to an isomorphism.
\end{proposition}

\begin{proof} Assume that $\true:\mathtt 1\to\Omega$ and $\true':\mathtt 1\to\Omega'$ are two subobject classifiers. Since $\mathtt 1$ is a terminal object, the morphisms $\true$ and $\true'$ are monomorphisms. Then there are unique characteristic functions $\chi:\Omega'\to\Omega$ and $\chi':\Omega\to\Omega'$ such that the upper and lower squares of the following diagram are pullbacks:  
$$
\xymatrix{
\mathtt 1\ar[d]\ar^{\true}[r]&\Omega\ar^{\chi'}[d]\\
\mathtt 1\ar[d]\ar^{\true'}[r]&\Omega'\ar^{\chi}[d]\\
\mathtt 1\ar^{\true}[r]&\Omega
}
$$
Then the external square also is a pullback and then $\chi\circ\chi'$ is the identity morphism of $\Omega$ by the definition of a pullback. By analogy we can prove that $\chi'\circ\chi=\mathsf 1_{\Omega'}$. This means that the morphism $\chi:\Omega'\to\Omega$ is an isomorphism.
\end{proof}

\begin{example} In the category of sets, a subobject classifier exists: it is the function $\true=\{\langle 0,1\rangle\}:1\to 2$.
\end{example}

The morphism $\true:1\to 2$ can be defined in any category with a terminal object $\mathtt 1$ and finite coproducts. Namely, let $\mathtt 2=\mathtt 1\sqcup\mathtt 1$ be a coproduct of two copies of $\mathtt1$ and $\mathtt{false}:\mathtt 1\to\mathtt 2$, $\true:\mathtt 1\to\mathtt 2$ be the first and second coordinate coprojections, respectively.

\begin{proposition} The doubleton $\mathtt 2$ endowed with the morphism $\true:\mathtt 1\to\mathtt 2$ is a subobject classifier in the category of sets.
\end{proposition}

\begin{proof} In the category of sets the terminal object $\mathtt 1$ is isomorphic to the natural number $1=\{0\}$ and the coproduct $\mathtt 2=\mathtt 1\sqcup \mathtt 1$ is isomorphic to the natural number $2=\{0,1\}$. Then the morphism $\true:\mathtt 1\to \mathtt 2$ can be identified with the function $\{\langle 0,1\rangle\}:1\mapsto\{1\}\subseteq 2$. Given any injective function $i:X\to Y$ between sets, consider the characteristic function $\chi:Y\to 2$ of the subset $i[X]$ of $Y$. By definition, $\chi$ is a unique function such that $$\chi(y)=\begin{cases}1&\mbox{if $y\in i[X]$};\\
0&\mbox{if $y\in Y\setminus i[X]$}.
\end{cases}
$$
We should prove that $\chi$ is a unique function making the square
$$\xymatrix{
X\ar^i[r]\ar_{u}[d]&Y\ar^\chi[d]\\
1\ar^{\true}[r]&2
}
$$
a pullback. The definition of the function $\chi$ ensures that this square is commutative. To prove that it is a pullback, take any set $Z$ and functions $f:Z\to Y$ and $g:Z\to 1$ such that $\chi\circ f=\true\circ g$.  The latter equality implies that $f[Z]\subseteq i[X]$. The injectivity of the function $i:X\to Y$ ensures that there exists a unique function $h:Z\to X$ such that $f=i\circ h$. The uniqueness of functions into $1$ guarantees that $g=u\circ h$. This means that the above square is indeed a pullback. 

To prove the uniqueness of the function $\chi$, take any function $\chi':Y\to 2$ for which the square 
$$\xymatrix{
X\ar^i[r]\ar_{u}[d]&Y\ar^{\chi'}[d]\\
1\ar^{\true}[r]&2
}
$$
is a pullback. The commutativity of this square implies that $\chi'[i[X]]\subseteq\{1\}$. Assuming that $\chi'\ne\chi$, we could find an element $y\in Y\setminus i[X]$ such that $\chi(y)=1$. Consider the function $f:1\to Y$ with $f(0)=y$ and observe that $\chi'\circ f=\true$. The pullback property of the square yields a unique function $h:1\to X$ such that $f=i\circ h$. Then $y=f(0)=i(h(0))\in i[X]$, which contradicts the choice of $y$.
\end{proof}

\begin{exercise} Let $\C,\C'$ be  categories possessing subobject classifiers $\true:\mathtt 1\to\Omega$ and $\true':\mathtt 1'\to\Omega$. Prove that $\langle\true,\true'\rangle$ is a subobject classifier of the product category $\C\times\C'$.
\end{exercise} 

\begin{exercise} Prove that the category $\Set\times\Set$ has a subobject classifier $\true:\mathtt 1\to\Omega$ with $|\Mor(\mathtt 1,\Omega)|=4$.
\end{exercise}

\begin{exercise} Prove that the category of functions $\Set^{\to}$ has a subobject classifier $\true:\mathtt 1\to\Omega$ with $|\Mor(\mathtt 1,\Omega)|=3$.
\end{exercise}

The existence of subobject classifiers impose some restrictions on a category. We recall that a category is {\em balanced} if each bimorphism (=mono+epi) is an isomorphism.

\begin{proposition}\label{p:topos-balanced} If a category $\C$ has a subobject classifier $\true:\mathtt 1\to\Omega$, then $\C$ is a balanced category.
\end{proposition}

\begin{proof}  Given an bimorphism $f:X\to Y$ in the category $\C$, find a unique $\C$-morphism $\chi:Y\to\Omega$ into the classifying object $\Omega$ making the square
$$
\xymatrix{
X\ar[d]\ar^f[r]&Y\ar^\chi[d]\\
\mathtt 1\ar_{\true}[r]&\Omega
}
$$a pullback. Then for the identity morphism $\mathsf 1_Y:Y\to Y$ and the unique morphism $Y\to 1$, the pullback property of this square implies the existence of a unique $\C$-morphism $h:Y\to X$ such that $f\circ h=\mathsf 1_Y$. Since $f$ is a monomorphism, the equality $f\circ (h\circ f)=(f\circ h)\circ f=\mathsf 1_Y\circ f=f\circ \mathsf 1_X$ implies $h\circ f=\mathsf 1_X$. Therefore, $f$ is an isomorphism with $f^{-1}=h$.
\end{proof}
 
\begin{proposition}\label{p:topos-equalizers} Assume that a category $\C$ has a subobject classifier $\true:\mathtt 1\to\Omega$. If $\C$ has binary squares, then it has limits of finite diagrams. It particular, it has equalizers and pullbacks.
\end{proposition} 

\begin{proof} To show that the category $\C$ has equalizers, fix any $\C$-morphisms $f,g:X\to Y$. Since $\C$ has binary products, it has a product $Y\times Y$. By definition of the product,  there exists a unique $\C$-morphism $\delta:Y\to Y\times Y$ such that $\pr_1\circ \delta=\mathsf 1_Y=\pr_2\circ\delta$, where  $\pr_1,\pr_2:Y\times Y\to Y$ are the coordinate projections of the product $Y\times Y$.  The latter equalities imply that $\delta$ is a monomorphism. Then there exists a $\C$-morphism $\chi:Y\times Y\to \Omega$ such that the square 
$$\xymatrix{
Y\ar[d]\ar^(.4)\delta[r]&Y\times Y\ar^\chi[d]\\
\mathtt 1\ar_{\true}[r]&\Omega
}
$$is a pullback. 

 By definition of the product $Y\times Y$, there exists a unique $\C$-morphism $(f,g):X\to Y\times Y$ such that $f=\pr_1\circ(f,g)$ and $g=\pr_2\circ(f,g)$. Now consider the morphism $h=\chi\circ(f,g):X\to\Omega$. By definition of a subobject classifier, there exists a pullback
$$\xymatrix{
E\ar[d]\ar^e[r]&X\ar^h[d]\\
\mathtt 1\ar_{\true}[r]&\Omega.
}
$$
It can be shown that the morphism $e:E\to X$ is an equalizer of the pair $(f,g)$.

The existence of binary products and equalizers implies the existence of pullbacks and limits of all finite diagrams, see Exercise~\ref{ex:p+e=>pb} and~\ref{ex:f-limits}.
\end{proof}

 


\section{Elementary Topoi}

{\small\em

\rightline{Lawvere's axioms for elementary topos helped many people}

\rightline{outside the community of specialists to enter into this field}

\rightline{and make a fruitful research in it.}

\rightline{Everyone who learns today the topos theory begins}

\rightline{with Lawvere's axioms for elementary topos.}

\rightline{This makes Lawvere's axiomatization of topos theory}

\rightline{a true success story of Axiomatic Method}

\rightline{in the twentieth century mathematics.}
}
\smallskip

\rightline{\small Andrei Rodin, ``Axiomatic Method and Category Theory'', 2014}
\bigskip

Elementary topoi were introduced by Lawvere and Tierney in 1968-69. Now the theory of elementary topoi is well-developed and is considered as a foundation of mathematics (alternative to Set Theory). The modern definition of an elementary topos  is very short.

\begin{definition} An \index{elementary topos}{\em elementary topos} is $\C$ is a cartesian closed category with a subobject classifier. 
\end{definition}

A standard example of an elementary topos is the category of sets. On the other hand, the categories $\Set\times\Set$ and $\Set^\to$ are elementary topoi, which are not equivalent to the category $\Set$ (because their subobject classifiers have more than two elements). 

For any $\C$-object $X$ of an elementary topos $\C$ with a subject classifier $\Omega$, we can consider the exponential object $\Omega^X$, called the \index{power object}{\em power object} of $X$. The power object $\Omega^X$ indexes all subobjects of $X$. Using the evaluation   morphism $\ev_{X,\Omega}:\Omega^X\times X\to\Omega$, for any global elements $s:\mathtt 1\to \Omega^X$ and $x:\mathtt 1\to X$, we can consider the morphism $\ev_{X,\Omega}\circ(s,x):\mathtt 1\to\Omega$ and compare it with the morphism $\true:\mathtt 1\to\Omega$. The equality $\ev_{X,\Omega}\circ(s,x)=\true$ can be interpreted as the indication that the global element $x$ ``belongs'' to the subobject $s$ of $X$. This allows to apply element-based arguments resembling those practiced in the classical Set Theory and Logics. The subject of Categorial Logic is very extensive and we will not develop it here referring the reader to the monographs \cite{Johnstone}, \cite{McLarty}, \cite{Streicher}.    

In this section we characterize elementary topoi, which are equivalent or isomorphic to the category of sets. 

 
\begin{definition} An elementary topos is called \index{elementary topos!well-pointed}{\em well-pointed} if its terminal object $\mathtt 1$ is a generator and its subobject classifier $\Omega$ is not a terminal object.
\end{definition} 

\begin{exercise} Show that the category of finite sets $\mathbf{FinSet}$ is a well-pointed elementary topos.
\end{exercise}

\begin{proposition}\label{p:wp=>tw} If an elementary topos $\C$ is well-pointed, then its subobject classifier $\Omega$ is two-valued in the sense that $|\Mor(\mathtt 1,\Omega)|=2$.
\end{proposition}

\begin{proof} Assume that an elementary topos $\C$ is well-pointed. Then its subobject classifier $\Omega$ is not a terminal object of the category $\C$. Because of the morphism $\true:\mathtt 1\to\Omega$, every $\C$-object $X$ has a morphism $X\to\Omega$. Since $\Omega$ is not terminal, there exists a  $\C$-object $X$ admitting two distinct morphisms $f,g:X\to\Omega$. Since $\mathtt 1$ is a generator, there exists a morphism $h:\mathtt 1\to X$ such that $f\circ h\ne g\circ h$. Consequently, $|\Mor(\mathtt 1,\Omega)|\ge2$.  Since the set $\Mor(\mathtt 1,\Omega)$ classifies subobjects of $\mathtt1$, the equality $|\Mor(\mathtt 1,\Omega)|=2$ will follow as soon as we show that $\mathtt 1$ has exactly two subobjects. 

Let $i:X\to\mathtt 1$ be any monomorphism. If there exists a morphism $x:1\to\mathtt X$, then $i\circ x=\mathsf 1_{\mathtt 1}$ and hence $i$ is an epimorphism. By Proposition~\ref{p:topos-balanced}, the category $\C$ is balanced, which implies that $i$ is an isomorphism. Assuming that $\mathtt 1$ has more than two subobjects, we can find two nonisomorphic monomophisms $u:U\to\mathtt 1$ and $v:V\to\mathtt 1$ such that $\Mor(\mathtt 1,U)=\emptyset=\Mor(\mathtt 1,V)$. 
Now consider the pullback
$$
\xymatrix{W\ar^{\phi}[r]\ar_\psi[d]&U\ar^u[d]\\
V\ar_v[r]&\mathtt 1
}
$$
which exists as the category $\C$ has binary products and equalizers according to Proposition~\ref{p:topos-equalizers}. We claim that the morphism $\Phi:W\to U$ is a monomorphism. In the opposite case we could find a $\C$-object $Z$ and distinct morphisms $f,g:Z\to W$ such that $\phi\circ f=\phi\circ g$. Since $\mathtt 1$ is a generator, there exists a $\C$-morphism $h:\mathtt 1\to Z$ such that $f\circ h\ne g\circ h$. Then the composition $\phi\circ f\circ h$ belongs to the class $\Mor(\mathtt 1,U)=\emptyset$, which is a desired contradiction showing that $\phi$ is a monomorphism. By analogy we can prove that $\psi$ is a monomorphism. Since the morphisms $u,v$ are not isomorphic, either $\phi$ or $\psi$ is not an isomormphism. We lose no generality assuming that $\phi$ is not an isomorphism. Then $U$ has two non-isomorphic monomorphisms: $\mathsf 1_U:U\to U$ and $\phi:W\to U$, which are classified by two distinct morphisms $\chi,\chi':U\to \Omega$. Since $\mathtt 1$ is a generator, there a morphism $\varphi:\mathtt 1\to U$ such that $\chi\circ\varphi\ne\chi'\circ\varphi$.  But $\varphi$ cannot exist as $\Mor(\mathtt 1,U)=\emptyset$.
This contradiction completes the proof of the equality $|\Mor(\mathtt 1,\Omega)|=2$.
\end{proof}

\begin{proposition}\label{p:topos-1sep} Each well-pointed elementary topos $\C$ is element-separating.
\end{proposition}

\begin{proof} Given any $\C$-morphism $x:\mathtt 1\to X$, observe that $x$ is a monomorphism (by the terminal property of $\mathtt 1$).  By definition of the subobject classifier, there exists a unique $\C$-morphism $\chi_x:X\to\Omega$ such that the commutative square 
$$
\xymatrix{
\mathtt 1\ar^x[r]\ar_{\mathsf 1_{\mathtt 1}}[d]&X\ar^{\chi_x}[d]\\
\mathtt 1\ar_{\true}[r]&\Omega
}
$$is a pullback. 

By Proposition~\ref{p:wp=>tw}, $|\Mor(\mathtt 1,\Omega)|=2$. Let $\false$ be the unique element of the set $\Mor(\mathtt 1,\Omega)\setminus\{\true\}$. The pullback property of the above square ensures that for any $y\in\Mor(\mathtt 1,X)\setminus\{x\}$ we have $\chi_x\circ y\ne\true$ and hence $\chi_x\circ y=\false$. Now we see that the morphisms $\chi_x:X\to\Omega$ and $\false:\mathtt 1\to\Omega$ witness that the category $\C$ is element-separating.
\end{proof}

Propositions~\ref{p:topos-balanced}, \ref{p:topos-equalizers}, \ref{p:topos-1sep} and Theorems~\ref{t:EC=Set}, \ref{t:GWO=Set} imply the following characterizations of the category $\Set$ (for the global choice principles $(\mathsf{EC})$ and $(\mathsf{GWO})$, see Section~\ref{s:GChoice}).
 
 \begin{theorem} Under $(\mathsf{EC})$, a category $\C$ is equivalent to the category $\Set$ if and only if $\C$ is a locally small well-pointed elementary topos that has arbitrary coproducts. 
 \end{theorem}
 
 \begin{theorem} Under $(\mathsf{GWO})$, a category $\C$ is isomorphic to the category $\Set$ if and only if 
 \begin{enumerate}
 \item[\textup{1)}] $\C$ is a well-pointed elementary topos;
 \item[\textup{2)}] $\C$ is locally small;
 \item[\textup{3)}] $\C$ has arbitrary coproducts;
 \item[\textup{4)}] $\C$ has a unique initial object;
 \item[\textup{5)}] for any non-initial $\C$-object $X$ the class of $\C$-objects that are isomorphic to $X$ is proper.
 \end{enumerate}
 \end{theorem}
 \newpage

\newpage

 \part*{Epilogue}
 
The material presented in this book is a reasonable minimum which a good (post-graduate) student in Mathematics should know about foundations of this science.

There are many nice textbooks that elaborate in detail on selected topics that were only briefly touched  in this textbook. 

In particular, in Mathematical Logic a classical textbook is that of Mendelson \cite{Mendelson}; there is also a new book of Kunen \cite{KunenFM}.

In Set Theory and Forcing recommended textbooks are those of Jech \cite{Jech} and Kunen \cite{KunenST}.

For surreal numbers we refer the interested reader to the original books of Conway \cite{ONAG} and Knuth \cite{Knuth}, and also to the survey paper \cite{Ehrlich21} of Ehrlich.

Mathematical Structures (of algebraic origin) and their relation to Category Theory are well-elaborated in the lecture notes of Bergman \cite{Bergman};  Model Theory can be further studied via the classical textbook Chang and Keisler \cite{Model}. The books \cite{BLGA} and \cite{OGTA} present foundations of Linear and Ordered Geometry, respectively. For a deeper study of Tarski's axioms of Euclidean and non-Euclidean geometry, we recommend the book \cite{SST}.

Category Theory can be studied using the classical textbook of Mac Lane \cite{McLane86}, and Topos Theory via the ``Elephant'' of Johnstone \cite{Johnstone}. A short and readable introduction to Category Theory and Categorical Logic is that of Streicher \cite{Streicher}. For a discussion of various foundations of Category Theory, see the expository paper of Shulman \cite{Shulman}.

\printindex

\end{document}

\end{document}